\documentclass[12pt]{amsart}

\usepackage{amssymb}
\makeatletter
\RequirePackage{xcolor}
\newcommand*{\Dfn}[1]{\emph{\color{blue}#1}\index{#1}}
\newcommand*{\qbinom}[2]{\genfrac{[}{]}{0pt}{}{#1}{#2}_{q}}

\newcommand*{\bC}{\mathbb{C}}
\newcommand*{\bR}{\mathbb{R}}
\newcommand*{\bQ}{\mathbb{Q}}
\newcommand*{\bZ}{\mathbb{Z}}
\newcommand*{\bN}{\mathbb{N}}

\newcommand*{\fg}{\mathfrak{g}}
\DeclareMathOperator{\Hom}{Hom}
\DeclareMathOperator{\End}{End}

\newtheorem{thm}{Theorem}[section]
\newtheorem{defn}[thm]{Definition}
\newtheorem{rem}[thm]{Remark}
\newtheorem{cor}[thm]{Corollary}
\newtheorem{ex}[thm]{Example}
\newtheorem{prop}[thm]{Proposition}

\newtheorem{lemma}[thm]{Lemma}

\DeclareFontFamily{OT1}{pzc}{}
\DeclareFontShape{OT1}{pzc}{m}{it}{<-> s * [1.10] pzcmi7t}{}
\DeclareMathAlphabet{\mathpzc}{OT1}{pzc}{m}{it}

\makeatother
\usepackage{tikz,bbm,etoolbox,mathtools,hyperref}

\numberwithin{equation}{section}

\usetikzlibrary{cd,svg.path,decorations.pathreplacing,decorations.markings}

\newcommand{\editorialnote}[2][]{}

\newcommand\Irr{\mathit{Irr}}

\newcommand\Dcat{\mathpzc{D}}

\newcommand\Icat{\mathpzc{Inv}\index{invariant tensors}}
\newcommand\Pcat{\mathpzc{P}}
\newcommand\Wcat{\Intcat} %
\newcommand\Intcat{\mathpzc{L}}\index{interpolating category}
\newcommand\threecat{\Intcat_{\le \obj3}}
\newcommand\Rcat{\mathpzc{R}\index{ribbon category}}
\newcommand\Vcat{\mathpzc{V}}
\newcommand\ev{\mathrm{ev}}
\newcommand\co{\mathrm{co}}
\newcommand\unit{\mathbbm{1}}
\DeclareMathOperator{\Aut}{Aut}

\DeclareMathOperator{\id}{id}

\newcommand\kk{\Bbbk}

\newcommand\Rot{\mathrm{Rot}}
\newcommand\Tr{\mathrm{Tr}\index{trace}}
\definecolor{sienna}{rgb}{0.91, 0.45, 0.32}
\newcommand\bB{\mathcal{B}}
\newcommand\osp{\mathfrak{osp}}
\newcommand\fso{\mathfrak{so}}

\newcommand\obj[1]{\mathit{#1}}
\newcommand\casimir{\mathbf{c}}
\newcommand\fU{\mathfrak{U}}
\newcommand\OSp{\mathit{OSp}_{1|2}}

\tikzset{bwwpicture/.style={rotate=90,baseline={([yshift=-\the\dimexpr\fontdimen22\textfont2\relax]current  bounding  box.center)}, red, ultra thick,scale=.9,line join=round}}
\tikzset{bwwcircle/.style={bwwpicture,execute at begin picture={\path (0,0) circle [radius=1.2]; \clip (0,0) circle [radius=1.0]; \fill[bgdfill] circle [radius=1.0];},
execute at end picture={\draw[bgddraw] circle [radius=1.0];}}}
\tikzset{bwwrect/.style 2 args={bwwpicture,execute at begin picture={\fill[bgdfill] #1 coordinate (a) rectangle #2 coordinate (b); \path (barycentric cs:a=1.11,b=-.11) rectangle (barycentric cs:a=-.11,b=1.11); \clip (a) rectangle (b);},
execute at end picture={\draw[bgddraw] #1 rectangle #2;}},bwwrect/.default={(0,0)}{(1.5,1.5)}}

\tikzstyle{bgddraw}=[blue, thin]
\tikzstyle{bgdfill}=[fill=blue!20]
\tikzstyle{wipe}=[double distance=1.6pt,double=red,blue!20,line width=3pt]

\def\trivalent{
	\begin{tikzpicture}[bwwcircle]
	\draw  (0,0) -- (270:1);
	\draw  (0,0) -- (30:1);
	\draw  (0,0) -- (150:1);
\end{tikzpicture}
}
\def\trivalenta{
  \begin{tikzpicture}[bwwpicture]
    \clip (0,0) circle [radius=1.0]; \fill[bgdfill] circle [radius=1.0];
    \draw  (0,0) -- (270:1);
    \draw  (0,0) -- (30:1);
    \draw  (0,0) -- (150:1);
\end{tikzpicture}
}

\def\circle{
	\begin{tikzpicture}[bwwcircle]
		\draw  (0,0) circle [radius=0.5];		
	\end{tikzpicture}
}

\def\empty{
	\begin{tikzpicture}[bwwcircle]
	\end{tikzpicture}
}

\def\tadpole{
	\begin{tikzpicture}[bwwcircle]
		\draw  (0,-0.2) -- (270:1);
	\draw  (0,-0.2) to[out=30,in=270] +(45:0.5) arc(0:180:0.3535) to[out=270,in=150] (0,-0.2);
      \end{tikzpicture}
}

\def\bigon{
	\begin{tikzpicture}[bwwcircle]
		\draw  (0,0.5) to [out=210, in=90] (-0.4,0) to [out=270, in=150] (0,-0.5);
	\draw  (0,0.5) to [out=330, in=90] (0.4,0) to [out=270, in=30](0,-0.5);
	\draw  (0,0.5) -- (0,1);
	\draw  (0,-0.5) -- (0,-1);
	\end{tikzpicture}
}

\def\linecirc{
	\begin{tikzpicture}[bwwcircle]
		\draw  (0,1) -- (0,-1);
	\end{tikzpicture}
}
\def\linecirca{
	\begin{tikzpicture}[bwwpicture]
          \clip (0,0) circle [radius=1.0]; \fill[bgdfill] circle [radius=1.0];
          \draw  (0,1) -- (0,-1);
	\end{tikzpicture}
}

\def\twistover{
	\begin{tikzpicture}[bwwcircle]
		\draw  (30:1) to [out=210,in=45] (0,0.2) to [out=225, in=180] (0,-0.5);
	\draw [wipe] (150:1) to [out=330,in=135] (0,0.2) to [out=-45, in=0] (0,-0.5);
	\draw  (0,-0.5) -- (0,-1);
	\end{tikzpicture}
}

\def\twistunder{
	\begin{tikzpicture}[bwwcircle]
          \draw  (150:1) to [out=330,in=135] (0,0.2) to [out=-45, in=0] (0,-0.5);
          \draw [wipe] (30:1) to [out=210,in=45] (0,0.2) to [out=225, in=180] (0,-0.5);
          \draw  (0,-0.5) -- (0,-1);
	\end{tikzpicture}
}

\def\trianglecirc{
	\begin{tikzpicture}[bwwcircle]
		\draw  (90:0.5) -- (90:1);
	\draw  (210:0.5) -- (210:1);
	\draw  (330:0.5) -- (330:1);
	\draw  (90:0.5) to [out=210, in=90] (210:0.5);
	\draw  (210:0.5) to [out=330, in=210] (330:0.5);
	\draw  (330:0.5) to [out=90, in=330] (90:0.5);
	\end{tikzpicture}
}

\def\overcirc{
\begin{tikzpicture}[bwwcircle]
		\draw  (45:1) -- (225:1);
	\draw [wipe] (135:1) -- (315:1);
\end{tikzpicture}
}

\def\overcirca{
\begin{tikzpicture}[bwwpicture]
  \clip (0,0) circle [radius=1.0]; \fill[bgdfill] circle [radius=1.0];
  \draw  (45:1) -- (225:1);
  \draw [wipe] (135:1) -- (315:1);
\end{tikzpicture}
}

\def\undercirc{
\begin{tikzpicture}[bwwcircle]
		\draw (135:1) -- (315:1);
	\draw [wipe] (45:1) -- (225:1);
\end{tikzpicture}
}

\def\Icirc{
\begin{tikzpicture}[bwwcircle]
		\draw  (45:1) to [out=225,in=135] (315:1);
	\draw  (135:1) to [out=315,in=45] (225:1);
\end{tikzpicture}
}

\def\Ucirc{
\begin{tikzpicture}[bwwcircle]
		\draw  (45:1) to [out=225,in=315] (135:1);
	\draw  (315:1) to [out=135,in=45] (225:1);
\end{tikzpicture}
}

\def\Kcirc{
\begin{tikzpicture}[bwwcircle]
		\draw  (45:1) to [out=225,in=30] (0,0.3);
	\draw  (135:1) to [out=315,in=150] (0,0.3);
	\draw  (225:1) to [out=45,in=210] (0,-0.3);
	\draw  (315:1) to [out=135,in=330] (0,-0.3);
	\draw  (0,0.3) -- (0,-0.3);
\end{tikzpicture}
}

\def\Hcirc{
\begin{tikzpicture}[bwwcircle]
		\draw  (45:1) to [out=225,in=60] (0.3,0);
	\draw  (135:1) to [out=315,in=120] (-0.3,0);
	\draw  (225:1) to [out=45,in=240] (-0.3,0);
	\draw  (315:1) to [out=135,in=300] (0.3,0);
	\draw  (0.3,0) -- (-0.3,0);
\end{tikzpicture}
}

\def\squarecirc{
\begin{tikzpicture}[bwwcircle]
		\draw  (45:0.5) -- (45:1);
	\draw  (135:0.5) -- (135:1);
	\draw  (225:0.5) -- (225:1);
	\draw  (315:0.5) -- (315:1);
	\draw  (45:0.5) to [out=165,in=15] (135:0.5);
	\draw  (135:0.5) to [out=255,in=105] (225:0.5);
	\draw  (225:0.5) to[out=345,in=195] (315:0.5);
	\draw  (315:0.5) to [out=75,in=285] (45:0.5);
\end{tikzpicture}
}

\def\caprect{
\begin{tikzpicture}[bwwrect,yscale=1.1]
	\draw  (1.17,0) arc [start angle=0,end angle=180,x radius=0.417,y radius=0.6];
\end{tikzpicture}
}

\def\cuprect{
\begin{tikzpicture}[bwwrect,yscale=-1.1]
	\draw  (1.17,0) arc [start angle=0,end angle=180,x radius=0.417,y radius=0.6];
\end{tikzpicture}
}

\newcommand\overrect[1][1]{
\begin{tikzpicture}[bwwrect,yscale=1.11,scale=#1]
	\draw  (1.17,1.5) to[out=270, in=90] (0.33,0);
	\draw [wipe] (0.33,1.5) to[out=270, in=90] (1.17,0);
\end{tikzpicture}
}

\newcommand\underrect[1][1]{
\begin{tikzpicture}[bwwrect,yscale=1.11,xscale=-1,scale=#1]
	\draw  (1.17,1.5) to[out=270, in=90] (0.33,0);
	\draw [wipe] (0.33,1.5) to[out=270, in=90] (1.17,0);
\end{tikzpicture}
}

\def\trivalentrect{\diag[.337]{diag2_1_0}}

\def\ribbon{
\begin{tikzpicture}[bwwrect={(-1.25,-2)}{(1.25,2)},yscale=0.75]
	\draw (-1,2)  to[out=270,in=90] (0,-1) arc(360:180:0.5);
	\draw (-1,-1) to[out=90,in=270] (1,1);
	\draw (1,1)[wipe] arc(0:180:0.5) to[out=270,in=90]  (1,-2);
	\draw[wipe] (-1,2)  to[out=270,in=90] (0,-1) arc(360:180:0.5);
\end{tikzpicture}
}

\def\topvertex{
	\begin{tikzpicture}[bwwpicture]%
	\clip (1,0) -- (-1,0) arc(180:360:1); \fill[bgdfill] (1,0) -- (-1,0) arc(180:360:1);
        \draw[bgddraw]  (1,0) -- (-1,0);%
	\draw (0,0) -- (0,-1);
	\end{tikzpicture}
}

\def\botvertex{
	\begin{tikzpicture}[bwwpicture]
	\clip (-1,0) -- (1,0) arc(0:180:1); \fill[bgdfill] (-1,0) -- (1,0) arc(0:180:1);
        \draw[bgddraw] (-1,0) -- (1,0);%
	\draw (0,0) -- (0,1);
	\end{tikzpicture}
}

\def\invlft{
	\begin{tikzpicture}[bwwrect={(0,0)}{(2.5,2.5)},xscale=1.1]
	\draw (0.5,2.5) to[out=270, in=150] (0.875,1.5);
	\draw (1.25,2.5) to[out=270, in=30] (0.875,1.5);
	\draw (0.875,1.5) -- (0.875,1) arc(180:360:0.5625) -- (2,2.5);
	\end{tikzpicture}
}

\def\invrht{
	\begin{tikzpicture}[bwwrect={(0,0)}{(2.5,2.5)},xscale=-1.1]
	\draw (0.5,2.5) to[out=270, in=150] (0.875,1.5);
	\draw (1.25,2.5) to[out=270, in=30] (0.875,1.5);
	\draw (0.875,1.5) -- (0.875,1) arc(180:360:0.5625) -- (2,2.5);
	\end{tikzpicture}
}

\def\Xpos{
  \begin{tikzpicture}[bwwcircle]
    \posdot{0,0};
    \draw (45:1) -- (225:1);
    \draw	(135:1) -- (315:1);
  \end{tikzpicture}
}

\def\topover{
  \begin{tikzpicture}[bwwcircle]
    \draw (45:1) to [out=225,in=105] (225:0.5);
    \draw [wipe] (135:1) to [out=315,in=75] (315:0.5);
    \draw (225:0.5) to [out=345,in=195] (315:0.5);
    \draw (225:0.5) -- (225:1);
    \draw (315:0.5) -- (315:1);
  \end{tikzpicture}
}

\def\leftover{
  \begin{tikzpicture}[bwwcircle,rotate=90]
    \draw (45:1) to [out=225,in=105] (225:0.5);
    \draw [wipe] (135:1) to [out=315,in=75] (315:0.5);
    \draw (225:0.5) to [out=345,in=195] (315:0.5);
    \draw (225:0.5) -- (225:1);
    \draw (315:0.5) -- (315:1);
  \end{tikzpicture}
}

\def\botover{
  \begin{tikzpicture}[bwwcircle,rotate=180]
    \draw (45:1) to [out=225,in=105] (225:0.5);
    \draw [wipe] (135:1) to [out=315,in=75] (315:0.5);
    \draw (225:0.5) to [out=345,in=195] (315:0.5);
    \draw (225:0.5) -- (225:1);
    \draw (315:0.5) -- (315:1);
  \end{tikzpicture}
}

\def\rightover{
  \begin{tikzpicture}[bwwcircle,rotate=270]
    \draw  (45:1) to [out=225,in=105] (225:0.5);
    \draw [wipe] (135:1) to [out=315,in=75] (315:0.5);
    \draw  (225:0.5) to [out=345,in=195] (315:0.5);
    \draw  (225:0.5) -- (225:1);
    \draw  (315:0.5) -- (315:1);
  \end{tikzpicture}
}

\def\topoverbar{
  \begin{tikzpicture}[bwwcircle,rotate=180,xscale=-1]
    \draw (45:1) to [out=225,in=105] (225:0.5);
    \draw [wipe] (135:1) to [out=315,in=75] (315:0.5);
    \draw (225:0.5) to [out=345,in=195] (315:0.5);
    \draw (225:0.5) -- (225:1);
    \draw (315:0.5) -- (315:1);
  \end{tikzpicture}
}

\def\rotrect{
	\begin{tikzpicture}[bwwrect={(-0.5,-0.8)}{(1.5,1)}]
	\draw (0.25,0.4) arc (0:180:0.25) -- (-0.25,-0.8);
	\draw (0.75,-0.2) arc (180:360:0.25) -- (1.25,1);
	\draw (0.75,0.4) -- (0.75,1);
	\draw (0.25,-0.2) -- (0.25,-0.8);
	\filldraw [fill=sienna, draw=black,thin] (0.1,-0.2) rectangle (0.9,0.4);
	\end{tikzpicture}
}

\def\twotworect{
	\begin{tikzpicture}[bwwrect={(-0.1,-0.8)}{(1.1,1)}]
	\draw (0.25,1) -- (0.25,-0.8);
	\draw (0.75,1) -- (0.75,-0.8);
	\filldraw [fill=sienna, draw=black,thin] (0.1,-0.2) rectangle (0.9,0.4);
	\end{tikzpicture}
}

\def\tracerect{
	\begin{tikzpicture}[bwwrect={(-0.1,-0.8)}{(1.5,1)}]
	\draw (0.25,1) -- (0.25,-0.8);
	\draw (0.75,-0.2) arc (180:360:0.25) -- (1.25,0.4) arc (0:180:0.25);
	\filldraw [fill=sienna, draw=black,thin] (0.1,-0.2) rectangle (0.9,0.4);
	\end{tikzpicture}
}

\def\tracegeneralrect{
	\begin{tikzpicture}[bwwrect={(-0.5,-0.8)}{(2,1)}]
          \draw[rounded corners=0.5cm] (0,-0.2) -- (0,-0.7) -- (1.75,-0.7) -- (1.75,0.9) -- (0,0.9) -- (0,0.4);
          \draw[rounded corners=0.25cm] (1,-0.2) -- (1,-0.45) -- (1.5,-0.45) -- (1.5,0.65) -- (1,0.65) -- (1,0.4);
	\filldraw [fill=sienna, draw=black,thin] (-0.2,-0.2) rectangle (1.2,0.4);
        \node[black] at (.7,.6) {$\vdots$};
        \node[black] at (.7,-.4) {$\vdots$};
	\end{tikzpicture}
}

\def\generalrect{
  \begin{tikzpicture}[
    execute at begin picture={
      \draw[thin,black,decorate,decoration={brace},yshift=1mm] (0,1) -- node[left] {$\obj r$} (1,1);
      \draw[thin,black,decorate,decoration={brace},yshift=-1mm] (1,-.8) -- node[right] {$\obj s$} (0,-.8);
    },bwwrect={(-0.5,-0.8)}{(1.5,1)}]
	\draw (0,1) -- (0,-0.8);
	\draw (1,1) -- (1,-0.8);
	\filldraw [fill=sienna, draw=black,thin] (-0.2,-0.2) rectangle (1.2,0.4);
        \node[black] at (.6,.8) {$\vdots$};
        \node[black] at (.6,-.6) {$\vdots$};
	\end{tikzpicture}
}
\def\generalrecta{
  \begin{tikzpicture}[
    execute at begin picture={
      \draw[thin,black,decorate,decoration={brace},yshift=1mm] (0,1) -- node[left] {$\obj r$} (1,1);
      \draw[thin,black,decorate,decoration={brace},yshift=-1mm] (1,-.8) -- node[right] {$\obj r$} (0,-.8);
    },bwwrect={(-0.5,-0.8)}{(1.5,1)}]
	\draw (0,1) -- (0,-0.8);
	\draw (1,1) -- (1,-0.8);
	\filldraw [fill=sienna, draw=black,thin] (-0.2,-0.2) rectangle (1.2,0.4);
        \node[black] at (.6,.8) {$\vdots$};
        \node[black] at (.6,-.6) {$\vdots$};
	\end{tikzpicture}
}

\def\rotgeneralrect{
	\begin{tikzpicture}[execute at begin picture={
      \draw[thin,black,decorate,decoration={brace},yshift=1mm] (.5,1) -- node[left] {$\obj r$} (1.5,1);
      \draw[thin,black,decorate,decoration={brace},yshift=-1mm] (.5,-.8) -- node[right] {$\obj s$} (-.5,-.8);
    },bwwrect={(-0.7,-0.8)}{(1.7,1)}]
	\draw (0,0.4) arc (0:180:0.25) -- (-0.5,-0.8);
	\draw (1,-0.2) arc (180:360:0.25) -- (1.5,1);
	\draw (1,0.4) -- (1,1);
	\draw (0.,-0.2) -- (0.,-0.8);
        \draw (0.5,0.4) -- (0.5,1);
        \draw (0.5,-0.2) -- (0.5,-0.8);
	\filldraw [fill=sienna, draw=black,thin] (-0.2,-0.2) rectangle (1.2,0.4);
        \node[black] at (.85,.8) {$\vdots$};
        \node[black] at (.35,-.6) {$\vdots$};
	\end{tikzpicture}
}
\def\onegen{\diag[.337]{diag2_2_0}}

\def\Ugen{\diag[.337]{diag2_2_1}}

\def\Kgen{\diag[.337]{diag2_2_2}}

\def\Hgen{\diag[.337]{diag2_2_3}}

\def\overgen{\overrect}

\def\undergen{\underrect}

\def\Xposgen{
  \begin{tikzpicture}[bwwrect,yscale=1.11]
    \posdot{0.75,0.75}
    \draw  (1.17,1.5) to[out=270, in=90] (0.33,0);
    \draw  (0.33,1.5) to[out=270, in=90] (1.17,0);
  \end{tikzpicture}
}

\def\Xneggen{
  \begin{tikzpicture}[bwwrect,yscale=1.11]
    \negdot{0.75,0.75}
    \draw  (1.17,1.5) to[out=270, in=90] (0.33,0);
    \draw  (0.33,1.5) to[out=270, in=90] (1.17,0);
  \end{tikzpicture}
}

\def\overgenb{
	\begin{tikzpicture}[bwwrect={(0,-0.4)}{(1.5,0.6)},scale=1.2]
	\draw (0.75,0.6) to[out=270,in=90] (0.25,-0.4);
	\fill[blue!20] (0.5,0.1) circle (0.2);
	\draw (0.25,0.6) to[out=270,in=90] (0.75,-0.4);
        \draw (1.25,-0.4) -- (1.25,0.6);
	\end{tikzpicture}
}

\def\overgena{
	\begin{tikzpicture}[bwwrect={(-.5,-0.4)}{(1,0.6)},scale=1.2]
	\draw (0.75,0.6) to[out=270,in=90] (0.25,-0.4);
	\fill[blue!20] (0.5,0.1) circle (0.2);
	\draw (0.25,0.6) to[out=270,in=90] (0.75,-0.4);		
        \draw (-.25,-0.4) -- (-.25,0.6);
	\end{tikzpicture}
}

\newcommand\diag[2][.3]{{\tikzset{every picture/.style={line cap=round,sharp corners,scale=#1}}%
\ifcsname e8diagram:#2\endcsname
\csname e8diagram:#2\endcsname
\else\PackageError{exceptional}{Missing embedded diagram: #2}{}\fi\relax}}

\newcommand\posdot[1]{
	\filldraw[black] (#1) -- +(0.09192,0.09192) arc(45:135:0.13) -- cycle;
	\filldraw[black] (#1) -- +(0.09192,-0.09192) arc(-45:-135:0.13) -- cycle;
}
\newcommand\negdot[1]{
	\filldraw[black] (#1) -- +(0.09192,0.09192) arc(45:-45:0.13) -- cycle;
	\filldraw[black] (#1) -- +(-0.09192,0.09192) arc(135:225:0.13) -- cycle;
}

\def\myfont{\mathsf}

\newcommand\qq{\myfont{q}}          %
\newcommand\pp{\myfont{p}}          %
\newcommand\tv{\myfont{t}}          %
\newcommand\nv{\myfont{n}} %
\newcommand\uu{\myfont{u}} %
\newcommand\vv{\myfont{v}} %
\newcommand\xx{\myfont{x}} %
\newcommand\yy{\myfont{y}} %

\newcommand\qint[2]{[#1 \,\nv+#2]} %
\newcommand\coeffs{\index{coefficients}\kk_{\pp,\qq}}
\newcommand\Ideal{\mathcal{I}}

\newcommand\dilute[1]{#1^{\square}}
\newcommand\dilutep[1]{#1^{\square'}}

\makeindex

\allowdisplaybreaks[1]

\expandafter\def\csname e8diagram:1\endcsname{%
\begin{tikzpicture}[execute at begin picture={\clip (-1.24,-2.736) rectangle ++(2.48,3.472);},x={(1.7363cm,0cm)},y={(0cm,-1.7363cm)},baseline={([yshift=-\the\dimexpr\fontdimen22\textfont2\relax]current  bounding  box.center)},line join=round]
\begin{scope}
\begin{scope}
\path[clip] (-1,.4) -- (-1,-2.4) -- (1,-2.4) -- (1,.4) -- cycle;
\path[fill={rgb,255:red,204;green,204;blue,255},draw=none] (-1,.4) -- (-1,-2.4) -- (1,-2.4) -- (1,.4) -- cycle;
\begin{scope}[every path/.append style={draw=red,line width=.05619cm}]
\path svg[xscale=1.7363cm,yscale=-1.7363cm] {M -1 0 Q 0 0 1 0};
\path svg[xscale=1.7363cm,yscale=-1.7363cm] {M -1 -1 Q 0 -1 1 -1};
\path svg[xscale=1.7363cm,yscale=-1.7363cm] {M -1 -2 Q 0 -2 1 -2};
\end{scope}
\path[draw=blue,line width=.01405cm] (-1,.4) -- (-1,-2.4) -- (1,-2.4) -- (1,.4) -- cycle;
\end{scope}
\end{scope}
\end{tikzpicture}
}
\expandafter\def\csname e8diagram:A3notword_0\endcsname{%
\begin{tikzpicture}[execute at begin picture={\clip (-1.24,-2.736) rectangle ++(2.48,3.472);},x={(1.7363cm,0cm)},y={(0cm,-1.7363cm)},baseline={([yshift=-\the\dimexpr\fontdimen22\textfont2\relax]current  bounding  box.center)},line join=round]
\begin{scope}
\begin{scope}
\path[clip] (-1,.4) -- (-1,-2.4) -- (1,-2.4) -- (1,.4) -- cycle;
\path[fill={rgb,255:red,204;green,204;blue,255},draw=none] (-1,.4) -- (-1,-2.4) -- (1,-2.4) -- (1,.4) -- cycle;
\begin{scope}[every path/.append style={draw=red,line width=.05619cm}]
\path svg[xscale=1.7363cm,yscale=-1.7363cm] {M -1 0 Q -.58147 -.5 -.61926 -1};
\path svg[xscale=1.7363cm,yscale=-1.7363cm] {M 1 0 Q .58147 -.5 .61926 -1};
\path svg[xscale=1.7363cm,yscale=-1.7363cm] {M 1 -1 Q .80963 -1 .61926 -1};
\path svg[xscale=1.7363cm,yscale=-1.7363cm] {M 1 -2 Q .58147 -1.5 .61926 -1};
\path svg[xscale=1.7363cm,yscale=-1.7363cm] {M -1 -2 Q -.58147 -1.5 -.61926 -1};
\path svg[xscale=1.7363cm,yscale=-1.7363cm] {M -1 -1 Q -.80963 -1 -.61926 -1};
\end{scope}
\path[draw=blue,line width=.01405cm] (-1,.4) -- (-1,-2.4) -- (1,-2.4) -- (1,.4) -- cycle;
\end{scope}
\end{scope}
\end{tikzpicture}
}
\expandafter\def\csname e8diagram:A3notword_1\endcsname{%
\begin{tikzpicture}[execute at begin picture={\clip (-1.24,-2.736) rectangle ++(2.48,3.472);},x={(1.7363cm,0cm)},y={(0cm,-1.7363cm)},baseline={([yshift=-\the\dimexpr\fontdimen22\textfont2\relax]current  bounding  box.center)},line join=round]
\begin{scope}
\begin{scope}
\path[clip] (-1,.4) -- (-1,-2.4) -- (1,-2.4) -- (1,.4) -- cycle;
\path[fill={rgb,255:red,204;green,204;blue,255},draw=none] (-1,.4) -- (-1,-2.4) -- (1,-2.4) -- (1,.4) -- cycle;
\begin{scope}[every path/.append style={draw=red,line width=.05619cm}]
\path svg[xscale=1.7363cm,yscale=-1.7363cm] {M 0 -.99394 Q .2759 -1.1854 .57894 -1.3769};
\path svg[xscale=1.7363cm,yscale=-1.7363cm] {M -1 0 Q -.33953 -.13133 0 -.26266};
\path svg[xscale=1.7363cm,yscale=-1.7363cm] {M -.57894 -1.3769 Q -.2759 -1.1854 0 -.99394};
\path svg[xscale=1.7363cm,yscale=-1.7363cm] {M 1 -1 Q .78353 -1.1885 .57894 -1.3769};
\path svg[xscale=1.7363cm,yscale=-1.7363cm] {M -1 -1 Q -.78353 -1.1885 -.57894 -1.3769};
\path svg[xscale=1.7363cm,yscale=-1.7363cm] {M 1 0 Q .33953 -.13133 0 -.26266};
\path svg[xscale=1.7363cm,yscale=-1.7363cm] {M -1 -2 Q -.61576 -1.6885 -.57894 -1.3769};
\path svg[xscale=1.7363cm,yscale=-1.7363cm] {M 1 -2 Q .61576 -1.6885 .57894 -1.3769};
\path svg[xscale=1.7363cm,yscale=-1.7363cm] {M 0 -.99394 Q 0 -.6283 0 -.26266};
\end{scope}
\path[draw=blue,line width=.01405cm] (-1,.4) -- (-1,-2.4) -- (1,-2.4) -- (1,.4) -- cycle;
\end{scope}
\end{scope}
\end{tikzpicture}
}
\expandafter\def\csname e8diagram:A3notword_2\endcsname{%
\begin{tikzpicture}[execute at begin picture={\clip (-1.24,-2.736) rectangle ++(2.48,3.472);},x={(1.7363cm,0cm)},y={(0cm,-1.7363cm)},baseline={([yshift=-\the\dimexpr\fontdimen22\textfont2\relax]current  bounding  box.center)},line join=round]
\begin{scope}
\begin{scope}
\path[clip] (-1,.4) -- (-1,-2.4) -- (1,-2.4) -- (1,.4) -- cycle;
\path[fill={rgb,255:red,204;green,204;blue,255},draw=none] (-1,.4) -- (-1,-2.4) -- (1,-2.4) -- (1,.4) -- cycle;
\begin{scope}[every path/.append style={draw=red,line width=.05619cm}]
\path svg[xscale=1.7363cm,yscale=-1.7363cm] {M 0 -1.0061 Q .2759 -.81457 .57894 -.62308};
\path svg[xscale=1.7363cm,yscale=-1.7363cm] {M -1 0 Q -.61576 -.31154 -.57894 -.62308};
\path svg[xscale=1.7363cm,yscale=-1.7363cm] {M 0 -1.7373 Q 0 -1.3717 0 -1.0061};
\path svg[xscale=1.7363cm,yscale=-1.7363cm] {M 1 -1 Q .78353 -.81154 .57894 -.62308};
\path svg[xscale=1.7363cm,yscale=-1.7363cm] {M -1 -1 Q -.78353 -.81154 -.57894 -.62308};
\path svg[xscale=1.7363cm,yscale=-1.7363cm] {M 1 -2 Q .33953 -1.8687 0 -1.7373};
\path svg[xscale=1.7363cm,yscale=-1.7363cm] {M 1 0 Q .61576 -.31154 .57894 -.62308};
\path svg[xscale=1.7363cm,yscale=-1.7363cm] {M -1 -2 Q -.33953 -1.8687 0 -1.7373};
\path svg[xscale=1.7363cm,yscale=-1.7363cm] {M 0 -1.0061 Q -.2759 -.81457 -.57894 -.62308};
\end{scope}
\path[draw=blue,line width=.01405cm] (-1,.4) -- (-1,-2.4) -- (1,-2.4) -- (1,.4) -- cycle;
\end{scope}
\end{scope}
\end{tikzpicture}
}
\expandafter\def\csname e8diagram:A3notword_3\endcsname{%
\begin{tikzpicture}[execute at begin picture={\clip (-1.24,-2.736) rectangle ++(2.48,3.472);},x={(1.7363cm,0cm)},y={(0cm,-1.7363cm)},baseline={([yshift=-\the\dimexpr\fontdimen22\textfont2\relax]current  bounding  box.center)},line join=round]
\begin{scope}
\begin{scope}
\path[clip] (-1,.4) -- (-1,-2.4) -- (1,-2.4) -- (1,.4) -- cycle;
\path[fill={rgb,255:red,204;green,204;blue,255},draw=none] (-1,.4) -- (-1,-2.4) -- (1,-2.4) -- (1,.4) -- cycle;
\begin{scope}[every path/.append style={draw=red,line width=.05619cm}]
\path svg[xscale=1.7363cm,yscale=-1.7363cm] {M -1 -1 Q -.70509 -.67416 -.4306 -.34831};
\path svg[xscale=1.7363cm,yscale=-1.7363cm] {M -1 0 Q -.57791 -.17416 -.4306 -.34831};
\path svg[xscale=1.7363cm,yscale=-1.7363cm] {M 1 -1 Q .70509 -1.3258 .4306 -1.6517};
\path svg[xscale=1.7363cm,yscale=-1.7363cm] {M 1 0 Q .05276 -.17416 -.4306 -.34831};
\path svg[xscale=1.7363cm,yscale=-1.7363cm] {M 1 -2 Q .57791 -1.8258 .4306 -1.6517};
\path svg[xscale=1.7363cm,yscale=-1.7363cm] {M -1 -2 Q -.05276 -1.8258 .4306 -1.6517};
\end{scope}
\path[draw=blue,line width=.01405cm] (-1,.4) -- (-1,-2.4) -- (1,-2.4) -- (1,.4) -- cycle;
\end{scope}
\end{scope}
\end{tikzpicture}
}
\expandafter\def\csname e8diagram:A3notword_4\endcsname{%
\begin{tikzpicture}[execute at begin picture={\clip (-1.24,-2.736) rectangle ++(2.48,3.472);},x={(1.7363cm,0cm)},y={(0cm,-1.7363cm)},baseline={([yshift=-\the\dimexpr\fontdimen22\textfont2\relax]current  bounding  box.center)},line join=round]
\begin{scope}
\begin{scope}
\path[clip] (-1,.4) -- (-1,-2.4) -- (1,-2.4) -- (1,.4) -- cycle;
\path[fill={rgb,255:red,204;green,204;blue,255},draw=none] (-1,.4) -- (-1,-2.4) -- (1,-2.4) -- (1,.4) -- cycle;
\begin{scope}[every path/.append style={draw=red,line width=.05619cm}]
\path svg[xscale=1.7363cm,yscale=-1.7363cm] {M -1 0 Q -.05276 -.17416 .4306 -.34831};
\path svg[xscale=1.7363cm,yscale=-1.7363cm] {M 1 0 Q .57791 -.17416 .4306 -.34831};
\path svg[xscale=1.7363cm,yscale=-1.7363cm] {M 1 -2 Q .05276 -1.8258 -.4306 -1.6517};
\path svg[xscale=1.7363cm,yscale=-1.7363cm] {M 1 -1 Q .70509 -.67416 .4306 -.34831};
\path svg[xscale=1.7363cm,yscale=-1.7363cm] {M -1 -2 Q -.57791 -1.8258 -.4306 -1.6517};
\path svg[xscale=1.7363cm,yscale=-1.7363cm] {M -1 -1 Q -.70509 -1.3258 -.4306 -1.6517};
\end{scope}
\path[draw=blue,line width=.01405cm] (-1,.4) -- (-1,-2.4) -- (1,-2.4) -- (1,.4) -- cycle;
\end{scope}
\end{scope}
\end{tikzpicture}
}
\expandafter\def\csname e8diagram:A3notword_5\endcsname{%
\begin{tikzpicture}[execute at begin picture={\clip (-1.24,-2.736) rectangle ++(2.48,3.472);},x={(1.7363cm,0cm)},y={(0cm,-1.7363cm)},baseline={([yshift=-\the\dimexpr\fontdimen22\textfont2\relax]current  bounding  box.center)},line join=round]
\begin{scope}
\begin{scope}
\path[clip] (-1,.4) -- (-1,-2.4) -- (1,-2.4) -- (1,.4) -- cycle;
\path[fill={rgb,255:red,204;green,204;blue,255},draw=none] (-1,.4) -- (-1,-2.4) -- (1,-2.4) -- (1,.4) -- cycle;
\begin{scope}[every path/.append style={draw=red,line width=.05619cm}]
\path svg[xscale=1.7363cm,yscale=-1.7363cm] {M -.25293 -1.1595 Q -.02003 -1.3122 .23374 -1.4648};
\path svg[xscale=1.7363cm,yscale=-1.7363cm] {M -1 0 Q -.42114 -.26761 -.23374 -.53522};
\path svg[xscale=1.7363cm,yscale=-1.7363cm] {M .23374 -1.4648 Q .23331 -1.1526 .25293 -.84046};
\path svg[xscale=1.7363cm,yscale=-1.7363cm] {M -.73118 -1.4601 Q -.47649 -1.3098 -.25293 -1.1595};
\path svg[xscale=1.7363cm,yscale=-1.7363cm] {M -1 -1 Q -.85089 -1.23 -.73118 -1.4601};
\path svg[xscale=1.7363cm,yscale=-1.7363cm] {M 1 -1 Q .85089 -.76996 .73118 -.53991};
\path svg[xscale=1.7363cm,yscale=-1.7363cm] {M -.23374 -.53522 Q .02003 -.68784 .25293 -.84046};
\path svg[xscale=1.7363cm,yscale=-1.7363cm] {M .25293 -.84046 Q .47649 -.69018 .73118 -.53991};
\path svg[xscale=1.7363cm,yscale=-1.7363cm] {M -1 -2 Q -.72241 -1.73 -.73118 -1.4601};
\path svg[xscale=1.7363cm,yscale=-1.7363cm] {M 1 0 Q .72241 -.26996 .73118 -.53991};
\path svg[xscale=1.7363cm,yscale=-1.7363cm] {M 1 -2 Q .42114 -1.7324 .23374 -1.4648};
\path svg[xscale=1.7363cm,yscale=-1.7363cm] {M -.25293 -1.1595 Q -.23331 -.84738 -.23374 -.53522};
\end{scope}
\path[draw=blue,line width=.01405cm] (-1,.4) -- (-1,-2.4) -- (1,-2.4) -- (1,.4) -- cycle;
\end{scope}
\end{scope}
\end{tikzpicture}
}
\expandafter\def\csname e8diagram:A3notword_6\endcsname{%
\begin{tikzpicture}[execute at begin picture={\clip (-1.24,-2.736) rectangle ++(2.48,3.472);},x={(1.7363cm,0cm)},y={(0cm,-1.7363cm)},baseline={([yshift=-\the\dimexpr\fontdimen22\textfont2\relax]current  bounding  box.center)},line join=round]
\begin{scope}
\begin{scope}
\path[clip] (-1,.4) -- (-1,-2.4) -- (1,-2.4) -- (1,.4) -- cycle;
\path[fill={rgb,255:red,204;green,204;blue,255},draw=none] (-1,.4) -- (-1,-2.4) -- (1,-2.4) -- (1,.4) -- cycle;
\begin{scope}[every path/.append style={draw=red,line width=.05619cm}]
\path svg[xscale=1.7363cm,yscale=-1.7363cm] {M -.25293 -.84046 Q -.47649 -.69018 -.73118 -.53991};
\path svg[xscale=1.7363cm,yscale=-1.7363cm] {M 1 -2 Q .72241 -1.73 .73118 -1.4601};
\path svg[xscale=1.7363cm,yscale=-1.7363cm] {M -.25293 -.84046 Q -.02003 -.68784 .23374 -.53522};
\path svg[xscale=1.7363cm,yscale=-1.7363cm] {M -1 0 Q -.72241 -.26996 -.73118 -.53991};
\path svg[xscale=1.7363cm,yscale=-1.7363cm] {M -.23374 -1.4648 Q -.23331 -1.1526 -.25293 -.84046};
\path svg[xscale=1.7363cm,yscale=-1.7363cm] {M -.23374 -1.4648 Q .02003 -1.3122 .25293 -1.1595};
\path svg[xscale=1.7363cm,yscale=-1.7363cm] {M 1 -1 Q .85089 -1.23 .73118 -1.4601};
\path svg[xscale=1.7363cm,yscale=-1.7363cm] {M .25293 -1.1595 Q .23331 -.84738 .23374 -.53522};
\path svg[xscale=1.7363cm,yscale=-1.7363cm] {M -1 -1 Q -.85089 -.76996 -.73118 -.53991};
\path svg[xscale=1.7363cm,yscale=-1.7363cm] {M .25293 -1.1595 Q .47649 -1.3098 .73118 -1.4601};
\path svg[xscale=1.7363cm,yscale=-1.7363cm] {M 1 0 Q .42114 -.26761 .23374 -.53522};
\path svg[xscale=1.7363cm,yscale=-1.7363cm] {M -1 -2 Q -.42114 -1.7324 -.23374 -1.4648};
\end{scope}
\path[draw=blue,line width=.01405cm] (-1,.4) -- (-1,-2.4) -- (1,-2.4) -- (1,.4) -- cycle;
\end{scope}
\end{scope}
\end{tikzpicture}
}
\expandafter\def\csname e8diagram:A3notword_7\endcsname{%
\begin{tikzpicture}[execute at begin picture={\clip (-1.24,-2.736) rectangle ++(2.48,3.472);},x={(1.7363cm,0cm)},y={(0cm,-1.7363cm)},baseline={([yshift=-\the\dimexpr\fontdimen22\textfont2\relax]current  bounding  box.center)},line join=round]
\begin{scope}
\begin{scope}
\path[clip] (-1,.4) -- (-1,-2.4) -- (1,-2.4) -- (1,.4) -- cycle;
\path[fill={rgb,255:red,204;green,204;blue,255},draw=none] (-1,.4) -- (-1,-2.4) -- (1,-2.4) -- (1,.4) -- cycle;
\begin{scope}[every path/.append style={draw=red,line width=.05619cm}]
\path svg[xscale=1.7363cm,yscale=-1.7363cm] {M 0 -1.3323 Q .23823 -1.1661 .471 -1};
\path svg[xscale=1.7363cm,yscale=-1.7363cm] {M .471 -1 Q .23823 -.83386 0 -.66771};
\path svg[xscale=1.7363cm,yscale=-1.7363cm] {M 0 -1.8928 Q 0 -1.6125 0 -1.3323};
\path svg[xscale=1.7363cm,yscale=-1.7363cm] {M -1 0 Q -.40867 -.05361 0 -.10722};
\path svg[xscale=1.7363cm,yscale=-1.7363cm] {M 1 -1 Q .7355 -1 .471 -1};
\path svg[xscale=1.7363cm,yscale=-1.7363cm] {M -1 -1 Q -.7355 -1 -.471 -1};
\path svg[xscale=1.7363cm,yscale=-1.7363cm] {M -.471 -1 Q -.23823 -.83386 0 -.66771};
\path svg[xscale=1.7363cm,yscale=-1.7363cm] {M 1 -2 Q .40867 -1.9464 0 -1.8928};
\path svg[xscale=1.7363cm,yscale=-1.7363cm] {M 0 -.66771 Q 0 -.38747 0 -.10722};
\path svg[xscale=1.7363cm,yscale=-1.7363cm] {M -1 -2 Q -.40867 -1.9464 0 -1.8928};
\path svg[xscale=1.7363cm,yscale=-1.7363cm] {M 1 0 Q .40867 -.05361 0 -.10722};
\path svg[xscale=1.7363cm,yscale=-1.7363cm] {M 0 -1.3323 Q -.23823 -1.1661 -.471 -1};
\end{scope}
\path[draw=blue,line width=.01405cm] (-1,.4) -- (-1,-2.4) -- (1,-2.4) -- (1,.4) -- cycle;
\end{scope}
\end{scope}
\end{tikzpicture}
}
\expandafter\def\csname e8diagram:E1\endcsname{%
\begin{tikzpicture}[execute at begin picture={\clip (-1.24,-2.736) rectangle ++(2.48,3.472);},x={(1.7363cm,0cm)},y={(0cm,-1.7363cm)},baseline={([yshift=-\the\dimexpr\fontdimen22\textfont2\relax]current  bounding  box.center)},line join=round]
\begin{scope}
\begin{scope}
\path[clip] (-1,.4) -- (-1,-2.4) -- (1,-2.4) -- (1,.4) -- cycle;
\path[fill={rgb,255:red,204;green,204;blue,255},draw=none] (-1,.4) -- (-1,-2.4) -- (1,-2.4) -- (1,.4) -- cycle;
\begin{scope}[every path/.append style={draw=red,line width=.05619cm}]
\path svg[xscale=1.7363cm,yscale=-1.7363cm] {M 1 -1 Q .29519 -1.5 1 -2};
\path svg[xscale=1.7363cm,yscale=-1.7363cm] {M -1 0 Q 0 0 1 0};
\path svg[xscale=1.7363cm,yscale=-1.7363cm] {M -1 -2 Q -.29519 -1.5 -1 -1};
\end{scope}
\path[draw=blue,line width=.01405cm] (-1,.4) -- (-1,-2.4) -- (1,-2.4) -- (1,.4) -- cycle;
\end{scope}
\end{scope}
\end{tikzpicture}
}
\expandafter\def\csname e8diagram:E1E2\endcsname{%
\begin{tikzpicture}[execute at begin picture={\clip (-1.24,-2.736) rectangle ++(2.48,3.472);},x={(1.7363cm,0cm)},y={(0cm,-1.7363cm)},baseline={([yshift=-\the\dimexpr\fontdimen22\textfont2\relax]current  bounding  box.center)},line join=round]
\begin{scope}
\begin{scope}
\path[clip] (-1,.4) -- (-1,-2.4) -- (1,-2.4) -- (1,.4) -- cycle;
\path[fill={rgb,255:red,204;green,204;blue,255},draw=none] (-1,.4) -- (-1,-2.4) -- (1,-2.4) -- (1,.4) -- cycle;
\begin{scope}[every path/.append style={draw=red,line width=.05619cm}]
\path svg[xscale=1.7363cm,yscale=-1.7363cm] {M -1 -2 Q -.29519 -1.5 -1 -1};
\path svg[xscale=1.7363cm,yscale=-1.7363cm] {M 1 0 Q .29519 -.5 1 -1};
\path svg[xscale=1.7363cm,yscale=-1.7363cm] {M -1 0 Q 0 -1 1 -2};
\end{scope}
\path[draw=blue,line width=.01405cm] (-1,.4) -- (-1,-2.4) -- (1,-2.4) -- (1,.4) -- cycle;
\end{scope}
\end{scope}
\end{tikzpicture}
}
\expandafter\def\csname e8diagram:E1H2\endcsname{%
\begin{tikzpicture}[execute at begin picture={\clip (-1.24,-2.736) rectangle ++(2.48,3.472);},x={(1.7363cm,0cm)},y={(0cm,-1.7363cm)},baseline={([yshift=-\the\dimexpr\fontdimen22\textfont2\relax]current  bounding  box.center)},line join=round]
\begin{scope}
\begin{scope}
\path[clip] (-1,.4) -- (-1,-2.4) -- (1,-2.4) -- (1,.4) -- cycle;
\path[fill={rgb,255:red,204;green,204;blue,255},draw=none] (-1,.4) -- (-1,-2.4) -- (1,-2.4) -- (1,.4) -- cycle;
\begin{scope}[every path/.append style={draw=red,line width=.05619cm}]
\path svg[xscale=1.7363cm,yscale=-1.7363cm] {M -1 0 Q -.19023 -.18082 .19976 -.36165};
\path svg[xscale=1.7363cm,yscale=-1.7363cm] {M -1 -2 Q -.29519 -1.5 -1 -1};
\path svg[xscale=1.7363cm,yscale=-1.7363cm] {M 1 0 Q .43399 -.18082 .19976 -.36165};
\path svg[xscale=1.7363cm,yscale=-1.7363cm] {M 1 -1 Q .77384 -1.0783 .54935 -1.1565};
\path svg[xscale=1.7363cm,yscale=-1.7363cm] {M 1 -2 Q .56447 -1.5783 .54935 -1.1565};
\path svg[xscale=1.7363cm,yscale=-1.7363cm] {M .19976 -.36165 Q .36201 -.75908 .54935 -1.1565};
\end{scope}
\path[draw=blue,line width=.01405cm] (-1,.4) -- (-1,-2.4) -- (1,-2.4) -- (1,.4) -- cycle;
\end{scope}
\end{scope}
\end{tikzpicture}
}
\expandafter\def\csname e8diagram:E1K2\endcsname{%
\begin{tikzpicture}[execute at begin picture={\clip (-1.24,-2.736) rectangle ++(2.48,3.472);},x={(1.7363cm,0cm)},y={(0cm,-1.7363cm)},baseline={([yshift=-\the\dimexpr\fontdimen22\textfont2\relax]current  bounding  box.center)},line join=round]
\begin{scope}
\begin{scope}
\path[clip] (-1,.4) -- (-1,-2.4) -- (1,-2.4) -- (1,.4) -- cycle;
\path[fill={rgb,255:red,204;green,204;blue,255},draw=none] (-1,.4) -- (-1,-2.4) -- (1,-2.4) -- (1,.4) -- cycle;
\begin{scope}[every path/.append style={draw=red,line width=.05619cm}]
\path svg[xscale=1.7363cm,yscale=-1.7363cm] {M -1 0 Q -.15987 -.45223 .13985 -.90447};
\path svg[xscale=1.7363cm,yscale=-1.7363cm] {M -1 -2 Q -.29519 -1.5 -1 -1};
\path svg[xscale=1.7363cm,yscale=-1.7363cm] {M 1 0 Q .65382 -.2951 .63282 -.59021};
\path svg[xscale=1.7363cm,yscale=-1.7363cm] {M 1 -2 Q .30492 -1.4522 .13985 -.90447};
\path svg[xscale=1.7363cm,yscale=-1.7363cm] {M 1 -1 Q .80777 -.7951 .63282 -.59021};
\path svg[xscale=1.7363cm,yscale=-1.7363cm] {M .13985 -.90447 Q .37327 -.74734 .63282 -.59021};
\end{scope}
\path[draw=blue,line width=.01405cm] (-1,.4) -- (-1,-2.4) -- (1,-2.4) -- (1,.4) -- cycle;
\end{scope}
\end{scope}
\end{tikzpicture}
}
\expandafter\def\csname e8diagram:E1K2H1\endcsname{%
\begin{tikzpicture}[execute at begin picture={\clip (-1.24,-2.736) rectangle ++(2.48,3.472);},x={(1.7363cm,0cm)},y={(0cm,-1.7363cm)},baseline={([yshift=-\the\dimexpr\fontdimen22\textfont2\relax]current  bounding  box.center)},line join=round]
\begin{scope}
\begin{scope}
\path[clip] (-1,.4) -- (-1,-2.4) -- (1,-2.4) -- (1,.4) -- cycle;
\path[fill={rgb,255:red,204;green,204;blue,255},draw=none] (-1,.4) -- (-1,-2.4) -- (1,-2.4) -- (1,.4) -- cycle;
\begin{scope}[every path/.append style={draw=red,line width=.05619cm}]
\path svg[xscale=1.7363cm,yscale=-1.7363cm] {M .5187 -.37987 Q .5847 -.62607 .67592 -.87227};
\path svg[xscale=1.7363cm,yscale=-1.7363cm] {M -1 0 Q -.29096 -.25246 -.01153 -.50492};
\path svg[xscale=1.7363cm,yscale=-1.7363cm] {M 1 -1 Q .83708 -.93614 .67592 -.87227};
\path svg[xscale=1.7363cm,yscale=-1.7363cm] {M -.01153 -.50492 Q .24253 -.4424 .5187 -.37987};
\path svg[xscale=1.7363cm,yscale=-1.7363cm] {M -.01153 -.50492 Q .16766 -.87989 .35323 -1.2549};
\path svg[xscale=1.7363cm,yscale=-1.7363cm] {M .35323 -1.2549 Q .50833 -1.0636 .67592 -.87227};
\path svg[xscale=1.7363cm,yscale=-1.7363cm] {M -1 -2 Q -.29519 -1.5 -1 -1};
\path svg[xscale=1.7363cm,yscale=-1.7363cm] {M 1 0 Q .62453 -.18994 .5187 -.37987};
\path svg[xscale=1.7363cm,yscale=-1.7363cm] {M 1 -2 Q .46277 -1.6274 .35323 -1.2549};
\end{scope}
\path[draw=blue,line width=.01405cm] (-1,.4) -- (-1,-2.4) -- (1,-2.4) -- (1,.4) -- cycle;
\end{scope}
\end{scope}
\end{tikzpicture}
}
\expandafter\def\csname e8diagram:E2\endcsname{%
\begin{tikzpicture}[execute at begin picture={\clip (-1.24,-2.736) rectangle ++(2.48,3.472);},x={(1.7363cm,0cm)},y={(0cm,-1.7363cm)},baseline={([yshift=-\the\dimexpr\fontdimen22\textfont2\relax]current  bounding  box.center)},line join=round]
\begin{scope}
\begin{scope}
\path[clip] (-1,.4) -- (-1,-2.4) -- (1,-2.4) -- (1,.4) -- cycle;
\path[fill={rgb,255:red,204;green,204;blue,255},draw=none] (-1,.4) -- (-1,-2.4) -- (1,-2.4) -- (1,.4) -- cycle;
\begin{scope}[every path/.append style={draw=red,line width=.05619cm}]
\path svg[xscale=1.7363cm,yscale=-1.7363cm] {M -1 -2 Q 0 -2 1 -2};
\path svg[xscale=1.7363cm,yscale=-1.7363cm] {M 1 0 Q .29519 -.5 1 -1};
\path svg[xscale=1.7363cm,yscale=-1.7363cm] {M -1 -1 Q -.29519 -.5 -1 0};
\end{scope}
\path[draw=blue,line width=.01405cm] (-1,.4) -- (-1,-2.4) -- (1,-2.4) -- (1,.4) -- cycle;
\end{scope}
\end{scope}
\end{tikzpicture}
}
\expandafter\def\csname e8diagram:E2E1\endcsname{%
\begin{tikzpicture}[execute at begin picture={\clip (-1.24,-2.736) rectangle ++(2.48,3.472);},x={(1.7363cm,0cm)},y={(0cm,-1.7363cm)},baseline={([yshift=-\the\dimexpr\fontdimen22\textfont2\relax]current  bounding  box.center)},line join=round]
\begin{scope}
\begin{scope}
\path[clip] (-1,.4) -- (-1,-2.4) -- (1,-2.4) -- (1,.4) -- cycle;
\path[fill={rgb,255:red,204;green,204;blue,255},draw=none] (-1,.4) -- (-1,-2.4) -- (1,-2.4) -- (1,.4) -- cycle;
\begin{scope}[every path/.append style={draw=red,line width=.05619cm}]
\path svg[xscale=1.7363cm,yscale=-1.7363cm] {M 1 -1 Q .29519 -1.5 1 -2};
\path svg[xscale=1.7363cm,yscale=-1.7363cm] {M -1 -1 Q -.29519 -.5 -1 0};
\path svg[xscale=1.7363cm,yscale=-1.7363cm] {M -1 -2 Q 0 -1 1 0};
\end{scope}
\path[draw=blue,line width=.01405cm] (-1,.4) -- (-1,-2.4) -- (1,-2.4) -- (1,.4) -- cycle;
\end{scope}
\end{scope}
\end{tikzpicture}
}
\expandafter\def\csname e8diagram:E2H1\endcsname{%
\begin{tikzpicture}[execute at begin picture={\clip (-1.24,-2.736) rectangle ++(2.48,3.472);},x={(1.7363cm,0cm)},y={(0cm,-1.7363cm)},baseline={([yshift=-\the\dimexpr\fontdimen22\textfont2\relax]current  bounding  box.center)},line join=round]
\begin{scope}
\begin{scope}
\path[clip] (-1,.4) -- (-1,-2.4) -- (1,-2.4) -- (1,.4) -- cycle;
\path[fill={rgb,255:red,204;green,204;blue,255},draw=none] (-1,.4) -- (-1,-2.4) -- (1,-2.4) -- (1,.4) -- cycle;
\begin{scope}[every path/.append style={draw=red,line width=.05619cm}]
\path svg[xscale=1.7363cm,yscale=-1.7363cm] {M 1 0 Q .56447 -.42175 .54935 -.84349};
\path svg[xscale=1.7363cm,yscale=-1.7363cm] {M 1 -2 Q .43399 -1.8192 .19976 -1.6384};
\path svg[xscale=1.7363cm,yscale=-1.7363cm] {M 1 -1 Q .77384 -.92175 .54935 -.84349};
\path svg[xscale=1.7363cm,yscale=-1.7363cm] {M -1 -1 Q -.29519 -.5 -1 0};
\path svg[xscale=1.7363cm,yscale=-1.7363cm] {M .19976 -1.6384 Q .36201 -1.2409 .54935 -.84349};
\path svg[xscale=1.7363cm,yscale=-1.7363cm] {M -1 -2 Q -.19023 -1.8192 .19976 -1.6384};
\end{scope}
\path[draw=blue,line width=.01405cm] (-1,.4) -- (-1,-2.4) -- (1,-2.4) -- (1,.4) -- cycle;
\end{scope}
\end{scope}
\end{tikzpicture}
}
\expandafter\def\csname e8diagram:E2K1\endcsname{%
\begin{tikzpicture}[execute at begin picture={\clip (-1.24,-2.736) rectangle ++(2.48,3.472);},x={(1.7363cm,0cm)},y={(0cm,-1.7363cm)},baseline={([yshift=-\the\dimexpr\fontdimen22\textfont2\relax]current  bounding  box.center)},line join=round]
\begin{scope}
\begin{scope}
\path[clip] (-1,.4) -- (-1,-2.4) -- (1,-2.4) -- (1,.4) -- cycle;
\path[fill={rgb,255:red,204;green,204;blue,255},draw=none] (-1,.4) -- (-1,-2.4) -- (1,-2.4) -- (1,.4) -- cycle;
\begin{scope}[every path/.append style={draw=red,line width=.05619cm}]
\path svg[xscale=1.7363cm,yscale=-1.7363cm] {M 1 0 Q .30492 -.54777 .13985 -1.0955};
\path svg[xscale=1.7363cm,yscale=-1.7363cm] {M 1 -1 Q .80777 -1.2049 .63282 -1.4098};
\path svg[xscale=1.7363cm,yscale=-1.7363cm] {M -1 -1 Q -.29519 -.5 -1 0};
\path svg[xscale=1.7363cm,yscale=-1.7363cm] {M 1 -2 Q .65382 -1.7049 .63282 -1.4098};
\path svg[xscale=1.7363cm,yscale=-1.7363cm] {M .13985 -1.0955 Q .37327 -1.2527 .63282 -1.4098};
\path svg[xscale=1.7363cm,yscale=-1.7363cm] {M -1 -2 Q -.15987 -1.5478 .13985 -1.0955};
\end{scope}
\path[draw=blue,line width=.01405cm] (-1,.4) -- (-1,-2.4) -- (1,-2.4) -- (1,.4) -- cycle;
\end{scope}
\end{scope}
\end{tikzpicture}
}
\expandafter\def\csname e8diagram:E2K1H2\endcsname{%
\begin{tikzpicture}[execute at begin picture={\clip (-1.24,-2.736) rectangle ++(2.48,3.472);},x={(1.7363cm,0cm)},y={(0cm,-1.7363cm)},baseline={([yshift=-\the\dimexpr\fontdimen22\textfont2\relax]current  bounding  box.center)},line join=round]
\begin{scope}
\begin{scope}
\path[clip] (-1,.4) -- (-1,-2.4) -- (1,-2.4) -- (1,.4) -- cycle;
\path[fill={rgb,255:red,204;green,204;blue,255},draw=none] (-1,.4) -- (-1,-2.4) -- (1,-2.4) -- (1,.4) -- cycle;
\begin{scope}[every path/.append style={draw=red,line width=.05619cm}]
\path svg[xscale=1.7363cm,yscale=-1.7363cm] {M .35323 -.74513 Q .50833 -.93643 .67592 -1.1277};
\path svg[xscale=1.7363cm,yscale=-1.7363cm] {M -1 -1 Q -.29519 -.5 -1 0};
\path svg[xscale=1.7363cm,yscale=-1.7363cm] {M 1 -1 Q .83708 -1.0639 .67592 -1.1277};
\path svg[xscale=1.7363cm,yscale=-1.7363cm] {M -.01153 -1.4951 Q .16766 -1.1201 .35323 -.74513};
\path svg[xscale=1.7363cm,yscale=-1.7363cm] {M -.01153 -1.4951 Q .24253 -1.5576 .5187 -1.6201};
\path svg[xscale=1.7363cm,yscale=-1.7363cm] {M .5187 -1.6201 Q .5847 -1.3739 .67592 -1.1277};
\path svg[xscale=1.7363cm,yscale=-1.7363cm] {M 1 0 Q .46277 -.37257 .35323 -.74513};
\path svg[xscale=1.7363cm,yscale=-1.7363cm] {M 1 -2 Q .62453 -1.8101 .5187 -1.6201};
\path svg[xscale=1.7363cm,yscale=-1.7363cm] {M -1 -2 Q -.29096 -1.7475 -.01153 -1.4951};
\end{scope}
\path[draw=blue,line width=.01405cm] (-1,.4) -- (-1,-2.4) -- (1,-2.4) -- (1,.4) -- cycle;
\end{scope}
\end{scope}
\end{tikzpicture}
}
\expandafter\def\csname e8diagram:H1\endcsname{%
\begin{tikzpicture}[execute at begin picture={\clip (-1.24,-2.736) rectangle ++(2.48,3.472);},x={(1.7363cm,0cm)},y={(0cm,-1.7363cm)},baseline={([yshift=-\the\dimexpr\fontdimen22\textfont2\relax]current  bounding  box.center)},line join=round]
\begin{scope}
\begin{scope}
\path[clip] (-1,.4) -- (-1,-2.4) -- (1,-2.4) -- (1,.4) -- cycle;
\path[fill={rgb,255:red,204;green,204;blue,255},draw=none] (-1,.4) -- (-1,-2.4) -- (1,-2.4) -- (1,.4) -- cycle;
\begin{scope}[every path/.append style={draw=red,line width=.05619cm}]
\path svg[xscale=1.7363cm,yscale=-1.7363cm] {M -1 -1 Q -.50204 -1.0991 0 -1.1983};
\path svg[xscale=1.7363cm,yscale=-1.7363cm] {M 1 -2 Q .36243 -1.9002 0 -1.8004};
\path svg[xscale=1.7363cm,yscale=-1.7363cm] {M 1 -1 Q .50204 -1.0991 0 -1.1983};
\path svg[xscale=1.7363cm,yscale=-1.7363cm] {M 0 -1.8004 Q 0 -1.4993 0 -1.1983};
\path svg[xscale=1.7363cm,yscale=-1.7363cm] {M -1 0 Q 0 0 1 0};
\path svg[xscale=1.7363cm,yscale=-1.7363cm] {M -1 -2 Q -.36243 -1.9002 0 -1.8004};
\end{scope}
\path[draw=blue,line width=.01405cm] (-1,.4) -- (-1,-2.4) -- (1,-2.4) -- (1,.4) -- cycle;
\end{scope}
\end{scope}
\end{tikzpicture}
}
\expandafter\def\csname e8diagram:H12\endcsname{%
\begin{tikzpicture}[execute at begin picture={\clip (-1.24,-2.736) rectangle ++(2.48,3.472);},x={(1.7363cm,0cm)},y={(0cm,-1.7363cm)},baseline={([yshift=-\the\dimexpr\fontdimen22\textfont2\relax]current  bounding  box.center)},line join=round]
\begin{scope}
\begin{scope}
\path[clip] (-1,.4) -- (-1,-2.4) -- (1,-2.4) -- (1,.4) -- cycle;
\path[fill={rgb,255:red,204;green,204;blue,255},draw=none] (-1,.4) -- (-1,-2.4) -- (1,-2.4) -- (1,.4) -- cycle;
\begin{scope}[every path/.append style={draw=red,line width=.05619cm}]
\path svg[xscale=1.7363cm,yscale=-1.7363cm] {M -.3967 -1.2453 Q 0 -1.2453 .3967 -1.2453};
\path svg[xscale=1.7363cm,yscale=-1.7363cm] {M 1 -1 Q .6977 -1.1227 .3967 -1.2453};
\path svg[xscale=1.7363cm,yscale=-1.7363cm] {M -.31926 -1.7384 Q -.34844 -1.4919 -.3967 -1.2453};
\path svg[xscale=1.7363cm,yscale=-1.7363cm] {M -.31926 -1.7384 Q 0 -1.7384 .31926 -1.7384};
\path svg[xscale=1.7363cm,yscale=-1.7363cm] {M -1 -1 Q -.6977 -1.1227 -.3967 -1.2453};
\path svg[xscale=1.7363cm,yscale=-1.7363cm] {M .31926 -1.7384 Q .34844 -1.4919 .3967 -1.2453};
\path svg[xscale=1.7363cm,yscale=-1.7363cm] {M -1 0 Q 0 0 1 0};
\path svg[xscale=1.7363cm,yscale=-1.7363cm] {M 1 -2 Q .53294 -1.8692 .31926 -1.7384};
\path svg[xscale=1.7363cm,yscale=-1.7363cm] {M -1 -2 Q -.53294 -1.8692 -.31926 -1.7384};
\end{scope}
\path[draw=blue,line width=.01405cm] (-1,.4) -- (-1,-2.4) -- (1,-2.4) -- (1,.4) -- cycle;
\end{scope}
\end{scope}
\end{tikzpicture}
}
\expandafter\def\csname e8diagram:H12E2\endcsname{%
\begin{tikzpicture}[execute at begin picture={\clip (-1.24,-2.736) rectangle ++(2.48,3.472);},x={(1.7363cm,0cm)},y={(0cm,-1.7363cm)},baseline={([yshift=-\the\dimexpr\fontdimen22\textfont2\relax]current  bounding  box.center)},line join=round]
\begin{scope}
\begin{scope}
\path[clip] (-1,.4) -- (-1,-2.4) -- (1,-2.4) -- (1,.4) -- cycle;
\path[fill={rgb,255:red,204;green,204;blue,255},draw=none] (-1,.4) -- (-1,-2.4) -- (1,-2.4) -- (1,.4) -- cycle;
\begin{scope}[every path/.append style={draw=red,line width=.05619cm}]
\path svg[xscale=1.7363cm,yscale=-1.7363cm] {M -.67592 -1.1277 Q -.50833 -.93643 -.35323 -.74513};
\path svg[xscale=1.7363cm,yscale=-1.7363cm] {M -1 0 Q -.46277 -.37257 -.35323 -.74513};
\path svg[xscale=1.7363cm,yscale=-1.7363cm] {M -.5187 -1.6201 Q -.5847 -1.3739 -.67592 -1.1277};
\path svg[xscale=1.7363cm,yscale=-1.7363cm] {M -.5187 -1.6201 Q -.24253 -1.5576 .01153 -1.4951};
\path svg[xscale=1.7363cm,yscale=-1.7363cm] {M 1 0 Q .29519 -.5 1 -1};
\path svg[xscale=1.7363cm,yscale=-1.7363cm] {M -1 -1 Q -.83708 -1.0639 -.67592 -1.1277};
\path svg[xscale=1.7363cm,yscale=-1.7363cm] {M .01153 -1.4951 Q -.16766 -1.1201 -.35323 -.74513};
\path svg[xscale=1.7363cm,yscale=-1.7363cm] {M 1 -2 Q .29096 -1.7475 .01153 -1.4951};
\path svg[xscale=1.7363cm,yscale=-1.7363cm] {M -1 -2 Q -.62453 -1.8101 -.5187 -1.6201};
\end{scope}
\path[draw=blue,line width=.01405cm] (-1,.4) -- (-1,-2.4) -- (1,-2.4) -- (1,.4) -- cycle;
\end{scope}
\end{scope}
\end{tikzpicture}
}
\expandafter\def\csname e8diagram:H12H2\endcsname{%
\begin{tikzpicture}[execute at begin picture={\clip (-1.24,-2.736) rectangle ++(2.48,3.472);},x={(1.7363cm,0cm)},y={(0cm,-1.7363cm)},baseline={([yshift=-\the\dimexpr\fontdimen22\textfont2\relax]current  bounding  box.center)},line join=round]
\begin{scope}
\begin{scope}
\path[clip] (-1,.4) -- (-1,-2.4) -- (1,-2.4) -- (1,.4) -- cycle;
\path[fill={rgb,255:red,204;green,204;blue,255},draw=none] (-1,.4) -- (-1,-2.4) -- (1,-2.4) -- (1,.4) -- cycle;
\begin{scope}[every path/.append style={draw=red,line width=.05619cm}]
\path svg[xscale=1.7363cm,yscale=-1.7363cm] {M -.57638 -1.2604 Q -.26603 -1.2326 .04069 -1.2049};
\path svg[xscale=1.7363cm,yscale=-1.7363cm] {M .04069 -1.2049 Q .25297 -1.0017 .47398 -.79856};
\path svg[xscale=1.7363cm,yscale=-1.7363cm] {M -.4751 -1.7578 Q -.5101 -1.5091 -.57638 -1.2604};
\path svg[xscale=1.7363cm,yscale=-1.7363cm] {M -1 0 Q -.26445 -.08944 .18013 -.17887};
\path svg[xscale=1.7363cm,yscale=-1.7363cm] {M -.4751 -1.7578 Q -.1395 -1.7572 .19562 -1.7566};
\path svg[xscale=1.7363cm,yscale=-1.7363cm] {M 1 -1 Q .73608 -.89928 .47398 -.79856};
\path svg[xscale=1.7363cm,yscale=-1.7363cm] {M -1 -1 Q -.78552 -1.1302 -.57638 -1.2604};
\path svg[xscale=1.7363cm,yscale=-1.7363cm] {M .19562 -1.7566 Q .10116 -1.4807 .04069 -1.2049};
\path svg[xscale=1.7363cm,yscale=-1.7363cm] {M .47398 -.79856 Q .33082 -.48871 .18013 -.17887};
\path svg[xscale=1.7363cm,yscale=-1.7363cm] {M 1 -2 Q .46391 -1.8783 .19562 -1.7566};
\path svg[xscale=1.7363cm,yscale=-1.7363cm] {M -1 -2 Q -.63275 -1.8789 -.4751 -1.7578};
\path svg[xscale=1.7363cm,yscale=-1.7363cm] {M 1 0 Q .47809 -.08944 .18013 -.17887};
\end{scope}
\path[draw=blue,line width=.01405cm] (-1,.4) -- (-1,-2.4) -- (1,-2.4) -- (1,.4) -- cycle;
\end{scope}
\end{scope}
\end{tikzpicture}
}
\expandafter\def\csname e8diagram:H12H22\endcsname{%
\begin{tikzpicture}[execute at begin picture={\clip (-1.24,-2.736) rectangle ++(2.48,3.472);},x={(1.7363cm,0cm)},y={(0cm,-1.7363cm)},baseline={([yshift=-\the\dimexpr\fontdimen22\textfont2\relax]current  bounding  box.center)},line join=round]
\begin{scope}
\begin{scope}
\path[clip] (-1,.4) -- (-1,-2.4) -- (1,-2.4) -- (1,.4) -- cycle;
\path[fill={rgb,255:red,204;green,204;blue,255},draw=none] (-1,.4) -- (-1,-2.4) -- (1,-2.4) -- (1,.4) -- cycle;
\begin{scope}[every path/.append style={draw=red,line width=.05619cm}]
\path svg[xscale=1.7363cm,yscale=-1.7363cm] {M -.11799 -1.2198 Q 0 -1 .11799 -.78024};
\path svg[xscale=1.7363cm,yscale=-1.7363cm] {M 1 -1 Q .83293 -.85509 .67519 -.71017};
\path svg[xscale=1.7363cm,yscale=-1.7363cm] {M .13625 -1.7745 Q -.01057 -1.4971 -.11799 -1.2198};
\path svg[xscale=1.7363cm,yscale=-1.7363cm] {M 1 0 Q .69593 -.10584 .56505 -.21167};
\path svg[xscale=1.7363cm,yscale=-1.7363cm] {M -.67519 -1.2898 Q -.39419 -1.2548 -.11799 -1.2198};
\path svg[xscale=1.7363cm,yscale=-1.7363cm] {M .67519 -.71017 Q .59965 -.46092 .56505 -.21167};
\path svg[xscale=1.7363cm,yscale=-1.7363cm] {M -1 0 Q -.43421 -.11277 -.13625 -.22554};
\path svg[xscale=1.7363cm,yscale=-1.7363cm] {M -.56505 -1.7883 Q -.59965 -1.5391 -.67519 -1.2898};
\path svg[xscale=1.7363cm,yscale=-1.7363cm] {M -.56505 -1.7883 Q -.21121 -1.7814 .13625 -1.7745};
\path svg[xscale=1.7363cm,yscale=-1.7363cm] {M -.13625 -.22554 Q .21121 -.21861 .56505 -.21167};
\path svg[xscale=1.7363cm,yscale=-1.7363cm] {M -1 -1 Q -.83293 -1.1449 -.67519 -1.2898};
\path svg[xscale=1.7363cm,yscale=-1.7363cm] {M 1 -2 Q .43421 -1.8872 .13625 -1.7745};
\path svg[xscale=1.7363cm,yscale=-1.7363cm] {M .11799 -.78024 Q .01057 -.50289 -.13625 -.22554};
\path svg[xscale=1.7363cm,yscale=-1.7363cm] {M .11799 -.78024 Q .39419 -.74521 .67519 -.71017};
\path svg[xscale=1.7363cm,yscale=-1.7363cm] {M -1 -2 Q -.69593 -1.8942 -.56505 -1.7883};
\end{scope}
\path[draw=blue,line width=.01405cm] (-1,.4) -- (-1,-2.4) -- (1,-2.4) -- (1,.4) -- cycle;
\end{scope}
\end{scope}
\end{tikzpicture}
}
\expandafter\def\csname e8diagram:H12H2H1\endcsname{%
\begin{tikzpicture}[execute at begin picture={\clip (-1.24,-2.736) rectangle ++(2.48,3.472);},x={(1.7363cm,0cm)},y={(0cm,-1.7363cm)},baseline={([yshift=-\the\dimexpr\fontdimen22\textfont2\relax]current  bounding  box.center)},line join=round]
\begin{scope}
\begin{scope}
\path[clip] (-1,.4) -- (-1,-2.4) -- (1,-2.4) -- (1,.4) -- cycle;
\path[fill={rgb,255:red,204;green,204;blue,255},draw=none] (-1,.4) -- (-1,-2.4) -- (1,-2.4) -- (1,.4) -- cycle;
\begin{scope}[every path/.append style={draw=red,line width=.05619cm}]
\path svg[xscale=1.7363cm,yscale=-1.7363cm] {M -.20203 -1.0627 Q -.01718 -.8911 .17323 -.7195};
\path svg[xscale=1.7363cm,yscale=-1.7363cm] {M 1 -1 Q .82298 -1.0493 .64692 -1.0987};
\path svg[xscale=1.7363cm,yscale=-1.7363cm] {M -.03925 -1.5368 Q -.12262 -1.2998 -.20203 -1.0627};
\path svg[xscale=1.7363cm,yscale=-1.7363cm] {M -.03925 -1.5368 Q .22216 -1.5935 .50626 -1.6502};
\path svg[xscale=1.7363cm,yscale=-1.7363cm] {M 1 0 Q .43694 -.06150 .065 -.12301};
\path svg[xscale=1.7363cm,yscale=-1.7363cm] {M -.71853 -1.2108 Q -.457 -1.1367 -.20203 -1.0627};
\path svg[xscale=1.7363cm,yscale=-1.7363cm] {M -.61923 -1.7315 Q -.6461 -1.4711 -.71853 -1.2108};
\path svg[xscale=1.7363cm,yscale=-1.7363cm] {M -.61923 -1.7315 Q -.30396 -1.6342 -.03925 -1.5368};
\path svg[xscale=1.7363cm,yscale=-1.7363cm] {M .17323 -.7195 Q .4073 -.90909 .64692 -1.0987};
\path svg[xscale=1.7363cm,yscale=-1.7363cm] {M -1 0 Q -.36172 -.06150 .065 -.12301};
\path svg[xscale=1.7363cm,yscale=-1.7363cm] {M .17323 -.7195 Q .12156 -.42125 .065 -.12301};
\path svg[xscale=1.7363cm,yscale=-1.7363cm] {M -1 -1 Q -.85655 -1.1054 -.71853 -1.2108};
\path svg[xscale=1.7363cm,yscale=-1.7363cm] {M .50626 -1.6502 Q .55976 -1.3744 .64692 -1.0987};
\path svg[xscale=1.7363cm,yscale=-1.7363cm] {M -1 -2 Q -.71249 -1.8658 -.61923 -1.7315};
\path svg[xscale=1.7363cm,yscale=-1.7363cm] {M 1 -2 Q .62369 -1.8251 .50626 -1.6502};
\end{scope}
\path[draw=blue,line width=.01405cm] (-1,.4) -- (-1,-2.4) -- (1,-2.4) -- (1,.4) -- cycle;
\end{scope}
\end{scope}
\end{tikzpicture}
}
\expandafter\def\csname e8diagram:H12H2K1\endcsname{%
\begin{tikzpicture}[execute at begin picture={\clip (-1.24,-2.736) rectangle ++(2.48,3.472);},x={(1.7363cm,0cm)},y={(0cm,-1.7363cm)},baseline={([yshift=-\the\dimexpr\fontdimen22\textfont2\relax]current  bounding  box.center)},line join=round]
\begin{scope}
\begin{scope}
\path[clip] (-1,.4) -- (-1,-2.4) -- (1,-2.4) -- (1,.4) -- cycle;
\path[fill={rgb,255:red,204;green,204;blue,255},draw=none] (-1,.4) -- (-1,-2.4) -- (1,-2.4) -- (1,.4) -- cycle;
\begin{scope}[every path/.append style={draw=red,line width=.05619cm}]
\path svg[xscale=1.7363cm,yscale=-1.7363cm] {M 1 -2 Q .72745 -1.7308 .73904 -1.4616};
\path svg[xscale=1.7363cm,yscale=-1.7363cm] {M -.39133 -.95307 Q -.2113 -.80095 -.02601 -.64883};
\path svg[xscale=1.7363cm,yscale=-1.7363cm] {M 1 -1 Q .85443 -1.2308 .73904 -1.4616};
\path svg[xscale=1.7363cm,yscale=-1.7363cm] {M -.16337 -1.4054 Q -.2751 -1.1792 -.39133 -.95307};
\path svg[xscale=1.7363cm,yscale=-1.7363cm] {M -.16337 -1.4054 Q .05417 -1.2966 .26002 -1.1879};
\path svg[xscale=1.7363cm,yscale=-1.7363cm] {M 1 0 Q .40748 -.05337 -.00261 -.10674};
\path svg[xscale=1.7363cm,yscale=-1.7363cm] {M -.81534 -1.1921 Q -.5982 -1.0726 -.39133 -.95307};
\path svg[xscale=1.7363cm,yscale=-1.7363cm] {M -.66384 -1.7138 Q -.71762 -1.453 -.81534 -1.1921};
\path svg[xscale=1.7363cm,yscale=-1.7363cm] {M -.66384 -1.7138 Q -.38154 -1.5596 -.16337 -1.4054};
\path svg[xscale=1.7363cm,yscale=-1.7363cm] {M -1 0 Q -.41046 -.05337 -.00261 -.10674};
\path svg[xscale=1.7363cm,yscale=-1.7363cm] {M -.02601 -.64883 Q -.01587 -.37778 -.00261 -.10674};
\path svg[xscale=1.7363cm,yscale=-1.7363cm] {M -1 -1 Q -.90486 -1.0961 -.81534 -1.1921};
\path svg[xscale=1.7363cm,yscale=-1.7363cm] {M .26002 -1.1879 Q .11609 -.91837 -.02601 -.64883};
\path svg[xscale=1.7363cm,yscale=-1.7363cm] {M .26002 -1.1879 Q .48514 -1.3247 .73904 -1.4616};
\path svg[xscale=1.7363cm,yscale=-1.7363cm] {M -1 -2 Q -.73497 -1.8569 -.66384 -1.7138};
\end{scope}
\path[draw=blue,line width=.01405cm] (-1,.4) -- (-1,-2.4) -- (1,-2.4) -- (1,.4) -- cycle;
\end{scope}
\end{scope}
\end{tikzpicture}
}
\expandafter\def\csname e8diagram:H12K2\endcsname{%
\begin{tikzpicture}[execute at begin picture={\clip (-1.24,-2.736) rectangle ++(2.48,3.472);},x={(1.7363cm,0cm)},y={(0cm,-1.7363cm)},baseline={([yshift=-\the\dimexpr\fontdimen22\textfont2\relax]current  bounding  box.center)},line join=round]
\begin{scope}
\begin{scope}
\path[clip] (-1,.4) -- (-1,-2.4) -- (1,-2.4) -- (1,.4) -- cycle;
\path[fill={rgb,255:red,204;green,204;blue,255},draw=none] (-1,.4) -- (-1,-2.4) -- (1,-2.4) -- (1,.4) -- cycle;
\begin{scope}[every path/.append style={draw=red,line width=.05619cm}]
\path svg[xscale=1.7363cm,yscale=-1.7363cm] {M -.66983 -1.234 Q -.41829 -1.1474 -.1748 -1.0608};
\path svg[xscale=1.7363cm,yscale=-1.7363cm] {M -.1748 -1.0608 Q -.122 -.75411 -.0732 -.4474};
\path svg[xscale=1.7363cm,yscale=-1.7363cm] {M -1 0 Q -.33867 -.2237 -.0732 -.4474};
\path svg[xscale=1.7363cm,yscale=-1.7363cm] {M -.52641 -1.7251 Q -.58355 -1.4795 -.66983 -1.234};
\path svg[xscale=1.7363cm,yscale=-1.7363cm] {M -.52641 -1.7251 Q -.20709 -1.6979 .09552 -1.6708};
\path svg[xscale=1.7363cm,yscale=-1.7363cm] {M 1 -1 Q .80616 -.72948 .6464 -.45896};
\path svg[xscale=1.7363cm,yscale=-1.7363cm] {M -1 -1 Q -.83197 -1.117 -.66983 -1.234};
\path svg[xscale=1.7363cm,yscale=-1.7363cm] {M .09552 -1.6708 Q -.05239 -1.3658 -.1748 -1.0608};
\path svg[xscale=1.7363cm,yscale=-1.7363cm] {M -.0732 -.4474 Q .28784 -.45318 .6464 -.45896};
\path svg[xscale=1.7363cm,yscale=-1.7363cm] {M 1 -2 Q .37836 -1.8354 .09552 -1.6708};
\path svg[xscale=1.7363cm,yscale=-1.7363cm] {M -1 -2 Q -.65459 -1.8626 -.52641 -1.7251};
\path svg[xscale=1.7363cm,yscale=-1.7363cm] {M 1 0 Q .68616 -.22948 .6464 -.45896};
\end{scope}
\path[draw=blue,line width=.01405cm] (-1,.4) -- (-1,-2.4) -- (1,-2.4) -- (1,.4) -- cycle;
\end{scope}
\end{scope}
\end{tikzpicture}
}
\expandafter\def\csname e8diagram:H12K2H1\endcsname{%
\begin{tikzpicture}[execute at begin picture={\clip (-1.24,-2.736) rectangle ++(2.48,3.472);},x={(1.7363cm,0cm)},y={(0cm,-1.7363cm)},baseline={([yshift=-\the\dimexpr\fontdimen22\textfont2\relax]current  bounding  box.center)},line join=round]
\begin{scope}
\begin{scope}
\path[clip] (-1,.4) -- (-1,-2.4) -- (1,-2.4) -- (1,.4) -- cycle;
\path[fill={rgb,255:red,204;green,204;blue,255},draw=none] (-1,.4) -- (-1,-2.4) -- (1,-2.4) -- (1,.4) -- cycle;
\begin{scope}[every path/.append style={draw=red,line width=.05619cm}]
\path svg[xscale=1.7363cm,yscale=-1.7363cm] {M -.32057 -.8841 Q -.27227 -.61018 -.23877 -.33625};
\path svg[xscale=1.7363cm,yscale=-1.7363cm] {M 1 -1 Q .82997 -.97684 .66016 -.95368};
\path svg[xscale=1.7363cm,yscale=-1.7363cm] {M -.08262 -1.3733 Q -.19963 -1.1287 -.32057 -.8841};
\path svg[xscale=1.7363cm,yscale=-1.7363cm] {M -.08262 -1.3733 Q .17111 -1.4568 .4466 -1.5402};
\path svg[xscale=1.7363cm,yscale=-1.7363cm] {M 1 0 Q .58044 -.17961 .43889 -.35921};
\path svg[xscale=1.7363cm,yscale=-1.7363cm] {M -.75171 -1.1625 Q -.53124 -1.0233 -.32057 -.8841};
\path svg[xscale=1.7363cm,yscale=-1.7363cm] {M .66016 -.95368 Q .53491 -.65644 .43889 -.35921};
\path svg[xscale=1.7363cm,yscale=-1.7363cm] {M -1 0 Q -.46387 -.16813 -.23877 -.33625};
\path svg[xscale=1.7363cm,yscale=-1.7363cm] {M -.59162 -1.6694 Q -.6551 -1.4159 -.75171 -1.1625};
\path svg[xscale=1.7363cm,yscale=-1.7363cm] {M -.59162 -1.6694 Q -.31071 -1.5213 -.08262 -1.3733};
\path svg[xscale=1.7363cm,yscale=-1.7363cm] {M -.23877 -.33625 Q .10328 -.34773 .43889 -.35921};
\path svg[xscale=1.7363cm,yscale=-1.7363cm] {M -1 -1 Q -.87412 -1.0812 -.75171 -1.1625};
\path svg[xscale=1.7363cm,yscale=-1.7363cm] {M .4466 -1.5402 Q .53804 -1.2469 .66016 -.95368};
\path svg[xscale=1.7363cm,yscale=-1.7363cm] {M -1 -2 Q -.68022 -1.8347 -.59162 -1.6694};
\path svg[xscale=1.7363cm,yscale=-1.7363cm] {M 1 -2 Q .56413 -1.7701 .4466 -1.5402};
\end{scope}
\path[draw=blue,line width=.01405cm] (-1,.4) -- (-1,-2.4) -- (1,-2.4) -- (1,.4) -- cycle;
\end{scope}
\end{scope}
\end{tikzpicture}
}
\expandafter\def\csname e8diagram:H12K2H12\endcsname{%
\begin{tikzpicture}[execute at begin picture={\clip (-1.24,-2.736) rectangle ++(2.48,3.472);},x={(1.7363cm,0cm)},y={(0cm,-1.7363cm)},baseline={([yshift=-\the\dimexpr\fontdimen22\textfont2\relax]current  bounding  box.center)},line join=round]
\begin{scope}
\begin{scope}
\path[clip] (-1,.4) -- (-1,-2.4) -- (1,-2.4) -- (1,.4) -- cycle;
\path[fill={rgb,255:red,204;green,204;blue,255},draw=none] (-1,.4) -- (-1,-2.4) -- (1,-2.4) -- (1,.4) -- cycle;
\begin{scope}[every path/.append style={draw=red,line width=.05619cm}]
\path svg[xscale=1.7363cm,yscale=-1.7363cm] {M 1 -2 Q .71334 -1.8264 .65228 -1.6528};
\path svg[xscale=1.7363cm,yscale=-1.7363cm] {M -.43983 -.79011 Q -.38806 -.52233 -.3636 -.25455};
\path svg[xscale=1.7363cm,yscale=-1.7363cm] {M -.2132 -1.2626 Q -.31948 -1.0263 -.43983 -.79011};
\path svg[xscale=1.7363cm,yscale=-1.7363cm] {M -.2132 -1.2626 Q 0 -1.2626 .2132 -1.2626};
\path svg[xscale=1.7363cm,yscale=-1.7363cm] {M 1 -1 Q .90831 -1.0681 .81944 -1.1363};
\path svg[xscale=1.7363cm,yscale=-1.7363cm] {M .65228 -1.6528 Q .71675 -1.3945 .81944 -1.1363};
\path svg[xscale=1.7363cm,yscale=-1.7363cm] {M -.81944 -1.1363 Q -.62279 -.96318 -.43983 -.79011};
\path svg[xscale=1.7363cm,yscale=-1.7363cm] {M 1 0 Q .56182 -.12727 .3636 -.25455};
\path svg[xscale=1.7363cm,yscale=-1.7363cm] {M .43983 -.79011 Q .38806 -.52233 .3636 -.25455};
\path svg[xscale=1.7363cm,yscale=-1.7363cm] {M .43983 -.79011 Q .62279 -.96318 .81944 -1.1363};
\path svg[xscale=1.7363cm,yscale=-1.7363cm] {M -1 0 Q -.56182 -.12727 -.3636 -.25455};
\path svg[xscale=1.7363cm,yscale=-1.7363cm] {M -.65228 -1.6528 Q -.71675 -1.3945 -.81944 -1.1363};
\path svg[xscale=1.7363cm,yscale=-1.7363cm] {M -.65228 -1.6528 Q -.40251 -1.4577 -.2132 -1.2626};
\path svg[xscale=1.7363cm,yscale=-1.7363cm] {M -1 -1 Q -.90831 -1.0681 -.81944 -1.1363};
\path svg[xscale=1.7363cm,yscale=-1.7363cm] {M -.3636 -.25455 Q 0 -.25455 .3636 -.25455};
\path svg[xscale=1.7363cm,yscale=-1.7363cm] {M .2132 -1.2626 Q .31948 -1.0263 .43983 -.79011};
\path svg[xscale=1.7363cm,yscale=-1.7363cm] {M -1 -2 Q -.71334 -1.8264 -.65228 -1.6528};
\path svg[xscale=1.7363cm,yscale=-1.7363cm] {M .2132 -1.2626 Q .40251 -1.4577 .65228 -1.6528};
\end{scope}
\path[draw=blue,line width=.01405cm] (-1,.4) -- (-1,-2.4) -- (1,-2.4) -- (1,.4) -- cycle;
\end{scope}
\end{scope}
\end{tikzpicture}
}
\expandafter\def\csname e8diagram:H1E2\endcsname{%
\begin{tikzpicture}[execute at begin picture={\clip (-1.24,-2.736) rectangle ++(2.48,3.472);},x={(1.7363cm,0cm)},y={(0cm,-1.7363cm)},baseline={([yshift=-\the\dimexpr\fontdimen22\textfont2\relax]current  bounding  box.center)},line join=round]
\begin{scope}
\begin{scope}
\path[clip] (-1,.4) -- (-1,-2.4) -- (1,-2.4) -- (1,.4) -- cycle;
\path[fill={rgb,255:red,204;green,204;blue,255},draw=none] (-1,.4) -- (-1,-2.4) -- (1,-2.4) -- (1,.4) -- cycle;
\begin{scope}[every path/.append style={draw=red,line width=.05619cm}]
\path svg[xscale=1.7363cm,yscale=-1.7363cm] {M -1 -1 Q -.77384 -.92175 -.54935 -.84349};
\path svg[xscale=1.7363cm,yscale=-1.7363cm] {M -1 0 Q -.56447 -.42175 -.54935 -.84349};
\path svg[xscale=1.7363cm,yscale=-1.7363cm] {M 1 -2 Q .19023 -1.8192 -.19976 -1.6384};
\path svg[xscale=1.7363cm,yscale=-1.7363cm] {M -.19976 -1.6384 Q -.36201 -1.2409 -.54935 -.84349};
\path svg[xscale=1.7363cm,yscale=-1.7363cm] {M -1 -2 Q -.43399 -1.8192 -.19976 -1.6384};
\path svg[xscale=1.7363cm,yscale=-1.7363cm] {M 1 0 Q .29519 -.5 1 -1};
\end{scope}
\path[draw=blue,line width=.01405cm] (-1,.4) -- (-1,-2.4) -- (1,-2.4) -- (1,.4) -- cycle;
\end{scope}
\end{scope}
\end{tikzpicture}
}
\expandafter\def\csname e8diagram:H1H2\endcsname{%
\begin{tikzpicture}[execute at begin picture={\clip (-1.24,-2.736) rectangle ++(2.48,3.472);},x={(1.7363cm,0cm)},y={(0cm,-1.7363cm)},baseline={([yshift=-\the\dimexpr\fontdimen22\textfont2\relax]current  bounding  box.center)},line join=round]
\begin{scope}
\begin{scope}
\path[clip] (-1,.4) -- (-1,-2.4) -- (1,-2.4) -- (1,.4) -- cycle;
\path[fill={rgb,255:red,204;green,204;blue,255},draw=none] (-1,.4) -- (-1,-2.4) -- (1,-2.4) -- (1,.4) -- cycle;
\begin{scope}[every path/.append style={draw=red,line width=.05619cm}]
\path svg[xscale=1.7363cm,yscale=-1.7363cm] {M -1 0 Q -.29162 -.09904 .11978 -.19807};
\path svg[xscale=1.7363cm,yscale=-1.7363cm] {M 1 -1 Q .66488 -.90908 .32975 -.81817};
\path svg[xscale=1.7363cm,yscale=-1.7363cm] {M -.11978 -1.8019 Q -.22762 -1.4919 -.32975 -1.1818};
\path svg[xscale=1.7363cm,yscale=-1.7363cm] {M -1 -1 Q -.66488 -1.0909 -.32975 -1.1818};
\path svg[xscale=1.7363cm,yscale=-1.7363cm] {M .32975 -.81817 Q .22762 -.50812 .11978 -.19807};
\path svg[xscale=1.7363cm,yscale=-1.7363cm] {M 1 -2 Q .29162 -1.901 -.11978 -1.8019};
\path svg[xscale=1.7363cm,yscale=-1.7363cm] {M 1 0 Q .43453 -.09904 .11978 -.19807};
\path svg[xscale=1.7363cm,yscale=-1.7363cm] {M -1 -2 Q -.43453 -1.901 -.11978 -1.8019};
\path svg[xscale=1.7363cm,yscale=-1.7363cm] {M -.32975 -1.1818 Q 0 -1 .32975 -.81817};
\end{scope}
\path[draw=blue,line width=.01405cm] (-1,.4) -- (-1,-2.4) -- (1,-2.4) -- (1,.4) -- cycle;
\end{scope}
\end{scope}
\end{tikzpicture}
}
\expandafter\def\csname e8diagram:H1H22\endcsname{%
\begin{tikzpicture}[execute at begin picture={\clip (-1.24,-2.736) rectangle ++(2.48,3.472);},x={(1.7363cm,0cm)},y={(0cm,-1.7363cm)},baseline={([yshift=-\the\dimexpr\fontdimen22\textfont2\relax]current  bounding  box.center)},line join=round]
\begin{scope}
\begin{scope}
\path[clip] (-1,.4) -- (-1,-2.4) -- (1,-2.4) -- (1,.4) -- cycle;
\path[fill={rgb,255:red,204;green,204;blue,255},draw=none] (-1,.4) -- (-1,-2.4) -- (1,-2.4) -- (1,.4) -- cycle;
\begin{scope}[every path/.append style={draw=red,line width=.05619cm}]
\path svg[xscale=1.7363cm,yscale=-1.7363cm] {M -1 0 Q -.46391 -.12169 -.19562 -.24339};
\path svg[xscale=1.7363cm,yscale=-1.7363cm] {M -.19562 -.24339 Q .1395 -.24278 .4751 -.24216};
\path svg[xscale=1.7363cm,yscale=-1.7363cm] {M -.18013 -1.8211 Q -.33082 -1.5113 -.47398 -1.2014};
\path svg[xscale=1.7363cm,yscale=-1.7363cm] {M 1 -1 Q .78552 -.86981 .57638 -.73962};
\path svg[xscale=1.7363cm,yscale=-1.7363cm] {M -1 -1 Q -.73608 -1.1007 -.47398 -1.2014};
\path svg[xscale=1.7363cm,yscale=-1.7363cm] {M -.04069 -.79513 Q -.10116 -.51926 -.19562 -.24339};
\path svg[xscale=1.7363cm,yscale=-1.7363cm] {M -.04069 -.79513 Q .26603 -.76737 .57638 -.73962};
\path svg[xscale=1.7363cm,yscale=-1.7363cm] {M 1 -2 Q .26445 -1.9106 -.18013 -1.8211};
\path svg[xscale=1.7363cm,yscale=-1.7363cm] {M .57638 -.73962 Q .5101 -.49089 .4751 -.24216};
\path svg[xscale=1.7363cm,yscale=-1.7363cm] {M -1 -2 Q -.47809 -1.9106 -.18013 -1.8211};
\path svg[xscale=1.7363cm,yscale=-1.7363cm] {M 1 0 Q .63275 -.12108 .4751 -.24216};
\path svg[xscale=1.7363cm,yscale=-1.7363cm] {M -.47398 -1.2014 Q -.25297 -.99829 -.04069 -.79513};
\end{scope}
\path[draw=blue,line width=.01405cm] (-1,.4) -- (-1,-2.4) -- (1,-2.4) -- (1,.4) -- cycle;
\end{scope}
\end{scope}
\end{tikzpicture}
}
\expandafter\def\csname e8diagram:H1H22H1\endcsname{%
\begin{tikzpicture}[execute at begin picture={\clip (-1.24,-2.736) rectangle ++(2.48,3.472);},x={(1.7363cm,0cm)},y={(0cm,-1.7363cm)},baseline={([yshift=-\the\dimexpr\fontdimen22\textfont2\relax]current  bounding  box.center)},line join=round]
\begin{scope}
\begin{scope}
\path[clip] (-1,.4) -- (-1,-2.4) -- (1,-2.4) -- (1,.4) -- cycle;
\path[fill={rgb,255:red,204;green,204;blue,255},draw=none] (-1,.4) -- (-1,-2.4) -- (1,-2.4) -- (1,.4) -- cycle;
\begin{scope}[every path/.append style={draw=red,line width=.05619cm}]
\path svg[xscale=1.7363cm,yscale=-1.7363cm] {M -.26162 -.73063 Q -.28658 -.45784 -.36484 -.18505};
\path svg[xscale=1.7363cm,yscale=-1.7363cm] {M -.26162 -.73063 Q 0 -.73063 .26162 -.73063};
\path svg[xscale=1.7363cm,yscale=-1.7363cm] {M 1 -1 Q .80373 -1.0676 .60905 -1.1353};
\path svg[xscale=1.7363cm,yscale=-1.7363cm] {M .36533 -1.6802 Q .47996 -1.4077 .60905 -1.1353};
\path svg[xscale=1.7363cm,yscale=-1.7363cm] {M 1 0 Q .58536 -.09252 .36484 -.18505};
\path svg[xscale=1.7363cm,yscale=-1.7363cm] {M -.60905 -1.1353 Q -.42984 -.93295 -.26162 -.73063};
\path svg[xscale=1.7363cm,yscale=-1.7363cm] {M .26162 -.73063 Q .28658 -.45784 .36484 -.18505};
\path svg[xscale=1.7363cm,yscale=-1.7363cm] {M -1 0 Q -.58536 -.09252 -.36484 -.18505};
\path svg[xscale=1.7363cm,yscale=-1.7363cm] {M -.36533 -1.6802 Q -.47996 -1.4077 -.60905 -1.1353};
\path svg[xscale=1.7363cm,yscale=-1.7363cm] {M -.36533 -1.6802 Q 0 -1.6802 .36533 -1.6802};
\path svg[xscale=1.7363cm,yscale=-1.7363cm] {M -.36484 -.18505 Q 0 -.18505 .36484 -.18505};
\path svg[xscale=1.7363cm,yscale=-1.7363cm] {M -1 -1 Q -.80373 -1.0676 -.60905 -1.1353};
\path svg[xscale=1.7363cm,yscale=-1.7363cm] {M 1 -2 Q .54505 -1.8401 .36533 -1.6802};
\path svg[xscale=1.7363cm,yscale=-1.7363cm] {M .60905 -1.1353 Q .42984 -.93295 .26162 -.73063};
\path svg[xscale=1.7363cm,yscale=-1.7363cm] {M -1 -2 Q -.54505 -1.8401 -.36533 -1.6802};
\end{scope}
\path[draw=blue,line width=.01405cm] (-1,.4) -- (-1,-2.4) -- (1,-2.4) -- (1,.4) -- cycle;
\end{scope}
\end{scope}
\end{tikzpicture}
}
\expandafter\def\csname e8diagram:H1H2H1\endcsname{%
\begin{tikzpicture}[execute at begin picture={\clip (-1.24,-2.736) rectangle ++(2.48,3.472);},x={(1.7363cm,0cm)},y={(0cm,-1.7363cm)},baseline={([yshift=-\the\dimexpr\fontdimen22\textfont2\relax]current  bounding  box.center)},line join=round]
\begin{scope}
\begin{scope}
\path[clip] (-1,.4) -- (-1,-2.4) -- (1,-2.4) -- (1,.4) -- cycle;
\path[fill={rgb,255:red,204;green,204;blue,255},draw=none] (-1,.4) -- (-1,-2.4) -- (1,-2.4) -- (1,.4) -- cycle;
\begin{scope}[every path/.append style={draw=red,line width=.05619cm}]
\path svg[xscale=1.7363cm,yscale=-1.7363cm] {M -.52332 -1.1303 Q -.25999 -.96 0 -.78968};
\path svg[xscale=1.7363cm,yscale=-1.7363cm] {M 0 -.78968 Q .25999 -.96 .52332 -1.1303};
\path svg[xscale=1.7363cm,yscale=-1.7363cm] {M 0 -.78968 Q 0 -.47644 0 -.16319};
\path svg[xscale=1.7363cm,yscale=-1.7363cm] {M -.33892 -1.6698 Q -.42322 -1.4001 -.52332 -1.1303};
\path svg[xscale=1.7363cm,yscale=-1.7363cm] {M -1 0 Q -.37846 -.08159 0 -.16319};
\path svg[xscale=1.7363cm,yscale=-1.7363cm] {M -.33892 -1.6698 Q 0 -1.6698 .33892 -1.6698};
\path svg[xscale=1.7363cm,yscale=-1.7363cm] {M 1 -1 Q .76115 -1.0652 .52332 -1.1303};
\path svg[xscale=1.7363cm,yscale=-1.7363cm] {M -1 -1 Q -.76115 -1.0652 -.52332 -1.1303};
\path svg[xscale=1.7363cm,yscale=-1.7363cm] {M .33892 -1.6698 Q .42322 -1.4001 .52332 -1.1303};
\path svg[xscale=1.7363cm,yscale=-1.7363cm] {M 1 -2 Q .52639 -1.8349 .33892 -1.6698};
\path svg[xscale=1.7363cm,yscale=-1.7363cm] {M -1 -2 Q -.52639 -1.8349 -.33892 -1.6698};
\path svg[xscale=1.7363cm,yscale=-1.7363cm] {M 1 0 Q .37846 -.08159 0 -.16319};
\end{scope}
\path[draw=blue,line width=.01405cm] (-1,.4) -- (-1,-2.4) -- (1,-2.4) -- (1,.4) -- cycle;
\end{scope}
\end{scope}
\end{tikzpicture}
}
\expandafter\def\csname e8diagram:H1H2H1H2\endcsname{%
\begin{tikzpicture}[execute at begin picture={\clip (-1.24,-2.736) rectangle ++(2.48,3.472);},x={(1.7363cm,0cm)},y={(0cm,-1.7363cm)},baseline={([yshift=-\the\dimexpr\fontdimen22\textfont2\relax]current  bounding  box.center)},line join=round]
\begin{scope}
\begin{scope}
\path[clip] (-1,.4) -- (-1,-2.4) -- (1,-2.4) -- (1,.4) -- cycle;
\path[fill={rgb,255:red,204;green,204;blue,255},draw=none] (-1,.4) -- (-1,-2.4) -- (1,-2.4) -- (1,.4) -- cycle;
\begin{scope}[every path/.append style={draw=red,line width=.05619cm}]
\path svg[xscale=1.7363cm,yscale=-1.7363cm] {M -.22114 -.8501 Q 0 -1 .22114 -1.1499};
\path svg[xscale=1.7363cm,yscale=-1.7363cm] {M -.22114 -.8501 Q -.21322 -.5669 -.22785 -.2837};
\path svg[xscale=1.7363cm,yscale=-1.7363cm] {M 1 -1 Q .82252 -.90964 .64829 -.81928};
\path svg[xscale=1.7363cm,yscale=-1.7363cm] {M .22785 -1.7163 Q .21322 -1.4331 .22114 -1.1499};
\path svg[xscale=1.7363cm,yscale=-1.7363cm] {M 1 0 Q .61719 -.14076 .46766 -.28153};
\path svg[xscale=1.7363cm,yscale=-1.7363cm] {M -.64829 -1.1807 Q -.42932 -1.0154 -.22114 -.8501};
\path svg[xscale=1.7363cm,yscale=-1.7363cm] {M .64829 -.81928 Q .54435 -.55041 .46766 -.28153};
\path svg[xscale=1.7363cm,yscale=-1.7363cm] {M -1 0 Q -.47115 -.14185 -.22785 -.2837};
\path svg[xscale=1.7363cm,yscale=-1.7363cm] {M -.46766 -1.7185 Q -.54435 -1.4496 -.64829 -1.1807};
\path svg[xscale=1.7363cm,yscale=-1.7363cm] {M -.46766 -1.7185 Q -.11951 -1.7174 .22785 -1.7163};
\path svg[xscale=1.7363cm,yscale=-1.7363cm] {M -.22785 -.2837 Q .11951 -.28262 .46766 -.28153};
\path svg[xscale=1.7363cm,yscale=-1.7363cm] {M -1 -1 Q -.82252 -1.0904 -.64829 -1.1807};
\path svg[xscale=1.7363cm,yscale=-1.7363cm] {M 1 -2 Q .47115 -1.8581 .22785 -1.7163};
\path svg[xscale=1.7363cm,yscale=-1.7363cm] {M .22114 -1.1499 Q .42932 -.98459 .64829 -.81928};
\path svg[xscale=1.7363cm,yscale=-1.7363cm] {M -1 -2 Q -.61719 -1.8592 -.46766 -1.7185};
\end{scope}
\path[draw=blue,line width=.01405cm] (-1,.4) -- (-1,-2.4) -- (1,-2.4) -- (1,.4) -- cycle;
\end{scope}
\end{scope}
\end{tikzpicture}
}
\expandafter\def\csname e8diagram:H1H2H1H2H1\endcsname{%
\begin{tikzpicture}[execute at begin picture={\clip (-1.24,-2.736) rectangle ++(2.48,3.472);},x={(1.7363cm,0cm)},y={(0cm,-1.7363cm)},baseline={([yshift=-\the\dimexpr\fontdimen22\textfont2\relax]current  bounding  box.center)},line join=round]
\begin{scope}
\begin{scope}
\path[clip] (-1,.4) -- (-1,-2.4) -- (1,-2.4) -- (1,.4) -- cycle;
\path[fill={rgb,255:red,204;green,204;blue,255},draw=none] (-1,.4) -- (-1,-2.4) -- (1,-2.4) -- (1,.4) -- cycle;
\begin{scope}[every path/.append style={draw=red,line width=.05619cm}]
\path svg[xscale=1.7363cm,yscale=-1.7363cm] {M 1 -2 Q .67381 -1.839 .57768 -1.6781};
\path svg[xscale=1.7363cm,yscale=-1.7363cm] {M -.39952 -.7434 Q -.19575 -.86428 0 -.98516};
\path svg[xscale=1.7363cm,yscale=-1.7363cm] {M -.39952 -.7434 Q -.37164 -.47045 -.38452 -.1975};
\path svg[xscale=1.7363cm,yscale=-1.7363cm] {M 1 -1 Q .87383 -1.068 .75007 -1.1359};
\path svg[xscale=1.7363cm,yscale=-1.7363cm] {M 0 -1.5138 Q 0 -1.2495 0 -.98516};
\path svg[xscale=1.7363cm,yscale=-1.7363cm] {M .39952 -.7434 Q .56726 -.93967 .75007 -1.1359};
\path svg[xscale=1.7363cm,yscale=-1.7363cm] {M 0 -1.5138 Q .27059 -1.596 .57768 -1.6781};
\path svg[xscale=1.7363cm,yscale=-1.7363cm] {M .39952 -.7434 Q .37164 -.47045 .38452 -.1975};
\path svg[xscale=1.7363cm,yscale=-1.7363cm] {M -.75007 -1.1359 Q -.56726 -.93967 -.39952 -.7434};
\path svg[xscale=1.7363cm,yscale=-1.7363cm] {M 1 0 Q .59265 -.09875 .38452 -.1975};
\path svg[xscale=1.7363cm,yscale=-1.7363cm] {M -.38452 -.1975 Q 0 -.1975 .38452 -.1975};
\path svg[xscale=1.7363cm,yscale=-1.7363cm] {M -.57768 -1.6781 Q -.6454 -1.407 -.75007 -1.1359};
\path svg[xscale=1.7363cm,yscale=-1.7363cm] {M -1 0 Q -.59265 -.09875 -.38452 -.1975};
\path svg[xscale=1.7363cm,yscale=-1.7363cm] {M -.57768 -1.6781 Q -.27059 -1.596 0 -1.5138};
\path svg[xscale=1.7363cm,yscale=-1.7363cm] {M -1 -1 Q -.87383 -1.068 -.75007 -1.1359};
\path svg[xscale=1.7363cm,yscale=-1.7363cm] {M .57768 -1.6781 Q .6454 -1.407 .75007 -1.1359};
\path svg[xscale=1.7363cm,yscale=-1.7363cm] {M -1 -2 Q -.67381 -1.839 -.57768 -1.6781};
\path svg[xscale=1.7363cm,yscale=-1.7363cm] {M 0 -.98516 Q .19575 -.86428 .39952 -.7434};
\end{scope}
\path[draw=blue,line width=.01405cm] (-1,.4) -- (-1,-2.4) -- (1,-2.4) -- (1,.4) -- cycle;
\end{scope}
\end{scope}
\end{tikzpicture}
}
\expandafter\def\csname e8diagram:H1H2K1\endcsname{%
\begin{tikzpicture}[execute at begin picture={\clip (-1.24,-2.736) rectangle ++(2.48,3.472);},x={(1.7363cm,0cm)},y={(0cm,-1.7363cm)},baseline={([yshift=-\the\dimexpr\fontdimen22\textfont2\relax]current  bounding  box.center)},line join=round]
\begin{scope}
\begin{scope}
\path[clip] (-1,.4) -- (-1,-2.4) -- (1,-2.4) -- (1,.4) -- cycle;
\path[fill={rgb,255:red,204;green,204;blue,255},draw=none] (-1,.4) -- (-1,-2.4) -- (1,-2.4) -- (1,.4) -- cycle;
\begin{scope}[every path/.append style={draw=red,line width=.05619cm}]
\path svg[xscale=1.7363cm,yscale=-1.7363cm] {M -.64724 -1.106 Q -.41966 -.94693 -.19782 -.78787};
\path svg[xscale=1.7363cm,yscale=-1.7363cm] {M -.19782 -.78787 Q -.13551 -.48227 -.06789 -.17666};
\path svg[xscale=1.7363cm,yscale=-1.7363cm] {M -.44605 -1.5939 Q -.53778 -1.3499 -.64724 -1.106};
\path svg[xscale=1.7363cm,yscale=-1.7363cm] {M -1 0 Q -.41248 -.08833 -.06789 -.17666};
\path svg[xscale=1.7363cm,yscale=-1.7363cm] {M -.44605 -1.5939 Q -.16983 -1.443 .06697 -1.2921};
\path svg[xscale=1.7363cm,yscale=-1.7363cm] {M 1 -1 Q .81637 -1.2342 .65793 -1.4684};
\path svg[xscale=1.7363cm,yscale=-1.7363cm] {M -1 -1 Q -.82307 -1.053 -.64724 -1.106};
\path svg[xscale=1.7363cm,yscale=-1.7363cm] {M .06697 -1.2921 Q -.06481 -1.04 -.19782 -.78787};
\path svg[xscale=1.7363cm,yscale=-1.7363cm] {M .06697 -1.2921 Q .35162 -1.3803 .65793 -1.4684};
\path svg[xscale=1.7363cm,yscale=-1.7363cm] {M -1 -2 Q -.57455 -1.7969 -.44605 -1.5939};
\path svg[xscale=1.7363cm,yscale=-1.7363cm] {M 1 0 Q .33202 -.08833 -.06789 -.17666};
\path svg[xscale=1.7363cm,yscale=-1.7363cm] {M 1 -2 Q .67944 -1.7342 .65793 -1.4684};
\end{scope}
\path[draw=blue,line width=.01405cm] (-1,.4) -- (-1,-2.4) -- (1,-2.4) -- (1,.4) -- cycle;
\end{scope}
\end{scope}
\end{tikzpicture}
}
\expandafter\def\csname e8diagram:H1K2\endcsname{%
\begin{tikzpicture}[execute at begin picture={\clip (-1.24,-2.736) rectangle ++(2.48,3.472);},x={(1.7363cm,0cm)},y={(0cm,-1.7363cm)},baseline={([yshift=-\the\dimexpr\fontdimen22\textfont2\relax]current  bounding  box.center)},line join=round]
\begin{scope}
\begin{scope}
\path[clip] (-1,.4) -- (-1,-2.4) -- (1,-2.4) -- (1,.4) -- cycle;
\path[fill={rgb,255:red,204;green,204;blue,255},draw=none] (-1,.4) -- (-1,-2.4) -- (1,-2.4) -- (1,.4) -- cycle;
\begin{scope}[every path/.append style={draw=red,line width=.05619cm}]
\path svg[xscale=1.7363cm,yscale=-1.7363cm] {M -1 0 Q -.39748 -.24101 -.18002 -.48203};
\path svg[xscale=1.7363cm,yscale=-1.7363cm] {M -.1763 -1.7498 Q -.32843 -1.4187 -.47994 -1.0876};
\path svg[xscale=1.7363cm,yscale=-1.7363cm] {M 1 -1 Q .77264 -.7441 .56813 -.48821};
\path svg[xscale=1.7363cm,yscale=-1.7363cm] {M -1 -1 Q -.73979 -1.0438 -.47994 -1.0876};
\path svg[xscale=1.7363cm,yscale=-1.7363cm] {M -.18002 -.48203 Q .19469 -.48512 .56813 -.48821};
\path svg[xscale=1.7363cm,yscale=-1.7363cm] {M 1 -2 Q .23725 -1.8749 -.1763 -1.7498};
\path svg[xscale=1.7363cm,yscale=-1.7363cm] {M 1 0 Q .63283 -.2441 .56813 -.48821};
\path svg[xscale=1.7363cm,yscale=-1.7363cm] {M -1 -2 Q -.45007 -1.8749 -.1763 -1.7498};
\path svg[xscale=1.7363cm,yscale=-1.7363cm] {M -.47994 -1.0876 Q -.32489 -.7848 -.18002 -.48203};
\end{scope}
\path[draw=blue,line width=.01405cm] (-1,.4) -- (-1,-2.4) -- (1,-2.4) -- (1,.4) -- cycle;
\end{scope}
\end{scope}
\end{tikzpicture}
}
\expandafter\def\csname e8diagram:H1K2E1\endcsname{%
\begin{tikzpicture}[execute at begin picture={\clip (-1.24,-2.736) rectangle ++(2.48,3.472);},x={(1.7363cm,0cm)},y={(0cm,-1.7363cm)},baseline={([yshift=-\the\dimexpr\fontdimen22\textfont2\relax]current  bounding  box.center)},line join=round]
\begin{scope}
\begin{scope}
\path[clip] (-1,.4) -- (-1,-2.4) -- (1,-2.4) -- (1,.4) -- cycle;
\path[fill={rgb,255:red,204;green,204;blue,255},draw=none] (-1,.4) -- (-1,-2.4) -- (1,-2.4) -- (1,.4) -- cycle;
\begin{scope}[every path/.append style={draw=red,line width=.05619cm}]
\path svg[xscale=1.7363cm,yscale=-1.7363cm] {M -.67592 -.87227 Q -.5847 -.62607 -.5187 -.37987};
\path svg[xscale=1.7363cm,yscale=-1.7363cm] {M 1 -1 Q .29519 -1.5 1 -2};
\path svg[xscale=1.7363cm,yscale=-1.7363cm] {M -1 0 Q -.62453 -.18994 -.5187 -.37987};
\path svg[xscale=1.7363cm,yscale=-1.7363cm] {M -.35323 -1.2549 Q -.50833 -1.0636 -.67592 -.87227};
\path svg[xscale=1.7363cm,yscale=-1.7363cm] {M -.35323 -1.2549 Q -.16766 -.87989 .01153 -.50492};
\path svg[xscale=1.7363cm,yscale=-1.7363cm] {M -1 -1 Q -.83708 -.93614 -.67592 -.87227};
\path svg[xscale=1.7363cm,yscale=-1.7363cm] {M .01153 -.50492 Q -.24253 -.4424 -.5187 -.37987};
\path svg[xscale=1.7363cm,yscale=-1.7363cm] {M 1 0 Q .29096 -.25246 .01153 -.50492};
\path svg[xscale=1.7363cm,yscale=-1.7363cm] {M -1 -2 Q -.46277 -1.6274 -.35323 -1.2549};
\end{scope}
\path[draw=blue,line width=.01405cm] (-1,.4) -- (-1,-2.4) -- (1,-2.4) -- (1,.4) -- cycle;
\end{scope}
\end{scope}
\end{tikzpicture}
}
\expandafter\def\csname e8diagram:H1K2E1H2\endcsname{%
\begin{tikzpicture}[execute at begin picture={\clip (-1.24,-2.736) rectangle ++(2.48,3.472);},x={(1.7363cm,0cm)},y={(0cm,-1.7363cm)},baseline={([yshift=-\the\dimexpr\fontdimen22\textfont2\relax]current  bounding  box.center)},line join=round]
\begin{scope}
\begin{scope}
\path[clip] (-1,.4) -- (-1,-2.4) -- (1,-2.4) -- (1,.4) -- cycle;
\path[fill={rgb,255:red,204;green,204;blue,255},draw=none] (-1,.4) -- (-1,-2.4) -- (1,-2.4) -- (1,.4) -- cycle;
\begin{scope}[every path/.append style={draw=red,line width=.05619cm}]
\path svg[xscale=1.7363cm,yscale=-1.7363cm] {M -.73579 -.96706 Q -.62349 -.71341 -.53776 -.45977};
\path svg[xscale=1.7363cm,yscale=-1.7363cm] {M 1 -2 Q .66796 -1.6505 .68691 -1.301};
\path svg[xscale=1.7363cm,yscale=-1.7363cm] {M -1 0 Q -.61974 -.22989 -.53776 -.45977};
\path svg[xscale=1.7363cm,yscale=-1.7363cm] {M -.46904 -1.4217 Q -.59368 -1.1944 -.73579 -.96706};
\path svg[xscale=1.7363cm,yscale=-1.7363cm] {M -.46904 -1.4217 Q -.28316 -1.1403 -.12957 -.85894};
\path svg[xscale=1.7363cm,yscale=-1.7363cm] {M 1 -1 Q .83823 -1.1505 .68691 -1.301};
\path svg[xscale=1.7363cm,yscale=-1.7363cm] {M -1 -1 Q -.86783 -.98353 -.73579 -.96706};
\path svg[xscale=1.7363cm,yscale=-1.7363cm] {M -.12957 -.85894 Q -.3139 -.65936 -.53776 -.45977};
\path svg[xscale=1.7363cm,yscale=-1.7363cm] {M -.12957 -.85894 Q .1215 -.74629 .38704 -.63364};
\path svg[xscale=1.7363cm,yscale=-1.7363cm] {M .38704 -.63364 Q .51304 -.96731 .68691 -1.301};
\path svg[xscale=1.7363cm,yscale=-1.7363cm] {M 1 0 Q .49859 -.31682 .38704 -.63364};
\path svg[xscale=1.7363cm,yscale=-1.7363cm] {M -1 -2 Q -.55679 -1.7109 -.46904 -1.4217};
\end{scope}
\path[draw=blue,line width=.01405cm] (-1,.4) -- (-1,-2.4) -- (1,-2.4) -- (1,.4) -- cycle;
\end{scope}
\end{scope}
\end{tikzpicture}
}
\expandafter\def\csname e8diagram:H1K2E1K2\endcsname{%
\begin{tikzpicture}[execute at begin picture={\clip (-1.24,-2.736) rectangle ++(2.48,3.472);},x={(1.7363cm,0cm)},y={(0cm,-1.7363cm)},baseline={([yshift=-\the\dimexpr\fontdimen22\textfont2\relax]current  bounding  box.center)},line join=round]
\begin{scope}
\begin{scope}
\path[clip] (-1,.4) -- (-1,-2.4) -- (1,-2.4) -- (1,.4) -- cycle;
\path[fill={rgb,255:red,204;green,204;blue,255},draw=none] (-1,.4) -- (-1,-2.4) -- (1,-2.4) -- (1,.4) -- cycle;
\begin{scope}[every path/.append style={draw=red,line width=.05619cm}]
\path svg[xscale=1.7363cm,yscale=-1.7363cm] {M -.73579 -1.0329 Q -.59368 -.8056 -.46904 -.57825};
\path svg[xscale=1.7363cm,yscale=-1.7363cm] {M -1 0 Q -.55679 -.28913 -.46904 -.57825};
\path svg[xscale=1.7363cm,yscale=-1.7363cm] {M -.53776 -1.5402 Q -.62349 -1.2866 -.73579 -1.0329};
\path svg[xscale=1.7363cm,yscale=-1.7363cm] {M -.53776 -1.5402 Q -.3139 -1.3406 -.12957 -1.1411};
\path svg[xscale=1.7363cm,yscale=-1.7363cm] {M 1 -1 Q .83823 -.8495 .68691 -.69901};
\path svg[xscale=1.7363cm,yscale=-1.7363cm] {M -1 -1 Q -.86783 -1.0165 -.73579 -1.0329};
\path svg[xscale=1.7363cm,yscale=-1.7363cm] {M -.12957 -1.1411 Q -.28316 -.85966 -.46904 -.57825};
\path svg[xscale=1.7363cm,yscale=-1.7363cm] {M -.12957 -1.1411 Q .1215 -1.2537 .38704 -1.3664};
\path svg[xscale=1.7363cm,yscale=-1.7363cm] {M .38704 -1.3664 Q .51304 -1.0327 .68691 -.69901};
\path svg[xscale=1.7363cm,yscale=-1.7363cm] {M -1 -2 Q -.61974 -1.7701 -.53776 -1.5402};
\path svg[xscale=1.7363cm,yscale=-1.7363cm] {M 1 0 Q .66796 -.3495 .68691 -.69901};
\path svg[xscale=1.7363cm,yscale=-1.7363cm] {M 1 -2 Q .49859 -1.6832 .38704 -1.3664};
\end{scope}
\path[draw=blue,line width=.01405cm] (-1,.4) -- (-1,-2.4) -- (1,-2.4) -- (1,.4) -- cycle;
\end{scope}
\end{scope}
\end{tikzpicture}
}
\expandafter\def\csname e8diagram:H1K2E1K2H1\endcsname{%
\begin{tikzpicture}[execute at begin picture={\clip (-1.24,-2.736) rectangle ++(2.48,3.472);},x={(1.7363cm,0cm)},y={(0cm,-1.7363cm)},baseline={([yshift=-\the\dimexpr\fontdimen22\textfont2\relax]current  bounding  box.center)},line join=round]
\begin{scope}
\begin{scope}
\path[clip] (-1,.4) -- (-1,-2.4) -- (1,-2.4) -- (1,.4) -- cycle;
\path[fill={rgb,255:red,204;green,204;blue,255},draw=none] (-1,.4) -- (-1,-2.4) -- (1,-2.4) -- (1,.4) -- cycle;
\begin{scope}[every path/.append style={draw=red,line width=.05619cm}]
\path svg[xscale=1.7363cm,yscale=-1.7363cm] {M 1 -2 Q .63029 -1.7517 .56652 -1.5035};
\path svg[xscale=1.7363cm,yscale=-1.7363cm] {M 1 -1 Q .89292 -1 .78584 -1};
\path svg[xscale=1.7363cm,yscale=-1.7363cm] {M -.2217 -1 Q -.37259 -.74827 -.56652 -.49655};
\path svg[xscale=1.7363cm,yscale=-1.7363cm] {M -.2217 -1 Q 0 -1 .2217 -1};
\path svg[xscale=1.7363cm,yscale=-1.7363cm] {M 1 0 Q .63029 -.24827 .56652 -.49655};
\path svg[xscale=1.7363cm,yscale=-1.7363cm] {M -.78584 -1 Q -.66235 -.74827 -.56652 -.49655};
\path svg[xscale=1.7363cm,yscale=-1.7363cm] {M .56652 -1.5035 Q .66235 -1.2517 .78584 -1};
\path svg[xscale=1.7363cm,yscale=-1.7363cm] {M -1 0 Q -.63029 -.24827 -.56652 -.49655};
\path svg[xscale=1.7363cm,yscale=-1.7363cm] {M -.56652 -1.5035 Q -.66235 -1.2517 -.78584 -1};
\path svg[xscale=1.7363cm,yscale=-1.7363cm] {M -.56652 -1.5035 Q -.37259 -1.2517 -.2217 -1};
\path svg[xscale=1.7363cm,yscale=-1.7363cm] {M .56652 -.49655 Q .66235 -.74827 .78584 -1};
\path svg[xscale=1.7363cm,yscale=-1.7363cm] {M -1 -1 Q -.89292 -1 -.78584 -1};
\path svg[xscale=1.7363cm,yscale=-1.7363cm] {M .2217 -1 Q .37259 -.74827 .56652 -.49655};
\path svg[xscale=1.7363cm,yscale=-1.7363cm] {M .2217 -1 Q .37259 -1.2517 .56652 -1.5035};
\path svg[xscale=1.7363cm,yscale=-1.7363cm] {M -1 -2 Q -.63029 -1.7517 -.56652 -1.5035};
\end{scope}
\path[draw=blue,line width=.01405cm] (-1,.4) -- (-1,-2.4) -- (1,-2.4) -- (1,.4) -- cycle;
\end{scope}
\end{scope}
\end{tikzpicture}
}
\expandafter\def\csname e8diagram:H1K2H1\endcsname{%
\begin{tikzpicture}[execute at begin picture={\clip (-1.24,-2.736) rectangle ++(2.48,3.472);},x={(1.7363cm,0cm)},y={(0cm,-1.7363cm)},baseline={([yshift=-\the\dimexpr\fontdimen22\textfont2\relax]current  bounding  box.center)},line join=round]
\begin{scope}
\begin{scope}
\path[clip] (-1,.4) -- (-1,-2.4) -- (1,-2.4) -- (1,.4) -- cycle;
\path[fill={rgb,255:red,204;green,204;blue,255},draw=none] (-1,.4) -- (-1,-2.4) -- (1,-2.4) -- (1,.4) -- cycle;
\begin{scope}[every path/.append style={draw=red,line width=.05619cm}]
\path svg[xscale=1.7363cm,yscale=-1.7363cm] {M -.5818 -1 Q -.44854 -.70783 -.3341 -.41566};
\path svg[xscale=1.7363cm,yscale=-1.7363cm] {M -1 0 Q -.50421 -.20783 -.3341 -.41566};
\path svg[xscale=1.7363cm,yscale=-1.7363cm] {M -.3341 -.41566 Q 0 -.41566 .3341 -.41566};
\path svg[xscale=1.7363cm,yscale=-1.7363cm] {M -.3341 -1.5843 Q -.44854 -1.2922 -.5818 -1};
\path svg[xscale=1.7363cm,yscale=-1.7363cm] {M -.3341 -1.5843 Q 0 -1.5843 .3341 -1.5843};
\path svg[xscale=1.7363cm,yscale=-1.7363cm] {M 1 -1 Q .7909 -1 .5818 -1};
\path svg[xscale=1.7363cm,yscale=-1.7363cm] {M -1 -1 Q -.7909 -1 -.5818 -1};
\path svg[xscale=1.7363cm,yscale=-1.7363cm] {M .3341 -1.5843 Q .44854 -1.2922 .5818 -1};
\path svg[xscale=1.7363cm,yscale=-1.7363cm] {M .5818 -1 Q .44854 -.70783 .3341 -.41566};
\path svg[xscale=1.7363cm,yscale=-1.7363cm] {M 1 -2 Q .50421 -1.7922 .3341 -1.5843};
\path svg[xscale=1.7363cm,yscale=-1.7363cm] {M -1 -2 Q -.50421 -1.7922 -.3341 -1.5843};
\path svg[xscale=1.7363cm,yscale=-1.7363cm] {M 1 0 Q .50421 -.20783 .3341 -.41566};
\end{scope}
\path[draw=blue,line width=.01405cm] (-1,.4) -- (-1,-2.4) -- (1,-2.4) -- (1,.4) -- cycle;
\end{scope}
\end{scope}
\end{tikzpicture}
}
\expandafter\def\csname e8diagram:H1K2H12\endcsname{%
\begin{tikzpicture}[execute at begin picture={\clip (-1.24,-2.736) rectangle ++(2.48,3.472);},x={(1.7363cm,0cm)},y={(0cm,-1.7363cm)},baseline={([yshift=-\the\dimexpr\fontdimen22\textfont2\relax]current  bounding  box.center)},line join=round]
\begin{scope}
\begin{scope}
\path[clip] (-1,.4) -- (-1,-2.4) -- (1,-2.4) -- (1,.4) -- cycle;
\path[fill={rgb,255:red,204;green,204;blue,255},draw=none] (-1,.4) -- (-1,-2.4) -- (1,-2.4) -- (1,.4) -- cycle;
\begin{scope}[every path/.append style={draw=red,line width=.05619cm}]
\path svg[xscale=1.7363cm,yscale=-1.7363cm] {M 1 -2 Q .68022 -1.8347 .59162 -1.6694};
\path svg[xscale=1.7363cm,yscale=-1.7363cm] {M -.43889 -.35921 Q -.10328 -.34773 .23877 -.33625};
\path svg[xscale=1.7363cm,yscale=-1.7363cm] {M 1 -1 Q .87412 -1.0812 .75171 -1.1625};
\path svg[xscale=1.7363cm,yscale=-1.7363cm] {M .08262 -1.3733 Q .19963 -1.1287 .32057 -.8841};
\path svg[xscale=1.7363cm,yscale=-1.7363cm] {M .08262 -1.3733 Q .31071 -1.5213 .59162 -1.6694};
\path svg[xscale=1.7363cm,yscale=-1.7363cm] {M 1 0 Q .46387 -.16813 .23877 -.33625};
\path svg[xscale=1.7363cm,yscale=-1.7363cm] {M -.66016 -.95368 Q -.53491 -.65644 -.43889 -.35921};
\path svg[xscale=1.7363cm,yscale=-1.7363cm] {M -1 0 Q -.58044 -.17961 -.43889 -.35921};
\path svg[xscale=1.7363cm,yscale=-1.7363cm] {M -.4466 -1.5402 Q -.53804 -1.2469 -.66016 -.95368};
\path svg[xscale=1.7363cm,yscale=-1.7363cm] {M -.4466 -1.5402 Q -.17111 -1.4568 .08262 -1.3733};
\path svg[xscale=1.7363cm,yscale=-1.7363cm] {M .59162 -1.6694 Q .6551 -1.4159 .75171 -1.1625};
\path svg[xscale=1.7363cm,yscale=-1.7363cm] {M -1 -1 Q -.82997 -.97684 -.66016 -.95368};
\path svg[xscale=1.7363cm,yscale=-1.7363cm] {M .32057 -.8841 Q .27227 -.61018 .23877 -.33625};
\path svg[xscale=1.7363cm,yscale=-1.7363cm] {M -1 -2 Q -.56413 -1.7701 -.4466 -1.5402};
\path svg[xscale=1.7363cm,yscale=-1.7363cm] {M .32057 -.8841 Q .53124 -1.0233 .75171 -1.1625};
\end{scope}
\path[draw=blue,line width=.01405cm] (-1,.4) -- (-1,-2.4) -- (1,-2.4) -- (1,.4) -- cycle;
\end{scope}
\end{scope}
\end{tikzpicture}
}
\expandafter\def\csname e8diagram:H1K2H1H2\endcsname{%
\begin{tikzpicture}[execute at begin picture={\clip (-1.24,-2.736) rectangle ++(2.48,3.472);},x={(1.7363cm,0cm)},y={(0cm,-1.7363cm)},baseline={([yshift=-\the\dimexpr\fontdimen22\textfont2\relax]current  bounding  box.center)},line join=round]
\begin{scope}
\begin{scope}
\path[clip] (-1,.4) -- (-1,-2.4) -- (1,-2.4) -- (1,.4) -- cycle;
\path[fill={rgb,255:red,204;green,204;blue,255},draw=none] (-1,.4) -- (-1,-2.4) -- (1,-2.4) -- (1,.4) -- cycle;
\begin{scope}[every path/.append style={draw=red,line width=.05619cm}]
\path svg[xscale=1.7363cm,yscale=-1.7363cm] {M -.4466 -.45979 Q -.17111 -.54325 .08262 -.6267};
\path svg[xscale=1.7363cm,yscale=-1.7363cm] {M 1 -1 Q .87412 -.91876 .75171 -.83751};
\path svg[xscale=1.7363cm,yscale=-1.7363cm] {M .23877 -1.6637 Q .27227 -1.3898 .32057 -1.1159};
\path svg[xscale=1.7363cm,yscale=-1.7363cm] {M 1 0 Q .68022 -.16531 .59162 -.33062};
\path svg[xscale=1.7363cm,yscale=-1.7363cm] {M -.66016 -1.0463 Q -.53804 -.75306 -.4466 -.45979};
\path svg[xscale=1.7363cm,yscale=-1.7363cm] {M .75171 -.83751 Q .6551 -.58406 .59162 -.33062};
\path svg[xscale=1.7363cm,yscale=-1.7363cm] {M -1 0 Q -.56413 -.2299 -.4466 -.45979};
\path svg[xscale=1.7363cm,yscale=-1.7363cm] {M -.43889 -1.6408 Q -.53491 -1.3436 -.66016 -1.0463};
\path svg[xscale=1.7363cm,yscale=-1.7363cm] {M -.43889 -1.6408 Q -.10328 -1.6523 .23877 -1.6637};
\path svg[xscale=1.7363cm,yscale=-1.7363cm] {M .08262 -.6267 Q .31071 -.47866 .59162 -.33062};
\path svg[xscale=1.7363cm,yscale=-1.7363cm] {M -1 -1 Q -.82997 -1.0232 -.66016 -1.0463};
\path svg[xscale=1.7363cm,yscale=-1.7363cm] {M 1 -2 Q .46387 -1.8319 .23877 -1.6637};
\path svg[xscale=1.7363cm,yscale=-1.7363cm] {M .32057 -1.1159 Q .19963 -.8713 .08262 -.6267};
\path svg[xscale=1.7363cm,yscale=-1.7363cm] {M .32057 -1.1159 Q .53124 -.97671 .75171 -.83751};
\path svg[xscale=1.7363cm,yscale=-1.7363cm] {M -1 -2 Q -.58044 -1.8204 -.43889 -1.6408};
\end{scope}
\path[draw=blue,line width=.01405cm] (-1,.4) -- (-1,-2.4) -- (1,-2.4) -- (1,.4) -- cycle;
\end{scope}
\end{scope}
\end{tikzpicture}
}
\expandafter\def\csname e8diagram:H1K2K1\endcsname{%
\begin{tikzpicture}[execute at begin picture={\clip (-1.24,-2.736) rectangle ++(2.48,3.472);},x={(1.7363cm,0cm)},y={(0cm,-1.7363cm)},baseline={([yshift=-\the\dimexpr\fontdimen22\textfont2\relax]current  bounding  box.center)},line join=round]
\begin{scope}
\begin{scope}
\path[clip] (-1,.4) -- (-1,-2.4) -- (1,-2.4) -- (1,.4) -- cycle;
\path[fill={rgb,255:red,204;green,204;blue,255},draw=none] (-1,.4) -- (-1,-2.4) -- (1,-2.4) -- (1,.4) -- cycle;
\begin{scope}[every path/.append style={draw=red,line width=.05619cm}]
\path svg[xscale=1.7363cm,yscale=-1.7363cm] {M -.64034 -.94024 Q -.52441 -.66229 -.43293 -.38434};
\path svg[xscale=1.7363cm,yscale=-1.7363cm] {M 1 -2 Q .68301 -1.7151 .67513 -1.4303};
\path svg[xscale=1.7363cm,yscale=-1.7363cm] {M -.43293 -.38434 Q -.12238 -.40456 .17946 -.42479};
\path svg[xscale=1.7363cm,yscale=-1.7363cm] {M -1 0 Q -.57119 -.19217 -.43293 -.38434};
\path svg[xscale=1.7363cm,yscale=-1.7363cm] {M -.36479 -1.4537 Q -.49384 -1.197 -.64034 -.94024};
\path svg[xscale=1.7363cm,yscale=-1.7363cm] {M -.36479 -1.4537 Q -.09121 -1.2753 .15631 -1.0969};
\path svg[xscale=1.7363cm,yscale=-1.7363cm] {M 1 -1 Q .82664 -1.2151 .67513 -1.4303};
\path svg[xscale=1.7363cm,yscale=-1.7363cm] {M -1 -1 Q -.82 -.97012 -.64034 -.94024};
\path svg[xscale=1.7363cm,yscale=-1.7363cm] {M .15631 -1.0969 Q .15812 -.76084 .17946 -.42479};
\path svg[xscale=1.7363cm,yscale=-1.7363cm] {M .15631 -1.0969 Q .40018 -1.2636 .67513 -1.4303};
\path svg[xscale=1.7363cm,yscale=-1.7363cm] {M 1 0 Q .40794 -.2124 .17946 -.42479};
\path svg[xscale=1.7363cm,yscale=-1.7363cm] {M -1 -2 Q -.49889 -1.7269 -.36479 -1.4537};
\end{scope}
\path[draw=blue,line width=.01405cm] (-1,.4) -- (-1,-2.4) -- (1,-2.4) -- (1,.4) -- cycle;
\end{scope}
\end{scope}
\end{tikzpicture}
}
\expandafter\def\csname e8diagram:H1K2K1H2\endcsname{%
\begin{tikzpicture}[execute at begin picture={\clip (-1.24,-2.736) rectangle ++(2.48,3.472);},x={(1.7363cm,0cm)},y={(0cm,-1.7363cm)},baseline={([yshift=-\the\dimexpr\fontdimen22\textfont2\relax]current  bounding  box.center)},line join=round]
\begin{scope}
\begin{scope}
\path[clip] (-1,.4) -- (-1,-2.4) -- (1,-2.4) -- (1,.4) -- cycle;
\path[fill={rgb,255:red,204;green,204;blue,255},draw=none] (-1,.4) -- (-1,-2.4) -- (1,-2.4) -- (1,.4) -- cycle;
\begin{scope}[every path/.append style={draw=red,line width=.05619cm}]
\path svg[xscale=1.7363cm,yscale=-1.7363cm] {M 1 -2 Q .59956 -1.7768 .50039 -1.5536};
\path svg[xscale=1.7363cm,yscale=-1.7363cm] {M -.50039 -.44637 Q -.23377 -.58162 0 -.71687};
\path svg[xscale=1.7363cm,yscale=-1.7363cm] {M 1 -1 Q .85852 -1 .71705 -1};
\path svg[xscale=1.7363cm,yscale=-1.7363cm] {M 0 -1.2831 Q 0 -1 0 -.71687};
\path svg[xscale=1.7363cm,yscale=-1.7363cm] {M 0 -1.2831 Q .23377 -1.4184 .50039 -1.5536};
\path svg[xscale=1.7363cm,yscale=-1.7363cm] {M 1 0 Q .59956 -.22318 .50039 -.44637};
\path svg[xscale=1.7363cm,yscale=-1.7363cm] {M -.71705 -1 Q -.59383 -.72318 -.50039 -.44637};
\path svg[xscale=1.7363cm,yscale=-1.7363cm] {M .50039 -.44637 Q .59383 -.72318 .71705 -1};
\path svg[xscale=1.7363cm,yscale=-1.7363cm] {M -1 0 Q -.59956 -.22318 -.50039 -.44637};
\path svg[xscale=1.7363cm,yscale=-1.7363cm] {M -.50039 -1.5536 Q -.59383 -1.2768 -.71705 -1};
\path svg[xscale=1.7363cm,yscale=-1.7363cm] {M -.50039 -1.5536 Q -.23377 -1.4184 0 -1.2831};
\path svg[xscale=1.7363cm,yscale=-1.7363cm] {M .50039 -1.5536 Q .59383 -1.2768 .71705 -1};
\path svg[xscale=1.7363cm,yscale=-1.7363cm] {M -1 -1 Q -.85852 -1 -.71705 -1};
\path svg[xscale=1.7363cm,yscale=-1.7363cm] {M 0 -.71687 Q .23377 -.58162 .50039 -.44637};
\path svg[xscale=1.7363cm,yscale=-1.7363cm] {M -1 -2 Q -.59956 -1.7768 -.50039 -1.5536};
\end{scope}
\path[draw=blue,line width=.01405cm] (-1,.4) -- (-1,-2.4) -- (1,-2.4) -- (1,.4) -- cycle;
\end{scope}
\end{scope}
\end{tikzpicture}
}
\expandafter\def\csname e8diagram:H2\endcsname{%
\begin{tikzpicture}[execute at begin picture={\clip (-1.24,-2.736) rectangle ++(2.48,3.472);},x={(1.7363cm,0cm)},y={(0cm,-1.7363cm)},baseline={([yshift=-\the\dimexpr\fontdimen22\textfont2\relax]current  bounding  box.center)},line join=round]
\begin{scope}
\begin{scope}
\path[clip] (-1,.4) -- (-1,-2.4) -- (1,-2.4) -- (1,.4) -- cycle;
\path[fill={rgb,255:red,204;green,204;blue,255},draw=none] (-1,.4) -- (-1,-2.4) -- (1,-2.4) -- (1,.4) -- cycle;
\begin{scope}[every path/.append style={draw=red,line width=.05619cm}]
\path svg[xscale=1.7363cm,yscale=-1.7363cm] {M -1 0 Q -.36243 -.09982 0 -.19964};
\path svg[xscale=1.7363cm,yscale=-1.7363cm] {M 1 -1 Q .50204 -.90087 0 -.80175};
\path svg[xscale=1.7363cm,yscale=-1.7363cm] {M 1 0 Q .36243 -.09982 0 -.19964};
\path svg[xscale=1.7363cm,yscale=-1.7363cm] {M 0 -.80175 Q 0 -.5007 0 -.19964};
\path svg[xscale=1.7363cm,yscale=-1.7363cm] {M -1 -2 Q 0 -2 1 -2};
\path svg[xscale=1.7363cm,yscale=-1.7363cm] {M -1 -1 Q -.50204 -.90087 0 -.80175};
\end{scope}
\path[draw=blue,line width=.01405cm] (-1,.4) -- (-1,-2.4) -- (1,-2.4) -- (1,.4) -- cycle;
\end{scope}
\end{scope}
\end{tikzpicture}
}
\expandafter\def\csname e8diagram:H22\endcsname{%
\begin{tikzpicture}[execute at begin picture={\clip (-1.24,-2.736) rectangle ++(2.48,3.472);},x={(1.7363cm,0cm)},y={(0cm,-1.7363cm)},baseline={([yshift=-\the\dimexpr\fontdimen22\textfont2\relax]current  bounding  box.center)},line join=round]
\begin{scope}
\begin{scope}
\path[clip] (-1,.4) -- (-1,-2.4) -- (1,-2.4) -- (1,.4) -- cycle;
\path[fill={rgb,255:red,204;green,204;blue,255},draw=none] (-1,.4) -- (-1,-2.4) -- (1,-2.4) -- (1,.4) -- cycle;
\begin{scope}[every path/.append style={draw=red,line width=.05619cm}]
\path svg[xscale=1.7363cm,yscale=-1.7363cm] {M -.31926 -.26156 Q 0 -.26156 .31926 -.26156};
\path svg[xscale=1.7363cm,yscale=-1.7363cm] {M -1 0 Q -.53294 -.13078 -.31926 -.26156};
\path svg[xscale=1.7363cm,yscale=-1.7363cm] {M 1 -1 Q .6977 -.87734 .3967 -.75468};
\path svg[xscale=1.7363cm,yscale=-1.7363cm] {M -.3967 -.75468 Q -.34844 -.50812 -.31926 -.26156};
\path svg[xscale=1.7363cm,yscale=-1.7363cm] {M -.3967 -.75468 Q 0 -.75468 .3967 -.75468};
\path svg[xscale=1.7363cm,yscale=-1.7363cm] {M -1 -1 Q -.6977 -.87734 -.3967 -.75468};
\path svg[xscale=1.7363cm,yscale=-1.7363cm] {M .3967 -.75468 Q .34844 -.50812 .31926 -.26156};
\path svg[xscale=1.7363cm,yscale=-1.7363cm] {M 1 0 Q .53294 -.13078 .31926 -.26156};
\path svg[xscale=1.7363cm,yscale=-1.7363cm] {M -1 -2 Q 0 -2 1 -2};
\end{scope}
\path[draw=blue,line width=.01405cm] (-1,.4) -- (-1,-2.4) -- (1,-2.4) -- (1,.4) -- cycle;
\end{scope}
\end{scope}
\end{tikzpicture}
}
\expandafter\def\csname e8diagram:H22H1\endcsname{%
\begin{tikzpicture}[execute at begin picture={\clip (-1.24,-2.736) rectangle ++(2.48,3.472);},x={(1.7363cm,0cm)},y={(0cm,-1.7363cm)},baseline={([yshift=-\the\dimexpr\fontdimen22\textfont2\relax]current  bounding  box.center)},line join=round]
\begin{scope}
\begin{scope}
\path[clip] (-1,.4) -- (-1,-2.4) -- (1,-2.4) -- (1,.4) -- cycle;
\path[fill={rgb,255:red,204;green,204;blue,255},draw=none] (-1,.4) -- (-1,-2.4) -- (1,-2.4) -- (1,.4) -- cycle;
\begin{scope}[every path/.append style={draw=red,line width=.05619cm}]
\path svg[xscale=1.7363cm,yscale=-1.7363cm] {M 1 -1 Q .73608 -1.1007 .47398 -1.2014};
\path svg[xscale=1.7363cm,yscale=-1.7363cm] {M -.57638 -.73962 Q -.5101 -.49089 -.4751 -.24216};
\path svg[xscale=1.7363cm,yscale=-1.7363cm] {M .18013 -1.8211 Q .33082 -1.5113 .47398 -1.2014};
\path svg[xscale=1.7363cm,yscale=-1.7363cm] {M -1 0 Q -.63275 -.12108 -.4751 -.24216};
\path svg[xscale=1.7363cm,yscale=-1.7363cm] {M .04069 -.79513 Q -.26603 -.76737 -.57638 -.73962};
\path svg[xscale=1.7363cm,yscale=-1.7363cm] {M .04069 -.79513 Q .10116 -.51926 .19562 -.24339};
\path svg[xscale=1.7363cm,yscale=-1.7363cm] {M -1 -1 Q -.78552 -.86981 -.57638 -.73962};
\path svg[xscale=1.7363cm,yscale=-1.7363cm] {M 1 -2 Q .47809 -1.9106 .18013 -1.8211};
\path svg[xscale=1.7363cm,yscale=-1.7363cm] {M .19562 -.24339 Q -.1395 -.24278 -.4751 -.24216};
\path svg[xscale=1.7363cm,yscale=-1.7363cm] {M 1 0 Q .46391 -.12169 .19562 -.24339};
\path svg[xscale=1.7363cm,yscale=-1.7363cm] {M -1 -2 Q -.26445 -1.9106 .18013 -1.8211};
\path svg[xscale=1.7363cm,yscale=-1.7363cm] {M .47398 -1.2014 Q .25297 -.99829 .04069 -.79513};
\end{scope}
\path[draw=blue,line width=.01405cm] (-1,.4) -- (-1,-2.4) -- (1,-2.4) -- (1,.4) -- cycle;
\end{scope}
\end{scope}
\end{tikzpicture}
}
\expandafter\def\csname e8diagram:H22H12\endcsname{%
\begin{tikzpicture}[execute at begin picture={\clip (-1.24,-2.736) rectangle ++(2.48,3.472);},x={(1.7363cm,0cm)},y={(0cm,-1.7363cm)},baseline={([yshift=-\the\dimexpr\fontdimen22\textfont2\relax]current  bounding  box.center)},line join=round]
\begin{scope}
\begin{scope}
\path[clip] (-1,.4) -- (-1,-2.4) -- (1,-2.4) -- (1,.4) -- cycle;
\path[fill={rgb,255:red,204;green,204;blue,255},draw=none] (-1,.4) -- (-1,-2.4) -- (1,-2.4) -- (1,.4) -- cycle;
\begin{scope}[every path/.append style={draw=red,line width=.05619cm}]
\path svg[xscale=1.7363cm,yscale=-1.7363cm] {M -.11799 -.78024 Q -.39419 -.74521 -.67519 -.71017};
\path svg[xscale=1.7363cm,yscale=-1.7363cm] {M -.11799 -.78024 Q -.01057 -.50289 .13625 -.22554};
\path svg[xscale=1.7363cm,yscale=-1.7363cm] {M 1 -1 Q .83293 -1.1449 .67519 -1.2898};
\path svg[xscale=1.7363cm,yscale=-1.7363cm] {M .56505 -1.7883 Q .59965 -1.5391 .67519 -1.2898};
\path svg[xscale=1.7363cm,yscale=-1.7363cm] {M 1 0 Q .43421 -.11277 .13625 -.22554};
\path svg[xscale=1.7363cm,yscale=-1.7363cm] {M .11799 -1.2198 Q 0 -1 -.11799 -.78024};
\path svg[xscale=1.7363cm,yscale=-1.7363cm] {M .13625 -.22554 Q -.21121 -.21861 -.56505 -.21167};
\path svg[xscale=1.7363cm,yscale=-1.7363cm] {M .11799 -1.2198 Q .39419 -1.2548 .67519 -1.2898};
\path svg[xscale=1.7363cm,yscale=-1.7363cm] {M -.13625 -1.7745 Q .01057 -1.4971 .11799 -1.2198};
\path svg[xscale=1.7363cm,yscale=-1.7363cm] {M -.13625 -1.7745 Q .21121 -1.7814 .56505 -1.7883};
\path svg[xscale=1.7363cm,yscale=-1.7363cm] {M -1 0 Q -.69593 -.10584 -.56505 -.21167};
\path svg[xscale=1.7363cm,yscale=-1.7363cm] {M -.67519 -.71017 Q -.59965 -.46092 -.56505 -.21167};
\path svg[xscale=1.7363cm,yscale=-1.7363cm] {M -1 -1 Q -.83293 -.85509 -.67519 -.71017};
\path svg[xscale=1.7363cm,yscale=-1.7363cm] {M 1 -2 Q .69593 -1.8942 .56505 -1.7883};
\path svg[xscale=1.7363cm,yscale=-1.7363cm] {M -1 -2 Q -.43421 -1.8872 -.13625 -1.7745};
\end{scope}
\path[draw=blue,line width=.01405cm] (-1,.4) -- (-1,-2.4) -- (1,-2.4) -- (1,.4) -- cycle;
\end{scope}
\end{scope}
\end{tikzpicture}
}
\expandafter\def\csname e8diagram:H22H1H2\endcsname{%
\begin{tikzpicture}[execute at begin picture={\clip (-1.24,-2.736) rectangle ++(2.48,3.472);},x={(1.7363cm,0cm)},y={(0cm,-1.7363cm)},baseline={([yshift=-\the\dimexpr\fontdimen22\textfont2\relax]current  bounding  box.center)},line join=round]
\begin{scope}
\begin{scope}
\path[clip] (-1,.4) -- (-1,-2.4) -- (1,-2.4) -- (1,.4) -- cycle;
\path[fill={rgb,255:red,204;green,204;blue,255},draw=none] (-1,.4) -- (-1,-2.4) -- (1,-2.4) -- (1,.4) -- cycle;
\begin{scope}[every path/.append style={draw=red,line width=.05619cm}]
\path svg[xscale=1.7363cm,yscale=-1.7363cm] {M 1 -1 Q .82298 -.95066 .64692 -.90132};
\path svg[xscale=1.7363cm,yscale=-1.7363cm] {M .64692 -.90132 Q .55976 -.62558 .50626 -.34983};
\path svg[xscale=1.7363cm,yscale=-1.7363cm] {M -.20203 -.9373 Q -.457 -.86326 -.71853 -.78922};
\path svg[xscale=1.7363cm,yscale=-1.7363cm] {M -.20203 -.9373 Q -.12262 -.70023 -.03925 -.46315};
\path svg[xscale=1.7363cm,yscale=-1.7363cm] {M 1 0 Q .62369 -.17492 .50626 -.34983};
\path svg[xscale=1.7363cm,yscale=-1.7363cm] {M .17323 -1.2805 Q -.01718 -1.1089 -.20203 -.9373};
\path svg[xscale=1.7363cm,yscale=-1.7363cm] {M .17323 -1.2805 Q .4073 -1.0909 .64692 -.90132};
\path svg[xscale=1.7363cm,yscale=-1.7363cm] {M .065 -1.877 Q .12156 -1.5787 .17323 -1.2805};
\path svg[xscale=1.7363cm,yscale=-1.7363cm] {M -.03925 -.46315 Q .22216 -.40649 .50626 -.34983};
\path svg[xscale=1.7363cm,yscale=-1.7363cm] {M -1 0 Q -.71249 -.13425 -.61923 -.2685};
\path svg[xscale=1.7363cm,yscale=-1.7363cm] {M -.03925 -.46315 Q -.30396 -.36582 -.61923 -.2685};
\path svg[xscale=1.7363cm,yscale=-1.7363cm] {M 1 -2 Q .43694 -1.9385 .065 -1.877};
\path svg[xscale=1.7363cm,yscale=-1.7363cm] {M -1 -1 Q -.85655 -.89461 -.71853 -.78922};
\path svg[xscale=1.7363cm,yscale=-1.7363cm] {M -1 -2 Q -.36172 -1.9385 .065 -1.877};
\path svg[xscale=1.7363cm,yscale=-1.7363cm] {M -.71853 -.78922 Q -.6461 -.52886 -.61923 -.2685};
\end{scope}
\path[draw=blue,line width=.01405cm] (-1,.4) -- (-1,-2.4) -- (1,-2.4) -- (1,.4) -- cycle;
\end{scope}
\end{scope}
\end{tikzpicture}
}
\expandafter\def\csname e8diagram:H22K1\endcsname{%
\begin{tikzpicture}[execute at begin picture={\clip (-1.24,-2.736) rectangle ++(2.48,3.472);},x={(1.7363cm,0cm)},y={(0cm,-1.7363cm)},baseline={([yshift=-\the\dimexpr\fontdimen22\textfont2\relax]current  bounding  box.center)},line join=round]
\begin{scope}
\begin{scope}
\path[clip] (-1,.4) -- (-1,-2.4) -- (1,-2.4) -- (1,.4) -- cycle;
\path[fill={rgb,255:red,204;green,204;blue,255},draw=none] (-1,.4) -- (-1,-2.4) -- (1,-2.4) -- (1,.4) -- cycle;
\begin{scope}[every path/.append style={draw=red,line width=.05619cm}]
\path svg[xscale=1.7363cm,yscale=-1.7363cm] {M -.1748 -.93918 Q -.41829 -.85261 -.66983 -.76603};
\path svg[xscale=1.7363cm,yscale=-1.7363cm] {M -.1748 -.93918 Q -.05239 -.63421 .09552 -.32924};
\path svg[xscale=1.7363cm,yscale=-1.7363cm] {M -.66983 -.76603 Q -.58355 -.52045 -.52641 -.27488};
\path svg[xscale=1.7363cm,yscale=-1.7363cm] {M 1 -1 Q .80616 -1.2705 .6464 -1.541};
\path svg[xscale=1.7363cm,yscale=-1.7363cm] {M -.0732 -1.5526 Q -.122 -1.2459 -.1748 -.93918};
\path svg[xscale=1.7363cm,yscale=-1.7363cm] {M -1 0 Q -.65459 -.13744 -.52641 -.27488};
\path svg[xscale=1.7363cm,yscale=-1.7363cm] {M -.0732 -1.5526 Q .28784 -1.5468 .6464 -1.541};
\path svg[xscale=1.7363cm,yscale=-1.7363cm] {M -1 -1 Q -.83197 -.88301 -.66983 -.76603};
\path svg[xscale=1.7363cm,yscale=-1.7363cm] {M .09552 -.32924 Q -.20709 -.30206 -.52641 -.27488};
\path svg[xscale=1.7363cm,yscale=-1.7363cm] {M 1 0 Q .37836 -.16462 .09552 -.32924};
\path svg[xscale=1.7363cm,yscale=-1.7363cm] {M 1 -2 Q .68616 -1.7705 .6464 -1.541};
\path svg[xscale=1.7363cm,yscale=-1.7363cm] {M -1 -2 Q -.33867 -1.7763 -.0732 -1.5526};
\end{scope}
\path[draw=blue,line width=.01405cm] (-1,.4) -- (-1,-2.4) -- (1,-2.4) -- (1,.4) -- cycle;
\end{scope}
\end{scope}
\end{tikzpicture}
}
\expandafter\def\csname e8diagram:H2E1\endcsname{%
\begin{tikzpicture}[execute at begin picture={\clip (-1.24,-2.736) rectangle ++(2.48,3.472);},x={(1.7363cm,0cm)},y={(0cm,-1.7363cm)},baseline={([yshift=-\the\dimexpr\fontdimen22\textfont2\relax]current  bounding  box.center)},line join=round]
\begin{scope}
\begin{scope}
\path[clip] (-1,.4) -- (-1,-2.4) -- (1,-2.4) -- (1,.4) -- cycle;
\path[fill={rgb,255:red,204;green,204;blue,255},draw=none] (-1,.4) -- (-1,-2.4) -- (1,-2.4) -- (1,.4) -- cycle;
\begin{scope}[every path/.append style={draw=red,line width=.05619cm}]
\path svg[xscale=1.7363cm,yscale=-1.7363cm] {M 1 -1 Q .29519 -1.5 1 -2};
\path svg[xscale=1.7363cm,yscale=-1.7363cm] {M -1 0 Q -.43399 -.18082 -.19976 -.36165};
\path svg[xscale=1.7363cm,yscale=-1.7363cm] {M 1 0 Q .19023 -.18082 -.19976 -.36165};
\path svg[xscale=1.7363cm,yscale=-1.7363cm] {M -.54935 -1.1565 Q -.36201 -.75908 -.19976 -.36165};
\path svg[xscale=1.7363cm,yscale=-1.7363cm] {M -1 -2 Q -.56447 -1.5783 -.54935 -1.1565};
\path svg[xscale=1.7363cm,yscale=-1.7363cm] {M -1 -1 Q -.77384 -1.0783 -.54935 -1.1565};
\end{scope}
\path[draw=blue,line width=.01405cm] (-1,.4) -- (-1,-2.4) -- (1,-2.4) -- (1,.4) -- cycle;
\end{scope}
\end{scope}
\end{tikzpicture}
}
\expandafter\def\csname e8diagram:H2H1\endcsname{%
\begin{tikzpicture}[execute at begin picture={\clip (-1.24,-2.736) rectangle ++(2.48,3.472);},x={(1.7363cm,0cm)},y={(0cm,-1.7363cm)},baseline={([yshift=-\the\dimexpr\fontdimen22\textfont2\relax]current  bounding  box.center)},line join=round]
\begin{scope}
\begin{scope}
\path[clip] (-1,.4) -- (-1,-2.4) -- (1,-2.4) -- (1,.4) -- cycle;
\path[fill={rgb,255:red,204;green,204;blue,255},draw=none] (-1,.4) -- (-1,-2.4) -- (1,-2.4) -- (1,.4) -- cycle;
\begin{scope}[every path/.append style={draw=red,line width=.05619cm}]
\path svg[xscale=1.7363cm,yscale=-1.7363cm] {M -1 0 Q -.43453 -.09904 -.11978 -.19807};
\path svg[xscale=1.7363cm,yscale=-1.7363cm] {M 1 -1 Q .66488 -1.0909 .32975 -1.1818};
\path svg[xscale=1.7363cm,yscale=-1.7363cm] {M .11978 -1.8019 Q .22762 -1.4919 .32975 -1.1818};
\path svg[xscale=1.7363cm,yscale=-1.7363cm] {M -.32975 -.81817 Q -.22762 -.50812 -.11978 -.19807};
\path svg[xscale=1.7363cm,yscale=-1.7363cm] {M -1 -1 Q -.66488 -.90908 -.32975 -.81817};
\path svg[xscale=1.7363cm,yscale=-1.7363cm] {M 1 -2 Q .43453 -1.901 .11978 -1.8019};
\path svg[xscale=1.7363cm,yscale=-1.7363cm] {M 1 0 Q .29162 -.09904 -.11978 -.19807};
\path svg[xscale=1.7363cm,yscale=-1.7363cm] {M -1 -2 Q -.29162 -1.901 .11978 -1.8019};
\path svg[xscale=1.7363cm,yscale=-1.7363cm] {M .32975 -1.1818 Q 0 -1 -.32975 -.81817};
\end{scope}
\path[draw=blue,line width=.01405cm] (-1,.4) -- (-1,-2.4) -- (1,-2.4) -- (1,.4) -- cycle;
\end{scope}
\end{scope}
\end{tikzpicture}
}
\expandafter\def\csname e8diagram:H2H12\endcsname{%
\begin{tikzpicture}[execute at begin picture={\clip (-1.24,-2.736) rectangle ++(2.48,3.472);},x={(1.7363cm,0cm)},y={(0cm,-1.7363cm)},baseline={([yshift=-\the\dimexpr\fontdimen22\textfont2\relax]current  bounding  box.center)},line join=round]
\begin{scope}
\begin{scope}
\path[clip] (-1,.4) -- (-1,-2.4) -- (1,-2.4) -- (1,.4) -- cycle;
\path[fill={rgb,255:red,204;green,204;blue,255},draw=none] (-1,.4) -- (-1,-2.4) -- (1,-2.4) -- (1,.4) -- cycle;
\begin{scope}[every path/.append style={draw=red,line width=.05619cm}]
\path svg[xscale=1.7363cm,yscale=-1.7363cm] {M -.04069 -1.2049 Q -.25297 -1.0017 -.47398 -.79856};
\path svg[xscale=1.7363cm,yscale=-1.7363cm] {M -.04069 -1.2049 Q .26603 -1.2326 .57638 -1.2604};
\path svg[xscale=1.7363cm,yscale=-1.7363cm] {M -.47398 -.79856 Q -.33082 -.48871 -.18013 -.17887};
\path svg[xscale=1.7363cm,yscale=-1.7363cm] {M -.19562 -1.7566 Q -.10116 -1.4807 -.04069 -1.2049};
\path svg[xscale=1.7363cm,yscale=-1.7363cm] {M -1 0 Q -.47809 -.08944 -.18013 -.17887};
\path svg[xscale=1.7363cm,yscale=-1.7363cm] {M -.19562 -1.7566 Q .1395 -1.7572 .4751 -1.7578};
\path svg[xscale=1.7363cm,yscale=-1.7363cm] {M 1 -1 Q .78552 -1.1302 .57638 -1.2604};
\path svg[xscale=1.7363cm,yscale=-1.7363cm] {M .4751 -1.7578 Q .5101 -1.5091 .57638 -1.2604};
\path svg[xscale=1.7363cm,yscale=-1.7363cm] {M -1 -1 Q -.73608 -.89928 -.47398 -.79856};
\path svg[xscale=1.7363cm,yscale=-1.7363cm] {M 1 -2 Q .63275 -1.8789 .4751 -1.7578};
\path svg[xscale=1.7363cm,yscale=-1.7363cm] {M -1 -2 Q -.46391 -1.8783 -.19562 -1.7566};
\path svg[xscale=1.7363cm,yscale=-1.7363cm] {M 1 0 Q .26445 -.08944 -.18013 -.17887};
\end{scope}
\path[draw=blue,line width=.01405cm] (-1,.4) -- (-1,-2.4) -- (1,-2.4) -- (1,.4) -- cycle;
\end{scope}
\end{scope}
\end{tikzpicture}
}
\expandafter\def\csname e8diagram:H2H12H2\endcsname{%
\begin{tikzpicture}[execute at begin picture={\clip (-1.24,-2.736) rectangle ++(2.48,3.472);},x={(1.7363cm,0cm)},y={(0cm,-1.7363cm)},baseline={([yshift=-\the\dimexpr\fontdimen22\textfont2\relax]current  bounding  box.center)},line join=round]
\begin{scope}
\begin{scope}
\path[clip] (-1,.4) -- (-1,-2.4) -- (1,-2.4) -- (1,.4) -- cycle;
\path[fill={rgb,255:red,204;green,204;blue,255},draw=none] (-1,.4) -- (-1,-2.4) -- (1,-2.4) -- (1,.4) -- cycle;
\begin{scope}[every path/.append style={draw=red,line width=.05619cm}]
\path svg[xscale=1.7363cm,yscale=-1.7363cm] {M -.60905 -.86473 Q -.47996 -.59228 -.36533 -.31982};
\path svg[xscale=1.7363cm,yscale=-1.7363cm] {M 1 -1 Q .80373 -.93237 .60905 -.86473};
\path svg[xscale=1.7363cm,yscale=-1.7363cm] {M .36484 -1.815 Q .28658 -1.5422 .26162 -1.2694};
\path svg[xscale=1.7363cm,yscale=-1.7363cm] {M 1 0 Q .54505 -.15991 .36533 -.31982};
\path svg[xscale=1.7363cm,yscale=-1.7363cm] {M -.26162 -1.2694 Q -.42984 -1.0671 -.60905 -.86473};
\path svg[xscale=1.7363cm,yscale=-1.7363cm] {M .60905 -.86473 Q .47996 -.59228 .36533 -.31982};
\path svg[xscale=1.7363cm,yscale=-1.7363cm] {M -.26162 -1.2694 Q 0 -1.2694 .26162 -1.2694};
\path svg[xscale=1.7363cm,yscale=-1.7363cm] {M -1 0 Q -.54505 -.15991 -.36533 -.31982};
\path svg[xscale=1.7363cm,yscale=-1.7363cm] {M -.36484 -1.815 Q -.28658 -1.5422 -.26162 -1.2694};
\path svg[xscale=1.7363cm,yscale=-1.7363cm] {M -.36484 -1.815 Q 0 -1.815 .36484 -1.815};
\path svg[xscale=1.7363cm,yscale=-1.7363cm] {M -.36533 -.31982 Q 0 -.31982 .36533 -.31982};
\path svg[xscale=1.7363cm,yscale=-1.7363cm] {M -1 -1 Q -.80373 -.93237 -.60905 -.86473};
\path svg[xscale=1.7363cm,yscale=-1.7363cm] {M 1 -2 Q .58536 -1.9075 .36484 -1.815};
\path svg[xscale=1.7363cm,yscale=-1.7363cm] {M .26162 -1.2694 Q .42984 -1.0671 .60905 -.86473};
\path svg[xscale=1.7363cm,yscale=-1.7363cm] {M -1 -2 Q -.58536 -1.9075 -.36484 -1.815};
\end{scope}
\path[draw=blue,line width=.01405cm] (-1,.4) -- (-1,-2.4) -- (1,-2.4) -- (1,.4) -- cycle;
\end{scope}
\end{scope}
\end{tikzpicture}
}
\expandafter\def\csname e8diagram:H2H1H2\endcsname{%
\begin{tikzpicture}[execute at begin picture={\clip (-1.24,-2.736) rectangle ++(2.48,3.472);},x={(1.7363cm,0cm)},y={(0cm,-1.7363cm)},baseline={([yshift=-\the\dimexpr\fontdimen22\textfont2\relax]current  bounding  box.center)},line join=round]
\begin{scope}
\begin{scope}
\path[clip] (-1,.4) -- (-1,-2.4) -- (1,-2.4) -- (1,.4) -- cycle;
\path[fill={rgb,255:red,204;green,204;blue,255},draw=none] (-1,.4) -- (-1,-2.4) -- (1,-2.4) -- (1,.4) -- cycle;
\begin{scope}[every path/.append style={draw=red,line width=.05619cm}]
\path svg[xscale=1.7363cm,yscale=-1.7363cm] {M 0 -1.2103 Q .25999 -1.04 .52332 -.86968};
\path svg[xscale=1.7363cm,yscale=-1.7363cm] {M .52332 -.86968 Q .42322 -.59993 .33892 -.33018};
\path svg[xscale=1.7363cm,yscale=-1.7363cm] {M -1 0 Q -.52639 -.16509 -.33892 -.33018};
\path svg[xscale=1.7363cm,yscale=-1.7363cm] {M 0 -1.8368 Q 0 -1.5236 0 -1.2103};
\path svg[xscale=1.7363cm,yscale=-1.7363cm] {M 1 -1 Q .76115 -.93484 .52332 -.86968};
\path svg[xscale=1.7363cm,yscale=-1.7363cm] {M -1 -1 Q -.76115 -.93484 -.52332 -.86968};
\path svg[xscale=1.7363cm,yscale=-1.7363cm] {M -.52332 -.86968 Q -.42322 -.59993 -.33892 -.33018};
\path svg[xscale=1.7363cm,yscale=-1.7363cm] {M 1 -2 Q .37846 -1.9184 0 -1.8368};
\path svg[xscale=1.7363cm,yscale=-1.7363cm] {M -.33892 -.33018 Q 0 -.33018 .33892 -.33018};
\path svg[xscale=1.7363cm,yscale=-1.7363cm] {M -1 -2 Q -.37846 -1.9184 0 -1.8368};
\path svg[xscale=1.7363cm,yscale=-1.7363cm] {M 1 0 Q .52639 -.16509 .33892 -.33018};
\path svg[xscale=1.7363cm,yscale=-1.7363cm] {M 0 -1.2103 Q -.25999 -1.04 -.52332 -.86968};
\end{scope}
\path[draw=blue,line width=.01405cm] (-1,.4) -- (-1,-2.4) -- (1,-2.4) -- (1,.4) -- cycle;
\end{scope}
\end{scope}
\end{tikzpicture}
}
\expandafter\def\csname e8diagram:H2H1H22\endcsname{%
\begin{tikzpicture}[execute at begin picture={\clip (-1.24,-2.736) rectangle ++(2.48,3.472);},x={(1.7363cm,0cm)},y={(0cm,-1.7363cm)},baseline={([yshift=-\the\dimexpr\fontdimen22\textfont2\relax]current  bounding  box.center)},line join=round]
\begin{scope}
\begin{scope}
\path[clip] (-1,.4) -- (-1,-2.4) -- (1,-2.4) -- (1,.4) -- cycle;
\path[fill={rgb,255:red,204;green,204;blue,255},draw=none] (-1,.4) -- (-1,-2.4) -- (1,-2.4) -- (1,.4) -- cycle;
\begin{scope}[every path/.append style={draw=red,line width=.05619cm}]
\path svg[xscale=1.7363cm,yscale=-1.7363cm] {M .20203 -.9373 Q .12262 -.70023 .03925 -.46315};
\path svg[xscale=1.7363cm,yscale=-1.7363cm] {M .20203 -.9373 Q .457 -.86326 .71853 -.78922};
\path svg[xscale=1.7363cm,yscale=-1.7363cm] {M 1 -1 Q .85655 -.89461 .71853 -.78922};
\path svg[xscale=1.7363cm,yscale=-1.7363cm] {M -.64692 -.90132 Q -.55976 -.62558 -.50626 -.34983};
\path svg[xscale=1.7363cm,yscale=-1.7363cm] {M -.17323 -1.2805 Q -.4073 -1.0909 -.64692 -.90132};
\path svg[xscale=1.7363cm,yscale=-1.7363cm] {M 1 0 Q .71249 -.13425 .61923 -.2685};
\path svg[xscale=1.7363cm,yscale=-1.7363cm] {M -.17323 -1.2805 Q .01718 -1.1089 .20203 -.9373};
\path svg[xscale=1.7363cm,yscale=-1.7363cm] {M .71853 -.78922 Q .6461 -.52886 .61923 -.2685};
\path svg[xscale=1.7363cm,yscale=-1.7363cm] {M -1 0 Q -.62369 -.17492 -.50626 -.34983};
\path svg[xscale=1.7363cm,yscale=-1.7363cm] {M -.065 -1.877 Q -.12156 -1.5787 -.17323 -1.2805};
\path svg[xscale=1.7363cm,yscale=-1.7363cm] {M .03925 -.46315 Q .30396 -.36582 .61923 -.2685};
\path svg[xscale=1.7363cm,yscale=-1.7363cm] {M -1 -1 Q -.82298 -.95066 -.64692 -.90132};
\path svg[xscale=1.7363cm,yscale=-1.7363cm] {M 1 -2 Q .36172 -1.9385 -.065 -1.877};
\path svg[xscale=1.7363cm,yscale=-1.7363cm] {M -.50626 -.34983 Q -.22216 -.40649 .03925 -.46315};
\path svg[xscale=1.7363cm,yscale=-1.7363cm] {M -1 -2 Q -.43694 -1.9385 -.065 -1.877};
\end{scope}
\path[draw=blue,line width=.01405cm] (-1,.4) -- (-1,-2.4) -- (1,-2.4) -- (1,.4) -- cycle;
\end{scope}
\end{scope}
\end{tikzpicture}
}
\expandafter\def\csname e8diagram:H2H1H2H1\endcsname{%
\begin{tikzpicture}[execute at begin picture={\clip (-1.24,-2.736) rectangle ++(2.48,3.472);},x={(1.7363cm,0cm)},y={(0cm,-1.7363cm)},baseline={([yshift=-\the\dimexpr\fontdimen22\textfont2\relax]current  bounding  box.center)},line join=round]
\begin{scope}
\begin{scope}
\path[clip] (-1,.4) -- (-1,-2.4) -- (1,-2.4) -- (1,.4) -- cycle;
\path[fill={rgb,255:red,204;green,204;blue,255},draw=none] (-1,.4) -- (-1,-2.4) -- (1,-2.4) -- (1,.4) -- cycle;
\begin{scope}[every path/.append style={draw=red,line width=.05619cm}]
\path svg[xscale=1.7363cm,yscale=-1.7363cm] {M -.64829 -.81928 Q -.54435 -.55041 -.46766 -.28153};
\path svg[xscale=1.7363cm,yscale=-1.7363cm] {M 1 -1 Q .82252 -1.0904 .64829 -1.1807};
\path svg[xscale=1.7363cm,yscale=-1.7363cm] {M .46766 -1.7185 Q .54435 -1.4496 .64829 -1.1807};
\path svg[xscale=1.7363cm,yscale=-1.7363cm] {M 1 0 Q .47115 -.14185 .22785 -.2837};
\path svg[xscale=1.7363cm,yscale=-1.7363cm] {M -.22114 -1.1499 Q -.42932 -.98459 -.64829 -.81928};
\path svg[xscale=1.7363cm,yscale=-1.7363cm] {M -.46766 -.28153 Q -.11951 -.28262 .22785 -.2837};
\path svg[xscale=1.7363cm,yscale=-1.7363cm] {M -.22114 -1.1499 Q 0 -1 .22114 -.8501};
\path svg[xscale=1.7363cm,yscale=-1.7363cm] {M -.22785 -1.7163 Q -.21322 -1.4331 -.22114 -1.1499};
\path svg[xscale=1.7363cm,yscale=-1.7363cm] {M -1 0 Q -.61719 -.14076 -.46766 -.28153};
\path svg[xscale=1.7363cm,yscale=-1.7363cm] {M -.22785 -1.7163 Q .11951 -1.7174 .46766 -1.7185};
\path svg[xscale=1.7363cm,yscale=-1.7363cm] {M -1 -1 Q -.82252 -.90964 -.64829 -.81928};
\path svg[xscale=1.7363cm,yscale=-1.7363cm] {M .22114 -.8501 Q .42932 -1.0154 .64829 -1.1807};
\path svg[xscale=1.7363cm,yscale=-1.7363cm] {M 1 -2 Q .61719 -1.8592 .46766 -1.7185};
\path svg[xscale=1.7363cm,yscale=-1.7363cm] {M -1 -2 Q -.47115 -1.8581 -.22785 -1.7163};
\path svg[xscale=1.7363cm,yscale=-1.7363cm] {M .22114 -.8501 Q .21322 -.5669 .22785 -.2837};
\end{scope}
\path[draw=blue,line width=.01405cm] (-1,.4) -- (-1,-2.4) -- (1,-2.4) -- (1,.4) -- cycle;
\end{scope}
\end{scope}
\end{tikzpicture}
}
\expandafter\def\csname e8diagram:H2H1K2\endcsname{%
\begin{tikzpicture}[execute at begin picture={\clip (-1.24,-2.736) rectangle ++(2.48,3.472);},x={(1.7363cm,0cm)},y={(0cm,-1.7363cm)},baseline={([yshift=-\the\dimexpr\fontdimen22\textfont2\relax]current  bounding  box.center)},line join=round]
\begin{scope}
\begin{scope}
\path[clip] (-1,.4) -- (-1,-2.4) -- (1,-2.4) -- (1,.4) -- cycle;
\path[fill={rgb,255:red,204;green,204;blue,255},draw=none] (-1,.4) -- (-1,-2.4) -- (1,-2.4) -- (1,.4) -- cycle;
\begin{scope}[every path/.append style={draw=red,line width=.05619cm}]
\path svg[xscale=1.7363cm,yscale=-1.7363cm] {M -.19782 -1.2121 Q -.06481 -.95999 .06697 -.70785};
\path svg[xscale=1.7363cm,yscale=-1.7363cm] {M .06697 -.70785 Q -.16983 -.55699 -.44605 -.40612};
\path svg[xscale=1.7363cm,yscale=-1.7363cm] {M .06697 -.70785 Q .35162 -.61974 .65793 -.53163};
\path svg[xscale=1.7363cm,yscale=-1.7363cm] {M -1 0 Q -.57455 -.20306 -.44605 -.40612};
\path svg[xscale=1.7363cm,yscale=-1.7363cm] {M -.06789 -1.8233 Q -.13551 -1.5177 -.19782 -1.2121};
\path svg[xscale=1.7363cm,yscale=-1.7363cm] {M 1 -1 Q .81637 -.76582 .65793 -.53163};
\path svg[xscale=1.7363cm,yscale=-1.7363cm] {M -1 -1 Q -.82307 -.947 -.64724 -.894};
\path svg[xscale=1.7363cm,yscale=-1.7363cm] {M -.64724 -.894 Q -.53778 -.65006 -.44605 -.40612};
\path svg[xscale=1.7363cm,yscale=-1.7363cm] {M 1 -2 Q .33202 -1.9117 -.06789 -1.8233};
\path svg[xscale=1.7363cm,yscale=-1.7363cm] {M -1 -2 Q -.41248 -1.9117 -.06789 -1.8233};
\path svg[xscale=1.7363cm,yscale=-1.7363cm] {M 1 0 Q .67944 -.26582 .65793 -.53163};
\path svg[xscale=1.7363cm,yscale=-1.7363cm] {M -.19782 -1.2121 Q -.41966 -1.0531 -.64724 -.894};
\end{scope}
\path[draw=blue,line width=.01405cm] (-1,.4) -- (-1,-2.4) -- (1,-2.4) -- (1,.4) -- cycle;
\end{scope}
\end{scope}
\end{tikzpicture}
}
\expandafter\def\csname e8diagram:H2K1\endcsname{%
\begin{tikzpicture}[execute at begin picture={\clip (-1.24,-2.736) rectangle ++(2.48,3.472);},x={(1.7363cm,0cm)},y={(0cm,-1.7363cm)},baseline={([yshift=-\the\dimexpr\fontdimen22\textfont2\relax]current  bounding  box.center)},line join=round]
\begin{scope}
\begin{scope}
\path[clip] (-1,.4) -- (-1,-2.4) -- (1,-2.4) -- (1,.4) -- cycle;
\path[fill={rgb,255:red,204;green,204;blue,255},draw=none] (-1,.4) -- (-1,-2.4) -- (1,-2.4) -- (1,.4) -- cycle;
\begin{scope}[every path/.append style={draw=red,line width=.05619cm}]
\path svg[xscale=1.7363cm,yscale=-1.7363cm] {M -.47994 -.91244 Q -.32843 -.58131 -.1763 -.25018};
\path svg[xscale=1.7363cm,yscale=-1.7363cm] {M -1 0 Q -.45007 -.12509 -.1763 -.25018};
\path svg[xscale=1.7363cm,yscale=-1.7363cm] {M 1 -1 Q .77264 -1.2559 .56813 -1.5118};
\path svg[xscale=1.7363cm,yscale=-1.7363cm] {M -.18002 -1.518 Q -.32489 -1.2152 -.47994 -.91244};
\path svg[xscale=1.7363cm,yscale=-1.7363cm] {M -.18002 -1.518 Q .19469 -1.5149 .56813 -1.5118};
\path svg[xscale=1.7363cm,yscale=-1.7363cm] {M -1 -1 Q -.73979 -.95622 -.47994 -.91244};
\path svg[xscale=1.7363cm,yscale=-1.7363cm] {M 1 0 Q .23725 -.12509 -.1763 -.25018};
\path svg[xscale=1.7363cm,yscale=-1.7363cm] {M 1 -2 Q .63283 -1.7559 .56813 -1.5118};
\path svg[xscale=1.7363cm,yscale=-1.7363cm] {M -1 -2 Q -.39748 -1.759 -.18002 -1.518};
\end{scope}
\path[draw=blue,line width=.01405cm] (-1,.4) -- (-1,-2.4) -- (1,-2.4) -- (1,.4) -- cycle;
\end{scope}
\end{scope}
\end{tikzpicture}
}
\expandafter\def\csname e8diagram:H2K1K2\endcsname{%
\begin{tikzpicture}[execute at begin picture={\clip (-1.24,-2.736) rectangle ++(2.48,3.472);},x={(1.7363cm,0cm)},y={(0cm,-1.7363cm)},baseline={([yshift=-\the\dimexpr\fontdimen22\textfont2\relax]current  bounding  box.center)},line join=round]
\begin{scope}
\begin{scope}
\path[clip] (-1,.4) -- (-1,-2.4) -- (1,-2.4) -- (1,.4) -- cycle;
\path[fill={rgb,255:red,204;green,204;blue,255},draw=none] (-1,.4) -- (-1,-2.4) -- (1,-2.4) -- (1,.4) -- cycle;
\begin{scope}[every path/.append style={draw=red,line width=.05619cm}]
\path svg[xscale=1.7363cm,yscale=-1.7363cm] {M -.64034 -1.0598 Q -.49384 -.80301 -.36479 -.54626};
\path svg[xscale=1.7363cm,yscale=-1.7363cm] {M -.36479 -.54626 Q -.09121 -.72469 .15631 -.90312};
\path svg[xscale=1.7363cm,yscale=-1.7363cm] {M -1 0 Q -.49889 -.27313 -.36479 -.54626};
\path svg[xscale=1.7363cm,yscale=-1.7363cm] {M -.43293 -1.6157 Q -.52441 -1.3377 -.64034 -1.0598};
\path svg[xscale=1.7363cm,yscale=-1.7363cm] {M -.43293 -1.6157 Q -.12238 -1.5954 .17946 -1.5752};
\path svg[xscale=1.7363cm,yscale=-1.7363cm] {M 1 -1 Q .82664 -.78487 .67513 -.56974};
\path svg[xscale=1.7363cm,yscale=-1.7363cm] {M -1 -1 Q -.82 -1.0299 -.64034 -1.0598};
\path svg[xscale=1.7363cm,yscale=-1.7363cm] {M .17946 -1.5752 Q .15812 -1.2392 .15631 -.90312};
\path svg[xscale=1.7363cm,yscale=-1.7363cm] {M .15631 -.90312 Q .40018 -.73643 .67513 -.56974};
\path svg[xscale=1.7363cm,yscale=-1.7363cm] {M 1 -2 Q .40794 -1.7876 .17946 -1.5752};
\path svg[xscale=1.7363cm,yscale=-1.7363cm] {M -1 -2 Q -.57119 -1.8078 -.43293 -1.6157};
\path svg[xscale=1.7363cm,yscale=-1.7363cm] {M 1 0 Q .68301 -.28487 .67513 -.56974};
\end{scope}
\path[draw=blue,line width=.01405cm] (-1,.4) -- (-1,-2.4) -- (1,-2.4) -- (1,.4) -- cycle;
\end{scope}
\end{scope}
\end{tikzpicture}
}
\expandafter\def\csname e8diagram:K1\endcsname{%
\begin{tikzpicture}[execute at begin picture={\clip (-1.24,-2.736) rectangle ++(2.48,3.472);},x={(1.7363cm,0cm)},y={(0cm,-1.7363cm)},baseline={([yshift=-\the\dimexpr\fontdimen22\textfont2\relax]current  bounding  box.center)},line join=round]
\begin{scope}
\begin{scope}
\path[clip] (-1,.4) -- (-1,-2.4) -- (1,-2.4) -- (1,.4) -- cycle;
\path[fill={rgb,255:red,204;green,204;blue,255},draw=none] (-1,.4) -- (-1,-2.4) -- (1,-2.4) -- (1,.4) -- cycle;
\begin{scope}[every path/.append style={draw=red,line width=.05619cm}]
\path svg[xscale=1.7363cm,yscale=-1.7363cm] {M 1 -1 Q .71484 -1.2446 .43978 -1.4893};
\path svg[xscale=1.7363cm,yscale=-1.7363cm] {M 1 -2 Q .55057 -1.7446 .43978 -1.4893};
\path svg[xscale=1.7363cm,yscale=-1.7363cm] {M -.43978 -1.4893 Q 0 -1.4893 .43978 -1.4893};
\path svg[xscale=1.7363cm,yscale=-1.7363cm] {M -1 0 Q 0 0 1 0};
\path svg[xscale=1.7363cm,yscale=-1.7363cm] {M -1 -2 Q -.55057 -1.7446 -.43978 -1.4893};
\path svg[xscale=1.7363cm,yscale=-1.7363cm] {M -1 -1 Q -.71484 -1.2446 -.43978 -1.4893};
\end{scope}
\path[draw=blue,line width=.01405cm] (-1,.4) -- (-1,-2.4) -- (1,-2.4) -- (1,.4) -- cycle;
\end{scope}
\end{scope}
\end{tikzpicture}
}
\expandafter\def\csname e8diagram:K1E2\endcsname{%
\begin{tikzpicture}[execute at begin picture={\clip (-1.24,-2.736) rectangle ++(2.48,3.472);},x={(1.7363cm,0cm)},y={(0cm,-1.7363cm)},baseline={([yshift=-\the\dimexpr\fontdimen22\textfont2\relax]current  bounding  box.center)},line join=round]
\begin{scope}
\begin{scope}
\path[clip] (-1,.4) -- (-1,-2.4) -- (1,-2.4) -- (1,.4) -- cycle;
\path[fill={rgb,255:red,204;green,204;blue,255},draw=none] (-1,.4) -- (-1,-2.4) -- (1,-2.4) -- (1,.4) -- cycle;
\begin{scope}[every path/.append style={draw=red,line width=.05619cm}]
\path svg[xscale=1.7363cm,yscale=-1.7363cm] {M -1 0 Q -.30492 -.54777 -.13985 -1.0955};
\path svg[xscale=1.7363cm,yscale=-1.7363cm] {M 1 -2 Q .15987 -1.5478 -.13985 -1.0955};
\path svg[xscale=1.7363cm,yscale=-1.7363cm] {M -.63282 -1.4098 Q -.37327 -1.2527 -.13985 -1.0955};
\path svg[xscale=1.7363cm,yscale=-1.7363cm] {M -1 -2 Q -.65382 -1.7049 -.63282 -1.4098};
\path svg[xscale=1.7363cm,yscale=-1.7363cm] {M 1 0 Q .29519 -.5 1 -1};
\path svg[xscale=1.7363cm,yscale=-1.7363cm] {M -1 -1 Q -.80777 -1.2049 -.63282 -1.4098};
\end{scope}
\path[draw=blue,line width=.01405cm] (-1,.4) -- (-1,-2.4) -- (1,-2.4) -- (1,.4) -- cycle;
\end{scope}
\end{scope}
\end{tikzpicture}
}
\expandafter\def\csname e8diagram:K1E2H1\endcsname{%
\begin{tikzpicture}[execute at begin picture={\clip (-1.24,-2.736) rectangle ++(2.48,3.472);},x={(1.7363cm,0cm)},y={(0cm,-1.7363cm)},baseline={([yshift=-\the\dimexpr\fontdimen22\textfont2\relax]current  bounding  box.center)},line join=round]
\begin{scope}
\begin{scope}
\path[clip] (-1,.4) -- (-1,-2.4) -- (1,-2.4) -- (1,.4) -- cycle;
\path[fill={rgb,255:red,204;green,204;blue,255},draw=none] (-1,.4) -- (-1,-2.4) -- (1,-2.4) -- (1,.4) -- cycle;
\begin{scope}[every path/.append style={draw=red,line width=.05619cm}]
\path svg[xscale=1.7363cm,yscale=-1.7363cm] {M -1 0 Q -.41782 -.38365 -.28192 -.7673};
\path svg[xscale=1.7363cm,yscale=-1.7363cm] {M -.65481 -1.3275 Q -.45074 -1.0474 -.28192 -.7673};
\path svg[xscale=1.7363cm,yscale=-1.7363cm] {M 1 -1 Q .82171 -.83623 .65481 -.67246};
\path svg[xscale=1.7363cm,yscale=-1.7363cm] {M -1 -1 Q -.82171 -1.1638 -.65481 -1.3275};
\path svg[xscale=1.7363cm,yscale=-1.7363cm] {M .28192 -1.2327 Q .45074 -.95258 .65481 -.67246};
\path svg[xscale=1.7363cm,yscale=-1.7363cm] {M 1 0 Q .65317 -.33623 .65481 -.67246};
\path svg[xscale=1.7363cm,yscale=-1.7363cm] {M 1 -2 Q .41782 -1.6164 .28192 -1.2327};
\path svg[xscale=1.7363cm,yscale=-1.7363cm] {M -1 -2 Q -.65317 -1.6638 -.65481 -1.3275};
\path svg[xscale=1.7363cm,yscale=-1.7363cm] {M -.28192 -.7673 Q 0 -1 .28192 -1.2327};
\end{scope}
\path[draw=blue,line width=.01405cm] (-1,.4) -- (-1,-2.4) -- (1,-2.4) -- (1,.4) -- cycle;
\end{scope}
\end{scope}
\end{tikzpicture}
}
\expandafter\def\csname e8diagram:K1E2K1\endcsname{%
\begin{tikzpicture}[execute at begin picture={\clip (-1.24,-2.736) rectangle ++(2.48,3.472);},x={(1.7363cm,0cm)},y={(0cm,-1.7363cm)},baseline={([yshift=-\the\dimexpr\fontdimen22\textfont2\relax]current  bounding  box.center)},line join=round]
\begin{scope}
\begin{scope}
\path[clip] (-1,.4) -- (-1,-2.4) -- (1,-2.4) -- (1,.4) -- cycle;
\path[fill={rgb,255:red,204;green,204;blue,255},draw=none] (-1,.4) -- (-1,-2.4) -- (1,-2.4) -- (1,.4) -- cycle;
\begin{scope}[every path/.append style={draw=red,line width=.05619cm}]
\path svg[xscale=1.7363cm,yscale=-1.7363cm] {M -1 0 Q -.46663 -.28147 -.31607 -.56294};
\path svg[xscale=1.7363cm,yscale=-1.7363cm] {M -.62458 -1.2532 Q -.4505 -.90806 -.31607 -.56294};
\path svg[xscale=1.7363cm,yscale=-1.7363cm] {M 1 -1 Q .80928 -1.1266 .62458 -1.2532};
\path svg[xscale=1.7363cm,yscale=-1.7363cm] {M -1 -1 Q -.80928 -1.1266 -.62458 -1.2532};
\path svg[xscale=1.7363cm,yscale=-1.7363cm] {M .31607 -.56294 Q .4505 -.90806 .62458 -1.2532};
\path svg[xscale=1.7363cm,yscale=-1.7363cm] {M 1 0 Q .46663 -.28147 .31607 -.56294};
\path svg[xscale=1.7363cm,yscale=-1.7363cm] {M -1 -2 Q -.62319 -1.6266 -.62458 -1.2532};
\path svg[xscale=1.7363cm,yscale=-1.7363cm] {M 1 -2 Q .62319 -1.6266 .62458 -1.2532};
\path svg[xscale=1.7363cm,yscale=-1.7363cm] {M -.31607 -.56294 Q 0 -.56294 .31607 -.56294};
\end{scope}
\path[draw=blue,line width=.01405cm] (-1,.4) -- (-1,-2.4) -- (1,-2.4) -- (1,.4) -- cycle;
\end{scope}
\end{scope}
\end{tikzpicture}
}
\expandafter\def\csname e8diagram:K1E2K1H2\endcsname{%
\begin{tikzpicture}[execute at begin picture={\clip (-1.24,-2.736) rectangle ++(2.48,3.472);},x={(1.7363cm,0cm)},y={(0cm,-1.7363cm)},baseline={([yshift=-\the\dimexpr\fontdimen22\textfont2\relax]current  bounding  box.center)},line join=round]
\begin{scope}
\begin{scope}
\path[clip] (-1,.4) -- (-1,-2.4) -- (1,-2.4) -- (1,.4) -- cycle;
\path[fill={rgb,255:red,204;green,204;blue,255},draw=none] (-1,.4) -- (-1,-2.4) -- (1,-2.4) -- (1,.4) -- cycle;
\begin{scope}[every path/.append style={draw=red,line width=.05619cm}]
\path svg[xscale=1.7363cm,yscale=-1.7363cm] {M -1 0 Q -.49859 -.31682 -.38704 -.63364};
\path svg[xscale=1.7363cm,yscale=-1.7363cm] {M .53776 -.45977 Q .62349 -.71341 .73579 -.96706};
\path svg[xscale=1.7363cm,yscale=-1.7363cm] {M -.68691 -1.301 Q -.51304 -.96731 -.38704 -.63364};
\path svg[xscale=1.7363cm,yscale=-1.7363cm] {M -1 -1 Q -.83823 -1.1505 -.68691 -1.301};
\path svg[xscale=1.7363cm,yscale=-1.7363cm] {M 1 -1 Q .86783 -.98353 .73579 -.96706};
\path svg[xscale=1.7363cm,yscale=-1.7363cm] {M .12957 -.85894 Q .3139 -.65936 .53776 -.45977};
\path svg[xscale=1.7363cm,yscale=-1.7363cm] {M .12957 -.85894 Q .28316 -1.1403 .46904 -1.4217};
\path svg[xscale=1.7363cm,yscale=-1.7363cm] {M 1 0 Q .61974 -.22989 .53776 -.45977};
\path svg[xscale=1.7363cm,yscale=-1.7363cm] {M .46904 -1.4217 Q .59368 -1.1944 .73579 -.96706};
\path svg[xscale=1.7363cm,yscale=-1.7363cm] {M -1 -2 Q -.66796 -1.6505 -.68691 -1.301};
\path svg[xscale=1.7363cm,yscale=-1.7363cm] {M -.38704 -.63364 Q -.1215 -.74629 .12957 -.85894};
\path svg[xscale=1.7363cm,yscale=-1.7363cm] {M 1 -2 Q .55679 -1.7109 .46904 -1.4217};
\end{scope}
\path[draw=blue,line width=.01405cm] (-1,.4) -- (-1,-2.4) -- (1,-2.4) -- (1,.4) -- cycle;
\end{scope}
\end{scope}
\end{tikzpicture}
}
\expandafter\def\csname e8diagram:K1H2\endcsname{%
\begin{tikzpicture}[execute at begin picture={\clip (-1.24,-2.736) rectangle ++(2.48,3.472);},x={(1.7363cm,0cm)},y={(0cm,-1.7363cm)},baseline={([yshift=-\the\dimexpr\fontdimen22\textfont2\relax]current  bounding  box.center)},line join=round]
\begin{scope}
\begin{scope}
\path[clip] (-1,.4) -- (-1,-2.4) -- (1,-2.4) -- (1,.4) -- cycle;
\path[fill={rgb,255:red,204;green,204;blue,255},draw=none] (-1,.4) -- (-1,-2.4) -- (1,-2.4) -- (1,.4) -- cycle;
\begin{scope}[every path/.append style={draw=red,line width=.05619cm}]
\path svg[xscale=1.7363cm,yscale=-1.7363cm] {M -1 0 Q -.23725 -.12509 .1763 -.25018};
\path svg[xscale=1.7363cm,yscale=-1.7363cm] {M 1 -1 Q .73979 -.95622 .47994 -.91244};
\path svg[xscale=1.7363cm,yscale=-1.7363cm] {M -.56813 -1.5118 Q -.19469 -1.5149 .18002 -1.518};
\path svg[xscale=1.7363cm,yscale=-1.7363cm] {M -1 -1 Q -.77264 -1.2559 -.56813 -1.5118};
\path svg[xscale=1.7363cm,yscale=-1.7363cm] {M .47994 -.91244 Q .32843 -.58131 .1763 -.25018};
\path svg[xscale=1.7363cm,yscale=-1.7363cm] {M 1 0 Q .45007 -.12509 .1763 -.25018};
\path svg[xscale=1.7363cm,yscale=-1.7363cm] {M 1 -2 Q .39748 -1.759 .18002 -1.518};
\path svg[xscale=1.7363cm,yscale=-1.7363cm] {M -1 -2 Q -.63283 -1.7559 -.56813 -1.5118};
\path svg[xscale=1.7363cm,yscale=-1.7363cm] {M .18002 -1.518 Q .32489 -1.2152 .47994 -.91244};
\end{scope}
\path[draw=blue,line width=.01405cm] (-1,.4) -- (-1,-2.4) -- (1,-2.4) -- (1,.4) -- cycle;
\end{scope}
\end{scope}
\end{tikzpicture}
}
\expandafter\def\csname e8diagram:K1H22\endcsname{%
\begin{tikzpicture}[execute at begin picture={\clip (-1.24,-2.736) rectangle ++(2.48,3.472);},x={(1.7363cm,0cm)},y={(0cm,-1.7363cm)},baseline={([yshift=-\the\dimexpr\fontdimen22\textfont2\relax]current  bounding  box.center)},line join=round]
\begin{scope}
\begin{scope}
\path[clip] (-1,.4) -- (-1,-2.4) -- (1,-2.4) -- (1,.4) -- cycle;
\path[fill={rgb,255:red,204;green,204;blue,255},draw=none] (-1,.4) -- (-1,-2.4) -- (1,-2.4) -- (1,.4) -- cycle;
\begin{scope}[every path/.append style={draw=red,line width=.05619cm}]
\path svg[xscale=1.7363cm,yscale=-1.7363cm] {M -1 0 Q -.37836 -.16462 -.09552 -.32924};
\path svg[xscale=1.7363cm,yscale=-1.7363cm] {M -.09552 -.32924 Q .20709 -.30206 .52641 -.27488};
\path svg[xscale=1.7363cm,yscale=-1.7363cm] {M -.6464 -1.541 Q -.28784 -1.5468 .0732 -1.5526};
\path svg[xscale=1.7363cm,yscale=-1.7363cm] {M 1 -1 Q .83197 -.88301 .66983 -.76603};
\path svg[xscale=1.7363cm,yscale=-1.7363cm] {M -1 -1 Q -.80616 -1.2705 -.6464 -1.541};
\path svg[xscale=1.7363cm,yscale=-1.7363cm] {M .1748 -.93918 Q .05239 -.63421 -.09552 -.32924};
\path svg[xscale=1.7363cm,yscale=-1.7363cm] {M .1748 -.93918 Q .41829 -.85261 .66983 -.76603};
\path svg[xscale=1.7363cm,yscale=-1.7363cm] {M 1 -2 Q .33867 -1.7763 .0732 -1.5526};
\path svg[xscale=1.7363cm,yscale=-1.7363cm] {M .66983 -.76603 Q .58355 -.52045 .52641 -.27488};
\path svg[xscale=1.7363cm,yscale=-1.7363cm] {M -1 -2 Q -.68616 -1.7705 -.6464 -1.541};
\path svg[xscale=1.7363cm,yscale=-1.7363cm] {M 1 0 Q .65459 -.13744 .52641 -.27488};
\path svg[xscale=1.7363cm,yscale=-1.7363cm] {M .0732 -1.5526 Q .122 -1.2459 .1748 -.93918};
\end{scope}
\path[draw=blue,line width=.01405cm] (-1,.4) -- (-1,-2.4) -- (1,-2.4) -- (1,.4) -- cycle;
\end{scope}
\end{scope}
\end{tikzpicture}
}
\expandafter\def\csname e8diagram:K1H2H1\endcsname{%
\begin{tikzpicture}[execute at begin picture={\clip (-1.24,-2.736) rectangle ++(2.48,3.472);},x={(1.7363cm,0cm)},y={(0cm,-1.7363cm)},baseline={([yshift=-\the\dimexpr\fontdimen22\textfont2\relax]current  bounding  box.center)},line join=round]
\begin{scope}
\begin{scope}
\path[clip] (-1,.4) -- (-1,-2.4) -- (1,-2.4) -- (1,.4) -- cycle;
\path[fill={rgb,255:red,204;green,204;blue,255},draw=none] (-1,.4) -- (-1,-2.4) -- (1,-2.4) -- (1,.4) -- cycle;
\begin{scope}[every path/.append style={draw=red,line width=.05619cm}]
\path svg[xscale=1.7363cm,yscale=-1.7363cm] {M -.06697 -1.2921 Q .16983 -1.443 .44605 -1.5939};
\path svg[xscale=1.7363cm,yscale=-1.7363cm] {M .44605 -1.5939 Q .53778 -1.3499 .64724 -1.106};
\path svg[xscale=1.7363cm,yscale=-1.7363cm] {M -1 0 Q -.33202 -.08833 .06789 -.17666};
\path svg[xscale=1.7363cm,yscale=-1.7363cm] {M -.65793 -1.4684 Q -.35162 -1.3803 -.06697 -1.2921};
\path svg[xscale=1.7363cm,yscale=-1.7363cm] {M -1 -1 Q -.81637 -1.2342 -.65793 -1.4684};
\path svg[xscale=1.7363cm,yscale=-1.7363cm] {M 1 -1 Q .82307 -1.053 .64724 -1.106};
\path svg[xscale=1.7363cm,yscale=-1.7363cm] {M .19782 -.78787 Q .13551 -.48227 .06789 -.17666};
\path svg[xscale=1.7363cm,yscale=-1.7363cm] {M .19782 -.78787 Q .41966 -.94693 .64724 -1.106};
\path svg[xscale=1.7363cm,yscale=-1.7363cm] {M 1 0 Q .41248 -.08833 .06789 -.17666};
\path svg[xscale=1.7363cm,yscale=-1.7363cm] {M -1 -2 Q -.67944 -1.7342 -.65793 -1.4684};
\path svg[xscale=1.7363cm,yscale=-1.7363cm] {M 1 -2 Q .57455 -1.7969 .44605 -1.5939};
\path svg[xscale=1.7363cm,yscale=-1.7363cm] {M -.06697 -1.2921 Q .06481 -1.04 .19782 -.78787};
\end{scope}
\path[draw=blue,line width=.01405cm] (-1,.4) -- (-1,-2.4) -- (1,-2.4) -- (1,.4) -- cycle;
\end{scope}
\end{scope}
\end{tikzpicture}
}
\expandafter\def\csname e8diagram:K1H2K1\endcsname{%
\begin{tikzpicture}[execute at begin picture={\clip (-1.24,-2.736) rectangle ++(2.48,3.472);},x={(1.7363cm,0cm)},y={(0cm,-1.7363cm)},baseline={([yshift=-\the\dimexpr\fontdimen22\textfont2\relax]current  bounding  box.center)},line join=round]
\begin{scope}
\begin{scope}
\path[clip] (-1,.4) -- (-1,-2.4) -- (1,-2.4) -- (1,.4) -- cycle;
\path[fill={rgb,255:red,204;green,204;blue,255},draw=none] (-1,.4) -- (-1,-2.4) -- (1,-2.4) -- (1,.4) -- cycle;
\begin{scope}[every path/.append style={draw=red,line width=.05619cm}]
\path svg[xscale=1.7363cm,yscale=-1.7363cm] {M -.23819 -1.1512 Q 0 -1.1512 .23819 -1.1512};
\path svg[xscale=1.7363cm,yscale=-1.7363cm] {M 1 -2 Q .70164 -1.7227 .7014 -1.4454};
\path svg[xscale=1.7363cm,yscale=-1.7363cm] {M .23819 -1.1512 Q .45571 -1.2983 .7014 -1.4454};
\path svg[xscale=1.7363cm,yscale=-1.7363cm] {M -1 0 Q -.39217 -.06799 0 -.13598};
\path svg[xscale=1.7363cm,yscale=-1.7363cm] {M -.7014 -1.4454 Q -.45571 -1.2983 -.23819 -1.1512};
\path svg[xscale=1.7363cm,yscale=-1.7363cm] {M -1 -1 Q -.83802 -1.2227 -.7014 -1.4454};
\path svg[xscale=1.7363cm,yscale=-1.7363cm] {M 1 -1 Q .83802 -1.2227 .7014 -1.4454};
\path svg[xscale=1.7363cm,yscale=-1.7363cm] {M 0 -.68246 Q .1183 -.91681 .23819 -1.1512};
\path svg[xscale=1.7363cm,yscale=-1.7363cm] {M 0 -.68246 Q 0 -.40922 0 -.13598};
\path svg[xscale=1.7363cm,yscale=-1.7363cm] {M 1 0 Q .39217 -.06799 0 -.13598};
\path svg[xscale=1.7363cm,yscale=-1.7363cm] {M -1 -2 Q -.70164 -1.7227 -.7014 -1.4454};
\path svg[xscale=1.7363cm,yscale=-1.7363cm] {M -.23819 -1.1512 Q -.1183 -.91681 0 -.68246};
\end{scope}
\path[draw=blue,line width=.01405cm] (-1,.4) -- (-1,-2.4) -- (1,-2.4) -- (1,.4) -- cycle;
\end{scope}
\end{scope}
\end{tikzpicture}
}
\expandafter\def\csname e8diagram:K1K2\endcsname{%
\begin{tikzpicture}[execute at begin picture={\clip (-1.24,-2.736) rectangle ++(2.48,3.472);},x={(1.7363cm,0cm)},y={(0cm,-1.7363cm)},baseline={([yshift=-\the\dimexpr\fontdimen22\textfont2\relax]current  bounding  box.center)},line join=round]
\begin{scope}
\begin{scope}
\path[clip] (-1,.4) -- (-1,-2.4) -- (1,-2.4) -- (1,.4) -- cycle;
\path[fill={rgb,255:red,204;green,204;blue,255},draw=none] (-1,.4) -- (-1,-2.4) -- (1,-2.4) -- (1,.4) -- cycle;
\begin{scope}[every path/.append style={draw=red,line width=.05619cm}]
\path svg[xscale=1.7363cm,yscale=-1.7363cm] {M -1 0 Q -.27888 -.29796 -.01375 -.59593};
\path svg[xscale=1.7363cm,yscale=-1.7363cm] {M -.62893 -1.493 Q -.30087 -1.4485 .01375 -1.4041};
\path svg[xscale=1.7363cm,yscale=-1.7363cm] {M 1 -1 Q .8015 -.75351 .62893 -.50703};
\path svg[xscale=1.7363cm,yscale=-1.7363cm] {M -1 -1 Q -.8015 -1.2465 -.62893 -1.493};
\path svg[xscale=1.7363cm,yscale=-1.7363cm] {M -.01375 -.59593 Q .30087 -.55148 .62893 -.50703};
\path svg[xscale=1.7363cm,yscale=-1.7363cm] {M 1 0 Q .66634 -.25351 .62893 -.50703};
\path svg[xscale=1.7363cm,yscale=-1.7363cm] {M 1 -2 Q .27888 -1.702 .01375 -1.4041};
\path svg[xscale=1.7363cm,yscale=-1.7363cm] {M -1 -2 Q -.66634 -1.7465 -.62893 -1.493};
\path svg[xscale=1.7363cm,yscale=-1.7363cm] {M .01375 -1.4041 Q 0 -1 -.01375 -.59593};
\end{scope}
\path[draw=blue,line width=.01405cm] (-1,.4) -- (-1,-2.4) -- (1,-2.4) -- (1,.4) -- cycle;
\end{scope}
\end{scope}
\end{tikzpicture}
}
\expandafter\def\csname e8diagram:K1K2H1\endcsname{%
\begin{tikzpicture}[execute at begin picture={\clip (-1.24,-2.736) rectangle ++(2.48,3.472);},x={(1.7363cm,0cm)},y={(0cm,-1.7363cm)},baseline={([yshift=-\the\dimexpr\fontdimen22\textfont2\relax]current  bounding  box.center)},line join=round]
\begin{scope}
\begin{scope}
\path[clip] (-1,.4) -- (-1,-2.4) -- (1,-2.4) -- (1,.4) -- cycle;
\path[fill={rgb,255:red,204;green,204;blue,255},draw=none] (-1,.4) -- (-1,-2.4) -- (1,-2.4) -- (1,.4) -- cycle;
\begin{scope}[every path/.append style={draw=red,line width=.05619cm}]
\path svg[xscale=1.7363cm,yscale=-1.7363cm] {M -.15631 -1.0969 Q .09121 -1.2753 .36479 -1.4537};
\path svg[xscale=1.7363cm,yscale=-1.7363cm] {M -1 0 Q -.40794 -.2124 -.17946 -.42479};
\path svg[xscale=1.7363cm,yscale=-1.7363cm] {M .36479 -1.4537 Q .49384 -1.197 .64034 -.94024};
\path svg[xscale=1.7363cm,yscale=-1.7363cm] {M -.67513 -1.4303 Q -.40018 -1.2636 -.15631 -1.0969};
\path svg[xscale=1.7363cm,yscale=-1.7363cm] {M -1 -1 Q -.82664 -1.2151 -.67513 -1.4303};
\path svg[xscale=1.7363cm,yscale=-1.7363cm] {M 1 -1 Q .82 -.97012 .64034 -.94024};
\path svg[xscale=1.7363cm,yscale=-1.7363cm] {M -.17946 -.42479 Q .12238 -.40456 .43293 -.38434};
\path svg[xscale=1.7363cm,yscale=-1.7363cm] {M .43293 -.38434 Q .52441 -.66229 .64034 -.94024};
\path svg[xscale=1.7363cm,yscale=-1.7363cm] {M 1 0 Q .57119 -.19217 .43293 -.38434};
\path svg[xscale=1.7363cm,yscale=-1.7363cm] {M -1 -2 Q -.68301 -1.7151 -.67513 -1.4303};
\path svg[xscale=1.7363cm,yscale=-1.7363cm] {M 1 -2 Q .49889 -1.7269 .36479 -1.4537};
\path svg[xscale=1.7363cm,yscale=-1.7363cm] {M -.15631 -1.0969 Q -.15812 -.76084 -.17946 -.42479};
\end{scope}
\path[draw=blue,line width=.01405cm] (-1,.4) -- (-1,-2.4) -- (1,-2.4) -- (1,.4) -- cycle;
\end{scope}
\end{scope}
\end{tikzpicture}
}
\expandafter\def\csname e8diagram:K1K2H1K2\endcsname{%
\begin{tikzpicture}[execute at begin picture={\clip (-1.24,-2.736) rectangle ++(2.48,3.472);},x={(1.7363cm,0cm)},y={(0cm,-1.7363cm)},baseline={([yshift=-\the\dimexpr\fontdimen22\textfont2\relax]current  bounding  box.center)},line join=round]
\begin{scope}
\begin{scope}
\path[clip] (-1,.4) -- (-1,-2.4) -- (1,-2.4) -- (1,.4) -- cycle;
\path[fill={rgb,255:red,204;green,204;blue,255},draw=none] (-1,.4) -- (-1,-2.4) -- (1,-2.4) -- (1,.4) -- cycle;
\begin{scope}[every path/.append style={draw=red,line width=.05619cm}]
\path svg[xscale=1.7363cm,yscale=-1.7363cm] {M .22509 -1.6035 Q .22885 -1.3867 .24322 -1.1699};
\path svg[xscale=1.7363cm,yscale=-1.7363cm] {M 1 -1 Q .9083 -.73517 .87076 -.47034};
\path svg[xscale=1.7363cm,yscale=-1.7363cm] {M -.37791 -.47321 Q -.1778 -.57469 .00399 -.67616};
\path svg[xscale=1.7363cm,yscale=-1.7363cm] {M -.34207 -1.2185 Q -.33683 -.84583 -.37791 -.47321};
\path svg[xscale=1.7363cm,yscale=-1.7363cm] {M 1 0 Q .82074 -.23517 .87076 -.47034};
\path svg[xscale=1.7363cm,yscale=-1.7363cm] {M -.34207 -1.2185 Q -.07928 -1.411 .22509 -1.6035};
\path svg[xscale=1.7363cm,yscale=-1.7363cm] {M .50181 -.68893 Q .67335 -.57964 .87076 -.47034};
\path svg[xscale=1.7363cm,yscale=-1.7363cm] {M -1 0 Q -.51971 -.23661 -.37791 -.47321};
\path svg[xscale=1.7363cm,yscale=-1.7363cm] {M -.79532 -1.4924 Q -.55288 -1.3554 -.34207 -1.2185};
\path svg[xscale=1.7363cm,yscale=-1.7363cm] {M .24322 -1.1699 Q .36156 -.9294 .50181 -.68893};
\path svg[xscale=1.7363cm,yscale=-1.7363cm] {M -1 -1 Q -.87787 -1.2462 -.79532 -1.4924};
\path svg[xscale=1.7363cm,yscale=-1.7363cm] {M .00399 -.67616 Q .12236 -.92302 .24322 -1.1699};
\path svg[xscale=1.7363cm,yscale=-1.7363cm] {M .00399 -.67616 Q .25335 -.68255 .50181 -.68893};
\path svg[xscale=1.7363cm,yscale=-1.7363cm] {M 1 -2 Q .44167 -1.8018 .22509 -1.6035};
\path svg[xscale=1.7363cm,yscale=-1.7363cm] {M -1 -2 Q -.76748 -1.7462 -.79532 -1.4924};
\end{scope}
\path[draw=blue,line width=.01405cm] (-1,.4) -- (-1,-2.4) -- (1,-2.4) -- (1,.4) -- cycle;
\end{scope}
\end{scope}
\end{tikzpicture}
}
\expandafter\def\csname e8diagram:K1K2K1\endcsname{%
\begin{tikzpicture}[execute at begin picture={\clip (-1.24,-2.736) rectangle ++(2.48,3.472);},x={(1.7363cm,0cm)},y={(0cm,-1.7363cm)},baseline={([yshift=-\the\dimexpr\fontdimen22\textfont2\relax]current  bounding  box.center)},line join=round]
\begin{scope}
\begin{scope}
\path[clip] (-1,.4) -- (-1,-2.4) -- (1,-2.4) -- (1,.4) -- cycle;
\path[fill={rgb,255:red,204;green,204;blue,255},draw=none] (-1,.4) -- (-1,-2.4) -- (1,-2.4) -- (1,.4) -- cycle;
\begin{scope}[every path/.append style={draw=red,line width=.05619cm}]
\path svg[xscale=1.7363cm,yscale=-1.7363cm] {M -.25545 -.90513 Q 0 -.90513 .25545 -.90513};
\path svg[xscale=1.7363cm,yscale=-1.7363cm] {M 1 -2 Q .66924 -1.6884 .66776 -1.3769};
\path svg[xscale=1.7363cm,yscale=-1.7363cm] {M -1 0 Q -.51249 -.1544 -.30717 -.3088};
\path svg[xscale=1.7363cm,yscale=-1.7363cm] {M .25545 -.90513 Q .26416 -.60696 .30717 -.3088};
\path svg[xscale=1.7363cm,yscale=-1.7363cm] {M .25545 -.90513 Q .44459 -1.141 .66776 -1.3769};
\path svg[xscale=1.7363cm,yscale=-1.7363cm] {M -.66776 -1.3769 Q -.44459 -1.141 -.25545 -.90513};
\path svg[xscale=1.7363cm,yscale=-1.7363cm] {M -1 -1 Q -.82588 -1.1884 -.66776 -1.3769};
\path svg[xscale=1.7363cm,yscale=-1.7363cm] {M 1 -1 Q .82588 -1.1884 .66776 -1.3769};
\path svg[xscale=1.7363cm,yscale=-1.7363cm] {M -.30717 -.3088 Q 0 -.3088 .30717 -.3088};
\path svg[xscale=1.7363cm,yscale=-1.7363cm] {M 1 0 Q .51249 -.1544 .30717 -.3088};
\path svg[xscale=1.7363cm,yscale=-1.7363cm] {M -1 -2 Q -.66924 -1.6884 -.66776 -1.3769};
\path svg[xscale=1.7363cm,yscale=-1.7363cm] {M -.25545 -.90513 Q -.26416 -.60696 -.30717 -.3088};
\end{scope}
\path[draw=blue,line width=.01405cm] (-1,.4) -- (-1,-2.4) -- (1,-2.4) -- (1,.4) -- cycle;
\end{scope}
\end{scope}
\end{tikzpicture}
}
\expandafter\def\csname e8diagram:K1K2K1H2\endcsname{%
\begin{tikzpicture}[execute at begin picture={\clip (-1.24,-2.736) rectangle ++(2.48,3.472);},x={(1.7363cm,0cm)},y={(0cm,-1.7363cm)},baseline={([yshift=-\the\dimexpr\fontdimen22\textfont2\relax]current  bounding  box.center)},line join=round]
\begin{scope}
\begin{scope}
\path[clip] (-1,.4) -- (-1,-2.4) -- (1,-2.4) -- (1,.4) -- cycle;
\path[fill={rgb,255:red,204;green,204;blue,255},draw=none] (-1,.4) -- (-1,-2.4) -- (1,-2.4) -- (1,.4) -- cycle;
\begin{scope}[every path/.append style={draw=red,line width=.05619cm}]
\path svg[xscale=1.7363cm,yscale=-1.7363cm] {M 1 -2 Q .58539 -1.7487 .4946 -1.4974};
\path svg[xscale=1.7363cm,yscale=-1.7363cm] {M .08889 -1.1029 Q .09391 -.84056 .1054 -.5782};
\path svg[xscale=1.7363cm,yscale=-1.7363cm] {M .08889 -1.1029 Q .27527 -1.3002 .4946 -1.4974};
\path svg[xscale=1.7363cm,yscale=-1.7363cm] {M 1 -1 Q .87674 -.98523 .75359 -.97045};
\path svg[xscale=1.7363cm,yscale=-1.7363cm] {M -.39961 -.36801 Q -.13085 -.47311 .1054 -.5782};
\path svg[xscale=1.7363cm,yscale=-1.7363cm] {M 1 0 Q .67632 -.19198 .60499 -.38396};
\path svg[xscale=1.7363cm,yscale=-1.7363cm] {M -.38355 -1.0044 Q -.37072 -.6862 -.39961 -.36801};
\path svg[xscale=1.7363cm,yscale=-1.7363cm] {M -.38355 -1.0044 Q -.1477 -1.0537 .08889 -1.1029};
\path svg[xscale=1.7363cm,yscale=-1.7363cm] {M .60499 -.38396 Q .65616 -.6772 .75359 -.97045};
\path svg[xscale=1.7363cm,yscale=-1.7363cm] {M -1 0 Q -.55447 -.18401 -.39961 -.36801};
\path svg[xscale=1.7363cm,yscale=-1.7363cm] {M -.75721 -1.4337 Q -.55074 -1.2191 -.38355 -1.0044};
\path svg[xscale=1.7363cm,yscale=-1.7363cm] {M -1 -1 Q -.86491 -1.2169 -.75721 -1.4337};
\path svg[xscale=1.7363cm,yscale=-1.7363cm] {M .4946 -1.4974 Q .61121 -1.2339 .75359 -.97045};
\path svg[xscale=1.7363cm,yscale=-1.7363cm] {M .1054 -.5782 Q .33942 -.48108 .60499 -.38396};
\path svg[xscale=1.7363cm,yscale=-1.7363cm] {M -1 -2 Q -.73335 -1.7169 -.75721 -1.4337};
\end{scope}
\path[draw=blue,line width=.01405cm] (-1,.4) -- (-1,-2.4) -- (1,-2.4) -- (1,.4) -- cycle;
\end{scope}
\end{scope}
\end{tikzpicture}
}
\expandafter\def\csname e8diagram:K1K2K1K2\endcsname{%
\begin{tikzpicture}[execute at begin picture={\clip (-1.24,-2.736) rectangle ++(2.48,3.472);},x={(1.7363cm,0cm)},y={(0cm,-1.7363cm)},baseline={([yshift=-\the\dimexpr\fontdimen22\textfont2\relax]current  bounding  box.center)},line join=round]
\begin{scope}
\begin{scope}
\path[clip] (-1,.4) -- (-1,-2.4) -- (1,-2.4) -- (1,.4) -- cycle;
\path[fill={rgb,255:red,204;green,204;blue,255},draw=none] (-1,.4) -- (-1,-2.4) -- (1,-2.4) -- (1,.4) -- cycle;
\begin{scope}[every path/.append style={draw=red,line width=.05619cm}]
\path svg[xscale=1.7363cm,yscale=-1.7363cm] {M 1 -2 Q .54131 -1.7865 .39862 -1.5731};
\path svg[xscale=1.7363cm,yscale=-1.7363cm] {M -.02752 -1.2663 Q 0 -1 .02752 -.73368};
\path svg[xscale=1.7363cm,yscale=-1.7363cm] {M -.02752 -1.2663 Q .16961 -1.4197 .39862 -1.5731};
\path svg[xscale=1.7363cm,yscale=-1.7363cm] {M 1 -1 Q .88893 -.76689 .81424 -.53378};
\path svg[xscale=1.7363cm,yscale=-1.7363cm] {M -.39862 -.42694 Q -.16961 -.58031 .02752 -.73368};
\path svg[xscale=1.7363cm,yscale=-1.7363cm] {M -.46424 -1.0535 Q -.41367 -.74021 -.39862 -.42694};
\path svg[xscale=1.7363cm,yscale=-1.7363cm] {M 1 0 Q .77399 -.26689 .81424 -.53378};
\path svg[xscale=1.7363cm,yscale=-1.7363cm] {M -.46424 -1.0535 Q -.24783 -1.1599 -.02752 -1.2663};
\path svg[xscale=1.7363cm,yscale=-1.7363cm] {M .46424 -.94653 Q .61808 -.74015 .81424 -.53378};
\path svg[xscale=1.7363cm,yscale=-1.7363cm] {M -1 0 Q -.54131 -.21347 -.39862 -.42694};
\path svg[xscale=1.7363cm,yscale=-1.7363cm] {M -.81424 -1.4662 Q -.61808 -1.2598 -.46424 -1.0535};
\path svg[xscale=1.7363cm,yscale=-1.7363cm] {M .39862 -1.5731 Q .41367 -1.2598 .46424 -.94653};
\path svg[xscale=1.7363cm,yscale=-1.7363cm] {M -1 -1 Q -.88893 -1.2331 -.81424 -1.4662};
\path svg[xscale=1.7363cm,yscale=-1.7363cm] {M .02752 -.73368 Q .24783 -.8401 .46424 -.94653};
\path svg[xscale=1.7363cm,yscale=-1.7363cm] {M -1 -2 Q -.77399 -1.7331 -.81424 -1.4662};
\end{scope}
\path[draw=blue,line width=.01405cm] (-1,.4) -- (-1,-2.4) -- (1,-2.4) -- (1,.4) -- cycle;
\end{scope}
\end{scope}
\end{tikzpicture}
}
\expandafter\def\csname e8diagram:K2\endcsname{%
\begin{tikzpicture}[execute at begin picture={\clip (-1.24,-2.736) rectangle ++(2.48,3.472);},x={(1.7363cm,0cm)},y={(0cm,-1.7363cm)},baseline={([yshift=-\the\dimexpr\fontdimen22\textfont2\relax]current  bounding  box.center)},line join=round]
\begin{scope}
\begin{scope}
\path[clip] (-1,.4) -- (-1,-2.4) -- (1,-2.4) -- (1,.4) -- cycle;
\path[fill={rgb,255:red,204;green,204;blue,255},draw=none] (-1,.4) -- (-1,-2.4) -- (1,-2.4) -- (1,.4) -- cycle;
\begin{scope}[every path/.append style={draw=red,line width=.05619cm}]
\path svg[xscale=1.7363cm,yscale=-1.7363cm] {M -1 0 Q -.55057 -.25537 -.43978 -.51073};
\path svg[xscale=1.7363cm,yscale=-1.7363cm] {M 1 0 Q .55057 -.25537 .43978 -.51073};
\path svg[xscale=1.7363cm,yscale=-1.7363cm] {M 1 -1 Q .71484 -.75537 .43978 -.51073};
\path svg[xscale=1.7363cm,yscale=-1.7363cm] {M -.43978 -.51073 Q 0 -.51073 .43978 -.51073};
\path svg[xscale=1.7363cm,yscale=-1.7363cm] {M -1 -2 Q 0 -2 1 -2};
\path svg[xscale=1.7363cm,yscale=-1.7363cm] {M -1 -1 Q -.71484 -.75537 -.43978 -.51073};
\end{scope}
\path[draw=blue,line width=.01405cm] (-1,.4) -- (-1,-2.4) -- (1,-2.4) -- (1,.4) -- cycle;
\end{scope}
\end{scope}
\end{tikzpicture}
}
\expandafter\def\csname e8diagram:K2E1\endcsname{%
\begin{tikzpicture}[execute at begin picture={\clip (-1.24,-2.736) rectangle ++(2.48,3.472);},x={(1.7363cm,0cm)},y={(0cm,-1.7363cm)},baseline={([yshift=-\the\dimexpr\fontdimen22\textfont2\relax]current  bounding  box.center)},line join=round]
\begin{scope}
\begin{scope}
\path[clip] (-1,.4) -- (-1,-2.4) -- (1,-2.4) -- (1,.4) -- cycle;
\path[fill={rgb,255:red,204;green,204;blue,255},draw=none] (-1,.4) -- (-1,-2.4) -- (1,-2.4) -- (1,.4) -- cycle;
\begin{scope}[every path/.append style={draw=red,line width=.05619cm}]
\path svg[xscale=1.7363cm,yscale=-1.7363cm] {M -1 -1 Q -.80777 -.7951 -.63282 -.59021};
\path svg[xscale=1.7363cm,yscale=-1.7363cm] {M 1 -1 Q .29519 -1.5 1 -2};
\path svg[xscale=1.7363cm,yscale=-1.7363cm] {M 1 0 Q .15987 -.45223 -.13985 -.90447};
\path svg[xscale=1.7363cm,yscale=-1.7363cm] {M -1 0 Q -.65382 -.2951 -.63282 -.59021};
\path svg[xscale=1.7363cm,yscale=-1.7363cm] {M -.13985 -.90447 Q -.37327 -.74734 -.63282 -.59021};
\path svg[xscale=1.7363cm,yscale=-1.7363cm] {M -1 -2 Q -.30492 -1.4522 -.13985 -.90447};
\end{scope}
\path[draw=blue,line width=.01405cm] (-1,.4) -- (-1,-2.4) -- (1,-2.4) -- (1,.4) -- cycle;
\end{scope}
\end{scope}
\end{tikzpicture}
}
\expandafter\def\csname e8diagram:K2E1H2\endcsname{%
\begin{tikzpicture}[execute at begin picture={\clip (-1.24,-2.736) rectangle ++(2.48,3.472);},x={(1.7363cm,0cm)},y={(0cm,-1.7363cm)},baseline={([yshift=-\the\dimexpr\fontdimen22\textfont2\relax]current  bounding  box.center)},line join=round]
\begin{scope}
\begin{scope}
\path[clip] (-1,.4) -- (-1,-2.4) -- (1,-2.4) -- (1,.4) -- cycle;
\path[fill={rgb,255:red,204;green,204;blue,255},draw=none] (-1,.4) -- (-1,-2.4) -- (1,-2.4) -- (1,.4) -- cycle;
\begin{scope}[every path/.append style={draw=red,line width=.05619cm}]
\path svg[xscale=1.7363cm,yscale=-1.7363cm] {M -1 0 Q -.65317 -.33623 -.65481 -.67246};
\path svg[xscale=1.7363cm,yscale=-1.7363cm] {M -.28192 -1.2327 Q -.45074 -.95258 -.65481 -.67246};
\path svg[xscale=1.7363cm,yscale=-1.7363cm] {M 1 -1 Q .82171 -1.1638 .65481 -1.3275};
\path svg[xscale=1.7363cm,yscale=-1.7363cm] {M -.28192 -1.2327 Q 0 -1 .28192 -.7673};
\path svg[xscale=1.7363cm,yscale=-1.7363cm] {M -1 -1 Q -.82171 -.83623 -.65481 -.67246};
\path svg[xscale=1.7363cm,yscale=-1.7363cm] {M .28192 -.7673 Q .45074 -1.0474 .65481 -1.3275};
\path svg[xscale=1.7363cm,yscale=-1.7363cm] {M 1 0 Q .41782 -.38365 .28192 -.7673};
\path svg[xscale=1.7363cm,yscale=-1.7363cm] {M -1 -2 Q -.41782 -1.6164 -.28192 -1.2327};
\path svg[xscale=1.7363cm,yscale=-1.7363cm] {M 1 -2 Q .65317 -1.6638 .65481 -1.3275};
\end{scope}
\path[draw=blue,line width=.01405cm] (-1,.4) -- (-1,-2.4) -- (1,-2.4) -- (1,.4) -- cycle;
\end{scope}
\end{scope}
\end{tikzpicture}
}
\expandafter\def\csname e8diagram:K2E1K2\endcsname{%
\begin{tikzpicture}[execute at begin picture={\clip (-1.24,-2.736) rectangle ++(2.48,3.472);},x={(1.7363cm,0cm)},y={(0cm,-1.7363cm)},baseline={([yshift=-\the\dimexpr\fontdimen22\textfont2\relax]current  bounding  box.center)},line join=round]
\begin{scope}
\begin{scope}
\path[clip] (-1,.4) -- (-1,-2.4) -- (1,-2.4) -- (1,.4) -- cycle;
\path[fill={rgb,255:red,204;green,204;blue,255},draw=none] (-1,.4) -- (-1,-2.4) -- (1,-2.4) -- (1,.4) -- cycle;
\begin{scope}[every path/.append style={draw=red,line width=.05619cm}]
\path svg[xscale=1.7363cm,yscale=-1.7363cm] {M -1 0 Q -.62319 -.37341 -.62458 -.74682};
\path svg[xscale=1.7363cm,yscale=-1.7363cm] {M -.31607 -1.4371 Q -.4505 -1.0919 -.62458 -.74682};
\path svg[xscale=1.7363cm,yscale=-1.7363cm] {M 1 -1 Q .80928 -.87341 .62458 -.74682};
\path svg[xscale=1.7363cm,yscale=-1.7363cm] {M -.31607 -1.4371 Q 0 -1.4371 .31607 -1.4371};
\path svg[xscale=1.7363cm,yscale=-1.7363cm] {M -1 -1 Q -.80928 -.87341 -.62458 -.74682};
\path svg[xscale=1.7363cm,yscale=-1.7363cm] {M .31607 -1.4371 Q .4505 -1.0919 .62458 -.74682};
\path svg[xscale=1.7363cm,yscale=-1.7363cm] {M 1 0 Q .62319 -.37341 .62458 -.74682};
\path svg[xscale=1.7363cm,yscale=-1.7363cm] {M 1 -2 Q .46663 -1.7185 .31607 -1.4371};
\path svg[xscale=1.7363cm,yscale=-1.7363cm] {M -1 -2 Q -.46663 -1.7185 -.31607 -1.4371};
\end{scope}
\path[draw=blue,line width=.01405cm] (-1,.4) -- (-1,-2.4) -- (1,-2.4) -- (1,.4) -- cycle;
\end{scope}
\end{scope}
\end{tikzpicture}
}
\expandafter\def\csname e8diagram:K2E1K2H1\endcsname{%
\begin{tikzpicture}[execute at begin picture={\clip (-1.24,-2.736) rectangle ++(2.48,3.472);},x={(1.7363cm,0cm)},y={(0cm,-1.7363cm)},baseline={([yshift=-\the\dimexpr\fontdimen22\textfont2\relax]current  bounding  box.center)},line join=round]
\begin{scope}
\begin{scope}
\path[clip] (-1,.4) -- (-1,-2.4) -- (1,-2.4) -- (1,.4) -- cycle;
\path[fill={rgb,255:red,204;green,204;blue,255},draw=none] (-1,.4) -- (-1,-2.4) -- (1,-2.4) -- (1,.4) -- cycle;
\begin{scope}[every path/.append style={draw=red,line width=.05619cm}]
\path svg[xscale=1.7363cm,yscale=-1.7363cm] {M -1 0 Q -.66796 -.3495 -.68691 -.69901};
\path svg[xscale=1.7363cm,yscale=-1.7363cm] {M .46904 -.57825 Q .59368 -.8056 .73579 -1.0329};
\path svg[xscale=1.7363cm,yscale=-1.7363cm] {M -.38704 -1.3664 Q -.51304 -1.0327 -.68691 -.69901};
\path svg[xscale=1.7363cm,yscale=-1.7363cm] {M -.38704 -1.3664 Q -.1215 -1.2537 .12957 -1.1411};
\path svg[xscale=1.7363cm,yscale=-1.7363cm] {M 1 -1 Q .86783 -1.0165 .73579 -1.0329};
\path svg[xscale=1.7363cm,yscale=-1.7363cm] {M -1 -1 Q -.83823 -.8495 -.68691 -.69901};
\path svg[xscale=1.7363cm,yscale=-1.7363cm] {M .12957 -1.1411 Q .28316 -.85966 .46904 -.57825};
\path svg[xscale=1.7363cm,yscale=-1.7363cm] {M .12957 -1.1411 Q .3139 -1.3406 .53776 -1.5402};
\path svg[xscale=1.7363cm,yscale=-1.7363cm] {M 1 0 Q .55679 -.28913 .46904 -.57825};
\path svg[xscale=1.7363cm,yscale=-1.7363cm] {M .53776 -1.5402 Q .62349 -1.2866 .73579 -1.0329};
\path svg[xscale=1.7363cm,yscale=-1.7363cm] {M -1 -2 Q -.49859 -1.6832 -.38704 -1.3664};
\path svg[xscale=1.7363cm,yscale=-1.7363cm] {M 1 -2 Q .61974 -1.7701 .53776 -1.5402};
\end{scope}
\path[draw=blue,line width=.01405cm] (-1,.4) -- (-1,-2.4) -- (1,-2.4) -- (1,.4) -- cycle;
\end{scope}
\end{scope}
\end{tikzpicture}
}
\expandafter\def\csname e8diagram:K2H1\endcsname{%
\begin{tikzpicture}[execute at begin picture={\clip (-1.24,-2.736) rectangle ++(2.48,3.472);},x={(1.7363cm,0cm)},y={(0cm,-1.7363cm)},baseline={([yshift=-\the\dimexpr\fontdimen22\textfont2\relax]current  bounding  box.center)},line join=round]
\begin{scope}
\begin{scope}
\path[clip] (-1,.4) -- (-1,-2.4) -- (1,-2.4) -- (1,.4) -- cycle;
\path[fill={rgb,255:red,204;green,204;blue,255},draw=none] (-1,.4) -- (-1,-2.4) -- (1,-2.4) -- (1,.4) -- cycle;
\begin{scope}[every path/.append style={draw=red,line width=.05619cm}]
\path svg[xscale=1.7363cm,yscale=-1.7363cm] {M -1 0 Q -.63283 -.2441 -.56813 -.48821};
\path svg[xscale=1.7363cm,yscale=-1.7363cm] {M 1 -1 Q .73979 -1.0438 .47994 -1.0876};
\path svg[xscale=1.7363cm,yscale=-1.7363cm] {M .1763 -1.7498 Q .32843 -1.4187 .47994 -1.0876};
\path svg[xscale=1.7363cm,yscale=-1.7363cm] {M .18002 -.48203 Q -.19469 -.48512 -.56813 -.48821};
\path svg[xscale=1.7363cm,yscale=-1.7363cm] {M -1 -1 Q -.77264 -.7441 -.56813 -.48821};
\path svg[xscale=1.7363cm,yscale=-1.7363cm] {M 1 0 Q .39748 -.24101 .18002 -.48203};
\path svg[xscale=1.7363cm,yscale=-1.7363cm] {M 1 -2 Q .45007 -1.8749 .1763 -1.7498};
\path svg[xscale=1.7363cm,yscale=-1.7363cm] {M -1 -2 Q -.23725 -1.8749 .1763 -1.7498};
\path svg[xscale=1.7363cm,yscale=-1.7363cm] {M .47994 -1.0876 Q .32489 -.7848 .18002 -.48203};
\end{scope}
\path[draw=blue,line width=.01405cm] (-1,.4) -- (-1,-2.4) -- (1,-2.4) -- (1,.4) -- cycle;
\end{scope}
\end{scope}
\end{tikzpicture}
}
\expandafter\def\csname e8diagram:K2H12\endcsname{%
\begin{tikzpicture}[execute at begin picture={\clip (-1.24,-2.736) rectangle ++(2.48,3.472);},x={(1.7363cm,0cm)},y={(0cm,-1.7363cm)},baseline={([yshift=-\the\dimexpr\fontdimen22\textfont2\relax]current  bounding  box.center)},line join=round]
\begin{scope}
\begin{scope}
\path[clip] (-1,.4) -- (-1,-2.4) -- (1,-2.4) -- (1,.4) -- cycle;
\path[fill={rgb,255:red,204;green,204;blue,255},draw=none] (-1,.4) -- (-1,-2.4) -- (1,-2.4) -- (1,.4) -- cycle;
\begin{scope}[every path/.append style={draw=red,line width=.05619cm}]
\path svg[xscale=1.7363cm,yscale=-1.7363cm] {M .1748 -1.0608 Q .122 -.75411 .0732 -.4474};
\path svg[xscale=1.7363cm,yscale=-1.7363cm] {M .1748 -1.0608 Q .41829 -1.1474 .66983 -1.234};
\path svg[xscale=1.7363cm,yscale=-1.7363cm] {M .0732 -.4474 Q -.28784 -.45318 -.6464 -.45896};
\path svg[xscale=1.7363cm,yscale=-1.7363cm] {M -.09552 -1.6708 Q .05239 -1.3658 .1748 -1.0608};
\path svg[xscale=1.7363cm,yscale=-1.7363cm] {M -1 0 Q -.68616 -.22948 -.6464 -.45896};
\path svg[xscale=1.7363cm,yscale=-1.7363cm] {M -.09552 -1.6708 Q .20709 -1.6979 .52641 -1.7251};
\path svg[xscale=1.7363cm,yscale=-1.7363cm] {M 1 -1 Q .83197 -1.117 .66983 -1.234};
\path svg[xscale=1.7363cm,yscale=-1.7363cm] {M .52641 -1.7251 Q .58355 -1.4795 .66983 -1.234};
\path svg[xscale=1.7363cm,yscale=-1.7363cm] {M -1 -1 Q -.80616 -.72948 -.6464 -.45896};
\path svg[xscale=1.7363cm,yscale=-1.7363cm] {M 1 0 Q .33867 -.2237 .0732 -.4474};
\path svg[xscale=1.7363cm,yscale=-1.7363cm] {M 1 -2 Q .65459 -1.8626 .52641 -1.7251};
\path svg[xscale=1.7363cm,yscale=-1.7363cm] {M -1 -2 Q -.37836 -1.8354 -.09552 -1.6708};
\end{scope}
\path[draw=blue,line width=.01405cm] (-1,.4) -- (-1,-2.4) -- (1,-2.4) -- (1,.4) -- cycle;
\end{scope}
\end{scope}
\end{tikzpicture}
}
\expandafter\def\csname e8diagram:K2H1H2\endcsname{%
\begin{tikzpicture}[execute at begin picture={\clip (-1.24,-2.736) rectangle ++(2.48,3.472);},x={(1.7363cm,0cm)},y={(0cm,-1.7363cm)},baseline={([yshift=-\the\dimexpr\fontdimen22\textfont2\relax]current  bounding  box.center)},line join=round]
\begin{scope}
\begin{scope}
\path[clip] (-1,.4) -- (-1,-2.4) -- (1,-2.4) -- (1,.4) -- cycle;
\path[fill={rgb,255:red,204;green,204;blue,255},draw=none] (-1,.4) -- (-1,-2.4) -- (1,-2.4) -- (1,.4) -- cycle;
\begin{scope}[every path/.append style={draw=red,line width=.05619cm}]
\path svg[xscale=1.7363cm,yscale=-1.7363cm] {M .19782 -1.2121 Q .41966 -1.0531 .64724 -.894};
\path svg[xscale=1.7363cm,yscale=-1.7363cm] {M .64724 -.894 Q .53778 -.65006 .44605 -.40612};
\path svg[xscale=1.7363cm,yscale=-1.7363cm] {M -1 0 Q -.67944 -.26582 -.65793 -.53163};
\path svg[xscale=1.7363cm,yscale=-1.7363cm] {M .06789 -1.8233 Q .13551 -1.5177 .19782 -1.2121};
\path svg[xscale=1.7363cm,yscale=-1.7363cm] {M 1 -1 Q .82307 -.947 .64724 -.894};
\path svg[xscale=1.7363cm,yscale=-1.7363cm] {M -.06697 -.70785 Q -.35162 -.61974 -.65793 -.53163};
\path svg[xscale=1.7363cm,yscale=-1.7363cm] {M -.06697 -.70785 Q .16983 -.55699 .44605 -.40612};
\path svg[xscale=1.7363cm,yscale=-1.7363cm] {M -1 -1 Q -.81637 -.76582 -.65793 -.53163};
\path svg[xscale=1.7363cm,yscale=-1.7363cm] {M 1 -2 Q .41248 -1.9117 .06789 -1.8233};
\path svg[xscale=1.7363cm,yscale=-1.7363cm] {M -1 -2 Q -.33202 -1.9117 .06789 -1.8233};
\path svg[xscale=1.7363cm,yscale=-1.7363cm] {M 1 0 Q .57455 -.20306 .44605 -.40612};
\path svg[xscale=1.7363cm,yscale=-1.7363cm] {M .19782 -1.2121 Q .06481 -.95999 -.06697 -.70785};
\end{scope}
\path[draw=blue,line width=.01405cm] (-1,.4) -- (-1,-2.4) -- (1,-2.4) -- (1,.4) -- cycle;
\end{scope}
\end{scope}
\end{tikzpicture}
}
\expandafter\def\csname e8diagram:K2H1K2\endcsname{%
\begin{tikzpicture}[execute at begin picture={\clip (-1.24,-2.736) rectangle ++(2.48,3.472);},x={(1.7363cm,0cm)},y={(0cm,-1.7363cm)},baseline={([yshift=-\the\dimexpr\fontdimen22\textfont2\relax]current  bounding  box.center)},line join=round]
\begin{scope}
\begin{scope}
\path[clip] (-1,.4) -- (-1,-2.4) -- (1,-2.4) -- (1,.4) -- cycle;
\path[fill={rgb,255:red,204;green,204;blue,255},draw=none] (-1,.4) -- (-1,-2.4) -- (1,-2.4) -- (1,.4) -- cycle;
\begin{scope}[every path/.append style={draw=red,line width=.05619cm}]
\path svg[xscale=1.7363cm,yscale=-1.7363cm] {M 0 -1.3175 Q .1183 -1.0832 .23819 -.84883};
\path svg[xscale=1.7363cm,yscale=-1.7363cm] {M .23819 -.84883 Q .45571 -.70173 .7014 -.55463};
\path svg[xscale=1.7363cm,yscale=-1.7363cm] {M -1 0 Q -.70164 -.27732 -.7014 -.55463};
\path svg[xscale=1.7363cm,yscale=-1.7363cm] {M 0 -1.864 Q 0 -1.5908 0 -1.3175};
\path svg[xscale=1.7363cm,yscale=-1.7363cm] {M 1 -1 Q .83802 -.77732 .7014 -.55463};
\path svg[xscale=1.7363cm,yscale=-1.7363cm] {M -.23819 -.84883 Q 0 -.84883 .23819 -.84883};
\path svg[xscale=1.7363cm,yscale=-1.7363cm] {M -.23819 -.84883 Q -.45571 -.70173 -.7014 -.55463};
\path svg[xscale=1.7363cm,yscale=-1.7363cm] {M -1 -1 Q -.83802 -.77732 -.7014 -.55463};
\path svg[xscale=1.7363cm,yscale=-1.7363cm] {M 1 -2 Q .39217 -1.932 0 -1.864};
\path svg[xscale=1.7363cm,yscale=-1.7363cm] {M -1 -2 Q -.39217 -1.932 0 -1.864};
\path svg[xscale=1.7363cm,yscale=-1.7363cm] {M 1 0 Q .70164 -.27732 .7014 -.55463};
\path svg[xscale=1.7363cm,yscale=-1.7363cm] {M 0 -1.3175 Q -.1183 -1.0832 -.23819 -.84883};
\end{scope}
\path[draw=blue,line width=.01405cm] (-1,.4) -- (-1,-2.4) -- (1,-2.4) -- (1,.4) -- cycle;
\end{scope}
\end{scope}
\end{tikzpicture}
}
\expandafter\def\csname e8diagram:K2K1\endcsname{%
\begin{tikzpicture}[execute at begin picture={\clip (-1.24,-2.736) rectangle ++(2.48,3.472);},x={(1.7363cm,0cm)},y={(0cm,-1.7363cm)},baseline={([yshift=-\the\dimexpr\fontdimen22\textfont2\relax]current  bounding  box.center)},line join=round]
\begin{scope}
\begin{scope}
\path[clip] (-1,.4) -- (-1,-2.4) -- (1,-2.4) -- (1,.4) -- cycle;
\path[fill={rgb,255:red,204;green,204;blue,255},draw=none] (-1,.4) -- (-1,-2.4) -- (1,-2.4) -- (1,.4) -- cycle;
\begin{scope}[every path/.append style={draw=red,line width=.05619cm}]
\path svg[xscale=1.7363cm,yscale=-1.7363cm] {M .01375 -.59593 Q -.30087 -.55148 -.62893 -.50703};
\path svg[xscale=1.7363cm,yscale=-1.7363cm] {M -1 0 Q -.66634 -.25351 -.62893 -.50703};
\path svg[xscale=1.7363cm,yscale=-1.7363cm] {M 1 -1 Q .8015 -1.2465 .62893 -1.493};
\path svg[xscale=1.7363cm,yscale=-1.7363cm] {M -.01375 -1.4041 Q 0 -1 .01375 -.59593};
\path svg[xscale=1.7363cm,yscale=-1.7363cm] {M -.01375 -1.4041 Q .30087 -1.4485 .62893 -1.493};
\path svg[xscale=1.7363cm,yscale=-1.7363cm] {M -1 -1 Q -.8015 -.75351 -.62893 -.50703};
\path svg[xscale=1.7363cm,yscale=-1.7363cm] {M 1 0 Q .27888 -.29796 .01375 -.59593};
\path svg[xscale=1.7363cm,yscale=-1.7363cm] {M 1 -2 Q .66634 -1.7465 .62893 -1.493};
\path svg[xscale=1.7363cm,yscale=-1.7363cm] {M -1 -2 Q -.27888 -1.702 -.01375 -1.4041};
\end{scope}
\path[draw=blue,line width=.01405cm] (-1,.4) -- (-1,-2.4) -- (1,-2.4) -- (1,.4) -- cycle;
\end{scope}
\end{scope}
\end{tikzpicture}
}
\expandafter\def\csname e8diagram:K2K1H2\endcsname{%
\begin{tikzpicture}[execute at begin picture={\clip (-1.24,-2.736) rectangle ++(2.48,3.472);},x={(1.7363cm,0cm)},y={(0cm,-1.7363cm)},baseline={([yshift=-\the\dimexpr\fontdimen22\textfont2\relax]current  bounding  box.center)},line join=round]
\begin{scope}
\begin{scope}
\path[clip] (-1,.4) -- (-1,-2.4) -- (1,-2.4) -- (1,.4) -- cycle;
\path[fill={rgb,255:red,204;green,204;blue,255},draw=none] (-1,.4) -- (-1,-2.4) -- (1,-2.4) -- (1,.4) -- cycle;
\begin{scope}[every path/.append style={draw=red,line width=.05619cm}]
\path svg[xscale=1.7363cm,yscale=-1.7363cm] {M -.15631 -.90312 Q -.40018 -.73643 -.67513 -.56974};
\path svg[xscale=1.7363cm,yscale=-1.7363cm] {M -.15631 -.90312 Q .09121 -.72469 .36479 -.54626};
\path svg[xscale=1.7363cm,yscale=-1.7363cm] {M -1 0 Q -.68301 -.28487 -.67513 -.56974};
\path svg[xscale=1.7363cm,yscale=-1.7363cm] {M -.17946 -1.5752 Q -.15812 -1.2392 -.15631 -.90312};
\path svg[xscale=1.7363cm,yscale=-1.7363cm] {M -.17946 -1.5752 Q .12238 -1.5954 .43293 -1.6157};
\path svg[xscale=1.7363cm,yscale=-1.7363cm] {M 1 -1 Q .82 -1.0299 .64034 -1.0598};
\path svg[xscale=1.7363cm,yscale=-1.7363cm] {M -1 -1 Q -.82664 -.78487 -.67513 -.56974};
\path svg[xscale=1.7363cm,yscale=-1.7363cm] {M .43293 -1.6157 Q .52441 -1.3377 .64034 -1.0598};
\path svg[xscale=1.7363cm,yscale=-1.7363cm] {M .36479 -.54626 Q .49384 -.80301 .64034 -1.0598};
\path svg[xscale=1.7363cm,yscale=-1.7363cm] {M 1 0 Q .49889 -.27313 .36479 -.54626};
\path svg[xscale=1.7363cm,yscale=-1.7363cm] {M 1 -2 Q .57119 -1.8078 .43293 -1.6157};
\path svg[xscale=1.7363cm,yscale=-1.7363cm] {M -1 -2 Q -.40794 -1.7876 -.17946 -1.5752};
\end{scope}
\path[draw=blue,line width=.01405cm] (-1,.4) -- (-1,-2.4) -- (1,-2.4) -- (1,.4) -- cycle;
\end{scope}
\end{scope}
\end{tikzpicture}
}
\expandafter\def\csname e8diagram:K2K1K2\endcsname{%
\begin{tikzpicture}[execute at begin picture={\clip (-1.24,-2.736) rectangle ++(2.48,3.472);},x={(1.7363cm,0cm)},y={(0cm,-1.7363cm)},baseline={([yshift=-\the\dimexpr\fontdimen22\textfont2\relax]current  bounding  box.center)},line join=round]
\begin{scope}
\begin{scope}
\path[clip] (-1,.4) -- (-1,-2.4) -- (1,-2.4) -- (1,.4) -- cycle;
\path[fill={rgb,255:red,204;green,204;blue,255},draw=none] (-1,.4) -- (-1,-2.4) -- (1,-2.4) -- (1,.4) -- cycle;
\begin{scope}[every path/.append style={draw=red,line width=.05619cm}]
\path svg[xscale=1.7363cm,yscale=-1.7363cm] {M -.25545 -1.0949 Q -.44459 -.85899 -.66776 -.62311};
\path svg[xscale=1.7363cm,yscale=-1.7363cm] {M -.25545 -1.0949 Q 0 -1.0949 .25545 -1.0949};
\path svg[xscale=1.7363cm,yscale=-1.7363cm] {M -1 0 Q -.66924 -.31156 -.66776 -.62311};
\path svg[xscale=1.7363cm,yscale=-1.7363cm] {M -.30717 -1.6912 Q -.26416 -1.393 -.25545 -1.0949};
\path svg[xscale=1.7363cm,yscale=-1.7363cm] {M -.30717 -1.6912 Q 0 -1.6912 .30717 -1.6912};
\path svg[xscale=1.7363cm,yscale=-1.7363cm] {M 1 -1 Q .82588 -.81156 .66776 -.62311};
\path svg[xscale=1.7363cm,yscale=-1.7363cm] {M .30717 -1.6912 Q .26416 -1.393 .25545 -1.0949};
\path svg[xscale=1.7363cm,yscale=-1.7363cm] {M -1 -1 Q -.82588 -.81156 -.66776 -.62311};
\path svg[xscale=1.7363cm,yscale=-1.7363cm] {M .25545 -1.0949 Q .44459 -.85899 .66776 -.62311};
\path svg[xscale=1.7363cm,yscale=-1.7363cm] {M 1 -2 Q .51249 -1.8456 .30717 -1.6912};
\path svg[xscale=1.7363cm,yscale=-1.7363cm] {M -1 -2 Q -.51249 -1.8456 -.30717 -1.6912};
\path svg[xscale=1.7363cm,yscale=-1.7363cm] {M 1 0 Q .66924 -.31156 .66776 -.62311};
\end{scope}
\path[draw=blue,line width=.01405cm] (-1,.4) -- (-1,-2.4) -- (1,-2.4) -- (1,.4) -- cycle;
\end{scope}
\end{scope}
\end{tikzpicture}
}
\expandafter\def\csname e8diagram:K2K1K2H1\endcsname{%
\begin{tikzpicture}[execute at begin picture={\clip (-1.24,-2.736) rectangle ++(2.48,3.472);},x={(1.7363cm,0cm)},y={(0cm,-1.7363cm)},baseline={([yshift=-\the\dimexpr\fontdimen22\textfont2\relax]current  bounding  box.center)},line join=round]
\begin{scope}
\begin{scope}
\path[clip] (-1,.4) -- (-1,-2.4) -- (1,-2.4) -- (1,.4) -- cycle;
\path[fill={rgb,255:red,204;green,204;blue,255},draw=none] (-1,.4) -- (-1,-2.4) -- (1,-2.4) -- (1,.4) -- cycle;
\begin{scope}[every path/.append style={draw=red,line width=.05619cm}]
\path svg[xscale=1.7363cm,yscale=-1.7363cm] {M 1 -2 Q .67632 -1.808 .60499 -1.616};
\path svg[xscale=1.7363cm,yscale=-1.7363cm] {M 1 -1 Q .87674 -1.0148 .75359 -1.0295};
\path svg[xscale=1.7363cm,yscale=-1.7363cm] {M .1054 -1.4218 Q .09391 -1.1594 .08889 -.89707};
\path svg[xscale=1.7363cm,yscale=-1.7363cm] {M .1054 -1.4218 Q .33942 -1.5189 .60499 -1.616};
\path svg[xscale=1.7363cm,yscale=-1.7363cm] {M 1 0 Q .58539 -.25129 .4946 -.50259};
\path svg[xscale=1.7363cm,yscale=-1.7363cm] {M -.38355 -.99561 Q -.55074 -.78093 -.75721 -.56625};
\path svg[xscale=1.7363cm,yscale=-1.7363cm] {M .4946 -.50259 Q .61121 -.76607 .75359 -1.0295};
\path svg[xscale=1.7363cm,yscale=-1.7363cm] {M -.38355 -.99561 Q -.1477 -.94634 .08889 -.89707};
\path svg[xscale=1.7363cm,yscale=-1.7363cm] {M -1 0 Q -.73335 -.28313 -.75721 -.56625};
\path svg[xscale=1.7363cm,yscale=-1.7363cm] {M -.39961 -1.632 Q -.37072 -1.3138 -.38355 -.99561};
\path svg[xscale=1.7363cm,yscale=-1.7363cm] {M -.39961 -1.632 Q -.13085 -1.5269 .1054 -1.4218};
\path svg[xscale=1.7363cm,yscale=-1.7363cm] {M .60499 -1.616 Q .65616 -1.3228 .75359 -1.0295};
\path svg[xscale=1.7363cm,yscale=-1.7363cm] {M -1 -1 Q -.86491 -.78313 -.75721 -.56625};
\path svg[xscale=1.7363cm,yscale=-1.7363cm] {M .08889 -.89707 Q .27527 -.69983 .4946 -.50259};
\path svg[xscale=1.7363cm,yscale=-1.7363cm] {M -1 -2 Q -.55447 -1.816 -.39961 -1.632};
\end{scope}
\path[draw=blue,line width=.01405cm] (-1,.4) -- (-1,-2.4) -- (1,-2.4) -- (1,.4) -- cycle;
\end{scope}
\end{scope}
\end{tikzpicture}
}
\expandafter\def\csname e8diagram:K2K1K2K1\endcsname{%
\begin{tikzpicture}[execute at begin picture={\clip (-1.24,-2.736) rectangle ++(2.48,3.472);},x={(1.7363cm,0cm)},y={(0cm,-1.7363cm)},baseline={([yshift=-\the\dimexpr\fontdimen22\textfont2\relax]current  bounding  box.center)},line join=round]
\begin{scope}
\begin{scope}
\path[clip] (-1,.4) -- (-1,-2.4) -- (1,-2.4) -- (1,.4) -- cycle;
\path[fill={rgb,255:red,204;green,204;blue,255},draw=none] (-1,.4) -- (-1,-2.4) -- (1,-2.4) -- (1,.4) -- cycle;
\begin{scope}[every path/.append style={draw=red,line width=.05619cm}]
\path svg[xscale=1.7363cm,yscale=-1.7363cm] {M 1 -2 Q .77399 -1.7331 .81424 -1.4662};
\path svg[xscale=1.7363cm,yscale=-1.7363cm] {M 1 -1 Q .88893 -1.2331 .81424 -1.4662};
\path svg[xscale=1.7363cm,yscale=-1.7363cm] {M .02752 -1.2663 Q 0 -1 -.02752 -.73368};
\path svg[xscale=1.7363cm,yscale=-1.7363cm] {M .02752 -1.2663 Q .24783 -1.1599 .46424 -1.0535};
\path svg[xscale=1.7363cm,yscale=-1.7363cm] {M 1 0 Q .54131 -.21347 .39862 -.42694};
\path svg[xscale=1.7363cm,yscale=-1.7363cm] {M -.46424 -.94653 Q -.61808 -.74015 -.81424 -.53378};
\path svg[xscale=1.7363cm,yscale=-1.7363cm] {M -.46424 -.94653 Q -.24783 -.8401 -.02752 -.73368};
\path svg[xscale=1.7363cm,yscale=-1.7363cm] {M -1 0 Q -.77399 -.26689 -.81424 -.53378};
\path svg[xscale=1.7363cm,yscale=-1.7363cm] {M -.39862 -1.5731 Q -.41367 -1.2598 -.46424 -.94653};
\path svg[xscale=1.7363cm,yscale=-1.7363cm] {M -.39862 -1.5731 Q -.16961 -1.4197 .02752 -1.2663};
\path svg[xscale=1.7363cm,yscale=-1.7363cm] {M .46424 -1.0535 Q .41367 -.74021 .39862 -.42694};
\path svg[xscale=1.7363cm,yscale=-1.7363cm] {M .46424 -1.0535 Q .61808 -1.2598 .81424 -1.4662};
\path svg[xscale=1.7363cm,yscale=-1.7363cm] {M -1 -1 Q -.88893 -.76689 -.81424 -.53378};
\path svg[xscale=1.7363cm,yscale=-1.7363cm] {M -.02752 -.73368 Q .16961 -.58031 .39862 -.42694};
\path svg[xscale=1.7363cm,yscale=-1.7363cm] {M -1 -2 Q -.54131 -1.7865 -.39862 -1.5731};
\end{scope}
\path[draw=blue,line width=.01405cm] (-1,.4) -- (-1,-2.4) -- (1,-2.4) -- (1,.4) -- cycle;
\end{scope}
\end{scope}
\end{tikzpicture}
}
\expandafter\def\csname e8diagram:bubble\endcsname{%
\begin{tikzpicture}[execute at begin picture={\clip (-1.24,-1.24) rectangle ++(2.48,2.48);},x={(2.1123cm,0cm)},y={(0cm,-2.1123cm)},baseline={([yshift=-\the\dimexpr\fontdimen22\textfont2\relax]current  bounding  box.center)},line join=round]
\begin{scope}
\path[clip] (0,0) circle[radius=1];
\path[fill={rgb,255:red,204;green,204;blue,255},draw=none] (0,0) circle[radius=1];
\begin{scope}[every path/.append style={draw=red,line width=.05619cm}]
\path (1,0) -- (.24167,0);
\path svg[xscale=2.1123cm,yscale=-2.1123cm] {M -.24167 0 Q 0 -.48333 .24167 0 Q 0 .48333 -.24167 0};
\path (-1,0) -- (-.24167,0);
\end{scope}
\path[draw=blue,line width=.01405cm] (0,0) circle[radius=1];
\end{scope}
\end{tikzpicture}
}
\expandafter\def\csname e8diagram:diag0_3_0\endcsname{%
\begin{tikzpicture}[execute at begin picture={\clip (-1.24,-2.736) rectangle ++(2.48,3.472);},x={(1.7363cm,0cm)},y={(0cm,-1.7363cm)},baseline={([yshift=-\the\dimexpr\fontdimen22\textfont2\relax]current  bounding  box.center)},line join=round]
\begin{scope}
\begin{scope}
\path[clip] (-1,.4) -- (-1,-2.4) -- (1,-2.4) -- (1,.4) -- cycle;
\path[fill={rgb,255:red,204;green,204;blue,255},draw=none] (-1,.4) -- (-1,-2.4) -- (1,-2.4) -- (1,.4) -- cycle;
\begin{scope}[every path/.append style={draw=red,line width=.05619cm}]
\path svg[xscale=1.7363cm,yscale=-1.7363cm] {M 1 -1 Q .56864 -1 .13729 -1};
\path svg[xscale=1.7363cm,yscale=-1.7363cm] {M 1 -2 Q .10611 -1.5 .13729 -1};
\path svg[xscale=1.7363cm,yscale=-1.7363cm] {M 1 0 Q .10611 -.5 .13729 -1};
\end{scope}
\path[draw=blue,line width=.01405cm] (-1,.4) -- (-1,-2.4) -- (1,-2.4) -- (1,.4) -- cycle;
\end{scope}
\end{scope}
\end{tikzpicture}
}
\expandafter\def\csname e8diagram:diag1_1_0\endcsname{%
\begin{tikzpicture}[execute at begin picture={\clip (-1.24,-.496) rectangle ++(2.48,.992);},x={(2.7736cm,0cm)},y={(0cm,-2.7736cm)},baseline={([yshift=-\the\dimexpr\fontdimen22\textfont2\relax]current  bounding  box.center)},line join=round]
\begin{scope}
\begin{scope}
\path[clip] (-1,.4) -- (-1,-.4) -- (1,-.4) -- (1,.4) -- cycle;
\path[fill={rgb,255:red,204;green,204;blue,255},draw=none] (-1,.4) -- (-1,-.4) -- (1,-.4) -- (1,.4) -- cycle;
\begin{scope}[every path/.append style={draw=red,line width=.05619cm}]
\path svg[xscale=2.7736cm,yscale=-2.7736cm] {M -1 0 Q 0 0 1 0};
\end{scope}
\path[draw=blue,line width=.01405cm] (-1,.4) -- (-1,-.4) -- (1,-.4) -- (1,.4) -- cycle;
\end{scope}
\end{scope}
\end{tikzpicture}
}
\expandafter\def\csname e8diagram:diag1_2_0\endcsname{%
\begin{tikzpicture}[execute at begin picture={\clip (-1.24,-1.616) rectangle ++(2.48,2.232);},x={(2.2204cm,0cm)},y={(0cm,-2.2204cm)},baseline={([yshift=-\the\dimexpr\fontdimen22\textfont2\relax]current  bounding  box.center)},line join=round]
\begin{scope}
\begin{scope}
\path[clip] (-1,.4) -- (-1,-1.4) -- (1,-1.4) -- (1,.4) -- cycle;
\path[fill={rgb,255:red,204;green,204;blue,255},draw=none] (-1,.4) -- (-1,-1.4) -- (1,-1.4) -- (1,.4) -- cycle;
\begin{scope}[every path/.append style={draw=red,line width=.05619cm}]
\path svg[xscale=2.2204cm,yscale=-2.2204cm] {M 1 0 Q .32218 -.25 .03299 -.5};
\path svg[xscale=2.2204cm,yscale=-2.2204cm] {M 1 -1 Q .32218 -.75 .03299 -.5};
\path svg[xscale=2.2204cm,yscale=-2.2204cm] {M -1 -.5 Q -.48351 -.5 .03299 -.5};
\end{scope}
\path[draw=blue,line width=.01405cm] (-1,.4) -- (-1,-1.4) -- (1,-1.4) -- (1,.4) -- cycle;
\end{scope}
\end{scope}
\end{tikzpicture}
}
\expandafter\def\csname e8diagram:diag1_3_0\endcsname{%
\begin{tikzpicture}[execute at begin picture={\clip (-1.24,-2.736) rectangle ++(2.48,3.472);},x={(1.7363cm,0cm)},y={(0cm,-1.7363cm)},baseline={([yshift=-\the\dimexpr\fontdimen22\textfont2\relax]current  bounding  box.center)},line join=round]
\begin{scope}
\begin{scope}
\path[clip] (-1,.4) -- (-1,-2.4) -- (1,-2.4) -- (1,.4) -- cycle;
\path[fill={rgb,255:red,204;green,204;blue,255},draw=none] (-1,.4) -- (-1,-2.4) -- (1,-2.4) -- (1,.4) -- cycle;
\begin{scope}[every path/.append style={draw=red,line width=.05619cm}]
\path svg[xscale=1.7363cm,yscale=-1.7363cm] {M 1 -1 Q .78605 -1 .57211 -1};
\path svg[xscale=1.7363cm,yscale=-1.7363cm] {M -.27262 -1 Q -.0215 -.78189 .25107 -.56377};
\path svg[xscale=1.7363cm,yscale=-1.7363cm] {M -.27262 -1 Q -.0215 -1.2181 .25107 -1.4362};
\path svg[xscale=1.7363cm,yscale=-1.7363cm] {M .25107 -1.4362 Q .40942 -1.2181 .57211 -1};
\path svg[xscale=1.7363cm,yscale=-1.7363cm] {M 1 0 Q .42711 -.28189 .25107 -.56377};
\path svg[xscale=1.7363cm,yscale=-1.7363cm] {M 1 -2 Q .42711 -1.7181 .25107 -1.4362};
\path svg[xscale=1.7363cm,yscale=-1.7363cm] {M .25107 -.56377 Q .40942 -.78189 .57211 -1};
\path svg[xscale=1.7363cm,yscale=-1.7363cm] {M -1 -1 Q -.63631 -1 -.27262 -1};
\end{scope}
\path[draw=blue,line width=.01405cm] (-1,.4) -- (-1,-2.4) -- (1,-2.4) -- (1,.4) -- cycle;
\end{scope}
\end{scope}
\end{tikzpicture}
}
\expandafter\def\csname e8diagram:diag1_3_1\endcsname{%
\begin{tikzpicture}[execute at begin picture={\clip (-1.24,-2.736) rectangle ++(2.48,3.472);},x={(1.7363cm,0cm)},y={(0cm,-1.7363cm)},baseline={([yshift=-\the\dimexpr\fontdimen22\textfont2\relax]current  bounding  box.center)},line join=round]
\begin{scope}
\begin{scope}
\path[clip] (-1,.4) -- (-1,-2.4) -- (1,-2.4) -- (1,.4) -- cycle;
\path[fill={rgb,255:red,204;green,204;blue,255},draw=none] (-1,.4) -- (-1,-2.4) -- (1,-2.4) -- (1,.4) -- cycle;
\begin{scope}[every path/.append style={draw=red,line width=.05619cm}]
\path svg[xscale=1.7363cm,yscale=-1.7363cm] {M -.08456 -1.2759 Q .19139 -1 .47997 -.72413};
\path svg[xscale=1.7363cm,yscale=-1.7363cm] {M -1 -1 Q -.54523 -1.1379 -.08456 -1.2759};
\path svg[xscale=1.7363cm,yscale=-1.7363cm] {M 1 0 Q .54086 -.36206 .47997 -.72413};
\path svg[xscale=1.7363cm,yscale=-1.7363cm] {M 1 -2 Q .20545 -1.6379 -.08456 -1.2759};
\path svg[xscale=1.7363cm,yscale=-1.7363cm] {M 1 -1 Q .73814 -.86206 .47997 -.72413};
\end{scope}
\path[draw=blue,line width=.01405cm] (-1,.4) -- (-1,-2.4) -- (1,-2.4) -- (1,.4) -- cycle;
\end{scope}
\end{scope}
\end{tikzpicture}
}
\expandafter\def\csname e8diagram:diag1_3_2\endcsname{%
\begin{tikzpicture}[execute at begin picture={\clip (-1.24,-2.736) rectangle ++(2.48,3.472);},x={(1.7363cm,0cm)},y={(0cm,-1.7363cm)},baseline={([yshift=-\the\dimexpr\fontdimen22\textfont2\relax]current  bounding  box.center)},line join=round]
\begin{scope}
\begin{scope}
\path[clip] (-1,.4) -- (-1,-2.4) -- (1,-2.4) -- (1,.4) -- cycle;
\path[fill={rgb,255:red,204;green,204;blue,255},draw=none] (-1,.4) -- (-1,-2.4) -- (1,-2.4) -- (1,.4) -- cycle;
\begin{scope}[every path/.append style={draw=red,line width=.05619cm}]
\path svg[xscale=1.7363cm,yscale=-1.7363cm] {M 1 -2 Q .54086 -1.6379 .47997 -1.2759};
\path svg[xscale=1.7363cm,yscale=-1.7363cm] {M -.08456 -.72412 Q .19139 -1 .47997 -1.2759};
\path svg[xscale=1.7363cm,yscale=-1.7363cm] {M -1 -1 Q -.54523 -.86206 -.08456 -.72412};
\path svg[xscale=1.7363cm,yscale=-1.7363cm] {M 1 0 Q .20545 -.36206 -.08456 -.72412};
\path svg[xscale=1.7363cm,yscale=-1.7363cm] {M 1 -1 Q .73814 -1.1379 .47997 -1.2759};
\end{scope}
\path[draw=blue,line width=.01405cm] (-1,.4) -- (-1,-2.4) -- (1,-2.4) -- (1,.4) -- cycle;
\end{scope}
\end{scope}
\end{tikzpicture}
}
\expandafter\def\csname e8diagram:diag1_3_3\endcsname{%
\begin{tikzpicture}[execute at begin picture={\clip (-1.24,-2.736) rectangle ++(2.48,3.472);},x={(1.7363cm,0cm)},y={(0cm,-1.7363cm)},baseline={([yshift=-\the\dimexpr\fontdimen22\textfont2\relax]current  bounding  box.center)},line join=round]
\begin{scope}
\begin{scope}
\path[clip] (-1,.4) -- (-1,-2.4) -- (1,-2.4) -- (1,.4) -- cycle;
\path[fill={rgb,255:red,204;green,204;blue,255},draw=none] (-1,.4) -- (-1,-2.4) -- (1,-2.4) -- (1,.4) -- cycle;
\begin{scope}[every path/.append style={draw=red,line width=.05619cm}]
\path svg[xscale=1.7363cm,yscale=-1.7363cm] {M 1 -1 Q .29519 -1.5 1 -2};
\path svg[xscale=1.7363cm,yscale=-1.7363cm] {M -1 -1 Q -.32749 -.5 1 0};
\end{scope}
\path[draw=blue,line width=.01405cm] (-1,.4) -- (-1,-2.4) -- (1,-2.4) -- (1,.4) -- cycle;
\end{scope}
\end{scope}
\end{tikzpicture}
}
\expandafter\def\csname e8diagram:diag1_3_4\endcsname{%
\begin{tikzpicture}[execute at begin picture={\clip (-1.24,-2.736) rectangle ++(2.48,3.472);},x={(1.7363cm,0cm)},y={(0cm,-1.7363cm)},baseline={([yshift=-\the\dimexpr\fontdimen22\textfont2\relax]current  bounding  box.center)},line join=round]
\begin{scope}
\begin{scope}
\path[clip] (-1,.4) -- (-1,-2.4) -- (1,-2.4) -- (1,.4) -- cycle;
\path[fill={rgb,255:red,204;green,204;blue,255},draw=none] (-1,.4) -- (-1,-2.4) -- (1,-2.4) -- (1,.4) -- cycle;
\begin{scope}[every path/.append style={draw=red,line width=.05619cm}]
\path svg[xscale=1.7363cm,yscale=-1.7363cm] {M 1 0 Q .29519 -.5 1 -1};
\path svg[xscale=1.7363cm,yscale=-1.7363cm] {M -1 -1 Q -.32749 -1.5 1 -2};
\end{scope}
\path[draw=blue,line width=.01405cm] (-1,.4) -- (-1,-2.4) -- (1,-2.4) -- (1,.4) -- cycle;
\end{scope}
\end{scope}
\end{tikzpicture}
}
\expandafter\def\csname e8diagram:diag2_1_0\endcsname{%
\begin{tikzpicture}[execute at begin picture={\clip (-1.24,-1.616) rectangle ++(2.48,2.232);},x={(2.2204cm,0cm)},y={(0cm,-2.2204cm)},baseline={([yshift=-\the\dimexpr\fontdimen22\textfont2\relax]current  bounding  box.center)},line join=round]
\begin{scope}
\begin{scope}
\path[clip] (-1,.4) -- (-1,-1.4) -- (1,-1.4) -- (1,.4) -- cycle;
\path[fill={rgb,255:red,204;green,204;blue,255},draw=none] (-1,.4) -- (-1,-1.4) -- (1,-1.4) -- (1,.4) -- cycle;
\begin{scope}[every path/.append style={draw=red,line width=.05619cm}]
\path svg[xscale=2.2204cm,yscale=-2.2204cm] {M -1 0 Q -.32218 -.25 -.03299 -.5};
\path svg[xscale=2.2204cm,yscale=-2.2204cm] {M 1 -.5 Q .48351 -.5 -.03299 -.5};
\path svg[xscale=2.2204cm,yscale=-2.2204cm] {M -1 -1 Q -.32218 -.75 -.03299 -.5};
\end{scope}
\path[draw=blue,line width=.01405cm] (-1,.4) -- (-1,-1.4) -- (1,-1.4) -- (1,.4) -- cycle;
\end{scope}
\end{scope}
\end{tikzpicture}
}
\expandafter\def\csname e8diagram:diag2_2_0\endcsname{%
\begin{tikzpicture}[execute at begin picture={\clip (-1.24,-1.616) rectangle ++(2.48,2.232);},x={(2.2204cm,0cm)},y={(0cm,-2.2204cm)},baseline={([yshift=-\the\dimexpr\fontdimen22\textfont2\relax]current  bounding  box.center)},line join=round]
\begin{scope}
\begin{scope}
\path[clip] (-1,.4) -- (-1,-1.4) -- (1,-1.4) -- (1,.4) -- cycle;
\path[fill={rgb,255:red,204;green,204;blue,255},draw=none] (-1,.4) -- (-1,-1.4) -- (1,-1.4) -- (1,.4) -- cycle;
\begin{scope}[every path/.append style={draw=red,line width=.05619cm}]
\path svg[xscale=2.2204cm,yscale=-2.2204cm] {M -1 0 Q 0 0 1 0};
\path svg[xscale=2.2204cm,yscale=-2.2204cm] {M -1 -1 Q 0 -1 1 -1};
\end{scope}
\path[draw=blue,line width=.01405cm] (-1,.4) -- (-1,-1.4) -- (1,-1.4) -- (1,.4) -- cycle;
\end{scope}
\end{scope}
\end{tikzpicture}
}
\expandafter\def\csname e8diagram:diag2_2_1\endcsname{%
\begin{tikzpicture}[execute at begin picture={\clip (-1.24,-1.616) rectangle ++(2.48,2.232);},x={(2.2204cm,0cm)},y={(0cm,-2.2204cm)},baseline={([yshift=-\the\dimexpr\fontdimen22\textfont2\relax]current  bounding  box.center)},line join=round]
\begin{scope}
\begin{scope}
\path[clip] (-1,.4) -- (-1,-1.4) -- (1,-1.4) -- (1,.4) -- cycle;
\path[fill={rgb,255:red,204;green,204;blue,255},draw=none] (-1,.4) -- (-1,-1.4) -- (1,-1.4) -- (1,.4) -- cycle;
\begin{scope}[every path/.append style={draw=red,line width=.05619cm}]
\path svg[xscale=2.2204cm,yscale=-2.2204cm] {M 1 0 Q .05277 -.5 1 -1};
\path svg[xscale=2.2204cm,yscale=-2.2204cm] {M -1 -1 Q -.05277 -.5 -1 0};
\end{scope}
\path[draw=blue,line width=.01405cm] (-1,.4) -- (-1,-1.4) -- (1,-1.4) -- (1,.4) -- cycle;
\end{scope}
\end{scope}
\end{tikzpicture}
}
\expandafter\def\csname e8diagram:diag2_2_2\endcsname{%
\begin{tikzpicture}[execute at begin picture={\clip (-1.24,-1.616) rectangle ++(2.48,2.232);},x={(2.2204cm,0cm)},y={(0cm,-2.2204cm)},baseline={([yshift=-\the\dimexpr\fontdimen22\textfont2\relax]current  bounding  box.center)},line join=round]
\begin{scope}
\begin{scope}
\path[clip] (-1,.4) -- (-1,-1.4) -- (1,-1.4) -- (1,.4) -- cycle;
\path[fill={rgb,255:red,204;green,204;blue,255},draw=none] (-1,.4) -- (-1,-1.4) -- (1,-1.4) -- (1,.4) -- cycle;
\begin{scope}[every path/.append style={draw=red,line width=.05619cm}]
\path svg[xscale=2.2204cm,yscale=-2.2204cm] {M 1 -1 Q .48126 -.75 .32222 -.5};
\path svg[xscale=2.2204cm,yscale=-2.2204cm] {M -.32222 -.5 Q 0 -.5 .32222 -.5};
\path svg[xscale=2.2204cm,yscale=-2.2204cm] {M -1 -1 Q -.48126 -.75 -.32222 -.5};
\path svg[xscale=2.2204cm,yscale=-2.2204cm] {M -1 0 Q -.48126 -.25 -.32222 -.5};
\path svg[xscale=2.2204cm,yscale=-2.2204cm] {M 1 0 Q .48126 -.25 .32222 -.5};
\end{scope}
\path[draw=blue,line width=.01405cm] (-1,.4) -- (-1,-1.4) -- (1,-1.4) -- (1,.4) -- cycle;
\end{scope}
\end{scope}
\end{tikzpicture}
}
\expandafter\def\csname e8diagram:diag2_2_3\endcsname{%
\begin{tikzpicture}[execute at begin picture={\clip (-1.24,-1.616) rectangle ++(2.48,2.232);},x={(2.2204cm,0cm)},y={(0cm,-2.2204cm)},baseline={([yshift=-\the\dimexpr\fontdimen22\textfont2\relax]current  bounding  box.center)},line join=round]
\begin{scope}
\begin{scope}
\path[clip] (-1,.4) -- (-1,-1.4) -- (1,-1.4) -- (1,.4) -- cycle;
\path[fill={rgb,255:red,204;green,204;blue,255},draw=none] (-1,.4) -- (-1,-1.4) -- (1,-1.4) -- (1,.4) -- cycle;
\begin{scope}[every path/.append style={draw=red,line width=.05619cm}]
\path svg[xscale=2.2204cm,yscale=-2.2204cm] {M 0 -.75959 Q 0 -.5 0 -.24041};
\path svg[xscale=2.2204cm,yscale=-2.2204cm] {M -1 -1 Q -.35701 -.87979 0 -.75959};
\path svg[xscale=2.2204cm,yscale=-2.2204cm] {M -1 0 Q -.35701 -.12021 0 -.24041};
\path svg[xscale=2.2204cm,yscale=-2.2204cm] {M 1 -1 Q .35701 -.87979 0 -.75959};
\path svg[xscale=2.2204cm,yscale=-2.2204cm] {M 1 0 Q .35701 -.12021 0 -.24041};
\end{scope}
\path[draw=blue,line width=.01405cm] (-1,.4) -- (-1,-1.4) -- (1,-1.4) -- (1,.4) -- cycle;
\end{scope}
\end{scope}
\end{tikzpicture}
}
\expandafter\def\csname e8diagram:diag2_2_4\endcsname{%
\begin{tikzpicture}[execute at begin picture={\clip (-1.24,-1.616) rectangle ++(2.48,2.232);},x={(2.2204cm,0cm)},y={(0cm,-2.2204cm)},baseline={([yshift=-\the\dimexpr\fontdimen22\textfont2\relax]current  bounding  box.center)},line join=round]
\begin{scope}
\begin{scope}
\path[clip] (-1,.4) -- (-1,-1.4) -- (1,-1.4) -- (1,.4) -- cycle;
\path[fill={rgb,255:red,204;green,204;blue,255},draw=none] (-1,.4) -- (-1,-1.4) -- (1,-1.4) -- (1,.4) -- cycle;
\begin{scope}[every path/.append style={draw=red,line width=.05619cm}]
\path svg[xscale=2.2204cm,yscale=-2.2204cm] {M 1 0 Q .49871 -.15116 .27405 -.30233};
\path svg[xscale=2.2204cm,yscale=-2.2204cm] {M -.27405 -.69767 Q -.25875 -.5 -.27405 -.30233};
\path svg[xscale=2.2204cm,yscale=-2.2204cm] {M -.27405 -.69767 Q 0 -.69767 .27405 -.69767};
\path svg[xscale=2.2204cm,yscale=-2.2204cm] {M .27405 -.69767 Q .25875 -.5 .27405 -.30233};
\path svg[xscale=2.2204cm,yscale=-2.2204cm] {M -1 0 Q -.49871 -.15116 -.27405 -.30233};
\path svg[xscale=2.2204cm,yscale=-2.2204cm] {M 1 -1 Q .49871 -.84884 .27405 -.69767};
\path svg[xscale=2.2204cm,yscale=-2.2204cm] {M -.27405 -.30233 Q 0 -.30233 .27405 -.30233};
\path svg[xscale=2.2204cm,yscale=-2.2204cm] {M -1 -1 Q -.49871 -.84884 -.27405 -.69767};
\end{scope}
\path[draw=blue,line width=.01405cm] (-1,.4) -- (-1,-1.4) -- (1,-1.4) -- (1,.4) -- cycle;
\end{scope}
\end{scope}
\end{tikzpicture}
}
\expandafter\def\csname e8diagram:diag2_3_0\endcsname{%
\begin{tikzpicture}[execute at begin picture={\clip (-1.24,-2.736) rectangle ++(2.48,3.472);},x={(1.7363cm,0cm)},y={(0cm,-1.7363cm)},baseline={([yshift=-\the\dimexpr\fontdimen22\textfont2\relax]current  bounding  box.center)},line join=round]
\begin{scope}
\begin{scope}
\path[clip] (-1,.4) -- (-1,-2.4) -- (1,-2.4) -- (1,.4) -- cycle;
\path[fill={rgb,255:red,204;green,204;blue,255},draw=none] (-1,.4) -- (-1,-2.4) -- (1,-2.4) -- (1,.4) -- cycle;
\begin{scope}[every path/.append style={draw=red,line width=.05619cm}]
\path svg[xscale=1.7363cm,yscale=-1.7363cm] {M 1 -1 Q .79652 -1 .59304 -1};
\path svg[xscale=1.7363cm,yscale=-1.7363cm] {M 1 -2 Q .56675 -1.5 .59304 -1};
\path svg[xscale=1.7363cm,yscale=-1.7363cm] {M -1 -1.5 Q -.39072 -1 -1 -.5};
\path svg[xscale=1.7363cm,yscale=-1.7363cm] {M 1 0 Q .56675 -.5 .59304 -1};
\end{scope}
\path[draw=blue,line width=.01405cm] (-1,.4) -- (-1,-2.4) -- (1,-2.4) -- (1,.4) -- cycle;
\end{scope}
\end{scope}
\end{tikzpicture}
}
\expandafter\def\csname e8diagram:diag2_3_1\endcsname{%
\begin{tikzpicture}[execute at begin picture={\clip (-1.24,-2.736) rectangle ++(2.48,3.472);},x={(1.7363cm,0cm)},y={(0cm,-1.7363cm)},baseline={([yshift=-\the\dimexpr\fontdimen22\textfont2\relax]current  bounding  box.center)},line join=round]
\begin{scope}
\begin{scope}
\path[clip] (-1,.4) -- (-1,-2.4) -- (1,-2.4) -- (1,.4) -- cycle;
\path[fill={rgb,255:red,204;green,204;blue,255},draw=none] (-1,.4) -- (-1,-2.4) -- (1,-2.4) -- (1,.4) -- cycle;
\begin{scope}[every path/.append style={draw=red,line width=.05619cm}]
\path svg[xscale=1.7363cm,yscale=-1.7363cm] {M 1 -1 Q .59841 -1.2497 .18583 -1.4994};
\path svg[xscale=1.7363cm,yscale=-1.7363cm] {M 1 -2 Q .39777 -1.7497 .18583 -1.4994};
\path svg[xscale=1.7363cm,yscale=-1.7363cm] {M -1 -.5 Q -.32407 -.25 1 0};
\path svg[xscale=1.7363cm,yscale=-1.7363cm] {M -1 -1.5 Q -.40699 -1.4997 .18583 -1.4994};
\end{scope}
\path[draw=blue,line width=.01405cm] (-1,.4) -- (-1,-2.4) -- (1,-2.4) -- (1,.4) -- cycle;
\end{scope}
\end{scope}
\end{tikzpicture}
}
\expandafter\def\csname e8diagram:diag2_3_10\endcsname{%
\begin{tikzpicture}[execute at begin picture={\clip (-1.24,-2.736) rectangle ++(2.48,3.472);},x={(1.7363cm,0cm)},y={(0cm,-1.7363cm)},baseline={([yshift=-\the\dimexpr\fontdimen22\textfont2\relax]current  bounding  box.center)},line join=round]
\begin{scope}
\begin{scope}
\path[clip] (-1,.4) -- (-1,-2.4) -- (1,-2.4) -- (1,.4) -- cycle;
\path[fill={rgb,255:red,204;green,204;blue,255},draw=none] (-1,.4) -- (-1,-2.4) -- (1,-2.4) -- (1,.4) -- cycle;
\begin{scope}[every path/.append style={draw=red,line width=.05619cm}]
\path svg[xscale=1.7363cm,yscale=-1.7363cm] {M .38721 -.53936 Q .51751 -.76968 .66045 -1};
\path svg[xscale=1.7363cm,yscale=-1.7363cm] {M 1 0 Q .51366 -.26968 .38721 -.53936};
\path svg[xscale=1.7363cm,yscale=-1.7363cm] {M 1 -2 Q .51366 -1.7303 .38721 -1.4606};
\path svg[xscale=1.7363cm,yscale=-1.7363cm] {M -.04559 -1 Q .15649 -.76968 .38721 -.53936};
\path svg[xscale=1.7363cm,yscale=-1.7363cm] {M -.04559 -1 Q .15649 -1.2303 .38721 -1.4606};
\path svg[xscale=1.7363cm,yscale=-1.7363cm] {M -1 -.5 Q -.80034 -.75 -.66807 -1};
\path svg[xscale=1.7363cm,yscale=-1.7363cm] {M .38721 -1.4606 Q .51751 -1.2303 .66045 -1};
\path svg[xscale=1.7363cm,yscale=-1.7363cm] {M -.66807 -1 Q -.35683 -1 -.04559 -1};
\path svg[xscale=1.7363cm,yscale=-1.7363cm] {M 1 -1 Q .83023 -1 .66045 -1};
\path svg[xscale=1.7363cm,yscale=-1.7363cm] {M -1 -1.5 Q -.80034 -1.25 -.66807 -1};
\end{scope}
\path[draw=blue,line width=.01405cm] (-1,.4) -- (-1,-2.4) -- (1,-2.4) -- (1,.4) -- cycle;
\end{scope}
\end{scope}
\end{tikzpicture}
}
\expandafter\def\csname e8diagram:diag2_3_11\endcsname{%
\begin{tikzpicture}[execute at begin picture={\clip (-1.24,-2.736) rectangle ++(2.48,3.472);},x={(1.7363cm,0cm)},y={(0cm,-1.7363cm)},baseline={([yshift=-\the\dimexpr\fontdimen22\textfont2\relax]current  bounding  box.center)},line join=round]
\begin{scope}
\begin{scope}
\path[clip] (-1,.4) -- (-1,-2.4) -- (1,-2.4) -- (1,.4) -- cycle;
\path[fill={rgb,255:red,204;green,204;blue,255},draw=none] (-1,.4) -- (-1,-2.4) -- (1,-2.4) -- (1,.4) -- cycle;
\begin{scope}[every path/.append style={draw=red,line width=.05619cm}]
\path svg[xscale=1.7363cm,yscale=-1.7363cm] {M .26814 -1.6266 Q .39366 -1.3885 .5211 -1.1503};
\path svg[xscale=1.7363cm,yscale=-1.7363cm] {M 1 0 Q .16452 -.18567 -.23811 -.37134};
\path svg[xscale=1.7363cm,yscale=-1.7363cm] {M 1 -2 Q .47304 -1.8133 .26814 -1.6266};
\path svg[xscale=1.7363cm,yscale=-1.7363cm] {M -.01082 -.93792 Q -.10859 -.65463 -.23811 -.37134};
\path svg[xscale=1.7363cm,yscale=-1.7363cm] {M -.01082 -.93792 Q .25291 -1.0441 .5211 -1.1503};
\path svg[xscale=1.7363cm,yscale=-1.7363cm] {M -1 -.5 Q -.63266 -.43567 -.23811 -.37134};
\path svg[xscale=1.7363cm,yscale=-1.7363cm] {M -.35916 -1.4096 Q -.17425 -1.1738 -.01082 -.93792};
\path svg[xscale=1.7363cm,yscale=-1.7363cm] {M -.35916 -1.4096 Q -.06403 -1.5181 .26814 -1.6266};
\path svg[xscale=1.7363cm,yscale=-1.7363cm] {M 1 -1 Q .75988 -1.0752 .5211 -1.1503};
\path svg[xscale=1.7363cm,yscale=-1.7363cm] {M -1 -1.5 Q -.67218 -1.4548 -.35916 -1.4096};
\end{scope}
\path[draw=blue,line width=.01405cm] (-1,.4) -- (-1,-2.4) -- (1,-2.4) -- (1,.4) -- cycle;
\end{scope}
\end{scope}
\end{tikzpicture}
}
\expandafter\def\csname e8diagram:diag2_3_12\endcsname{%
\begin{tikzpicture}[execute at begin picture={\clip (-1.24,-2.736) rectangle ++(2.48,3.472);},x={(1.7363cm,0cm)},y={(0cm,-1.7363cm)},baseline={([yshift=-\the\dimexpr\fontdimen22\textfont2\relax]current  bounding  box.center)},line join=round]
\begin{scope}
\begin{scope}
\path[clip] (-1,.4) -- (-1,-2.4) -- (1,-2.4) -- (1,.4) -- cycle;
\path[fill={rgb,255:red,204;green,204;blue,255},draw=none] (-1,.4) -- (-1,-2.4) -- (1,-2.4) -- (1,.4) -- cycle;
\begin{scope}[every path/.append style={draw=red,line width=.05619cm}]
\path svg[xscale=1.7363cm,yscale=-1.7363cm] {M 1 0 Q .628 -.30302 .59453 -.60603};
\path svg[xscale=1.7363cm,yscale=-1.7363cm] {M 1 -2 Q .31956 -1.78 .03959 -1.5599};
\path svg[xscale=1.7363cm,yscale=-1.7363cm] {M -.5253 -.83867 Q -.23264 -.90177 .05622 -.96486};
\path svg[xscale=1.7363cm,yscale=-1.7363cm] {M -1 -.5 Q -.73937 -.66934 -.5253 -.83867};
\path svg[xscale=1.7363cm,yscale=-1.7363cm] {M .05622 -.96486 Q .31163 -.78545 .59453 -.60603};
\path svg[xscale=1.7363cm,yscale=-1.7363cm] {M -.55984 -1.3817 Q -.52498 -1.1102 -.5253 -.83867};
\path svg[xscale=1.7363cm,yscale=-1.7363cm] {M -.55984 -1.3817 Q -.27163 -1.4708 .03959 -1.5599};
\path svg[xscale=1.7363cm,yscale=-1.7363cm] {M 1 -1 Q .79032 -.80302 .59453 -.60603};
\path svg[xscale=1.7363cm,yscale=-1.7363cm] {M -1 -1.5 Q -.77265 -1.4408 -.55984 -1.3817};
\path svg[xscale=1.7363cm,yscale=-1.7363cm] {M .03959 -1.5599 Q .04648 -1.2624 .05622 -.96486};
\end{scope}
\path[draw=blue,line width=.01405cm] (-1,.4) -- (-1,-2.4) -- (1,-2.4) -- (1,.4) -- cycle;
\end{scope}
\end{scope}
\end{tikzpicture}
}
\expandafter\def\csname e8diagram:diag2_3_13\endcsname{%
\begin{tikzpicture}[execute at begin picture={\clip (-1.24,-2.736) rectangle ++(2.48,3.472);},x={(1.7363cm,0cm)},y={(0cm,-1.7363cm)},baseline={([yshift=-\the\dimexpr\fontdimen22\textfont2\relax]current  bounding  box.center)},line join=round]
\begin{scope}
\begin{scope}
\path[clip] (-1,.4) -- (-1,-2.4) -- (1,-2.4) -- (1,.4) -- cycle;
\path[fill={rgb,255:red,204;green,204;blue,255},draw=none] (-1,.4) -- (-1,-2.4) -- (1,-2.4) -- (1,.4) -- cycle;
\begin{scope}[every path/.append style={draw=red,line width=.05619cm}]
\path svg[xscale=1.7363cm,yscale=-1.7363cm] {M .05622 -1.0351 Q .31163 -1.2146 .59453 -1.394};
\path svg[xscale=1.7363cm,yscale=-1.7363cm] {M 1 0 Q .31956 -.22004 .03959 -.44009};
\path svg[xscale=1.7363cm,yscale=-1.7363cm] {M 1 -2 Q .628 -1.697 .59453 -1.394};
\path svg[xscale=1.7363cm,yscale=-1.7363cm] {M -.55984 -.61835 Q -.27163 -.52922 .03959 -.44009};
\path svg[xscale=1.7363cm,yscale=-1.7363cm] {M -1 -.5 Q -.77265 -.55917 -.55984 -.61835};
\path svg[xscale=1.7363cm,yscale=-1.7363cm] {M -.5253 -1.1613 Q -.52498 -.88984 -.55984 -.61835};
\path svg[xscale=1.7363cm,yscale=-1.7363cm] {M -.5253 -1.1613 Q -.23264 -1.0982 .05622 -1.0351};
\path svg[xscale=1.7363cm,yscale=-1.7363cm] {M 1 -1 Q .79032 -1.197 .59453 -1.394};
\path svg[xscale=1.7363cm,yscale=-1.7363cm] {M -1 -1.5 Q -.73937 -1.3307 -.5253 -1.1613};
\path svg[xscale=1.7363cm,yscale=-1.7363cm] {M .05622 -1.0351 Q .04648 -.73761 .03959 -.44009};
\end{scope}
\path[draw=blue,line width=.01405cm] (-1,.4) -- (-1,-2.4) -- (1,-2.4) -- (1,.4) -- cycle;
\end{scope}
\end{scope}
\end{tikzpicture}
}
\expandafter\def\csname e8diagram:diag2_3_14\endcsname{%
\begin{tikzpicture}[execute at begin picture={\clip (-1.24,-2.736) rectangle ++(2.48,3.472);},x={(1.7363cm,0cm)},y={(0cm,-1.7363cm)},baseline={([yshift=-\the\dimexpr\fontdimen22\textfont2\relax]current  bounding  box.center)},line join=round]
\begin{scope}
\begin{scope}
\path[clip] (-1,.4) -- (-1,-2.4) -- (1,-2.4) -- (1,.4) -- cycle;
\path[fill={rgb,255:red,204;green,204;blue,255},draw=none] (-1,.4) -- (-1,-2.4) -- (1,-2.4) -- (1,.4) -- cycle;
\begin{scope}[every path/.append style={draw=red,line width=.05619cm}]
\path svg[xscale=1.7363cm,yscale=-1.7363cm] {M -.35916 -.59035 Q -.06403 -.48185 .26814 -.37335};
\path svg[xscale=1.7363cm,yscale=-1.7363cm] {M 1 -2 Q .16452 -1.8143 -.23811 -1.6287};
\path svg[xscale=1.7363cm,yscale=-1.7363cm] {M 1 0 Q .47304 -.18668 .26814 -.37335};
\path svg[xscale=1.7363cm,yscale=-1.7363cm] {M -.01082 -1.0621 Q -.17425 -.82622 -.35916 -.59035};
\path svg[xscale=1.7363cm,yscale=-1.7363cm] {M -.01082 -1.0621 Q .25291 -.95589 .5211 -.8497};
\path svg[xscale=1.7363cm,yscale=-1.7363cm] {M -1 -.5 Q -.67218 -.54518 -.35916 -.59035};
\path svg[xscale=1.7363cm,yscale=-1.7363cm] {M .5211 -.8497 Q .39366 -.61152 .26814 -.37335};
\path svg[xscale=1.7363cm,yscale=-1.7363cm] {M 1 -1 Q .75988 -.92485 .5211 -.8497};
\path svg[xscale=1.7363cm,yscale=-1.7363cm] {M -.23811 -1.6287 Q -.10859 -1.3454 -.01082 -1.0621};
\path svg[xscale=1.7363cm,yscale=-1.7363cm] {M -1 -1.5 Q -.63266 -1.5643 -.23811 -1.6287};
\end{scope}
\path[draw=blue,line width=.01405cm] (-1,.4) -- (-1,-2.4) -- (1,-2.4) -- (1,.4) -- cycle;
\end{scope}
\end{scope}
\end{tikzpicture}
}
\expandafter\def\csname e8diagram:diag2_3_15\endcsname{%
\begin{tikzpicture}[execute at begin picture={\clip (-1.24,-2.736) rectangle ++(2.48,3.472);},x={(1.7363cm,0cm)},y={(0cm,-1.7363cm)},baseline={([yshift=-\the\dimexpr\fontdimen22\textfont2\relax]current  bounding  box.center)},line join=round]
\begin{scope}
\begin{scope}
\path[clip] (-1,.4) -- (-1,-2.4) -- (1,-2.4) -- (1,.4) -- cycle;
\path[fill={rgb,255:red,204;green,204;blue,255},draw=none] (-1,.4) -- (-1,-2.4) -- (1,-2.4) -- (1,.4) -- cycle;
\begin{scope}[every path/.append style={draw=red,line width=.05619cm}]
\path svg[xscale=1.7363cm,yscale=-1.7363cm] {M .20927 -1.5104 Q .36999 -1.2552 .53442 -1};
\path svg[xscale=1.7363cm,yscale=-1.7363cm] {M 1 0 Q .41397 -.24479 .20927 -.48958};
\path svg[xscale=1.7363cm,yscale=-1.7363cm] {M 1 -2 Q .41397 -1.7552 .20927 -1.5104};
\path svg[xscale=1.7363cm,yscale=-1.7363cm] {M -.44551 -.69816 Q -.13076 -.59387 .20927 -.48958};
\path svg[xscale=1.7363cm,yscale=-1.7363cm] {M -1 -.5 Q -.70829 -.59908 -.44551 -.69816};
\path svg[xscale=1.7363cm,yscale=-1.7363cm] {M .20927 -.48958 Q .36999 -.74479 .53442 -1};
\path svg[xscale=1.7363cm,yscale=-1.7363cm] {M -.44551 -1.3018 Q -.42841 -1 -.44551 -.69816};
\path svg[xscale=1.7363cm,yscale=-1.7363cm] {M -.44551 -1.3018 Q -.13076 -1.4061 .20927 -1.5104};
\path svg[xscale=1.7363cm,yscale=-1.7363cm] {M 1 -1 Q .76721 -1 .53442 -1};
\path svg[xscale=1.7363cm,yscale=-1.7363cm] {M -1 -1.5 Q -.70829 -1.4009 -.44551 -1.3018};
\end{scope}
\path[draw=blue,line width=.01405cm] (-1,.4) -- (-1,-2.4) -- (1,-2.4) -- (1,.4) -- cycle;
\end{scope}
\end{scope}
\end{tikzpicture}
}
\expandafter\def\csname e8diagram:diag2_3_2\endcsname{%
\begin{tikzpicture}[execute at begin picture={\clip (-1.24,-2.736) rectangle ++(2.48,3.472);},x={(1.7363cm,0cm)},y={(0cm,-1.7363cm)},baseline={([yshift=-\the\dimexpr\fontdimen22\textfont2\relax]current  bounding  box.center)},line join=round]
\begin{scope}
\begin{scope}
\path[clip] (-1,.4) -- (-1,-2.4) -- (1,-2.4) -- (1,.4) -- cycle;
\path[fill={rgb,255:red,204;green,204;blue,255},draw=none] (-1,.4) -- (-1,-2.4) -- (1,-2.4) -- (1,.4) -- cycle;
\begin{scope}[every path/.append style={draw=red,line width=.05619cm}]
\path svg[xscale=1.7363cm,yscale=-1.7363cm] {M 1 -2 Q -.01418 -1.6713 -.46204 -1.3426};
\path svg[xscale=1.7363cm,yscale=-1.7363cm] {M -1 -1.5 Q -.71978 -1.4213 -.46204 -1.3426};
\path svg[xscale=1.7363cm,yscale=-1.7363cm] {M 1 0 Q .29519 -.5 1 -1};
\path svg[xscale=1.7363cm,yscale=-1.7363cm] {M -1 -.5 Q -.6666 -.92132 -.46204 -1.3426};
\end{scope}
\path[draw=blue,line width=.01405cm] (-1,.4) -- (-1,-2.4) -- (1,-2.4) -- (1,.4) -- cycle;
\end{scope}
\end{scope}
\end{tikzpicture}
}
\expandafter\def\csname e8diagram:diag2_3_3\endcsname{%
\begin{tikzpicture}[execute at begin picture={\clip (-1.24,-2.736) rectangle ++(2.48,3.472);},x={(1.7363cm,0cm)},y={(0cm,-1.7363cm)},baseline={([yshift=-\the\dimexpr\fontdimen22\textfont2\relax]current  bounding  box.center)},line join=round]
\begin{scope}
\begin{scope}
\path[clip] (-1,.4) -- (-1,-2.4) -- (1,-2.4) -- (1,.4) -- cycle;
\path[fill={rgb,255:red,204;green,204;blue,255},draw=none] (-1,.4) -- (-1,-2.4) -- (1,-2.4) -- (1,.4) -- cycle;
\begin{scope}[every path/.append style={draw=red,line width=.05619cm}]
\path svg[xscale=1.7363cm,yscale=-1.7363cm] {M -1 -1.5 Q -.6666 -1.0787 -.46204 -.65736};
\path svg[xscale=1.7363cm,yscale=-1.7363cm] {M -1 -.5 Q -.71978 -.57868 -.46204 -.65736};
\path svg[xscale=1.7363cm,yscale=-1.7363cm] {M 1 -1 Q .29519 -1.5 1 -2};
\path svg[xscale=1.7363cm,yscale=-1.7363cm] {M 1 0 Q -.01418 -.32868 -.46204 -.65736};
\end{scope}
\path[draw=blue,line width=.01405cm] (-1,.4) -- (-1,-2.4) -- (1,-2.4) -- (1,.4) -- cycle;
\end{scope}
\end{scope}
\end{tikzpicture}
}
\expandafter\def\csname e8diagram:diag2_3_4\endcsname{%
\begin{tikzpicture}[execute at begin picture={\clip (-1.24,-2.736) rectangle ++(2.48,3.472);},x={(1.7363cm,0cm)},y={(0cm,-1.7363cm)},baseline={([yshift=-\the\dimexpr\fontdimen22\textfont2\relax]current  bounding  box.center)},line join=round]
\begin{scope}
\begin{scope}
\path[clip] (-1,.4) -- (-1,-2.4) -- (1,-2.4) -- (1,.4) -- cycle;
\path[fill={rgb,255:red,204;green,204;blue,255},draw=none] (-1,.4) -- (-1,-2.4) -- (1,-2.4) -- (1,.4) -- cycle;
\begin{scope}[every path/.append style={draw=red,line width=.05619cm}]
\path svg[xscale=1.7363cm,yscale=-1.7363cm] {M 1 -1 Q .59841 -.75032 .18583 -.50065};
\path svg[xscale=1.7363cm,yscale=-1.7363cm] {M -1 -1.5 Q -.32407 -1.75 1 -2};
\path svg[xscale=1.7363cm,yscale=-1.7363cm] {M -1 -.5 Q -.40699 -.50032 .18583 -.50065};
\path svg[xscale=1.7363cm,yscale=-1.7363cm] {M 1 0 Q .39777 -.25032 .18583 -.50065};
\end{scope}
\path[draw=blue,line width=.01405cm] (-1,.4) -- (-1,-2.4) -- (1,-2.4) -- (1,.4) -- cycle;
\end{scope}
\end{scope}
\end{tikzpicture}
}
\expandafter\def\csname e8diagram:diag2_3_5\endcsname{%
\begin{tikzpicture}[execute at begin picture={\clip (-1.24,-2.736) rectangle ++(2.48,3.472);},x={(1.7363cm,0cm)},y={(0cm,-1.7363cm)},baseline={([yshift=-\the\dimexpr\fontdimen22\textfont2\relax]current  bounding  box.center)},line join=round]
\begin{scope}
\begin{scope}
\path[clip] (-1,.4) -- (-1,-2.4) -- (1,-2.4) -- (1,.4) -- cycle;
\path[fill={rgb,255:red,204;green,204;blue,255},draw=none] (-1,.4) -- (-1,-2.4) -- (1,-2.4) -- (1,.4) -- cycle;
\begin{scope}[every path/.append style={draw=red,line width=.05619cm}]
\path svg[xscale=1.7363cm,yscale=-1.7363cm] {M 1 0 Q .2332 -.2168 -.10394 -.43359};
\path svg[xscale=1.7363cm,yscale=-1.7363cm] {M 1 -2 Q .2332 -1.7832 -.10394 -1.5664};
\path svg[xscale=1.7363cm,yscale=-1.7363cm] {M .33104 -1 Q .12519 -.7168 -.10394 -.43359};
\path svg[xscale=1.7363cm,yscale=-1.7363cm] {M -1 -.5 Q -.56011 -.4668 -.10394 -.43359};
\path svg[xscale=1.7363cm,yscale=-1.7363cm] {M 1 -1 Q .66552 -1 .33104 -1};
\path svg[xscale=1.7363cm,yscale=-1.7363cm] {M -.10394 -1.5664 Q .12519 -1.2832 .33104 -1};
\path svg[xscale=1.7363cm,yscale=-1.7363cm] {M -1 -1.5 Q -.56011 -1.5332 -.10394 -1.5664};
\end{scope}
\path[draw=blue,line width=.01405cm] (-1,.4) -- (-1,-2.4) -- (1,-2.4) -- (1,.4) -- cycle;
\end{scope}
\end{scope}
\end{tikzpicture}
}
\expandafter\def\csname e8diagram:diag2_3_6\endcsname{%
\begin{tikzpicture}[execute at begin picture={\clip (-1.24,-2.736) rectangle ++(2.48,3.472);},x={(1.7363cm,0cm)},y={(0cm,-1.7363cm)},baseline={([yshift=-\the\dimexpr\fontdimen22\textfont2\relax]current  bounding  box.center)},line join=round]
\begin{scope}
\begin{scope}
\path[clip] (-1,.4) -- (-1,-2.4) -- (1,-2.4) -- (1,.4) -- cycle;
\path[fill={rgb,255:red,204;green,204;blue,255},draw=none] (-1,.4) -- (-1,-2.4) -- (1,-2.4) -- (1,.4) -- cycle;
\begin{scope}[every path/.append style={draw=red,line width=.05619cm}]
\path svg[xscale=1.7363cm,yscale=-1.7363cm] {M 1 0 Q .59351 -.36005 .56754 -.7201};
\path svg[xscale=1.7363cm,yscale=-1.7363cm] {M 1 -2 Q .34166 -1.6619 .13516 -1.3238};
\path svg[xscale=1.7363cm,yscale=-1.7363cm] {M .13516 -1.3238 Q .33908 -1.022 .56754 -.7201};
\path svg[xscale=1.7363cm,yscale=-1.7363cm] {M -1 -.5 Q -.75245 -.79534 -.5899 -1.0907};
\path svg[xscale=1.7363cm,yscale=-1.7363cm] {M 1 -1 Q .78078 -.86005 .56754 -.7201};
\path svg[xscale=1.7363cm,yscale=-1.7363cm] {M -.5899 -1.0907 Q -.23354 -1.2072 .13516 -1.3238};
\path svg[xscale=1.7363cm,yscale=-1.7363cm] {M -1 -1.5 Q -.76749 -1.2953 -.5899 -1.0907};
\end{scope}
\path[draw=blue,line width=.01405cm] (-1,.4) -- (-1,-2.4) -- (1,-2.4) -- (1,.4) -- cycle;
\end{scope}
\end{scope}
\end{tikzpicture}
}
\expandafter\def\csname e8diagram:diag2_3_7\endcsname{%
\begin{tikzpicture}[execute at begin picture={\clip (-1.24,-2.736) rectangle ++(2.48,3.472);},x={(1.7363cm,0cm)},y={(0cm,-1.7363cm)},baseline={([yshift=-\the\dimexpr\fontdimen22\textfont2\relax]current  bounding  box.center)},line join=round]
\begin{scope}
\begin{scope}
\path[clip] (-1,.4) -- (-1,-2.4) -- (1,-2.4) -- (1,.4) -- cycle;
\path[fill={rgb,255:red,204;green,204;blue,255},draw=none] (-1,.4) -- (-1,-2.4) -- (1,-2.4) -- (1,.4) -- cycle;
\begin{scope}[every path/.append style={draw=red,line width=.05619cm}]
\path svg[xscale=1.7363cm,yscale=-1.7363cm] {M 1 0 Q .11491 -.23663 -.28605 -.47327};
\path svg[xscale=1.7363cm,yscale=-1.7363cm] {M 1 -2 Q .55172 -1.7045 .46414 -1.4089};
\path svg[xscale=1.7363cm,yscale=-1.7363cm] {M -1 -.5 Q -.64554 -.48663 -.28605 -.47327};
\path svg[xscale=1.7363cm,yscale=-1.7363cm] {M 1 -1 Q .72813 -1.2045 .46414 -1.4089};
\path svg[xscale=1.7363cm,yscale=-1.7363cm] {M -.27474 -1.1767 Q -.26452 -.82501 -.28605 -.47327};
\path svg[xscale=1.7363cm,yscale=-1.7363cm] {M -.27474 -1.1767 Q .08284 -1.2928 .46414 -1.4089};
\path svg[xscale=1.7363cm,yscale=-1.7363cm] {M -1 -1.5 Q -.61055 -1.3384 -.27474 -1.1767};
\end{scope}
\path[draw=blue,line width=.01405cm] (-1,.4) -- (-1,-2.4) -- (1,-2.4) -- (1,.4) -- cycle;
\end{scope}
\end{scope}
\end{tikzpicture}
}
\expandafter\def\csname e8diagram:diag2_3_8\endcsname{%
\begin{tikzpicture}[execute at begin picture={\clip (-1.24,-2.736) rectangle ++(2.48,3.472);},x={(1.7363cm,0cm)},y={(0cm,-1.7363cm)},baseline={([yshift=-\the\dimexpr\fontdimen22\textfont2\relax]current  bounding  box.center)},line join=round]
\begin{scope}
\begin{scope}
\path[clip] (-1,.4) -- (-1,-2.4) -- (1,-2.4) -- (1,.4) -- cycle;
\path[fill={rgb,255:red,204;green,204;blue,255},draw=none] (-1,.4) -- (-1,-2.4) -- (1,-2.4) -- (1,.4) -- cycle;
\begin{scope}[every path/.append style={draw=red,line width=.05619cm}]
\path svg[xscale=1.7363cm,yscale=-1.7363cm] {M 1 0 Q .55172 -.29555 .46414 -.59109};
\path svg[xscale=1.7363cm,yscale=-1.7363cm] {M 1 -2 Q .11491 -1.7634 -.28605 -1.5267};
\path svg[xscale=1.7363cm,yscale=-1.7363cm] {M -.27474 -.82326 Q .08284 -.70717 .46414 -.59109};
\path svg[xscale=1.7363cm,yscale=-1.7363cm] {M -1 -.5 Q -.61055 -.66163 -.27474 -.82326};
\path svg[xscale=1.7363cm,yscale=-1.7363cm] {M 1 -1 Q .72813 -.79555 .46414 -.59109};
\path svg[xscale=1.7363cm,yscale=-1.7363cm] {M -.28605 -1.5267 Q -.26452 -1.175 -.27474 -.82326};
\path svg[xscale=1.7363cm,yscale=-1.7363cm] {M -1 -1.5 Q -.64554 -1.5134 -.28605 -1.5267};
\end{scope}
\path[draw=blue,line width=.01405cm] (-1,.4) -- (-1,-2.4) -- (1,-2.4) -- (1,.4) -- cycle;
\end{scope}
\end{scope}
\end{tikzpicture}
}
\expandafter\def\csname e8diagram:diag2_3_9\endcsname{%
\begin{tikzpicture}[execute at begin picture={\clip (-1.24,-2.736) rectangle ++(2.48,3.472);},x={(1.7363cm,0cm)},y={(0cm,-1.7363cm)},baseline={([yshift=-\the\dimexpr\fontdimen22\textfont2\relax]current  bounding  box.center)},line join=round]
\begin{scope}
\begin{scope}
\path[clip] (-1,.4) -- (-1,-2.4) -- (1,-2.4) -- (1,.4) -- cycle;
\path[fill={rgb,255:red,204;green,204;blue,255},draw=none] (-1,.4) -- (-1,-2.4) -- (1,-2.4) -- (1,.4) -- cycle;
\begin{scope}[every path/.append style={draw=red,line width=.05619cm}]
\path svg[xscale=1.7363cm,yscale=-1.7363cm] {M 1 0 Q .34166 -.3381 .13516 -.67619};
\path svg[xscale=1.7363cm,yscale=-1.7363cm] {M 1 -2 Q .59351 -1.64 .56754 -1.2799};
\path svg[xscale=1.7363cm,yscale=-1.7363cm] {M .13516 -.67619 Q .33908 -.97805 .56754 -1.2799};
\path svg[xscale=1.7363cm,yscale=-1.7363cm] {M -1 -.5 Q -.76749 -.70466 -.5899 -.90932};
\path svg[xscale=1.7363cm,yscale=-1.7363cm] {M 1 -1 Q .78078 -1.14 .56754 -1.2799};
\path svg[xscale=1.7363cm,yscale=-1.7363cm] {M -.5899 -.90932 Q -.23354 -.79276 .13516 -.67619};
\path svg[xscale=1.7363cm,yscale=-1.7363cm] {M -1 -1.5 Q -.75245 -1.2047 -.5899 -.90932};
\end{scope}
\path[draw=blue,line width=.01405cm] (-1,.4) -- (-1,-2.4) -- (1,-2.4) -- (1,.4) -- cycle;
\end{scope}
\end{scope}
\end{tikzpicture}
}
\expandafter\def\csname e8diagram:diag3_0_0\endcsname{%
\begin{tikzpicture}[execute at begin picture={\clip (-1.24,-2.736) rectangle ++(2.48,3.472);},x={(1.7363cm,0cm)},y={(0cm,-1.7363cm)},baseline={([yshift=-\the\dimexpr\fontdimen22\textfont2\relax]current  bounding  box.center)},line join=round]
\begin{scope}
\begin{scope}
\path[clip] (-1,.4) -- (-1,-2.4) -- (1,-2.4) -- (1,.4) -- cycle;
\path[fill={rgb,255:red,204;green,204;blue,255},draw=none] (-1,.4) -- (-1,-2.4) -- (1,-2.4) -- (1,.4) -- cycle;
\begin{scope}[every path/.append style={draw=red,line width=.05619cm}]
\path svg[xscale=1.7363cm,yscale=-1.7363cm] {M -1 -1 Q -.56864 -1 -.13729 -1};
\path svg[xscale=1.7363cm,yscale=-1.7363cm] {M -1 0 Q -.10611 -.5 -.13729 -1};
\path svg[xscale=1.7363cm,yscale=-1.7363cm] {M -1 -2 Q -.10611 -1.5 -.13729 -1};
\end{scope}
\path[draw=blue,line width=.01405cm] (-1,.4) -- (-1,-2.4) -- (1,-2.4) -- (1,.4) -- cycle;
\end{scope}
\end{scope}
\end{tikzpicture}
}
\expandafter\def\csname e8diagram:diag3_1_0\endcsname{%
\begin{tikzpicture}[execute at begin picture={\clip (-1.24,-2.736) rectangle ++(2.48,3.472);},x={(1.7363cm,0cm)},y={(0cm,-1.7363cm)},baseline={([yshift=-\the\dimexpr\fontdimen22\textfont2\relax]current  bounding  box.center)},line join=round]
\begin{scope}
\begin{scope}
\path[clip] (-1,.4) -- (-1,-2.4) -- (1,-2.4) -- (1,.4) -- cycle;
\path[fill={rgb,255:red,204;green,204;blue,255},draw=none] (-1,.4) -- (-1,-2.4) -- (1,-2.4) -- (1,.4) -- cycle;
\begin{scope}[every path/.append style={draw=red,line width=.05619cm}]
\path svg[xscale=1.7363cm,yscale=-1.7363cm] {M -1 0 Q -.42711 -.28189 -.25107 -.56377};
\path svg[xscale=1.7363cm,yscale=-1.7363cm] {M -.25107 -1.4362 Q -.40942 -1.2181 -.57211 -1};
\path svg[xscale=1.7363cm,yscale=-1.7363cm] {M -.25107 -1.4362 Q .0215 -1.2181 .27262 -1};
\path svg[xscale=1.7363cm,yscale=-1.7363cm] {M .27262 -1 Q .0215 -.78189 -.25107 -.56377};
\path svg[xscale=1.7363cm,yscale=-1.7363cm] {M -1 -1 Q -.78605 -1 -.57211 -1};
\path svg[xscale=1.7363cm,yscale=-1.7363cm] {M 1 -1 Q .63631 -1 .27262 -1};
\path svg[xscale=1.7363cm,yscale=-1.7363cm] {M -.57211 -1 Q -.40942 -.78189 -.25107 -.56377};
\path svg[xscale=1.7363cm,yscale=-1.7363cm] {M -1 -2 Q -.42711 -1.7181 -.25107 -1.4362};
\end{scope}
\path[draw=blue,line width=.01405cm] (-1,.4) -- (-1,-2.4) -- (1,-2.4) -- (1,.4) -- cycle;
\end{scope}
\end{scope}
\end{tikzpicture}
}
\expandafter\def\csname e8diagram:diag3_1_1\endcsname{%
\begin{tikzpicture}[execute at begin picture={\clip (-1.24,-2.736) rectangle ++(2.48,3.472);},x={(1.7363cm,0cm)},y={(0cm,-1.7363cm)},baseline={([yshift=-\the\dimexpr\fontdimen22\textfont2\relax]current  bounding  box.center)},line join=round]
\begin{scope}
\begin{scope}
\path[clip] (-1,.4) -- (-1,-2.4) -- (1,-2.4) -- (1,.4) -- cycle;
\path[fill={rgb,255:red,204;green,204;blue,255},draw=none] (-1,.4) -- (-1,-2.4) -- (1,-2.4) -- (1,.4) -- cycle;
\begin{scope}[every path/.append style={draw=red,line width=.05619cm}]
\path svg[xscale=1.7363cm,yscale=-1.7363cm] {M .08456 -1.2759 Q -.19139 -1 -.47997 -.72413};
\path svg[xscale=1.7363cm,yscale=-1.7363cm] {M -1 -2 Q -.20545 -1.6379 .08456 -1.2759};
\path svg[xscale=1.7363cm,yscale=-1.7363cm] {M -1 -1 Q -.73814 -.86206 -.47997 -.72413};
\path svg[xscale=1.7363cm,yscale=-1.7363cm] {M 1 -1 Q .54523 -1.1379 .08456 -1.2759};
\path svg[xscale=1.7363cm,yscale=-1.7363cm] {M -1 0 Q -.54086 -.36206 -.47997 -.72413};
\end{scope}
\path[draw=blue,line width=.01405cm] (-1,.4) -- (-1,-2.4) -- (1,-2.4) -- (1,.4) -- cycle;
\end{scope}
\end{scope}
\end{tikzpicture}
}
\expandafter\def\csname e8diagram:diag3_1_2\endcsname{%
\begin{tikzpicture}[execute at begin picture={\clip (-1.24,-2.736) rectangle ++(2.48,3.472);},x={(1.7363cm,0cm)},y={(0cm,-1.7363cm)},baseline={([yshift=-\the\dimexpr\fontdimen22\textfont2\relax]current  bounding  box.center)},line join=round]
\begin{scope}
\begin{scope}
\path[clip] (-1,.4) -- (-1,-2.4) -- (1,-2.4) -- (1,.4) -- cycle;
\path[fill={rgb,255:red,204;green,204;blue,255},draw=none] (-1,.4) -- (-1,-2.4) -- (1,-2.4) -- (1,.4) -- cycle;
\begin{scope}[every path/.append style={draw=red,line width=.05619cm}]
\path svg[xscale=1.7363cm,yscale=-1.7363cm] {M 1 -1 Q .54523 -.86206 .08456 -.72412};
\path svg[xscale=1.7363cm,yscale=-1.7363cm] {M -.47997 -1.2759 Q -.19139 -1 .08456 -.72412};
\path svg[xscale=1.7363cm,yscale=-1.7363cm] {M -1 -2 Q -.54086 -1.6379 -.47997 -1.2759};
\path svg[xscale=1.7363cm,yscale=-1.7363cm] {M -1 -1 Q -.73814 -1.1379 -.47997 -1.2759};
\path svg[xscale=1.7363cm,yscale=-1.7363cm] {M -1 0 Q -.20545 -.36206 .08456 -.72412};
\end{scope}
\path[draw=blue,line width=.01405cm] (-1,.4) -- (-1,-2.4) -- (1,-2.4) -- (1,.4) -- cycle;
\end{scope}
\end{scope}
\end{tikzpicture}
}
\expandafter\def\csname e8diagram:diag3_1_3\endcsname{%
\begin{tikzpicture}[execute at begin picture={\clip (-1.24,-2.736) rectangle ++(2.48,3.472);},x={(1.7363cm,0cm)},y={(0cm,-1.7363cm)},baseline={([yshift=-\the\dimexpr\fontdimen22\textfont2\relax]current  bounding  box.center)},line join=round]
\begin{scope}
\begin{scope}
\path[clip] (-1,.4) -- (-1,-2.4) -- (1,-2.4) -- (1,.4) -- cycle;
\path[fill={rgb,255:red,204;green,204;blue,255},draw=none] (-1,.4) -- (-1,-2.4) -- (1,-2.4) -- (1,.4) -- cycle;
\begin{scope}[every path/.append style={draw=red,line width=.05619cm}]
\path svg[xscale=1.7363cm,yscale=-1.7363cm] {M -1 0 Q .32749 -.5 1 -1};
\path svg[xscale=1.7363cm,yscale=-1.7363cm] {M -1 -2 Q -.29519 -1.5 -1 -1};
\end{scope}
\path[draw=blue,line width=.01405cm] (-1,.4) -- (-1,-2.4) -- (1,-2.4) -- (1,.4) -- cycle;
\end{scope}
\end{scope}
\end{tikzpicture}
}
\expandafter\def\csname e8diagram:diag3_1_4\endcsname{%
\begin{tikzpicture}[execute at begin picture={\clip (-1.24,-2.736) rectangle ++(2.48,3.472);},x={(1.7363cm,0cm)},y={(0cm,-1.7363cm)},baseline={([yshift=-\the\dimexpr\fontdimen22\textfont2\relax]current  bounding  box.center)},line join=round]
\begin{scope}
\begin{scope}
\path[clip] (-1,.4) -- (-1,-2.4) -- (1,-2.4) -- (1,.4) -- cycle;
\path[fill={rgb,255:red,204;green,204;blue,255},draw=none] (-1,.4) -- (-1,-2.4) -- (1,-2.4) -- (1,.4) -- cycle;
\begin{scope}[every path/.append style={draw=red,line width=.05619cm}]
\path svg[xscale=1.7363cm,yscale=-1.7363cm] {M -1 -1 Q -.29519 -.5 -1 0};
\path svg[xscale=1.7363cm,yscale=-1.7363cm] {M -1 -2 Q .32749 -1.5 1 -1};
\end{scope}
\path[draw=blue,line width=.01405cm] (-1,.4) -- (-1,-2.4) -- (1,-2.4) -- (1,.4) -- cycle;
\end{scope}
\end{scope}
\end{tikzpicture}
}
\expandafter\def\csname e8diagram:diag3_2_proj1\endcsname{%
\begin{tikzpicture}[execute at begin picture={\clip (-1.24,-2.736) rectangle ++(2.48,3.472);},x={(1.7363cm,0cm)},y={(0cm,-1.7363cm)},baseline={([yshift=-\the\dimexpr\fontdimen22\textfont2\relax]current  bounding  box.center)},line join=round]
\begin{scope}
\begin{scope}
\path[clip] (-1,.4) -- (-1,-2.4) -- (1,-2.4) -- (1,.4) -- cycle;
\path[fill={rgb,255:red,204;green,204;blue,255},draw=none] (-1,.4) -- (-1,-2.4) -- (1,-2.4) -- (1,.4) -- cycle;
\begin{scope}[every path/.append style={draw=red,line width=.05619cm}]
\draw (-1,0) -- (-.5,-.5) (.5,-.5) -- (1,-.5);
\draw(-1,-1) -- (-.75,-1.5) -- (-.5,-1.5) (.5,-1.5) -- (1,-1.5);
\draw(-1,-2) -- (-.75,-1.5);
\end{scope}
\path[draw=blue,line width=.01405cm] (-1,.4) -- (-1,-2.4) -- (1,-2.4) -- (1,.4) -- cycle;
\end{scope}
  \filldraw [fill=sienna, draw=black,thin] (-.5,-.3) rectangle (.5,-1.7);
\node at (0,-1) {$\scriptstyle \pi_R$};
\end{scope}
\end{tikzpicture}
}
\expandafter\def\csname e8diagram:diag3_2_proj2\endcsname{%
\begin{tikzpicture}[execute at begin picture={\clip (-1.24,-2.736) rectangle ++(2.48,3.472);},x={(1.7363cm,0cm)},y={(0cm,-1.7363cm)},baseline={([yshift=-\the\dimexpr\fontdimen22\textfont2\relax]current  bounding  box.center)},line join=round]
\begin{scope}
\begin{scope}
\path[clip] (-1,.4) -- (-1,-2.4) -- (1,-2.4) -- (1,.4) -- cycle;
\path[fill={rgb,255:red,204;green,204;blue,255},draw=none] (-1,.4) -- (-1,-2.4) -- (1,-2.4) -- (1,.4) -- cycle;
\begin{scope}[every path/.append style={draw=red,line width=.05619cm}]
\draw (-1,0) -- (-.75,-.5) -- (-.5,-.5) (.5,-.5) -- (1,-.5);
\draw(-1,-1) -- (-.75,-.5) (.5,-1.5) -- (1,-1.5);
\draw(-1,-2) -- (-.5,-1.5);
\end{scope}
\path[draw=blue,line width=.01405cm] (-1,.4) -- (-1,-2.4) -- (1,-2.4) -- (1,.4) -- cycle;
\end{scope}
  \filldraw [fill=sienna, draw=black,thin] (-.5,-.3) rectangle (.5,-1.7);
\node at (0,-1) {$\scriptstyle \pi_R$};
\end{scope}
\end{tikzpicture}
}
\expandafter\def\csname e8diagram:dilute-ex\endcsname{%
\begin{tikzpicture}[execute at begin picture={\clip (-1.24,-3.856) rectangle ++(2.48,4.712);},x={(1.3913cm,0cm)},y={(0cm,-1.3913cm)},baseline={([yshift=-\the\dimexpr\fontdimen22\textfont2\relax]current  bounding  box.center)},line join=round]
\begin{scope}
\begin{scope}
\path[clip] (-1,.4) -- (-1,-3.4) -- (1,-3.4) -- (1,.4) -- cycle;
\path[fill={rgb,255:red,204;green,204;blue,255},draw=none] (-1,.4) -- (-1,-3.4) -- (1,-3.4) -- (1,.4) -- cycle;
\begin{scope}[every path/.append style={draw=red,line width=.05619cm}]
\path svg[xscale=1.3913cm,yscale=-1.3913cm] {M -1 -1 Q -.66504 -1.254 -.36435 -1.508};
\path svg[xscale=1.3913cm,yscale=-1.3913cm] {M 1 -.5 Q .63553 -1.0053 .42818 -1.5107};
\path svg[xscale=1.3913cm,yscale=-1.3913cm] {M 1 -2.5 Q .637 -2.0053 .42818 -1.5107};
\path svg[xscale=1.3913cm,yscale=-1.3913cm] {M -1 -2 Q -.84252 -2.1541 -.68117 -2.3083};
\path svg[xscale=1.3913cm,yscale=-1.3913cm] {M -.36435 -1.508 Q .03191 -1.5093 .42818 -1.5107};
\path svg[xscale=1.3913cm,yscale=-1.3913cm] {M -1 -3 Q -.67562 -2.6541 -.68117 -2.3083};
\path svg[xscale=1.3913cm,yscale=-1.3913cm] {M -.68117 -2.3083 Q -.49239 -1.9081 -.36435 -1.508};
\end{scope}
\path[draw=blue,line width=.01405cm] (-1,.4) -- (-1,-3.4) -- (1,-3.4) -- (1,.4) -- cycle;
\end{scope}
\begin{scope}
\node[fill=white,rectangle,inner sep=2pt,draw=black,scale=.99114,line width=.00281cm] at (-1,0) {};
\node[fill=white,rectangle,inner sep=2pt,draw=black,scale=.99114,line width=.00281cm] at (1,-1.5) {};
\end{scope}
\end{scope}
\end{tikzpicture}
}
\expandafter\def\csname e8diagram:dilute0_1_0\endcsname{%
\begin{tikzpicture}[execute at begin picture={\clip (-1.24,-.496) rectangle ++(2.48,.992);},x={(2.7736cm,0cm)},y={(0cm,-2.7736cm)},baseline={([yshift=-\the\dimexpr\fontdimen22\textfont2\relax]current  bounding  box.center)},line join=round]
\begin{scope}
\begin{scope}
\path[clip] (-1,.4) -- (-1,-.4) -- (1,-.4) -- (1,.4) -- cycle;
\path[fill={rgb,255:red,204;green,204;blue,255},draw=none] (-1,.4) -- (-1,-.4) -- (1,-.4) -- (1,.4) -- cycle;
\begin{scope}[every path/.append style={draw=red,line width=.05619cm}]
\end{scope}
\path[draw=blue,line width=.01405cm] (-1,.4) -- (-1,-.4) -- (1,-.4) -- (1,.4) -- cycle;
\end{scope}
\begin{scope}
\node[fill=white,rectangle,inner sep=2pt,draw=black,scale=.99115,line width=.00281cm] at (1,0) {};
\end{scope}
\end{scope}
\end{tikzpicture}
}
\expandafter\def\csname e8diagram:dilute1_1_0\endcsname{%
\begin{tikzpicture}[execute at begin picture={\clip (-1.24,-.496) rectangle ++(2.48,.992);},x={(2.7736cm,0cm)},y={(0cm,-2.7736cm)},baseline={([yshift=-\the\dimexpr\fontdimen22\textfont2\relax]current  bounding  box.center)},line join=round]
\begin{scope}
\begin{scope}
\path[clip] (-1,.4) -- (-1,-.4) -- (1,-.4) -- (1,.4) -- cycle;
\path[fill={rgb,255:red,204;green,204;blue,255},draw=none] (-1,.4) -- (-1,-.4) -- (1,-.4) -- (1,.4) -- cycle;
\begin{scope}[every path/.append style={draw=red,line width=.05619cm}]
\path svg[xscale=2.7736cm,yscale=-2.7736cm] {M -1 0 Q 0 0 1 0};
\end{scope}
\path[draw=blue,line width=.01405cm] (-1,.4) -- (-1,-.4) -- (1,-.4) -- (1,.4) -- cycle;
\end{scope}
\end{scope}
\end{tikzpicture}
}
\expandafter\def\csname e8diagram:dilute1_1_1\endcsname{%
\begin{tikzpicture}[execute at begin picture={\clip (-1.24,-.496) rectangle ++(2.48,.992);},x={(2.7736cm,0cm)},y={(0cm,-2.7736cm)},baseline={([yshift=-\the\dimexpr\fontdimen22\textfont2\relax]current  bounding  box.center)},line join=round]
\begin{scope}
\begin{scope}
\path[clip] (-1,.4) -- (-1,-.4) -- (1,-.4) -- (1,.4) -- cycle;
\path[fill={rgb,255:red,204;green,204;blue,255},draw=none] (-1,.4) -- (-1,-.4) -- (1,-.4) -- (1,.4) -- cycle;
\begin{scope}[every path/.append style={draw=red,line width=.05619cm}]
\end{scope}
\path[draw=blue,line width=.01405cm] (-1,.4) -- (-1,-.4) -- (1,-.4) -- (1,.4) -- cycle;
\end{scope}
\begin{scope}
\node[fill=white,rectangle,inner sep=2pt,draw=black,scale=.99115,line width=.00281cm] at (-1,0) {};
\node[fill=white,rectangle,inner sep=2pt,draw=black,scale=.99115,line width=.00281cm] at (1,0) {};
\end{scope}
\end{scope}
\end{tikzpicture}
}
\expandafter\def\csname e8diagram:dilute1_2_0\endcsname{%
\begin{tikzpicture}[execute at begin picture={\clip (-1.24,-1.616) rectangle ++(2.48,2.232);},x={(2.2204cm,0cm)},y={(0cm,-2.2204cm)},baseline={([yshift=-\the\dimexpr\fontdimen22\textfont2\relax]current  bounding  box.center)},line join=round]
\begin{scope}
\begin{scope}
\path[clip] (-1,.4) -- (-1,-1.4) -- (1,-1.4) -- (1,.4) -- cycle;
\path[fill={rgb,255:red,204;green,204;blue,255},draw=none] (-1,.4) -- (-1,-1.4) -- (1,-1.4) -- (1,.4) -- cycle;
\begin{scope}[every path/.append style={draw=red,line width=.05619cm}]
\path svg[xscale=2.2204cm,yscale=-2.2204cm] {M 1 0 Q .32218 -.25 .03299 -.5};
\path svg[xscale=2.2204cm,yscale=-2.2204cm] {M 1 -1 Q .32218 -.75 .03299 -.5};
\path svg[xscale=2.2204cm,yscale=-2.2204cm] {M -1 -.5 Q -.48351 -.5 .03299 -.5};
\end{scope}
\path[draw=blue,line width=.01405cm] (-1,.4) -- (-1,-1.4) -- (1,-1.4) -- (1,.4) -- cycle;
\end{scope}
\end{scope}
\end{tikzpicture}
}
\expandafter\def\csname e8diagram:dilute1_2_1\endcsname{%
\begin{tikzpicture}[execute at begin picture={\clip (-1.24,-1.616) rectangle ++(2.48,2.232);},x={(2.2204cm,0cm)},y={(0cm,-2.2204cm)},baseline={([yshift=-\the\dimexpr\fontdimen22\textfont2\relax]current  bounding  box.center)},line join=round]
\begin{scope}
\begin{scope}
\path[clip] (-1,.4) -- (-1,-1.4) -- (1,-1.4) -- (1,.4) -- cycle;
\path[fill={rgb,255:red,204;green,204;blue,255},draw=none] (-1,.4) -- (-1,-1.4) -- (1,-1.4) -- (1,.4) -- cycle;
\begin{scope}[every path/.append style={draw=red,line width=.05619cm}]
\path svg[xscale=2.2204cm,yscale=-2.2204cm] {M -1 -.5 Q -.24598 -.25 1 0};
\end{scope}
\path[draw=blue,line width=.01405cm] (-1,.4) -- (-1,-1.4) -- (1,-1.4) -- (1,.4) -- cycle;
\end{scope}
\begin{scope}
\node[fill=white,rectangle,inner sep=2pt,draw=black,scale=.99115,line width=.00281cm] at (1,-1) {};
\end{scope}
\end{scope}
\end{tikzpicture}
}
\expandafter\def\csname e8diagram:dilute1_2_2\endcsname{%
\begin{tikzpicture}[execute at begin picture={\clip (-1.24,-1.616) rectangle ++(2.48,2.232);},x={(2.2204cm,0cm)},y={(0cm,-2.2204cm)},baseline={([yshift=-\the\dimexpr\fontdimen22\textfont2\relax]current  bounding  box.center)},line join=round]
\begin{scope}
\begin{scope}
\path[clip] (-1,.4) -- (-1,-1.4) -- (1,-1.4) -- (1,.4) -- cycle;
\path[fill={rgb,255:red,204;green,204;blue,255},draw=none] (-1,.4) -- (-1,-1.4) -- (1,-1.4) -- (1,.4) -- cycle;
\begin{scope}[every path/.append style={draw=red,line width=.05619cm}]
\path svg[xscale=2.2204cm,yscale=-2.2204cm] {M -1 -.5 Q -.24598 -.75 1 -1};
\end{scope}
\path[draw=blue,line width=.01405cm] (-1,.4) -- (-1,-1.4) -- (1,-1.4) -- (1,.4) -- cycle;
\end{scope}
\begin{scope}
\node[fill=white,rectangle,inner sep=2pt,draw=black,scale=.99115,line width=.00281cm] at (1,0) {};
\end{scope}
\end{scope}
\end{tikzpicture}
}
\expandafter\def\csname e8diagram:dilute1_2_3\endcsname{%
\begin{tikzpicture}[execute at begin picture={\clip (-1.24,-1.616) rectangle ++(2.48,2.232);},x={(2.2204cm,0cm)},y={(0cm,-2.2204cm)},baseline={([yshift=-\the\dimexpr\fontdimen22\textfont2\relax]current  bounding  box.center)},line join=round]
\begin{scope}
\begin{scope}
\path[clip] (-1,.4) -- (-1,-1.4) -- (1,-1.4) -- (1,.4) -- cycle;
\path[fill={rgb,255:red,204;green,204;blue,255},draw=none] (-1,.4) -- (-1,-1.4) -- (1,-1.4) -- (1,.4) -- cycle;
\begin{scope}[every path/.append style={draw=red,line width=.05619cm}]
\path svg[xscale=2.2204cm,yscale=-2.2204cm] {M 1 0 Q .05277 -.5 1 -1};
\end{scope}
\path[draw=blue,line width=.01405cm] (-1,.4) -- (-1,-1.4) -- (1,-1.4) -- (1,.4) -- cycle;
\end{scope}
\begin{scope}
\node[fill=white,rectangle,inner sep=2pt,draw=black,scale=.99115,line width=.00281cm] at (-1,-.5) {};
\end{scope}
\end{scope}
\end{tikzpicture}
}
\expandafter\def\csname e8diagram:dilute1_2_4\endcsname{%
\begin{tikzpicture}[execute at begin picture={\clip (-1.24,-1.616) rectangle ++(2.48,2.232);},x={(2.2204cm,0cm)},y={(0cm,-2.2204cm)},baseline={([yshift=-\the\dimexpr\fontdimen22\textfont2\relax]current  bounding  box.center)},line join=round]
\begin{scope}
\begin{scope}
\path[clip] (-1,.4) -- (-1,-1.4) -- (1,-1.4) -- (1,.4) -- cycle;
\path[fill={rgb,255:red,204;green,204;blue,255},draw=none] (-1,.4) -- (-1,-1.4) -- (1,-1.4) -- (1,.4) -- cycle;
\begin{scope}[every path/.append style={draw=red,line width=.05619cm}]
\end{scope}
\path[draw=blue,line width=.01405cm] (-1,.4) -- (-1,-1.4) -- (1,-1.4) -- (1,.4) -- cycle;
\end{scope}
\begin{scope}
\node[fill=white,rectangle,inner sep=2pt,draw=black,scale=.99115,line width=.00281cm] at (-1,-.5) {};
\node[fill=white,rectangle,inner sep=2pt,draw=black,scale=.99115,line width=.00281cm] at (1,0) {};
\node[fill=white,rectangle,inner sep=2pt,draw=black,scale=.99115,line width=.00281cm] at (1,-1) {};
\end{scope}
\end{scope}
\end{tikzpicture}
}
\expandafter\def\csname e8diagram:dilute2_0_0\endcsname{%
\begin{tikzpicture}[execute at begin picture={\clip (-1.24,-1.616) rectangle ++(2.48,2.232);},x={(2.2204cm,0cm)},y={(0cm,-2.2204cm)},baseline={([yshift=-\the\dimexpr\fontdimen22\textfont2\relax]current  bounding  box.center)},line join=round]
\begin{scope}
\begin{scope}
\path[clip] (-1,.4) -- (-1,-1.4) -- (1,-1.4) -- (1,.4) -- cycle;
\path[fill={rgb,255:red,204;green,204;blue,255},draw=none] (-1,.4) -- (-1,-1.4) -- (1,-1.4) -- (1,.4) -- cycle;
\begin{scope}[every path/.append style={draw=red,line width=.05619cm}]
\end{scope}
\path[draw=blue,line width=.01405cm] (-1,.4) -- (-1,-1.4) -- (1,-1.4) -- (1,.4) -- cycle;
\end{scope}
\begin{scope}
\node[fill=white,rectangle,inner sep=2pt,draw=black,scale=.99115,line width=.00281cm] at (-1,-1) {};
\node[fill=white,rectangle,inner sep=2pt,draw=black,scale=.99115,line width=.00281cm] at (-1,0) {};
\end{scope}
\end{scope}
\end{tikzpicture}
}
\expandafter\def\csname e8diagram:dilute2_0_1\endcsname{%
\begin{tikzpicture}[execute at begin picture={\clip (-1.24,-1.616) rectangle ++(2.48,2.232);},x={(2.2204cm,0cm)},y={(0cm,-2.2204cm)},baseline={([yshift=-\the\dimexpr\fontdimen22\textfont2\relax]current  bounding  box.center)},line join=round]
\begin{scope}
\begin{scope}
\path[clip] (-1,.4) -- (-1,-1.4) -- (1,-1.4) -- (1,.4) -- cycle;
\path[fill={rgb,255:red,204;green,204;blue,255},draw=none] (-1,.4) -- (-1,-1.4) -- (1,-1.4) -- (1,.4) -- cycle;
\begin{scope}[every path/.append style={draw=red,line width=.05619cm}]
\path svg[xscale=2.2204cm,yscale=-2.2204cm] {M -1 -1 Q .39446 -.5 -1 0};
\end{scope}
\path[draw=blue,line width=.01405cm] (-1,.4) -- (-1,-1.4) -- (1,-1.4) -- (1,.4) -- cycle;
\end{scope}
\end{scope}
\end{tikzpicture}
}
\expandafter\def\csname e8diagram:dilute2_1_0\endcsname{%
\begin{tikzpicture}[execute at begin picture={\clip (-1.24,-1.616) rectangle ++(2.48,2.232);},x={(2.2204cm,0cm)},y={(0cm,-2.2204cm)},baseline={([yshift=-\the\dimexpr\fontdimen22\textfont2\relax]current  bounding  box.center)},line join=round]
\begin{scope}
\begin{scope}
\path[clip] (-1,.4) -- (-1,-1.4) -- (1,-1.4) -- (1,.4) -- cycle;
\path[fill={rgb,255:red,204;green,204;blue,255},draw=none] (-1,.4) -- (-1,-1.4) -- (1,-1.4) -- (1,.4) -- cycle;
\begin{scope}[every path/.append style={draw=red,line width=.05619cm}]
\path svg[xscale=2.2204cm,yscale=-2.2204cm] {M -1 0 Q .24598 -.25 1 -.5};
\end{scope}
\path[draw=blue,line width=.01405cm] (-1,.4) -- (-1,-1.4) -- (1,-1.4) -- (1,.4) -- cycle;
\end{scope}
\begin{scope}
\node[fill=white,rectangle,inner sep=2pt,draw=black,scale=.99115,line width=.00281cm] at (-1,-1) {};
\end{scope}
\end{scope}
\end{tikzpicture}
}
\expandafter\def\csname e8diagram:dilute2_1_1\endcsname{%
\begin{tikzpicture}[execute at begin picture={\clip (-1.24,-1.616) rectangle ++(2.48,2.232);},x={(2.2204cm,0cm)},y={(0cm,-2.2204cm)},baseline={([yshift=-\the\dimexpr\fontdimen22\textfont2\relax]current  bounding  box.center)},line join=round]
\begin{scope}
\begin{scope}
\path[clip] (-1,.4) -- (-1,-1.4) -- (1,-1.4) -- (1,.4) -- cycle;
\path[fill={rgb,255:red,204;green,204;blue,255},draw=none] (-1,.4) -- (-1,-1.4) -- (1,-1.4) -- (1,.4) -- cycle;
\begin{scope}[every path/.append style={draw=red,line width=.05619cm}]
\path svg[xscale=2.2204cm,yscale=-2.2204cm] {M -1 -1 Q .24598 -.75 1 -.5};
\end{scope}
\path[draw=blue,line width=.01405cm] (-1,.4) -- (-1,-1.4) -- (1,-1.4) -- (1,.4) -- cycle;
\end{scope}
\begin{scope}
\node[fill=white,rectangle,inner sep=2pt,draw=black,scale=.99115,line width=.00281cm] at (-1,0) {};
\end{scope}
\end{scope}
\end{tikzpicture}
}
\expandafter\def\csname e8diagram:dilute2_1_2\endcsname{%
\begin{tikzpicture}[execute at begin picture={\clip (-1.24,-1.616) rectangle ++(2.48,2.232);},x={(2.2204cm,0cm)},y={(0cm,-2.2204cm)},baseline={([yshift=-\the\dimexpr\fontdimen22\textfont2\relax]current  bounding  box.center)},line join=round]
\begin{scope}
\begin{scope}
\path[clip] (-1,.4) -- (-1,-1.4) -- (1,-1.4) -- (1,.4) -- cycle;
\path[fill={rgb,255:red,204;green,204;blue,255},draw=none] (-1,.4) -- (-1,-1.4) -- (1,-1.4) -- (1,.4) -- cycle;
\begin{scope}[every path/.append style={draw=red,line width=.05619cm}]
\path svg[xscale=2.2204cm,yscale=-2.2204cm] {M -1 0 Q -.32218 -.25 -.03299 -.5};
\path svg[xscale=2.2204cm,yscale=-2.2204cm] {M 1 -.5 Q .48351 -.5 -.03299 -.5};
\path svg[xscale=2.2204cm,yscale=-2.2204cm] {M -1 -1 Q -.32218 -.75 -.03299 -.5};
\end{scope}
\path[draw=blue,line width=.01405cm] (-1,.4) -- (-1,-1.4) -- (1,-1.4) -- (1,.4) -- cycle;
\end{scope}
\end{scope}
\end{tikzpicture}
}
\expandafter\def\csname e8diagram:dilute2_2_0\endcsname{%
\begin{tikzpicture}[execute at begin picture={\clip (-1.24,-1.616) rectangle ++(2.48,2.232);},x={(2.2204cm,0cm)},y={(0cm,-2.2204cm)},baseline={([yshift=-\the\dimexpr\fontdimen22\textfont2\relax]current  bounding  box.center)},line join=round]
\begin{scope}
\begin{scope}
\path[clip] (-1,.4) -- (-1,-1.4) -- (1,-1.4) -- (1,.4) -- cycle;
\path[fill={rgb,255:red,204;green,204;blue,255},draw=none] (-1,.4) -- (-1,-1.4) -- (1,-1.4) -- (1,.4) -- cycle;
\begin{scope}[every path/.append style={draw=red,line width=.05619cm}]
\path svg[xscale=2.2204cm,yscale=-2.2204cm] {M -1 0 Q 0 0 1 0};
\path svg[xscale=2.2204cm,yscale=-2.2204cm] {M -1 -1 Q 0 -1 1 -1};
\end{scope}
\path[draw=blue,line width=.01405cm] (-1,.4) -- (-1,-1.4) -- (1,-1.4) -- (1,.4) -- cycle;
\end{scope}
\end{scope}
\end{tikzpicture}
}
\expandafter\def\csname e8diagram:dilute2_2_1\endcsname{%
\begin{tikzpicture}[execute at begin picture={\clip (-1.24,-1.616) rectangle ++(2.48,2.232);},x={(2.2204cm,0cm)},y={(0cm,-2.2204cm)},baseline={([yshift=-\the\dimexpr\fontdimen22\textfont2\relax]current  bounding  box.center)},line join=round]
\begin{scope}
\begin{scope}
\path[clip] (-1,.4) -- (-1,-1.4) -- (1,-1.4) -- (1,.4) -- cycle;
\path[fill={rgb,255:red,204;green,204;blue,255},draw=none] (-1,.4) -- (-1,-1.4) -- (1,-1.4) -- (1,.4) -- cycle;
\begin{scope}[every path/.append style={draw=red,line width=.05619cm}]
\path svg[xscale=2.2204cm,yscale=-2.2204cm] {M 1 0 Q .05277 -.5 1 -1};
\path svg[xscale=2.2204cm,yscale=-2.2204cm] {M -1 -1 Q -.05277 -.5 -1 0};
\end{scope}
\path[draw=blue,line width=.01405cm] (-1,.4) -- (-1,-1.4) -- (1,-1.4) -- (1,.4) -- cycle;
\end{scope}
\end{scope}
\end{tikzpicture}
}
\expandafter\def\csname e8diagram:dilute2_2_10\endcsname{%
\begin{tikzpicture}[execute at begin picture={\clip (-1.24,-1.616) rectangle ++(2.48,2.232);},x={(2.2204cm,0cm)},y={(0cm,-2.2204cm)},baseline={([yshift=-\the\dimexpr\fontdimen22\textfont2\relax]current  bounding  box.center)},line join=round]
\begin{scope}
\begin{scope}
\path[clip] (-1,.4) -- (-1,-1.4) -- (1,-1.4) -- (1,.4) -- cycle;
\path[fill={rgb,255:red,204;green,204;blue,255},draw=none] (-1,.4) -- (-1,-1.4) -- (1,-1.4) -- (1,.4) -- cycle;
\begin{scope}[every path/.append style={draw=red,line width=.05619cm}]
\path svg[xscale=2.2204cm,yscale=-2.2204cm] {M -1 -1 Q -.2186 -.83054 .19505 -.66107};
\path svg[xscale=2.2204cm,yscale=-2.2204cm] {M 1 0 Q .38605 -.33054 .19505 -.66107};
\path svg[xscale=2.2204cm,yscale=-2.2204cm] {M 1 -1 Q .44374 -.83054 .19505 -.66107};
\end{scope}
\path[draw=blue,line width=.01405cm] (-1,.4) -- (-1,-1.4) -- (1,-1.4) -- (1,.4) -- cycle;
\end{scope}
\begin{scope}
\node[fill=white,rectangle,inner sep=2pt,draw=black,scale=.99115,line width=.00281cm] at (-1,0) {};
\end{scope}
\end{scope}
\end{tikzpicture}
}
\expandafter\def\csname e8diagram:dilute2_2_11\endcsname{%
\begin{tikzpicture}[execute at begin picture={\clip (-1.24,-1.616) rectangle ++(2.48,2.232);},x={(2.2204cm,0cm)},y={(0cm,-2.2204cm)},baseline={([yshift=-\the\dimexpr\fontdimen22\textfont2\relax]current  bounding  box.center)},line join=round]
\begin{scope}
\begin{scope}
\path[clip] (-1,.4) -- (-1,-1.4) -- (1,-1.4) -- (1,.4) -- cycle;
\path[fill={rgb,255:red,204;green,204;blue,255},draw=none] (-1,.4) -- (-1,-1.4) -- (1,-1.4) -- (1,.4) -- cycle;
\begin{scope}[every path/.append style={draw=red,line width=.05619cm}]
\path svg[xscale=2.2204cm,yscale=-2.2204cm] {M -1 -1 Q -.44374 -.83054 -.19505 -.66107};
\path svg[xscale=2.2204cm,yscale=-2.2204cm] {M -1 0 Q -.38605 -.33054 -.19505 -.66107};
\path svg[xscale=2.2204cm,yscale=-2.2204cm] {M 1 -1 Q .2186 -.83054 -.19505 -.66107};
\end{scope}
\path[draw=blue,line width=.01405cm] (-1,.4) -- (-1,-1.4) -- (1,-1.4) -- (1,.4) -- cycle;
\end{scope}
\begin{scope}
\node[fill=white,rectangle,inner sep=2pt,draw=black,scale=.99115,line width=.00281cm] at (1,0) {};
\end{scope}
\end{scope}
\end{tikzpicture}
}
\expandafter\def\csname e8diagram:dilute2_2_12\endcsname{%
\begin{tikzpicture}[execute at begin picture={\clip (-1.24,-1.616) rectangle ++(2.48,2.232);},x={(2.2204cm,0cm)},y={(0cm,-2.2204cm)},baseline={([yshift=-\the\dimexpr\fontdimen22\textfont2\relax]current  bounding  box.center)},line join=round]
\begin{scope}
\begin{scope}
\path[clip] (-1,.4) -- (-1,-1.4) -- (1,-1.4) -- (1,.4) -- cycle;
\path[fill={rgb,255:red,204;green,204;blue,255},draw=none] (-1,.4) -- (-1,-1.4) -- (1,-1.4) -- (1,.4) -- cycle;
\begin{scope}[every path/.append style={draw=red,line width=.05619cm}]
\path svg[xscale=2.2204cm,yscale=-2.2204cm] {M -1 -1 Q -.38605 -.66946 -.19505 -.33893};
\path svg[xscale=2.2204cm,yscale=-2.2204cm] {M -1 0 Q -.44374 -.16946 -.19505 -.33893};
\path svg[xscale=2.2204cm,yscale=-2.2204cm] {M 1 0 Q .2186 -.16946 -.19505 -.33893};
\end{scope}
\path[draw=blue,line width=.01405cm] (-1,.4) -- (-1,-1.4) -- (1,-1.4) -- (1,.4) -- cycle;
\end{scope}
\begin{scope}
\node[fill=white,rectangle,inner sep=2pt,draw=black,scale=.99115,line width=.00281cm] at (1,-1) {};
\end{scope}
\end{scope}
\end{tikzpicture}
}
\expandafter\def\csname e8diagram:dilute2_2_13\endcsname{%
\begin{tikzpicture}[execute at begin picture={\clip (-1.24,-1.616) rectangle ++(2.48,2.232);},x={(2.2204cm,0cm)},y={(0cm,-2.2204cm)},baseline={([yshift=-\the\dimexpr\fontdimen22\textfont2\relax]current  bounding  box.center)},line join=round]
\begin{scope}
\begin{scope}
\path[clip] (-1,.4) -- (-1,-1.4) -- (1,-1.4) -- (1,.4) -- cycle;
\path[fill={rgb,255:red,204;green,204;blue,255},draw=none] (-1,.4) -- (-1,-1.4) -- (1,-1.4) -- (1,.4) -- cycle;
\begin{scope}[every path/.append style={draw=red,line width=.05619cm}]
\path svg[xscale=2.2204cm,yscale=-2.2204cm] {M 1 0 Q .05277 -.5 1 -1};
\end{scope}
\path[draw=blue,line width=.01405cm] (-1,.4) -- (-1,-1.4) -- (1,-1.4) -- (1,.4) -- cycle;
\end{scope}
\begin{scope}
\node[fill=white,rectangle,inner sep=2pt,draw=black,scale=.99115,line width=.00281cm] at (-1,-1) {};
\node[fill=white,rectangle,inner sep=2pt,draw=black,scale=.99115,line width=.00281cm] at (-1,0) {};
\end{scope}
\end{scope}
\end{tikzpicture}
}
\expandafter\def\csname e8diagram:dilute2_2_14\endcsname{%
\begin{tikzpicture}[execute at begin picture={\clip (-1.24,-1.616) rectangle ++(2.48,2.232);},x={(2.2204cm,0cm)},y={(0cm,-2.2204cm)},baseline={([yshift=-\the\dimexpr\fontdimen22\textfont2\relax]current  bounding  box.center)},line join=round]
\begin{scope}
\begin{scope}
\path[clip] (-1,.4) -- (-1,-1.4) -- (1,-1.4) -- (1,.4) -- cycle;
\path[fill={rgb,255:red,204;green,204;blue,255},draw=none] (-1,.4) -- (-1,-1.4) -- (1,-1.4) -- (1,.4) -- cycle;
\begin{scope}[every path/.append style={draw=red,line width=.05619cm}]
\path svg[xscale=2.2204cm,yscale=-2.2204cm] {M -1 -1 Q -.05277 -.5 -1 0};
\end{scope}
\path[draw=blue,line width=.01405cm] (-1,.4) -- (-1,-1.4) -- (1,-1.4) -- (1,.4) -- cycle;
\end{scope}
\begin{scope}
\node[fill=white,rectangle,inner sep=2pt,draw=black,scale=.99115,line width=.00281cm] at (1,0) {};
\node[fill=white,rectangle,inner sep=2pt,draw=black,scale=.99115,line width=.00281cm] at (1,-1) {};
\end{scope}
\end{scope}
\end{tikzpicture}
}
\expandafter\def\csname e8diagram:dilute2_2_15\endcsname{%
\begin{tikzpicture}[execute at begin picture={\clip (-1.24,-1.616) rectangle ++(2.48,2.232);},x={(2.2204cm,0cm)},y={(0cm,-2.2204cm)},baseline={([yshift=-\the\dimexpr\fontdimen22\textfont2\relax]current  bounding  box.center)},line join=round]
\begin{scope}
\begin{scope}
\path[clip] (-1,.4) -- (-1,-1.4) -- (1,-1.4) -- (1,.4) -- cycle;
\path[fill={rgb,255:red,204;green,204;blue,255},draw=none] (-1,.4) -- (-1,-1.4) -- (1,-1.4) -- (1,.4) -- cycle;
\begin{scope}[every path/.append style={draw=red,line width=.05619cm}]
\end{scope}
\path[draw=blue,line width=.01405cm] (-1,.4) -- (-1,-1.4) -- (1,-1.4) -- (1,.4) -- cycle;
\end{scope}
\begin{scope}
\node[fill=white,rectangle,inner sep=2pt,draw=black,scale=.99115,line width=.00281cm] at (-1,-1) {};
\node[fill=white,rectangle,inner sep=2pt,draw=black,scale=.99115,line width=.00281cm] at (-1,0) {};
\node[fill=white,rectangle,inner sep=2pt,draw=black,scale=.99115,line width=.00281cm] at (1,0) {};
\node[fill=white,rectangle,inner sep=2pt,draw=black,scale=.99115,line width=.00281cm] at (1,-1) {};
\end{scope}
\end{scope}
\end{tikzpicture}
}
\expandafter\def\csname e8diagram:dilute2_2_2\endcsname{%
\begin{tikzpicture}[execute at begin picture={\clip (-1.24,-1.616) rectangle ++(2.48,2.232);},x={(2.2204cm,0cm)},y={(0cm,-2.2204cm)},baseline={([yshift=-\the\dimexpr\fontdimen22\textfont2\relax]current  bounding  box.center)},line join=round]
\begin{scope}
\begin{scope}
\path[clip] (-1,.4) -- (-1,-1.4) -- (1,-1.4) -- (1,.4) -- cycle;
\path[fill={rgb,255:red,204;green,204;blue,255},draw=none] (-1,.4) -- (-1,-1.4) -- (1,-1.4) -- (1,.4) -- cycle;
\begin{scope}[every path/.append style={draw=red,line width=.05619cm}]
\path svg[xscale=2.2204cm,yscale=-2.2204cm] {M 1 -1 Q .48126 -.75 .32222 -.5};
\path svg[xscale=2.2204cm,yscale=-2.2204cm] {M -.32222 -.5 Q 0 -.5 .32222 -.5};
\path svg[xscale=2.2204cm,yscale=-2.2204cm] {M -1 -1 Q -.48126 -.75 -.32222 -.5};
\path svg[xscale=2.2204cm,yscale=-2.2204cm] {M -1 0 Q -.48126 -.25 -.32222 -.5};
\path svg[xscale=2.2204cm,yscale=-2.2204cm] {M 1 0 Q .48126 -.25 .32222 -.5};
\end{scope}
\path[draw=blue,line width=.01405cm] (-1,.4) -- (-1,-1.4) -- (1,-1.4) -- (1,.4) -- cycle;
\end{scope}
\end{scope}
\end{tikzpicture}
}
\expandafter\def\csname e8diagram:dilute2_2_3\endcsname{%
\begin{tikzpicture}[execute at begin picture={\clip (-1.24,-1.616) rectangle ++(2.48,2.232);},x={(2.2204cm,0cm)},y={(0cm,-2.2204cm)},baseline={([yshift=-\the\dimexpr\fontdimen22\textfont2\relax]current  bounding  box.center)},line join=round]
\begin{scope}
\begin{scope}
\path[clip] (-1,.4) -- (-1,-1.4) -- (1,-1.4) -- (1,.4) -- cycle;
\path[fill={rgb,255:red,204;green,204;blue,255},draw=none] (-1,.4) -- (-1,-1.4) -- (1,-1.4) -- (1,.4) -- cycle;
\begin{scope}[every path/.append style={draw=red,line width=.05619cm}]
\path svg[xscale=2.2204cm,yscale=-2.2204cm] {M 0 -.75959 Q 0 -.5 0 -.24041};
\path svg[xscale=2.2204cm,yscale=-2.2204cm] {M -1 -1 Q -.35701 -.87979 0 -.75959};
\path svg[xscale=2.2204cm,yscale=-2.2204cm] {M -1 0 Q -.35701 -.12021 0 -.24041};
\path svg[xscale=2.2204cm,yscale=-2.2204cm] {M 1 -1 Q .35701 -.87979 0 -.75959};
\path svg[xscale=2.2204cm,yscale=-2.2204cm] {M 1 0 Q .35701 -.12021 0 -.24041};
\end{scope}
\path[draw=blue,line width=.01405cm] (-1,.4) -- (-1,-1.4) -- (1,-1.4) -- (1,.4) -- cycle;
\end{scope}
\end{scope}
\end{tikzpicture}
}
\expandafter\def\csname e8diagram:dilute2_2_4\endcsname{%
\begin{tikzpicture}[execute at begin picture={\clip (-1.24,-1.616) rectangle ++(2.48,2.232);},x={(2.2204cm,0cm)},y={(0cm,-2.2204cm)},baseline={([yshift=-\the\dimexpr\fontdimen22\textfont2\relax]current  bounding  box.center)},line join=round]
\begin{scope}
\begin{scope}
\path[clip] (-1,.4) -- (-1,-1.4) -- (1,-1.4) -- (1,.4) -- cycle;
\path[fill={rgb,255:red,204;green,204;blue,255},draw=none] (-1,.4) -- (-1,-1.4) -- (1,-1.4) -- (1,.4) -- cycle;
\begin{scope}[every path/.append style={draw=red,line width=.05619cm}]
\path svg[xscale=2.2204cm,yscale=-2.2204cm] {M 1 0 Q .49871 -.15116 .27405 -.30233};
\path svg[xscale=2.2204cm,yscale=-2.2204cm] {M -.27405 -.69767 Q -.25875 -.5 -.27405 -.30233};
\path svg[xscale=2.2204cm,yscale=-2.2204cm] {M -.27405 -.69767 Q 0 -.69767 .27405 -.69767};
\path svg[xscale=2.2204cm,yscale=-2.2204cm] {M .27405 -.69767 Q .25875 -.5 .27405 -.30233};
\path svg[xscale=2.2204cm,yscale=-2.2204cm] {M -1 0 Q -.49871 -.15116 -.27405 -.30233};
\path svg[xscale=2.2204cm,yscale=-2.2204cm] {M 1 -1 Q .49871 -.84884 .27405 -.69767};
\path svg[xscale=2.2204cm,yscale=-2.2204cm] {M -.27405 -.30233 Q 0 -.30233 .27405 -.30233};
\path svg[xscale=2.2204cm,yscale=-2.2204cm] {M -1 -1 Q -.49871 -.84884 -.27405 -.69767};
\end{scope}
\path[draw=blue,line width=.01405cm] (-1,.4) -- (-1,-1.4) -- (1,-1.4) -- (1,.4) -- cycle;
\end{scope}
\end{scope}
\end{tikzpicture}
}
\expandafter\def\csname e8diagram:dilute2_2_5\endcsname{%
\begin{tikzpicture}[execute at begin picture={\clip (-1.24,-1.616) rectangle ++(2.48,2.232);},x={(2.2204cm,0cm)},y={(0cm,-2.2204cm)},baseline={([yshift=-\the\dimexpr\fontdimen22\textfont2\relax]current  bounding  box.center)},line join=round]
\begin{scope}
\begin{scope}
\path[clip] (-1,.4) -- (-1,-1.4) -- (1,-1.4) -- (1,.4) -- cycle;
\path[fill={rgb,255:red,204;green,204;blue,255},draw=none] (-1,.4) -- (-1,-1.4) -- (1,-1.4) -- (1,.4) -- cycle;
\begin{scope}[every path/.append style={draw=red,line width=.05619cm}]
\path svg[xscale=2.2204cm,yscale=-2.2204cm] {M -1 -1 Q 0 -1 1 -1};
\end{scope}
\path[draw=blue,line width=.01405cm] (-1,.4) -- (-1,-1.4) -- (1,-1.4) -- (1,.4) -- cycle;
\end{scope}
\begin{scope}
\node[fill=white,rectangle,inner sep=2pt,draw=black,scale=.99115,line width=.00281cm] at (-1,0) {};
\node[fill=white,rectangle,inner sep=2pt,draw=black,scale=.99115,line width=.00281cm] at (1,0) {};
\end{scope}
\end{scope}
\end{tikzpicture}
}
\expandafter\def\csname e8diagram:dilute2_2_6\endcsname{%
\begin{tikzpicture}[execute at begin picture={\clip (-1.24,-1.616) rectangle ++(2.48,2.232);},x={(2.2204cm,0cm)},y={(0cm,-2.2204cm)},baseline={([yshift=-\the\dimexpr\fontdimen22\textfont2\relax]current  bounding  box.center)},line join=round]
\begin{scope}
\begin{scope}
\path[clip] (-1,.4) -- (-1,-1.4) -- (1,-1.4) -- (1,.4) -- cycle;
\path[fill={rgb,255:red,204;green,204;blue,255},draw=none] (-1,.4) -- (-1,-1.4) -- (1,-1.4) -- (1,.4) -- cycle;
\begin{scope}[every path/.append style={draw=red,line width=.05619cm}]
\path svg[xscale=2.2204cm,yscale=-2.2204cm] {M -1 0 Q 0 -.5 1 -1};
\end{scope}
\path[draw=blue,line width=.01405cm] (-1,.4) -- (-1,-1.4) -- (1,-1.4) -- (1,.4) -- cycle;
\end{scope}
\begin{scope}
\node[fill=white,rectangle,inner sep=2pt,draw=black,scale=.99115,line width=.00281cm] at (-1,-1) {};
\node[fill=white,rectangle,inner sep=2pt,draw=black,scale=.99115,line width=.00281cm] at (1,0) {};
\end{scope}
\end{scope}
\end{tikzpicture}
}
\expandafter\def\csname e8diagram:dilute2_2_7\endcsname{%
\begin{tikzpicture}[execute at begin picture={\clip (-1.24,-1.616) rectangle ++(2.48,2.232);},x={(2.2204cm,0cm)},y={(0cm,-2.2204cm)},baseline={([yshift=-\the\dimexpr\fontdimen22\textfont2\relax]current  bounding  box.center)},line join=round]
\begin{scope}
\begin{scope}
\path[clip] (-1,.4) -- (-1,-1.4) -- (1,-1.4) -- (1,.4) -- cycle;
\path[fill={rgb,255:red,204;green,204;blue,255},draw=none] (-1,.4) -- (-1,-1.4) -- (1,-1.4) -- (1,.4) -- cycle;
\begin{scope}[every path/.append style={draw=red,line width=.05619cm}]
\path svg[xscale=2.2204cm,yscale=-2.2204cm] {M -1 -1 Q 0 -.5 1 0};
\end{scope}
\path[draw=blue,line width=.01405cm] (-1,.4) -- (-1,-1.4) -- (1,-1.4) -- (1,.4) -- cycle;
\end{scope}
\begin{scope}
\node[fill=white,rectangle,inner sep=2pt,draw=black,scale=.99115,line width=.00281cm] at (-1,0) {};
\node[fill=white,rectangle,inner sep=2pt,draw=black,scale=.99115,line width=.00281cm] at (1,-1) {};
\end{scope}
\end{scope}
\end{tikzpicture}
}
\expandafter\def\csname e8diagram:dilute2_2_8\endcsname{%
\begin{tikzpicture}[execute at begin picture={\clip (-1.24,-1.616) rectangle ++(2.48,2.232);},x={(2.2204cm,0cm)},y={(0cm,-2.2204cm)},baseline={([yshift=-\the\dimexpr\fontdimen22\textfont2\relax]current  bounding  box.center)},line join=round]
\begin{scope}
\begin{scope}
\path[clip] (-1,.4) -- (-1,-1.4) -- (1,-1.4) -- (1,.4) -- cycle;
\path[fill={rgb,255:red,204;green,204;blue,255},draw=none] (-1,.4) -- (-1,-1.4) -- (1,-1.4) -- (1,.4) -- cycle;
\begin{scope}[every path/.append style={draw=red,line width=.05619cm}]
\path svg[xscale=2.2204cm,yscale=-2.2204cm] {M -1 0 Q 0 0 1 0};
\end{scope}
\path[draw=blue,line width=.01405cm] (-1,.4) -- (-1,-1.4) -- (1,-1.4) -- (1,.4) -- cycle;
\end{scope}
\begin{scope}
\node[fill=white,rectangle,inner sep=2pt,draw=black,scale=.99115,line width=.00281cm] at (-1,-1) {};
\node[fill=white,rectangle,inner sep=2pt,draw=black,scale=.99115,line width=.00281cm] at (1,-1) {};
\end{scope}
\end{scope}
\end{tikzpicture}
}
\expandafter\def\csname e8diagram:dilute2_2_9\endcsname{%
\begin{tikzpicture}[execute at begin picture={\clip (-1.24,-1.616) rectangle ++(2.48,2.232);},x={(2.2204cm,0cm)},y={(0cm,-2.2204cm)},baseline={([yshift=-\the\dimexpr\fontdimen22\textfont2\relax]current  bounding  box.center)},line join=round]
\begin{scope}
\begin{scope}
\path[clip] (-1,.4) -- (-1,-1.4) -- (1,-1.4) -- (1,.4) -- cycle;
\path[fill={rgb,255:red,204;green,204;blue,255},draw=none] (-1,.4) -- (-1,-1.4) -- (1,-1.4) -- (1,.4) -- cycle;
\begin{scope}[every path/.append style={draw=red,line width=.05619cm}]
\path svg[xscale=2.2204cm,yscale=-2.2204cm] {M -1 0 Q -.2186 -.16946 .19505 -.33893};
\path svg[xscale=2.2204cm,yscale=-2.2204cm] {M 1 0 Q .44374 -.16946 .19505 -.33893};
\path svg[xscale=2.2204cm,yscale=-2.2204cm] {M 1 -1 Q .38605 -.66946 .19505 -.33893};
\end{scope}
\path[draw=blue,line width=.01405cm] (-1,.4) -- (-1,-1.4) -- (1,-1.4) -- (1,.4) -- cycle;
\end{scope}
\begin{scope}
\node[fill=white,rectangle,inner sep=2pt,draw=black,scale=.99115,line width=.00281cm] at (-1,-1) {};
\end{scope}
\end{scope}
\end{tikzpicture}
}
\expandafter\def\csname e8diagram:dilute3_0_0\endcsname{%
\begin{tikzpicture}[execute at begin picture={\clip (-1.24,-2.736) rectangle ++(2.48,3.472);},x={(1.7363cm,0cm)},y={(0cm,-1.7363cm)},baseline={([yshift=-\the\dimexpr\fontdimen22\textfont2\relax]current  bounding  box.center)},line join=round]
\begin{scope}
\begin{scope}
\path[clip] (-1,.4) -- (-1,-2.4) -- (1,-2.4) -- (1,.4) -- cycle;
\path[fill={rgb,255:red,204;green,204;blue,255},draw=none] (-1,.4) -- (-1,-2.4) -- (1,-2.4) -- (1,.4) -- cycle;
\begin{scope}[every path/.append style={draw=red,line width=.05619cm}]
\end{scope}
\path[draw=blue,line width=.01405cm] (-1,.4) -- (-1,-2.4) -- (1,-2.4) -- (1,.4) -- cycle;
\end{scope}
\begin{scope}
\node[fill=white,rectangle,inner sep=2pt,draw=black,scale=.99115,line width=.00281cm] at (-1,-2) {};
\node[fill=white,rectangle,inner sep=2pt,draw=black,scale=.99115,line width=.00281cm] at (-1,-1) {};
\node[fill=white,rectangle,inner sep=2pt,draw=black,scale=.99115,line width=.00281cm] at (-1,0) {};
\end{scope}
\end{scope}
\end{tikzpicture}
}
\expandafter\def\csname e8diagram:dilute3_0_1\endcsname{%
\begin{tikzpicture}[execute at begin picture={\clip (-1.24,-2.736) rectangle ++(2.48,3.472);},x={(1.7363cm,0cm)},y={(0cm,-1.7363cm)},baseline={([yshift=-\the\dimexpr\fontdimen22\textfont2\relax]current  bounding  box.center)},line join=round]
\begin{scope}
\begin{scope}
\path[clip] (-1,.4) -- (-1,-2.4) -- (1,-2.4) -- (1,.4) -- cycle;
\path[fill={rgb,255:red,204;green,204;blue,255},draw=none] (-1,.4) -- (-1,-2.4) -- (1,-2.4) -- (1,.4) -- cycle;
\begin{scope}[every path/.append style={draw=red,line width=.05619cm}]
\path svg[xscale=1.7363cm,yscale=-1.7363cm] {M -1 -1 Q -.09039 -.5 -1 0};
\end{scope}
\path[draw=blue,line width=.01405cm] (-1,.4) -- (-1,-2.4) -- (1,-2.4) -- (1,.4) -- cycle;
\end{scope}
\begin{scope}
\node[fill=white,rectangle,inner sep=2pt,draw=black,scale=.99115,line width=.00281cm] at (-1,-2) {};
\end{scope}
\end{scope}
\end{tikzpicture}
}
\expandafter\def\csname e8diagram:dilute3_0_2\endcsname{%
\begin{tikzpicture}[execute at begin picture={\clip (-1.24,-2.736) rectangle ++(2.48,3.472);},x={(1.7363cm,0cm)},y={(0cm,-1.7363cm)},baseline={([yshift=-\the\dimexpr\fontdimen22\textfont2\relax]current  bounding  box.center)},line join=round]
\begin{scope}
\begin{scope}
\path[clip] (-1,.4) -- (-1,-2.4) -- (1,-2.4) -- (1,.4) -- cycle;
\path[fill={rgb,255:red,204;green,204;blue,255},draw=none] (-1,.4) -- (-1,-2.4) -- (1,-2.4) -- (1,.4) -- cycle;
\begin{scope}[every path/.append style={draw=red,line width=.05619cm}]
\path svg[xscale=1.7363cm,yscale=-1.7363cm] {M -1 -2 Q .6668 -1 -1 0};
\end{scope}
\path[draw=blue,line width=.01405cm] (-1,.4) -- (-1,-2.4) -- (1,-2.4) -- (1,.4) -- cycle;
\end{scope}
\begin{scope}
\node[fill=white,rectangle,inner sep=2pt,draw=black,scale=.99115,line width=.00281cm] at (-1,-1) {};
\end{scope}
\end{scope}
\end{tikzpicture}
}
\expandafter\def\csname e8diagram:dilute3_0_3\endcsname{%
\begin{tikzpicture}[execute at begin picture={\clip (-1.24,-2.736) rectangle ++(2.48,3.472);},x={(1.7363cm,0cm)},y={(0cm,-1.7363cm)},baseline={([yshift=-\the\dimexpr\fontdimen22\textfont2\relax]current  bounding  box.center)},line join=round]
\begin{scope}
\begin{scope}
\path[clip] (-1,.4) -- (-1,-2.4) -- (1,-2.4) -- (1,.4) -- cycle;
\path[fill={rgb,255:red,204;green,204;blue,255},draw=none] (-1,.4) -- (-1,-2.4) -- (1,-2.4) -- (1,.4) -- cycle;
\begin{scope}[every path/.append style={draw=red,line width=.05619cm}]
\path svg[xscale=1.7363cm,yscale=-1.7363cm] {M -1 -2 Q -.09039 -1.5 -1 -1};
\end{scope}
\path[draw=blue,line width=.01405cm] (-1,.4) -- (-1,-2.4) -- (1,-2.4) -- (1,.4) -- cycle;
\end{scope}
\begin{scope}
\node[fill=white,rectangle,inner sep=2pt,draw=black,scale=.99115,line width=.00281cm] at (-1,0) {};
\end{scope}
\end{scope}
\end{tikzpicture}
}
\expandafter\def\csname e8diagram:dilute3_0_4\endcsname{%
\begin{tikzpicture}[execute at begin picture={\clip (-1.24,-2.736) rectangle ++(2.48,3.472);},x={(1.7363cm,0cm)},y={(0cm,-1.7363cm)},baseline={([yshift=-\the\dimexpr\fontdimen22\textfont2\relax]current  bounding  box.center)},line join=round]
\begin{scope}
\begin{scope}
\path[clip] (-1,.4) -- (-1,-2.4) -- (1,-2.4) -- (1,.4) -- cycle;
\path[fill={rgb,255:red,204;green,204;blue,255},draw=none] (-1,.4) -- (-1,-2.4) -- (1,-2.4) -- (1,.4) -- cycle;
\begin{scope}[every path/.append style={draw=red,line width=.05619cm}]
\path svg[xscale=1.7363cm,yscale=-1.7363cm] {M -1 -1 Q -.56864 -1 -.13729 -1};
\path svg[xscale=1.7363cm,yscale=-1.7363cm] {M -1 0 Q -.10611 -.5 -.13729 -1};
\path svg[xscale=1.7363cm,yscale=-1.7363cm] {M -1 -2 Q -.10611 -1.5 -.13729 -1};
\end{scope}
\path[draw=blue,line width=.01405cm] (-1,.4) -- (-1,-2.4) -- (1,-2.4) -- (1,.4) -- cycle;
\end{scope}
\end{scope}
\end{tikzpicture}
}
\expandafter\def\csname e8diagram:dilute3_1_0\endcsname{%
\begin{tikzpicture}[execute at begin picture={\clip (-1.24,-2.736) rectangle ++(2.48,3.472);},x={(1.7363cm,0cm)},y={(0cm,-1.7363cm)},baseline={([yshift=-\the\dimexpr\fontdimen22\textfont2\relax]current  bounding  box.center)},line join=round]
\begin{scope}
\begin{scope}
\path[clip] (-1,.4) -- (-1,-2.4) -- (1,-2.4) -- (1,.4) -- cycle;
\path[fill={rgb,255:red,204;green,204;blue,255},draw=none] (-1,.4) -- (-1,-2.4) -- (1,-2.4) -- (1,.4) -- cycle;
\begin{scope}[every path/.append style={draw=red,line width=.05619cm}]
\path svg[xscale=1.7363cm,yscale=-1.7363cm] {M -1 0 Q .32749 -.5 1 -1};
\end{scope}
\path[draw=blue,line width=.01405cm] (-1,.4) -- (-1,-2.4) -- (1,-2.4) -- (1,.4) -- cycle;
\end{scope}
\begin{scope}
\node[fill=white,rectangle,inner sep=2pt,draw=black,scale=.99115,line width=.00281cm] at (-1,-2) {};
\node[fill=white,rectangle,inner sep=2pt,draw=black,scale=.99115,line width=.00281cm] at (-1,-1) {};
\end{scope}
\end{scope}
\end{tikzpicture}
}
\expandafter\def\csname e8diagram:dilute3_1_1\endcsname{%
\begin{tikzpicture}[execute at begin picture={\clip (-1.24,-2.736) rectangle ++(2.48,3.472);},x={(1.7363cm,0cm)},y={(0cm,-1.7363cm)},baseline={([yshift=-\the\dimexpr\fontdimen22\textfont2\relax]current  bounding  box.center)},line join=round]
\begin{scope}
\begin{scope}
\path[clip] (-1,.4) -- (-1,-2.4) -- (1,-2.4) -- (1,.4) -- cycle;
\path[fill={rgb,255:red,204;green,204;blue,255},draw=none] (-1,.4) -- (-1,-2.4) -- (1,-2.4) -- (1,.4) -- cycle;
\begin{scope}[every path/.append style={draw=red,line width=.05619cm}]
\path svg[xscale=1.7363cm,yscale=-1.7363cm] {M -1 -2 Q -.33638 -1.6667 -.12436 -1.3334};
\path svg[xscale=1.7363cm,yscale=-1.7363cm] {M -1 -1 Q -.56579 -1.1667 -.12436 -1.3334};
\path svg[xscale=1.7363cm,yscale=-1.7363cm] {M 1 -1 Q .44589 -1.1667 -.12436 -1.3334};
\end{scope}
\path[draw=blue,line width=.01405cm] (-1,.4) -- (-1,-2.4) -- (1,-2.4) -- (1,.4) -- cycle;
\end{scope}
\begin{scope}
\node[fill=white,rectangle,inner sep=2pt,draw=black,scale=.99115,line width=.00281cm] at (-1,0) {};
\end{scope}
\end{scope}
\end{tikzpicture}
}
\expandafter\def\csname e8diagram:dilute3_1_10\endcsname{%
\begin{tikzpicture}[execute at begin picture={\clip (-1.24,-2.736) rectangle ++(2.48,3.472);},x={(1.7363cm,0cm)},y={(0cm,-1.7363cm)},baseline={([yshift=-\the\dimexpr\fontdimen22\textfont2\relax]current  bounding  box.center)},line join=round]
\begin{scope}
\begin{scope}
\path[clip] (-1,.4) -- (-1,-2.4) -- (1,-2.4) -- (1,.4) -- cycle;
\path[fill={rgb,255:red,204;green,204;blue,255},draw=none] (-1,.4) -- (-1,-2.4) -- (1,-2.4) -- (1,.4) -- cycle;
\begin{scope}[every path/.append style={draw=red,line width=.05619cm}]
\path svg[xscale=1.7363cm,yscale=-1.7363cm] {M -1 -1 Q -.29519 -.5 -1 0};
\path svg[xscale=1.7363cm,yscale=-1.7363cm] {M -1 -2 Q .32749 -1.5 1 -1};
\end{scope}
\path[draw=blue,line width=.01405cm] (-1,.4) -- (-1,-2.4) -- (1,-2.4) -- (1,.4) -- cycle;
\end{scope}
\end{scope}
\end{tikzpicture}
}
\expandafter\def\csname e8diagram:dilute3_1_2\endcsname{%
\begin{tikzpicture}[execute at begin picture={\clip (-1.24,-2.736) rectangle ++(2.48,3.472);},x={(1.7363cm,0cm)},y={(0cm,-1.7363cm)},baseline={([yshift=-\the\dimexpr\fontdimen22\textfont2\relax]current  bounding  box.center)},line join=round]
\begin{scope}
\begin{scope}
\path[clip] (-1,.4) -- (-1,-2.4) -- (1,-2.4) -- (1,.4) -- cycle;
\path[fill={rgb,255:red,204;green,204;blue,255},draw=none] (-1,.4) -- (-1,-2.4) -- (1,-2.4) -- (1,.4) -- cycle;
\begin{scope}[every path/.append style={draw=red,line width=.05619cm}]
\path svg[xscale=1.7363cm,yscale=-1.7363cm] {M -1 -2 Q -.19379 -1.5 .07137 -1};
\path svg[xscale=1.7363cm,yscale=-1.7363cm] {M -1 0 Q -.19379 -.5 .07137 -1};
\path svg[xscale=1.7363cm,yscale=-1.7363cm] {M 1 -1 Q .53568 -1 .07137 -1};
\end{scope}
\path[draw=blue,line width=.01405cm] (-1,.4) -- (-1,-2.4) -- (1,-2.4) -- (1,.4) -- cycle;
\end{scope}
\begin{scope}
\node[fill=white,rectangle,inner sep=2pt,draw=black,scale=.99115,line width=.00281cm] at (-1,-1) {};
\end{scope}
\end{scope}
\end{tikzpicture}
}
\expandafter\def\csname e8diagram:dilute3_1_3\endcsname{%
\begin{tikzpicture}[execute at begin picture={\clip (-1.24,-2.736) rectangle ++(2.48,3.472);},x={(1.7363cm,0cm)},y={(0cm,-1.7363cm)},baseline={([yshift=-\the\dimexpr\fontdimen22\textfont2\relax]current  bounding  box.center)},line join=round]
\begin{scope}
\begin{scope}
\path[clip] (-1,.4) -- (-1,-2.4) -- (1,-2.4) -- (1,.4) -- cycle;
\path[fill={rgb,255:red,204;green,204;blue,255},draw=none] (-1,.4) -- (-1,-2.4) -- (1,-2.4) -- (1,.4) -- cycle;
\begin{scope}[every path/.append style={draw=red,line width=.05619cm}]
\path svg[xscale=1.7363cm,yscale=-1.7363cm] {M -1 -1 Q -.56579 -.83331 -.12436 -.66662};
\path svg[xscale=1.7363cm,yscale=-1.7363cm] {M -1 0 Q -.33638 -.33331 -.12436 -.66662};
\path svg[xscale=1.7363cm,yscale=-1.7363cm] {M 1 -1 Q .44589 -.83331 -.12436 -.66662};
\end{scope}
\path[draw=blue,line width=.01405cm] (-1,.4) -- (-1,-2.4) -- (1,-2.4) -- (1,.4) -- cycle;
\end{scope}
\begin{scope}
\node[fill=white,rectangle,inner sep=2pt,draw=black,scale=.99115,line width=.00281cm] at (-1,-2) {};
\end{scope}
\end{scope}
\end{tikzpicture}
}
\expandafter\def\csname e8diagram:dilute3_1_4\endcsname{%
\begin{tikzpicture}[execute at begin picture={\clip (-1.24,-2.736) rectangle ++(2.48,3.472);},x={(1.7363cm,0cm)},y={(0cm,-1.7363cm)},baseline={([yshift=-\the\dimexpr\fontdimen22\textfont2\relax]current  bounding  box.center)},line join=round]
\begin{scope}
\begin{scope}
\path[clip] (-1,.4) -- (-1,-2.4) -- (1,-2.4) -- (1,.4) -- cycle;
\path[fill={rgb,255:red,204;green,204;blue,255},draw=none] (-1,.4) -- (-1,-2.4) -- (1,-2.4) -- (1,.4) -- cycle;
\begin{scope}[every path/.append style={draw=red,line width=.05619cm}]
\path svg[xscale=1.7363cm,yscale=-1.7363cm] {M -1 -1 Q 0 -1 1 -1};
\end{scope}
\path[draw=blue,line width=.01405cm] (-1,.4) -- (-1,-2.4) -- (1,-2.4) -- (1,.4) -- cycle;
\end{scope}
\begin{scope}
\node[fill=white,rectangle,inner sep=2pt,draw=black,scale=.99115,line width=.00281cm] at (-1,-2) {};
\node[fill=white,rectangle,inner sep=2pt,draw=black,scale=.99115,line width=.00281cm] at (-1,0) {};
\end{scope}
\end{scope}
\end{tikzpicture}
}
\expandafter\def\csname e8diagram:dilute3_1_5\endcsname{%
\begin{tikzpicture}[execute at begin picture={\clip (-1.24,-2.736) rectangle ++(2.48,3.472);},x={(1.7363cm,0cm)},y={(0cm,-1.7363cm)},baseline={([yshift=-\the\dimexpr\fontdimen22\textfont2\relax]current  bounding  box.center)},line join=round]
\begin{scope}
\begin{scope}
\path[clip] (-1,.4) -- (-1,-2.4) -- (1,-2.4) -- (1,.4) -- cycle;
\path[fill={rgb,255:red,204;green,204;blue,255},draw=none] (-1,.4) -- (-1,-2.4) -- (1,-2.4) -- (1,.4) -- cycle;
\begin{scope}[every path/.append style={draw=red,line width=.05619cm}]
\path svg[xscale=1.7363cm,yscale=-1.7363cm] {M -1 -2 Q .32749 -1.5 1 -1};
\end{scope}
\path[draw=blue,line width=.01405cm] (-1,.4) -- (-1,-2.4) -- (1,-2.4) -- (1,.4) -- cycle;
\end{scope}
\begin{scope}
\node[fill=white,rectangle,inner sep=2pt,draw=black,scale=.99115,line width=.00281cm] at (-1,-1) {};
\node[fill=white,rectangle,inner sep=2pt,draw=black,scale=.99115,line width=.00281cm] at (-1,0) {};
\end{scope}
\end{scope}
\end{tikzpicture}
}
\expandafter\def\csname e8diagram:dilute3_1_6\endcsname{%
\begin{tikzpicture}[execute at begin picture={\clip (-1.24,-2.736) rectangle ++(2.48,3.472);},x={(1.7363cm,0cm)},y={(0cm,-1.7363cm)},baseline={([yshift=-\the\dimexpr\fontdimen22\textfont2\relax]current  bounding  box.center)},line join=round]
\begin{scope}
\begin{scope}
\path[clip] (-1,.4) -- (-1,-2.4) -- (1,-2.4) -- (1,.4) -- cycle;
\path[fill={rgb,255:red,204;green,204;blue,255},draw=none] (-1,.4) -- (-1,-2.4) -- (1,-2.4) -- (1,.4) -- cycle;
\begin{scope}[every path/.append style={draw=red,line width=.05619cm}]
\path svg[xscale=1.7363cm,yscale=-1.7363cm] {M -1 0 Q -.42711 -.28189 -.25107 -.56377};
\path svg[xscale=1.7363cm,yscale=-1.7363cm] {M -.25107 -1.4362 Q -.40942 -1.2181 -.57211 -1};
\path svg[xscale=1.7363cm,yscale=-1.7363cm] {M -.25107 -1.4362 Q .0215 -1.2181 .27262 -1};
\path svg[xscale=1.7363cm,yscale=-1.7363cm] {M .27262 -1 Q .0215 -.78189 -.25107 -.56377};
\path svg[xscale=1.7363cm,yscale=-1.7363cm] {M -1 -1 Q -.78605 -1 -.57211 -1};
\path svg[xscale=1.7363cm,yscale=-1.7363cm] {M 1 -1 Q .63631 -1 .27262 -1};
\path svg[xscale=1.7363cm,yscale=-1.7363cm] {M -.57211 -1 Q -.40942 -.78189 -.25107 -.56377};
\path svg[xscale=1.7363cm,yscale=-1.7363cm] {M -1 -2 Q -.42711 -1.7181 -.25107 -1.4362};
\end{scope}
\path[draw=blue,line width=.01405cm] (-1,.4) -- (-1,-2.4) -- (1,-2.4) -- (1,.4) -- cycle;
\end{scope}
\end{scope}
\end{tikzpicture}
}
\expandafter\def\csname e8diagram:dilute3_1_7\endcsname{%
\begin{tikzpicture}[execute at begin picture={\clip (-1.24,-2.736) rectangle ++(2.48,3.472);},x={(1.7363cm,0cm)},y={(0cm,-1.7363cm)},baseline={([yshift=-\the\dimexpr\fontdimen22\textfont2\relax]current  bounding  box.center)},line join=round]
\begin{scope}
\begin{scope}
\path[clip] (-1,.4) -- (-1,-2.4) -- (1,-2.4) -- (1,.4) -- cycle;
\path[fill={rgb,255:red,204;green,204;blue,255},draw=none] (-1,.4) -- (-1,-2.4) -- (1,-2.4) -- (1,.4) -- cycle;
\begin{scope}[every path/.append style={draw=red,line width=.05619cm}]
\path svg[xscale=1.7363cm,yscale=-1.7363cm] {M .08456 -1.2759 Q -.19139 -1 -.47997 -.72413};
\path svg[xscale=1.7363cm,yscale=-1.7363cm] {M -1 -2 Q -.20545 -1.6379 .08456 -1.2759};
\path svg[xscale=1.7363cm,yscale=-1.7363cm] {M -1 -1 Q -.73814 -.86206 -.47997 -.72413};
\path svg[xscale=1.7363cm,yscale=-1.7363cm] {M 1 -1 Q .54523 -1.1379 .08456 -1.2759};
\path svg[xscale=1.7363cm,yscale=-1.7363cm] {M -1 0 Q -.54086 -.36206 -.47997 -.72413};
\end{scope}
\path[draw=blue,line width=.01405cm] (-1,.4) -- (-1,-2.4) -- (1,-2.4) -- (1,.4) -- cycle;
\end{scope}
\end{scope}
\end{tikzpicture}
}
\expandafter\def\csname e8diagram:dilute3_1_8\endcsname{%
\begin{tikzpicture}[execute at begin picture={\clip (-1.24,-2.736) rectangle ++(2.48,3.472);},x={(1.7363cm,0cm)},y={(0cm,-1.7363cm)},baseline={([yshift=-\the\dimexpr\fontdimen22\textfont2\relax]current  bounding  box.center)},line join=round]
\begin{scope}
\begin{scope}
\path[clip] (-1,.4) -- (-1,-2.4) -- (1,-2.4) -- (1,.4) -- cycle;
\path[fill={rgb,255:red,204;green,204;blue,255},draw=none] (-1,.4) -- (-1,-2.4) -- (1,-2.4) -- (1,.4) -- cycle;
\begin{scope}[every path/.append style={draw=red,line width=.05619cm}]
\path svg[xscale=1.7363cm,yscale=-1.7363cm] {M 1 -1 Q .54523 -.86206 .08456 -.72412};
\path svg[xscale=1.7363cm,yscale=-1.7363cm] {M -.47997 -1.2759 Q -.19139 -1 .08456 -.72412};
\path svg[xscale=1.7363cm,yscale=-1.7363cm] {M -1 -2 Q -.54086 -1.6379 -.47997 -1.2759};
\path svg[xscale=1.7363cm,yscale=-1.7363cm] {M -1 -1 Q -.73814 -1.1379 -.47997 -1.2759};
\path svg[xscale=1.7363cm,yscale=-1.7363cm] {M -1 0 Q -.20545 -.36206 .08456 -.72412};
\end{scope}
\path[draw=blue,line width=.01405cm] (-1,.4) -- (-1,-2.4) -- (1,-2.4) -- (1,.4) -- cycle;
\end{scope}
\end{scope}
\end{tikzpicture}
}
\expandafter\def\csname e8diagram:dilute3_1_9\endcsname{%
\begin{tikzpicture}[execute at begin picture={\clip (-1.24,-2.736) rectangle ++(2.48,3.472);},x={(1.7363cm,0cm)},y={(0cm,-1.7363cm)},baseline={([yshift=-\the\dimexpr\fontdimen22\textfont2\relax]current  bounding  box.center)},line join=round]
\begin{scope}
\begin{scope}
\path[clip] (-1,.4) -- (-1,-2.4) -- (1,-2.4) -- (1,.4) -- cycle;
\path[fill={rgb,255:red,204;green,204;blue,255},draw=none] (-1,.4) -- (-1,-2.4) -- (1,-2.4) -- (1,.4) -- cycle;
\begin{scope}[every path/.append style={draw=red,line width=.05619cm}]
\path svg[xscale=1.7363cm,yscale=-1.7363cm] {M -1 0 Q .32749 -.5 1 -1};
\path svg[xscale=1.7363cm,yscale=-1.7363cm] {M -1 -2 Q -.29519 -1.5 -1 -1};
\end{scope}
\path[draw=blue,line width=.01405cm] (-1,.4) -- (-1,-2.4) -- (1,-2.4) -- (1,.4) -- cycle;
\end{scope}
\end{scope}
\end{tikzpicture}
}
\expandafter\def\csname e8diagram:dilute3_2_0\endcsname{%
\begin{tikzpicture}[execute at begin picture={\clip (-1.24,-2.736) rectangle ++(2.48,3.472);},x={(1.7363cm,0cm)},y={(0cm,-1.7363cm)},baseline={([yshift=-\the\dimexpr\fontdimen22\textfont2\relax]current  bounding  box.center)},line join=round]
\begin{scope}
\begin{scope}
\path[clip] (-1,.4) -- (-1,-2.4) -- (1,-2.4) -- (1,.4) -- cycle;
\path[fill={rgb,255:red,204;green,204;blue,255},draw=none] (-1,.4) -- (-1,-2.4) -- (1,-2.4) -- (1,.4) -- cycle;
\begin{scope}[every path/.append style={draw=red,line width=.05619cm}]
\path svg[xscale=1.7363cm,yscale=-1.7363cm] {M 1 -1.5 Q .73248 -1.4551 .47742 -1.4102};
\path svg[xscale=1.7363cm,yscale=-1.7363cm] {M -.33937 -1.1117 Q .01971 -1.003 .3788 -.89441};
\path svg[xscale=1.7363cm,yscale=-1.7363cm] {M -1 -1 Q -.66967 -1.0558 -.33937 -1.1117};
\path svg[xscale=1.7363cm,yscale=-1.7363cm] {M 1 -.5 Q .65936 -.69721 .3788 -.89441};
\path svg[xscale=1.7363cm,yscale=-1.7363cm] {M -.14374 -1.5999 Q -.24206 -1.3558 -.33937 -1.1117};
\path svg[xscale=1.7363cm,yscale=-1.7363cm] {M -.14374 -1.5999 Q .18151 -1.505 .47742 -1.4102};
\path svg[xscale=1.7363cm,yscale=-1.7363cm] {M -1 -2 Q -.39122 -1.7999 -.14374 -1.5999};
\path svg[xscale=1.7363cm,yscale=-1.7363cm] {M .47742 -1.4102 Q .414 -1.1523 .3788 -.89441};
\end{scope}
\path[draw=blue,line width=.01405cm] (-1,.4) -- (-1,-2.4) -- (1,-2.4) -- (1,.4) -- cycle;
\end{scope}
\begin{scope}
\node[fill=white,rectangle,inner sep=2pt,draw=black,scale=.99115,line width=.00281cm] at (-1,0) {};
\end{scope}
\end{scope}
\end{tikzpicture}
}
\expandafter\def\csname e8diagram:dilute3_2_1\endcsname{%
\begin{tikzpicture}[execute at begin picture={\clip (-1.24,-2.736) rectangle ++(2.48,3.472);},x={(1.7363cm,0cm)},y={(0cm,-1.7363cm)},baseline={([yshift=-\the\dimexpr\fontdimen22\textfont2\relax]current  bounding  box.center)},line join=round]
\begin{scope}
\begin{scope}
\path[clip] (-1,.4) -- (-1,-2.4) -- (1,-2.4) -- (1,.4) -- cycle;
\path[fill={rgb,255:red,204;green,204;blue,255},draw=none] (-1,.4) -- (-1,-2.4) -- (1,-2.4) -- (1,.4) -- cycle;
\begin{scope}[every path/.append style={draw=red,line width=.05619cm}]
\path svg[xscale=1.7363cm,yscale=-1.7363cm] {M 1 -1.5 Q .61114 -1.5302 .20944 -1.5603};
\path svg[xscale=1.7363cm,yscale=-1.7363cm] {M 1 -.5 Q .46283 -.72806 .00620 -.95612};
\path svg[xscale=1.7363cm,yscale=-1.7363cm] {M .20944 -1.5603 Q .09589 -1.2582 .00620 -.95612};
\path svg[xscale=1.7363cm,yscale=-1.7363cm] {M -1 -2 Q -.16764 -1.7802 .20944 -1.5603};
\path svg[xscale=1.7363cm,yscale=-1.7363cm] {M -1 -1 Q -.497 -.97806 .00620 -.95612};
\end{scope}
\path[draw=blue,line width=.01405cm] (-1,.4) -- (-1,-2.4) -- (1,-2.4) -- (1,.4) -- cycle;
\end{scope}
\begin{scope}
\node[fill=white,rectangle,inner sep=2pt,draw=black,scale=.99115,line width=.00281cm] at (-1,0) {};
\end{scope}
\end{scope}
\end{tikzpicture}
}
\expandafter\def\csname e8diagram:dilute3_2_10\endcsname{%
\begin{tikzpicture}[execute at begin picture={\clip (-1.24,-2.736) rectangle ++(2.48,3.472);},x={(1.7363cm,0cm)},y={(0cm,-1.7363cm)},baseline={([yshift=-\the\dimexpr\fontdimen22\textfont2\relax]current  bounding  box.center)},line join=round]
\begin{scope}
\begin{scope}
\path[clip] (-1,.4) -- (-1,-2.4) -- (1,-2.4) -- (1,.4) -- cycle;
\path[fill={rgb,255:red,204;green,204;blue,255},draw=none] (-1,.4) -- (-1,-2.4) -- (1,-2.4) -- (1,.4) -- cycle;
\begin{scope}[every path/.append style={draw=red,line width=.05619cm}]
\path svg[xscale=1.7363cm,yscale=-1.7363cm] {M -.5211 -.8497 Q -.39366 -.61152 -.26814 -.37335};
\path svg[xscale=1.7363cm,yscale=-1.7363cm] {M 1 -1.5 Q .63266 -1.5643 .23811 -1.6287};
\path svg[xscale=1.7363cm,yscale=-1.7363cm] {M -1 0 Q -.47304 -.18668 -.26814 -.37335};
\path svg[xscale=1.7363cm,yscale=-1.7363cm] {M .01082 -1.0621 Q -.25291 -.95589 -.5211 -.8497};
\path svg[xscale=1.7363cm,yscale=-1.7363cm] {M .01082 -1.0621 Q .17425 -.82622 .35916 -.59035};
\path svg[xscale=1.7363cm,yscale=-1.7363cm] {M -1 -1 Q -.75988 -.92485 -.5211 -.8497};
\path svg[xscale=1.7363cm,yscale=-1.7363cm] {M .35916 -.59035 Q .06403 -.48185 -.26814 -.37335};
\path svg[xscale=1.7363cm,yscale=-1.7363cm] {M 1 -.5 Q .67218 -.54518 .35916 -.59035};
\path svg[xscale=1.7363cm,yscale=-1.7363cm] {M .23811 -1.6287 Q .10859 -1.3454 .01082 -1.0621};
\path svg[xscale=1.7363cm,yscale=-1.7363cm] {M -1 -2 Q -.16452 -1.8143 .23811 -1.6287};
\end{scope}
\path[draw=blue,line width=.01405cm] (-1,.4) -- (-1,-2.4) -- (1,-2.4) -- (1,.4) -- cycle;
\end{scope}
\end{scope}
\end{tikzpicture}
}
\expandafter\def\csname e8diagram:dilute3_2_11\endcsname{%
\begin{tikzpicture}[execute at begin picture={\clip (-1.24,-2.736) rectangle ++(2.48,3.472);},x={(1.7363cm,0cm)},y={(0cm,-1.7363cm)},baseline={([yshift=-\the\dimexpr\fontdimen22\textfont2\relax]current  bounding  box.center)},line join=round]
\begin{scope}
\begin{scope}
\path[clip] (-1,.4) -- (-1,-2.4) -- (1,-2.4) -- (1,.4) -- cycle;
\path[fill={rgb,255:red,204;green,204;blue,255},draw=none] (-1,.4) -- (-1,-2.4) -- (1,-2.4) -- (1,.4) -- cycle;
\begin{scope}[every path/.append style={draw=red,line width=.05619cm}]
\path svg[xscale=1.7363cm,yscale=-1.7363cm] {M -1 0 Q -.16452 -.18567 .23811 -.37134};
\path svg[xscale=1.7363cm,yscale=-1.7363cm] {M 1 -1.5 Q .67218 -1.4548 .35916 -1.4096};
\path svg[xscale=1.7363cm,yscale=-1.7363cm] {M -.5211 -1.1503 Q -.25291 -1.0441 .01082 -.93792};
\path svg[xscale=1.7363cm,yscale=-1.7363cm] {M -1 -1 Q -.75988 -1.0752 -.5211 -1.1503};
\path svg[xscale=1.7363cm,yscale=-1.7363cm] {M .01082 -.93792 Q .10859 -.65463 .23811 -.37134};
\path svg[xscale=1.7363cm,yscale=-1.7363cm] {M -.26814 -1.6266 Q -.39366 -1.3885 -.5211 -1.1503};
\path svg[xscale=1.7363cm,yscale=-1.7363cm] {M -.26814 -1.6266 Q .06403 -1.5181 .35916 -1.4096};
\path svg[xscale=1.7363cm,yscale=-1.7363cm] {M 1 -.5 Q .63266 -.43567 .23811 -.37134};
\path svg[xscale=1.7363cm,yscale=-1.7363cm] {M -1 -2 Q -.47304 -1.8133 -.26814 -1.6266};
\path svg[xscale=1.7363cm,yscale=-1.7363cm] {M .35916 -1.4096 Q .17425 -1.1738 .01082 -.93792};
\end{scope}
\path[draw=blue,line width=.01405cm] (-1,.4) -- (-1,-2.4) -- (1,-2.4) -- (1,.4) -- cycle;
\end{scope}
\end{scope}
\end{tikzpicture}
}
\expandafter\def\csname e8diagram:dilute3_2_12\endcsname{%
\begin{tikzpicture}[execute at begin picture={\clip (-1.24,-2.736) rectangle ++(2.48,3.472);},x={(1.7363cm,0cm)},y={(0cm,-1.7363cm)},baseline={([yshift=-\the\dimexpr\fontdimen22\textfont2\relax]current  bounding  box.center)},line join=round]
\begin{scope}
\begin{scope}
\path[clip] (-1,.4) -- (-1,-2.4) -- (1,-2.4) -- (1,.4) -- cycle;
\path[fill={rgb,255:red,204;green,204;blue,255},draw=none] (-1,.4) -- (-1,-2.4) -- (1,-2.4) -- (1,.4) -- cycle;
\begin{scope}[every path/.append style={draw=red,line width=.05619cm}]
\path svg[xscale=1.7363cm,yscale=-1.7363cm] {M .55984 -1.3817 Q .52498 -1.1102 .5253 -.83867};
\path svg[xscale=1.7363cm,yscale=-1.7363cm] {M -1 0 Q -.628 -.30302 -.59453 -.60603};
\path svg[xscale=1.7363cm,yscale=-1.7363cm] {M 1 -1.5 Q .77265 -1.4408 .55984 -1.3817};
\path svg[xscale=1.7363cm,yscale=-1.7363cm] {M -.05622 -.96486 Q -.31163 -.78545 -.59453 -.60603};
\path svg[xscale=1.7363cm,yscale=-1.7363cm] {M -.05622 -.96486 Q .23264 -.90177 .5253 -.83867};
\path svg[xscale=1.7363cm,yscale=-1.7363cm] {M -1 -1 Q -.79032 -.80302 -.59453 -.60603};
\path svg[xscale=1.7363cm,yscale=-1.7363cm] {M -.03959 -1.5599 Q -.04648 -1.2624 -.05622 -.96486};
\path svg[xscale=1.7363cm,yscale=-1.7363cm] {M -.03959 -1.5599 Q .27163 -1.4708 .55984 -1.3817};
\path svg[xscale=1.7363cm,yscale=-1.7363cm] {M 1 -.5 Q .73937 -.66934 .5253 -.83867};
\path svg[xscale=1.7363cm,yscale=-1.7363cm] {M -1 -2 Q -.31956 -1.78 -.03959 -1.5599};
\end{scope}
\path[draw=blue,line width=.01405cm] (-1,.4) -- (-1,-2.4) -- (1,-2.4) -- (1,.4) -- cycle;
\end{scope}
\end{scope}
\end{tikzpicture}
}
\expandafter\def\csname e8diagram:dilute3_2_13\endcsname{%
\begin{tikzpicture}[execute at begin picture={\clip (-1.24,-2.736) rectangle ++(2.48,3.472);},x={(1.7363cm,0cm)},y={(0cm,-1.7363cm)},baseline={([yshift=-\the\dimexpr\fontdimen22\textfont2\relax]current  bounding  box.center)},line join=round]
\begin{scope}
\begin{scope}
\path[clip] (-1,.4) -- (-1,-2.4) -- (1,-2.4) -- (1,.4) -- cycle;
\path[fill={rgb,255:red,204;green,204;blue,255},draw=none] (-1,.4) -- (-1,-2.4) -- (1,-2.4) -- (1,.4) -- cycle;
\begin{scope}[every path/.append style={draw=red,line width=.05619cm}]
\path svg[xscale=1.7363cm,yscale=-1.7363cm] {M -.03959 -.44009 Q .27163 -.52922 .55984 -.61835};
\path svg[xscale=1.7363cm,yscale=-1.7363cm] {M -1 0 Q -.31956 -.22004 -.03959 -.44009};
\path svg[xscale=1.7363cm,yscale=-1.7363cm] {M 1 -1.5 Q .73937 -1.3307 .5253 -1.1613};
\path svg[xscale=1.7363cm,yscale=-1.7363cm] {M -.05622 -1.0351 Q -.04648 -.73761 -.03959 -.44009};
\path svg[xscale=1.7363cm,yscale=-1.7363cm] {M -.05622 -1.0351 Q .23264 -1.0982 .5253 -1.1613};
\path svg[xscale=1.7363cm,yscale=-1.7363cm] {M -1 -1 Q -.79032 -1.197 -.59453 -1.394};
\path svg[xscale=1.7363cm,yscale=-1.7363cm] {M .5253 -1.1613 Q .52498 -.88984 .55984 -.61835};
\path svg[xscale=1.7363cm,yscale=-1.7363cm] {M -.59453 -1.394 Q -.31163 -1.2146 -.05622 -1.0351};
\path svg[xscale=1.7363cm,yscale=-1.7363cm] {M 1 -.5 Q .77265 -.55917 .55984 -.61835};
\path svg[xscale=1.7363cm,yscale=-1.7363cm] {M -1 -2 Q -.628 -1.697 -.59453 -1.394};
\end{scope}
\path[draw=blue,line width=.01405cm] (-1,.4) -- (-1,-2.4) -- (1,-2.4) -- (1,.4) -- cycle;
\end{scope}
\end{scope}
\end{tikzpicture}
}
\expandafter\def\csname e8diagram:dilute3_2_14\endcsname{%
\begin{tikzpicture}[execute at begin picture={\clip (-1.24,-2.736) rectangle ++(2.48,3.472);},x={(1.7363cm,0cm)},y={(0cm,-1.7363cm)},baseline={([yshift=-\the\dimexpr\fontdimen22\textfont2\relax]current  bounding  box.center)},line join=round]
\begin{scope}
\begin{scope}
\path[clip] (-1,.4) -- (-1,-2.4) -- (1,-2.4) -- (1,.4) -- cycle;
\path[fill={rgb,255:red,204;green,204;blue,255},draw=none] (-1,.4) -- (-1,-2.4) -- (1,-2.4) -- (1,.4) -- cycle;
\begin{scope}[every path/.append style={draw=red,line width=.05619cm}]
\path svg[xscale=1.7363cm,yscale=-1.7363cm] {M -1 0 Q -.2332 -.2168 .10394 -.43359};
\path svg[xscale=1.7363cm,yscale=-1.7363cm] {M 1 -1.5 Q .56011 -1.5332 .10394 -1.5664};
\path svg[xscale=1.7363cm,yscale=-1.7363cm] {M -.33104 -1 Q -.12519 -.7168 .10394 -.43359};
\path svg[xscale=1.7363cm,yscale=-1.7363cm] {M -1 -1 Q -.66552 -1 -.33104 -1};
\path svg[xscale=1.7363cm,yscale=-1.7363cm] {M 1 -.5 Q .56011 -.4668 .10394 -.43359};
\path svg[xscale=1.7363cm,yscale=-1.7363cm] {M .10394 -1.5664 Q -.12519 -1.2832 -.33104 -1};
\path svg[xscale=1.7363cm,yscale=-1.7363cm] {M -1 -2 Q -.2332 -1.7832 .10394 -1.5664};
\end{scope}
\path[draw=blue,line width=.01405cm] (-1,.4) -- (-1,-2.4) -- (1,-2.4) -- (1,.4) -- cycle;
\end{scope}
\end{scope}
\end{tikzpicture}
}
\expandafter\def\csname e8diagram:dilute3_2_15\endcsname{%
\begin{tikzpicture}[execute at begin picture={\clip (-1.24,-2.736) rectangle ++(2.48,3.472);},x={(1.7363cm,0cm)},y={(0cm,-1.7363cm)},baseline={([yshift=-\the\dimexpr\fontdimen22\textfont2\relax]current  bounding  box.center)},line join=round]
\begin{scope}
\begin{scope}
\path[clip] (-1,.4) -- (-1,-2.4) -- (1,-2.4) -- (1,.4) -- cycle;
\path[fill={rgb,255:red,204;green,204;blue,255},draw=none] (-1,.4) -- (-1,-2.4) -- (1,-2.4) -- (1,.4) -- cycle;
\begin{scope}[every path/.append style={draw=red,line width=.05619cm}]
\path svg[xscale=1.7363cm,yscale=-1.7363cm] {M -1 0 Q -.11491 -.23663 .28605 -.47327};
\path svg[xscale=1.7363cm,yscale=-1.7363cm] {M 1 -1.5 Q .61055 -1.3384 .27474 -1.1767};
\path svg[xscale=1.7363cm,yscale=-1.7363cm] {M .27474 -1.1767 Q .26452 -.82501 .28605 -.47327};
\path svg[xscale=1.7363cm,yscale=-1.7363cm] {M -1 -1 Q -.72813 -1.2045 -.46414 -1.4089};
\path svg[xscale=1.7363cm,yscale=-1.7363cm] {M 1 -.5 Q .64554 -.48663 .28605 -.47327};
\path svg[xscale=1.7363cm,yscale=-1.7363cm] {M -.46414 -1.4089 Q -.08284 -1.2928 .27474 -1.1767};
\path svg[xscale=1.7363cm,yscale=-1.7363cm] {M -1 -2 Q -.55172 -1.7045 -.46414 -1.4089};
\end{scope}
\path[draw=blue,line width=.01405cm] (-1,.4) -- (-1,-2.4) -- (1,-2.4) -- (1,.4) -- cycle;
\end{scope}
\end{scope}
\end{tikzpicture}
}
\expandafter\def\csname e8diagram:dilute3_2_16\endcsname{%
\begin{tikzpicture}[execute at begin picture={\clip (-1.24,-2.736) rectangle ++(2.48,3.472);},x={(1.7363cm,0cm)},y={(0cm,-1.7363cm)},baseline={([yshift=-\the\dimexpr\fontdimen22\textfont2\relax]current  bounding  box.center)},line join=round]
\begin{scope}
\begin{scope}
\path[clip] (-1,.4) -- (-1,-2.4) -- (1,-2.4) -- (1,.4) -- cycle;
\path[fill={rgb,255:red,204;green,204;blue,255},draw=none] (-1,.4) -- (-1,-2.4) -- (1,-2.4) -- (1,.4) -- cycle;
\begin{scope}[every path/.append style={draw=red,line width=.05619cm}]
\path svg[xscale=1.7363cm,yscale=-1.7363cm] {M -1 0 Q -.55172 -.29555 -.46414 -.59109};
\path svg[xscale=1.7363cm,yscale=-1.7363cm] {M 1 -1.5 Q .64554 -1.5134 .28605 -1.5267};
\path svg[xscale=1.7363cm,yscale=-1.7363cm] {M .27474 -.82326 Q -.08284 -.70717 -.46414 -.59109};
\path svg[xscale=1.7363cm,yscale=-1.7363cm] {M -1 -1 Q -.72813 -.79555 -.46414 -.59109};
\path svg[xscale=1.7363cm,yscale=-1.7363cm] {M 1 -.5 Q .61055 -.66163 .27474 -.82326};
\path svg[xscale=1.7363cm,yscale=-1.7363cm] {M .28605 -1.5267 Q .26452 -1.175 .27474 -.82326};
\path svg[xscale=1.7363cm,yscale=-1.7363cm] {M -1 -2 Q -.11491 -1.7634 .28605 -1.5267};
\end{scope}
\path[draw=blue,line width=.01405cm] (-1,.4) -- (-1,-2.4) -- (1,-2.4) -- (1,.4) -- cycle;
\end{scope}
\end{scope}
\end{tikzpicture}
}
\expandafter\def\csname e8diagram:dilute3_2_17\endcsname{%
\begin{tikzpicture}[execute at begin picture={\clip (-1.24,-2.736) rectangle ++(2.48,3.472);},x={(1.7363cm,0cm)},y={(0cm,-1.7363cm)},baseline={([yshift=-\the\dimexpr\fontdimen22\textfont2\relax]current  bounding  box.center)},line join=round]
\begin{scope}
\begin{scope}
\path[clip] (-1,.4) -- (-1,-2.4) -- (1,-2.4) -- (1,.4) -- cycle;
\path[fill={rgb,255:red,204;green,204;blue,255},draw=none] (-1,.4) -- (-1,-2.4) -- (1,-2.4) -- (1,.4) -- cycle;
\begin{scope}[every path/.append style={draw=red,line width=.05619cm}]
\path svg[xscale=1.7363cm,yscale=-1.7363cm] {M 1 -.5 Q .40699 -.50032 -.18583 -.50065};
\path svg[xscale=1.7363cm,yscale=-1.7363cm] {M -1 -2 Q .32407 -1.75 1 -1.5};
\path svg[xscale=1.7363cm,yscale=-1.7363cm] {M -1 -1 Q -.59841 -.75032 -.18583 -.50065};
\path svg[xscale=1.7363cm,yscale=-1.7363cm] {M -1 0 Q -.39777 -.25032 -.18583 -.50065};
\end{scope}
\path[draw=blue,line width=.01405cm] (-1,.4) -- (-1,-2.4) -- (1,-2.4) -- (1,.4) -- cycle;
\end{scope}
\end{scope}
\end{tikzpicture}
}
\expandafter\def\csname e8diagram:dilute3_2_18\endcsname{%
\begin{tikzpicture}[execute at begin picture={\clip (-1.24,-2.736) rectangle ++(2.48,3.472);},x={(1.7363cm,0cm)},y={(0cm,-1.7363cm)},baseline={([yshift=-\the\dimexpr\fontdimen22\textfont2\relax]current  bounding  box.center)},line join=round]
\begin{scope}
\begin{scope}
\path[clip] (-1,.4) -- (-1,-2.4) -- (1,-2.4) -- (1,.4) -- cycle;
\path[fill={rgb,255:red,204;green,204;blue,255},draw=none] (-1,.4) -- (-1,-2.4) -- (1,-2.4) -- (1,.4) -- cycle;
\begin{scope}[every path/.append style={draw=red,line width=.05619cm}]
\path svg[xscale=1.7363cm,yscale=-1.7363cm] {M 1 -1.5 Q .40699 -1.4997 -.18583 -1.4994};
\path svg[xscale=1.7363cm,yscale=-1.7363cm] {M -1 -2 Q -.39777 -1.7497 -.18583 -1.4994};
\path svg[xscale=1.7363cm,yscale=-1.7363cm] {M -1 0 Q .32407 -.25 1 -.5};
\path svg[xscale=1.7363cm,yscale=-1.7363cm] {M -1 -1 Q -.59841 -1.2497 -.18583 -1.4994};
\end{scope}
\path[draw=blue,line width=.01405cm] (-1,.4) -- (-1,-2.4) -- (1,-2.4) -- (1,.4) -- cycle;
\end{scope}
\end{scope}
\end{tikzpicture}
}
\expandafter\def\csname e8diagram:dilute3_2_2\endcsname{%
\begin{tikzpicture}[execute at begin picture={\clip (-1.24,-2.736) rectangle ++(2.48,3.472);},x={(1.7363cm,0cm)},y={(0cm,-1.7363cm)},baseline={([yshift=-\the\dimexpr\fontdimen22\textfont2\relax]current  bounding  box.center)},line join=round]
\begin{scope}
\begin{scope}
\path[clip] (-1,.4) -- (-1,-2.4) -- (1,-2.4) -- (1,.4) -- cycle;
\path[fill={rgb,255:red,204;green,204;blue,255},draw=none] (-1,.4) -- (-1,-2.4) -- (1,-2.4) -- (1,.4) -- cycle;
\begin{scope}[every path/.append style={draw=red,line width=.05619cm}]
\path svg[xscale=1.7363cm,yscale=-1.7363cm] {M -1 -1 Q -.05611 -.75 1 -.5};
\path svg[xscale=1.7363cm,yscale=-1.7363cm] {M -1 -2 Q .32407 -1.75 1 -1.5};
\end{scope}
\path[draw=blue,line width=.01405cm] (-1,.4) -- (-1,-2.4) -- (1,-2.4) -- (1,.4) -- cycle;
\end{scope}
\begin{scope}
\node[fill=white,rectangle,inner sep=2pt,draw=black,scale=.99115,line width=.00281cm] at (-1,0) {};
\end{scope}
\end{scope}
\end{tikzpicture}
}
\expandafter\def\csname e8diagram:dilute3_2_3\endcsname{%
\begin{tikzpicture}[execute at begin picture={\clip (-1.24,-2.736) rectangle ++(2.48,3.472);},x={(1.7363cm,0cm)},y={(0cm,-1.7363cm)},baseline={([yshift=-\the\dimexpr\fontdimen22\textfont2\relax]current  bounding  box.center)},line join=round]
\begin{scope}
\begin{scope}
\path[clip] (-1,.4) -- (-1,-2.4) -- (1,-2.4) -- (1,.4) -- cycle;
\path[fill={rgb,255:red,204;green,204;blue,255},draw=none] (-1,.4) -- (-1,-2.4) -- (1,-2.4) -- (1,.4) -- cycle;
\begin{scope}[every path/.append style={draw=red,line width=.05619cm}]
\path svg[xscale=1.7363cm,yscale=-1.7363cm] {M -1 0 Q -.29579 -.30937 -.04675 -.61874};
\path svg[xscale=1.7363cm,yscale=-1.7363cm] {M 1 -1.5 Q .74839 -1.3891 .5264 -1.2782};
\path svg[xscale=1.7363cm,yscale=-1.7363cm] {M -.04675 -.61874 Q .24394 -.67026 .5264 -.72179};
\path svg[xscale=1.7363cm,yscale=-1.7363cm] {M 1 -.5 Q .74839 -.61089 .5264 -.72179};
\path svg[xscale=1.7363cm,yscale=-1.7363cm] {M -.04675 -1.3813 Q -.04385 -1 -.04675 -.61874};
\path svg[xscale=1.7363cm,yscale=-1.7363cm] {M -.04675 -1.3813 Q .24394 -1.3297 .5264 -1.2782};
\path svg[xscale=1.7363cm,yscale=-1.7363cm] {M -1 -2 Q -.29579 -1.6906 -.04675 -1.3813};
\path svg[xscale=1.7363cm,yscale=-1.7363cm] {M .5264 -1.2782 Q .50928 -1 .5264 -.72179};
\end{scope}
\path[draw=blue,line width=.01405cm] (-1,.4) -- (-1,-2.4) -- (1,-2.4) -- (1,.4) -- cycle;
\end{scope}
\begin{scope}
\node[fill=white,rectangle,inner sep=2pt,draw=black,scale=.99115,line width=.00281cm] at (-1,-1) {};
\end{scope}
\end{scope}
\end{tikzpicture}
}
\expandafter\def\csname e8diagram:dilute3_2_4\endcsname{%
\begin{tikzpicture}[execute at begin picture={\clip (-1.24,-2.736) rectangle ++(2.48,3.472);},x={(1.7363cm,0cm)},y={(0cm,-1.7363cm)},baseline={([yshift=-\the\dimexpr\fontdimen22\textfont2\relax]current  bounding  box.center)},line join=round]
\begin{scope}
\begin{scope}
\path[clip] (-1,.4) -- (-1,-2.4) -- (1,-2.4) -- (1,.4) -- cycle;
\path[fill={rgb,255:red,204;green,204;blue,255},draw=none] (-1,.4) -- (-1,-2.4) -- (1,-2.4) -- (1,.4) -- cycle;
\begin{scope}[every path/.append style={draw=red,line width=.05619cm}]
\path svg[xscale=1.7363cm,yscale=-1.7363cm] {M -1 0 Q -.11948 -.2901 .25322 -.58019};
\path svg[xscale=1.7363cm,yscale=-1.7363cm] {M 1 -1.5 Q .6191 -1.4599 .25322 -1.4198};
\path svg[xscale=1.7363cm,yscale=-1.7363cm] {M 1 -.5 Q .6191 -.5401 .25322 -.58019};
\path svg[xscale=1.7363cm,yscale=-1.7363cm] {M .25322 -1.4198 Q .23406 -1 .25322 -.58019};
\path svg[xscale=1.7363cm,yscale=-1.7363cm] {M -1 -2 Q -.11948 -1.7099 .25322 -1.4198};
\end{scope}
\path[draw=blue,line width=.01405cm] (-1,.4) -- (-1,-2.4) -- (1,-2.4) -- (1,.4) -- cycle;
\end{scope}
\begin{scope}
\node[fill=white,rectangle,inner sep=2pt,draw=black,scale=.99115,line width=.00281cm] at (-1,-1) {};
\end{scope}
\end{scope}
\end{tikzpicture}
}
\expandafter\def\csname e8diagram:dilute3_2_5\endcsname{%
\begin{tikzpicture}[execute at begin picture={\clip (-1.24,-2.736) rectangle ++(2.48,3.472);},x={(1.7363cm,0cm)},y={(0cm,-1.7363cm)},baseline={([yshift=-\the\dimexpr\fontdimen22\textfont2\relax]current  bounding  box.center)},line join=round]
\begin{scope}
\begin{scope}
\path[clip] (-1,.4) -- (-1,-2.4) -- (1,-2.4) -- (1,.4) -- cycle;
\path[fill={rgb,255:red,204;green,204;blue,255},draw=none] (-1,.4) -- (-1,-2.4) -- (1,-2.4) -- (1,.4) -- cycle;
\begin{scope}[every path/.append style={draw=red,line width=.05619cm}]
\path svg[xscale=1.7363cm,yscale=-1.7363cm] {M -1 -2 Q .32407 -1.75 1 -1.5};
\path svg[xscale=1.7363cm,yscale=-1.7363cm] {M -1 0 Q .32407 -.25 1 -.5};
\end{scope}
\path[draw=blue,line width=.01405cm] (-1,.4) -- (-1,-2.4) -- (1,-2.4) -- (1,.4) -- cycle;
\end{scope}
\begin{scope}
\node[fill=white,rectangle,inner sep=2pt,draw=black,scale=.99115,line width=.00281cm] at (-1,-1) {};
\end{scope}
\end{scope}
\end{tikzpicture}
}
\expandafter\def\csname e8diagram:dilute3_2_6\endcsname{%
\begin{tikzpicture}[execute at begin picture={\clip (-1.24,-2.736) rectangle ++(2.48,3.472);},x={(1.7363cm,0cm)},y={(0cm,-1.7363cm)},baseline={([yshift=-\the\dimexpr\fontdimen22\textfont2\relax]current  bounding  box.center)},line join=round]
\begin{scope}
\begin{scope}
\path[clip] (-1,.4) -- (-1,-2.4) -- (1,-2.4) -- (1,.4) -- cycle;
\path[fill={rgb,255:red,204;green,204;blue,255},draw=none] (-1,.4) -- (-1,-2.4) -- (1,-2.4) -- (1,.4) -- cycle;
\begin{scope}[every path/.append style={draw=red,line width=.05619cm}]
\path svg[xscale=1.7363cm,yscale=-1.7363cm] {M -1 0 Q -.39122 -.20007 -.14374 -.40014};
\path svg[xscale=1.7363cm,yscale=-1.7363cm] {M 1 -1.5 Q .65936 -1.3028 .3788 -1.1056};
\path svg[xscale=1.7363cm,yscale=-1.7363cm] {M -.14374 -.40014 Q .18151 -.49496 .47742 -.58978};
\path svg[xscale=1.7363cm,yscale=-1.7363cm] {M -1 -1 Q -.66967 -.94416 -.33937 -.88833};
\path svg[xscale=1.7363cm,yscale=-1.7363cm] {M 1 -.5 Q .73248 -.54489 .47742 -.58978};
\path svg[xscale=1.7363cm,yscale=-1.7363cm] {M -.33937 -.88833 Q -.24206 -.64423 -.14374 -.40014};
\path svg[xscale=1.7363cm,yscale=-1.7363cm] {M -.33937 -.88833 Q .01971 -.99696 .3788 -1.1056};
\path svg[xscale=1.7363cm,yscale=-1.7363cm] {M .3788 -1.1056 Q .414 -.84768 .47742 -.58978};
\end{scope}
\path[draw=blue,line width=.01405cm] (-1,.4) -- (-1,-2.4) -- (1,-2.4) -- (1,.4) -- cycle;
\end{scope}
\begin{scope}
\node[fill=white,rectangle,inner sep=2pt,draw=black,scale=.99115,line width=.00281cm] at (-1,-2) {};
\end{scope}
\end{scope}
\end{tikzpicture}
}
\expandafter\def\csname e8diagram:dilute3_2_7\endcsname{%
\begin{tikzpicture}[execute at begin picture={\clip (-1.24,-2.736) rectangle ++(2.48,3.472);},x={(1.7363cm,0cm)},y={(0cm,-1.7363cm)},baseline={([yshift=-\the\dimexpr\fontdimen22\textfont2\relax]current  bounding  box.center)},line join=round]
\begin{scope}
\begin{scope}
\path[clip] (-1,.4) -- (-1,-2.4) -- (1,-2.4) -- (1,.4) -- cycle;
\path[fill={rgb,255:red,204;green,204;blue,255},draw=none] (-1,.4) -- (-1,-2.4) -- (1,-2.4) -- (1,.4) -- cycle;
\begin{scope}[every path/.append style={draw=red,line width=.05619cm}]
\path svg[xscale=1.7363cm,yscale=-1.7363cm] {M -1 0 Q -.16764 -.21984 .20944 -.43968};
\path svg[xscale=1.7363cm,yscale=-1.7363cm] {M 1 -1.5 Q .46283 -1.2719 .00620 -1.0439};
\path svg[xscale=1.7363cm,yscale=-1.7363cm] {M 1 -.5 Q .61114 -.46984 .20944 -.43968};
\path svg[xscale=1.7363cm,yscale=-1.7363cm] {M .00620 -1.0439 Q .09589 -.74178 .20944 -.43968};
\path svg[xscale=1.7363cm,yscale=-1.7363cm] {M -1 -1 Q -.497 -1.0219 .00620 -1.0439};
\end{scope}
\path[draw=blue,line width=.01405cm] (-1,.4) -- (-1,-2.4) -- (1,-2.4) -- (1,.4) -- cycle;
\end{scope}
\begin{scope}
\node[fill=white,rectangle,inner sep=2pt,draw=black,scale=.99115,line width=.00281cm] at (-1,-2) {};
\end{scope}
\end{scope}
\end{tikzpicture}
}
\expandafter\def\csname e8diagram:dilute3_2_8\endcsname{%
\begin{tikzpicture}[execute at begin picture={\clip (-1.24,-2.736) rectangle ++(2.48,3.472);},x={(1.7363cm,0cm)},y={(0cm,-1.7363cm)},baseline={([yshift=-\the\dimexpr\fontdimen22\textfont2\relax]current  bounding  box.center)},line join=round]
\begin{scope}
\begin{scope}
\path[clip] (-1,.4) -- (-1,-2.4) -- (1,-2.4) -- (1,.4) -- cycle;
\path[fill={rgb,255:red,204;green,204;blue,255},draw=none] (-1,.4) -- (-1,-2.4) -- (1,-2.4) -- (1,.4) -- cycle;
\begin{scope}[every path/.append style={draw=red,line width=.05619cm}]
\path svg[xscale=1.7363cm,yscale=-1.7363cm] {M -1 0 Q .32407 -.25 1 -.5};
\path svg[xscale=1.7363cm,yscale=-1.7363cm] {M -1 -1 Q -.05611 -1.25 1 -1.5};
\end{scope}
\path[draw=blue,line width=.01405cm] (-1,.4) -- (-1,-2.4) -- (1,-2.4) -- (1,.4) -- cycle;
\end{scope}
\begin{scope}
\node[fill=white,rectangle,inner sep=2pt,draw=black,scale=.99115,line width=.00281cm] at (-1,-2) {};
\end{scope}
\end{scope}
\end{tikzpicture}
}
\expandafter\def\csname e8diagram:dilute3_2_9\endcsname{%
\begin{tikzpicture}[execute at begin picture={\clip (-1.24,-2.736) rectangle ++(2.48,3.472);},x={(1.7363cm,0cm)},y={(0cm,-1.7363cm)},baseline={([yshift=-\the\dimexpr\fontdimen22\textfont2\relax]current  bounding  box.center)},line join=round]
\begin{scope}
\begin{scope}
\path[clip] (-1,.4) -- (-1,-2.4) -- (1,-2.4) -- (1,.4) -- cycle;
\path[fill={rgb,255:red,204;green,204;blue,255},draw=none] (-1,.4) -- (-1,-2.4) -- (1,-2.4) -- (1,.4) -- cycle;
\begin{scope}[every path/.append style={draw=red,line width=.05619cm}]
\path svg[xscale=1.7363cm,yscale=-1.7363cm] {M .44551 -1.3018 Q .42841 -1 .44551 -.69816};
\path svg[xscale=1.7363cm,yscale=-1.7363cm] {M -1 0 Q -.41397 -.24479 -.20927 -.48958};
\path svg[xscale=1.7363cm,yscale=-1.7363cm] {M 1 -1.5 Q .70829 -1.4009 .44551 -1.3018};
\path svg[xscale=1.7363cm,yscale=-1.7363cm] {M -.53442 -1 Q -.36999 -.74479 -.20927 -.48958};
\path svg[xscale=1.7363cm,yscale=-1.7363cm] {M -1 -1 Q -.76721 -1 -.53442 -1};
\path svg[xscale=1.7363cm,yscale=-1.7363cm] {M -.20927 -.48958 Q .13076 -.59387 .44551 -.69816};
\path svg[xscale=1.7363cm,yscale=-1.7363cm] {M -.20927 -1.5104 Q -.36999 -1.2552 -.53442 -1};
\path svg[xscale=1.7363cm,yscale=-1.7363cm] {M -.20927 -1.5104 Q .13076 -1.4061 .44551 -1.3018};
\path svg[xscale=1.7363cm,yscale=-1.7363cm] {M 1 -.5 Q .70829 -.59908 .44551 -.69816};
\path svg[xscale=1.7363cm,yscale=-1.7363cm] {M -1 -2 Q -.41397 -1.7552 -.20927 -1.5104};
\end{scope}
\path[draw=blue,line width=.01405cm] (-1,.4) -- (-1,-2.4) -- (1,-2.4) -- (1,.4) -- cycle;
\end{scope}
\end{scope}
\end{tikzpicture}
}
\expandafter\def\csname e8diagram:dilute3_2_proj1\endcsname{%
\begin{tikzpicture}[execute at begin picture={\clip (-1.24,-2.736) rectangle ++(2.48,3.472);},x={(1.7363cm,0cm)},y={(0cm,-1.7363cm)},baseline={([yshift=-\the\dimexpr\fontdimen22\textfont2\relax]current  bounding  box.center)},line join=round]
\begin{scope}
\begin{scope}
\path[clip] (-1,.4) -- (-1,-2.4) -- (1,-2.4) -- (1,.4) -- cycle;
\path[fill={rgb,255:red,204;green,204;blue,255},draw=none] (-1,.4) -- (-1,-2.4) -- (1,-2.4) -- (1,.4) -- cycle;
\begin{scope}[every path/.append style={draw=red,line width=.05619cm}]
\draw (-1,0) -- (-.5,-.25) (.5,-.25) -- (1,-.5);
\draw(-1,-1) -- (-.5,-1.25) (.5,-1.25) -- (1,-1.5);
\end{scope}
\path[draw=blue,line width=.01405cm] (-1,.4) -- (-1,-2.4) -- (1,-2.4) -- (1,.4) -- cycle;
\end{scope}
  \filldraw [fill=sienna, draw=black,thin] (-.5,-.05) rectangle (.5,-1.45);
\node at (0,-.75) {$\scriptstyle \pi_R$};
\begin{scope}
\node[fill=white,rectangle,inner sep=2pt,draw=black,scale=.99115,line width=.00281cm] at (-1,-2) {};
\end{scope}
\end{scope}
\end{tikzpicture}
}
\expandafter\def\csname e8diagram:dilute3_2_proj2\endcsname{%
\begin{tikzpicture}[execute at begin picture={\clip (-1.24,-2.736) rectangle ++(2.48,3.472);},x={(1.7363cm,0cm)},y={(0cm,-1.7363cm)},baseline={([yshift=-\the\dimexpr\fontdimen22\textfont2\relax]current  bounding  box.center)},line join=round]
\begin{scope}
\begin{scope}
\path[clip] (-1,.4) -- (-1,-2.4) -- (1,-2.4) -- (1,.4) -- cycle;
\path[fill={rgb,255:red,204;green,204;blue,255},draw=none] (-1,.4) -- (-1,-2.4) -- (1,-2.4) -- (1,.4) -- cycle;
\begin{scope}[every path/.append style={draw=red,line width=.05619cm}]
\draw (-1,0) -- (-.5,-.5) (.5,-.5) -- (1,-.5);
\draw(-1,-2) -- (-.5,-1.5) (.5,-1.5) -- (1,-1.5);
\end{scope}
\path[draw=blue,line width=.01405cm] (-1,.4) -- (-1,-2.4) -- (1,-2.4) -- (1,.4) -- cycle;
\end{scope}
  \filldraw [fill=sienna, draw=black,thin] (-.5,-.3) rectangle (.5,-1.7);
\node at (0,-1) {$\scriptstyle \pi_R$};
\begin{scope}
\node[fill=white,rectangle,inner sep=2pt,draw=black,scale=.99115,line width=.00281cm] at (-1,-1) {};
\end{scope}
\end{scope}
\end{tikzpicture}
}
\expandafter\def\csname e8diagram:dilute3_2_proj3\endcsname{%
\begin{tikzpicture}[execute at begin picture={\clip (-1.24,-2.736) rectangle ++(2.48,3.472);},x={(1.7363cm,0cm)},y={(0cm,-1.7363cm)},baseline={([yshift=-\the\dimexpr\fontdimen22\textfont2\relax]current  bounding  box.center)},line join=round]
\begin{scope}
\begin{scope}
\path[clip] (-1,.4) -- (-1,-2.4) -- (1,-2.4) -- (1,.4) -- cycle;
\path[fill={rgb,255:red,204;green,204;blue,255},draw=none] (-1,.4) -- (-1,-2.4) -- (1,-2.4) -- (1,.4) -- cycle;
\begin{scope}[every path/.append style={draw=red,line width=.05619cm}]
\draw (-1,-1) -- (-.5,-.75) (.5,-.75) -- (1,-.5);
\draw(-1,-2) -- (-.5,-1.75) (.5,-1.75) -- (1,-1.5);
\end{scope}
\path[draw=blue,line width=.01405cm] (-1,.4) -- (-1,-2.4) -- (1,-2.4) -- (1,.4) -- cycle;
\end{scope}
  \filldraw [fill=sienna, draw=black,thin] (-.5,-.55) rectangle (.5,-1.95);
\node at (0,-1.25) {$\scriptstyle \pi_R$};
\begin{scope}
\node[fill=white,rectangle,inner sep=2pt,draw=black,scale=.99115,line width=.00281cm] at (-1,0) {};
\end{scope}
\end{scope}
\end{tikzpicture}
}
\expandafter\def\csname e8diagram:overlap-squarehexa-mod\endcsname{%
\begin{tikzpicture}[execute at begin picture={\clip (-1.24,-1.24) rectangle ++(2.48,2.48);},x={(2.11230cm,0cm)},y={(0cm,-2.11230cm)},baseline=(current  bounding  box.center),every path/.style={,},line join=round]
\begin{scope}
\path[clip] (0,0) circle[radius=1];
\path[fill=blue!20,draw=none] (0,0) circle[radius=1];
\begin{scope}[every path/.append style={draw=red,ultra thick}]
\path (.5,-.86603) -- (.49353,-.4197);
\path (-.49353,.4197) -- (-.18417,.17994);
\path (.49353,.4197) -- (.72509,0);
\path (-.18417,-.17994) -- (-.18417,.17994);
\path (1,0) -- (.72509,0);
\path (-.18417,-.17994) -- (.18417,-.17994);
\path (.49353,-.4197) -- (.72509,0);
\path (-.72509,0) -- (-.49353,.4197);
\path (.5,.86603) -- (.49353,.4197);
\path (.18417,.17994) -- (.49353,.4197);
\path (-.5,.86603) -- (-.49353,.4197);
\path (-.49353,-.4197) -- (-.72509,0);
\path (-.49353,-.4197) -- (-.18417,-.17994);
\path (.18417,-.17994) -- (.18417,.17994);
\path (.18417,-.17994) -- (.49353,-.4197);
\path (-1,0) -- (-.72509,0);
\path (-.18417,.17994) -- (.18417,.17994);
\path (-.5,-.86603) -- (-.49353,-.4197);
\end{scope}
\path[draw=blue,thin] (0,0) circle[radius=1];
\draw[dotted,yellow!50!black] (-.26,0) ellipse [x radius=0.65,y radius=0.55];
\draw[dotted,cyan!50!black] (.26,0) ellipse [x radius=0.65,y radius=0.55];
\end{scope}
\end{tikzpicture}
}
\expandafter\def\csname e8diagram:round2_0\endcsname{%
\begin{tikzpicture}[execute at begin picture={\clip (-1.24,-1.24) rectangle ++(2.48,2.48);},x={(2.1123cm,0cm)},y={(0cm,-2.1123cm)},baseline={([yshift=-\the\dimexpr\fontdimen22\textfont2\relax]current  bounding  box.center)},line join=round]
\begin{scope}
\path[clip] (0,0) circle[radius=1];
\path[fill={rgb,255:red,204;green,204;blue,255},draw=none] (0,0) circle[radius=1];
\begin{scope}[every path/.append style={draw=red,line width=.05619cm}]
\path svg[xscale=2.1123cm,yscale=-2.1123cm] {M -1 0 C -.29412 0 .29412 0 1 0};
\end{scope}
\path[draw=blue,line width=.01405cm] (0,0) circle[radius=1];
\end{scope}
\end{tikzpicture}
}
\expandafter\def\csname e8diagram:round3_0\endcsname{%
\begin{tikzpicture}[execute at begin picture={\clip (-1.24,-1.24) rectangle ++(2.48,2.48);},x={(2.1123cm,0cm)},y={(0cm,-2.1123cm)},baseline={([yshift=-\the\dimexpr\fontdimen22\textfont2\relax]current  bounding  box.center)},line join=round]
\begin{scope}
\path[clip] (0,0) circle[radius=1];
\path[fill={rgb,255:red,204;green,204;blue,255},draw=none] (0,0) circle[radius=1];
\begin{scope}[every path/.append style={draw=red,line width=.05619cm}]
\path svg[xscale=2.1123cm,yscale=-2.1123cm] {M .5 .86603 C .3125 .54127 0 0 0 0};
\path svg[xscale=2.1123cm,yscale=-2.1123cm] {M .5 -.86603 C .3125 -.54127 0 0 0 0};
\path svg[xscale=2.1123cm,yscale=-2.1123cm] {M -1 0 C -.625 0 0 0 0 0};
\end{scope}
\path[draw=blue,line width=.01405cm] (0,0) circle[radius=1];
\end{scope}
\end{tikzpicture}
}
\expandafter\def\csname e8diagram:round4_0\endcsname{%
\begin{tikzpicture}[execute at begin picture={\clip (-1.24,-1.24) rectangle ++(2.48,2.48);},x={(2.1123cm,0cm)},y={(0cm,-2.1123cm)},baseline={([yshift=-\the\dimexpr\fontdimen22\textfont2\relax]current  bounding  box.center)},line join=round]
\begin{scope}
\path[clip] (0,0) circle[radius=1];
\path[fill={rgb,255:red,204;green,204;blue,255},draw=none] (0,0) circle[radius=1];
\begin{scope}[every path/.append style={draw=red,line width=.05619cm}]
\path svg[xscale=2.1123cm,yscale=-2.1123cm] {M -.70711 .70711 C -.32141 .32141 .32141 .32141 .70711 .70711};
\path svg[xscale=2.1123cm,yscale=-2.1123cm] {M -.70711 -.70711 C -.32141 -.32141 .32141 -.32141 .70711 -.70711};
\end{scope}
\path[draw=blue,line width=.01405cm] (0,0) circle[radius=1];
\end{scope}
\end{tikzpicture}
}
\expandafter\def\csname e8diagram:round4_1\endcsname{%
\begin{tikzpicture}[execute at begin picture={\clip (-1.24,-1.24) rectangle ++(2.48,2.48);},x={(2.1123cm,0cm)},y={(0cm,-2.1123cm)},baseline={([yshift=-\the\dimexpr\fontdimen22\textfont2\relax]current  bounding  box.center)},line join=round]
\begin{scope}
\path[clip] (0,0) circle[radius=1];
\path[fill={rgb,255:red,204;green,204;blue,255},draw=none] (0,0) circle[radius=1];
\begin{scope}[every path/.append style={draw=red,line width=.05619cm}]
\path svg[xscale=2.1123cm,yscale=-2.1123cm] {M .70711 .70711 C .32141 .32141 .32141 -.32141 .70711 -.70711};
\path svg[xscale=2.1123cm,yscale=-2.1123cm] {M -.70711 -.70711 C -.32141 -.32141 -.32141 .32141 -.70711 .70711};
\end{scope}
\path[draw=blue,line width=.01405cm] (0,0) circle[radius=1];
\end{scope}
\end{tikzpicture}
}
\expandafter\def\csname e8diagram:round4_2\endcsname{%
\begin{tikzpicture}[execute at begin picture={\clip (-1.24,-1.24) rectangle ++(2.48,2.48);},x={(2.1123cm,0cm)},y={(0cm,-2.1123cm)},baseline={([yshift=-\the\dimexpr\fontdimen22\textfont2\relax]current  bounding  box.center)},line join=round]
\begin{scope}
\path[clip] (0,0) circle[radius=1];
\path[fill={rgb,255:red,204;green,204;blue,255},draw=none] (0,0) circle[radius=1];
\begin{scope}[every path/.append style={draw=red,line width=.05619cm}]
\path svg[xscale=2.1123cm,yscale=-2.1123cm] {M .70711 -.70711 C .51912 -.51912 .38532 0 .38532 0};
\path svg[xscale=2.1123cm,yscale=-2.1123cm] {M -.38532 0 C -.38532 0 .38532 0 .38532 0};
\path svg[xscale=2.1123cm,yscale=-2.1123cm] {M -.70711 -.70711 C -.51912 -.51912 -.38532 0 -.38532 0};
\path svg[xscale=2.1123cm,yscale=-2.1123cm] {M -.70711 .70711 C -.51912 .51912 -.38532 0 -.38532 0};
\path svg[xscale=2.1123cm,yscale=-2.1123cm] {M .70711 .70711 C .51912 .51912 .38532 0 .38532 0};
\end{scope}
\path[draw=blue,line width=.01405cm] (0,0) circle[radius=1];
\end{scope}
\end{tikzpicture}
}
\expandafter\def\csname e8diagram:round4_3\endcsname{%
\begin{tikzpicture}[execute at begin picture={\clip (-1.24,-1.24) rectangle ++(2.48,2.48);},x={(2.1123cm,0cm)},y={(0cm,-2.1123cm)},baseline={([yshift=-\the\dimexpr\fontdimen22\textfont2\relax]current  bounding  box.center)},line join=round]
\begin{scope}
\path[clip] (0,0) circle[radius=1];
\path[fill={rgb,255:red,204;green,204;blue,255},draw=none] (0,0) circle[radius=1];
\begin{scope}[every path/.append style={draw=red,line width=.05619cm}]
\path svg[xscale=2.1123cm,yscale=-2.1123cm] {M 0 -.38532 C 0 -.38532 0 .38532 0 .38532};
\path svg[xscale=2.1123cm,yscale=-2.1123cm] {M -.70711 -.70711 C -.51912 -.51912 0 -.38532 0 -.38532};
\path svg[xscale=2.1123cm,yscale=-2.1123cm] {M -.70711 .70711 C -.51912 .51912 0 .38532 0 .38532};
\path svg[xscale=2.1123cm,yscale=-2.1123cm] {M .70711 -.70711 C .51912 -.51912 0 -.38532 0 -.38532};
\path svg[xscale=2.1123cm,yscale=-2.1123cm] {M .70711 .70711 C .51912 .51912 0 .38532 0 .38532};
\end{scope}
\path[draw=blue,line width=.01405cm] (0,0) circle[radius=1];
\end{scope}
\end{tikzpicture}
}
\expandafter\def\csname e8diagram:round4_4\endcsname{%
\begin{tikzpicture}[execute at begin picture={\clip (-1.24,-1.24) rectangle ++(2.48,2.48);},x={(2.1123cm,0cm)},y={(0cm,-2.1123cm)},baseline={([yshift=-\the\dimexpr\fontdimen22\textfont2\relax]current  bounding  box.center)},line join=round]
\begin{scope}
\path[clip] (0,0) circle[radius=1];
\path[fill={rgb,255:red,204;green,204;blue,255},draw=none] (0,0) circle[radius=1];
\begin{scope}[every path/.append style={draw=red,line width=.05619cm}]
\path svg[xscale=2.1123cm,yscale=-2.1123cm] {M .70711 .70711 C .58847 .58847 .29722 .29722 .29722 .29722};
\path svg[xscale=2.1123cm,yscale=-2.1123cm] {M -.29722 -.29722 C -.29722 -.29722 -.29722 .29722 -.29722 .29722};
\path svg[xscale=2.1123cm,yscale=-2.1123cm] {M -.29722 -.29722 C -.29722 -.29722 .29722 -.29722 .29722 -.29722};
\path svg[xscale=2.1123cm,yscale=-2.1123cm] {M .29722 -.29722 C .29722 -.29722 .29722 .29722 .29722 .29722};
\path svg[xscale=2.1123cm,yscale=-2.1123cm] {M -.70711 .70711 C -.58847 .58847 -.29722 .29722 -.29722 .29722};
\path svg[xscale=2.1123cm,yscale=-2.1123cm] {M .70711 -.70711 C .58847 -.58847 .29722 -.29722 .29722 -.29722};
\path svg[xscale=2.1123cm,yscale=-2.1123cm] {M -.29722 .29722 C -.29722 .29722 .29722 .29722 .29722 .29722};
\path svg[xscale=2.1123cm,yscale=-2.1123cm] {M -.70711 -.70711 C -.58847 -.58847 -.29722 -.29722 -.29722 -.29722};
\end{scope}
\path[draw=blue,line width=.01405cm] (0,0) circle[radius=1];
\end{scope}
\end{tikzpicture}
}
\expandafter\def\csname e8diagram:round5_0\endcsname{%
\begin{tikzpicture}[execute at begin picture={\clip (-1.24,-1.24) rectangle ++(2.48,2.48);},x={(2.1123cm,0cm)},y={(0cm,-2.1123cm)},baseline={([yshift=-\the\dimexpr\fontdimen22\textfont2\relax]current  bounding  box.center)},line join=round]
\begin{scope}
\path[clip] (0,0) circle[radius=1];
\path[fill={rgb,255:red,204;green,204;blue,255},draw=none] (0,0) circle[radius=1];
\begin{scope}[every path/.append style={draw=red,line width=.05619cm}]
\path svg[xscale=2.1123cm,yscale=-2.1123cm] {M 1 0 C .88318 0 .53048 0 .53048 0};
\path svg[xscale=2.1123cm,yscale=-2.1123cm] {M .30902 -.95106 C .19656 -.60495 .53048 0 .53048 0};
\path svg[xscale=2.1123cm,yscale=-2.1123cm] {M -.80902 -.58779 C -.44228 -.32134 -.44228 .32134 -.80902 .58779};
\path svg[xscale=2.1123cm,yscale=-2.1123cm] {M .30902 .95106 C .19656 .60495 .53048 0 .53048 0};
\end{scope}
\path[draw=blue,line width=.01405cm] (0,0) circle[radius=1];
\end{scope}
\end{tikzpicture}
}
\expandafter\def\csname e8diagram:round5_1\endcsname{%
\begin{tikzpicture}[execute at begin picture={\clip (-1.24,-1.24) rectangle ++(2.48,2.48);},x={(2.1123cm,0cm)},y={(0cm,-2.1123cm)},baseline={([yshift=-\the\dimexpr\fontdimen22\textfont2\relax]current  bounding  box.center)},line join=round]
\begin{scope}
\path[clip] (0,0) circle[radius=1];
\path[fill={rgb,255:red,204;green,204;blue,255},draw=none] (0,0) circle[radius=1];
\begin{scope}[every path/.append style={draw=red,line width=.05619cm}]
\path svg[xscale=2.1123cm,yscale=-2.1123cm] {M 1 0 C .63608 0 .16393 -.50451 .16393 -.50451};
\path svg[xscale=2.1123cm,yscale=-2.1123cm] {M .30902 -.95106 C .27292 -.83996 .16393 -.50451 .16393 -.50451};
\path svg[xscale=2.1123cm,yscale=-2.1123cm] {M -.80902 .58779 C -.44228 .32134 .16894 .51994 .30902 .95106};
\path svg[xscale=2.1123cm,yscale=-2.1123cm] {M -.80902 -.58779 C -.5146 -.37388 .16393 -.50451 .16393 -.50451};
\end{scope}
\path[draw=blue,line width=.01405cm] (0,0) circle[radius=1];
\end{scope}
\end{tikzpicture}
}
\expandafter\def\csname e8diagram:round5_10\endcsname{%
\begin{tikzpicture}[execute at begin picture={\clip (-1.24,-1.24) rectangle ++(2.48,2.48);},x={(2.1123cm,0cm)},y={(0cm,-2.1123cm)},baseline={([yshift=-\the\dimexpr\fontdimen22\textfont2\relax]current  bounding  box.center)},line join=round]
\begin{scope}
\path[clip] (0,0) circle[radius=1];
\path[fill={rgb,255:red,204;green,204;blue,255},draw=none] (0,0) circle[radius=1];
\begin{scope}[every path/.append style={draw=red,line width=.05619cm}]
\path svg[xscale=2.1123cm,yscale=-2.1123cm] {M .3031 .43246 C .3031 .43246 .62986 0 .62986 0};
\path svg[xscale=2.1123cm,yscale=-2.1123cm] {M .30902 .95106 C .26608 .8189 .3031 .43246 .3031 .43246};
\path svg[xscale=2.1123cm,yscale=-2.1123cm] {M .30902 -.95106 C .26608 -.8189 .3031 -.43246 .3031 -.43246};
\path svg[xscale=2.1123cm,yscale=-2.1123cm] {M -.05836 0 C -.05836 0 .3031 .43246 .3031 .43246};
\path svg[xscale=2.1123cm,yscale=-2.1123cm] {M -.05836 0 C -.05836 0 .3031 -.43246 .3031 -.43246};
\path svg[xscale=2.1123cm,yscale=-2.1123cm] {M -.80902 .58779 C -.66112 .48033 -.64361 0 -.64361 0};
\path svg[xscale=2.1123cm,yscale=-2.1123cm] {M .3031 -.43246 C .3031 -.43246 .62986 0 .62986 0};
\path svg[xscale=2.1123cm,yscale=-2.1123cm] {M -.64361 0 C -.64361 0 -.05836 0 -.05836 0};
\path svg[xscale=2.1123cm,yscale=-2.1123cm] {M 1 0 C .92404 0 .62986 0 .62986 0};
\path svg[xscale=2.1123cm,yscale=-2.1123cm] {M -.80902 -.58779 C -.66112 -.48033 -.64361 0 -.64361 0};
\end{scope}
\path[draw=blue,line width=.01405cm] (0,0) circle[radius=1];
\end{scope}
\end{tikzpicture}
}
\expandafter\def\csname e8diagram:round5_11\endcsname{%
\begin{tikzpicture}[execute at begin picture={\clip (-1.24,-1.24) rectangle ++(2.48,2.48);},x={(2.1123cm,0cm)},y={(0cm,-2.1123cm)},baseline={([yshift=-\the\dimexpr\fontdimen22\textfont2\relax]current  bounding  box.center)},line join=round]
\begin{scope}
\path[clip] (0,0) circle[radius=1];
\path[fill={rgb,255:red,204;green,204;blue,255},draw=none] (0,0) circle[radius=1];
\begin{scope}[every path/.append style={draw=red,line width=.05619cm}]
\path svg[xscale=2.1123cm,yscale=-2.1123cm] {M .19464 -.59903 C .19464 -.59903 .50496 -.15463 .50496 -.15463};
\path svg[xscale=2.1123cm,yscale=-2.1123cm] {M .30902 .95106 C .25252 .77719 -.19889 .61211 -.19889 .61211};
\path svg[xscale=2.1123cm,yscale=-2.1123cm] {M .30902 -.95106 C .28554 -.87882 .19464 -.59903 .19464 -.59903};
\path svg[xscale=2.1123cm,yscale=-2.1123cm] {M -.01804 .05551 C -.01804 .05551 -.19889 .61211 -.19889 .61211};
\path svg[xscale=2.1123cm,yscale=-2.1123cm] {M -.01804 .05551 C -.01804 .05551 .50496 -.15463 .50496 -.15463};
\path svg[xscale=2.1123cm,yscale=-2.1123cm] {M -.80902 .58779 C -.66112 .48033 -.19889 .61211 -.19889 .61211};
\path svg[xscale=2.1123cm,yscale=-2.1123cm] {M -.31763 -.4219 C -.31763 -.4219 -.01804 .05551 -.01804 .05551};
\path svg[xscale=2.1123cm,yscale=-2.1123cm] {M -.31763 -.4219 C -.31763 -.4219 .19464 -.59903 .19464 -.59903};
\path svg[xscale=2.1123cm,yscale=-2.1123cm] {M 1 0 C .86104 0 .50496 -.15463 .50496 -.15463};
\path svg[xscale=2.1123cm,yscale=-2.1123cm] {M -.80902 -.58779 C -.6966 -.50611 -.31763 -.4219 -.31763 -.4219};
\end{scope}
\path[draw=blue,line width=.01405cm] (0,0) circle[radius=1];
\end{scope}
\end{tikzpicture}
}
\expandafter\def\csname e8diagram:round5_12\endcsname{%
\begin{tikzpicture}[execute at begin picture={\clip (-1.24,-1.24) rectangle ++(2.48,2.48);},x={(2.1123cm,0cm)},y={(0cm,-2.1123cm)},baseline={([yshift=-\the\dimexpr\fontdimen22\textfont2\relax]current  bounding  box.center)},line join=round]
\begin{scope}
\path[clip] (0,0) circle[radius=1];
\path[fill={rgb,255:red,204;green,204;blue,255},draw=none] (0,0) circle[radius=1];
\begin{scope}[every path/.append style={draw=red,line width=.05619cm}]
\path svg[xscale=2.1123cm,yscale=-2.1123cm] {M .30902 .95106 C .25252 .77719 .52069 .3783 .52069 .3783};
\path svg[xscale=2.1123cm,yscale=-2.1123cm] {M .30902 -.95106 C .26608 -.8189 .00898 -.52802 .00898 -.52802};
\path svg[xscale=2.1123cm,yscale=-2.1123cm] {M -.4994 .17171 C -.4994 .17171 .04722 .03431 .04722 .03431};
\path svg[xscale=2.1123cm,yscale=-2.1123cm] {M -.80902 .58779 C -.6966 .50611 -.4994 .17171 -.4994 .17171};
\path svg[xscale=2.1123cm,yscale=-2.1123cm] {M .04722 .03431 C .04722 .03431 .52069 .3783 .52069 .3783};
\path svg[xscale=2.1123cm,yscale=-2.1123cm] {M -.50957 -.37022 C -.50957 -.37022 -.4994 .17171 -.4994 .17171};
\path svg[xscale=2.1123cm,yscale=-2.1123cm] {M -.50957 -.37022 C -.50957 -.37022 .00898 -.52802 .00898 -.52802};
\path svg[xscale=2.1123cm,yscale=-2.1123cm] {M 1 0 C .81719 0 .52069 .3783 .52069 .3783};
\path svg[xscale=2.1123cm,yscale=-2.1123cm] {M -.80902 -.58779 C -.74756 -.54314 -.50957 -.37022 -.50957 -.37022};
\path svg[xscale=2.1123cm,yscale=-2.1123cm] {M .00898 -.52802 C .00898 -.52802 .04722 .03431 .04722 .03431};
\end{scope}
\path[draw=blue,line width=.01405cm] (0,0) circle[radius=1];
\end{scope}
\end{tikzpicture}
}
\expandafter\def\csname e8diagram:round5_13\endcsname{%
\begin{tikzpicture}[execute at begin picture={\clip (-1.24,-1.24) rectangle ++(2.48,2.48);},x={(2.1123cm,0cm)},y={(0cm,-2.1123cm)},baseline={([yshift=-\the\dimexpr\fontdimen22\textfont2\relax]current  bounding  box.center)},line join=round]
\begin{scope}
\path[clip] (0,0) circle[radius=1];
\path[fill={rgb,255:red,204;green,204;blue,255},draw=none] (0,0) circle[radius=1];
\begin{scope}[every path/.append style={draw=red,line width=.05619cm}]
\path svg[xscale=2.1123cm,yscale=-2.1123cm] {M .04722 -.03431 C .04722 -.03431 .52069 -.3783 .52069 -.3783};
\path svg[xscale=2.1123cm,yscale=-2.1123cm] {M .30902 .95106 C .26608 .8189 .00898 .52802 .00898 .52802};
\path svg[xscale=2.1123cm,yscale=-2.1123cm] {M .30902 -.95106 C .25252 -.77719 .52069 -.3783 .52069 -.3783};
\path svg[xscale=2.1123cm,yscale=-2.1123cm] {M -.50957 .37022 C -.50957 .37022 .00898 .52802 .00898 .52802};
\path svg[xscale=2.1123cm,yscale=-2.1123cm] {M -.80902 .58779 C -.74756 .54314 -.50957 .37022 -.50957 .37022};
\path svg[xscale=2.1123cm,yscale=-2.1123cm] {M -.4994 -.17171 C -.4994 -.17171 -.50957 .37022 -.50957 .37022};
\path svg[xscale=2.1123cm,yscale=-2.1123cm] {M -.4994 -.17171 C -.4994 -.17171 .04722 -.03431 .04722 -.03431};
\path svg[xscale=2.1123cm,yscale=-2.1123cm] {M 1 0 C .81719 0 .52069 -.3783 .52069 -.3783};
\path svg[xscale=2.1123cm,yscale=-2.1123cm] {M -.80902 -.58779 C -.6966 -.50611 -.4994 -.17171 -.4994 -.17171};
\path svg[xscale=2.1123cm,yscale=-2.1123cm] {M .04722 -.03431 C .04722 -.03431 .00898 .52802 .00898 .52802};
\end{scope}
\path[draw=blue,line width=.01405cm] (0,0) circle[radius=1];
\end{scope}
\end{tikzpicture}
}
\expandafter\def\csname e8diagram:round5_14\endcsname{%
\begin{tikzpicture}[execute at begin picture={\clip (-1.24,-1.24) rectangle ++(2.48,2.48);},x={(2.1123cm,0cm)},y={(0cm,-2.1123cm)},baseline={([yshift=-\the\dimexpr\fontdimen22\textfont2\relax]current  bounding  box.center)},line join=round]
\begin{scope}
\path[clip] (0,0) circle[radius=1];
\path[fill={rgb,255:red,204;green,204;blue,255},draw=none] (0,0) circle[radius=1];
\begin{scope}[every path/.append style={draw=red,line width=.05619cm}]
\path svg[xscale=2.1123cm,yscale=-2.1123cm] {M -.31763 .4219 C -.31763 .4219 .19464 .59903 .19464 .59903};
\path svg[xscale=2.1123cm,yscale=-2.1123cm] {M .30902 -.95106 C .25252 -.77719 -.19889 -.61211 -.19889 -.61211};
\path svg[xscale=2.1123cm,yscale=-2.1123cm] {M .30902 .95106 C .28554 .87882 .19464 .59903 .19464 .59903};
\path svg[xscale=2.1123cm,yscale=-2.1123cm] {M -.01804 -.05551 C -.01804 -.05551 -.31763 .4219 -.31763 .4219};
\path svg[xscale=2.1123cm,yscale=-2.1123cm] {M -.01804 -.05551 C -.01804 -.05551 .50496 .15463 .50496 .15463};
\path svg[xscale=2.1123cm,yscale=-2.1123cm] {M -.80902 .58779 C -.6966 .50611 -.31763 .4219 -.31763 .4219};
\path svg[xscale=2.1123cm,yscale=-2.1123cm] {M .50496 .15463 C .50496 .15463 .19464 .59903 .19464 .59903};
\path svg[xscale=2.1123cm,yscale=-2.1123cm] {M 1 0 C .86104 0 .50496 .15463 .50496 .15463};
\path svg[xscale=2.1123cm,yscale=-2.1123cm] {M -.19889 -.61211 C -.19889 -.61211 -.01804 -.05551 -.01804 -.05551};
\path svg[xscale=2.1123cm,yscale=-2.1123cm] {M -.80902 -.58779 C -.66112 -.48033 -.19889 -.61211 -.19889 -.61211};
\end{scope}
\path[draw=blue,line width=.01405cm] (0,0) circle[radius=1];
\end{scope}
\end{tikzpicture}
}
\expandafter\def\csname e8diagram:round5_15\endcsname{%
\begin{tikzpicture}[execute at begin picture={\clip (-1.24,-1.24) rectangle ++(2.48,2.48);},x={(2.1123cm,0cm)},y={(0cm,-2.1123cm)},baseline={([yshift=-\the\dimexpr\fontdimen22\textfont2\relax]current  bounding  box.center)},line join=round]
\begin{scope}
\path[clip] (0,0) circle[radius=1];
\path[fill={rgb,255:red,204;green,204;blue,255},draw=none] (0,0) circle[radius=1];
\begin{scope}[every path/.append style={draw=red,line width=.05619cm}]
\path svg[xscale=2.1123cm,yscale=-2.1123cm] {M .1582 -.4869 C .1582 -.4869 .51195 0 .51195 0};
\path svg[xscale=2.1123cm,yscale=-2.1123cm] {M .30902 .95106 C .27038 .83213 .1582 .4869 .1582 .4869};
\path svg[xscale=2.1123cm,yscale=-2.1123cm] {M .30902 -.95106 C .27038 -.83213 .1582 -.4869 .1582 -.4869};
\path svg[xscale=2.1123cm,yscale=-2.1123cm] {M -.41418 .30092 C -.41418 .30092 .1582 .4869 .1582 .4869};
\path svg[xscale=2.1123cm,yscale=-2.1123cm] {M -.80902 .58779 C -.70785 .51429 -.41418 .30092 -.41418 .30092};
\path svg[xscale=2.1123cm,yscale=-2.1123cm] {M .1582 .4869 C .1582 .4869 .51195 0 .51195 0};
\path svg[xscale=2.1123cm,yscale=-2.1123cm] {M -.41418 -.30092 C -.41418 -.30092 -.41418 .30092 -.41418 .30092};
\path svg[xscale=2.1123cm,yscale=-2.1123cm] {M -.41418 -.30092 C -.41418 -.30092 .1582 -.4869 .1582 -.4869};
\path svg[xscale=2.1123cm,yscale=-2.1123cm] {M 1 0 C .87496 0 .51195 0 .51195 0};
\path svg[xscale=2.1123cm,yscale=-2.1123cm] {M -.80902 -.58779 C -.70785 -.51429 -.41418 -.30092 -.41418 -.30092};
\end{scope}
\path[draw=blue,line width=.01405cm] (0,0) circle[radius=1];
\end{scope}
\end{tikzpicture}
}
\expandafter\def\csname e8diagram:round5_2\endcsname{%
\begin{tikzpicture}[execute at begin picture={\clip (-1.24,-1.24) rectangle ++(2.48,2.48);},x={(2.1123cm,0cm)},y={(0cm,-2.1123cm)},baseline={([yshift=-\the\dimexpr\fontdimen22\textfont2\relax]current  bounding  box.center)},line join=round]
\begin{scope}
\path[clip] (0,0) circle[radius=1];
\path[fill={rgb,255:red,204;green,204;blue,255},draw=none] (0,0) circle[radius=1];
\begin{scope}[every path/.append style={draw=red,line width=.05619cm}]
\path svg[xscale=2.1123cm,yscale=-2.1123cm] {M .30902 -.95106 C .19656 -.60495 -.42917 -.31181 -.42917 -.31181};
\path svg[xscale=2.1123cm,yscale=-2.1123cm] {M -.80902 -.58779 C -.71451 -.51912 -.42917 -.31181 -.42917 -.31181};
\path svg[xscale=2.1123cm,yscale=-2.1123cm] {M .30902 .95106 C .16894 .51994 .54669 0 1 0};
\path svg[xscale=2.1123cm,yscale=-2.1123cm] {M -.80902 .58779 C -.5146 .37388 -.42917 -.31181 -.42917 -.31181};
\end{scope}
\path[draw=blue,line width=.01405cm] (0,0) circle[radius=1];
\end{scope}
\end{tikzpicture}
}
\expandafter\def\csname e8diagram:round5_3\endcsname{%
\begin{tikzpicture}[execute at begin picture={\clip (-1.24,-1.24) rectangle ++(2.48,2.48);},x={(2.1123cm,0cm)},y={(0cm,-2.1123cm)},baseline={([yshift=-\the\dimexpr\fontdimen22\textfont2\relax]current  bounding  box.center)},line join=round]
\begin{scope}
\path[clip] (0,0) circle[radius=1];
\path[fill={rgb,255:red,204;green,204;blue,255},draw=none] (0,0) circle[radius=1];
\begin{scope}[every path/.append style={draw=red,line width=.05619cm}]
\path svg[xscale=2.1123cm,yscale=-2.1123cm] {M -.80902 -.58779 C -.5146 -.37388 -.42917 .31181 -.42917 .31181};
\path svg[xscale=2.1123cm,yscale=-2.1123cm] {M -.80902 .58779 C -.71451 .51912 -.42917 .31181 -.42917 .31181};
\path svg[xscale=2.1123cm,yscale=-2.1123cm] {M 1 0 C .54669 0 .16894 -.51994 .30902 -.95106};
\path svg[xscale=2.1123cm,yscale=-2.1123cm] {M .30902 .95106 C .19656 .60495 -.42917 .31181 -.42917 .31181};
\end{scope}
\path[draw=blue,line width=.01405cm] (0,0) circle[radius=1];
\end{scope}
\end{tikzpicture}
}
\expandafter\def\csname e8diagram:round5_4\endcsname{%
\begin{tikzpicture}[execute at begin picture={\clip (-1.24,-1.24) rectangle ++(2.48,2.48);},x={(2.1123cm,0cm)},y={(0cm,-2.1123cm)},baseline={([yshift=-\the\dimexpr\fontdimen22\textfont2\relax]current  bounding  box.center)},line join=round]
\begin{scope}
\path[clip] (0,0) circle[radius=1];
\path[fill={rgb,255:red,204;green,204;blue,255},draw=none] (0,0) circle[radius=1];
\begin{scope}[every path/.append style={draw=red,line width=.05619cm}]
\path svg[xscale=2.1123cm,yscale=-2.1123cm] {M 1 0 C .63608 0 .16393 .50451 .16393 .50451};
\path svg[xscale=2.1123cm,yscale=-2.1123cm] {M -.80902 -.58779 C -.44228 -.32134 .16894 -.51994 .30902 -.95106};
\path svg[xscale=2.1123cm,yscale=-2.1123cm] {M -.80902 .58779 C -.5146 .37388 .16393 .50451 .16393 .50451};
\path svg[xscale=2.1123cm,yscale=-2.1123cm] {M .30902 .95106 C .27292 .83996 .16393 .50451 .16393 .50451};
\end{scope}
\path[draw=blue,line width=.01405cm] (0,0) circle[radius=1];
\end{scope}
\end{tikzpicture}
}
\expandafter\def\csname e8diagram:round5_5\endcsname{%
\begin{tikzpicture}[execute at begin picture={\clip (-1.24,-1.24) rectangle ++(2.48,2.48);},x={(2.1123cm,0cm)},y={(0cm,-2.1123cm)},baseline={([yshift=-\the\dimexpr\fontdimen22\textfont2\relax]current  bounding  box.center)},line join=round]
\begin{scope}
\path[clip] (0,0) circle[radius=1];
\path[fill={rgb,255:red,204;green,204;blue,255},draw=none] (0,0) circle[radius=1];
\begin{scope}[every path/.append style={draw=red,line width=.05619cm}]
\path svg[xscale=2.1123cm,yscale=-2.1123cm] {M .30902 .95106 C .25995 .80003 -.09558 .56256 -.09558 .56256};
\path svg[xscale=2.1123cm,yscale=-2.1123cm] {M .30902 -.95106 C .25995 -.80003 -.09558 -.56256 -.09558 -.56256};
\path svg[xscale=2.1123cm,yscale=-2.1123cm] {M .32161 0 C .32161 0 -.09558 .56256 -.09558 .56256};
\path svg[xscale=2.1123cm,yscale=-2.1123cm] {M -.80902 .58779 C -.61957 .45014 -.09558 .56256 -.09558 .56256};
\path svg[xscale=2.1123cm,yscale=-2.1123cm] {M 1 0 C .78362 0 .32161 0 .32161 0};
\path svg[xscale=2.1123cm,yscale=-2.1123cm] {M -.09558 -.56256 C -.09558 -.56256 .32161 0 .32161 0};
\path svg[xscale=2.1123cm,yscale=-2.1123cm] {M -.80902 -.58779 C -.61957 -.45014 -.09558 -.56256 -.09558 -.56256};
\end{scope}
\path[draw=blue,line width=.01405cm] (0,0) circle[radius=1];
\end{scope}
\end{tikzpicture}
}
\expandafter\def\csname e8diagram:round5_6\endcsname{%
\begin{tikzpicture}[execute at begin picture={\clip (-1.24,-1.24) rectangle ++(2.48,2.48);},x={(2.1123cm,0cm)},y={(0cm,-2.1123cm)},baseline={([yshift=-\the\dimexpr\fontdimen22\textfont2\relax]current  bounding  box.center)},line join=round]
\begin{scope}
\path[clip] (0,0) circle[radius=1];
\path[fill={rgb,255:red,204;green,204;blue,255},draw=none] (0,0) circle[radius=1];
\begin{scope}[every path/.append style={draw=red,line width=.05619cm}]
\path svg[xscale=2.1123cm,yscale=-2.1123cm] {M .30902 .95106 C .23665 .72835 .50549 .26474 .50549 .26474};
\path svg[xscale=2.1123cm,yscale=-2.1123cm] {M .30902 -.95106 C .24215 -.74527 .09938 -.30587 .09938 -.30587};
\path svg[xscale=2.1123cm,yscale=-2.1123cm] {M .09938 -.30587 C .09938 -.30587 .50549 .26474 .50549 .26474};
\path svg[xscale=2.1123cm,yscale=-2.1123cm] {M -.80902 .58779 C -.61957 .45014 -.56456 -.08294 -.56456 -.08294};
\path svg[xscale=2.1123cm,yscale=-2.1123cm] {M 1 0 C .8412 0 .50549 .26474 .50549 .26474};
\path svg[xscale=2.1123cm,yscale=-2.1123cm] {M -.56456 -.08294 C -.56456 -.08294 .09938 -.30587 .09938 -.30587};
\path svg[xscale=2.1123cm,yscale=-2.1123cm] {M -.80902 -.58779 C -.68055 -.49445 -.56456 -.08294 -.56456 -.08294};
\end{scope}
\path[draw=blue,line width=.01405cm] (0,0) circle[radius=1];
\end{scope}
\end{tikzpicture}
}
\expandafter\def\csname e8diagram:round5_7\endcsname{%
\begin{tikzpicture}[execute at begin picture={\clip (-1.24,-1.24) rectangle ++(2.48,2.48);},x={(2.1123cm,0cm)},y={(0cm,-2.1123cm)},baseline={([yshift=-\the\dimexpr\fontdimen22\textfont2\relax]current  bounding  box.center)},line join=round]
\begin{scope}
\path[clip] (0,0) circle[radius=1];
\path[fill={rgb,255:red,204;green,204;blue,255},draw=none] (0,0) circle[radius=1];
\begin{scope}[every path/.append style={draw=red,line width=.05619cm}]
\path svg[xscale=2.1123cm,yscale=-2.1123cm] {M .30902 .95106 C .23665 .72835 -.25334 .5113 -.25334 .5113};
\path svg[xscale=2.1123cm,yscale=-2.1123cm] {M .30902 -.95106 C .25995 -.80003 .40799 -.39894 .40799 -.39894};
\path svg[xscale=2.1123cm,yscale=-2.1123cm] {M -.80902 .58779 C -.68055 .49445 -.25334 .5113 -.25334 .5113};
\path svg[xscale=2.1123cm,yscale=-2.1123cm] {M 1 0 C .76583 0 .40799 -.39894 .40799 -.39894};
\path svg[xscale=2.1123cm,yscale=-2.1123cm] {M -.26019 -.18904 C -.26019 -.18904 -.25334 .5113 -.25334 .5113};
\path svg[xscale=2.1123cm,yscale=-2.1123cm] {M -.26019 -.18904 C -.26019 -.18904 .40799 -.39894 .40799 -.39894};
\path svg[xscale=2.1123cm,yscale=-2.1123cm] {M -.80902 -.58779 C -.63396 -.4606 -.26019 -.18904 -.26019 -.18904};
\end{scope}
\path[draw=blue,line width=.01405cm] (0,0) circle[radius=1];
\end{scope}
\end{tikzpicture}
}
\expandafter\def\csname e8diagram:round5_8\endcsname{%
\begin{tikzpicture}[execute at begin picture={\clip (-1.24,-1.24) rectangle ++(2.48,2.48);},x={(2.1123cm,0cm)},y={(0cm,-2.1123cm)},baseline={([yshift=-\the\dimexpr\fontdimen22\textfont2\relax]current  bounding  box.center)},line join=round]
\begin{scope}
\path[clip] (0,0) circle[radius=1];
\path[fill={rgb,255:red,204;green,204;blue,255},draw=none] (0,0) circle[radius=1];
\begin{scope}[every path/.append style={draw=red,line width=.05619cm}]
\path svg[xscale=2.1123cm,yscale=-2.1123cm] {M .30902 .95106 C .25995 .80003 .40799 .39894 .40799 .39894};
\path svg[xscale=2.1123cm,yscale=-2.1123cm] {M .30902 -.95106 C .23665 -.72835 -.25334 -.5113 -.25334 -.5113};
\path svg[xscale=2.1123cm,yscale=-2.1123cm] {M -.26019 .18904 C -.26019 .18904 .40799 .39894 .40799 .39894};
\path svg[xscale=2.1123cm,yscale=-2.1123cm] {M -.80902 .58779 C -.63396 .4606 -.26019 .18904 -.26019 .18904};
\path svg[xscale=2.1123cm,yscale=-2.1123cm] {M 1 0 C .76583 0 .40799 .39894 .40799 .39894};
\path svg[xscale=2.1123cm,yscale=-2.1123cm] {M -.25334 -.5113 C -.25334 -.5113 -.26019 .18904 -.26019 .18904};
\path svg[xscale=2.1123cm,yscale=-2.1123cm] {M -.80902 -.58779 C -.68055 -.49445 -.25334 -.5113 -.25334 -.5113};
\end{scope}
\path[draw=blue,line width=.01405cm] (0,0) circle[radius=1];
\end{scope}
\end{tikzpicture}
}
\expandafter\def\csname e8diagram:round5_9\endcsname{%
\begin{tikzpicture}[execute at begin picture={\clip (-1.24,-1.24) rectangle ++(2.48,2.48);},x={(2.1123cm,0cm)},y={(0cm,-2.1123cm)},baseline={([yshift=-\the\dimexpr\fontdimen22\textfont2\relax]current  bounding  box.center)},line join=round]
\begin{scope}
\path[clip] (0,0) circle[radius=1];
\path[fill={rgb,255:red,204;green,204;blue,255},draw=none] (0,0) circle[radius=1];
\begin{scope}[every path/.append style={draw=red,line width=.05619cm}]
\path svg[xscale=2.1123cm,yscale=-2.1123cm] {M .30902 .95106 C .24215 .74527 .09938 .30587 .09938 .30587};
\path svg[xscale=2.1123cm,yscale=-2.1123cm] {M .30902 -.95106 C .23665 -.72835 .50549 -.26474 .50549 -.26474};
\path svg[xscale=2.1123cm,yscale=-2.1123cm] {M .09938 .30587 C .09938 .30587 .50549 -.26474 .50549 -.26474};
\path svg[xscale=2.1123cm,yscale=-2.1123cm] {M -.80902 .58779 C -.68055 .49445 -.56456 .08294 -.56456 .08294};
\path svg[xscale=2.1123cm,yscale=-2.1123cm] {M 1 0 C .8412 0 .50549 -.26474 .50549 -.26474};
\path svg[xscale=2.1123cm,yscale=-2.1123cm] {M -.56456 .08294 C -.56456 .08294 .09938 .30587 .09938 .30587};
\path svg[xscale=2.1123cm,yscale=-2.1123cm] {M -.80902 -.58779 C -.61957 -.45014 -.56456 .08294 -.56456 .08294};
\end{scope}
\path[draw=blue,line width=.01405cm] (0,0) circle[radius=1];
\end{scope}
\end{tikzpicture}
}
\expandafter\def\csname e8diagram:squarepenta-closed\endcsname{%
\begin{tikzpicture}[baseline=(current  bounding  box.center),line join=round,x={(2.11230cm,0cm)},y={(0cm,-2.11230cm)}]
\begin{scope}
\path[fill={rgb,255:red,204;green,204;blue,255},draw=none] (0,0) circle[radius=1];
\begin{scope}[every path/.append style={draw=red,ultra thick}]
\path (.30390,-.49896) -- (.63294,0);
\path (.30390,-.49896) -- (.30902,-.95106) .. controls (-72:1.5) and (252:2) .. (216:1.5);
\path (-.1228,.23853) -- (.30390,.49896);
\path (.63294,0) -- (1,0) .. controls (2,0) and (72:2) .. (.30902,.95106) -- (.30390,.49896);
\path (-.1228,-.23853) -- (-.1228,.23853);
\path (-.1228,-.23853) -- (.30390,-.49896);
\path (-.61957,.27727) -- (-.1228,.23853);
\path (216:1.5) .. controls (180:2) and (144:1.5) .. (-.80902,.58779) -- (-.61957,.27727);
\path (-.61957,-.27727) -- (-.61957,.27727);
\path (-.61957,-.27727) -- (-.1228,-.23853);
\path (.30390,.49896) -- (.63294,0);
\path (216:1.5) -- (-.80902,-.58779) -- (-.61957,-.27727);
\end{scope}
\path[draw=blue,line width=.01405cm] (0,0) circle[radius=1];
\end{scope}
\end{tikzpicture}
}
\expandafter\def\csname e8diagram:squarepenta-closed2\endcsname{%
\begin{tikzpicture}[baseline=(current  bounding  box.center),line join=round,x={(2.11230cm,0cm)},y={(0cm,-2.11230cm)}]
\begin{scope}
\path[fill={rgb,255:red,204;green,204;blue,255},draw=none] (0,0) circle[radius=1];
\begin{scope}[every path/.append style={draw=red,ultra thick}]
\path (.30390,-.49896) -- (.63294,0);
\path (-.1228,.23853) -- (.30390,.49896);
\path (-.1228,-.23853) -- (-.1228,.23853);
\path (-.1228,-.23853) -- (.30390,-.49896);
\path (-.61957,.27727) -- (-.1228,.23853);
\path (-.61957,-.27727) -- (-.61957,.27727);
\path (-.61957,-.27727) -- (-.1228,-.23853);
\path (.30390,.49896) -- (.63294,0);
\path (.63294,0) -- (1,0);
\path (-.80902,.58779) -- (-.61957,.27727);
\path (-.80902,-.58779) -- (-.61957,-.27727);
\path (.30390,-.49896) -- (.30902,-.95106);
\path (.30902,.95106) -- (.30390,.49896);
\begin{scope}[rotate=-72]
\path (1,0) .. controls (2,0) and (72:2) .. (.30902,.95106);
\path (216:1.5) .. controls (180:2) and (144:1.5) .. (-.80902,.58779);
\path (216:1.5) -- (-.80902,-.58779);
\path (.30902,-.95106) .. controls (-72:1.5) and (252:2) .. (216:1.5);
\end{scope}
\end{scope}
\path[draw=blue,line width=.01405cm] (0,0) circle[radius=1];
\end{scope}
\end{tikzpicture}
}
\expandafter\def\csname e8diagram:squarepenta\endcsname{%
\begin{tikzpicture}[execute at begin picture={\clip (-1.24,-1.24) rectangle ++(2.48,2.48);},x={(2.1123cm,0cm)},y={(0cm,-2.1123cm)},baseline={([yshift=-\the\dimexpr\fontdimen22\textfont2\relax]current  bounding  box.center)},line join=round]
\begin{scope}
\path[clip] (0,0) circle[radius=1];
\path[fill={rgb,255:red,204;green,204;blue,255},draw=none] (0,0) circle[radius=1];
\begin{scope}[every path/.append style={draw=red,line width=.05619cm}]
\path (.3039,-.49896) -- (.63294,0);
\path (.30902,-.95106) -- (.3039,-.49896);
\path (-.1228,.23853) -- (.3039,.49896);
\path (1,0) -- (.63294,0);
\path (-.1228,-.23853) -- (-.1228,.23853);
\path (-.1228,-.23853) -- (.3039,-.49896);
\path (.30902,.95106) -- (.3039,.49896);
\path (-.61957,.27727) -- (-.1228,.23853);
\path (-.80902,.58779) -- (-.61957,.27727);
\path (-.61957,-.27727) -- (-.61957,.27727);
\path (-.61957,-.27727) -- (-.1228,-.23853);
\path (.3039,.49896) -- (.63294,0);
\path (-.80902,-.58779) -- (-.61957,-.27727);
\end{scope}
\path[draw=blue,line width=.01405cm] (0,0) circle[radius=1];
\end{scope}
\end{tikzpicture}
}
\expandafter\def\csname e8diagram:squaresquare\endcsname{%
\begin{tikzpicture}[execute at begin picture={\clip (-1.24,-1.24) rectangle ++(2.48,2.48);},x={(2.1123cm,0cm)},y={(0cm,-2.1123cm)},baseline={([yshift=-\the\dimexpr\fontdimen22\textfont2\relax]current  bounding  box.center)},line join=round]
\begin{scope}
\path[clip] (0,0) circle[radius=1];
\path[fill={rgb,255:red,204;green,204;blue,255},draw=none] (0,0) circle[radius=1];
\begin{scope}[every path/.append style={draw=red,line width=.05619cm}]
\path (0,.22611) -- (.48828,.30114);
\path (-.48828,-.30114) -- (-.48828,.30114);
\path (-.48828,-.30114) -- (0,-.22611);
\path (.70711,.70711) -- (.48828,.30114);
\path (0,-.22611) -- (0,.22611);
\path (0,-.22611) -- (.48828,-.30114);
\path (-.70711,.70711) -- (-.48828,.30114);
\path (.48828,-.30114) -- (.48828,.30114);
\path (.70711,-.70711) -- (.48828,-.30114);
\path (-.48828,.30114) -- (0,.22611);
\path (-.70711,-.70711) -- (-.48828,-.30114);
\end{scope}
\path[draw=blue,line width=.01405cm] (0,0) circle[radius=1];
\end{scope}
\end{tikzpicture}
}
\expandafter\def\csname e8diagram:triangle\endcsname{%
\begin{tikzpicture}[execute at begin picture={\clip (-1.24,-1.24) rectangle ++(2.48,2.48);},x={(2.1123cm,0cm)},y={(0cm,-2.1123cm)},baseline={([yshift=-\the\dimexpr\fontdimen22\textfont2\relax]current  bounding  box.center)},line join=round]
\begin{scope}
\path[clip] (0,0) circle[radius=1];
\path[fill={rgb,255:red,204;green,204;blue,255},draw=none] (0,0) circle[radius=1];
\begin{scope}[every path/.append style={draw=red,line width=.05619cm}]
\path (-.32619,0) -- (.1631,-.28249);
\path (.1631,.28249) -- (.1631,-.28249);
\path (-1,0) -- (-.32619,0);
\path (.5,.86603) -- (.1631,.28249);
\path (-.32619,0) -- (.1631,.28249);
\path (.5,-.86603) -- (.1631,-.28249);
\end{scope}
\path[draw=blue,line width=.01405cm] (0,0) circle[radius=1];
\end{scope}
\end{tikzpicture}
}

\begin{document}

\title{A uniform trigonometric \texorpdfstring{$R$}{R}-matrix for the exceptional series}

\author{Bruce W.\ Westbury}
\address{Bruce W.\ Westbury}
\email{brucewestbury@gmail.com}
\author{Paul Zinn-Justin}
\address{Paul Zinn-Justin, School of Mathematics and Statistics, The University of Melbourne, 
Victoria 3010, Australia}
\email{pzinn@unimelb.edu.au}

\begin{abstract}
	The exceptional series is a finite list of points on a projective line with a simple Lie algebra
	attached to each point. This list of Lie algebras includes the five exceptional Lie algebras.
	We give a uniform trigonometric $R$-matrix for the exceptional series in the representation $L\oplus I$,
        where $L$ is the quantum deformation of the adjoint representation and $I$ is the trivial representation.
        We construct a sixteen dimensional algebra, $\dilute A(\obj2)$, which interpolates the algebras
	$\End(\otimes^2(L\oplus I))$ and a 287 dimensional algebra, $\dilute A(\obj3)$,
        which interpolates the algebras
	$\End(\otimes^3(L\oplus I))$. The $R$-matrix lives in
	$\dilute A(\obj2)$ and satisfies the Yang--Baxter equation in $\dilute A(\obj3)$;
        it interpolates the trigonometric $R$-matrices for the points in the exceptional series.
\end{abstract}



\maketitle
\thispagestyle{empty}

\tableofcontents


\section{Introduction} This paper connects two topics in the representation theory of quantum groups.
One topic is the Yang--Baxter equation which is a parametrised version of the braid relation.
The other topic is the web categories introduced in \cite{kuperberg1996}.
We want to apply these ideas to the exceptional series.
We expand on these various points below.

\subsection{The Yang--Baxter equation}
The Yang--Baxter equation originated in the study of integrable models in statistical mechanics and quantum field theory;
an expository article is \cite{jimbo1989}.
The data for an $R$-matrix is a vector space, $V$, and an operator $\check R(\uu_1,\uu_2)\colon V\otimes V\to V\otimes V$
where $\uu_1$, $\uu_2$ are parameters.
We say that this data is an $R$-matrix if $\check R(\uu_1,\uu_2)$ satisfies the Yang--Baxter equation
\begin{equation}\label{eq:YB}
\check R_1(\uu_1,\uu_2)\check R_2(\uu_1,\uu_3)\check R_1(\uu_2,\uu_3) = \check R_2(\uu_2,\uu_3)\check R_1(\uu_1,\uu_3)\check R_2(\uu_1,\uu_2)
\end{equation}
for all $\uu_1,\uu_2,\uu_3$,
where $\check R_1(\uu_1,\uu_2) = \check R(\uu_1,\uu_2) \otimes \id_V$ and $\check R_2(\uu_1,\uu_2) = \id_V \otimes \check R(\uu_1,\uu_2)$.
If $\check R(\uu_1,\uu_2)$  is constant this is the braid relation. The $R$-matrices we consider have the property that $\check R(\uu_1,\uu_2)$ only depends on $\uu_1/\uu_2$;
these are known as trigonometric $R$-matrices.
$R$-matrices with the property that $\check R(\uu_1,\uu_2)$ only depends on $\uu_1-\uu_2$
are known as rational $R$-matrices; typically, trigonometric $R$-matrices depend on a quantum parameter $q$
in such a way that taking the limit $q\to 1$ appropriately results in a rational $R$-matrix.

Let $C^{(1)}$ be an affine Cartan matrix (which we choose to be untwisted for simplicity),
and $U_q(C^{(1)})$ the quantised enveloping algebra. Given a representation $V$ of $U_q(C^{(1)})$,
there is a natural construction of a trigonometric $R$-matrix which is a $U_q(C^{(1)})$-intertwiner
$\check R(\uu_1,\uu_2)\colon V_{\uu_1}\otimes V_{\uu_2}\to V_{\uu_2}\otimes V_{\uu_1}$ where $V_{\uu_1}$ is $V$ composed with a certain family of automorphisms, see \cite{delius2006}.
Let $C$ be the finite type Cartan matrix given by removing the 0 node from $C^{(1)}$;
as a $U_q(C)$-module, $V_{\uu_1} \cong V_{\uu_2}\cong V$, and
$\check R(\uu_1,\uu_2)\in \End_{U_q(C)}(\otimes^2V)$, so we can view the Yang--Baxter as an equation in $\End_{U_q(C)}(\otimes^3V)$,
which is much smaller than the matrix algebra and has a nice diagrammatic description, as we discuss in \S\ref{ssec:introweb}.

Let $L$ be the quantised adjoint representation and $I$ be the trivial representation of $U_q(C)$.
In general the action on $L$ does not extend to an action of $U_q(C^{(1)})$.
However the action on $L\oplus I$ does extend, so this construction gives an $R$-matrix on $V=L\oplus I$,
$\check R(\uu_1,\uu_2)\in \End_{U_q(C)}(\otimes^2(L\oplus I))$. Our goal is to construct these $R$-matrices explicitly
for the exceptional simple Lie algebras. Our main result is much stronger in that instead of treating each case separately
we give a single formula. The idea is that this formula
involves an additional parameter and the individual cases corresponding to each specific Cartan type $C$ are given by specialising this parameter.
The complication is that each $R$-matrix is an element of a different algebra, so we first need to interpolate the algebras,
which we discuss in \S\ref{ssec:introinterp}.

\subsection{Web categories}\label{ssec:introweb}
Given a Cartan type $C$, its quantised enveloping algebra $U_q(C)$ and a $U_q(C)$-module $V$, let $\Icat(V)$ be the category of invariant tensors. A web category is then a construction of $\Icat(V)$ as a
monoidal category defined by generators and relations. In this paper we are interested in the cases where $C$ is an exceptional finite type Cartan type and $V$ is the quantised adjoint representation $L$, or as above $L\oplus I$.
Previous work on web categories, \cite{kuperberg1996}, \cite{Cautis2014}, has concentrated on fundamental representations and so has not covered adjoint representations (with the exception of $G_2$).

We follow the conventional approach to web categories using string diagrams, see \cite{savage2020}, \cite{east2023}.
We draw tensors as string diagrams in a rectangle, with composition (resp.\ tensor product) being horizontal (resp.\ vertical) concatenation.
The use of string diagrams gained acceptance with the application of the representation theory of
quantised enveloping algebras to knot theory. It provides an abstraction of tensor calculus in the
sense that we work with the tensor operations of tensor product, contractions, raising and lowering operators
(and forming linear combinations).

The starting points for the construction of a web category for an adjoint representation are the basic invariant tensors given by the Lie bracket as a $(2,1)$-tensor, the Killing form as a $(2,0)$-tensor and the quadratic Casimir as a $(0,2)$-tensor. These are generators in the sense that any invariant tensor can be built
from these using tensor operations, \cite{etingof2019}. There are also relations, for example the Jacobi relation,
which arise since the same tensor may be built in many ways. One of the central problems is to determine a full set of relations for $\Icat(V)$.

\subsection{Interpolation}\label{ssec:introinterp}
The problem of constructing web categories can be extended to series.
We shall not formalise this idea yet, and instead give some examples. The first ones are: the Birman--Wenzl--Murakami category (a.k.a.\ the $q$-Brauer category),
the oriented Brauer category and its quantisation and the partition category. These examples are discussed in
\cite{Deligne2007}. Two further examples are series of Lie superalgebras; the quantum queer (a.k.a.\ isomeric) supercategory, \cite{savage2024} and the periplectic $q$-Brauer category, \cite{rui2025}. Another example is \cite{mcnamara2024}.

If we are given a $U_q(C^{(1)})$-module, $V_{\uu}$, such that there is a web category for the $U_q(C)$-module, $V$,
then the $R$-matrix can be written in the graphical calculus. If we are given a
web category interpolating a series of representations such that each representation in the series admits an $R$-matrix
then we expect that there is an $R$-matrix in the web category which specialises to these $R$-matrices.
This is known for the seminal examples consisting of the defining representations of classical groups.

\subsubsection{Birman--Wenzl--Murakami category} Let us briefly describe the case of the vector representations of the orthogonal and symplectic groups.
The Brauer category, \cite{Brauer1937} is constructed combinatorially as the cobordism category of unoriented 1-manifolds and algebraically as the
free rigid symmetric monoidal category on an object with a non-degenerate symmetric inner product. It follows from the algebraic description that there
is a functor from the Brauer category to $\Icat(V)$ for $V$ the vector representation of an orthogonal or symplectic group. The quantum analogue is the
Birman--Wenzl--Murakami category whose endomorphism algebras are the Birman--Wenzl--Murakami algebras.

There are two parameters, the deformation parameter $q$, and $z$, a coordinate function for the series. The Birman--Wenzl--Murakami algebras have a well-known finite presentation.
The $k$-strand algebra has generators $U_i$ and $S_i$, $i=1,\ldots,k-1$, and the relations which depend on the parameters are:
\begin{gather*}
	U_i^2 = \delta\,U_i \\
	U_iS_i^{\pm 1} = z^{\mp 2}U_i \\
	S_i-S_i^{-1} = (q^2-q^{-2})(1-U_i) %
\end{gather*}
The remaining relations are the tangle relations and the commuting relations. The quantum dimension, $\delta$, is
\begin{equation*}
\delta \coloneqq \left(\frac{z\,q-z^{-1}q^{-1}}{q^2-q^{-2}}\right)(z^{-1}q+z\,q^{-1}) = 1 + \left(\frac{z^2-z^{-2}}{q^2-q^{-2}}\right)
\end{equation*}

For the coefficient ring $\bQ(z,q)$ each Birman--Wenzl--Murakami algebra is a finite dimensional multi-matrix algebra.
However we can also take the coefficient ring to be the subring of $\bQ(z,q)$ generated by $q^{\pm 1}$, $z^{\pm 1}$, $(q+q^{-1})^{-1}$ and the elements
	$\frac{z^rq^s - z^{-r}q^{-s}}{q-q^{-1}}$
for $r,s\in\bZ$.
Doing so allows us to define for each $n\in\bZ$ a homomorphism to the localisation $\bQ[q,q^{-1},(q+q^{-1})^{-1}]$
given by $z\mapsto q^{n-1}$.
For $n>2$ there is a full functor to the category of invariant tensors
of the vector representation of $\fso(n)$, see \cite{Wenzl1990a}, and similarly for $n$
negative even and $\mathfrak{sp}(-n)$.

There is also a homomorphism to the polynomial ring $\bQ[\nv]$ given by $q\mapsto 1$, $z\mapsto 1$ and
$ \frac{z^rq^s - z^{-r}q^{-s}}{q-q^{-1}} \mapsto r(\nv-1) + s $
for $r,s\in\bZ$. For example, $\delta\mapsto \nv$. This change of base gives the Brauer category, using
$	\frac{S-S^{-1}}{q^2-q^{-2}} \mapsto 1-U $.

The $R$-matrix is defined by
\begin{align*}
\check R(u) &\coloneqq
\left(\frac{u-u^{-1}}{q-q^{-1}}\right)\left(\frac{u^{-1}z\,q^{-1}S-u\,z^{-1}qS^{-1}}{q-q^{-1}}\right) + (q+q^{-1})\left(\frac{u\,z^{-1}q-u^{-1}z\,q^{-1}}{q-q^{-1}}\right)
\end{align*}
The specialisation above $z\mapsto q^{n-1}$ gives the trigonometric $R$-matrices in types $B$, $C$ and $D$ given in \cite{Jimbo1986}.
By taking $q\to1$ appropriately, one obtains the rational $R$-matrix in \cite[\S3.2]{MacKay_2007}.

\subsubsection{Interpolating series}
The idea that $R$-matrices in a series can be interpolated originated in \cite{Westbury2003} which
considered the rows of the Freudenthal magic square, \cite{Freudenthal1964}. The representations in the
first three rows are given by the tensor product graph method, \cite{Delius1994}. This method only applies
to a representation $V$ such that $V\otimes V$ is multiplicity free and so does not apply to $L\oplus I$.
If $V\otimes V$ is multiplicity free then the $R$-matrix is determined by its eigenvalues and
the tensor product graph method determines these eigenvalues. The observation in \cite{Westbury2003}
is that for two of the rows there is a simple interpolation of these eigenvalues. The string diagrams
for these two interpolating $R$-matrices are given in \cite{MacKay_2007}.
The remaining case is the exceptional series. Although the tensor product graph method does not apply
to $L\oplus I$, the rational $R$-matrices were given on a case by case basis in \cite{chari1991a}
and an interpolating formula was given in  \cite{Westbury2003}.

\subsection{The exceptional series}
The proposal that the exceptional simple Lie algebras should form a series originated in a 1995 preprint by Pierre Vogel which was eventually published in \cite{Vogel2011}.
The exceptional series was then investigated in \cite{deligne1996b}, \cite{deligne1996a} which presented uniform data for the Lie algebras
in the exceptional series. This data includes uniform decompositions of the tensor powers of the adjoint representation up to the fourth tensor power (and more generally,
uniform decompositions of Schur functors applied to the adjoint representation). This data also includes linear functions giving the values of the quadratic Casimir
on these composition factors and rational functions giving the dimensions of the composition factors.
(We present this data in \S\ref{numerology} but only up to the third tensor power as this is all we will need.)
The points on the exceptional line are given in Table~\ref{fig:nvalues} below; throughout the paper we say that
$C$ is an exceptional Cartan type if it appears in the first row of Table~\ref{fig:nvalues}.

This data is character theory. Replacing the Lie algebras by their quantised enveloping algebras, the only modification
needed is to replace the dimensions by the quantum dimensions. It is not known if this data can be extended to higher tensor
powers.

We now give a more detailed overview of our approach to this problem.
We work with quantised enveloping algebras over
a certain coefficient ring $\kk_q\subset \bQ(q)$. $\kk_q$ admits a homomorphism $q\mapsto 1$
and this makes precise the slogan that the representation theory of the quantum group $U_q(C)$
associated to the Cartan type $C$ is the same as that of the universal enveloping algebra $U(C)$ provided
$q$ is not a root of unity. We also include $\Gamma$, the group of Dynkin diagram automorphisms of $C$. If $\Gamma$ is non trivial
then the tensor powers of the adjoint representation have too many composition factors and do not fit a
uniform pattern: this is remedied by considering the smash product $\fU_q(C)$ of $\Gamma$ and $U_q(C)$.

In order to deal with the whole of the exceptional series, we work with a coefficient ring, $\coeffs$ %
which is the subring of the field $\bQ(\pp,\qq)$ such that the following specialisations exist:
a homomorphism $\psi\colon \coeffs \to \bQ(\nv)$ (a precise version of the limit $q\to 1$)
and homomorphisms $\theta_C\colon\coeffs\to\kk_q$
corresponding to the point of the exceptional line associated to the Cartan type $C$.

An interpolating category, $\Intcat$, is a $\coeffs$-linear ribbon category
whose monoid of objects is $\bN$, with $\End_{\Intcat}(\obj0)\cong\coeffs$,
and that admits full ribbon functors
$\Psi_C\colon \Intcat\otimes_{\theta_C}\kk_q \to \Icat_{\fU_q(C)}(L_C)$ for the (deformed) adjoint representation $L_C$ of each exceptional type $C$.
(See also \cite{deligne1996b} and \cite{morrison2026} for similar definitions.)
Although the idea of an interpolating category informs this paper, it is unknown if it exists.
Instead we construct a ``truncated'' web category $\threecat$ whose objects are $\{\obj0,\obj1,\obj2,\obj3\}$.
Our first main result is to show that $\threecat$ satisfies the expected properties of the full subcategory of $\Wcat$ with those objects.
The endomorphism algebras are interpolating algebras:
\begin{itemize}
	\item There are surjective homomorphisms $\End_{\threecat}(i)\otimes_{\theta_C}\kk_q\to \End_{\fU_q(C)}(\otimes^i L_C)$.
	\item The $\bQ(\pp,\qq)$-algebras $\End_{\threecat}(i)\otimes_{\coeffs}\bQ(\pp,\qq)$ are multi-matrix algebras.
	\item The branching rules for the inclusion $\End_{\threecat}(\obj i)\otimes_{\coeffs}\bQ(\pp,\qq) \subset \End_{\threecat}(i+1)\otimes_{\coeffs}\bQ(\pp,\qq)$
	are the branching rules for the exceptional series.
\end{itemize}

Once we have constructed the interpolating algebras for the adjoint representation $L$,
it is straightforward to construct analogues for $L\oplus I$,  together with
surjective homomorphisms to $\End_{\fU_q(C)}(\otimes^i(L_C\oplus I_C))$, $0\le i\le 3$,
for each exceptional Cartan type $C$.  These are the interpolating algebras that are needed to define
the interpolating $R$-matrix and to show that it satisfies the Yang--Baxter equation.

\editorialnote{}

\subsection{Plan of the paper}
The organisation of the paper is as follows:

In \S\ref{numerology} we summarise the background on the exceptional series, we give the uniform representations for each point in the exceptional series and the
uniform Bratteli diagram. We then give the linear functions which interpolate the values of the Casimir and the rational functions which interpolate the dimensions.

In \S\ref{sec:ribbon} we construct the free ribbon category on a trivalent vertex invariant under rotation. The morphisms in this category
are trivalent graphs drawn in a rectangle. We also allow crossings and each crossing is either an over or under crossing.

In \S\ref{sec:coeffs} we first introduce the quantum groups and module categories associated to an exceptional Cartan type. Then we construct the quantum
analogue of the adjoint representation as an object of the module category. Then we introduce the ring of scalars with the properties needed for interpolating categories.
This includes an algebraic definition of the $q\to 1$ limit.

In \S\ref{sec:twostring} we find relations in the rank 5 module of diagrams with four boundary points by two different approaches.
In particular, this constructs a rank 5 commutative algebra which interpolates the algebras $\End(\otimes^2L_C)$ for $C$ in the exceptional series.

In \S\ref{sec:threestring} we find a new relation for diagrams with five boundary points, and then proceed to
build the algebra which interpolates the $\End(\otimes^3L_C)$; it is a multi-matrix algebra of rank 80 whose irreducible representations we give explicitly.
Putting together all the constructions above, we define the category with objects  $\obj0\leqslant i\leqslant \obj3$,
which interpolates the categories with objects $\otimes^i(L_C)$ for $C$ an exceptional Cartan type.

In \S\ref{sec:dilute} we construct the category with objects $\obj0\leqslant i\leqslant \obj3$
which interpolates the categories with objects $\otimes^i(L_C\oplus I_C)$,
$C$ an exceptional Cartan type. We refer to this as the dilute category of the category constructed in the previous section.
The algebra $\End(\obj3)$ is a multi-matrix algebra of rank 287 and we give the irreducible representations explicitly.

Finally, in \S\ref{sec:yangbaxter}, we construct the $R$-matrix. First, this has the properties of an $R$-matrix, namely it satisfies the Yang--Baxter equation, crossing symmetry
and unitarity. Second, it interpolates the $R$-matrix for the points in the exceptional series. Thirdly, in the limit $q\to1$, it gives the rational $R$-matrix
proposed in \cite[\S2.3]{Westbury2003}.

In Appendix~\ref{app:skein} we give details of the relations.

\subsection*{SageMath and Macaulay2 computations}
Some of the explicit calculations supporting the proofs were performed using the open-source mathematics software systems SageMath~\cite{sagemath} and Macaulay2~\cite{M2}.
All data relevant to the paper's computerised calculations are provided in the companion repository
\url{https://github.com/BruceWestbury/exceptional-r-matrices} (v1.0.0 recorded at \url{https://doi.org/10.5281/zenodo.23886980})
including the exact representation matrices, scalar coefficients and basis conventions.
The repository also contains the \texttt{ExceptionalDiagrams} Macaulay2 package, a SageMath example,
and exact verification programs with their recorded results.

An interactive introduction to the Macaulay2 package is available in the tutorial
\url{https://www.unimelb-macaulay2.cloud.edu.au/#tutorial-ExceptionalDiagrams}.

\editorialnote{}

\section{Background}\label{numerology}
The exceptional series is a projective line over $\bQ$ with 10 distinguished points, which are
shown in Table~\ref{fig:nvalues}; this list is also given in \cite[Table~1]{morrison2026}.
Eight of these points are labelled by a simple Lie algebra of the given Cartan type, and this list of simple Lie
algebras includes all five exceptional simple Lie algebras. There are numerous choices
of parameterisations on the line. Our parameter is $\nv$ and the label of a point is $n$.

\begin{table}
  \[
    \begin{array}{c|cccccccccc}
      C & \varnothing & \OSp & A_1 & A_2 & G_2 & D_4 & F_4 & E_6 & E_7 & E_8 \\ \hline
      n & 1/5 & 1/4 & 1/3 & 1/2 & 2/3 & 1 & 3/2 & 2 & 3 & 5 \\
    \end{array}
  \]
  \caption{The exceptional Cartan types}\label{fig:nvalues}
\end{table}
\editorialnote[gray]{}
\editorialnote[gray]{}

The parameter $n$ is related to other parameters by
\begin{align*}
	n &= m/2+1 & n &= -1/\lambda & n &= \check{h}/6 \\
	m &= 2n-2 & \lambda &= -1/n & \check{h} &= 6n
\end{align*}
The parameter $m$ is the dimension of the composition algebra in the Freudenthal--Tits construction of the exceptional
simple Lie algebras (except $G_2$). The parameter $\check{h}$ is the dual Coxeter number and $\lambda$ is the parameter in
\cite{cohen1996} and \cite{morrison2026}. The parameter $a$ in \cite{cohen1996} is $a=1/(6\nv)$.

There is a $\bC$-linear category associated to each term in the exceptional series. Each category is semisimple abelian
and also rigid symmetric monoidal. For each Cartan type, $C$, in the exceptional series, let $\fg_C$ be the associated simple Lie algebra
and write $G_C=\Aut(\fg_C)$ for its automorphism group. The category associated to $C$ is the category of
finite dimensional $\mathcal{O}(G_C)$-comodules where $\mathcal{O}(G_C)$ is the coordinate ring of the algebraic group $\Aut(\fg_C)$.
The connected component of $\Aut(\fg_C)$, $\Aut_0(\fg_C)$, has Lie algebra $\fg_C$ and trivial centre.

The two points which are not labelled by a simple Lie algebra are $n=1/5$ which is the zero Lie algebra, $n=1/4$ which is labelled by the simple super Lie algebra
$\osp(1|2)$. The category associated to $n=1/5$ is the category of finite dimensional complex vector spaces.
The category associated to $n=1/4$ is the category of finite dimensional graded representations of the Lie superalgebra $\osp(1|2)$.

\begin{rem} The exceptional series has an unexpected symmetry. This symmetry is the involution $n\mapsto 1/n$.
It was observed in \cite{DeligneGross} that the
exceptional series (excluding $\osp(1|2)$) is an increasing sequence of Lie subalgebras of $\mathfrak{e}_8$. Moreover,
each pair $(1/n,n)$ gives a dual reductive pair of Lie subalgebras of $\mathfrak{e}_8$. This suggests that there should be an additional
point on the exceptional series corresponding to $n=4$. This point also has the property that the evaluations of the dimension formulae
are positive integers. This point is discussed in \cite{landsberg2006} and \cite{Westbury2006} but since the representation theory
is not understood we do not include it in this paper.
\end{rem}

\subsection{Numerology}
In this section we review the numerology of the exceptional series presented in \cite{cohen1996}.
We only present the results up to the third tensor power as this is all we need.
We write $L$ for the adjoint representation, $I$ for the trivial representation, but otherwise we adhere to the notation in \cite{cohen1996}.

The decomposition of $\otimes^2L$ is
\begin{equation}\label{eq:tensorsquare}
\otimes^2L \cong I \oplus L \oplus Y_2 \oplus Y_2^* \oplus X_2
\end{equation}

For the Cartan types in Table~\ref{fig:nvalues} the highest weights of these representations are
\begin{equation}\label{eq:weights}
	\begin{array}{c|cccccccccc}
	C & \varnothing & \mathit{OSp}_{1|2} & A_1 & A_2 & G_2 & D_4 & F_4 & E_6 & E_7 & E_8 \\ \hline
	L & & & 2\omega_1 & \omega_1+\omega_2 & \omega_2 & \omega_2 & \omega_1 & \omega_2 & \omega_1 & \omega_8 \\
	Y_2 & & & 4\omega_1 & 2\omega_1+2\omega_2 & 2\omega_2 & 2\omega_2 & 2\omega_1 & 2\omega_2 & 2\omega_1 & 2\omega_8 \\
	Y_2^* & & & - & \omega_1+\omega_2 & 2\omega_1 & 2\omega_1^\ast & 2\omega_4 & \omega_1+\omega_6 & \omega_6 & \omega_1 \\
	X_2 & & & - & 3\omega_1^\ast & 3\omega_1 & \omega_1+\omega_3+\omega_4 & \omega_2 & \omega_4 & \omega_3 & \omega_7
\end{array}
\end{equation}

A star on a highest weight denotes its orbit under Dynkin diagram automorphisms.

The following table gives the decompositions of $V\otimes L$ for $V$ a composition factor of $\otimes^2L$.
These can also be displayed as edges in the Bratteli diagram.
\begin{equation}\label{eq:bratteli}
	\begin{array}{c|c|c|ccc|ccccccc}
		& I & L & X_2 & Y_2 & Y_2^\ast & A & Y_3 & Y_3^\ast & C & C^\ast & X_3 \\ \hline
		I & & 1 & & & & & & & & & & \\
		L & 1 & 1 & 1 & 1 & 1 & & & & & & & \\
		Y_2 & & 1 & 1 & 1 & & 1 & 1 & & 1 & & \\
		X_2 & & 1 & 1 & 1 & 1 & 1 & & & 1 & 1 & 1\\
		Y^\ast_2 & & 1 & 1 & & 1 & 1 & & 1 & & 1 &
	\end{array}
\end{equation}

The rank of the free $\coeffs$-module $\Hom_{\Intcat}(r,s)$ only depends on $r+s$ since we have raising and lowering operators.
Based on the Bratteli diagram in \cite{cohen1996}, we know that the ranks of these $\coeffs$-modules for $0\leqslant r+s \leqslant 9$ should be
\begin{equation}\label{eq:invdims}
	\begin{tabular}{cccccccccc}
		0 & 1 & 2 & 3 & 4 & 5 & 6 & 7 & 8 & 9 \\ \hline
		1 & 0 & 1 & 1 & 5 & 16 & 80&  436 &  2891 &  22248
	\end{tabular}
\end{equation}

In particular, $\Hom_{\threecat}(r,s)$ should reproduce these numbers up to $r+s = 6$.

The values of the Casimir are computed using \cite[Proposition~11.36]{carter2005}.
\begin{equation}\label{eq:casimir}
	\casimir(\lambda) = \langle \lambda,\lambda+2\rho\rangle
\end{equation}
where the Killing form is normalised so that $\langle \theta,\theta\rangle = 2$, where
$\theta$ is the highest root.

This gives the following values on the composition factors of $\otimes^2L$:
\begin{equation}\label{eq:cas_values}
	\begin{array}{c|cccccc}
		V & I & L & Y_2 & Y_2^* & X_2 \\ \hline
		\casimir(V) & 0 & 12\,\nv & 24\,\nv+4 & 20\,\nv-4 & 24\,\nv
	\end{array}
\end{equation}

Although we will not make use of these, we record the values on the remaining composition factors of $\otimes^3L$.
The values of the quadratic Casimir, $\casimir$ are:
\begin{equation}\label{eq:cvalues}
	\begin{tabular}{c|cccccc}
		$V$ & $A$ & $Y_3$ & $Y_3^\ast$ & $C$ & $C^\ast$ & $X_3$ \\ \hline
		$\casimir(V)$  & $32\nv$ & $36\nv+12$ & $24\nv-12$ & $36\nv+6$ & $30\nv-6$ & $36\nv$
	\end{tabular}
\end{equation}

For the Cartan types $\OSp$, $A_1$, $A_2$, $G_2$, $D_4$, $F_4$, $E_6$, $E_7$, $E_8$, these values interpolate the values of the quadratic Casimir in the sense that: for each such $C$ %
and each $V$ in \eqref{eq:cas_values} and \eqref{eq:cvalues}, the quadratic Casimir acts on $V(C)$ by the scalar given by substituting $\nv=n_C$ in $\casimir(V)$. This is proved by case by case calculations.

\subsection{Quantum dimensions}\label{sec:qdims}
The following are the quantum dimension formulae. There are (at least) three methods for computing these.
The first method, used in \cite{cohen1996}, is to use the branching rules and the rules (1) and (2) in \cite{deligne1996b}
to obtain linear relations for the quantum dimensions. The second method is to use the results of \cite[\S3]{Landsberg2002}
which are obtained from the quantum Weyl dimension formula. The approach in \S\ref{sec:qdimsid} is to compute traces of idempotents.
It is remarkable that all three methods give the same results.

The quantum dimensions of the composition factors of $L\otimes L$ were first given in \cite{Tuba2001} and are:
\begin{gather}\label{eq:qdims2}
	\dim_q(I) = 1 \\
	\dim_q(L) = \label{eq:qdimL}\frac{\qint 61 [5\,\nv-1][4\,\nv]}{[\nv+1][2\,\nv]} \\
	\dim_q(X_2) = \frac{\qint 61 \qint 62 [5\,\nv] [4\,\nv-2] \qint 41 [5\,\nv-1] [3\,\nv-1]}%
	{[\nv] [2] \qint 22 \qint 31 [2\,\nv-1] [\nv+1]} \\
	\dim_q(Y_2) = \frac{[6\,\nv] \qint 63 [5\,\nv] [4\,\nv] \qint 41 [5\,\nv-1]}%
	{[2] [\nv+1] \qint 21 [2\,\nv] [\nv+2]} \\
	\dim_q(Y_2^\ast) = \frac{[6\,\nv] \qint 61 [5\,\nv] [4\,\nv] [3\,\nv-1] [3\,\nv-3]}%
	{\qint 22 [2\,\nv] [\nv-1] [\nv+2] [\nv+1]}
\end{gather}

The quantum dimensions of the composition factors of $\otimes^3 L$ were computed by taking the trace of an idempotent in $A(\obj3)$, see \S\ref{sec:threestring}.
These quantum dimensions have also been obtained in \cite[Table~4]{morrison2026} and most of these are given in \cite[\S3]{Landsberg2002}.
The quantum dimensions of composition factors which don't appear in $L\otimes L$ are:
\begin{gather}\label{eq:qdims3}
	\dim_q(X_3) = \frac{\qint 61 \qint 62 \qint 63 \qint 51 [4\,\nv-1] [4\,\nv] [4\,\nv-2] [5\,\nv-1][5\,\nv][3\,\nv-3]}%
	{[3][2]\qint 21 \qint 22 [2\,\nv] [\nv-1]\qint 33 [\nv+1] [\nv]}\!\!\!\!\\
	\dim_q(Y_3) = \frac{[6\,\nv] \qint 61 \qint 65 \qint 51 \qint 41 [4\,\nv] \qint 42 [5\,\nv-1] [5\,\nv]}%
	{[2][\nv+1] [3] \qint 21 \qint 22 [2\nv] [\nv+2] [\nv+3]} \\
	\dim_q(Y^\ast_3) = -\frac{[6\,\nv] \qint 61 [5\,\nv] [5\,\nv-1] [4\,\nv-1] [4\,\nv] [3\,\nv-1] [2\,\nv-2] [\nv-5] }%
	{[2][\nv-1] [\nv+1] [\nv+2] \qint 23 \qint 22 [2\nv] \qint 33} \\
	\dim_q(A) = \frac{[6\,\nv] \qint 61 \qint 63 \qint 51 [4\,\nv-1] \qint 41 [5\,\nv-1] [3\,\nv-1] [3\,\nv-3]}%
	{ [\nv+1]^2 \qint 23 \qint 21 [2\,\nv] [\nv-1] [\nv+3]} \\
	\dim_q(C) = \frac{[6\,\nv] \qint 62 \qint 64 \qint 51 [4\,\nv] [4\,\nv-2] \qint 42 [5\,\nv-1] [5\,\nv] [3\,\nv-1]  }%
	{[3] [\nv] \qint 23 \qint 21 \qint 32 [\nv+2]  [2\,\nv-1] [\nv+1]  } \\
	\dim_q(C^\ast) = \frac{[6\,\nv] \qint 61 \qint 62 [5\,\nv] [4\,\nv] [4\,\nv-2] [4\,\nv-1] \qint 41 [2\,\nv-2] [2\,\nv-4] }%
	{[\nv] [\nv-1] [\nv-2] [\nv+1]^2 [\nv+2] [\nv+3] \qint 31 \qint 33}
\end{gather}



\section{Ribbon category}\label{sec:ribbon}
Tortile categories were introduced in \cite{joyal1991} and the definition is algebraic. Ribbon categories were introduced in \cite{Reshetikhin1990} and the definition is topological (or diagrammatic). It is folklore that these two definitions are equivalent.
It is also shown in \cite{Reshetikhin1990} that the category of finite dimensional type I representations of a quantised
enveloping algebra is a linear ribbon category.

In this section we construct a free ribbon category. The result that our construction is a free ribbon category is also
folklore but related results are Theorem~2.3, Theorem~3.7 and Theorem~4.5 in \cite{Joyal1991a} which give diagrammatic constructions of the free monoidal category, the free braided category and the free balanced category
and \cite{Shum1994} which shows that the category of framed tangles is the free ribbon category on one object.

Our convention is to write maps on the right. Diagrams are read from left to right.

\subsection{Pivotal category}\label{ssec:pivot}
A \Dfn{rectangle} will mean a rectangle embedded in the Euclidean plane with horizontal
and vertical edges.
\begin{defn} A \Dfn{planar diagram} is a closed subset $D$ of a rectangle such that
	every point in $D$ which is in the boundary of the rectangle has an open neighbourhood
	equivalent to one of the open neighbourhoods in \eqref{eq:edge}
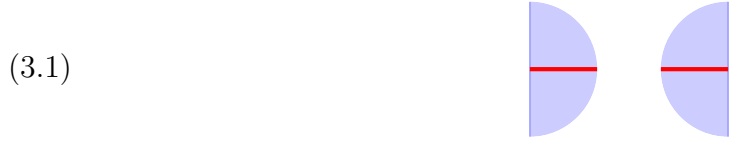
\begin{equation}\label{eq:edge}
	\topvertex \qquad \botvertex
\end{equation}
and every point of
$D$ in the interior of the rectangle has an open neighbourhood equivalent
to one of the open neighbourhoods in \eqref{eq:interior}.
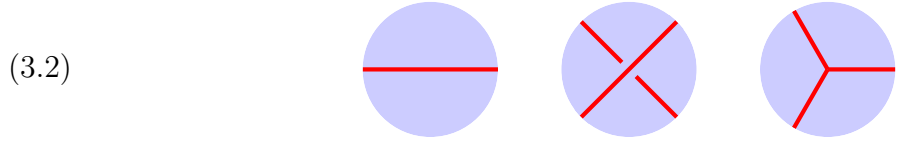
\begin{equation}\label{eq:interior}
	\linecirca \qquad \overcirca \qquad \trivalenta
\end{equation}

\end{defn}
	The equivalence relation is graph isomorphism which preserves
\begin{itemize}
	\item the cyclic ordering of the edges incident to each vertex
	\item the over/under marking at each crossing
	\item the ordered set of input vertices and the ordered set of output vertices
\end{itemize}

Planar diagrams are the morphisms of a category, $\Pcat$.
\begin{defn}\label{defn:pivotal}
	The set of objects of $\Pcat$ is $\bN$. A morphism $r\to s$ of the category $\Pcat$ is an equivalence class of planar diagrams
	with $r$ boundary points on the left hand edge and $s$ boundary points on the right hand edge.
Composition of morphisms is given by juxtaposing horizontally rectangles. The identity morphisms are rectangles with horizontal lines.
\end{defn}

The category $\Pcat$ is strict monoidal with tensor product given by stacking of rectangles.
The category $\Pcat$ is strict pivotal since every object is self-dual and
the functor $\ast\colon \Pcat\to \Pcat^{\mathrm{op}}$ rotates a diagram through a half revolution.
The category $\Pcat$ is also strict spherical and spatial
and is the free spatial spherical category on a rotationally invariant trivalent vertex, \cite[Conjecture~4.16]{selinger2010}
and \cite[Theorem~7.8]{delpeuch2022}.

The \Dfn{generators} of $\Pcat$ are the diagrams in the  first row \eqref{eq:catgens} and their names are given in the second row.
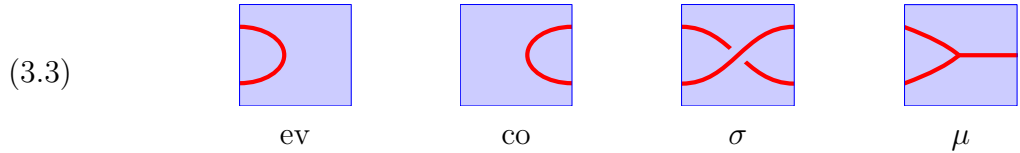
\begin{equation}\label{eq:catgens}
\arraycolsep=16pt
  \begin{array}{cccc}
    \cuprect  &
    \caprect  &
    \overrect  &
    \trivalentrect \\
    \ev & \co & \sigma & \mu
  \end{array}
\end{equation}

Let $\sigma$ be the braid generator in \eqref{eq:catgens}. The braid generator, $\sigma$, is invertible and the inverse, $\sigma^{-1}$,
is given by rotating $\sigma$ through a quarter turn (in either direction).

Recall that a braided monoidal category is a monoidal category with an isomorphism $\sigma_{U\,V} \colon U\otimes V\to V\otimes U$
for all objects $U,V$. These morphisms are required to make the following diagram commute
\begin{equation}\label{eq:braid}
	\begin{tikzcd}[black]
		U\otimes V \arrow[r , "f\otimes 1_V"] \arrow[d , "\sigma_{U\, V}"'] & W\otimes V \arrow[d , "\sigma_{W\,V}"] \\
		V\otimes U \arrow[r , "1_V\otimes f"'] & V \otimes W
	\end{tikzcd}
\end{equation}
for all objects $U,V,W$ and all morphisms $f\colon U\to W$.
It is sufficient to require that this holds for each of the generating morphisms in \eqref{eq:catgens} and for $V$ the generating object.

The conditions that the relations \eqref{eq:braid} hold for the generators, \eqref{eq:catgens} are shown in \eqref{eq:braidrelations1}:
\begin{equation}\label{eq:braidrelations1}
\begin{split}
&	\begin{tikzpicture}[bwwrect={(-1,-1)}{(1,1)},yscale=1.2]
		\draw (0.5,1) to[out=225,in=90] (-0.5,0) to[out=270,in=135] (0.5,-1);
		\draw[wipe] (-0.5,1) to[out=315,in=90] (0.5,0) to[out=270,in=45] (-0.5,-1);
	\end{tikzpicture}
	=
	\begin{tikzpicture}[bwwrect={(-1,-1)}{(1,1)},yscale=1.2]
		\draw (0.5,1) -- (0.5,-1);
		\draw (-0.5,1) -- (-0.5,-1);
	\end{tikzpicture}
	\qquad
	\begin{tikzpicture}[bwwrect={(-1,-1)}{(1,1)},yscale=1.2]
          \begin{scope}[rounded corners=10pt]
          \draw (0.7,1) -- (-.3,0) -- (-0.7,-1);
          \draw[wipe] (0,1) -- (.3,0) -- (0,-1);
          \draw[wipe] (-0.7,1) -- (-.3,0) -- (0.7,-1);
        \end{scope}
      \end{tikzpicture}
	=
	\begin{tikzpicture}[bwwrect={(-1,-1)}{(1,1)},yscale=1.2]
          \begin{scope}[rounded corners=10pt]
          \draw (0.7,1) -- (.3,0) -- (-0.7,-1);
          \draw[wipe] (0,1) -- (-.3,0) -- (0,-1);
          \draw[wipe] (-0.7,1) -- (.3,0) -- (0.7,-1);
        \end{scope}
      \end{tikzpicture}
\\
&	\begin{tikzpicture}[bwwrect={(-1,-1)}{(1,1)},yscale=1.2]
		\draw  (-0.2,-0.2) -- (-0.5,-1);
		\draw  (-0.2,-0.2) -- (0.7,1);
		\draw  (-0.2,-0.2) -- (0,1);
		\draw[wipe,rounded corners=10pt]  (-0.7,1) -- (0.3,0.3) -- (0.5,-1);
	\end{tikzpicture}
	=
	\begin{tikzpicture}[bwwrect={(-1,-1)}{(1,1)},yscale=1.2]
		\draw  (0.2,0.3) -- (-0.5,-1);
		\draw  (0.2,0.3) -- (0.7,1);
		\draw  (0.2,0.3) -- (0,1);
		\draw[wipe,rounded corners=10pt]  (-0.7,1) -- (-0.3,-0.3) -- (0.5,-1);
	\end{tikzpicture}
      \end{split}
    \end{equation}
The quotient of $\Pcat$ by the monoidal ideal generated by the relations \eqref{eq:braidrelations1} is a braided category.

The ribbon relation is shown in \eqref{eq:ribbon}:
\begin{equation}\label{eq:ribbon}
\ribbon
=
\begin{tikzpicture}[bwwrect={(-1.25,-2)}{(1.25,2)},yscale=0.75]
	\draw (-1,2) to[out=270,in=90] (1,-2);
\end{tikzpicture}
\end{equation}
\begin{defn} The category $\Rcat$ is the quotient of $\Pcat$ by the monoidal ideal generated by the relations \eqref{eq:braidrelations1} and \eqref{eq:ribbon}.
\end{defn}

The \Dfn{bar involution} on $\Rcat$ is the involution on morphisms which switches over and under crossings.

The category $\Rcat$ is a ribbon category but is not the free ribbon category on a self-dual object with a trivalent vertex
as the relation \eqref{eq:invariance}
holds in $\Pcat$ (from its definition), and this condition is not required for a ribbon category:
\begin{equation}\label{eq:invariance}
	\invlft = \invrht
\end{equation}
\editorialnote[gray]{}
However, the category $\Rcat$ is the free ribbon category on a self-dual object with a trivalent vertex which is invariant under rotation.

\begin{rem}
  The existence of $\co$ and $\ev$ allows us to move boundary points from one side to the other.
 It is therefore sometimes convenient
  to draw diagrams in a {\em disk}\/ rather than a rectangle. The interpretation is that the diagram can be deformed into a rectangle,
  but one can freely choose any sequence of adjacent boundary points
  and declare them to be left endpoints, the complement being the right endpoints. Of course such a choice should be made
  consistently when writing an equality between diagrams; for example, the relations \eqref{eq:braidrelations1} can be rewritten
\begin{equation}\label{braidrelations2}
	\begin{tikzpicture}[bwwcircle]
		\draw (45:1) to[out=225,in=90] (-0.5,0) to[out=270,in=135] (315:1);
		\draw[wipe] (135:1) to[out=315,in=90] (0.5,0) to[out=270,in=45] (225:1);
	\end{tikzpicture}
	=
	\begin{tikzpicture}[bwwcircle]
		\draw (45:1) to[out=225] (315:1);
		\draw (135:1) to[out=315,in=45] (225:1);
	\end{tikzpicture}
	,\
	\begin{tikzpicture}[bwwcircle]
		\def\x{0.45}
		\draw (30:1) to[out=210,in=15] (60:\x) to[out=195,in=45] (180:\x) to[out=225,in=30] (210:1);
		\draw [blue!20, line width =3mm] (150:1) to[out=330,in=135] (180:\x) to[out=315,in=165] (300:\x) to[out=345,in=150] (330:1);
		\draw (150:1) to[out=330,in=135] (180:\x) to[out=315,in=165] (300:\x) to[out=345,in=150] (330:1);
		\draw [wipe] (90:1) to[out=270,in=105] (60:\x) to[out=285,in=75] (300:\x) to[out=255,in=90] (270:1);
	\end{tikzpicture}
	=
	\begin{tikzpicture}[bwwcircle]
		\def\x{0.45}
		\draw (30:1) to[out=210,in=45] (0:\x) to[out=225,in=15] (240:\x) to[out=195,in=30] (210:1);
		\draw [blue!20, line width =3mm] (150:1) to[out=330,in=165] (120:\x) to[out=345,in=135] (0:\x) to[out=315,in=150] (330:1);
		\draw (150:1) to[out=330,in=165] (120:\x) to[out=345,in=135] (0:\x) to[out=315,in=150] (330:1);
		\draw [wipe] (90:1) to[out=270,in=75] (120:\x) to[out=255,in=105] (240:\x) to[out=285,in=90] (270:1);
	\end{tikzpicture}
,\
	\begin{tikzpicture}[bwwcircle]
		\draw  (0,0) -- (270:1);
		\draw  (0,0) -- (30:1);
		\draw  (0,0) -- (150:1);
		\draw[wipe]  (180:1) to[out=0,in=180] (0,0.3) to[out=0,in=180] (0:1);
	\end{tikzpicture}
	=
	\begin{tikzpicture}[bwwcircle]
		\draw  (0,0) -- (270:1);
		\draw  (0,0) -- (30:1);
		\draw  (0,0) -- (150:1);
		\draw[wipe]  (180:1) to[out=0,in=180] (0,-0.3) to[out=0,in=180] (0:1);
	\end{tikzpicture}
\end{equation}
\end{rem}

Recall that a self-dual object in a monoidal category is a triple $(V,\ev_V,\co_V)$ where $V$ is an object of $\Vcat$
and $\ev_V\colon V\otimes V \to \unit$, $\co_V\colon \unit \to V\otimes V$ satisfy the zig-zag relations.
\begin{defn}\label{defn:free}
	Let $\Vcat$ be a ribbon category. The category $T(\Vcat)$ has objects pairs consisting of
	a self-dual object $(V,\ev_V,\co_V)$ and a morphism $\mu_V\colon V\otimes V\to V$
	such that the following diagram commutes:
	\begin{center}
		\begin{tikzcd}[black]
			V\otimes V \otimes V \arrow[r,"\mu_V\otimes 1_V"] \arrow[d,"1_V\otimes\mu_V"'] & V\otimes V \arrow[d,"\ev_V"] \\
			V\otimes V \arrow[r,"\ev_V"'] & \unit
		\end{tikzcd}
	\end{center}
	A morphism $(V,\ev_V,\co_V,\mu_V) \to (W,\ev_W,\co_W,\mu_W)$ is a morphism $\phi\colon V\to W$ such that
	\[ (\phi\otimes\phi)\circ \ev_W = \ev_V \ ,\quad \co_V\circ (\phi\otimes\phi) =\co_W \ ,\quad
	(\phi\otimes\phi)\circ\mu_W = \mu_V\circ \phi  \]
\end{defn}

The tuple $(\obj1,\ev,\co,\mu)$ is an object of $T(\Rcat)$ where $\ev,\co,\mu$ are the generators in \eqref{eq:catgens}.

Let $\Vcat$ be a ribbon category. Define $\Hom(\Rcat,\Vcat)$ to be the category whose objects are ribbon functors $\Rcat \to \Vcat$
and whose morphisms are natural transformations of ribbon functors.

The universal property of $\Rcat$ is:
\begin{prop}\label{prop:free}
For any ribbon category $\Vcat$, the  functor $\Hom(\Rcat,\Vcat) \to T(\Vcat)$ which sends $\Phi\colon \Rcat\to\Vcat$
to the tuple $(\Phi(\obj1),\Phi(\ev),\Phi(\co),\Phi(\mu))$ is an equivalence of categories.
\end{prop}

\begin{proof} It is folklore that this can be proved using the method of proof of Reidemeister's theorem.
There are similar results in \cite{Freyd1989} and \cite{yetter1989}.
\end{proof}

\subsection{Rotation, trace, inner product}\label{ssec:rottrace}
Various basic operations can be performed on diagram categories.

\begin{defn}\label{defn:rot}
  The rotation map is the linear map, $\Rot\colon \Hom_\Dcat(\obj r,\obj s)\to \Hom_\Dcat(\obj r,\obj s)$, given by
  \begin{equation}\label{eq:rot}
    \Rot\colon\generalrect\:\mapsto\:\rotgeneralrect
  \end{equation}
\end{defn}

Note that if we think of the diagram as drawn on a disk, then this operation is simply a rotation by $2\pi/\obj k$,
where $\obj k=\obj r+\obj s$ is the number of external endpoints (spaced regularly).

\begin{defn}\label{defn:trace}
  The trace map is the linear map, $\Tr\colon \End_\Dcat(\obj r)\to \End_\Dcat(\obj 0)$, given by
  \begin{equation}\label{eq:trace}
    \Tr\colon\generalrecta\:\mapsto\:\tracegeneralrect
  \end{equation}

  \editorialnote[gray]{}
  The inner product is the map $\left<\cdot,\cdot\right>\colon\Hom_\Dcat(\obj r,\obj s)\otimes \Hom_\Dcat(\obj s,\obj r)\to \End_\Dcat(\obj 0)$ given by
  \begin{equation}\label{eq:inner}
    \left<a,b\right> \coloneqq \Tr(ab)
  \end{equation}
\end{defn}
One has the trace property $\Tr(ab)=\Tr(ba)$.

In the categories of interest, we expect $\End_\Dcat(\obj 0)$ to be isomorphic to the base ring,
and the inner product to be a perfect pairing
between $\Hom_\Dcat(\obj r,\obj s)$ and $\Hom_\Dcat(\obj s,\obj r)$.



\section{Quantum groups and coefficient rings}\label{sec:coeffs}

\subsection{Quantised enveloping algebras}\label{sec:qea}
The standard set-up for the representation theory of quantised enveloping algebras is to start with the Laurent polynomial ring, $\bZ[q^{\pm 1}] $.
The quantum integers $[n]_q\in \bZ[q^{\pm 1}]$ for $n\in \bZ$ are defined by
\begin{equation*}
	[n]_q = \frac{q^n-q^{-n}}{q-q^{-1}}
\end{equation*}

\begin{defn} The ring $\kk_q$ is constructed from $\bQ[q^{\pm 1}]$ by adjoining a multiplicative inverse of $[n]_q$ for $n>1$.
\end{defn}

The homomorphism $\psi\colon \kk_q \to \bQ$ is defined by
\begin{equation}\label{eq:q1}
	q\mapsto 1 \quad , \quad [n]_q \mapsto n
\end{equation}

\begin{rem} It is more common to use the ring
\begin{equation}\label{defnA}
	\{f/g \in \bQ(q) : g(1) \ne 0 \}
      \end{equation}
This is a local ring with field of fractions $\bQ(q)$ and residue field $\bQ$.
However, the representation theory is very similar for these two rings, \cite[Chapter~5]{jantzen1996}.
\end{rem}

The bar involution on $\kk_q$ is defined by
\begin{equation}\label{eq:barq}
	q^{\pm 1} \mapsto q^{\mp 1} \quad , \quad [n]_q \mapsto [n]_q \text{ for $n\in\bZ$}
\end{equation}

The following is essentially the Definition in \cite[\S 4.2]{jantzen1996}. A minor innovation is that we have introduced $[H_i]$ generators so that we can specialise to $q=1$.

Let $C$ be a Cartan type with $\Phi$ the finite root system and $\alpha_i$, $i \in I$, a set of simple roots.  There is a unique inner product $(\ ,\ )$ on the real vector space $\bR \Phi$ generated by the root system $\Phi$
such that $(\alpha,\alpha)=2$ for all short roots. Thus $(\lambda,\mu)=d_C\langle\lambda,\mu\rangle$, where $d_C$ is half the squared length of a long root
and $\langle\ ,\ \rangle$ is the form in \eqref{eq:casimir}. For $i,j \in I$, we have the Cartan matrix
\begin{equation} \label{Cartan}
	\frac{2(\alpha_i, \alpha_j)}{(\alpha_i, \alpha_i)} = a_{ij}
\end{equation}
In what follows, we will write $(i,j)$ for $(\alpha_i, \alpha_j)$, $i,j \in I$.  For each $i \in I$, define
\[ d_i = \frac{(i,i)}{2} \]

For $i \in I$, $a,n \in \bZ$, with $n>0$, we define
\begin{gather*}
	q_i := q^{d_i},\qquad
	[a]_i := \frac{q_i^a - q_i^{-a}}{q_i - q_i^{-1}} = \frac{q^{ad_i} - q^{-ad_i}}{q^{d_i} - q^{-d_i}},
	\\
	[n]_i^! := [n]_i [n-1]_i \dotsb [1]_i,\qquad
	{\qbinom{a}{n}}_i := \frac{[a]_i [a-1]_i \dotsb [a-n+1]_i}{[n]_i^!},\qquad
	{\qbinom{a}{0}}_i := 1.
\end{gather*}

The quantized enveloping algebra $U_q(C)$ is the $\kk_q$-algebra with generators $E_i, F_i, [H_i], K_i^{\pm 1}$, $i \in I$, and relations (for $i,j \in I$)
\begin{align*}
	K_i K_i^{-1} &= 1 = K_i^{-1} K_i,&
	K_i K_j &= K_j K_i, \\
	K_i E_j K_i^{-1} &= q^{(i,j)} E_j,&
	K_i F_j K_i^{-1} &= q^{-(i,j)} F_j, \\
	E_i F_j - F_j E_i &= \delta_{ij} [H_i],&
	(q_i-q_i^{-1})[H_i] &= K_i - K_i^{-1},
\end{align*}
where $\delta_{ij}$ is the Kronecker delta, and (for $i \ne j$),
\[
\sum_{r=0}^{1-a_{ij}} (-1)^r {\qbinom{1-a_{ij}}{r}}_i E_i^{1-a_{ij}-r} E_j E_i^r = 0, \qquad
\sum_{r=0}^{1-a_{ij}} (-1)^r {\qbinom{1-a_{ij}}{r}}_i F_i^{1-a_{ij}-r} F_j F_i^r = 0.
\]

\begin{rem} Substituting $q=1$ gives the Serre relations for $U(C)$ together with the relations that $K_i^2=1$ and $K_i$ is central. The type I representations
	have $K_i=1$ for all $i$ when $q=1$.
\end{rem}

There is a unique Hopf algebra structure on $U_q(C)$ such that, for all $i \in I$,
\begin{align} \label{HopfE}
	\Delta(E_i) &= E_i \otimes 1 + K_i \otimes E_i,&
	\varepsilon(E_i) &= 0,&
	S(E_i) &= - K_i^{-1} E_i,
	\\ \label{HopfF}
	\Delta(F_i) &= F_i \otimes K_i^{-1} + 1 \otimes F_i,&
	\varepsilon(F_i) &=0,&
	S(F_i) &= - F_i K_i,
	\\ \label{HopfK}
	\Delta(K_i) &= K_i \otimes K_i,&
	\varepsilon(K_i) &=1,&
	S(K_i) &= K_i^{-1}.
\end{align}

The \Dfn{bar involution} is the involution of $U_q(C)$ as a ring given by
\begin{equation}\label{eq:barqg}
	q \mapsto q^{-1} \quad,\quad E_i\mapsto E_i \quad,\quad F_i\mapsto F_i \quad,\quad K_i\mapsto K_i^{-1} \quad,\quad [H_i]\mapsto [H_i]
\end{equation}
Let $\Gamma$ be the group of diagram automorphisms and let $\gamma\colon I\to I$ be an element of $\Gamma$. This defines an automorphism of $U_q(C)$, $\gamma\colon a\mapsto a\triangleleft \gamma$ for $a\in U_q(C)$,  by
\begin{equation}\label{eq:semidirect}
	E_i\triangleleft \gamma = E_{\gamma(i)}, F_i\triangleleft \gamma = F_{\gamma(i)}, [H_i]\triangleleft \gamma = [H_{\gamma(i)}], K_i^{\pm 1}\triangleleft \gamma = K^{\pm 1}_{\gamma(i)},
\end{equation}
The algebra $ \Gamma\sharp U_q(C)$ is the vector space $\bQ\Gamma\otimes_\bQ U_q(C)$ with multiplication
\begin{equation}
	(\gamma\otimes a)(\gamma'\otimes a') = \gamma\gamma' \otimes (a\triangleleft \gamma')a'
\end{equation}
for $\gamma,\gamma'\in\Gamma$ and $a,a'\in U_q(C)$.
The coproduct is $\Delta(\gamma\otimes a)=(\gamma\otimes\gamma)\Delta(a)$. The unit is $1\otimes 1$. The counit is $\varepsilon\otimes\varepsilon$.

Define $\fU_q(C)$ to be the smash product $\Gamma\sharp U_q(C)$.
The category $\fU_q(C)$-mod is the full subcategory of the category of right $\fU_q(C)$-modules whose objects are type I $U_q(C)$-modules and finitely generated free $\kk_q$-modules.
The \Dfn{bar involution} extends to an involution on $\fU_q(C)$ by $\gamma\mapsto\gamma$ for $\gamma\in\Gamma$.

\begin{prop} The category $\fU_q(C)$-mod is a ribbon category.
\end{prop}

\begin{proof}
The category $U_{q}(C)$-mod is the category of right type I $U_q(C)$-modules.
This is a ribbon category, \cite{Reshetikhin1990}, following \cite{drinfeld1989}.

 The braiding on $U_q(C)$-mod is given by the universal $R$-matrix.
This is an element of a completed tensor product $U_q(C)\widehat{\otimes} U_q(C)$.
The construction is given in \cite[Chapter~7]{jantzen1996}. An inspection of this construction
reveals that it is invariant under the action of $\Gamma$ and hence gives a braiding on $\fU_q(C)$-mod.

The ribbon element lies in a completion of $U_q(C)$. By \cite[Proposition~2.14(4)]{Leduc1997}, it acts on the irreducible
module with highest weight $\Lambda$ by the constant
\begin{equation*}
	q^{-(\Lambda,\Lambda+2\rho)}=q^{-d_C\casimir(\Lambda)}
\end{equation*}
where $\rho$ is half the sum of the positive roots. It is clear that this is invariant under $\Gamma$.
\end{proof}

Every $\Hom$ space in this category is a finitely generated free $\kk_q$-module. Taking the specialisation $\kk_q\to \bQ$, $q\mapsto 1$, gives the category
of finite dimensional representations of the semidirect product algebra $\bQ\Gamma \ltimes U(C)$  extended by central elements $K_i$.

The quantum analogue of the adjoint representation is given in \cite[\S 5A.5]{jantzen1996} (also \cite{lusztig2016}).
Let $\Phi$ be the set of roots and let $\{\alpha_i:i\in I\}$ be the simple roots. A basis is the set
\begin{equation}
	\{v_\alpha:\alpha\in\Phi\} \amalg \{h_i:i\in I\}
\end{equation}

For $i\in I$ and $\beta\in\Phi$ with $\beta+\alpha_i\notin\Phi$ and $\beta\ne \alpha_i$; the $i$-string of roots through $\beta$ is the set
\begin{equation}
	\{\beta -k\,\alpha_i: 0\leqslant k\leqslant \left\langle \beta,\check\alpha_i\right\rangle\}
\end{equation}

For $\alpha\in\Phi$, $v_\alpha$ has weight $\alpha$; and for $i\in I$, $h_i$ has weight $0$.

For $\beta -k\,\alpha_i$ in the $i$-string through $\beta$, put
\begin{align}
	F_i\,v_{\beta -k\,\alpha_i} &= \begin{cases}
		[k+1]_i\,v_{\beta -(k+1)\,\alpha_i} & \text{if $k\ne \left\langle \beta,\check\alpha_i\right\rangle$} \\
		0 & \text{if $k= \left\langle \beta,\check\alpha_i\right\rangle$} \\
	\end{cases} \\
	E_i\,v_{\beta -k\,\alpha_i} &= \begin{cases}
		[\left\langle \beta,\check\alpha_i\right\rangle+1-k]_i\,v_{\beta -(k-1)\,\alpha_i} & \text{if $k\ne 0$} \\
0 & \text{if $k= 0$} \\
\end{cases}
\end{align}

For $i\in I$, put
\begin{align}
F_i\,v_{\alpha_i} &= h_i  & F_i\,h_i &= [2]_i v_{-\alpha_i}  & F_i\,v_{-\alpha_i} &= 0 \\
E_i\,v_{\alpha_i} &= 0  & E_i\,h_i &= [2]_i v_{\alpha_i}  & E_i\,v_{-\alpha_i} &= h_i
\end{align}

and finally, for $i\ne j$, put
\begin{equation}
	E_i\,h_j = [-\left\langle \alpha_i,\check\alpha_j\right\rangle]_j\,v_{\alpha_i}
	\quad,\quad
	F_i\,h_j = [-\left\langle \alpha_i,\check\alpha_j\right\rangle]_j\,v_{-\alpha_i}
\end{equation}

The action of $\Gamma$ on the $q$-deformed adjoint representation goes as follows:
if $\theta$ is the highest root, then $\gamma v_\theta=\epsilon(\gamma)v_\theta$ for $\gamma\in\Gamma$,
where $\epsilon$ is a character of $\Gamma$
which is the sign representation for $C=A_2,D_4,E_6$ (these are the only nontrivial cases of $\Gamma$ that we consider below).
$v_\theta$ is a cyclic vector for the action of $U_q(C)$, so the action of $\Gamma$ is extended to the whole representation
by use of \eqref{eq:semidirect}.

\editorialnote{}

\subsection{Lie superalgebra}\label{ssec:osp12}
\newcommand{\fsl}{\mathfrak{sl}}

Our approach to the Lie superalgebra follows \cite{xu2018} but this simplifies in the case $\osp(1|2)$.
The details in this case are known to experts but since we did not find a reference we include a summary
here for the convenience of the reader.

\begin{defn} The quantum group $U_q(\fsl(2))$ is the $\bZ[q,q^{-1}]$-algebra with generators
	\begin{equation}
		E',F',[H'],K',K'^{-1}
	\end{equation} and defining relations
	\begin{gather*}
		K'\,K'^{-1} = 1 = K'^{-1}\,K' \\
		K'\,E' = q^2\,E'\,K' \quad,\quad K'\,F' = q^{-2}\,F'\,K' \\
		E'\,F' - F'\,E' = [H'] \\
		(q-q^{-1})[H'] = K'-K'^{-1}
	\end{gather*}

	The coproduct, $\Delta$, antipode, $S$, and counit, $\varepsilon$ are:
	\begin{gather*}
		\Delta(E') = E'\otimes 1+K'\otimes E',\quad \Delta(F') = F'\otimes K'^{-1}+1\otimes F', \quad \Delta(K') = K'\otimes K' \\
		S(E') = -K'^{-1}E',\quad S(F') = -F'K',\quad S(K') = K'^{-1} \\
		\varepsilon(E')=0,\quad \varepsilon(F')=0,\quad \varepsilon(K')=1
	\end{gather*}
\end{defn}

\begin{defn} The quantum group $U_q(\osp(1|2))$ is the $\bZ[q,q^{-1}]$-algebra with even generators $[H],K,K^{-1},\sigma$,
	odd generators $E,F$ and defining relations
	\begin{gather*}
		\sigma^2 = 1 \\
		\sigma\,E\,\sigma = -E \quad,\quad \sigma\,F\,\sigma = -F \\
		\sigma\,K^{\pm 1}\,\sigma = K^{\pm 1}  \quad,\quad \sigma\,[H]\,\sigma = [H] \\
		K\,K^{-1} = 1 = K^{-1}\,K \\
		K\,E = q\,E\,K \quad,\quad K\,F = q^{-1}F\,K \\
		E\,F + F\,E = [H] \\
		(q-q^{-1})[H] = K-K^{-1}
	\end{gather*}

	The coproduct, $\Delta$, antipode, $S$, and counit, $\varepsilon$, are:
	\begin{gather*}
		\Delta(\sigma) = \sigma\otimes\sigma,\quad \Delta(K) = K\otimes K \\
		\Delta(E) = E\otimes \sigma+K\otimes E,\quad \Delta(F) = F\otimes K^{-1}\sigma+1\otimes F \\
		S(E) = -K^{-1}E\sigma,\quad S(F) = -FK\sigma,\quad S(K) = K^{-1}, \quad S(\sigma)=\sigma \\
		\varepsilon(E)=0,\quad \varepsilon(F)=0,\quad \varepsilon(K)=1,\quad \varepsilon(\sigma)=1 \\
	\end{gather*}
\end{defn}

\begin{prop}\label{prop:osp12}
  There is a Hopf algebra homomorphism $\theta\colon U_{q'}(\fsl(2))
	\to U_{q}(\osp(1|2))$, $q=-q'^2$, given by
	\begin{gather*}\label{eq:iso}
		\theta\colon E'\mapsto -(q'+q'^{-1})\sigma E ,\quad  \theta\colon F'\mapsto F,\\  \theta\colon [H']\mapsto -(q'+q'^{-1})\sigma [H] , \quad \theta\colon K'^{\pm 1}\mapsto \sigma K^{\pm 1}
	\end{gather*}
\end{prop}

\begin{proof}
This is a direct calculation.
	\begin{multline*}
		\theta(K'E') = -(q'+q'^{-1})\sigma K \sigma E = - (q'+q'^{-1})K E \\ = -q (q'+q'^{-1})E K = q (q'+q'^{-1})\sigma E \sigma K = \theta(-q E'K' )
	\end{multline*}
	\begin{multline*}
		\theta(E'F'-F'E') =-(q'+q'^{-1})(\sigma E F - F \sigma E) \\ = -(q'+q'^{-1})\sigma(E F + F E) = -(q'+q'^{-1})\sigma[H] = \theta([H'])
	\end{multline*}
	\begin{gather*}
		\theta(F'K') = F \sigma K = -\sigma F K = -q\, \sigma K F = \theta(-q\, K'F') \\
		\theta((q'-q'^{-1})[H']) = -(q'^2-q'^{-2})\sigma [H] = \sigma(K-K^{-1}) = \theta(K'-K'^{-1})
	\end{gather*}
\end{proof}

\subsection{Blow-up}\label{sec:blowup}
The exceptional series is a set of points on the projective line with coordinate $\nv$.
The quantum parameter in \S\ref{sec:qea} is $q$. For the quantum analogue of the
exceptional series we need two parameters which, heuristically, are $q$ and $q^\nv$.
In this section we define a ring $\coeffs$ which will be our ring of coefficients.

The ring $A_{\pp,\qq}$ is a blow-up of the Laurent polynomial ring $\bZ[\pp^{\pm 1},\qq^{\pm 1}]$:
\begin{defn}\label{defn:qz}
	The ring $A_{\pp,\qq}$ is the quotient of $\bZ[\pp^{\pm 1},\qq^{\pm 1},[\nv]]$ by the relation
	\begin{equation*}
		(\qq-\qq^{-1})[\nv] = (\pp-\pp^{-1})
	\end{equation*}
\end{defn}

For $s\in\bZ$, the quantum integer $[s]$ is the Laurent polynomial
\begin{equation*}
	[s] = \frac{\qq^s-\qq^{-s}}{\qq-\qq^{-1}}
\end{equation*}
For $r\in\bZ$, define $[r\,\nv]\in A_{\pp,\qq}$ by
\begin{equation*}
	[r\,\nv] = [\nv]\left(\frac{\pp^r-\pp^{-r}}{\pp-\pp^{-1}}\right)
\end{equation*}

\begin{lemma}
	\begin{equation*}
		(\qq^s-\qq^{-s})[r\,\nv] = [s](\pp^r-\pp^{-r})
	\end{equation*}
\end{lemma}
\begin{proof}
	\begin{equation*}
		(\qq-\qq^{-1})[r\,\nv] =
		(\qq-\qq^{-1})[\nv]\left(\frac{\pp^r-\pp^{-r}}{\pp-\pp^{-1}}\right) =
		(\pp^r-\pp^{-r})
	\end{equation*}
\end{proof}
For $r,s\in\bZ$, define $\qint rs\in A_{\pp,\qq}$ by
\begin{equation*}
	\qint rs = \qq^s[r\,\nv] + \pp^{-r}[s] = \qq^{-s}[r\,\nv] + \pp^r[s]
\end{equation*}

Note that
\begin{equation*}
	\qq^s[r\,\nv] + \pp^{-r}[s] - \qq^{-s}[r\,\nv] - \pp^r[s]
	= (\qq^s-\qq^{-s})[r\,\nv] -[s](\pp^r-\pp^{-r})
	= 0
\end{equation*}

\subsection{Specialisations}\label{sec:spec}
The ring $A_{\pp,\qq}$ is an integral domain and the field of fractions is the field
of rational functions $\bQ(\pp,\qq)$. The homomorphism $A_{\pp,\qq}\to \bQ(\pp,\qq)$ is given by
\begin{equation*}
	\pp\mapsto \pp \quad,\quad \qq\mapsto \qq \quad,\quad [r\,\nv+s] \mapsto \frac{\pp^r\qq^s-\pp^{-r}\qq^{-s}}{\qq-\qq^{-1}}
\end{equation*}

We also have the following homomorphisms:
\begin{itemize}
	\item For $r,s\in\bZ$ with $s\ne0$ there is a homomorphism $\theta_{r,s}\colon A_{\pp,\qq}\to \kk_q$ given by
	\begin{equation*}
		\pp\mapsto q^r , \qq\mapsto q^s , \qint uv \mapsto \frac{[ur+vs]_q}{[s]_q}
	\end{equation*}
	\item There is a homomorphism $\psi\colon A_{\pp,\qq}\to\bZ[\nv]$ given by
	\begin{equation*}
		\pp\mapsto 1 , \qq\mapsto 1 , \qint uv \mapsto (u\nv+v)
	\end{equation*}
\end{itemize}

For $r,s\in\bZ$ with $s\ne0$, the following diagram is commutative
\begin{equation*}
	\begin{tikzcd}[black]
		A_{\pp,\qq} \arrow[r,"\theta_{r,s}"] \arrow[d,"\psi"']  & \kk_q \arrow[d, "q\mapsto 1"] \\
		\bZ[\nv] \arrow[r,"\nv\mapsto r/s"'] & \bQ
	\end{tikzcd}
\end{equation*}
since both composites are the homomorphism $\pp,\qq\mapsto 1$ and $\qint uv \mapsto (ur+vs)/s$.

The \Dfn{bar involution} is the involution of $A_{\pp,\qq}$ given by
\begin{equation}\label{eq:bar}
	\pp\mapsto \pp^{-1} , \qq\mapsto \qq^{-1} , \qint rs \mapsto \qint rs
\end{equation}
For $r,s\in\bZ$ with $s\ne0$, the homomorphism $\theta_{r,s}$ intertwines the bar involutions.

For an exceptional simple Lie algebra, the normalisation in \S\ref{sec:qea} gives $d_C=3$ for $C=G_2$, $d_C=2$ for $C=F_4$ and $d_C=1$ otherwise.
For $C=\OSp$ set $d_C=2$, in accordance with \S\ref{ssec:osp12}, and for $C=\varnothing$ set $d_C=1$.
Let $s_C$ be the denominator of $d_C n_C$. Then we define $\theta_C\colon A_{\pp,\qq} \to \kk_{q^{1/s_C}}$ by $\pp \mapsto q^{d_C n_C}$ and $\qq \mapsto q^{d_C}$.

\begin{ex}
The key property of $\dim_q(L)$ in \eqref{eq:qdimL} is that it interpolates the quantum dimensions in the following sense.
For each exceptional Cartan type, $C$, substituting $\nv=n_C$ (from Table~\ref{fig:nvalues}) in $\dim_q(L)$ gives

\begin{equation*}
	\begin{array}{cccccccccc}
	\varnothing & \OSp & A_1 & A_2 & G_2 & D_4 & F_4 & E_6 & E_7 & E_8 \\ \hline
	0 & \scriptstyle{[3]-[2]} & \scriptstyle{[3]} & \scriptstyle{[2][4]} & \frac{[7][8][15]}{[3][4][5]} & \frac{[4][4][7]}{[2][2]} & \frac{[12][13][20]}{[2][5][6]} & \frac{[8][9][13]}{[3][4]} & \frac{[12][14][19]}{[4][6]} & \frac{[20][24][31]}{[6][10]}
\end{array}
\end{equation*}

These agree with the quantum dimensions calculated from the Weyl dimension formula.
\end{ex}

Define $\Sigma\subset A_{\pp,\qq}$ to be the set of ratios
\begin{equation*}
	f=\frac{\prod_i[r_i\nv+s_i]}{\prod_j[r'_j\nv+s'_j]}
\end{equation*}
of finite products of nonzero quantum integers such that $f\in A_{\pp,\qq}$ and $\theta_C(f)\ne0$ for every exceptional Cartan type $C$.
Specialisation is applied to $f$ after cancellation; individual numerator and denominator factors may vanish.
This set is multiplicative. Define the coefficient ring $\coeffs=\Sigma^{-1}A_{\pp,\qq}$.
Every element of $\Sigma$ is fixed by the bar involution, which therefore extends to $\coeffs$.

The role of cancellation can be explained using cyclotomic polynomials.
For $m\ge1$, let $\Phi_m(z)$ denote the normalised $m$th cyclotomic polynomial in $z^2$, namely %
\begin{equation*}
	\Phi_m(z)=\prod_{d\mid m}(z^d-z^{-d})^{\mu(m/d)},
\end{equation*}
where $\mu$ is the M\"obius function. For $(r,s)\ne(0,0)$ and $k\ge1$,
\begin{equation*}
	[k(r\nv+s)]=[r\nv+s]\prod_{\substack{m\mid k\\m>1}}\Phi_m(\pp^r\qq^s).
\end{equation*}
For $m=1$, $\Phi_1(z)=z-z^{-1}$, and
\begin{equation*}
	\frac{\Phi_1(\pp^r\qq^s)}{\Phi_1(\qq)}=[r\nv+s],
	\qquad \psi([r\nv+s])=r\nv+s.
\end{equation*}
Such a quantum integer belongs to $\Sigma$ precisely when $rn_C+s\ne0$ for every $C$.
After reducing each pair $(r,s)$ to a primitive pair, any factors vanishing at an exceptional type must cancel in a ratio belonging to $\Sigma$.

For $m>1$, $\Phi_m(z)$ is invariant under $z\mapsto z^{-1}$ and $\Phi_m(1)\ne0$.
Moreover,
\begin{equation*}
	\Phi_m(\pp^r\qq^s)=\prod_{d\mid m}[d(r\nv+s)]^{\mu(m/d)}
\end{equation*}
is a Laurent polynomial and belongs to $\Sigma$ for every $(r,s)\ne(0,0)$.
Indeed, under $\theta_C$ its argument becomes $t^e$, where $t=q^{1/s_C}$ and $e\in\bZ$.
If $e\ne0$, its image is a ratio of products of nonzero integer quantum integers, hence a unit of $\kk_t$;
if $e=0$, its image is the nonzero integer $\Phi_m(1)$, also a unit. Its classical specialisation is likewise $\Phi_m(1)$.

This factorisation shows that every element of $\Sigma$ maps to a unit under each $\theta_C$, and under $\psi$ to a rational function regular and nonzero at every $n_C$.
Thus the specialisations extend to homomorphisms $\theta_C\colon\coeffs\to\kk_{q^{1/s_C}}$ and $\psi\colon\coeffs\to\bQ(\nv)$.
Quantum integers vanishing at an exceptional type remain nonunits.

For each exceptional Cartan type $C$, the following diagram is commutative
\begin{equation*}
	\begin{tikzcd}[black]
        \coeffs \arrow[r,"\theta_{C}"] \arrow[d,"\psi"']  & \kk_{q^{1/s_C}} \arrow[d, "q\mapsto 1"] \\
		\mathrm{Im}(\psi)\subset\bQ(\nv) \arrow[r,"\nv\mapsto n_C"'] & \bQ
	\end{tikzcd}
\end{equation*}

Applying $\psi$ to the quantum dimensions in \eqref{eq:qdims2} and \eqref{eq:qdims3} gives the dimension formulae in
\cite{cohen1996} and \cite[Chapter~17]{Cvitanovic2008}.

\subsection{Multi-variable}\label{ssec:blowups}
In \S\ref{sec:yangbaxter} we will need additional parameters.

Given variables $\tv_1,\ldots,\tv_k$ and $\qq$, we write $A_{\tv_1,\ldots,\tv_k,\qq}$ for the subring of $\bQ(\tv_1,\ldots,\tv_k,\qq)$
generated by $\tv_1^{\pm},\ldots,\tv_k^{\pm},\qq^{\pm}$ and
\begin{equation*}
	\frac{\tv_1^{a_1}\dotsb \tv_k^{a_k}\qq^a - \tv_1^{-a_1}\dotsb \tv_k^{-a_k}\qq^{-a}}{\qq-\qq^{-1}}
\end{equation*}
for $a_1,\dotsc,a_k,a\in\bZ$.

There is a homomorphism $A_{\tv_1,\ldots,\tv_k,\qq}\to \bQ[\xx_1,\ldots,\xx_k]$ given by $\tv_i\mapsto 1$ for $1\leqslant i\leqslant k$ and
\begin{equation*}
	\frac{\tv_1^{a_1}\dotsb \tv_k^{a_k}\qq^a - \tv_1^{-a_1}\dotsb \tv_k^{-a_k}\qq^{-a}}{\qq-\qq^{-1}} \mapsto a_1\,\xx_1+\dotsb + a_k\,\xx_k + a
\end{equation*}
for $a_1,\dotsc,a_k,a\in\bZ$.

The bar involution inverts {\em all}\/ variables of $A_{\tv_1,\ldots,\tv_k,\qq}$.

\subsection{Functors}\label{sec:functors}
For $C$ a Cartan type, the quantum analogue of the adjoint representation is an object $L_C\in \fU_q(C)$-mod.
Let $\Icat(L_C)\subset \fU_q(C)$-mod be the full subcategory on the objects $\otimes^k L_C$ for $k\geqslant 0$.
In this section we construct, for each exceptional Cartan type, $C$, an $\kk_q$-linear ribbon functor $\Psi_C\colon\kk_q\Rcat\to \Icat(L_C)$.

The category $\Icat(L_C)$ inherits the structure of a ribbon category from $\fU_q(C)$-mod. It follows from Proposition~\ref{prop:free}
that, to construct the functor $\Psi_C$ it is sufficient to construct an element of the category $T(\Icat(L_C))$, defined in Definition~\ref{defn:free}.
Over $\bC$, the tuple is $(L_C,\ev_C,\co_C,\mu_C)$ where $\ev_C$ is the Killing form, $\co_C$ is determined by $\ev_C$ since the Killing form is
nondegenerate and $\mu_C$ is the Lie bracket. Taking the split real form of $L_C$ gives a tuple defined over $\bR$ and the real form has a rational form
(which can be defined by the Serre presentation) to give a tuple defined over $\bQ$. By averaging over $\Gamma$, if necessary, we can take this tuple
to be invariant under $\Gamma$. Specialising using the homomorphism $\psi$ defined in \eqref{eq:q1} is a functor $\fU_q(C)\text{-mod} \to \fU(C)\text{-mod}$.
The tuple defined over $\bQ$ can be lifted to a tuple defined over $\kk_q$. This can be seen by an explicit construction, or alternatively,
it can be deduced from the result that $\Hom$-spaces in $\fU_q(C)\text{-mod}$ are finitely generated free $\kk_q$-modules.%

For any exceptional Cartan type, $C$, the functor $\Psi_C$ is full. This is proved in \cite[Proposition~3.4]{ETINGOF2006}.
This is a version of the first fundamental theorem of tensor invariant theory originally
stated (without proof) in \cite{Littlewood1944}. The idea that informs this paper is that the interpolating category, if it exists, can be defined
as a quotient of $\coeffs\,\Rcat$ using the homomorphisms $\theta_C\colon \coeffs \to \kk_{q^{1/s_C}}$. In \S\ref{sec:twostring} and \S\ref{sec:threestring}
we find relations on $\coeffs\,\Rcat$ such that the quotient of $\coeffs\,\Rcat$ by the monoidal ideal generated by these relations approximates
an interpolating category. It is an open problem to determine if there is a finite set of relations such that any closed diagram can be evaluated.



\section{Two string relations}\label{sec:twostring}
In this section we construct a commutative $\coeffs$-algebra, $A(\obj2)$, which is a free $\coeffs$-module of rank five.
This algebra is constructed to interpolate the algebras $\End_{\fU_q(C)}(L_C\otimes L_C)$ for $C$ an exceptional Cartan type.
One can think of $A(\obj2)$ as a subquotient of $\End_{\coeffs} \Rcat(\obj 2)$, and so we use diagrams to describe its elements.

\editorialnote[gray]{}

These relations are a specialisation of relations given in \cite{Westbury2015},\cite[\S5.4]{westbury2022} and are also given in \cite{morrison2026}.

The following relation follows from $\Hom(\obj0,\obj1)=0$ which is a consequence of the assumption that $\obj1$ is irreducible.
\begin{equation}\label{eq:tadpole}
	\tadpole = 0
\end{equation}

\begin{defn}\label{def:basicrels}
  The  scalars $\pp$, $\delta$, $\phi$ and $\tau$ are defined by the relations:
	\begin{equation*}
		\circle	= \delta\:	\empty
		\qquad
		\bigon	= \phi\: \linecirc
	\end{equation*}

	\begin{equation*}
		\twistover =-\pp^6 \: \trivalent
		\qquad
		\twistunder = -\pp^{-6} \: \trivalent
	\end{equation*}

	\begin{equation}\label{triangle}
		\trianglecirc = \tau \: \trivalent
	\end{equation}
\end{defn}

The motivation for these relations comes from the observation that, for an exceptional Cartan type $C\ne\varnothing$,
the invariant spaces $\Hom_{\fU_q(C)}(I_C,\otimes^iL_C)$, for $i=0,2,3$, are free modules of rank one over the coefficient ring with bases
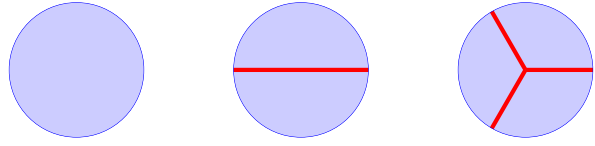
\begin{equation*}
	\empty \qquad \linecirc \qquad \trivalent
\end{equation*}
Hence there are relations of the form Definition~\ref{def:basicrels} for each exceptional Cartan type and we
are introducing interpolating relations.

These take the values
\begin{align}\label{eq:values}
	\delta &\coloneqq \frac{\qint 61 [5\,\nv-1][4\,\nv]}{[\nv+1][2\,\nv]} \\
	\phi &\coloneqq \frac{\left[3\,\nv\right]\ \left[6\,\nv+2\right]\ \left[4\,\nv-2\right]}{\left[\nv\right]\ \left[3\,\nv+1\right]\ \left[2\,\nv-1\right]} \\
	\tau &\coloneqq \frac{[2\,\nv]}{[4\,\nv]}\left(\frac{[6\,\nv+2]}{[3\,\nv+1]}\frac{[4\,\nv-2]}{[2\nv-1]}\frac{[3\,\nv]}{[\nv]}+(\qq-\qq^{-1})^2[\nv+1]\frac{[5\nv]}{[\nv]}\right)\\
	&= \frac{\left[2\,\nv\right]}{\left[\nv\right]}\left(\pp^4 + \pp^2\qq^2 - \pp^2 + \qq^2 - 1 + \qq^{-2} - \pp^{-2} + \pp^{-2}\qq^{-2} + \pp^{-4}\right)
\end{align}
The scalar $\delta$ is the quantum dimension $\dim_q(L)$ in \eqref{eq:qdimL}.
The scalar $\phi$ is a unit in $\coeffs$, since
\(
	\phi=\Phi_3(\pp)\Phi_2(\pp^3\qq)\Phi_2(\pp^2\qq^{-1})
\),
and each factor is a unit by \S\ref{sec:spec}.

We can also  apply the homomorphism $\psi\colon \coeffs\to \bQ(\nv)$. This gives
\begin{equation}
	\delta \mapsto 2\frac{(6\,\nv+1) (5\,\nv-1)}{(\nv+1)} \quad,\quad
	\phi \mapsto 12 \quad,\quad
	\tau \mapsto 6
\end{equation}
This dimension formula is given in \cite{cohen1996} and \cite[Chapter~17]{Cvitanovic2008}.

The constants $\phi$ and $\tau$ depend on a choice of normalisation. However they both scale by the same factor so the ratio
$\tau/\phi$ is independent of the choice of normalisation, see \eqref{eq:ratio}. Our choice of normalisation is dictated (up to sign) by
the requirements that $\phi,\tau\in \bZ[\pp^{\pm 1},\qq^{\pm 1}]$ and $\phi$, $\tau$ are coprime and invariant under the bar involution, \eqref{eq:bar}.
We shall return to this issue in \S\ref{ssec:norm},
where we shall discuss what happens when we do not impose such requirements on $\tau$, $\phi$, leading to a free normalisation
parameter in their definition.

\begin{prop}\label{prop:skein} The skein relation is
\begin{multline}\label{eq:skein}
	\frac{[6\nv]}{[3\nv][2\nv]}\:
	\left(\overcirc \ -\ %
          \undercirc\right)
        \ =
\\
	-(\qq-\qq^{-1})^3[\nv+1]\: \left(\Icirc
          \ - \
        \Ucirc
	\right) \\
	-(\qq-\qq^{-1})\:
	\left(\Kcirc
          \ - \
          \Hcirc \right)
\end{multline}
\end{prop}

\begin{rem} All the coefficients are in $\coeffs$. Applying the homomorphism $\psi\colon \coeffs\to \bQ(\nv)$ gives the relation
\begin{equation}
	\overcirc \ -\ \undercirc \ =\ 0
\end{equation}
\end{rem}

\begin{prop}\label{prop:jacobi} The deformed Jacobi relation is
	\begin{multline*}
\pp^{-3}\:\botover = (\qq-\qq^{-1})\frac{[\nv+1]}{[\nv]}\left(\pp^{-4}\,\Icirc + \pp^{4}\Ucirc + \overcirc\right) \\
-\pp\: \Kcirc + \pp^{-1}\:\Hcirc
	\end{multline*}
\end{prop}

\begin{rem} Applying the homomorphism $\psi\colon \coeffs\to \bQ(\nv)$ gives the Jacobi relation.
\end{rem}

\begin{cor}
	\begin{multline}
	\pp^{3}\:\topoverbar = -(\qq-\qq^{-1})\frac{[\nv+1]}{[\nv]}\left(\pp^{4}\,\Icirc + \pp^{-4}\Ucirc + \undercirc\right) \\
	-\pp^{-1}\: \Kcirc + \pp\:\Hcirc
\end{multline}
\end{cor}
\begin{proof}
  Apply the bar involution, \eqref{eq:bar}.
\end{proof}

\begin{prop}\label{prop:crossing} The relation that expresses the braid element in terms of a basis of planar diagrams is
\begin{multline*}
	\left(\frac{[6\,\nv]}{[3\,\nv]}\right)^2\frac{[\nv+1]}{[2\,\nv]}\,\overcirc
        \ = \ S_I\
        \Icirc
        \ +\ S_U\
        \Ucirc \\
        \ +\ S_K\
        \Kcirc
        \ +\ S_H\
        \Hcirc
        \ +\ S_{H^2}\
        \squarecirc
\end{multline*}
\end{prop}

\editorialnote[gray]{}
Recall from \S\ref{ssec:pivot} that the bar involution on $\Rcat$ is the involution on morphisms which switches over and under crossings
and the bar involution on $\coeffs$ is defined in \eqref{eq:bar}.
Since $S^{-1}$ is given by rotating $S$ and $S^{-1}=\overline{S}$ we have
\begin{equation*}
	\overline{S_I} = S_U \quad \overline{S_K} = S_H \quad \overline{S_{H^2}} = S_{H^2}
\end{equation*}
where the bar involution is defined in \eqref{eq:bar}.

The coefficients $S_I,S_K,S_{H^2}$ are:
\begin{align}\label{eq:coeffsS}
	S_I &= \frac{[\nv+1]}{[\nv]} (-\pp^{3}- \pp^{-1} \qq^{2} + \pp^{-1} + \pp^{-3} \qq^{2} - \pp^{-3}  - \pp^{-3}\qq^{-2} - \pp^{-5} + \pp^{-5}\qq^{-2}) \\
  S_K &= \frac 1{\qq-\qq^{-1}}  ( -\pp^{4} \qq^{2} - \pp^{2} \qq^{2} + 3\, \pp^{2} + \qq^{2} - \pp^{2}\qq^{-2} - \qq^{-2} - 2\, \pp^{-2}  + \pp^{-2}\qq^{-2} + \pp^{-4} ) \\\notag
  &=\qq+\qq^{-1}-\qq\,\pp^{2}+(2-\qq^{-2})\left[2\,\nv\right]-\qq\,\left[4\,\nv+1\right]
\\\label{eq:coeffsSb}
	S_{H^2} &= [\nv]
\end{align}

We give two proofs of Proposition~\ref{prop:skein}, Proposition~\ref{prop:jacobi} and Proposition~\ref{prop:crossing}. The two proofs differ in their initial assumptions. The first proof is in \S\ref{sec:diagrams} and uses
the values of the coefficients $\delta$,$\phi$,$\tau$ and is based on ribbon relations.
The second proof is in \S\ref{sec:representation} and uses the values of the quadratic Casimir in
\eqref{eq:cas_values} to determine a five dimensional representation of the
three string braid group. The second proof is conceptually simpler as it is based on linear algebra but
requires a computer algebra system for the calculations.

\subsection{Representation}\label{sec:representation}
We introduce six elements of $\End_\Rcat(\obj2)$ together with their names in Figure~\ref{fig:elements}.
\begin{figure}[ht]
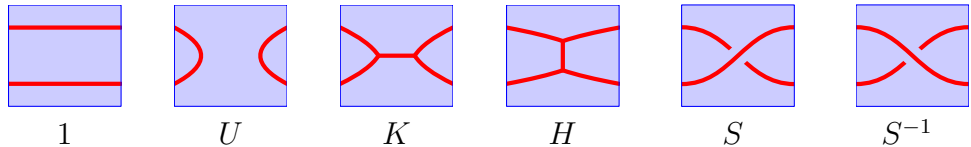

	\begin{tabular}{cccccc}
		\onegen	& \Ugen	& \Kgen	& \Hgen	& \overgen	& \undergen	 \\[5mm]
		1 & $U$ & $K$ & $H$ & $S$ & $S^{-1}$
	\end{tabular}
	\caption{Elements of two point algebra}\label{fig:elements}
\end{figure}

Our assumption is that these six elements span $A(\obj2)$. The justification is that, for each exceptional Cartan type, $C$,
the image of this set under the functor $\Psi_C$, defined in \S\ref{sec:functors}, spans $\End_{\fU_q(C)}(L_C\otimes L_C)$.

For each exceptional Cartan type, $C$, we can identify the vector spaces $\End_{\fU_{q}(C)}(L_C\otimes L_C)$ and $\Hom_{\fU_{q}(C)}(L_C,\otimes^3L_C)$
since $L_C$ is self-dual. The $\kk_{q^{1/s_C}}$-module $\Hom_{\fU_{q}(C)}(L_C,\otimes^3L_C)$ has a right action of the three string braid group, $B_3$.
In this section we construct a representation $B_3 \to M_5(\coeffs)$ where $M_5(\coeffs)$ is the $5\times 5$ matrix algebra with entries in $\coeffs$
which interpolates the actions on $\Hom_{\fU_{q}(C)}(L_C,\otimes^3L_C)$.

\begin{lemma} The table \eqref{eq:eigenvalues} interpolates the eigenvalues of the braid matrices acting on $\otimes^2L_C$ for $C$ an exceptional Cartan type.
\begin{equation}\label{eq:eigenvalues}
	\begin{array}{c|ccccc}
		V & I & L & X_2 & Y_2 & Y_2^* \\ \hline
		\lambda_V & \pp^{12} & -\pp^{6} & -1 & \qq^{-2} & \pp^2\qq^2
	\end{array}
\end{equation}
\end{lemma}
\begin{proof} For $V_C$ a composition factor of $\otimes^2L_C$, $V_C$ is the eigenspace of the braid matrix with eigenvalue
$\pm q^{\frac{d_C}{2}\left(2\casimir(L_C)-\casimir(V_C)\right)}$, see \cite[Theorem~(1.5)]{Reshetikhin1987}. It follows from \eqref{eq:cas_values} that $\theta_C(\lambda_V) = \pm q^{\frac{d_C}{2}\left(2\casimir(L_C)-\casimir(V_C)\right)}$.
\end{proof}

The irreducible five dimensional representations of the three string braid group are classified in \cite[\S4]{malle2010} and \cite{Tuba2001}.
Their isomorphism classes are determined by the eigenvalues and the action of the central element $(\sigma_1\sigma_2)^3$.
The matrix below chooses a basis for the representation with eigenvalues \eqref{eq:eigenvalues} and central scalar $\pp^{24}$.

Let $P$ be the antidiagonal matrix
\begin{equation*}
	P_{i\,j} = \begin{cases}
		1 & \text{if $i+j=6$} \\
		0 & \text{otherwise}
	\end{cases}
\end{equation*}

Define $S$ to be the matrix
\begin{equation}\label{eq:braidmatrix}
  \left(\begin{smallmatrix}
\qq^{-2}&0&0&0&0\\
\pp^{3}\qq^{-1}&-\pp^{6}&0&0&0\\
-1&\pp^{3}\qq-\pp^{3}\qq^{-1}&-1&0&0\\
\pp^{-3}\qq&-\pp^{4}-\qq^{2}+1&\pp^{5}\qq-\pp\,\qq+\pp^{-3}\qq&\pp^{2}\qq^{2}&0\\
\qq^{2}&-\pp^{7}\qq+\pp^{7}\qq^{-1}-\pp^{3}\qq^{3}+\pp^{3}\qq&\pp^{12}+\pp^{8}\qq^{2}-\pp^{8}-\pp^{4}\qq^{2}+\pp^{4}+\qq^{2}&\pp^{9}\qq-\pp^{9}\qq^{-1}+\pp^{5}\qq^{3}-\pp^{5}\qq&\pp^{12}
\end{smallmatrix}\right)
\end{equation}

The matrix $S$ has the following properties:
	\begin{itemize}
		\item $S$ has entries in $\bZ[\pp^{\pm 1},\qq^{\pm 1}]$.
		\item $S$ is lower triangular with diagonal entries the eigenvalues in \eqref{eq:eigenvalues}. (Here, a choice of ordering of the eigenvalues has to be made.)
		\item $S$ satisfies $(SP)^3=\pp^{12}$.
		\item $\overline{S} = S^{-1}$.
	\end{itemize}
	Here $\overline{S}$ is the matrix given by applying the bar involution, \eqref{eq:bar}, to each entry of $S$.
\begin{prop}\label{prop:braidrepunique}
	The matrices $S_1=S$ and $S_2=PSP$ define an absolutely irreducible representation of $B_3$ over $\bQ(\pp,\qq)$.
	Up to isomorphism, it is the unique absolutely irreducible five dimensional representation with eigenvalues \eqref{eq:eigenvalues}
	and $(S_1S_2)^3=\pp^{24}\operatorname{id}$.
\end{prop}
\begin{proof}
The identity $(SP)^3=\pp^{12}\operatorname{id}$ gives $S_1S_2S_1=S_2S_1S_2=\pp^{12}P$ and hence the stated central action.
Absolute irreducibility follows from the rank-one projector argument in the proof of Proposition~\ref{prop:intrep} below.
Since $\det S=\pp^{20}$, the central scalar selects the fifth root $\pp^4$ of $\det S$;
uniqueness up to isomorphism follows from \cite[Theorem~2.9]{Tuba2001}.
\end{proof}

\subsubsection{Algebra}
In what follows, we group the labels of composition factors as follows:
\begin{equation}\label{eq:irr}
	\begin{split}
		\Irr^0&=\left\{I\right\}
		\\
		\Irr^1&=\left\{L\right\}
		\\
		\Irr^2&=\left\{X_2,Y_2,Y^*_2\right\}
		\\
		\Irr^3&=\left\{X_3,Y_3,Y^*_3,A,C,C^*\right\}
	\end{split}
\end{equation}
and $\Irr^{\le i} = \bigsqcup_{j=0}^i \Irr^j$.

Consider $S$ as a matrix with entries in $\bQ(\pp,\qq)$. Then $S$ has a spectral decomposition and
this gives a basis of $A(\obj 2)$ of orthogonal idempotents $\pi_R$, $R\in\Irr^{\le 2}$, which we write as
\begin{equation}\label{eq:A2idem}
	S = \pp^{12}\pi_{I} - \pp^{6} \pi_{L} + \qq^{-2} \pi_{Y_2} - \pi_{X_2} + \pp^2\qq^2 \pi_{Y_2^*}
\end{equation}

Put $K=\phi\,\pi_{L}$ and $U=\delta\,\pi_{I}$ and note that these are defined over $\coeffs$.

Define $U_1=\delta\, \pi_I$ and $K_1=\phi\,\pi_L$. Put $U_2=PU_1P$ and $K_2=PK_1P$.
The element $H_1$ is determined by the conditions $H_1U_2 = K_2U_1U_2$ and $U_2H_1 = U_2U_1K_2$.
Put $H_2=PH_1P$.

Consider the $\coeffs$-algebra map $f$ from $\coeffs\left<1,U,K,H,S,S^{-1}\right>\subseteq\End_{\coeffs\Rcat}(\obj2)$ to $M_5(\coeffs)$
that sends $1,U,K,H,S,S^{-1}$ to $1,U_1,K_1,H_1,S_1,S_1^{-1}$, respectively.

\begin{defn}\label{def:A2}
The $\coeffs$-algebra $A(\obj2)$ is $\coeffs\left<1,U,K,H,S,S^{-1}\right>/\mathrm{Ker}\ f$.
\end{defn}

The algebra $A(\obj2)$ is a free $\coeffs$-module of rank 5. Note that, after setting $q=1$, $H_1$ is not an element of the subalgebra
generated by $S_1$. Therefore $H_1$ is not an element of the $\coeffs$-subalgebra generated by $S_1$.

\subsubsection{Quantum dimensions}\label{sec:qdimsid}
This representation also gives quantum dimensions. The element $U_2$ is a rank one quasi-idempotent so we have a linear map
$\varepsilon\colon A(\obj2) \to \End(\obj1)\cong\coeffs $ determined by $U_2 a_1 U_2 = \varepsilon(a_1) U_2$ for $a_1\in A(\obj2)$.
In terms of diagrams, the conditional expectation is the linear map given by
	\begin{equation}\label{fig:eps}
		\varepsilon\colon\twotworect\:\mapsto\:\tracerect
	\end{equation}
The map $\varepsilon$ on the elements in Figure~\ref{fig:elements} is given by
\begin{equation*}
\begin{tabular}{cccccc}
1 & $U$ & $K$ & $H$ & $S$ & $S^{-1}$ \\ \hline
$\delta$ & 1 & $\phi$ & 0 & $\pp^{-12}$ & $\pp^{12}$
\end{tabular}
\end{equation*}

The conditional expectation is closely related to the trace, see Definition~\ref{defn:trace};
the two are proportional, and the quantum dimension is given by
\begin{equation}
	\dim_q V = \Tr(\pi_V) = \delta\, \varepsilon(\pi_V)
\end{equation}
It is clear that these interpolate the quantum dimension for each exceptional Cartan type.

\begin{rem}
The spectral decompositions of $S_1$ and $S_2$ give two bases of $A(\obj2)$. The $6j$ symbols are the entries
of the change of basis matrix. The tetrahedron symbols are a modification which gives a symmetric matrix.
This calculation appears in \cite{Tuba2001} and in \cite{Westbury2003a}.
\end{rem}

\subsection{Diagrams}\label{sec:diagrams}
In this section we deduce the relations diagrammatically, following similar calculations in \cite{SavageBWW2024}. We take indeterminates $\delta$, $\pp$, $\phi$
and use the function field $\bQ(\delta,\pp,\phi)$. We obtain relations over $\bQ(\pp,\qq)$ by substituting the values in \eqref{eq:values} and deduce
the relations from Proposition~\ref{prop:skein} to Proposition~\ref{prop:crossing}.

We take $\bB=\{1,U,K,H,S\}$ as a basis of $A(\obj2)$ over $\bQ(\delta,\pp,\phi)$.
The multiplication table of $A(\obj2)$ with respect to the basis $\bB$ is shown in \eqref{eq:table}.

\begin{equation}\label{eq:table}
	\begin{tabular}{c|cccc}
		& $U$ & $K$ & $H$ & $S$ \\ \hline
		$U$ & $\delta\, U$ & 0 & $\phi\, U$ & $\pp^{12}\, U$ \\
		$K$ & 0 & $\phi K$ & $\tau K$ & $-\pp^6\, K$ \\
		$H$ & $\phi\, U$ & $\tau\, K$ & $H^2$ & $\pp^3\, X_v$ \\
		$S$ & $\pp^{12}\, U$ & $-\pp^6\, K$ & $\pp^3\, X_v$ & $S^2$ \\
	\end{tabular}
\end{equation}
The aim of this section is to write the structure constants of the algebra $A(\obj2)$ with respect to the basis $\bB$
as rational functions in $\delta$, $\pp$, $\phi$.

Define $a_I$ by \eqref{eq:aI},
\begin{equation*}
	a_I = \pp^{-4}\phi\left(\frac{\pp^{5}-\pp^{-5}}%
	{\pp^{-8}+\delta+\pp^{8}}\right)
\end{equation*}
Define $\tau$ by \eqref{eq:tau}
\begin{equation*}
	\tau = \frac{\pp^{4}a_I(\pp^{5}-\pp^{-5})+\phi}{\pp^{2}+\pp^{-2}}
\end{equation*}
The coefficients of the skein relation, \eqref{eq:skeinb}, are given in Proposition~\ref{skeinrelation},
	\begin{align*}
	z &= \frac{(\pp^{-6}-\pp^6)\phi+(\pp^{12}-\pp^{-12})(\phi-\tau)}{(\tau-\phi)\delta+2\phi-\tau} \\
	w &= \frac{(\pp^{-6}-\pp^6)(\delta-1)+(\pp^{12}-\pp^{-12})}{(\tau-\phi)\delta+2\phi-\tau}
\end{align*}
The relation for $X_v$ is given in Proposition~\ref{HSrelation}.

\subsubsection{Eigenvalues}
The eigenvalue relations  give
\begin{gather*}
	U\,S^{\pm 1}=\pp^{\pm 12} U = S^{\pm 1}\,U \\
	K\,S^{\pm 1}=-\pp^{\pm 6} K = S^{\pm 1}\,K
\end{gather*}

The relations in Definition~\ref{def:basicrels} give
\begin{equation*}
	U\,H=\phi\, U = H\,U \qquad K^2 = \phi\, K
\end{equation*}

\subsubsection{Skein relation}
The rotation map is defined in Definition~\ref{defn:rot}. On $A(\obj2)$ the linear map, $\Rot\colon A(\obj2)\to A(\obj2)$, is given by
\begin{equation}\label{fig:rot}
	\Rot\colon\twotworect\:\mapsto\:\rotrect
\end{equation}

The map $\Rot$ on the elements in Figure~\ref{fig:elements} is given by
\begin{equation*}
	\begin{tabular}{cccccc}
		1 & $U$ & $K$ & $H$ & $S$ & $S^{-1}$ \\ \hline
		$U$ & 1 & $H$ & $K$ & $S^{-1}$ & $S$
	\end{tabular}
\end{equation*}
Hence $\Rot$ has order two.

On the subspace spanned by $1,U,K,H$, the eigenvalues $\pm 1$ each have multiplicity 2.
On $A(\obj2)$, the eigenvalue $1$ has multiplicity 3 and the eigenvalue $-1$ has multiplicity 2.
The alternative is excluded using the bar involution, \eqref{eq:bar}.

The three elements $(S - S^{-1})$,  $(1-U)$ and $(H-K)$ are all in the $-1$ eigenspace and so
there is a linear relation
\begin{equation}\label{eq:skeinb}
	S - S^{-1} = z(1-U)+w(H-K)
\end{equation}
where the coefficients are to be determined.

\begin{prop}\label{skeinrelation}
	The coefficients in the skein relation, \eqref{eq:skeinb}, are given by
	\begin{align*}
		z &= \frac{(\pp^{-6}-\pp^6)\phi+(\pp^{12}-\pp^{-12})(\phi-\tau)}{(\tau-\phi)\delta+2\phi-\tau} \\
		w &= \frac{(\pp^{-6}-\pp^6)(\delta-1)+(\pp^{12}-\pp^{-12})}{(\tau-\phi)\delta+2\phi-\tau}
	\end{align*}
\end{prop}

\begin{proof}
	Multiplying the skein relation by $U$ and by $K$ gives
	\begin{equation*}
		\pp^{12}-\pp^{-12} = z(1-\delta) + w \phi
		\qquad
		-\pp^6+\pp^{-6} = z + w(\tau-\phi)
	\end{equation*}
	Solving this pair of equations for $z$ and $w$ gives the result.
\end{proof}

\subsubsection{Two string relations}
\begin{lemma} The following relations hold:
	\begin{equation*}
		\topover\: =-\pp^6\:\leftover\: =\:\botover\: =-\pp^6\:\rightover
	\end{equation*}
\end{lemma}

\begin{defn} Define the diagram
	\begin{equation*}
		\Xpos
	\end{equation*}
	to be any of the following

	\begin{equation*}
		\pp^{-3} \: \topover \: =-\pp^3 \: \leftover =\pp^{-3}\:\botover	\: =-\pp^3 \: \rightover
	\end{equation*}
\end{defn}

Then we define the elements $X_v$ and $X_h$ by
\begin{equation}
	\begin{tabular}{cc}
 \Xposgen	& \Xneggen \\[5mm]
 $X_v$ & $X_h$
\end{tabular}
\end{equation}
Equivalently, $X_v = \pp^{-3}S\,H = \pp^{-3}H\,S$ and $X_h = \pp^3 S^{-1}\,H = \pp^3 H\,S^{-1}$.
The rotation map extends by $\Rot(X_v)=-X_h$ and $\Rot(X_h)=-X_v$.

\begin{prop}\label{HSrelation} The following relation holds:
	\begin{equation*}
		X_v = a_I(1+\pp^{8} U+\pp^{4} S) + \pp^{-1} H
		-\pp\, K
	\end{equation*}
	where the coefficient $a_I$ is to be determined.
\end{prop}
\begin{proof}
	We have a relation
	\begin{equation}\label{eq:Xh}
		X_v = a_I + a_U U + a_K K + a_H H + a_S S
	\end{equation}
	where the coefficients are to be determined.

	Multiply \eqref{eq:Xh} by $S^{-1}$ to get
	\begin{equation*}
		\pp^{-3} H = a_I S^{-1}  + \pp^{-12} a_U U - \pp^{-6} a_K K + \pp^{-3} a_H X_h + a_S
	\end{equation*}
	Rotating gives
	\begin{equation*}
		\pp^{-3} K = a_I S  + \pp^{-12} a_U - \pp^{-6} a_K H - \pp^{-3} a_H X_v + a_S U
	\end{equation*}
	Solve for $a_HX_v$ to get
	\begin{equation}\label{eq:HS}
		a_H X_v = \pp^{-9} a_U + \pp^3\, a_S U - K -\pp^{-3} a_K H + \pp^3\, a_I S
	\end{equation}
	Multiply \eqref{eq:Xh} by $a_H$ to get a second equation for $a_H X_v$.
	Since $\bB$ is a basis, the individual coefficients can be compared.

	Comparing coefficients of $K$ gives $a_K a_H = -1$.
	Comparing coefficients of $H$ gives $a_H^2 = -\pp^{-3} a_K$.

	Solving this pair of equations gives
	\begin{equation*}
		a_K = -\pp\quad,\quad a_H = \pp^{-1}
	\end{equation*}

	The remaining three equations are
	\begin{equation*}
		\pp^{-9}a_U = \pp^{-1} a_I \quad,\quad \pp^3\,a_S = \pp^{-1} a_U  \quad,\quad \pp^{-1} a_S = \pp^3\,a_I
	\end{equation*}
	These give
	\begin{equation*}
		a_U = \pp^{8} a_I,\quad a_S = \pp^{4} a_I
	\end{equation*}
\end{proof}

\begin{cor}\label{HSinv}
	\begin{equation*}
		X_h = -a_I(\pp^{8}+U+\pp^{4} S^{-1}) - \pp^{-1} K + \pp\, H
	\end{equation*}

\end{cor}
\begin{proof}
	This is given by rotating the expression for
	$X_v$ in Proposition~\ref{HSrelation}.
\end{proof}

Next we determine the coefficients $a_I$ and $\tau$.

\begin{lemma} If $1+\pp^{8}\delta+\pp^{16}\ne 0$
	\begin{equation}\label{eq:aI}
		a_I = \pp^{-4}\phi\left(\frac{\pp^{5}-\pp^{-5}}%
		{\pp^{-8}+\delta+\pp^{8}}\right)
	\end{equation}
      \end{lemma}

\begin{proof}
	Multiplying Proposition~\ref{HSrelation} by $U$ gives
	\begin{equation*}
		\pp^9\phi = a_I(1+\pp^{8}\delta+\pp^{16}) + \pp^{-1}\phi
	\end{equation*}
	Now solve for $a_I$.
\end{proof}

\begin{lemma} If $(\pp^{2}+\pp^{-2})\ne 0$
	\begin{equation}%
		\tau = \frac{\pp^{4}a_I(\pp^{5}-\pp^{-5})+\phi}{\pp^{2}+\pp^{-2}}
	\end{equation}
\end{lemma}
\begin{proof}
	Multiplying Proposition~\ref{HSrelation} by $K$ gives
	\begin{equation}\label{eq:tau}
		-\pp^3\tau = a_I(1-\pp^{10}) - \pp\,\phi + \pp^{-1}\tau
	\end{equation}
	Solve for $\tau$ to get
	\begin{equation*}
		\tau = \frac{a_I(1-\pp^{10}) - \pp\,\phi}{-\pp^3-\pp^{-1}}
		= \frac{\pp^{4}a_I(\pp^{5}-\pp^{-5})+\phi}{\pp^{2}+\pp^{-2}}
	\end{equation*}
\end{proof}

The two coefficients $\phi$ and $\tau$ are defined up to a scalar factor.
However the ratio $\tau/\phi$ is independent of the scalar factor.
Substituting \eqref{eq:aI} in \eqref{eq:tau} gives
\begin{equation}\label{eq:ratio}
	\tau/\phi = \frac1{\pp^{2}+\pp^{-2}}\left(\frac{(\pp^{5}-\pp^{-5})^2}{\pp^{-8}+\delta+\pp^{8}}+1\right)
\end{equation}

The two entries of the multiplication table, \eqref{eq:table}, which have not been determined are $H^2$ and $S^2$.
The entry for $S^2$ can be found by multiplying the skein relation, \eqref{eq:skeinb}, by $S$ and simplifying.
The entry for $H^2$ can be found by multiplying the skein relation, \eqref{eq:skeinb}, by $H$ and simplifying.

\subsection{Ideals}\label{ssec:ideals}
The functors $\Psi_C$ preserve additional structure. The categories $\Icat(L_C)$ for $C$ an exceptional Cartan type and the category $\Rcat$
have an increasing sequence of ideals which are preserved by $\Psi_C$.

These ideals of diagram categories are defined using cut paths, \cite{Westbury1995}, \cite{kuperberg1996}, \cite{Halverson_2005}, \cite{Rubey2014}.
A \Dfn{cut path} in a diagram is a path between the two sides
of the rectangle which does not pass through any vertex of the diagram and each intersection with an edge of the diagram
is transversal. The \Dfn{cut length} (a.k.a.\ propagating number) of a diagram is the minimum number of intersections of a cut path with the diagram. Denote the cut length of a diagram $D$ by $\ell(D)$.

\begin{defn} Let $\Dcat$ be a diagram category. For $r,s\geqslant 0$ and $k\ge-1$, the subspace $J_\Dcat^{(k)}(\obj r,\obj s)\subseteq \Hom_\Dcat(\obj r,\obj s)$
	is the linear span of the set of diagrams
	\begin{equation*}
		J_\Dcat^{(k)}(\obj r,\obj s) \coloneqq \left\langle D \in \Hom_\Dcat(\obj r,\obj s) | \ell(D)\leqslant k   \right\rangle
	\end{equation*}
\end{defn}

If $D$ and $D'$ are composable diagrams then $\ell(D\circ D')\leqslant\min(\ell(D),\ell(D'))$.
It follows that, for $k\geqslant 0$, $J_\Dcat^{(k)}$ is an ideal. This defines an increasing sequence of ideals in $\Dcat$
\begin{equation}\label{eq:ideals}
	0=J_\Dcat^{(-1)}\subseteq J_\Dcat^{(0)} \subseteq J_\Dcat^{(1)} \subseteq J_\Dcat^{(2)} \subseteq \dotsb \subseteq \Dcat
\end{equation}
which satisfies
\begin{lemma}\label{lem:ideals}
	For $k,l\in\bN$,
	\begin{equation}
		J_\Dcat^{(k)} \circ J_\Dcat^{(l)} \subseteq J_\Dcat^{(\min(k,l))} \text{ and } J_\Dcat^{(k)} \otimes J_\Dcat^{(l)} \subseteq J_\Dcat^{(k+l)}
	\end{equation}
\end{lemma}

For $V$ any representation,  the ideals of  $\Icat(V)$, are defined by:
\begin{defn}
	For $\obj r,\obj s \in \bN$ and $0\leqslant k\leqslant\min(r,s)$, let $J_{\Icat(V)}^{(k)}(\obj r,\obj s)$ be the subspace of $ \Hom_{\Icat(V)}(\obj r,\obj s)$ spanned by the image of the compositions
	\[
	\Hom_{\Icat(V)} \left( \otimes^{\obj r}V, \otimes^{\obj t}V \right) \otimes \Hom_{\Icat(V)} \left( \otimes^{\obj t}V, \otimes^{\obj s}V \right)
	\to \Hom_{\Icat(V)}\left( \otimes^{\obj r}V, \otimes^{\obj s}V \right)
	\]
	for $0\leqslant t\leqslant k$.
\end{defn}

\begin{ex} The algebra $A(\obj2)$ has an increasing sequence of ideals
	\begin{equation}
		0=J^{(-1)}(\obj2,\obj2)\subseteq J^{(0)}(\obj2,\obj2) \subseteq J^{(1)}(\obj2,\obj2) \subseteq J^{(2)}(\obj 2,\obj2) = A(\obj2)
	\end{equation}
	The ideal $J^{(k)}(\obj2,\obj2)$ is defined as $\Phi^{-1}(J^{(k)}_{\coeffs\Rcat}(\obj2,\obj2)\cap \Phi(A(\obj2)))$
        where $\Phi$ is the map making $A(\obj2)$ a subquotient of $\End_{\coeffs\Rcat}(\obj2)$.

	The diagram $\Ugen$ has cut length 0, the diagram $\Kgen$ has cut length 1 and the diagrams $\onegen,\ \Hgen,\ \overgen,\ \undergen$ have cut length 2.

	Then $\left\{\Ugen\right\}$ is a basis of $J^{(0)}(\obj2,\obj2)$, $\left\{\Ugen,\ \Kgen\right\}$ is a basis of $J^{(1)}(\obj2,\obj2)$ and $\left\{\onegen,\ \Ugen,\ \Kgen,\ \Hgen,\ \overgen\right\}$ is a basis of $J^{(2)}(\obj2,\obj2)= A(\obj2)$.
\end{ex}

Define $\hat{A}(\obj2)$ to be the quotient algebra $A(\obj2)/J^{(1)}(\obj2,\obj2)$. It has a basis
\begin{equation*}
	\left\{\onegen,\ \Hgen,\ \overgen\right\}
\end{equation*}

\subsection{Trivalent diagrams}\label{ssec:A2diag}
Instead of taking $A(\obj2)$ to be the quotient of the free module with basis in Figure~\ref{fig:elements} by the skein relation, \eqref{eq:skein},
we can equivalently view $A(\obj2)$ as the free $\coeffs$-module with basis $\{1,U,K,H,H^2\}$.
(The expansion of $S$ in the latter basis is the subject of Proposition~\ref{prop:crossing}, and one can see that
the coefficients \eqref{eq:coeffsS}--\eqref{eq:coeffsSb} are in $\coeffs$.)
The multiplication table with respect to this basis is given by the table
\eqref{eq:table} (with $S$ omitted) and the relation which expresses $H^3$ in the new basis.
The latter is found as follows.

The eigenvalues of $H$, $K$ and $U$ are:
\begin{equation}\label{eq:Heigenvalues}
	\begin{array}{c|ccccc}
		\text{eigenvalues}\ \backslash\ V & I & L & X_2 & Y_2 & Y_2^* \\ \hline
		H & \phi & \tau & \xi=(\qq-\qq^{-1})^2[\nv+1] & \zeta=-\frac{[4\,\nv-2]}{[2\,\nv-1][\nv]} & \zeta^*=\frac{[6\,\nv+2][\nv+1]}{[3\,\nv+1][\nv]}
		\\
		K & 0 & \phi & 0 &0 &0
		\\
		U & \delta & 0 &0 &0 &0
	\end{array}
\end{equation}

The relation for $H$ is therefore
\begin{equation}
	\label{eq:Hcubed}
	(H-\xi)(H-\zeta)(H-\zeta^*)
	= \frac1\delta (\phi-\xi)(\phi-\zeta)(\phi-\zeta^{*})\,U
	+ \frac1\phi (\tau-\xi)(\tau-\zeta)(\tau-\zeta^{*})\,K
\end{equation}
Here $\phi$ is a unit in $\coeffs$, whereas $\delta$ is not, since $\theta_{\varnothing}(\delta)=0$.
The quotient multiplying $U$ nevertheless belongs to $\coeffs$; see \S\ref{ssec:squaresquare}
for explicit values of the coefficients in \eqref{eq:Hcubed}.

\begin{prop}\label{prop:A2}
	$A(\obj2)$ is isomorphic to the $\coeffs$-algebra with generators $H$, $K$, $U$ and relations given by \eqref{eq:Hcubed}
	and
	\begin{equation}\label{eq:A2rels}
		U^2=\delta U \quad UK=KU=0\quad K^2=\phi K\quad HU=UH=\phi U\quad HK=KH=\tau K
	\end{equation}
\end{prop}
This can be considered as an alternative definition of $A(\obj2)$.

\subsubsection{Top part of $A(\obj2)$}\label{ssec:A2top}
If we quotient $A(\obj2)$ by $J^{(1)}(\obj2,\obj2)$, we find that $\hat{A}(\obj2)$ is generated as a $\coeffs$-algebra by $H$, with sole relation
\begin{equation}\label{eq:Hcubedtop}
	(H-\xi)(H-\zeta)(H-\zeta^*)=0
\end{equation}
so that another basis of it is $\left\{\diag{diag2_2_0},\diag{diag2_2_3},\diag{diag2_2_4}\right\}$.
The skein relation \eqref{eq:skein} becomes
\begin{equation}\label{eq:skeintop}
(q-q^{-1}) (H-\xi)=
	\frac{\left[6\,\nv\right]}{\left[2\,\nv\right]\ \left[3\,\nv\right]}
	(S-S^{-1})
\end{equation}

\subsubsection{Bimodules}\label{ssec:A2bi}
For future purposes, we define, for $i=\obj0,\obj1$, left $A(\obj2)$-modules, $B(\obj2,\obj i)$ and right $A(\obj2)$-modules, $B(\obj i,\obj 2)$.
They are all free rank 1 modules over $\coeffs$.

\begin{itemize}
	\item $B(\obj0,\obj2)$, whose  generator is $\co=\caprect$.

	It is a right $A(\obj2)$-module with the obvious relations (cf Appendix~\ref{app:skein}):
	\[
	\co H = \phi \, \co \qquad \co K = 0 \qquad \co U = \delta\, \co
	\]
	\item $B(\obj2,\obj0)$, whose generator is $\ev=\cuprect$ and similar relations as left $A(\obj2)$-module.
	\item $B(\obj1,\obj2)$, whose generator is $\lambda=\diag{diag1_2_0}$.

	It is a right $A(\obj2)$-module with the obvious relations (cf Appendix~\ref{app:skein}):
	\[
	\lambda H = \tau\,\lambda \qquad \lambda K = \phi\,\lambda \qquad \lambda U = 0
	\]
	\item $B(\obj2,\obj1)$, whose generator is $\mu=\diag{diag2_1_0}$ and similar relations as left $A(\obj2)$-module.
\end{itemize}

\subsection{Gram determinant}
With respect to the ordered basis $(1,U,K,H,H^2)$ of $A(\obj2)$,
the inner product \eqref{eq:inner} has Gram matrix
\begin{equation}\label{eq:gram2}
	G(\obj2) = \delta \begin{pmatrix}
		\delta & 1 & \phi & 0 & \phi^{2} \\
		1 & \delta & 0 & \phi & \phi^{2} \\
		\phi & 0 & \phi^2 & \tau\phi & \phi\tau^2 \\
		0 & \phi & \tau\phi & \phi^2 & \phi\tau^2 \\
		\phi^{2} & \phi^{2} & \phi\tau^2 & \phi\tau^2 & \phi(b_1+b_2\tau+b_3\phi+b_4\tau^2)
	\end{pmatrix}
\end{equation}
where the coefficients $b_1,\ldots,b_4$ are given in Appendix~\ref{ssec:squaresquare}.

\begin{prop}\label{prop:Gram2}
	The determinant of the Gram matrix $G(\obj2)$ is
	\begin{multline}\label{eq:Gram2}
		\det G(\obj2) = \frac{\left[6\,\nv+2\right]\left[4\,\nv+1\right]^{2}}%
		{\left[\nv\right]^{2}\left[\nv+1\right] \left[2\,\nv+1\right]}
		\left(\frac{\qint 61 [4\,\nv]}{[\nv+1][2\,\nv]}\right)^5
		\boxed{
			[5\,\nv-1]^5
			\left[3\,\nv-1\right]^{2}
		} \\
		\times
		\left(\frac{\left[6\,\nv\right]}{\left[\nv\right]}\right)^4
		\left(\frac{\left[5\,\nv\right]}{\left[\nv\right]}\right)^3
		\left(\frac{\left[6\,\nv+2\right]}{\left[3\,\nv+1\right]}\right)^2
		\left(\frac{\left[6\,\nv+3\right]}{\left[3\,\nv+1\right]}\right)
		\left(\frac{\left[6\,\nv\right]\left[\nv\right]}{\left[3\,\nv\right]\left[2\,\nv\right]}\right)^4
		\left(\frac{\left[3\,\nv-3\right]}{\left[\nv-1\right]}\right)
		\left(\frac{\left[4\,\nv-2\right]}{\left[2\,\nv-1\right]}\right)^3
	\end{multline}
\end{prop}
All factors outside the box are units in $\coeffs$ by \S\ref{sec:spec}.
The boxed factors will be discussed in \S\ref{ssec:trunccat}.
\begin{proof}
	This is a direct calculation, using computer algebra.
\end{proof}

\subsection{Interpolation}
We introduce interpolating representations and show that \eqref{eq:braidmatrix} interpolates the representations
$\Hom_{\fU_{q}(C)}(L_C,\otimes^3L_C)$ of the three string braid group for $C$ an exceptional Cartan type.
Let $B_3$ be the three string braid group.

\begin{defn} An \Dfn{interpolating representation} is a $\coeffs$-linear representation, $\rho$, of $B_3$ with a symmetric invariant inner product
	such that:
	\begin{itemize}
		\item The underlying $\coeffs$-module is finitely generated and free.
		\item After the base change $\coeffs\to \bQ(\pp,\qq)$, the representation, $\rho\otimes \bQ(\pp,\qq)$, is irreducible and the inner product is non-degenerate.
		\item For $C$ an exceptional Cartan type, let $\rho_C=\rho\otimes_{\theta_C} \kk_{q^{1/s_C}}$ be the representation given by the base change $\theta_C\colon \coeffs\to \kk_{q^{1/s_C}}$.
		Then there is an isomorphism, preserving the inner product,
		\begin{equation*}
			\rho_C/\mathcal{N}_C \cong \Hom_{\fU_{q}(C)}(L_C,\otimes^3L_C)
		\end{equation*}
		where $\mathcal{N}_C$ is the null space of the inner product.
	\end{itemize}
\end{defn}

\begin{prop}\label{prop:intrep} The representation  of $B_3$ given in \eqref{eq:braidmatrix} is an interpolating representation.
\end{prop}

\begin{proof}
Using self-duality, view the basis $(1,U,K,H,H^2)$ of $A(\obj2)$ as diagrams with one input and three outputs.
A direct calculation identifies their braid action with \eqref{eq:braidmatrix} by a change of basis invertible over $\coeffs$;
the pairing has Gram matrix \eqref{eq:gram2}.

Over $\bQ(\pp,\qq)$, put $E=U_2/\delta$ and let $\pi_R$ be the spectral projectors of $S_1$.
By \S\ref{sec:qdimsid}, $E$ has rank one and
\[
E\pi_R E=\frac{\dim_q R}{\delta^2}E.
\]
In an eigenbasis of $S_1$, write $E=vw^t$. Then $v_Rw_R=\dim_q R/\delta^2\ne0$, so every $\pi_R E\pi_T$ is a nonzero multiple of a matrix unit.
Since $E$ is a polynomial in $S_2$, the braid matrices generate the full matrix algebra; this proves absolute irreducibility.
The inner product is non-degenerate by \eqref{eq:Gram2}.

For $C\notin\{\varnothing,A_1\}$, the five specialised eigenvalues remain distinct and $\delta$ and all five quantum dimensions remain nonzero,
with the ratios in \eqref{eq:qdims2} interpreted after cancellation. The same argument proves absolute irreducibility over $\bQ(q^{1/s_C})$.
The quantum-group multiplicity space is likewise absolutely irreducible by \cite[Lemma~3.2]{Tuba2001}.
The ribbon identity gives its full twist the scalar $\theta_C(\pp^{24})$, so \cite[Theorem~2.9]{Tuba2001} identifies the two representations over this field.

To obtain the assertion over $\kk_{q^{1/s_C}}$, use evaluation of this diagram basis.
The Gram-minor argument in the proof of Proposition~\ref{prop:A2interp} below shows that this map is surjective with kernel $\mathcal N_C$,
also for $A_1$ and $\varnothing$, where the target ranks are $3$ and $0$.
Evaluation preserves the braid action and the inner product, giving the required isomorphism.
\end{proof}

\begin{defn}\label{def:interpalgebra} For $k\geqslant 0$, an \Dfn{interpolating algebra} is a $\coeffs$-linear algebra, $A(k)$, such that:
	\begin{itemize}
		\item The underlying $\coeffs$-module is finitely generated and free.
		\item After the base change $\coeffs\to \bQ(\pp,\qq)$, the algebra, $A(k)\otimes \bQ(\pp,\qq)$, is split semisimple (a.k.a.\ a multimatrix algebra) and the inner product is non-degenerate.
		\item For $C$ an exceptional Cartan type, let $A_C(k)=A(k)\otimes_{\theta_C} \kk_{q^{1/s_C}}$ be the $\kk_{q^{1/s_C}}$-algebra given by the base change $\theta_C\colon \coeffs\to \kk_{q^{1/s_C}}$.
		Then there is an isomorphism, preserving the ideals and the trace map,
		\begin{equation*}
			A_C(k)/\mathcal{N}_C \cong \End_{\fU_{q}(C)}(\otimes^k L_C)
		\end{equation*}
		where $\mathcal{N}_C$ is the null space	of the inner product.
	\end{itemize}
\end{defn}

\begin{prop}\label{prop:A2interp}
  The algebra $A(\obj2)$ is an interpolating algebra.
\end{prop}
\begin{proof}
The algebra is free of rank five over $\coeffs$; over $\bQ(\pp,\qq)$ it has a basis of orthogonal idempotents and non-degenerate trace form.
Diagram evaluation gives a trace-preserving algebra homomorphism
$\theta_C(2)\colon A_C(2)\to \End_{\fU_{q}(C)}(\otimes^2 L_C)$.
For $C\notin\{\varnothing,A_1\}$, the target is free of rank five and \eqref{eq:Gram2} specialises to a unit.
For $C=A_1$, the target has rank three, and the leading $3\times3$ principal minor of \eqref{eq:gram2} is
\[
\left.\delta^3\phi^2(\delta^2-\delta-1)\right|_{A_1}
=\theta_{A_1}(\delta^3\phi^2)[5],
\]
a unit since $\theta_{A_1}(\delta)=[3]$ and $\phi$ is a unit.
A unit Gram minor of size equal to the target rank implies that the corresponding image vectors form a basis.
Thus $\theta_C(2)$ is surjective, and nondegeneracy of the target trace pairing identifies its kernel with $\mathcal N_C$.
For $C=\varnothing$, both the target and the specialised Gram matrix are zero.
Finally, the images of $U$ and of $\{U,K\}$ generate the corresponding cut-length zero and one ideals, so the isomorphism preserves the ideal filtration.
\end{proof}



\section{Three string relations}\label{sec:threestring}
In this section we construct a $\coeffs$-algebra, $A(\obj3)$, which is a free $\coeffs$-module of rank 80.
This algebra is constructed to interpolate the algebras $\End_{\Icat(L_C)}(\otimes^3 L_C)$ for $C$ in
the exceptional series.
We also construct an $A(\obj2)$-$A(\obj3)$ bimodule $B(\obj2,\obj3)$ which is a free $\coeffs$-module of rank 16.
This bimodule is constructed to interpolate the bimodules $\Hom_{\Icat(L_C)}(\otimes^2 L_C,\otimes^3 L_C)$.
All algebras and bimodules can be viewed as subquotients of the corresponding $\Hom$ spaces of the category $\coeffs \Rcat$,
and we define them using diagrams.

\subsection{\texorpdfstring{The bimodules $B(\obj i,\obj3)$ and $B(\obj3,\obj i)$}{The bimodules B(i,3) and B(3,i)}}\label{ssec:bimodules}

Before working with the algebra $A(\obj3)$, we need to consider the interpolating bimodules $B(\obj i,\obj3)$ and $B(\obj3,\obj i)$ for $i=\obj0,\obj1,\obj2$.

The bimodules $B(\obj0,\obj3)$ and $B(\obj3,\obj0)$ are rank 1 free $\coeffs$-modules whose generators are the following diagrams:
\begin{align*}
  B(\obj0,\obj3)&=\coeffs \diag{diag0_3_0}
  \\
  B(\obj3,\obj0)&=\coeffs \diag{diag3_0_0}
\end{align*}

The bimodules $B(\obj1,\obj3)$ and $B(\obj3,\obj1)$ are rank 5 free $\coeffs$-modules
whose generators are the following diagrams:
\begin{align*}
  B(\obj1,\obj3)&=\coeffs\text{-span of }\left\{\foreach \x in {3,2,4,1,0} { \diag{diag1_3_\x}\ifnum\x=0\else,\ \fi }\right\}
  \\
  B(\obj3,\obj1)&=\coeffs\text{-span of }\left\{\foreach \x in {3,2,4,1,0} { \diag{diag3_1_\x}\ifnum\x=0\else,\ \fi }\right\}
\end{align*}

They have already appeared implicitly in \S\ref{sec:representation}; the connection will be made explicitly in \S\ref{ssec:braidA3}.

We now discuss the more interesting case of $B(\obj2,\obj3)$ ($B(\obj3,\obj2)$ will be treated similarly).
We learn from the Bratteli diagram that $B(\obj2,\obj3)$ should contain an increasing sequence of bimodules
\begin{equation}
	0\subset J^{(0)}(\obj2,\obj3) \subset J^{(1)}(\obj2,\obj3) \subset J^{(2)}(\obj2,\obj3)
\end{equation}
for $0\leqslant i\leqslant 2$,
such that the generic dimensions of the successive quotient bimodules $J^{(i)}(\obj2,\obj3)/J^{(i-1)}(\obj2,\obj3)$ are $1$, $5$, $10$, and we expect that they are spanned by diagrams of cut length $i$.

We shall build $B(\obj2,\obj3)$ diagrammatically.
The main new result here is the ``square-pentagon relation'' which is also discussed in \cite{morrison2026}.

Consider the following set, $\mathcal B(\obj5)$, of $16$ diagrams:
\begin{align*}
  &\foreach \x in {0,...,7} { \diag{round5_\x},\ }
\\
  &\foreach \x [count=\i] in {8,...,15} { \diag{round5_\x} \ifnum\i=8\else,\ \fi }
\end{align*}

A simple argument to convince ourselves that these diagrams should be considered linearly independent
is to compute their Gram matrix; using the diagrammatic relations
that we already have at our disposal (see Appendix~\ref{ssec:basicrels}--\ref{ssec:squaresquare}),
this can be performed, and we only provide its determinant:

\begin{multline}\label{eq:Gram23}
\frac{\left[6\nv\right]^{42}\left[5\nv\right]^{13}\left[4\nv\right]^{16}\left[6\nv+2\right]^{16}\left[6\nv+4\right]\left[4\nv+1\right]^{10}\left[6\nv+3\right]^{6}\left[6\nv+1\right]^{16}}{\left[\nv\right]^{29}\left[3\nv\right]^{26}\left[2\nv\right]^{36}\left[\nv+1\right]^{16}\left[3\nv+1\right]^{16}\left[3\nv+2\right]\ \left[2\nv+1\right]^{6}}
\\
\left(\frac{\left[4\nv-2\right]}{ \left[2\nv-1\right]}\right)^{16}
\left(\frac{\left[3\nv-3\right]}{ \left[\nv-1\right]}\right)^6
\frac{\left[2\nv-4\right]}{ \left[\nv-2\right]}
\boxed{
  \left[5\nv-1\right]^{16}
  \left[3\nv-1\right]^{10}
  }
\end{multline}

The expressions outside the box are units in $\coeffs$ by \S\ref{sec:spec}, with ratios interpreted after cancellation.
The boxed factors will be discussed in \S\ref{ssec:trunccat}.

We therefore declare $B(\obj2,\obj3)$ (resp.\ $B(\obj3,\obj2)$) to be the free module over $\coeffs$ with basis $\mathcal B(\obj2,\obj3)$ (resp.\ $\mathcal B(\obj3,\obj2)$), which is the set of diagrams obtained from those
of $\mathcal B(\obj5)$ by interpreting the endpoints as being
two (resp.\ three) left boundary points and three (resp.\ two) right boundary points.
For $B(\obj2,\obj3)$ we obtain the following $1+5+10$ diagrams:
\begin{align*}
  &\diag{diag2_3_0},\\
  &\foreach \x in {2,3,6,9,10} {\diag{diag2_3_\x},\ }\\
  &\foreach \x in {1,4,5,7,8,11,12,13,14,15} {\diag{diag2_3_\x}\ifnum\x=15\else,\fi}
\end{align*}

Of course these are not the only diagrams with 5 external legs one can think of. The simplest example that is not in this list is
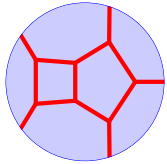
\[
\diag[.5]{squarepenta}
\]
and all its rotations. Again, one can compute the Gram matrix of the 16 basis diagrams plus the one above: we obtain a degenerate matrix.
The trace pairing on the corresponding Hom-space in each module category is nondegenerate, so a null vector of this Gram matrix gives a linear relation among the images of the 17 diagrams;
we therefore impose the same relation in our interpolating category, namely
\begin{equation}\label{eq:squarepenta}
\diag{squarepenta} = \sum_{k=0}^{15} c_k\, \mathcal B(\obj5)_{k}
\end{equation}
where the coefficients $c_k\in \coeffs$ are listed in Appendix~\ref{ssec:squarepenta},
and the elements of $\mathcal B(\obj5)$ are numbered from $0$ to $15$ in the same order as listed above.
Note that this relation can be freely rotated since our basis $\mathcal B(\obj5)$ is invariant under rotation.

We expect $A(\obj2)$ to act on $B(\obj2,\obj3)$ by left concatenation of diagrams, and indeed
\begin{prop}\label{prop:A2diag5}
$B(\obj2,\obj3)$ is a left $A(\obj2)$-module.
$B(\obj3,\obj2)$ is a right $A(\obj2)$-module.
\end{prop}
\begin{proof}
  For the purposes of this Proposition, our bimodules must be viewed as subquotients of the corresponding $\Hom_{\coeffs\Rcat}$ spaces;
  in other words, the action is defined in the natural diagrammatic way:
  we concatenate the diagrams, and then simplify using the known rules including the new square-pentagon relation.

  Let us choose $B(\obj2,\obj3)$ ($B(\obj3,\obj2)$ is treated similarly).
  We need to check that upon multiplication by any of the generators of $A(\obj2)$, namely $H$, $K$ and $U$,
  the result is a linear combination of the 16 basis diagrams.

  This can be proved by direct inspection. In fact, in all but three cases,
  only the ``old'' rules of Appendix~\ref{ssec:basicrels}--\ref{ssec:squaresquare} are needed;
  the three exceptions are
  \begin{align*}
    \diag[0.25]{diag2_2_3} \hspace{-2mm} \diag{diag2_3_11}
    \\
    \diag[0.25]{diag2_2_3} \hspace{-2mm} \diag{diag2_3_14}
    \\
    \diag[0.25]{diag2_2_3} \hspace{-2mm} \diag{diag2_3_15}
  \end{align*}
  In each case, we obtain a particular rotation of the square pentagon, which can therefore be expanded as a linear combination
  of our 16 basis diagrams.
\end{proof}

Once we have defined $A(\obj3)$, we shall also be able to prove that
$B(\obj2,\obj3)$ and $B(\obj3,\obj2)$ are right/left $A(\obj3)$-modules (Proposition~\ref{prop:A3bimodules});
we postpone until then the description of the representation content of $B(\obj2,\obj3)$ and $B(\obj3,\obj2)$.

\subsection{\texorpdfstring{Construction of $\hat{A}(\obj3)$}{Construction of hat A(3)}}\label{ssec:A3top}
As a warm-up, we shall first define the top part $\hat{A}(\obj3)$. For the purposes of this section,
$\hat{A}(\obj3)$ is a standalone algebra -- we shall show afterwards that it is indeed the quotient of the larger
algebra $A(\obj3)$ by the ideal $J^{(2)}(\obj3,\obj3)$ generated by diagrams of cut length $\leqslant 2$.

In the same way that $\hat{A}(\obj2)=A(\obj2)/J^{(1)}(\obj2,\obj2)$ is generated by $H$ (cf \S\ref{ssec:A2top}),
we expect $\hat{A}(\obj3)$ to be generated by
\begin{align*}
H_1&:=H\otimes 1 = \diag{H1}
\\
H_2&:=1\otimes H = \diag{H2}
\end{align*}
Each $H_i$, $i=1,2$, must satisfy the same cubic equation \eqref{eq:Hcubedtop}
as $H$. Furthermore, we must take into account the square-pentagon
relation \eqref{eq:squarepenta}.

A typical example is the word $H_2^2H_1H_2=\diag{H22H1H2}$, which contains a square and a pentagon next to each other, and is therefore amenable to the rule
\eqref{eq:squarepenta}. However, substituting the r.h.s.\ of \eqref{eq:squarepenta} into
$H_2^2H_1H_2$ results in a linear combination of diagrams that includes the diagram
\diag{A3notword_7}; the latter is a valid would-be element of $\hat{A}(\obj3)$ but is {\em not}\/
a word in $H_1$, $H_2$. Note that
\diag{A3notword_7} is invariant under the symmetries of the rectangle, thus appears in the expansion
of $H_2^2H_1H_2$, $H_2H_1H_2^2$, $H_1^2H_2H_1$, $H_1H_2H_1^2$ with the {\em same}\/ coefficient. We can therefore eliminate it by taking differences.

And so, applying \eqref{eq:squarepenta} to $H_2^2H_1H_2-H_2H_1H_2^2$
and keeping only diagrams of cut length $3$ results in a simple combination of words in $H_1$ and $H_2$, namely
\[
\diag{H22H1H2}-\diag{H2H1H22} = \xi\left(\diag{H22H1}-\diag{H1H22}\right)+\xi^2\left(\diag{H1H2}-\diag{H2H1}\right)
\]
where $\xi=\left(\qq-\qq^{-1}\right)\left(\pp\qq-\pp^{-1}\qq^{-1}\right)$.

\editorialnote[gray]{}

The exact same relation holds by applying a vertical mirror symmetry, which amounts to exchanging $H_1$ and $H_2$.
A similar but slightly different relation is
\begin{multline*}
\diag{H2H1H22}-\diag{H12H2H1}=c_{15}\left(\diag{H2H1H2}-\diag{H1H2H1}\right)+c_{11}\left(\diag{H22H1}-\diag{H2H12}\right)
\\
+c_{12}\left(\diag{H1H22}-\diag{H12H2}\right)+c_1\left(\diag{H1}-\diag{H2}\right)
\end{multline*}
where the coefficients are given in Appendix~\ref{ssec:squarepenta}.
There are more relations of the same kind, but they are obviously linear combinations of the previous three.

These considerations justify the following definition:
\begin{defn}
$\hat{A}(\obj3)$ is the $\coeffs$-algebra with generators $H_1$, $H_2$ and relations
\begin{gather}\label{eq:Hrels}
(H_i-\xi)(H_i-\zeta)(H_i-\zeta^*)
=0\qquad i=1,2
\\\label{eq:Hrelsa}
H_2^2H_1H_2
-
H_2H_1H_2^2
=\xi(
H_2^2H_1
-
H_1H_2^2
)
+\xi^2(
H_1H_2
-
H_2H_1
)
\\\label{eq:Hrelsaa}
H_1H_2H_1^2
-
H_1^2H_2H_1
=\xi(
H_2H_1^2
-
H_1^2H_2
)
+\xi^2(
H_1H_2
-
H_2H_1
)
\\
\label{eq:Hrelsb}
\begin{aligned}
  H_2^2H_1H_2-H_1^2H_2H_1
&= c_{15}(H_{2}H_{1}H_{2}-H_{1}H_{2}H_{1})
+c_{11}(H_{1}H_{2}^2-H_{2}H_{1}^2)\\
&\quad+c_{12}(H_{2}^2H_{1}-H_{1}^2H_{2})
+(c_{5}-c_{7})(H_{1}H_{2}-H_{2}H_{1})\\
&\quad+c_{1}(H_{1}-H_{2}).
\end{aligned}
\end{gather}
where
$\xi=\left(\qq-\qq^{-1}\right)\left(\pp\qq-\pp^{-1}\qq^{-1}\right)$,
$\zeta=-\frac{\left[4\,\nv-2\right]}{\left[\nv\right]\left[2\,\nv-1\right]}$,
$\zeta^*=\frac{\left[\nv+1\right]\ \left[6\,\nv+2\right]}{\left[\nv\right]\ \left[3\,\nv+1\right]}$.
\end{defn}

\editorialnote[gray]{}

\begin{lemma}\label{lem:A3toprep}
  The following matrices define representations of $\hat{A}(\obj3)$ over $\coeffs$
  which become irreducible and pairwise nonisomorphic over $\bQ(\pp,\qq)$:%
  \def\ha{\zeta}\def\hb{\zeta^*}
  \begin{align}\label{eq:A3irrep}
\rho_{\obj3,X_3}(H_1)&=\begin{pmatrix}\xi\end{pmatrix} & \rho_{\obj3,X_3}(H_2)&=\begin{pmatrix}\xi\end{pmatrix}
\\
\rho_{\obj3,Y_3}(H_1)&=\begin{pmatrix}\ha\end{pmatrix} & \rho_{\obj3,Y_3}(H_2)&=\begin{pmatrix}\ha\end{pmatrix}
\\
\rho_{\obj3,Y_3^*}(H_1)&=\begin{pmatrix}\hb\end{pmatrix} & \rho_{\obj3,Y_3^*}(H_2)&=\begin{pmatrix}\hb\end{pmatrix}
\\[1mm]
\rho_{\obj3,A}(H_1)&=\begin{pmatrix}\ha&0&0\\-\frac{\left[6\,\nv\right]}{\left[2\,\nv\right]\ \left[3\,\nv\right]}&\xi&0\\0&-\frac{\left[6\,\nv\right]\ \left[\nv+1\right]}{\left[3\,\nv\right]\ \left[\nv\right]}&\hb\end{pmatrix}\!\!\! & \rho_{\obj3,A}(H_2)&=\begin{pmatrix}\hb&-\frac{\left[6\,\nv\right]\ \left[\nv+1\right]}{\left[3\,\nv\right]\ \left[\nv\right]}&0\\0&\xi&-\frac{\left[6\,\nv\right]}{\left[2\,\nv\right]\ \left[3\,\nv\right]}\\0&0&\ha\end{pmatrix}\!\!\!
\\[1mm]
\rho_{\obj3,C}(H_1)&=\begin{pmatrix}\ha&0\\-\frac{\left[6\,\nv\right]}{\left[2\,\nv\right]\ \left[3\,\nv\right]}&\xi\end{pmatrix} & \rho_{\obj3,C}(H_2)&=\begin{pmatrix}\xi&-\frac{\left[6\,\nv\right]}{\left[2\,\nv\right]\ \left[3\,\nv\right]}\\0&\ha\end{pmatrix}
\\\label{eq:A3irrepb}
\rho_{\obj3,C^*}(H_1)&=\begin{pmatrix}\hb&0\\\frac{\left[6\,\nv\right]\ \left[\nv+1\right]}{\left[2\,\nv\right]\ \left[3\,\nv\right]}&\xi\end{pmatrix} & \rho_{\obj3,C^*}(H_2)&=\begin{pmatrix}\xi&\frac{\left[6\,\nv\right]\ \left[\nv+1\right]}{\left[2\,\nv\right]\ \left[3\,\nv\right]}\\0&\hb\end{pmatrix}
\end{align}
\end{lemma}
\begin{proof}
  One checks relations \eqref{eq:Hrels}--\eqref{eq:Hrelsb} in each representation by direct computation.

  We prove irreducibility over the fraction field $\bQ(\pp,\qq)$. It follows from the following two facts, which both amount to solving a linear system.

  Firstly, by direct calculation, one observes that given $R\in\Irr^3=\left\{X_3,Y_3,Y^*_3,A,C,C^*\right\}$,
  any matrix that commutes with both $\rho_R(H_1)$ and
  $\rho_R(H_2)$ is proportional to the identity. (Indeed, such a matrix must be lower triangular
  because $\rho_R(H_1)$ is lower triangular with distinct diagonal entries, and similarly must be upper triangular
  because of $\rho_R(H_2)$, so must be diagonal; and since the entries of $\rho_R(H_1)$ right below the diagonal
  are nonzero, commuting with it forces this diagonal matrix to be proportional to the identity.)

  Secondly, note that $\hat{A}(\obj3)$ possesses an
  anti-involution that leaves $H_1$ and $H_2$ invariant. This means that given any module $V$, its dual
  $V^*$ is also a module.
  In the case of each of the modules $V_R$ given by the matrices above,
  one can show that the module $V_R^*$ is isomorphic to $V_R$, and the
  corresponding nondegenerate bilinear form on $V_R$ is symmetric and can be chosen positive definite at $\pp=\qq=2$, hence on a nonempty real open neighbourhood.
  It is therefore anisotropic over $\bQ(\pp,\qq)$, since a nonzero rational vector is regular and nonzero somewhere in that neighbourhood.
 For example, for $V_A$, the symmetric matrix $M=\left(\!\begin{smallmatrix}
\frac{\left[\nv\right]\ \left[2\,\nv+2\right]}{\left[2\,\nv\right]\ \left[\nv+1\right]}&-1&\frac{\left[\nv\right]}{\left[\nv+2\right]\ \left[2\,\nv\right]}\\
-1&\frac{\left[2\right]\ \left[\nv+2\right]+\left[\nv+1\right]}{\left[\nv+2\right]}&-1\\
\frac{\left[\nv\right]}{\left[\nv+2\right]\ \left[2\,\nv\right]}&-1&\frac{\left[\nv\right]\ \left[2\,\nv+2\right]}{\left[2\,\nv\right]\ \left[\nv+1\right]}
  \end{smallmatrix}\!\right)$ satisfies $M \rho_{\obj3,A}(H_i)=\rho_{\obj3,A}(H_i)^T M$, $i=1,2$, and $M=P^T D P$
with $D$ the diagonal matrix with diagonal entries
\begin{equation*}
\left\{\frac{\left[\nv\right]\ \left[2\,\nv+2\right]}{\left[2\,\nv\right]\ \left[\nv+1\right]},\,\frac{\left[2\right]\ \left[2\,\nv+3\right]\ \left[\nv+1\right]}{\left[\nv+2\right]\ \left[2\,\nv+2\right]},\,\frac{\left[\nv+3\right]\ \left[2\,\nv+3\right]\ \left[\nv\right]}{\left[2\right]\ \left[2\,\nv\right]}\right\}.
\end{equation*}
At this stage the reasoning is standard: any submodule $W\subseteq V_R$ has a complementary
submodule $W^\perp$ w.r.t.\ the anisotropic quadratic form, the projector on the first summand of $W\oplus W^\perp$ commutes with $\rho_R(H_1)$
and $\rho_R(H_2)$ and is therefore proportional to the identity, hence $W=0$ or $W=V_R$, i.e., $V_R$ is irreducible.

Finally, the $V_R$ are clearly pairwise distinct since the eigenvalues of, say $H_1$, differ.
\end{proof}

Define $\hat{A}(\obj3)_{\mathrm{loc}}:=\hat{A}(\obj3)  \otimes_{\coeffs} \bQ(\pp,\qq)$:
\begin{cor}\label{cor:dimA3}
  $\hat{A}(\obj3)_{\mathrm{loc}}$ is an algebra of dimension at least $20$.
\end{cor}
\begin{proof}
By summing the squares of sizes of the matrices in the Lemma above, we obtain the statement.
\end{proof}

One can now conclude
\begin{prop}\label{prop:A3top}
  $\hat{A}(\obj3)$ is a free $\coeffs$-module of rank $20$, with basis the set
  $\mathcal B^{(3)}_{\mathrm{orig}}(\obj3,\obj3)$ consisting of
\begin{align*}
	1 &= \diag{1} & H_1 &= \diag{H1} & H_2 &= \diag{H2} & H_1^2 &= \diag{H12} \\
	H_1H_2 &= \diag{H1H2} & H_2H_1 &= \diag{H2H1} & H_2^2 &= \diag{H22} & H_1^2H_2 &= \diag{H12H2} \\
	H_1H_2H_1 &= \diag{H1H2H1} & H_1H_2^2 &= \diag{H1H22} & H_2H_1^2 &= \diag{H2H12} & H_2H_1H_2 &= \diag{H2H1H2} \\
	H_2^2H_1 &= \diag{H22H1} & H_1^2H_2H_1 &= \diag{H12H2H1} & H_1^2H_2^2 &= \diag{H12H22} & H_1H_2H_1H_2 &= \diag{H1H2H1H2} \\
	H_1H_2^2H_1 &= \diag{H1H22H1} & H_2H_1H_2H_1 &= \diag{H2H1H2H1} & H_2^2H_1^2 &= \diag{H22H12} & H_1H_2H_1H_2H_1 &= \diag{H1H2H1H2H1}
\end{align*}
\end{prop}
\begin{proof}
Consider $\coeffs\left< H_1,H_2\right>$ with the graded lexicographic monomial order, with the convention $H_1<H_2$.
In all that follows we only write the leading term of each expression, for brevity.
We start from the five relations \eqref{eq:Hrels}--\eqref{eq:Hrelsb}, taking $g_5$ to be the difference of \eqref{eq:Hrelsb} and \eqref{eq:Hrelsa}, with all terms brought to the left:
\begin{align*}
g_1&=H_1^3+\cdots
\\
g_2&=H_2^3+\cdots
\\
g_3&=H_2^2H_1H_2+\cdots
\\
g_4&=H_1H_2H_1^2+\cdots
\\
g_5&=H_2H_1H_2^2+\cdots
\end{align*}
We can then create new ones by {\em overlap} (see \cite{Sikora2007}).

It may be instructive to do an example of such a procedure diagrammatically.
Consider the following diagram:
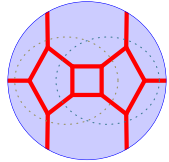
\[
\diag[.5]{overlap-squarehexa-mod}
\]
One can apply the square-pentagon relation \eqref{eq:squarepenta} to either left or right of the diagram.
This leads to a nontrivial diagrammatic identity where
no square-pentagons are involved, showing the subtlety in dealing with such diagrammatic rules. This ``new'' rule involves adjacent square-hexagons,
though it does not allow to remove them entirely, since they always appear in differences; in particular, in $A(\obj3)/J^{(2)}(\obj3,\obj3)$ (i.e., keeping only diagrams of cut length $3$), we expect an identity of the form
\begin{multline*}
\diag{H2H12H2}-\diag{H1H22H1}=d_0\left(\diag{H1H2H1}-\diag{H2H1H2}\right)+d_1\left(\diag{H22H1}+\diag{H1H22}-\diag{H12H2}-\diag{H2H12}\right)
\\
+d_2\left(\diag{H12}-\diag{H22}\right)+d_3\left(\diag{H1}-\diag{H2}\right)
\end{multline*}
where
\begin{align*}
d_0&=
2\frac{[\nv+1]}{\left[\nv\right]^2}
-\frac{\left[6\,\nv\right]}{\left[3\,\nv\right]}
\\
d_1&=\frac{[\nv+1]}{\left[\nv\right]^2}
\\
d_2&=-\frac{\left[6\,\nv\right]\ [\nv+1]}{\left[\nv\right]^2 \left[3\,\nv\right]}
\\
d_3&=-\frac{\left[6\,\nv\right]\ [\nv+1]^2}{\left[\nv\right]^4 \left[3\,\nv\right]}
\end{align*}
\editorialnote[gray]{}

One can search more systematically for overlaps in the relations of $\hat A(\obj3)$ (this is best performed by computer), and the result is as follows:
\begin{align*}
  R_{g_1,\ldots,g_5}(H_2H_1g_2-g_5H_2)&= g_6 & g_6&=H_1^2H_2H_1H_2+\cdots
  \\
  R_{g_1,\ldots,g_6}(H_1g_6-g_1H_2H_1H_2)&=\left(\frac{[6\,\nv][\nv+1]}{[\nv][2\,\nv][3\,\nv]}\right)^2 g_7 & g_7&=H_2H_1^2H_2+\cdots
  \\
  R_{g_1,\ldots,g_7}(g_3H_2 - H_2 g_5)&= g_8 & g_8&=H_1H_2^2H_1^2+\cdots
  \\
  R_{g_1,\ldots,g_8}(g_4H_2-H_1g_7) &=   g_9 & g_9&=H_1^2H_2^2H_1+\cdots
  \\
R_{g_1,\ldots,g_9}(H_1H_2g_7-g_8H_2) &= \frac{\left[6\,\nv\right]\ \left[2\,\nv\right]\ [\nv+1]}{\left[3\,\nv\right]\ \left[\nv\right]^{3}}
 g_{10}  & g_{10}&=H_2H_1H_2H_1H_2+\cdots
\end{align*}
where $R_{g_1,\ldots,g_k}$ denotes remainder by the successive division by $g_1,\ldots,g_k$.
$g_7$ coincides with the diagrammatic identity above.
Note that the prefactors of $g_7$ and $g_{10}$ are invertible in $\coeffs$,
so all the coefficients of $g_6,\ldots,g_{10}$ live in $\coeffs$.
$\{g_1,\ldots,g_{10}\}$ is in fact a minimal Gr\"obner basis for the ideal generated by $g_1,\ldots,g_5$. \editorialnote[gray]{}

Now consider the words as in Proposition~\ref{prop:A3top}:
\foreach\s/\t [count=\i] in
{1/1,H_1/H1,H_2/H2,H_1^2/H12,H_1H_2/H1H2,H_2H_1/H2H1,H_2^2/H22,H_1^2H_2/H12H2,H_1H_2H_1/H1H2H1,H_1H_2^2/H1H22,H_2H_1^2/H2H12,H_2H_1H_2/H2H1H2,H_2^2H_1/H22H1,H_1^2H_2H_1/H12H2H1,H_1^2H_2^2/H12H22,H_1H_2H_1H_2/H1H2H1H2,H_1H_2^2H_1/H1H22H1,H_2H_1H_2H_1/H2H1H2H1,H_2^2H_1^2/H22H12,H_1H_2H_1H_2H_1/H1H2H1H2H1}
{ $\s$\ifnum\i=20.\else,\fi }
\tikzset{every picture/.style={}}
Multiplying by either $H_1$ or $H_2$ (say, on the right) any of the 20 words above results in either
another word in the list, or a multiple of the leading term of one of the equations above:
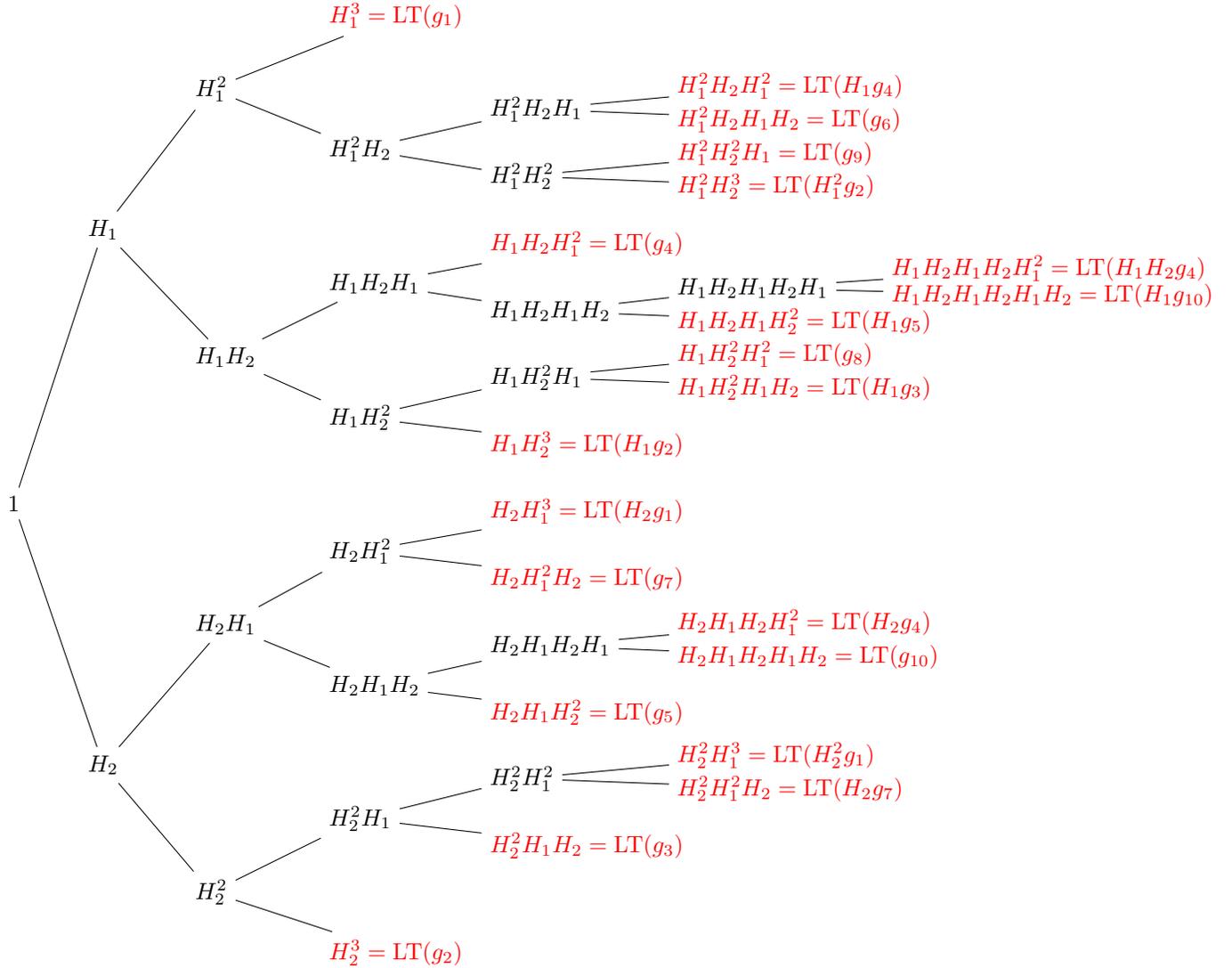
\newcommand\LT[1]{\mathrm{LT}(#1)}
\begin{center}
  \begin{tikzpicture}[scale=.9,transform shape,
    every node/.style={anchor=base west,node font=\footnotesize},
    growth parent anchor = west,
    level 1/.style={sibling distance=80mm,level distance=12mm},
    level 2/.style={sibling distance=40mm,level distance=16mm},
    level 3/.style={sibling distance=20mm,level distance=20mm},
    level 4/.style={sibling distance=10mm,level distance=24mm},
    level 5/.style={sibling distance=5mm,level distance=28mm},
    level 6/.style={sibling distance=4mm,level distance=32mm}]
    \node {1} [grow'=right]
    child {
      node{$H_1$}
      child {
        node {$H_1^2$}
        child {
          node[red] {$H_1^3=\LT{g_1}$}
        }
        child {
          node {$H_1^2H_2$}
          child {
            node {$H_1^2H_2H_1$}
            child {
              node[red] {$H_1^2H_2H_1^2=\LT{H_1g_4}$}
            }
            child {
              node[red] {$H_1^2H_2H_1H_2=\LT{g_6}$}
            }
          }
          child {
            node {$H_1^2H_2^2$}
            child {
              node[red] {$H_1^2H_2^2H_1=\LT{g_9}$}
            }
            child {
              node[red] {$H_1^2H_2^3=\LT{H_1^2g_2}$}
            }
          }
        }
      }
      child {
        node {$H_1H_2$}
        child {
          node {$H_1H_2H_1$}
          child {
            node[red] {$H_1H_2H_1^2=\LT{g_4}$}
          }
          child {
            node {$H_1H_2H_1H_2$}
            child {
              node {$H_1H_2H_1H_2H_1$}
              child {
                node[red] {$H_1H_2H_1H_2H_1^2=\LT{H_1H_2g_4}$}
              }
              child {
                node[red] {$H_1H_2H_1H_2H_1H_2=\LT{H_1g_{10}}$}
              }
            }
            child {
              node[red] {$H_1H_2H_1H_2^2=\LT{H_1g_5}$}
            }
          }
        }
        child {
          node {$H_1H_2^2$}
          child {
            node {$H_1H_2^2H_1$}
            child {
              node[red] {$H_1H_2^2H_1^2=\LT{g_{8}}$}
            }
            child {
              node[red] {$H_1H_2^2H_1H_2=\LT{H_1g_3}$}
            }
          }
          child {
            node[red] {$H_1H_2^3=\LT{H_1g_2}$}
          }
        }
      }
    }
    child {
      node{$H_2$}
      child {
        node {$H_2H_1$}
        child {
          node {$H_2H_1^2$}
          child {
            node[red] {$H_2H_1^3=\LT{H_2g_1}$}
          }
          child {
            node[red] {$H_2H_1^2H_2=\LT{g_7}$}
          }
        }
        child {
          node {$H_2H_1H_2$}
          child {
            node {$H_2H_1H_2H_1$}
            child {
              node[red] {$H_2H_1H_2H_1^2=\LT{H_2g_4}$}
            }
            child {
              node[red] {$H_2H_1H_2H_1H_2=\LT{g_{10}}$}
            }
          }
          child {
            node[red] {$H_2H_1H_2^2=\LT{g_5}$}
          }
        }
      }
      child {
        node {$H_2^2$}
        child {
          node {$H_2^2H_1$}
          child {
            node {$H_2^2H_1^2$}
            child {
              node[red] {$H_2^2H_1^3=\LT{H_2^2g_1}$}
            }
            child {
              node[red] {$H_2^2H_1^2H_2=\LT{H_2g_7}$}
            }
          }
          child {
            node[red] {$H_2^2H_1H_2=\LT{g_3}$}
          }
        }
        child {
          node[red] {$H_2^3=\LT{g_2}$}
        }
      }
    };
  \end{tikzpicture}
\end{center}
One concludes by induction on the monomial order
that the span of these words is stable by multiplication by $H_1$, $H_2$
on the right.
Since it contains $1$, it is equal to the whole of $\hat{A}(\obj3)$,
so the latter is a quotient of a free $\coeffs$-module of rank $20$.
We conclude using Corollary~\ref{cor:dimA3}.
\end{proof}

In particular, we can strengthen Corollary~\ref{cor:dimA3} to
\begin{cor}\label{cor:A3topsemisimple}
  $\hat{A}(\obj3)_{\mathrm{loc}}$ is a direct sum of matrix
  algebras over $\bQ(\pp,\qq)$, \editorialnote[gray]{}
    \editorialnote[gray]{}
  where the sizes of the blocks are given by \eqref{eq:A3irrep}--\eqref{eq:A3irrepb}.

  The map from $\hat{A}(\obj3)$ to the direct sum of matrix algebras over $\coeffs$ given by \eqref{eq:A3irrep}--\eqref{eq:A3irrepb} is injective. \editorialnote[gray]{}
\end{cor}

  As observed at the start of \S\ref{ssec:A3top},
  one could apply the square-pentagon relation \eqref{eq:squarepenta} to
  $H_1^2H_2H_1=\diag{H12H2H1}$, resulting in a linear combination of other basis diagrams and of the diagram
  \diag{A3notword_7}. The coefficient of the latter, $c_{11}=-\frac{\left[6\,\nv\right][\nv+1]}{[\nv]\ \left[2\,\nv\right]\ \left[3\nv\right]}$, is an invertible element of $\coeffs$, and so one could use it as a substitute for
  $H_1^2H_2H_1$ in the basis $\mathcal B^{(3)}(\obj3,\obj3)$.
  For practical purposes, this substitution is convenient, and we denote
  \begin{equation}\label{eq:B3subst}
    \mathcal B^{(3)}(\obj3,\obj3)=
    \left(\mathcal B^{(3)}_{\mathrm{orig}}(\obj3,\obj3)\backslash \left\{\diag{H12H2H1}\right\}\right) \cup \left\{ \diag{A3notword_7}\right\}
  \end{equation}

\subsection{\texorpdfstring{The braiding of $\hat{A}(\obj3)$}{The braiding of hat A(3)}}
Define the elements
\begin{align}\label{eq:topHtoS}
  S_i &=
        \frac{[2\,\nv][3\,\nv]^2}{[\nv+1][6\,\nv]^2}\left(
        S_{H^2}\, H_i^2+S_H\, H_i+S_I
        \right)\qquad i=1,2
\end{align}
where the coefficients are given in \eqref{eq:coeffsS}--\eqref{eq:coeffsSb}. This is nothing but the top part of
the expression in Proposition~\ref{prop:crossing}, viewed as an element of $A(\obj2)$.

Diagrammatically, $S_1=\overgena$, $S_2=\overgenb$.

We have the following
\begin{prop}\label{prop:A3topS}
  The $S_i$ satisfy
  \begin{align}\label{eq:braidtop}
0&=S_1S_2S_1-S_2S_1S_2
\\
0&=  (S_i+1)(S_i-\qq^{-2})(S_i-\qq^2\pp^2)\qquad i=1,2 \label{eq:Scubic}
    \\\label{eq:addrel}
    0&=S_{1}-S_{2}
+S_{1}^{2}-S_{2}^{2}
+S_{1}^{2}S_{2}-S_{2}^{2}S_{1}
\\\notag
&+S_{2}S_{1}^{2}-S_{1}S_{2}^{2}
+S_{2}S_{1}^{2}S_{2}-S_{1}S_{2}^{2}S_{1}
+S_{2}^{2}S_{1}S_{2}^{2}-S_{1}^2S_{2}S_{1}^2
  \end{align}
\end{prop}
\begin{proof}
In principle, one can derive these equations directly from
the relations \eqref{eq:Hrels}--\eqref{eq:Hrelsb}; this can be done
systematically by using a Gr\"obner basis, cf.\ the proof of Proposition~\ref{prop:A3top}.
Here we provide an alternative proof.
First, we find the image of $S_1$ and $S_2$ in the various representations
\eqref{eq:A3irrep}--\eqref{eq:A3irrepb}:
\begin{align}\label{eq:Sirrep}
\rho_{\obj3,X_3}(S_1)&=\begin{pmatrix}-1\end{pmatrix} & \rho_{\obj3,X_3}(S_2)&=\begin{pmatrix}-1\end{pmatrix}
\\
\rho_{\obj3,Y_3}(S_1)&=\begin{pmatrix}\qq^{-2}\end{pmatrix} & \rho_{\obj3,Y_3}(S_2)&=\begin{pmatrix}\qq^{-2}\end{pmatrix}
\\
\rho_{\obj3,Y_3^*}(S_1)&=\begin{pmatrix}\qq^2\pp^2\end{pmatrix} & \rho_{\obj3,Y_3^*}(S_2)&=\begin{pmatrix}\qq^2\pp^2\end{pmatrix}
\\
\rho_{\obj3,A}(S_1)&=\begin{pmatrix}\qq^{-2}&0&0\\\qq^{-1}&-1&0\\1&-\qq(1+\pp^2)&\qq^2\pp^2\end{pmatrix} & \rho_{\obj3,A}(S_2)&=\begin{pmatrix}\qq^2\pp^2&-\qq(1+\pp^2)&1\\0&-1&\qq^{-1}\\0&0&\qq^{-2}\end{pmatrix}
\\[1mm]
\rho_{\obj3,C}(S_1)&=\begin{pmatrix}\qq^{-2}&0\\\qq^{-1}&-1\end{pmatrix} & \rho_{\obj3,C}(S_2)&=\begin{pmatrix}-1&\qq^{-1}\\0&\qq^{-2}\end{pmatrix}
\\\label{eq:Sirrepb}
\rho_{\obj3,C^*}(S_1)&=\begin{pmatrix}\qq^2\pp^2&0\\\qq\pp&-1\end{pmatrix} & \rho_{\obj3,C^*}(S_2)&=\begin{pmatrix}-1&\qq\pp\\0&\qq^2\pp^2\end{pmatrix}
\end{align}
We then check the relations in these representations, and conclude
using Corollary~\ref{cor:A3topsemisimple}.
\end{proof}

\begin{rem}\label{rem:A3top}
In the generic case (i.e., tensoring with the fraction field),
the relation \eqref{eq:topHtoS} can be inverted, cf.\ \eqref{eq:skeintop}
\begin{equation}\label{eq:topStoH}
H_i = \xi +
\frac{\left[6\,\nv\right]}{\left[2\,\nv\right]\ \left[3\,\nv\right]}
\frac{S_i-S_i^{-1}}{\qq-\qq^{-1}}
\qquad i=1,2
\end{equation}
so that $\hat{A}(\obj3)_{\mathrm{loc}}$ is also the algebra generated by $S_1,S_2$. %

The latter is, according to \eqref{eq:braidtop}, a quotient of the three-string braid group algebra;
and more precisely, of a cubic Hecke algebra, cf.\ \eqref{eq:Scubic}. This algebra is generically a multimatrix algebra of dimension 24 with irreducible representations of dimensions 3,2,2,2,1,1,1. They are given by \eqref{eq:Sirrep}--\eqref{eq:Sirrepb}, plus
the additional
\begin{equation}\label{eq:intruder}
S_1=\begin{pmatrix}\qq^{-2}&0\\\pp&\qq^2\pp^2\end{pmatrix} \qquad S_2=\begin{pmatrix}\qq^2\pp^2&-\pp\\0&\qq^{-2}\end{pmatrix}
\end{equation}
The square of the r.h.s.\ of \eqref{eq:addrel} is up to normalisation the central idempotent associated to
the irreducible representation \eqref{eq:intruder}.
Therefore,
$\hat{A}(\obj3)_{\mathrm{loc}}$ is isomorphic
to the algebra with generators $S_1$, $S_2$ and relations as in
Proposition~\ref{prop:A3topS}.
\end{rem}


\subsection{\texorpdfstring{Construction of $A(\obj3)$}{Construction of A(3)}}\label{ssec:A3}
We now define the whole of $A(\obj3)$.
The construction is considerably more complicated than that of $\hat{A}(\obj3)$, but there is no
significant conceptual difference, and we shall skip many of the technical
details; in particular, the Gr\"obner basis calculations, which are best
performed by computer, will be omitted.

In the same way that $A(\obj2)$ can be considered as an algebra generated
by $H$, $K$, $U$, we expect $A(\obj3)$ to have generators
$H_1$, $H_2$, $K_1$, $K_2$, $U_1$, $U_2$; this is correct over $\bQ(\pp,\qq)$. The situation over $\coeffs$ is more complicated,
and will be discussed later.
The relations are all the diagrammatic relations that were already encountered in \S\ref{sec:twostring} in the study of $A(\obj2)$, including the skein relation \eqref{eq:skein},
as well as the new square-pentagon relation \eqref{eq:squarepenta}. See Appendix~\ref{app:skein} for a list
of all relations written explicitly.

We encounter the same difficulty as for $\hat{A}(\obj3)$
when trying to implement \eqref{eq:squarepenta} as a relation:
when we try to simplify say $H_2^2H_1H_2=\diag{H22H1H2}$
by replacing the square-pentagon with the r.h.s.\ of \eqref{eq:squarepenta},
several diagrams that appear are not words in the generators.
We first have \diag{A3notword_7}, which already appeared in $\hat{A}(\obj3)$,
and which can be cancelled by taking differences of two such quartic
words, for example $H_2^2H_1H_2-H_1^2H_2H_1$
(this particular choice is convenient because over the fraction field,
it allows to write a single quartic equation)
Secondly, four more diagrams appear, namely
$Z_1:=\diag{A3notword_1}$,
$Z_2:=\diag{A3notword_2}$,
$Z_3:=\diag{A3notword_3}$,
$Z_4:=\diag{A3notword_4}$.

$Z_1$ and $Z_2$ can be easily taken care of by noting that
$K_1H_2K_1=\diag{K1H2K1}=\tau Z_1$, and similarly
$K_2H_1K_2=\diag{K2H1K2}=\tau Z_2$. However, $\tau$ is not invertible in $\coeffs$.
This is one of the reasons that we work over the fraction field for now.
We can then invert $\tau$ and express $Z_1$ and $Z_2$ in terms of words.

In order to deal with $Z_3$ and $Z_4$, we consider
$H_1^2H_2K_1=\diag{H12H2K1}$ and note that after applying
the square-square relation \eqref{eq:squaresquare}, one has
$H_1^2H_2K_1=b_0 Z_3+\cdots$ where the dots are diagrams that are expressible
as words.
\editorialnote[gray]{}

This allows to express $Z_3$ (and similarly $Z_4$) as a linear
combination of words on condition that one invert $b_0$.

In the end, we are led to the following:
\begin{defn}\label{def:tildeA3}
  $A(\obj3)_{\mathrm{loc}}$ is the $\bQ(\pp,\qq)$-algebra with generators $H_1$, $H_2$, $K_1$, $K_2$, $U_1$, $U_2$ and relations
  {\footnotesize
\begin{gather}\label{eq:A3rels}
  H_i K_i = K_i H_i = \tau\, K_i
  \\
  H_i U_i = U_i H_i = \phi\, U_i
  \\
  K_i^2 = \phi\, K_i
  \\
  K_i U_i = U_i K_i = 0
  \\
  U_i^2 = \delta\, U_i
  \\
  H_1 K_2 H_1 = H_2 K_1 H_2
  \\
  K_i U_j K_i = H_j U_i H_j
  \\
  K_i U_j H_i = H_j U_i K_j
  \\
  K_i U_j U_i = H_j U_i
  \\
  U_i U_j K_i = U_i H_j
  \\
  U_i U_j U_i = U_i
  \\
(H_i-\xi)
(H_i-\zeta)
(H_i-\zeta^*)
= b_2 K_i + b_1 U_i
\\\label{eq:A3relsb}
  H_2^2H_1H_2-H_1^2H_2H_1
  =
  (c_2-c_1)b_0^{-1}(H_1^2H_2K_1-H_2^2H_1K_2)
  \\\notag
  +c_{15}(H_{2}H_{1}H_{2}-H_{1}H_{2}H_{1})
  +c_{11}(H_{1}H_{2}^2-H_{2}H_{1}^2)
  +c_{12}(H_{2}^2H_{1}-H_{1}^2H_{2})
  +c_{10}{}(K_{2}H_{1}H_{2}-K_{1}H_{2}H_{1})
  \\\notag
  +(c_{12}+(c_2-c_1)b_0^{-1}b_4)(H_{2}H_{1}K_{2}-H_{1}H_{2}K_{1})
  +(c_{6}+(c_2-c_1)b_0^{-1}b_2)\tau^{-1}(K_2H_1K_2-K_1H_2K_1)
  \\\notag
  +c_{6}(K_{2}H_{1}-K_{1}H_{2})
  +(c_{7}+(c_2-c_1)b_0^{-1}b_3)(H_{1}K_{2}-H_{2}K_{1})
  +\left(c_{0}+(c_2-c_1)b_0^{-1}b_1\right)(U_{2}H_{1}-U_{1}H_{2})
  \\\notag
  +\left(c_{5}{}-c_{7}{}\right)(H_{1}H_{2}-H_{2}H_{1})
  +c_{2}(H_1U_2-H_{2}U_{1})
  +c_{1}(H_{1}-H_{2})
\end{gather}
}
where $i=1,2$, $j=3-i$, and the coefficients $b_k$ and $c_k$ are given in Appendices~\ref{ssec:squaresquare} and~\ref{ssec:squarepenta}, respectively.
\editorialnote[gray]{}
\editorialnote[gray]{}
\end{defn}

\begin{rem}\label{rem:A3def}
Using the relations above, it is possible to express $U_i$ and $K_i$ in terms of $H_i$:
indeed one easily shows that $H_i$ satisfies the quintic equation
\[
(H_i-\xi)(H_i-\zeta)(H_i-\zeta^*)(H_i-\phi)(H_i-\tau)=0
\]
and then $U_i$ and $K_i$ can be obtained by Lagrange interpolation:
\begin{align*}
U_i&=\delta \frac{(H_i-\xi)(H_i-\zeta)(H_i-\zeta^*)(H_i-\tau)}{(\phi-\xi)(\phi-\zeta)(\phi-\zeta^*)(\phi-\tau)}
\\
K_i&=\phi \frac{(H_i-\xi)(H_i-\zeta)(H_i-\zeta^*)(H_i-\phi)}{(\tau-\xi)(\tau-\zeta)(\tau-\zeta^*)(\tau-\phi)}
\end{align*}
\editorialnote[gray]{}
However, it is convenient to include these generators in order to define a diagrammatic basis of $A(\obj3)$.
\end{rem}

\editorialnote[gray]{}

\begin{prop}\label{prop:tildeA3} The algebra
$A(\obj3)_{\mathrm{loc}}$ possesses the following set of generators as a vector space over $\bQ(\pp,\qq)$:
the same words of $\mathcal B^{(3)}_{\mathrm{orig}}(\obj3,\obj3)$ in Proposition~\ref{prop:A3top} (which we now view as elements of $A(\obj3)_{\mathrm{loc}}$), as well as the following set $\mathcal B_{\mathrm{orig}}^{(2)}(\obj3,\obj3)$:

{\scriptsize
\begin{align*}
H_2^2K_1 &= \diag{H22K1} & H_2H_1K_2 &= \diag{H2H1K2} & H_2K_1 &= \diag{H2K1} & H_2K_1K_2 &= \diag{H2K1K2} \\
H_1H_2K_1 &= \diag{H1H2K1} & H_1^2K_2 &= \diag{H12K2} & H_1^2K_2H_1 &= \diag{H12K2H1} & H_1^2K_2H_1^2 &= \diag{H12K2H12} \\
H_1K_2 &= \diag{H1K2} & H_1K_2H_1 &= \diag{H1K2H1} & H_1K_2H_1H_2 &= \diag{H1K2H1H2} & H_1K_2H_1^2 &= \diag{H1K2H12} \\
H_1K_2K_1 &= \diag{H1K2K1} & H_1K_2K_1H_2 &= \diag{H1K2K1H2} & K_2 &= \diag{K2} & K_2H_1 &= \diag{K2H1} \\
K_2H_1H_2 &= \diag{K2H1H2} & K_2H_1^2 &= \diag{K2H12} & K_2H_1K_2 &= \diag{K2H1K2} & K_2K_1 &= \diag{K2K1} \\
K_2K_1H_2 &= \diag{K2K1H2} & K_2K_1K_2 &= \diag{K2K1K2} & K_2K_1K_2H_1 &= \diag{K2K1K2H1} & K_2K_1K_2K_1 &= \diag{K2K1K2K1} \\
K_1 &= \diag{K1} & K_1H_2 &= \diag{K1H2} & K_1H_2^2 &= \diag{K1H22} & K_1H_2H_1 &= \diag{K1H2H1} \\
K_1H_2K_1 &= \diag{K1H2K1} & K_1K_2 &= \diag{K1K2} & K_1K_2H_1 &= \diag{K1K2H1}  & K_1K_2H_1K_2 &= \diag{K1K2H1K2} \\
K_1K_2K_1 &= \diag{K1K2K1} & K_1K_2K_1H_2 &= \diag{K1K2K1H2} & K_1K_2K_1K_2 &= \diag{K1K2K1K2}
\end{align*}
}

and the set $\mathcal B^{(1)}(\obj3,\obj3)$ consists of the 25 elements
{\scriptsize
\begin{alignat*}{5}
U_1&=\diag[.27]{E1},\  & U_1H_2&=\diag[.27]{E1H2},\  & U_1U_2&=\diag[.27]{E1E2},\  & U_1K_2&=\diag[.27]{E1K2},\  & U_1K_2H_1&=\diag[.27]{E1K2H1},\  \\
H_2U_1&=\diag[.27]{H2E1},\  & K_1U_2K_1&=\diag[.27]{K1E2K1},\  & K_1U_2&=\diag[.27]{K1E2},\  & K_1U_2H_1&=\diag[.27]{K1E2H1},\  & K_1U_2K_1H_2&=\diag[.27]{K1E2K1H2},\  \\
U_2U_1&=\diag[.27]{E2E1},\  & U_2K_1&=\diag[.27]{E2K1},\  & U_2&=\diag[.27]{E2},\  & U_2H_1&=\diag[.27]{E2H1},\  & U_2K_1H_2&=\diag[.27]{E2K1H2} \\
K_2U_1&=\diag[.27]{K2E1},\  & K_2U_1H_2&=\diag[.27]{K2E1H2},\  & H_1U_2&=\diag[.27]{H1E2},\  & K_2U_1K_2&=\diag[.27]{K2E1K2},\  & K_2U_1K_2H_1&=\diag[.27]{K2E1K2H1},\  \\
H_1K_2U_1&=\diag[.27]{H1K2E1},\  & H_1K_2U_1H_2&=\diag[.27]{H1K2E1H2},\  & H_1^2U_2&=\diag[.27]{H12E2},\  & H_1K_2U_1K_2&=\diag[.27]{H1K2E1K2},\  & H_1K_2U_1K_2H_1&=\diag[.27]{H1K2E1K2H1}
\end{alignat*}
}
\end{prop}
\begin{proof}
  Consider $\bQ(\pp,\qq)\left<H_1,H_2,K_1,K_2,U_1,U_2\right>$ with the graded lexicographic monomial order, with
  $U_1<U_2<K_1<K_2<H_1<H_2$. A computer calculation produces a Gr\"obner basis of the ideal
  generated by the relations of $A(\obj3)_{\mathrm{loc}}$, with 111 elements;
  the basis elements above are exactly all the words that are not divisible by the leading monomial of any
  element of that Gr\"obner basis.
\end{proof}

Among the diagrams listed above, we note that 8 of them can be simplified using our diagrammatic relations:
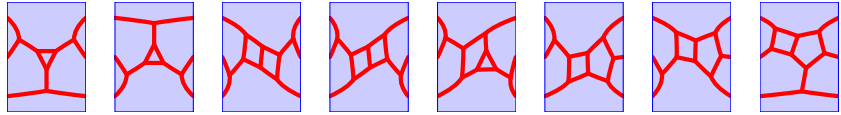
\[
  \diag{K1H2K1}
  \
  \diag{K2H1K2}
  \
  \diag{K2K1K2K1}
  \
  \diag{K1K2K1K2}
  \
  \diag{K1K2H1K2}
  \
  \diag{K1K2K1H2}
  \
  \diag{K2K1K2H1}
  \
  \diag{H12H2H1}
  \
\]

Conversely, here are 8 diagrams that are missing in the list:

\[
\foreach\i in {1,2,3,4,5,0,6,7} {\diag{A3notword_\i}\ }
\]

The first four diagrams in the second display are $Z_1,\ldots,Z_4$, introduced earlier in this section.
Since $K_1H_2K_1=\tau Z_1$ and $K_2H_1K_2=\tau Z_2$, replacing the first two diagrams in the first display
by $Z_1,Z_2$, respectively, preserves their span over $\bQ(\pp,\qq)$, but changes their span over $\coeffs$.
More generally, one can compute explicitly the $8\times 8$ change-of-basis matrix expressing the diagrams in the first display
in terms of those in the second; it has entries in $\coeffs$ and is invertible over $\bQ(\pp,\qq)$, but not over $\coeffs$.

We therefore obtain the following improved generating sets:
$\mathcal{B}^{(3)}(\obj3,\obj3)$ is given by \eqref{eq:B3subst},
$\mathcal{B}^{(1)}(\obj3,\obj3)$ is unchanged, and
\begin{align}\label{eq:A3subst2}
  \mathcal B^{(2)}(\obj3,\obj3) &=\left(\mathcal B_{\mathrm{orig}}^{(2)}(\obj3,\obj3) \backslash \left\{
                          \diag{K1H2K1},  \diag{K2H1K2},  \diag{K2K1K2K1},  \diag{K1K2K1K2},  \diag{K1K2H1K2},  \diag{K1K2K1H2},  \diag{K2K1K2H1}\right\}\right)
                          \\\notag
                        &\cup \left\{
                          \foreach\i in {1,...,6} {\diag{A3notword_\i}\ifnum\i=6\else,\fi}
                          \right\}
  \\\label{eq:A3subst0}
  \mathcal B^{(0)}(\obj3,\obj3) &=\left\{ \diag{A3notword_0} \right\}
\end{align}
Their $\bQ(\pp,\qq)$-span is the same as the original generating sets, but their $\coeffs$-span is larger.
\editorialnote[gray]{}

At this stage, we can finally introduce $A(\obj3)$:
\begin{defn}\label{def:A3}
Define $J^{(i)}(\obj3,\obj3)$ to be the $\coeffs$-linear span of $\bigsqcup_{j=0}^i \mathcal{B}^{(j)}(\obj3,\obj3)$ inside $A(\obj3)_{\mathrm{loc}}$;
and $A(\obj3)=J^{(3)}(\obj3,\obj3)$. (Conventionally, $J^{(-1)}(\obj3,\obj3)=0$.)
\end{defn}
\editorialnote[gray]{}
Note that by definition $A(\obj3)_{\mathrm{loc}}=A(\obj3)\otimes_{\coeffs} \bQ(\pp,\qq)$.

\begin{lemma}
  $A(\obj3)$ is a subring of $A(\obj3)_{\mathrm{loc}}$, and $J^{(i)}(\obj3,\obj3)$ are ideals of it, $i=0,\ldots,3$.
\end{lemma}
\begin{proof}
  Consider two elements $x$ and $y$ in $\mathcal B(\obj3,\obj3):=\bigsqcup_{j=0}^3 \mathcal B^{(j)}(\obj3,\obj3)$;
  because $\mathcal B(\obj3,\obj3)$ is a generating set over $\bQ(\pp,\qq)$, one can express
  the product $x\,y$ as a linear combination of elements of $\mathcal B(\obj3,\obj3)$. (In fact,
  we shall soon prove that $\mathcal B(\obj3,\obj3)$ is a basis, so that linear combination is unique.)
  The first part of the Lemma says that the coefficients of the expansion of $x\,y$ belong to $\coeffs$.
  This check is best performed by computer; for practical purposes, note that one only needs to check it
  for $x$ in a set of generators as $\coeffs$-algebra; that means, for $x$ equal to either
  $H_1$, $H_2$, $K_1$, $K_2$, $U_1$, $U_2$, or the first four additions in \eqref{eq:A3subst2}, previously denoted
  $Z_1$, $\ldots$, $Z_4$, or the addition in \eqref{eq:A3subst0}, which we denote $Z_0=\diag{A3notword_0}$
  (the other substitutions did not change the $\coeffs$-span).

   \editorialnote[gray]{}

  The fact that the $J^{(i)}(\obj3,\obj3)$ are ideals of $A(\obj3)$ follows immediately
  from the fact that $\mathcal B^{(j)}(\obj3,\obj3)$ is the subset of $\mathcal B(\obj3,\obj3)$ of diagrams of cut length $j$,
  and the discussion of \S\ref{ssec:ideals} applies, cf the first part of Lemma~\ref{lem:ideals}.
  \editorialnote[gray]{}
\end{proof}

\begin{lemma}\label{lem:A3rep}
$A(\obj3)_{\mathrm{loc}}$ admits the following irreducible, pairwise distinct, representations, where either $H_1$, $H_2$ are mapped to \eqref{eq:A3irrep}--\eqref{eq:A3irrepb}, or $H_1$ is mapped to the following matrices
    \def\ha{\zeta}\def\hb{\zeta^*}
\begin{align}
\rho_{\obj3,X_2}(H_1)  &=\left(\!\begin{smallmatrix}
\xi&0&0&0\\
\frac{\left[6\nv\right]}{\left[2\nv\right]\left[3\nv\right]}&\zeta&0&0\\
\frac{\left[6\nv\right]}{\left[2\nv\right]\left[3\nv\right]}&\frac{\left[6\nv\right]}{\left[3\nv\right]}&\zeta^*&0\\
\xi&\frac{\left[6\nv\right]\left(\left[2\nv\right]\left[2\right]+1\right)}{\left[2\nv\right]\left[3\nv\right]}&
-\frac{\left[6\nv\right]\left(\left[6\nv+2\right]\left[\nv\right]-\left[3\nv+1\right]\right)}{\left[2\nv\right]\left[3\nv\right]\left[3\nv+1\right]}
&\tau
\end{smallmatrix}\!\right)
\\
\rho_{\obj3,Y_2}(H_1)  &=\left(\!\begin{smallmatrix}
    \ha&0&0\\
-\frac{\left[6\nv\right]}{\left[3\nv\right]\left[2\nv\right]}&\xi&0\\
\ha&-\frac{\left[3\nv-1\right]\left(\qq^{2}\pp+\pp^{3}-\pp-\pp^{-1}+\pp^{-3}+\qq^{-2}\pp^{-1}\right)}{\left[\nv\right]}&\tau
\end{smallmatrix}\!\right)
\\
\rho_{\obj3,Y_2^*}(H_1)&=\left(\!\begin{smallmatrix}
\hb&0&0\\
-\frac{\left[6\nv\right]\left[\nv+1\right]}{\left[3\nv\right]\left[2\nv\right]}&\xi&0\\
\hb&\frac{\left[4\nv+1\right]\left(\qq^{2}\pp+\pp^{3}-\pp-\pp^{-1}+\pp^{-3}+\qq^{-2}\pp^{-1}\right)}{\left[\nv\right]}&\tau
\end{smallmatrix}\!\right)
\\
\rho_{\obj3,L}(H_1)&=\left(\!\begin{smallmatrix}
\zeta
&0&0&0&0\\
-\frac{\left[2\nv-2\right]\left[2\nv\right]-\left[\nv-1\right]^{2}}{\left[\nv-1\right]\left[\nv\right]}&\tau&0&0&0\\
\frac{\left[4\nv-2\right]}{\left[\nv\right]\left[2\nv-1\right]}&-\frac{\qq^{2}\pp+\pp^{3}-\pp-\pp^{-1}+\pp^{-3}+\qq^{-2}\pp^{-1}}{\left[\nv\right]}&\xi&0&0\\
-\frac{\left[3\nv\right]\left[4\nv-2\right]}{\left[\nv\right]\left[2\nv-1\right]}&\scriptscriptstyle\qq^{2}\pp^{3}+\qq^{2}\pp+\qq^{2}\pp^{-1}-2\pp-2\pp^{-1}+\qq^{-2}\pp+\qq^{-2}\pp^{-1}+\qq^{-2}\pp^{-3}&\frac{\left[\nv+1\right]\left[12\nv\right]}{\left[3\nv\right]\left[4\nv\right]}&\zeta^*&0\\
0&0&\frac{\left[4\nv-2\right]\left[6\nv+1\right]}{\left[\nv\right]\left[2\nv-1\right]}&\frac{\left[4\nv-2\right]\left[6\nv+2\right]}{\left[\nv\right]\left[2\nv-1\right]\left[3\nv+1\right]}&\phi
\end{smallmatrix}\!\right)\!\!\!\!\!\!\!\!\!\! \\
\rho_{\obj3,I}(H_1)&=\left(\!\begin{smallmatrix}
\tau
\end{smallmatrix}\!\right)
\end{align}%
and $H_2$ is mapped to the same matrices rotated 180 degrees.

These restrict to representations of $A(\obj3)$ (over $\coeffs$).
\end{lemma}
\begin{proof}
As mentioned in Remark~\ref{rem:A3def}, the $H_1$, $H_2$ generate $A(\obj3)_{\mathrm{loc}}$, so fixing $H_1$ and $H_2$ as above suffices to define representations.
Checking relations \eqref{eq:A3rels}--\eqref{eq:A3relsb} is a tedious but elementary exercise.

Irreducibility and pairwise distinctness follow from the same arguments as in the proof of Lemma~\ref{lem:A3toprep},
and we shall not repeat them here. For example, for the most complicated case of $L$, we find that there exists $M$ such that
$M\rho_{\obj3,L}(H_i)=\rho_{\obj3,L}(H_i)^TM$, and $M=P^TDP$ with $D$ diagonal with eigenvalues
{\tiny
\begin{multline*}
 1,\,\frac{\left[2\,\nv-1\right]\left[2\right]\left[4\,\nv\right]\left[6\,\nv\right]\left[2\,\nv\right]\left[6\,\nv+2\right]\left[3\,\nv-3\right]}{\left[\nv+2\right]\left[4\,\nv-2\right]\left[4\,\nv+1\right]\left[5\,\nv-1\right]\left[\nv\right]\left[3\,\nv+1\right]\left[\nv-1\right]},\,\frac{\left[6\,\nv+3\right]\left[3\,\nv-1\right]\left[6\,\nv+1\right]\left[3\,\nv-3\right]}{\left[2\,\nv+1\right]\left[2\,\nv+2\right]\left[4\,\nv+1\right]\left[5\,\nv-1\right]\left[2\right]\left[\nv-1\right]},\\\frac{\left[5\,\nv\right]\left[3\,\nv-1\right]\left[6\,\nv\right]^{3}\left[3\,\nv-3\right]}{\left[3\,\nv\right]^{4}\left[2\,\nv+2\right]\left[\nv+2\right]\left[5\,\nv-1\right]\left[\nv-1\right]},\,\frac{\left[5\,\nv\right]\left[3\,\nv-1\right]\left[\nv+1\right]\left[6\,\nv\right]\left[2\,\nv\right]\left[3\,\nv-3\right]}{\left[6\,\nv+1\right]\left[2\,\nv+2\right]\left[\nv+2\right]\left[4\,\nv\right]\left[\nv-1\right]}
\end{multline*}
}
which has the right positivity requirement.

\editorialnote[gray]{}

The matrices above have coefficients in $\coeffs$. One checks that the same is true of the matrices of
all the generators of $A(\obj3)$ as a $\coeffs$-algebra, namely,
$K_1$, $K_2$, $U_1$, $U_2$, and $Z_0$, $\ldots$, $Z_4$; e.g., for $Z_1$,
one checks that $K_1H_2K_1$ is sent to a matrix with entries in $\tau\coeffs$. This shows the restriction property
to $A(\obj3)$.
\end{proof}

Putting together Proposition~\ref{prop:tildeA3} and Lemma~\ref{lem:A3rep} allows to conclude:
\begin{cor}\label{cor:A3semisimple}
  $A(\obj3)_{\mathrm{loc}}$ is of dimension $80$, and a direct sum of matrix
  algebras over $\bQ(\pp,\qq)$,
  where the sizes of the blocks are given by the representations of Lemma~\ref{lem:A3rep}.

  $A(\obj3)$ is a rank $80$ free $\coeffs$-module, and the map from $A(\obj3)$ to the direct sum of matrix algebras over $\coeffs$ given by the same representations is injective.

  More generally, the $J^{(i)}(\obj3,\obj3)$ are free $\coeffs$-modules; the quotient $\coeffs$-modules
  \[ J^{(i)}(\obj3,\obj3)/J^{(i-1)}(\obj3,\obj3) \]
   have basis $\mathcal B^{(i)}(\obj3,\obj3)$.
\end{cor}

Recall the notation \eqref{eq:irr} for labels of irreps.
  \begin{prop}\label{prop:A3bimodules2}
    For each $j=0,\ldots,3$,
    the natural $A(\obj3)$-$A(\obj3)$ bimodule map from $J^{(j)}(\obj3,\obj3)$
    to the matrix algebra over $\coeffs$ given by the irreps $\Irr^j$, defined in \eqref{eq:irr},
    has kernel $J^{(j-1)}(\obj3,\obj3)$, and after quotienting by the kernel and tensoring with $\bQ(\pp,\qq)$,
    becomes an isomorphism of $A(\obj3)_{\mathrm{loc}}$-$A(\obj3)_{\mathrm{loc}}$ bimodules.

 In particular, $\hat{A}(\obj3)$ is isomorphic to the quotient $A(\obj3)/J^{(2)}(\obj3,\obj3)$ under the natural
 map $H_k\mapsto H_k$, $k=1,2$.
\end{prop}
\begin{proof}
  Let us first show that every element of $\mathcal B^{(i)}(\obj3,\obj3)$ is in the kernel of the map to the matrix algebra of irreps $\Irr^j$ for $i<j$, defined in \eqref{eq:irr}.
  One can use Remark~\ref{rem:A3def}, which gives the correspondence between the eigenvalues of $H_k$ and those
  of $K_k$ and $U_k$, $k=1,2$. We immediately find that $\rho_R(U_1)=\rho_R(U_2)=0$ for all $R\in \Irr^j$, $j=2,3$,
  defined in \eqref{eq:irr}, because $\rho_R(H_k)$ doesn't have the eigenvalue
  $\phi$. Similarly, $\rho_R(K_1)=\rho_R(K_2)=0$ for all $R\in \Irr^3$ because $\rho_R(H_k)$ doesn't have the eigenvalue $\tau$.
  This implies the statement at $i=1,2$. Remains $i=0$, that is, one must show that
  $\rho_R(Z_0)=0$ for all $R\in \Irr^j$, $j\ge 1$. Now $Z_0$ is obviously central and the module that it generates is of rank $1$; this implies
  that it is proportional (in $\bQ(\pp,\qq)$) to the primitive central idempotent of an irrep of dimension $1$. We compute $H_1Z_0= \tau Z_0$ and recognize the eigenvalue
  of $\rho_{\obj3,I}(H_1)$.  So the kernel contains $J^{(j-1)}(\obj 3,\obj3)$.

  One can then check that $J^{(i)}(\obj3,\obj3)/J^{(i-1)}(\obj3,\obj3)$ has the same rank as the matrix algebra it is sent to:
  \[
    1=1^2,\qquad 25=5^2,\qquad 34=4^2+3^2+3^2,\qquad 20=3^2+2^2+2^2+1^2+1^2+1^2
  \]
  so that after tensoring with the fraction field, the map is an isomorphism. Furthermore, because $J^{(i-1)}(\obj3,\obj3)$ is generated by a subset of the basis of $J^{(i)}(\obj 3,\obj3)$ as a free $\coeffs$-module, this also implies that the kernel of the map is $J^{(i-1)}(\obj3,\obj3)$.

  We now have two injective maps into the matrix algebra of $\Irr^3$, defined in \eqref{eq:irr}:
  one, from $\hat A(\obj3)$, was defined in Lemma~\ref{lem:A3toprep}; the second map from $A(\obj3)/J^{(2)}(\obj3,\obj3)$ was just obtained now.
  They manifestly have the same image because both $\hat A(\obj 3)$ and $A(\obj3)/J^{(2)}(\obj3,\obj3)$ are generated as $\coeffs$-algebras
  by the two generators $H_1$, $H_2$ and these are sent to the same matrices by their respective maps.
  Composing one map with the inverse of the other leads to an isomorphism where
  $H_k\mapsto H_k$.

  \editorialnote[gray]{}

\end{proof}

We now return to the bimodules introduced in \S\ref{ssec:bimodules}:
\begin{prop}\label{prop:A3bimodules}
For each $\obj i=\obj0,\ldots,\obj3$,
$B(\obj 3,\obj i)$ is a left $A(\obj3)$-module,
$B(\obj i,\obj 3)$ is a right $A(\obj3)$-module,
 and 
\[
B(\obj3,\obj i)\otimes_{A(\obj i)} B(\obj i,\obj3)
\cong
\begin{cases}
J^{(i)}(\obj3,\obj3) & i\ne \obj1,
\\
J^{(1)}(\obj3,\obj3)/J^{(0)}(\obj3,\obj3)& i=\obj1.
\end{cases}
\]
\end{prop}
\begin{proof}
  We assume in the proof $i\le 2$, otherwise the statement is trivial.

  The proof of the first part of the Proposition is an explicit computation, similar to that of Proposition~\ref{prop:A2diag5};
  e.g., for $B(\obj 3,\obj 0)$, one computes, using the rules of Appendix~\ref{app:skein},
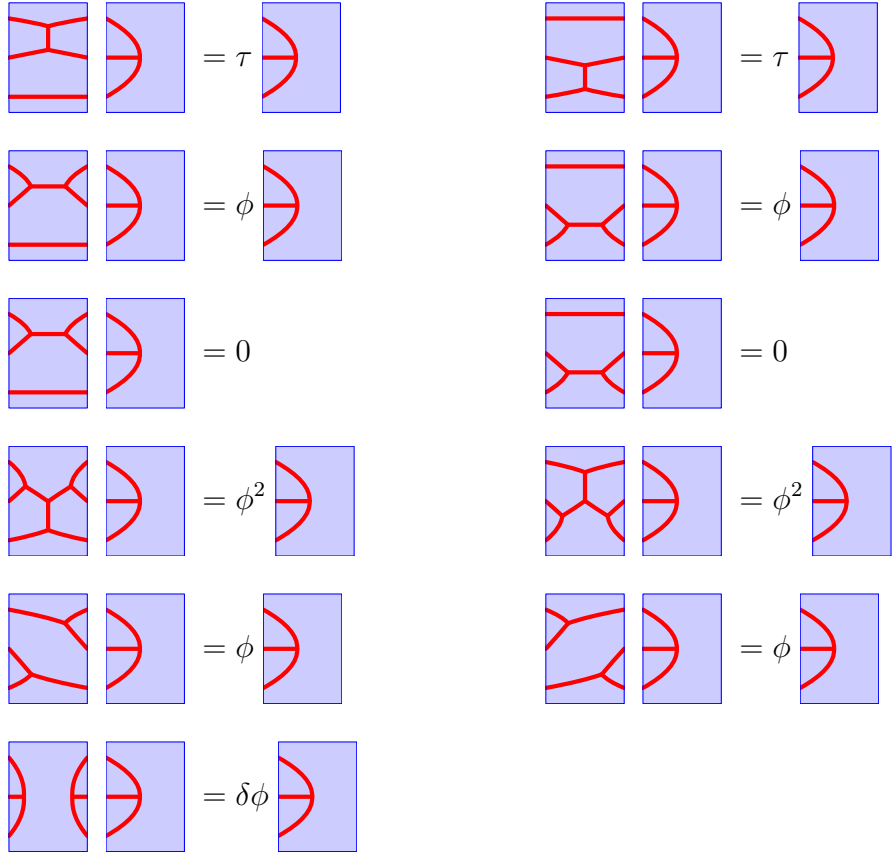
  \begin{align*}
    \diag{H1}\diag{diag3_0_0}&=\tau \diag{diag3_0_0}
    &
    \diag{H2}\diag{diag3_0_0}&=\tau \diag{diag3_0_0}
    \\
    \diag{K1}\diag{diag3_0_0}&=\phi \diag{diag3_0_0}
    &
    \diag{K2}\diag{diag3_0_0}&=\phi \diag{diag3_0_0}
    \\
    \diag{E1}\diag{diag3_0_0}&=0
    &
    \diag{E2}\diag{diag3_0_0}&=0
    \\
    \diag{A3notword_1}\diag{diag3_0_0} &=\phi^2 \diag{diag3_0_0}
    &
    \diag{A3notword_2}\diag{diag3_0_0} &=\phi^2 \diag{diag3_0_0}
    \\
    \diag{A3notword_3}\diag{diag3_0_0} &=\phi \diag{diag3_0_0}
    &
    \diag{A3notword_4}\diag{diag3_0_0} &=\phi \diag{diag3_0_0}
    \\
    \diag{A3notword_0}\diag{diag3_0_0} &=\delta\phi\diag{diag3_0_0}
  \end{align*}
  \editorialnote[gray]{}
  Since the r.h.s.\ of every rule of Appendix~\ref{app:skein} involves expressions with coefficients in $\coeffs$,
  all one needs to check is that for each diagram obtained by concatenation of one of the generators of $A(3)$ with a basis diagram,
  one can apply repeatedly substitutions of the form ``l.h.s. $\mapsto$ r.h.s.'' to reduce it to a linear combination of basis diagrams.
  This is not significantly harder for $i=1,2$, though best performed by computer.\footnote{This is in contradistinction with the analysis that
    was performed for $A(\obj3)$ itself (and more precisely, $\hat{A}(\obj3)$, cf the proof of Proposition~\ref{prop:A3top}),
    which involved some nontrivial overlap between various equations derived from those rules.}

  Composition gives an $A(\obj3)$-$A(\obj3)$ bimodule map
  $\mu_i\colon B(\obj3,\obj i)\otimes_{A(\obj i)}B(\obj i,\obj3)\to J^{(i)}(\obj3,\obj3)$,
  by Lemma~\ref{lem:ideals}.
  For $i=0$, the tensor product of the two displayed generators maps to $Z_0$, the basis of $J^{(0)}(\obj3,\obj3)$.
  For $i=1$, the 25 basis tensors map to the square array $\mathcal B^{(1)}(\obj3,\obj3)$ in Proposition~\ref{prop:tildeA3}.
  Their classes form a basis of $J^{(1)}(\obj3,\obj3)/J^{(0)}(\obj3,\obj3)$, so the composite of $\mu_1$ with the quotient map is an isomorphism.

  For $i=2$, write $T=B(\obj3,\obj2)\otimes_{A(\obj2)}B(\obj2,\obj3)$.
  Its presentation has $16^2=256$ generators and the 768 balancing relations
  \[
    (ma)\otimes n-m\otimes(an),\qquad a\in\{U,K,H\},
  \]
  where $m,n$ run through the diagram bases. A computer algebra calculation reduces this presentation
  using 196 pivots, each equal up to sign to $1$, $c_{11}$ or $\phi$, with no remaining relations on the other 60 generators.
  Here $c_{11}=-[6\nv][\nv+1]/([\nv][2\nv][3\nv])$ is the coefficient in Appendix~\ref{ssec:squarepenta};
  both $c_{11}$ and $\phi$ are units by \S\ref{sec:spec}. Thus $T$ is free of rank 60 over $\coeffs$.
  The composition calculation also gives
  \(
    \phi J^{(2)}(\obj3,\obj3)\subseteq\operatorname{im}\mu_2
  \).
  Since $\phi$ is a unit, $\mu_2$ is surjective. Its target is free of rank $1+25+34=60$ by Corollary~\ref{cor:A3semisimple},
  so $\mu_2$ is an isomorphism.
\end{proof}

Over $\bQ(\pp,\qq)$, the representation content of these bimodules can be written uniformly as
\[
B(\obj i,\obj3)_{\mathrm{loc}}\cong\bigoplus_{R\in\mathcal I_i}V_{\obj i,R}\otimes V_{\obj3,R}^{*},
\qquad i=0,1,2,3,
\]
where $\mathcal I_0=\Irr^0$, $\mathcal I_1=\Irr^1$, $\mathcal I_2=\Irr^{\le2}$ and $\mathcal I_3=\Irr^{\le3}$.
Here $V_{\obj i,R}$ is the simple left $A(\obj i)_{\mathrm{loc}}$-module labelled by $R$, and the star denotes the linear dual, with its natural right action.
For $i=0,1$ the unique simple module is one-dimensional; for $i=2$ its character is given by \eqref{eq:Heigenvalues}, and for $i=3$ the modules are those of Lemma~\ref{lem:A3rep}.
For $i\le2$, the decomposition follows from Proposition~\ref{prop:A3bimodules} and the nondegenerate trace pairings;
for $i=2$, the matching of labels is checked using the ranks of the idempotents $\pi_R$ and the traces of $H_1$ on their images.
For $i=3$, this is the regular bimodule decomposition of Corollary~\ref{cor:A3semisimple}.
Reflection gives the analogous decompositions of $B(\obj3,\obj i)_{\mathrm{loc}}$.

\subsection{Gram determinant}\label{sec:Gram3}
We give without proof the determinant of the Gram matrix of the basis
$\mathcal B(\obj3,\obj3)$, obtained by computer:
\begin{multline}\label{eq:Gram3}
	\scriptstyle  \frac{\left[4\nv\right]^{80}\left[6\nv\right]^{361}\left[4\nv+2\right]^{5}\left[6\nv+2\right]^{75}\left[6\nv+4\right]^{10}\left[4\nv+1\right]^{65}\left[5\nv\right]^{75}\left[6\nv+3\right]^{45}\left[5\nv+1\right]^{15}\left[6\nv+1\right]^{80}\left[6\nv+5\right]}{\left[\nv\right]^{175}\left[\nv+1\right]^{60}\left[2\nv\right]^{267}\left[3\nv\right]^{286}\left[3\nv+1\right]^{70}\left[3\nv+2\right]^{10}\left[2\nv+1\right]^{50}}
	\\\scriptstyle
	\frac{\left[2\nv-2\right]^{5}\left[3\nv-3\right]^{45}}{\left[\nv-1\right]^{50}}
	\frac{\left[2\nv-4\right]^{10}}{\left[\nv-2\right]^{10}}
	\boxed{\scriptstyle\left[5\nv-1\right]^{80}\left[4\nv-1\right]^{15}\left[3\nv-1\right]^{65}
		\frac{\left[4\nv-2\right]^{75}}{\left[2\nv-1\right]^{70}}
		\left[\nv-5\right]}
\end{multline}
The expressions outside the box are units in $\coeffs$ by \S\ref{sec:spec}, with ratios interpreted after cancellation.
The boxed factors give precisely the exceptional zeros discussed below.

\subsection{\texorpdfstring{The braiding of $A(\obj3)$}{The braiding of A(3)}}\label{ssec:braidA3}
Define
\begin{equation}\label{eq:HtoSA3}
  S_i
  =
  \frac{[2\,\nv][3\,\nv]^2}{[\nv+1][6\,\nv]^2}\left(
    S_{H^2}\, H_i^2+S_H\, H_i+S_I+S_K\, K_i + S_U\, U_i
  \right)\qquad i=1,2
\end{equation}
where the coefficients are given in \eqref{eq:coeffsS}--\eqref{eq:coeffsSb}.
Note that the image of $S_i$ in $\hat{A}(\obj3)$ reproduces \eqref{eq:topHtoS}.

\begin{prop}\label{prop:A3S}
The $S_i$ satisfy
  \begin{align}\label{eq:braidA3}
0&=S_1S_2S_1-S_2S_1S_2
\\
0&=  (S_i+1)(S_i-\qq^{-2})(S_i-\qq^2\pp^2)(S_i+\pp^6)(S_i-\pp^{12})\qquad i=1,2 \label{eq:Squintic}
  \end{align}
\end{prop}
\begin{proof}
We check the relations explicitly in the various representations of Lemma~\ref{lem:A3rep}.
The expressions of $S_1$ and $S_2$ in
the first six are given in \eqref{eq:Sirrep}--\eqref{eq:Sirrepb}; the remaining ones are:
\begin{align}\label{eq:Sirrepc}
\rho_{\obj3,X_2}(S_1)  &=
\left(\!\begin{smallmatrix}
-1&0&0&0\\
-\qq^{-1}&\qq^{-2}&0&0\\
\qq&\left(\pp-1\right)\left(\pp+1\right)&\qq^{2}\pp^{2}&0\\
1&-\qq^{-1}\pp^{2}\left(\pp^{2}-1+\pp^{-2}\right)&\qq\pp^{4}\left(\pp^{2}-1+\pp^{-2}\right)&-\pp^{6}
\end{smallmatrix}\!\right)
\\\label{eq:Sirrepdd}
\rho_{\obj3,Y_2}(S_1)  &=
\left(\!\begin{smallmatrix}
\qq^{-2}&0&0\\
\qq^{-1}&-1&0\\
1&\qq^{-1}\pp^{2}\left(\pp^{2}-\qq\pp^{-1}\right)\left(\pp^{2}+\qq\pp^{-1}\right)&-\pp^{6}
\end{smallmatrix}\!\right)
\\\label{eq:Sirrepd}
\rho_{\obj3,Y_2^*}(S_1)&=
\left(\!\begin{smallmatrix}
\qq^{2}\pp^{2}&0&0\\
-\qq\pp&-1&0\\
1&-\qq^{-1}\pp^{3}\left(\qq\pp^{2}-\pp^{-2}\right)\left(\qq\pp^{2}+\pp^{-2}\right)&-\pp^{6}
\end{smallmatrix}\!\right)
\\\label{eq:Sirrepe}
\rho_{\obj3,L}(S_1)&=
\left(\!\begin{smallmatrix}
\qq^{-2}&0&0&0&0\\
\qq^{-1}\pp^{3}&-\pp^{6}&0&0&0\\
-1&\pp^{3}\left(\qq-\qq^{-1}\right)&-1&0&0\\
\qq\pp^{-3}&-\pp^{4}-\qq^{2}+1&\qq\pp\left(\pp^{4}-1+\pp^{-4}\right)&\qq^{2}\pp^{2}&0\\
\qq^{2}&-\pp^{5}\qq\left(\qq-\qq^{-1}\right)\left(\qq^{-1}\pp^{2}+\qq\pp^{-2}\right)&\qq\pp^{6}\left(\qq^{-1}\pp^{2}+\qq\pp^{-2}\right)\left(\pp^{4}-1+\pp^{-4}\right)&\pp^{7}\qq\left(\qq-\qq^{-1}\right)\left(\qq^{-1}\pp^{2}+\qq\pp^{-2}\right)&\pp^{12}
\end{smallmatrix}\!\right)\!\!\!\!\!\!\!\!\!\!\!\!\!\!\!\!\!\!\!\!\!\!\!\!\!\!
\\\label{eq:Sirrepf}
\rho_{\obj3,I}(S_1)&=\left(\!\begin{smallmatrix}
-\pp^6
\end{smallmatrix}\!\right)
\end{align}%
and $S_2$ is mapped to the same matrices rotated 180 degrees. $\rho_{\obj3,L}(S_1)$ was already given in
\eqref{eq:braidmatrix}; furthermore, all the matrices above satisfy the same properties given after \eqref{eq:braidmatrix},
namely they have entries in $\bZ[\pp^{\pm1},\qq^{\pm1}]$, they are lower triangular, their inverses are their conjugate under \eqref{eq:bar}, and $\rho_{\obj3,R}(S_1S_2S_1)$ is antidiagonal
with constant entries on the antidiagonal. For $R\in\Irr^2$, these matrices use the diagrammatic bases constructed in Section~\ref{ssec:ybe2}.
\end{proof}

\begin{rem}
Let us go back to the generic case, i.e., work in $\bQ(\pp,\qq)$.
Note that $A(\obj3)_{\mathrm{loc}}$ is generated by $S_1,S_2$; it is therefore according to \eqref{eq:braidA3}
a quotient of the three-string braid group algebra by the quintic relation \eqref{eq:Squintic}.
The latter is an algebra of dimension $600$;
this means that the relations of Proposition~\ref{prop:A3S} are {\em not} all the defining relations of $A(\obj3)_{\mathrm{loc}}$.
(Compare with Remark~\ref{rem:A3top}.)
To generate the whole ideal of relations of $A(\obj3)_{\mathrm{loc}}$,
one needs one more relation of degree $7$, which is too complicated to write down explicitly.
\end{rem}

\subsection{The truncated category}\label{ssec:trunccat}
We are now in a position to define $\threecat$, as advertised in the introduction. Its objects are $\obj0,\obj1,\obj2,\obj3$.
With the usual identification of morphisms with diagrams:
firstly, the endomorphism algebras $\End_{\threecat}(\obj k)$ are the algebras $A(\obj k)$, which are:
\begin{itemize}
	\item $A(\obj0)\cong \coeffs$, where we identify the unit with the empty diagram.
	\item $A(\obj1)\cong \coeffs$, where we identify the unit with the trivial diagram \diag{diag1_1_0}.
	\item $A(\obj2)$ was defined in \S\ref{sec:twostring}, cf Definition~\ref{def:A2} and Proposition~\ref{prop:A2}.
	\item $A(\obj3)$ was defined in \S\ref{sec:threestring}, cf Definitions~\ref{def:tildeA3} and \ref{def:A3}.
\end{itemize}

The $\Hom_{\threecat}(\obj r,\obj s)$, $r\ne s$, are the bimodules $B(\obj r,\obj s)$, which are:
\begin{itemize}
	\item $B(\obj0,\obj1)=B(\obj1,\obj0)=0$.
	\item $B(\obj i,\obj 2)$ and $B(\obj 2,\obj i)$, $\obj i=\obj 0,\obj 1$ were defined in \S\ref{ssec:A2bi}.
	\item Finally, $B(\obj i,\obj3)$ and $B(\obj3,\obj i)$, $\obj i=\obj0,\obj1,\obj2$, were defined in \S\ref{ssec:bimodules}, see Proposition~\ref{prop:A2diag5} and \ref{prop:A3bimodules} for their bimodule structure.
\end{itemize}

Let $\Ideal$ be the tensor ideal generated by the various relations that we have imposed on our diagrammatic algebra $\coeffs\Rcat$, and which are
summarised for the reader's convenience in Appendix~\ref{app:skein}.
The construction of $\threecat$ was based on imposing these relations; this means that
we have a functor $\Phi$ from $\threecat$ to $\coeffs\Rcat/\Ideal$. As discussed in
\S\ref{sec:functors}, we also have a full functor $\Psi_C$ from $\kk_q\Rcat$ to $\Icat(L_C)$
for each $C$ in our list of 10 points of the exceptional series,
which factors through $\kk_q\Rcat/\Ideal$ because the relations of $\Ideal$ interpolate those of the exceptional series;
we still call the resulting full functor $\Psi_C$.
\editorialnote[gray]{}

It is natural to conjecture that $\Phi$ is also a full functor, and that it makes $\threecat$ a full subcategory of $\coeffs\Rcat/\Ideal$.
Such conjectures are typically very difficult to prove. More generally, one may hope that the full interpolating category
$\Intcat$ is nothing but $\coeffs\Rcat/\Ideal$.

We prove here a more modest result. Let us denote $\Xi_C$ the composition of $\Phi$, of the base change $\theta_C: \coeffs \to \kk_{q^{1/s_C}}$, and of $\Psi_C$.
\begin{thm}\label{thm:full}
	For each exceptional Cartan type $C$, $\Xi_C$ is a full functor
	from $\threecat$ to $\Icat(L_C)$.
\end{thm}
Note that this result is {\em not}\/ a consequence of the first fundamental theorem of tensor invariant theory, as we have selected a particular subset of diagrams
in each $\Hom$ space.
\begin{proof}	
	Under the functors $\Psi_C$, $\theta_C$ and $\Phi$, the trace operations defined in Definition~\ref{defn:trace} on the various diagram categories are mapped to each other.
	In turn, this means that the inner product of the images of morphisms under $\Xi_C$ can be computed
	inside $\threecat$; the determinants of the Gram matrices of the bases of $A(\obj2)$, $B(\obj2,\obj3)$, $A(\obj3)$ have been computed in \eqref{eq:Gram2},
	\eqref{eq:Gram23}, \eqref{eq:Gram3}, respectively. Because the $\coeffs$-modules $\Hom_{\threecat}(\obj i,\obj j)$ are isomorphic to each other as long as $i+j$
	is kept constant, this covers all the cases $i+j=\obj4,\obj5,\obj6$; the cases $i+j\le \obj3$ are trivial and left as an exercise.
The case $i+j=\obj4$ already follows from Proposition~\ref{prop:A2interp} using the identifications induced by self-duality, but
we include it below in a uniform argument for $i+j=\obj4,\obj5,\obj6$.

	The rank of the Gram matrix is the rank of the image of $\Hom_{\threecat}(\obj i,\obj j)$ in $\Hom_{\Icat(L_C)}(\obj i,\obj j)$ as a free $\kk_{q^{1/s_C}}$-module.
	A simple linear algebra lemma states that the order of the zero of the determinant of the Gram matrix is an upper bound for the
	corank of the Gram matrix. Here, we need a slight refinement (because $\kk_{q^{1/s_C}}$ is not a field), so we go through the proof.
	Write $G$ for the Gram matrix, $G_0$ for its specialisation at $n=n_C=a/b$ in lowest terms, and $G=G_0+xG_1$ where
	$x=[b\nv-a]$. One has:
	\begin{equation}\label{eq:funnylemma}
		\det G = \det (G_0+xG_1)=\sum_{\substack{I\subseteq\{1,\ldots,d\}\\J\subseteq\{1,\ldots,d\}\\ |I|=|J|}}(-1)^{\sum_{i\in I}i+\sum_{j\in J}j}
		x^{|I|}
		\det (G_0)_{\bar I,\bar J} \det (G_1)_{I,J}
	\end{equation}
	where $\bar I$, $\bar J$ denote the complements of $I$, $J$, and $d$ is the size of $G$.

	If the rank of $G_0$ is $r$, then all its minors of size $>r$ vanish, which means $\det G$ has a zero of order at least $d-r$.

	For reference, we show now the ranks of $\Hom_{\Icat(L_C)}(i,j)$ for $k=i+j\le \obj6$; entries in bold indicate ranks
	that differ from the generic rank $d$, i.e., that of $\Hom_{\threecat}(i,j)$:
	\begin{equation}
		\begin{array}{r|c|cccccccccc}
			$C$ & n\backslash k & \obj0 & \obj1 & \obj2 & \obj3 & \obj4 & \obj5 & \obj6 \\ \hline
			\text{generic}& & 1 & 0 & 1 & 1 & 5 & 16 & 80  \\ \hline
			\varnothing& 1/5 & 1 & 0 & \mathbf{0} & \mathbf{0} & \mathbf{0} & \mathbf{0}  & \mathbf{0}  \\
			\OSp & 1/4 & 1 & 0 & 1 & 1 & 5 & 16 & \mathbf{65}  \\
			A_1& 1/3 & 1 & 0 & 1 & 1 & \mathbf{3} & \mathbf{6}  & \mathbf{15} \\
			A_2& 1/2 & 1 & 0 & 1 & 1 & 5 & 16 & \mathbf{75} \\
			G_2& 2/3 & 1 & 0 & 1 & 1 & 5 & 16 & 80  \\
			D_4& 1 & 1 & 0 & 1 & 1 & 5 & 16 & 80\\
			F_4& 3/2 & 1 & 0 & 1 & 1 & 5 & 16 & 80\\
			E_6& 2 & 1 & 0 & 1 & 1 & 5 & 16 & 80\\
			E_7& 3 & 1 & 0 & 1 & 1 & 5 & 16 & 80\\
			E_8& 5 & 1 & 0 & 1 & 1 & 5 & 16 & \mathbf{79}\\
		\end{array}
	\end{equation}
	The content of this table can for example be extracted from the data of \cite{cohen1996}, except for the $\OSp$ row;
	for the latter, the representation correspondence induced by the homomorphism in Proposition~\ref{prop:osp12} identifies the invariant subspaces with those for $U_{q'}(\mathfrak{sl}(2))$. Thus
	the $\OSp$ row is nothing but the series of dimensions of invariant subspaces of tensor powers of the five-dimensional irreducible
	representation of $U_{q'}(\mathfrak{sl}(2))$, which can be easily computed.

	For $i+j=k$, let $m_k$ be the order of vanishing of $\det G$ in $x$.
	The boxed factors in \eqref{eq:Gram2}, \eqref{eq:Gram23}, \eqref{eq:Gram3} give
	\[
	\begin{array}{c|ccccc|c}
		C & \varnothing & \OSp & A_1 & A_2 & E_8 & \text{other }C \\ \hline
		m_4 & 5 & 0 & 2 & 0 & 0 & 0 \\
		m_5 & 16 & 0 & 10 & 0 & 0 & 0 \\
		m_6 & 80 & 15 & 65 & 5 & 1 & 0
	\end{array}
	\]
	In each case $m_k=d-r$, where $r$ is the target rank above, so the rank bound proves surjectivity over $\bQ(q^{1/s_C})$.

	Moreover, $c_C=\left.(x^{-(d-r)}\det G)\right|_C$ is a unit of $R=\kk_{q^{1/s_C}}$.
	Indeed, every nonvanishing factor $\theta_C([u\nv+v])$ is a ratio of nonzero integer quantum integers, hence a unit of $R$.
	Every vanishing factor has $u=kb$, $v=-ka$ for a nonzero integer $k$, and
	\[
	\left.\frac{[k(b\nv-a)]}{x}\right|_C=k\in\bQ^\times.
	\]
	Applying this to numerator and denominator factors proves the assertion after cancellation.
	For example, at $C=A_2$, $x=[2\nv-1]$, and the relevant quotient in \eqref{eq:Gram3} gives
	\[
	\left.x^{-5}\frac{[4\nv-2]^{75}}{[2\nv-1]^{70}}\right|_{A_2}
	=\left.\left(\frac{[4\nv-2]}{[2\nv-1]}\right)^{75}\right|_{A_2}=2^{75}.
	\]

	\begin{samepage}
	Since $R$ is a PID, choose source and target bases putting the $d\times r$ matrix of images in Smith normal form:
	\[
	P=\begin{pmatrix}
		\operatorname{diag}(\alpha_1,\ldots,\alpha_r) \\
		0_{(d-r)\times r}
	\end{pmatrix}.
	\]
	\end{samepage}
	The source basis change multiplies $\det G$ by a unit square, so $c_C$ remains a unit.
	Now $G_0=P G_0' P^t$, where $G_0'$ is the Gram matrix of the chosen target basis.
	Thus $G_0$ is zero outside its top left $r\times r$ square, and \eqref{eq:funnylemma} gives the coefficient of $x^{d-r}$ as
	\[
	c_C=\prod_{i=1}^r \alpha_i^2\ \det G_0'\ \det(G_1|_C)_{I,I},
	\qquad I=\{r+1,\ldots,d\}.
	\]
	Since $c_C$ is a unit, all the $\alpha_i$ are units, proving surjectivity.

	We conclude that $\Xi_C$ is a full functor.
\end{proof}

Recall, Definition~\ref{def:interpalgebra} which defines an interpolating algebra. As a corollary of Theorem~\ref{thm:full}, we have:
\begin{prop}\label{prop:A3interp} The algebra $A(\obj3)$ is an interpolating algebra.
\end{prop}
\begin{proof} The underlying $\coeffs$-module is finitely generated and free by Corollary~\ref{cor:A3semisimple}.
	After the base change $\coeffs\to \bQ(\pp,\qq)$, the algebra, $A(3)\otimes \bQ(\pp,\qq)$ is the direct sum of the matrix algebras given in Lemma~\ref{lem:A3rep}.
	The trace pairing is non-degenerate by \eqref{eq:Gram3}.
	For $C$ an exceptional Cartan type, let $A_C(3)=A(3)\otimes_{\theta_C} \kk_{q^{1/s_C}}$ be the $\kk_{q^{1/s_C}}$-algebra given by the base change $\theta_C\colon \coeffs\to \kk_{q^{1/s_C}}$.
	By Theorem~\ref{thm:full}, evaluation is a surjective homomorphism preserving the ideals and the trace map.
	Nondegeneracy of the target trace pairing identifies its kernel with $\mathcal{N}_C$, so it induces the isomorphism
	\begin{equation*}
		A_C(3)/\mathcal{N}_C \cong \End_{\fU_{q}(C)}(\otimes^3 L_C).
	\end{equation*}
\end{proof}


\section{Dilute categories}\label{sec:dilute}

\subsection{The dilute category}\label{ssec:dilute}

We introduce a notion of ``dilute'' category attached to each of our diagrammatic categories.

Given a linear category $\Dcat$ whose objects are elements of $\bN$, we define a new linear category
$\dilute\Dcat$ whose objects are the same as those of $\Dcat$;
the morphisms are
\[
  \Hom_{\dilute\Dcat}(\obj k,\obj l)=\bigoplus_{\substack{a\subseteq \{1,\ldots,k\}\\ b\subseteq\{1,\ldots,l\}}}
  \Hom_{\dilutep\Dcat}(a,b)
\]
where $\Hom_{\dilutep\Dcat}(a,b)$ is simply a copy of $\Hom_{\Dcat}(|a|,|b|)$; we denote $\varphi_{\Dcat,a,b}$ the isomorphism
from the latter to the former.

The composition rule is
    \[
      \varphi_{\Dcat,a_1,b_1}(u) \circ \varphi_{\Dcat,a_2,b_2}(v) = \begin{cases} \varphi_{\Dcat,a_1,b_2}(u\circ v) & b_1=a_2 \\ 0 & b_1\ne a_2\end{cases} \qquad
      \begin{aligned}
        &a_1\subseteq \{1,\ldots,k\}
        \\
        &a_2,b_1\subseteq\{1,\ldots,l\}
        \\
        &b_2\subseteq\{1,\ldots,m\}
      \end{aligned}
    \]
The identity morphisms are $\dilute{\id}_k:=\sum_{a\subseteq\{1,\ldots,k\}} \varphi_{\Dcat,a,a}(\id_{|a|})$.

    If $\Dcat$ is monoidal, so is $\dilute\Dcat$, where
    \[
      \varphi_{\Dcat,a_1,b_1}(u) \otimes \varphi_{\Dcat,a_2,b_2}(v) = \varphi_{\Dcat,a_1\cup(a_2+k_1),b_1\cup(b_2+l_1)}(u\otimes v)\qquad
      \begin{aligned}
        &a_i \subseteq \{1,\ldots,k_i\}
        \\
        & b_i\subseteq \{1,\ldots,l_i\}
      \end{aligned}
    \]

 If $\Dcat$ is braided (resp.\ ribbon), so is $\dilute\Dcat$, where
 \begin{equation}\label{eq:Sdilute}
   \dilute S = S + \varphi_{\Dcat,\{1\},\{2\}}(\id_{\obj1})
   + \varphi_{\Dcat,\{2\},\{1\}}(\id_{\obj1})
   + \varphi_{\Dcat,\varnothing,\varnothing}(\id_{\obj0})
 \end{equation}

 Write $[k]=\{1,\ldots,k\}$. The maps $\varphi_{\Dcat,[k],[l]}$ identify ordinary diagrams with the full-occupancy blocks of $\dilute\Dcat$ and preserve composition and, when defined, tensor products.
 They send $\id_k$ to the full-occupancy idempotent $e_k=\varphi_{\Dcat,[k],[k]}(\id_k)$.
 Equivalently, they define a fully faithful functor $\Dcat\to\operatorname{Kar}(\dilute\Dcat)$, where $\operatorname{Kar}$ denotes idempotent completion, sending $k$ to $(k,e_k)$.
   
    Whenever we have an object-preserving linear functor $\mathcal F$ between two such categories $\Dcat_1$ and $\Dcat_2$, we have a corresponding functor
    \[
      \dilute{\mathcal F}(\varphi_{\Dcat_1,a,b}(u)) = \varphi_{\Dcat_2,a,b}(\mathcal F(u))
    \]
    between $\dilute \Dcat_1$ and $\dilute \Dcat_2$ such that $\dilute{(\mathcal F\circ\mathcal G)} = \dilute{\mathcal F} \circ \dilute{\mathcal G}$.

      \begin{rem}
        We could have defined the category $\dilutep\Dcat$, whose objects are subsets of $\{1,\ldots,k\}$ for any $k$,
        and whose morphisms are the $\Hom_{\dilutep\Dcat}(a,b)$ introduced above; however, this definition is neither
        necessary nor convenient for our purposes.
      \end{rem}

      The category $\dilute\Pcat$ has a natural interpretation in terms of diagrams which are identical to the ones
      defined in \S\ref{ssec:pivot}, with the added possibility of marks along either boundary:
      
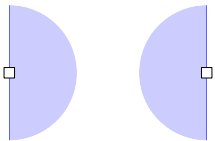
\[
	\begin{tikzpicture}[bwwpicture]%
          \begin{scope}
	\clip (1,0) -- (-1,0) arc(180:360:1); \fill[bgdfill] (1,0) -- (-1,0) arc(180:360:1);
        \draw[bgddraw]  (1,0) -- (-1,0);%
  \end{scope}
  \node[fill=white,rectangle,inner sep=2pt,draw=black,scale=.99114,line width=.00281cm] {};
	\end{tikzpicture}
        \qquad
	\begin{tikzpicture}[bwwpicture]
          \begin{scope}
	\clip (-1,0) -- (1,0) arc(0:180:1); \fill[bgdfill] (-1,0) -- (1,0) arc(0:180:1);
        \draw[bgddraw] (-1,0) -- (1,0);%
      \end{scope}
  \node[fill=white,rectangle,inner sep=2pt,draw=black,scale=.99114,line width=.00281cm] {};
    \end{tikzpicture}
      \]
      The points of contact of the (red) lines with the boundary correspond to the elements of the subset
      $a\subseteq \{1,\ldots,k\}$, whereas the new marks correspond to its complement, reading the boundary top to bottom;
      for instance,
      \[
        \diag[.5]{dilute-ex}
      \]
      belongs to $\Hom_{\dilute\Pcat}(\obj4,\obj3)$, and more precisely to the summand indexed by subsets $a=\{1,2,3\}$ and $b=\{1,3\}$.

      The composition rule simply means that locations of marks should agree on the common boundary when concatenating, in which case these marks are erased; otherwise the result is zero.

      Also, the dilute braiding \eqref{eq:Sdilute} is nothing but
      \[
        \dilute S
   = \overrect + \diag[.337]{dilute2_2_7} + \diag[.337]{dilute2_2_6} + \diag[.337]{dilute2_2_15}
 \]

      The same interpretation applies to our other categories, namely $\dilute\Rcat$ and $\dilute\threecat$; the latter will be discussed in detail in the next
      section.
      
      The set of objects of the category $\Icat(L_C)$, namely modules $L_C^{\otimes k}$, $k\in\bN$, is in natural
      bijection with $\bN$, and therefore $\Icat(L_C)$ has a dilute counterpart $\dilute{\Icat(L_C)}$ as well, which is the full subcategory of $\fU_{\kk_{q^{1/s_C}}}(C)$-mod
      whose objects are the modules $(L_C\oplus I_C)^{\otimes k}$, $k\in\bN$.
      An immediate corollary of Theorem~\ref{thm:full} is
      \begin{cor}
  For each exceptional Cartan type $C$, $\dilute{\Xi_C}$ is a full functor
  from $\dilute\threecat$ to $\dilute{\Icat(L_C)}$.
      \end{cor}
      
\subsection{The dilute truncated category}\label{ssec:diltrunc}
We provide for illustration a few bases of $\dilute B(r,s):= \Hom_{\dilute\threecat}(r,s)$:
\def\littleleftbrace{\left\{\vphantom{\diag{dilute0_1_0}}\right.}
\def\littlerightbrace{\left.\vphantom{\diag{dilute0_1_0}}\right\}}
\def\bigleftbrace{\left\{\vphantom{\diag{dilute2_2_0}}\right.}
\def\bigrightbrace{\left.\vphantom{\diag{dilute2_2_0}}\right\}}
\begin{align*}
  \dilute B(\obj0,\obj1)
  =\coeffs\text{-span of }\littleleftbrace&\diag{dilute0_1_0}\littlerightbrace
  \\
  \dilute B(\obj1,\obj1)
  =\coeffs\text{-span of }\littleleftbrace&
                            \diag[.28]{dilute1_1_0},\ \diag[.28]{dilute1_1_1}
                            \littlerightbrace
  \\
  \dilute B(\obj1,\obj2)
  =\coeffs\text{-span of }\bigleftbrace&
                            \foreach\i in {0,...,4} { \diag[.28]{dilute1_2_\i}\ifnum\i=4\else,\ \fi }
                            \bigrightbrace
  \\
  \dilute B(\obj2,\obj2)
  =\coeffs\text{-span of }\bigleftbrace&
                            \foreach\i in {0,...,4} { \diag[.28]{dilute2_2_\i},\ }
  \\
                          &\foreach\i in {5,...,9} { \diag[.28]{dilute2_2_\i},\ }
  \\
                          &\foreach\i in {10,...,15} { \diag[.28]{dilute2_2_\i}\ifnum\i=15\else,\ \fi}\bigrightbrace
\end{align*}

      The representation content of these algebras and bimodules is easy to obtain from that of the nondilute one: they are related by the binomial transform
      \[
        \dilute V_{k,R} = \bigoplus_{j=0}^k \binom kj V_{j,R}
      \]
      where $V_{k,R}$ (resp.\ $\dilute V_{k,R}$) is the space of the irreducible representation of $A(k)$ (resp.\ $\dilute A(k)$) indexed by $R$.

      In terms of dimensions, we have
      \[
	\begin{array}{r|c|c|c|ccc|ccccccc}
          &{k\backslash R} 
          & I & L & X_2 & Y_2 & Y_2^\ast & X_3 & Y_3 & Y_3^\ast & A & C & C^\ast  \\ \hline
          &\obj0 & 1 &&&&&&&&&\\
          \llap{$\dim V_{k,R}$}&\obj1 & 0 & 1 &&&&&&&\\
          &\obj2 & 1 & 1 & 1 & 1 & 1 \\
          &\obj3 & 1 & 5 & 4 & 3 & 3 & 1 & 1 & 1 & 3 & 2 &2 \\\hline
          &\obj0 & 1 &&&&&&&&&\\
          \llap{$\dim \dilute V_{k,R}$}&1 & 1 & 1 &&&&&&\\
          &\obj2 & 2 & 3 & 1 & 1 & 1 \\
          &\obj3 & 5 & 11 & 7 & 6 & 6 &  1 & 1 & 1 & 3 & 2 &2 
        \end{array}
      \]
      The rank of $\dilute B(r,s)$ as a free $\coeffs$-module is $\sum_R \dim \dilute V_{r,R}\dim \dilute V_{s,R}$.
      As expected, it only depends on $r+s$ and forms the sequence $(1,1,2,5,16,62,287)$ for $r+s=0,\ldots,6$.

      The category $\dilute\threecat$ possesses an increasing sequence of ideals $\dilute J{}^{(i)}$ given by cut length,
      cf \S\ref{ssec:ideals}. 
      The rank of $\dilute J{}^{(i)}(\obj r,\obj s)/\dilute J{}^{(i-1)}(\obj r,\obj s)$ is $\sum_{R\in\mathit{Irr}^i} \dim \dilute V_{r,R} \dim \dilute V_{s,R}$.
      In particular, for the algebras $\dilute A(\obj k):=\End_{\dilute\threecat}(\obj k)$, we find:
      \[
        \begin{array} {c | c c c c}
          k \backslash i & 0 & 1 & 2 & 3 \\\hline
          \obj0 & 1 \\
          \obj1 & 1 & 1\\
          \obj2 & 4 & 9 & 3\\
          \obj3 & 25 & 121 & 121 &20
        \end{array}
      \]
      As is obvious from the diagrams, the top part of $\dilute A(k)$ is the same algebra $\hat{A}(k)$ as in $\threecat$.

We have the dilute analogue of Propositions~\ref{prop:A2interp} and \ref{prop:A3interp}, which follows directly from them:
\begin{prop}\label{prop:dilA3interp} The algebras $\dilute{A}(\obj2)$ and $\dilute{A}(\obj3)$  are interpolating algebras.
\end{prop}



\newcommand\coeffsu{{\kk_{\pp,\qq,\uu}}\index{coefficients}}%
\newcommand\coeffsuv{{\kk_{\pp,\qq,\uu,\vv}}\index{coefficients}}%

\section{The Yang--Baxter equation}\label{sec:yangbaxter}
We now proceed to define the $R$-matrix. As advertised in the introduction,
it will interpolate the $R$-matrices of $U_q(C^{(1)})$ mapping $V_{\uu_1}\otimes V_{\uu_2}$ to $V_{\uu_2}\otimes V_{\uu_1}$,
where the $\uu_1$ and $\uu_2$ are ``spectral parameters'', affecting the action of $U_q(C^{(1)})$ but not that of the subalgebra $U_q(C)$, so
that as a representation of $U_q(C)$, $V_{\uu_1}\cong V_{\uu_2}\cong V=L\oplus I$.
In \S\ref{ssec:specparam} we introduce more precisely extra variables $\uu$ and $\vv$ corresponding to ratios
of spectral parameters, and in \S\ref{ssec:norm} we change to a more
convenient convention of normalisation of our diagrams, which makes the expression of the $R$-matrix to be given in \S\ref{ssec:Rmat} more symmetric.
The rest of this section is devoted to studying (and proving) the various properties of the $R$-matrix.

\subsection{The spectral parameters}\label{ssec:specparam}
In this section we introduce additional variables $\uu$, $\vv$ (corresponding to ratios of spectral parameters, e.g.,
in the Yang--Baxter equation \eqref{eq:YB}, we shall substitute $\uu=\uu_1/\uu_2$, $\vv=\uu_2/\uu_3$),
following the prescription of \S\ref{ssec:blowups}.
We extend the notation of Section~\ref{sec:blowup} and define
\begin{equation*}
	[a_n\nv + a_x\xx +a_y\yy + a] \coloneqq \frac{\pp^{a_n}\uu^{a_x}\vv^{a_y}\qq^{a}-\pp^{-a_n}\uu^{-a_x}\vv^{-a_y}\qq^{-a}}{\qq-\qq^{-1}}
\end{equation*}
for $a_n,a_x,a_y,a\in\bZ$. Heuristically, $\uu=q^\xx$ and $\vv=q^\yy$.

The ring $A_{\pp,\qq,\uu,\vv} \subset \bQ(\pp,\qq,\uu,\vv)$ is then defined to be the subring generated by $\pp^\pm,\uu^\pm,\vv^\pm,\qq^\pm$
and $[a_n\nv + a_x\xx +a_y\yy + a]$ for $a_n,a_x,a_y,a\in\bZ$.

The homomorphism $\psi\colon A_{\pp,\qq,\uu,\vv}\to \bQ[\nv,\xx,\yy]$ is defined by $\pp,\qq,\uu,\vv\mapsto 1$ and
\begin{equation}\label{eq:psispectral}
	[a_n\nv + a_x\xx +a_y\yy + a]\mapsto a_n\nv + a_x\xx +a_y\yy + a \qquad\text{for $a_n,a_x,a_y,a\in\bZ$}
\end{equation}

$A_{\pp,\qq,\uu,\vv}$ possesses subrings $A_{\pp,\qq,\uu}$ and $A_{\pp,\qq}$ corresponding respectively
to elements with no dependence on $\vv$, or $\uu$ and $\vv$.

Finally, we allow some localisation: we define $\coeffsuv$ to be the subring of $\bQ(\pp,\qq,\uu,\vv)$ generated by $A_{\pp,\qq,\uu,\vv}$,
$\coeffs$ (cf \S\ref{sec:spec}), and the inverses of $[z+k\,\nv+l]$ for $z\in\{\xx,\yy,\xx+\yy\}$ and $k,l\in\bZ$ satisfying $[k\,\nv+l]\in\Sigma$.
Define $\coeffsu$ analogously using $z=\xx$. At $\uu=\vv=1$, every inverted factor specialises to an element of $\Sigma$, so evaluation defines a homomorphism $\coeffsuv\to\coeffs$; in particular, $\uu=1$ defines $\coeffsu\to\coeffs$.
This localisation corresponds to the well-known fact that the $R$-matrix of $U_q(C^{(1)})$
is not well-defined for all values of $\uu_1$ and $\uu_2$ (and correspondingly, there are specialisations of $\uu=\uu_1/\uu_2$ for which
$V_{\uu_1} \otimes V_{\uu_2}$ and $V_{\uu_2}\otimes V_{\uu_1}$ are {\em not}\/ isomorphic).

\subsection{The change of normalisation}\label{ssec:norm}
We come back to the issue of normalisation of $\phi$ and $\tau$; the latter was fixed in \eqref{eq:values}, and we now discuss how to relax this normalisation condition.

Given an invertible element $\alpha\in \coeffs^\times$, consider the following map from $B(r,s)$ to $B(r,s)$
viewed as a free $\coeffs$-module with basis of diagrams $D$:
\[
D\in B(r,s)\mapsto \alpha^{(-|D|+r-s)/2}D
\]
where $|D|$ denotes the number of trivalent vertices of $D$, and we have used the easy lemma that $|D|\equiv r-s\pmod 2$.

In particular, note the action on the generators:
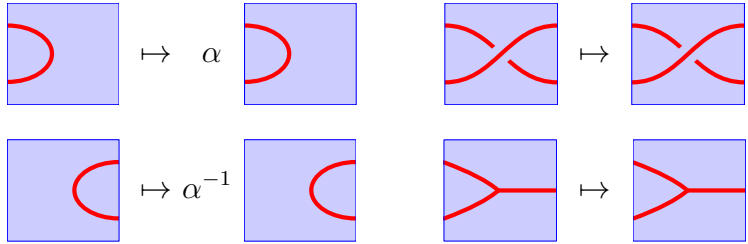
\begin{alignat*}{3}
  \cuprect &\mapsto&\ \alpha\ &\cuprect
  &\qquad
  \overrect&\mapsto \overrect
  \\
  \caprect &\mapsto&\ \alpha^{-1}&\caprect
  &\qquad
  \trivalentrect &\mapsto \trivalentrect
\end{alignat*}
Put together, these maps define an equivalence of categories between $\threecat$ and $\threecat(\alpha)$, where the latter
differs from $\threecat$ only in that:
\begin{itemize}
\item All the parameters stay the same (e.g., $\delta=
	\begin{tikzpicture}[bwwrect={(-.8,-.8)}{(.8,.8)},scale=.7]
		\draw  (0,0) circle [radius=0.5];
	\end{tikzpicture}
$ is manifestly invariant) except $\phi$ and $\tau$ take their general form
\begin{align*}
\phi &= \alpha\,\frac{\left[3\,\nv\right]\ \left[6\,\nv+2\right]\ \left[4\,\nv-2\right]}{\left[\nv\right]\ \left[3\,\nv+1\right]\ \left[2\,\nv-1\right]} \\
	\tau &= \alpha\,\frac{[2\,\nv]}{[4\,\nv]}\left(\frac{[6\,\nv+2]}{[3\,\nv+1]}\frac{[4\,\nv-2]}{[2\nv-1]}\frac{[3\,\nv]}{[\nv]}+(\qq-\qq^{-1})^2[\nv+1]\frac{[5\nv]}{[\nv]}\right)
\end{align*}
To see this, apply the map above to the relation $\phi\,\delta=
\begin{tikzpicture}[bwwrect={(-1.5,-1)}{(1.25,1)},rotate=-90,scale=.9]
\draw (0,.5) -- (.5,.5) arc [start angle=90,end angle=-90,radius=0.5] -- (0,-.5);
\draw (0,.5) .. controls (-.4,.25) .. (-.5,.25) arc [start angle=90,end angle=270,radius=0.25] -- (-.5,-.25) .. controls (-.4,-.25) .. (0,-.5);
\draw (0,.5) .. controls (-.4,.75) .. (-.5,.75) arc [start angle=90,end angle=270,radius=0.75] -- (-.5,-.75) .. controls (-.4,-.75) .. (0,-.5);
\end{tikzpicture}
$, and similarly for $\tau$.
\item All the relations of its algebras and bimodules are transformed in the obvious way: given a relation
$
r = \sum_{i=1}^k \kappa_i D_i \in B(r,s)
$
where the $\kappa_i\in\coeffs$ and the $D_i$ are diagrams, replace it with
$
r(\alpha) = \sum_{i=1}^k \kappa_i \alpha^{(-|D_i|+r-s)/2}D_i
$.
\end{itemize}
The same argument works for the corresponding dilute categories, where diagrams are rescaled by a factor
of $\alpha^{(-|D|+r-s)/2}$ where $r$ (resp.\ $s$) is now the number of left (resp.\ right) external legs.

In this section, we shall use a different (and perhaps less natural) value of $\alpha$ than
in the rest of the paper, where it was set to $1$; namely, we shall work in $\threecat(1/[\nv+1])$.
For instance, the skein relation \eqref{eq:skein} now takes the form
\begin{multline}\label{eq:skeinc}
	\frac{[6\nv]}{[3\nv][2\nv]}\:
	\left(\overcirc \ -\
          \undercirc\right)
        \ =
\\
	-(\qq-\qq^{-1})[\nv+1]\:\biggl(
        (\qq-\qq^{-1})^2\left(\Icirc\ -\ \Ucirc\right)
        +\Kcirc\ -\ \Hcirc\biggr)
\end{multline}
This convention will also make the expression of the $R$-matrix more symmetric.

\subsection{The \texorpdfstring{$R$}{R}-matrix}\label{ssec:Rmat}
We define the $R$-matrix, denoted $\check R(\uu)$ by:
\begin{multline}\label{eq:Rmat}
  \check R(\uu) = \frac{1}{[\xx+\nv+1]\ [\xx-1]}
  \left(%
    [\nv+1]-\frac{\pp^{-1}\uu^{-1} \dilute S - \pp\, \uu\, \dilute S{}^{-1}}{\qq-\qq^{-1}}[\xx]
\right.
\\
+\frac{\left[2\,\nv\right]\ \left[3\,\nv\right]\ \left[\nv+1\right]\ [\xx]}{\left[6\,\nv\right]\ \left[\xx+2\,\nv\right]}
\left(
-[\nv]K+T-\frac{1}{[\nv]}B-\frac{[2\xx]}{[\xx]\ [\xx+3\,\nv]}E\right.
\\
\left.\left.
+\frac{1}{[\nv]}\left(\frac{[6\,\nv]\ [\xx+2\,\nv]}{[3\,\nv]\ [\xx+3\,\nv]}-1\right)U
-\left(\frac{[2\,\xx+12\,\nv]}{[\xx+6\,\nv]\ [\xx+3\,\nv]}+\frac{2}{[\nv]}\right)O\right)\right)
\end{multline}
where
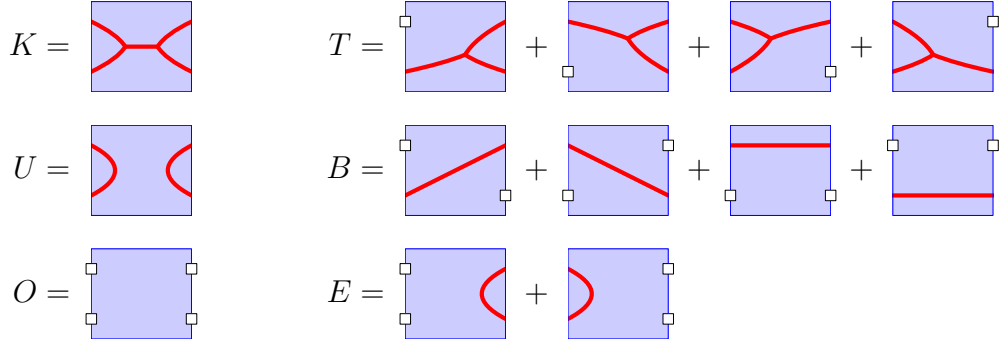
\begin{align*}
K&=\diag{dilute2_2_2}
&
T&=\diag{dilute2_2_9}+\diag{dilute2_2_10}+\diag{dilute2_2_11}+\diag{dilute2_2_12}
\\
U&=\diag{dilute2_2_1}
&
B&=\diag{dilute2_2_6}+\diag{dilute2_2_7}+\diag{dilute2_2_5}+\diag{dilute2_2_8}
\\
O&=\diag{dilute2_2_15}
&
E&=\diag{dilute2_2_13}+\diag{dilute2_2_14}
\end{align*}
and we also recall from \S\ref{ssec:dilute}
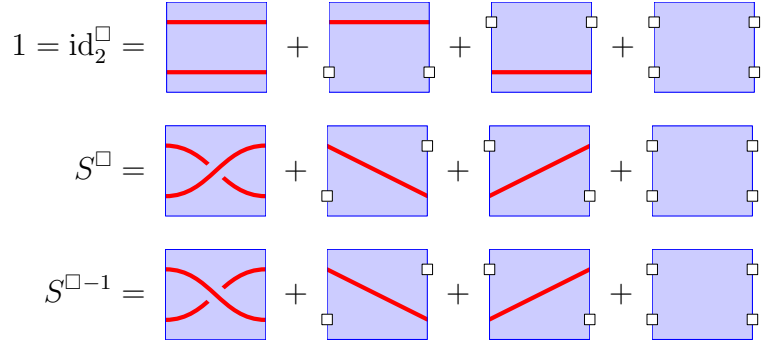
\begin{align*}
1=\dilute{\id}_{\obj2}&=\diag{dilute2_2_0}+\diag{dilute2_2_5}+\diag{dilute2_2_8}+\diag{dilute2_2_15}
\\
   \dilute S &= \overrect[.89] + \diag{dilute2_2_7} + \diag{dilute2_2_6} + \diag{dilute2_2_15}
\\
   \dilute S{}^{-1} &= \underrect[.89] + \diag{dilute2_2_7} + \diag{dilute2_2_6} + \diag{dilute2_2_15}
\end{align*}

We claim that $\check R(\uu)\in\dilute A(\obj2)\otimes_{\coeffs} \coeffsu$.
As defined in \eqref{eq:Rmat}, $\check R(\uu)$ naively only makes sense after tensoring with the fraction field.
The ratios $[2\xx]/[\xx]=\uu+\uu^{-1}$ and $[2\xx+12\nv]/[\xx+6\nv]=\pp^6\uu+\pp^{-6}\uu^{-1}$ are Laurent polynomials and require no additional inverses.
Most coefficients obviously live in $\coeffsu$, except the part
$\frac{\pp^{-1}\uu^{-1} \dilute S - \pp\, \uu\, \dilute S{}^{-1}}{\qq-\qq^{-1}}$ which requires discussion.
First, the ``dilute'' parts of $\dilute S$ and $\dilute S{}^{-1}$ are identical and will contribute a coefficient of $-[\xx+\nv]$, so we can simplify
this expression to $\frac{\pp^{-1}\uu^{-1} S - \pp\, \uu\, S{}^{-1}}{\qq-\qq^{-1}}$. Now write
\[
\frac{\pp^{-1}\uu^{-1} S - \pp\, \uu\, S{}^{-1}}{\qq-\qq^{-1}}
=
\frac{\pp^{-1}\uu^{-1} - \pp\, \uu}{\qq-\qq^{-1}}S^{-1}
+
\frac{\pp^{-1}\uu^{-1} (S - S{}^{-1})}{\qq-\qq^{-1}}
\]
and apply the skein relation \eqref{eq:skeinc}.

\editorialnote[gray]{}
For future purposes, we note that the combinations of diagrams in \eqref{eq:Rmat} are not algebraically independent; for example, we have
\begin{align}
  \label{eq:genrelsa}
T^2&=2K+\phi\, B
\\
E^2&=U+\delta\,O
\\  \label{eq:genrelsb}
E^2\dilute S&=p^{12}U+\delta\,O
\end{align}

\subsection{Main theorem}
Define $\check R_1(\uu)=\check R(\uu)\otimes \dilute{\id}_{\obj1}$, $\check R_2(\uu)=\dilute{\id}_{\obj1}\otimes\check R(\uu)$.
\begin{thm}\label{thm:main}
The $R$-matrix \eqref{eq:Rmat} satisfies the following identities:
\begin{itemize}
\item Special values:
\begin{align}\label{eq:specvals}
\lim_{\uu\to0}\check R(\uu) &= \dilute S
\\
\check R(1) &= -1
\\\label{eq:specvalsb}
\lim_{\uu\to\infty}\check R(\uu) &= \dilute S{}^{-1}
\end{align}
\item Invariance under bar involution:
\begin{equation}\label{eq:conj}
\overline{\check R(\uu)}=\check R(\uu)
\end{equation}
\item The unitarity equation:
\begin{equation}\label{eq:unit}
\check R(\uu) \check R(\uu^{-1}) = 1
\end{equation}
\item The Yang--Baxter equation:
\begin{equation}\label{eq:YBE}
\check R_1(\uu) \check R_2(\uu\,\vv) \check R_1(\vv) = \check R_2(\vv) \check R_1(\uu\,\vv) \check R_2(\uu)
\end{equation}
\item The crossing symmetry, interpreted over $\bQ(\pp,\qq,\uu)$:
\begin{equation}\label{eq:crossing}
\Rot(\check R(\uu)) = \frac{\left[\xx\right]\ \left[\nv+\xx\right]\ \left[3\,\nv+\xx+1\right]\ \left[2\,\nv+\xx-1\right]}{\left[\nv+\xx+1\right]\ \left[\xx-1\right]\ \left[2\,\nv+\xx\right]\ \left[3\,\nv+\xx\right]}
\ \check R(\pp^{-3}\uu^{-1})
\end{equation}
\end{itemize}
\end{thm}

The proof of the equalities \eqref{eq:specvals}--\eqref{eq:conj} is elementary and left as an exercise.

In what follows, we provide proofs of the remaining identities of Theorem~\ref{thm:main}.

\begin{rem}
Applying the homomorphism $\psi$ in \eqref{eq:psispectral} to $\check R(\uu)$ results in the rational $R$-matrix $\check R(\xx)$, which was anticipated in
\cite{Westbury2003}, and which interpolates the $R$-matrices associated to {\em Yangians}\/ of exceptional Lie algebras that appear in \cite{chari1991a}.
\end{rem}

\subsection{Proof of the Yang--Baxter equation for the top part}\label{sec:YBEtop}
In this section, we show how one can prove the Yang--Baxter equation
\eqref{eq:YBE} directly, though we only give the details for the top part
${\dilute A}(\obj3)/{\dilute J}^{(2)}(\obj3,\obj3)\cong\hat{A}(\obj 3)$.
The $R$-matrix simplifies there to
\[
  \check R(\uu) = \frac{1}{[\xx+\nv+1]\ [\xx-1]}
  \left(
    [\nv+1]-\frac{\pp^{-1}\uu^{-1} S - \pp\, \uu\, S^{-1}}{\qq-\qq^{-1}}[\xx]
  \right)
\]
For practical purposes, one can get rid of the denominator before
substituting this expression into \eqref{eq:YBE}. Computing the difference
between the l.h.s.\ and the r.h.s.\ produces a Laurent polynomial in $\uu$
and $\vv$, all the coefficients of which must vanish. By direct calculation,
one finds that these coefficients give rise to 8 linearly independent equations:
\begin{align*}
  0&=S_1S_2S_1-S_2S_1S_2
  \\
  0&=S_2S_1S_2^{-1}-S_1^{-1}S_2S_1
  \\
  0&=S_2S_1^{-1}S_2^{-1}-S_1^{-1}S_2^{-1}S_1
  \\
  0&=S_1S_2S_1^{-1}-S_2^{-1}S_1S_2
  \\
  0&=S_1S_2^{-1}S_1^{-1}-S_2^{-1}S_1^{-1}S_2
  \\
  0&=S_1^{-1}S_2^{-1}S_1^{-1}-S_2^{-1}S_1^{-1}S_2^{-1}
  \\
  0&=  S_1^{-2}(S_1+ss^*)(S_1+1)(S_1-s)(S_1-s^*)-  S_2^{-2}(S_2+ss^*)(S_2+1)(S_2-s)(S_2-s^*)
  \\
  0&=
ss^*(S_2S_1^{-1}S_2-S_1S_2^{-1}S_1)
+(ss^*)^{2}(S_1^{-1}S_2S_1^{-1}-S_2^{-1}S_1S_2^{-1})
\\
&+s(s-1)s^*(s^*-1)(S_2S_1^{-1}+S_1^{-1}S_2-S_2^{-1}S_1-S_1S_2^{-1})
\\
&
+(s-1)(s^*-1)(S_2^{2}-S_1^{2})
+(s-1)^{2}(s^*-1)^{2}(S_2-S_1)
\end{align*}
where for compactness we have written $s=\qq^{-2}$, $s^*=\qq^2\pp^2$.
The first 6 equations are equivalent to the braid relation \eqref{eq:braidtop}.
The seventh equation is a direct consequence of the cubic relation \eqref{eq:Scubic} satisfied by $S_1$ and $S_2$.
Finally, substituting $S_i^{-1}=-1+s^{-1}+s^*{}^{-1}+(s^{-1}+s^*{}^{-1}-(ss^*)^{-1})S_i-(ss^*)^{-1}S_i^2$ (from the cubic relation again) in the last equation results in
\begin{multline*}
S_1^{2}S_2S_1^{2}-S_2^{2}S_1S_2^{2}
+(s^*+s-1)(S_2S_1S_2^{2}+S_2^{2}S_1S_2-S_1^{2}S_2S_1-S_1S_2S_1^{2})
+S_1S_2^{2}S_1-S_2S_1^{2}S_2
\\
+(s^*+s-2)(s^*+s-1)(S_1S_2S_1-S_2S_1S_2)
-S_1^{2}S_2+S_1S_2^{2}-S_2S_1^{2}+S_2^{2}S_1
-S_1^{2}+S_2^{2}-S_1+S_2
\end{multline*}
which up to the braid relation \eqref{eq:braidtop} is nothing but the additional relation \eqref{eq:addrel} satisfied by $S_1$ and $S_2$.

In principle, one can use the same strategy to prove the Yang--Baxter equation
in the whole of $\dilute A(\obj3)$, either by using a noncommutative Gr\"obner
basis as in the proof of Proposition~\ref{prop:tildeA3},
or in fact by simple application
of the diagrammatic rules of Appendix~\ref{app:skein} --
as has been noted before, only the top part involves nontrivial overlap,
so once the top part has been shown to be zero using the reasoning above,
the remaining expression in ${\dilute J}^{(2)}(\obj3,\obj3)$ is amenable to direct
simplification using diagrammatic substitutions. We have successfully
implemented both methods by computer, but neither proof is human-readable.

Instead, we propose in what follows a proof of both unitarity equation
\eqref{eq:unit} and Yang--Baxter equation \eqref{eq:YBE} by going over
to representations, in the same spirit as some of the proofs of
\S\ref{sec:threestring}.

\subsection{Proof of unitarity}\label{sec:unit}
We refer to \S\ref{ssec:diltrunc} for the representation content of $\dilute A(2)$.

The top representations are one-dimensional, sending $S$ to $-1$, $\qq^{-2}$, $\qq^2\pp^2$, respectively,
and all other generators in \eqref{eq:Rmat} to $0$. Explicitly, one finds
\begin{align*}
\dilute\rho_{2,X_2}(\check R(\uu)) &= -1
\\
\dilute\rho_{2,Y_2}(\check R(\uu)) &= \frac{[\xx+1]}{[\xx-1]}
\\
\dilute\rho_{2,Y_2^*}(\check R(\uu)) &=\frac{[\xx-\nv-1]}{[\xx+\nv+1]}
\end{align*}
which obviously satisfy \eqref{eq:unit}.

There is a single cut length $1$ representation, of dimension $3$,
which one can think of as the action on the top part of $\dilute B(\obj2,\obj1)$, which is spanned by
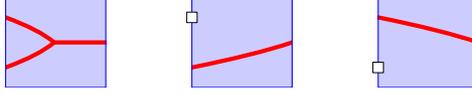
\[
\diag{dilute2_1_2}
\qquad
\diag{dilute2_1_0}
\qquad
\diag{dilute2_1_1}
\]

In this basis, the matrices of the generators occurring in \eqref{eq:Rmat} are
\begin{align*}
\dilute\rho_{2,L}(\dilute S) &=
\begin{pmatrix}
-\pp^6 & 0 & 0
\\
0 & 0 & 1
\\
  0 & 1 & 0
\end{pmatrix}
&
\dilute\rho_{2,L}(T)&=\begin{pmatrix}
0 & 1 & 1
\\
\phi & 0 & 0
\\
\phi & 0 & 0
\end{pmatrix}
\\
\dilute\rho_{2,L}(B)&=\begin{pmatrix}
0 & 0 & 0
\\
0 & 1 & 1
\\
0 & 1 & 1
\end{pmatrix}
&
\dilute\rho_{2,L}(K)&=\begin{pmatrix}
\phi & 0 & 0
\\
0 & 0 & 0
\\
0 & 0 & 0
\end{pmatrix}
\end{align*}
all other matrices being zero.

It is natural to perform the change of basis $A=\left(\begin{smallmatrix}1&0&0\\0&1&1\\0&1&-1\end{smallmatrix}\right)$ to separate the trivial
eigenvector of all these matrices; then
\[
A\dilute\rho_{2,L}(\check R(\uu))A^{-1} = \left(\begin{smallmatrix}
-\frac{\left[\nv-\xx+1\right]\left[\xx+1\right]\left[2\nv-\xx\right]}{\left[\nv+\xx+1\right]\left[\xx-1\right]\left[2\nv+\xx\right]}
\left(
1
-\frac{2\left[6\nv-2\xx\right]\left[2\nv\right]\left[3\nv\right]\left[\nv+1\right]}{\left[\nv-\xx+1\right]\left[\xx+1\right]\left[2\nv-\xx\right]\left[3\nv-\xx\right]\left[6\nv\right]}
\right)
&\frac{\left[2\nv\right]\left[3\nv\right]\left[\xx\right]\left[\nv+1\right]}{\left[\nv+\xx+1\right]\left[\xx-1\right]\left[6\nv\right]\left[2\nv+\xx\right]} &0
\\
\frac{2\left[6\nv+2\right]\left[4\nv-2\right]\left[2\nv\right]\left[3\nv\right]^{2}\left[\xx\right]}{\left[3\nv+1\right]\left[2\nv-1\right]\left[\nv+\xx+1\right]\left[\xx-1\right]\left[6\nv\right]\left[2\nv+\xx\right]\left[\nv\right]}
&1+\frac{2\left[6\nv+2\xx\right]\left[2\nv\right]\left[3\nv\right]\left[\nv+1\right]}{\left[\nv+\xx+1\right]\left[\xx-1\right]\left[6\nv\right]\left[2\nv+\xx\right]\left[3\nv+\xx\right]}&0
\\
0&0&-1
\end{smallmatrix}
\right)
\]

There is a single cut length $0$ representation, of dimension $2$,
which one can think of as the action on $\dilute B(\obj2,\obj0)$, which is spanned by
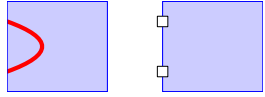
\[
\diag{dilute2_0_1}
\quad
\diag{dilute2_0_0}
\]
The corresponding matrices are
\begin{gather*}
\dilute\rho_{2,I}(\dilute S)=\begin{pmatrix}
\pp^{12} & 0
\\
0 & 1
\end{pmatrix}
\quad
\dilute\rho_{2,I}(H)=\begin{pmatrix}
\phi & 0
\\
0 & 0
\end{pmatrix}
\\
\dilute\rho_{2,I}(E)=\begin{pmatrix}
0 & 1
\\
\delta & 0
\end{pmatrix}
\quad
\dilute\rho_{2,I}(U)=\begin{pmatrix}
\delta & 0
\\
0 & 0
\end{pmatrix}
\quad
\dilute\rho_{2,I}(O)=\begin{pmatrix}
0 & 0
\\
0 & 1
\end{pmatrix}
\end{gather*}
all other matrices being zero.

Computing, one finds:
\[
\dilute\rho_{2,I}(\check R(\uu)) = \left(\begin{smallmatrix}
-\frac{\left[\nv-\xx+1\right]\left[\xx+1\right]\left[2\nv-\xx\right]\left[3\nv-\xx\right]}{\left[\nv+\xx+1\right]\left[\xx-1\right]\left[2\nv+\xx\right]\left[3\nv+\xx\right]}
\left(
1
-\frac{\left[\nv+1\right]\left[2\nv\right]\left[3\nv\right]\left[2\xx\right]\left[6\nv-\xx\right]}{\left[6\nv\right]\left[\nv-\xx+1\right]\left[\xx+1\right]\left[2\nv-\xx\right]\left[3\nv-\xx\right]\left[\xx\right]}
\right)
&-\frac{\left[2\nv\right]\left[3\nv\right]\left[2\xx\right]\left[\nv+1\right]}{\left[\nv+\xx+1\right]\left[\xx-1\right]\left[6\nv\right]\left[2\nv+\xx\right]\left[3\nv+\xx\right]}
\\
-\frac{\left[4\nv\right]\left[6\nv+1\right]\left[5\nv-1\right]\left[3\nv\right]\left[2\xx\right]}{\left[\nv+\xx+1\right]\left[\xx-1\right]\left[6\nv\right]\left[2\nv+\xx\right]\left[3\nv+\xx\right]}
& 1+\frac{\left[\nv+1\right]\left[2\nv\right]\left[3\nv\right]\left[2\xx\right]\left[6\nv+\xx\right]}{\left[6\nv\right]\left[\nv+\xx+1\right]\left[\xx-1\right]\left[2\nv+\xx\right]\left[3\nv+\xx\right]\left[\xx\right]}
\end{smallmatrix}\right)
\]

In both representations above, one can check by direct calculation that the unitarity equation \eqref{eq:unit} is satisfied.

As a byproduct, we've found an alternative description (in $\bQ(\pp,\qq,\uu)$) of the $R$-matrix.
Recall the idempotents $\pi_R$ of $A(\obj2)$ with $R\in \mathit{Irr}^{\le2}=\{I,L,X_2,Y_2,Y_2^*\}$, cf \eqref{eq:A2idem}.
Because the top part of $\dilute A(\obj2)$ coincides with that of $A(\obj2)$, the idempotents $\pi_R$, $R\in
\mathit{Irr}^{2}=\{X_2,Y_2,Y_2^*\}$, are still central idempotents inside $\dilute A(\obj2)$. Then
\[
\check R(\uu) = -\pi_{X_2} + \frac{[\xx+1]}{[\xx-1]} \pi_{Y_2} + \frac{[\xx-\nv-1]}{[\xx+\nv+1]} \pi_{Y_2^*} + \check R_L(\uu) + \check R_{I}(\uu)
\]
where the last two terms are obtained by multiplying $\check R(\uu)$ by the central idempotents corresponding to irreps $\dilute V_{2,L}$
and $\dilute V_{2,I}$ respectively; their matrix representations are given above.

\subsection{Proof of Yang--Baxter equation}\label{ssec:YBEproof}
We now prove \eqref{eq:YBE} in each representation of $\dilute A(\obj3)$, cf. \S\ref{ssec:diltrunc}, grouping them by cut length as in \eqref{eq:irr}.

\subsubsection{Cut length three}
This is the top part of $\dilute A(\obj3)$, which is just $\hat{A}(\obj 3)$.
In these representations, all generators are sent to zero except $S_1$ and $S_2$ which are sent to \eqref{eq:Sirrep}--\eqref{eq:Sirrepb}.
One can easily compute the image of l.h.s.\ and r.h.s.\ of \eqref{eq:YBE} under such representations, and check that they are equal: in the same order of representations, one finds
\begin{gather*}
\left(\!\begin{array}{c}
-1
\end{array}\!\right)
\\
\left(\!\begin{array}{c}
\frac{\left[\xx+1\right]\ \left[\yy+1\right]\ \left[\xx+\yy+1\right]}{\left[\xx-1\right]\ \left[\yy-1\right]\ \left[\xx+\yy-1\right]}
\end{array}\!\right)
\\
\left(\!\begin{array}{c}
\frac{\left[-\nv+\xx-1\right]\ \left[-\nv+\yy-1\right]\ \left[-\nv+\xx+\yy-1\right]}{\left[\nv+\xx+1\right]\ \left[\nv+\yy+1\right]\ \left[\nv+\xx+\yy+1\right]}
\end{array}\!\right)
  \\
\left(\!\begin{array}{ccc}
  \frac{\left[\nv+1\right]\ \left[-\nv+\xx-1\right]\ \left[\xx+1\right]}{\left[\yy-1\right]\ \left[\nv+\yy+1\right]\ \left[\xx+\yy-1\right]\ \left[\nv+\xx+\yy+1\right]}&\frac{\left[\nv+1\right]\ \left[2\,\nv\right]\ \left[\xx+1\right]\ \left[\xx+\yy\right]}{\left[\nv\right]\ \left[\nv+\yy+1\right]\ \left[\xx+\yy-1\right]\ \left[\nv+\xx+\yy+1\right]}&\frac{\left[\xx+1\right]\ \left[-\nv+\yy-1\right]\ \left[\xx+\yy\right]\ \left[\nv+\xx+\yy\right]}{\left[\xx-1\right]\ \left[\nv+\yy+1\right]\ \left[\xx+\yy-1\right]\ \left[\nv+\xx+\yy+1\right]}\\
\frac{\left[-\nv+\xx-1\right]\ \left[\xx+\yy\right]}{\left[\yy-1\right]\ \left[\xx+\yy-1\right]\ \left[\nv+\xx+\yy+1\right]}&\frac{\left[\xx+\yy\right]\ \left[-\nv+\xx+\yy\right]+\left[\nv+1\right]}{\left[\xx+\yy-1\right]\ \left[\nv+\xx+\yy+1\right]}&\frac{\left[-\nv+\yy-1\right]\ \left[\xx+\yy\right]}{\left[\xx-1\right]\ \left[\xx+\yy-1\right]\ \left[\nv+\xx+\yy+1\right]}\\
\frac{\left[-\nv+\xx-1\right]\ \left[\yy+1\right]\ \left[\xx+\yy\right]\ \left[\nv+\xx+\yy\right]}{\left[\nv+\xx+1\right]\ \left[\yy-1\right]\ \left[\xx+\yy-1\right]\ \left[\nv+\xx+\yy+1\right]}&\frac{\left[\nv+1\right]\ \left[2\,\nv\right]\ \left[\yy+1\right]\ \left[\xx+\yy\right]}{\left[\nv\right]\ \left[\nv+\xx+1\right]\ \left[\xx+\yy-1\right]\ \left[\nv+\xx+\yy+1\right]}&\frac{\left[\nv+1\right]\ \left[-\nv+\yy-1\right]\ \left[\yy+1\right]}{\left[\xx-1\right]\ \left[\nv+\xx+1\right]\ \left[\xx+\yy-1\right]\ \left[\nv+\xx+\yy+1\right]}
\end{array}\!\right)
\\
\left(\!\begin{array}{cc}
-\frac{\left[\xx+1\right]}{\left[\yy-1\right]\ \left[\xx+\yy-1\right]}&-\frac{\left[\xx+1\right]\ \left[\xx+\yy\right]}{\left[\xx-1\right]\ \left[\xx+\yy-1\right]}\\
-\frac{\left[\yy+1\right]\ \left[\xx+\yy\right]}{\left[\yy-1\right]\ \left[\xx+\yy-1\right]}&-\frac{\left[\yy+1\right]}{\left[\xx-1\right]\ \left[\xx+\yy-1\right]}
\end{array}\!\right)
\\
\left(\!\begin{array}{cc}
\frac{\left[\nv+1\right]\ \left[-\nv+\xx-1\right]}{\left[\nv+\yy+1\right]\ \left[\nv+\xx+\yy+1\right]}&-\frac{\left[-\nv+\xx-1\right]\ \left[\xx+\yy\right]}{\left[\nv+\xx+1\right]\ \left[\nv+\xx+\yy+1\right]}\\
-\frac{\left[-\nv+\yy-1\right]\ \left[\xx+\yy\right]}{\left[\nv+\yy+1\right]\ \left[\nv+\xx+\yy+1\right]}&\frac{\left[\nv+1\right]\ \left[-\nv+\yy-1\right]}{\left[\nv+\xx+1\right]\ \left[\nv+\xx+\yy+1\right]}
\end{array}\!\right)
\end{gather*}

\subsubsection{Cut length two}\label{ssec:ybe2}
In this section we implicitly work over the fraction field in some intermediate steps, noting that this is harmless for the purpose of proving identities such as \eqref{eq:YBE}.

We can think of the cut length $2$ representations as arising from the action of $\dilute A(\obj3)$ on $\dilute B(\obj3,\obj2)/{\dilute J}^{(1)}(\obj3,\obj2)$.
Here's a basis of the latter:
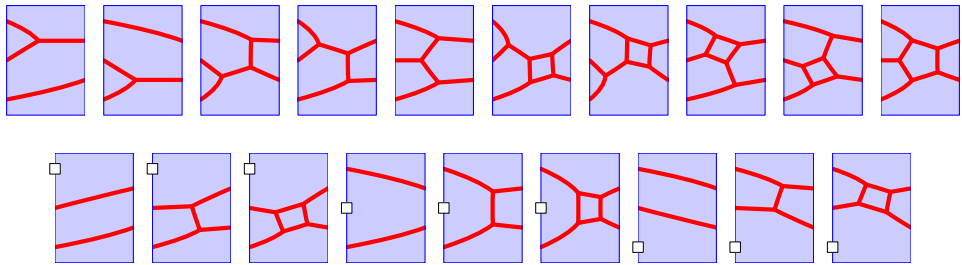
\begin{gather*}
\foreach \i in {18, 17, 16, 15, 14, 13, 12, 11, 10, 9} { \diag{dilute3_2_\i} }
\\
\foreach \i in { 8, 7, 6, 5, 4, 3, 2, 1, 0} { \diag{dilute3_2_\i} }
\end{gather*}
The three representations can be obtained by multiplying on the right by $\pi_R$, $R\in\Irr^2$, see \eqref{eq:irr}. In fact, it is clear that in the second row we only need to consider
the first in the list of three diagrams with the same location of the mark; i.e.,
\begin{equation}\label{eq:dilutecutlength2}
\dilute V_{3,R} = V_{3,R} \oplus \text{span}\left\{
\diag{dilute3_2_proj1},\diag{dilute3_2_proj2},\diag{dilute3_2_proj3}
\right\}
\qquad R\in \Irr^2
\end{equation}

The nondilute part $V_{3,R}$ has already been investigated as an $A(\obj3)$ representation; in particular, the action of $S_1$ and $S_2$ is known,
cf. \eqref{eq:Sirrepc}--\eqref{eq:Sirrepd}. Let us recover these matrices diagrammatically.

First note that there is an easy eigenspace of $S_1$, namely $\text{span}\left\{\diag{dilute3_2_18},\diag{dilute3_2_15},\diag{dilute3_2_13}\right\}$, with eigenvalue $-\pp^6$
(cf. Definition~\ref{def:basicrels}). This gives us an eigenvector for each $R\in \Irr^2$, see \eqref{eq:irr}:
\[
\omega_R= \diag{diag3_2_proj1}\qquad R\in\Irr^2
\]
In diagrammatic terms, the antidiagonal matrix $P$ with $1$s on the antidiagonal is interpreted as the operation of vertical mirror image of diagrams
of $\dilute B(\obj3,\obj2)/{\dilute J}^{(1)}(\obj3,\obj2)$. Therefore, we set
\[
\alpha_R=P \omega_R = \diag{diag3_2_proj2}\qquad R\in \Irr^2
\]
which is an eigenvector of $S_2$ with eigenvalue $-\pp^6$.

Define
\begin{align*}
  \beta_{Y_2}&=\qq\, S_1\alpha_{Y_2}-\qq^{-1}\alpha_{Y_2}-\qq\, \omega_{Y_2}
  \\
  \beta_{Y_2^*}&=-\qq^{-1}\pp^{-1}S_1\alpha_{Y_2^*}+\qq\pp\,\alpha_{Y_2^*}+\qq^{-1}\pp^{-1}\omega_{Y_2^*}
\end{align*}
Then the bases $(\alpha_R,\beta_R,\omega_R)$, $R=Y_2,Y_2^*$,  give rise to the matrices \eqref{eq:Sirrepdd}--\eqref{eq:Sirrepd}.

Finally, for the irrep $X_2$, we set
\begin{align*}
\gamma_{X_2}&=\qq^{-2}\pp^{-6}\frac{\left[2\,\nv-1\right]}{\left(\qq-\qq^{-1}\right)\left[2\right]\ \left[4\,\nv-2\right]}(S_1+1)(S_1-\qq^{-2})(S_1+\pp^6)\alpha_{X_2}-\frac{\left[\nv\right]\ \left[2\,\nv-1\right]\ \left[6\,\nv\right]}{\left[2\,\nv\right]\ \left[3\,\nv\right]\ \left[4\,\nv-2\right]}\omega_{X_2}
\\
\beta_{X_2}&=P \gamma_{X_2}
\end{align*}
The basis $(\alpha_{X_2},\beta_{X_2},\gamma_{X_2},\omega_{X_2})$ gives rise to the matrix \eqref{eq:Sirrepc}.

We now write the matrices for the various generators appearing in the definition \eqref{eq:Rmat} of the $R$-matrix.
For the second summand of \eqref{eq:dilutecutlength2}, we shall use the diagrammatic basis -- one could change bases so that $S_1$ is lower triangular in every block,
but there is no particular reason to do so.
We find:
\begin{align*}
  \dilute\rho_{\obj3,X_2}(\dilute S_1)&=\left(\!\begin{array}{c|c}
\begin{matrix}
-1&0&0&0\\
-\qq^{-1}&\qq^{-2}&0&0\\
\qq&\left(\pp-1\right)\left(\pp+1\right)&\qq^{2}\pp^{2}&0\\
1&-\qq^{-1}\pp^{2}\left(\pp^{2}-1+\pp^{-2}\right)&\qq\pp^{4}\left(\pp^{2}-1+\pp^{-2}\right)&-\pp^{6}
\end{matrix}&0\\\hline
0&\begin{matrix}0&1&0\\1&0&0\\0&0&-1\end{matrix}
\end{array}\!\right)
\\
  \dilute\rho_{\obj3,X_2}(T_1)&=\left(\!\begin{array}{c|c}
    0&\begin{matrix}0&0&0 \\ 0&0&0 \\ 0&0&0 \\ 1&1&0\end{matrix}
    \\\hline
    \begin{matrix}\left(\qq-\qq^{-1}\right)^{2}&\frac{\left[2\nv-2\right]\left[6\nv\right]}{\left[\nv-1\right]\left[\nv+1\right]\left[2\nv\right]}&-\frac{\left[6\nv\right]\left[6\nv+2\right]}{\left[\nv+1\right]\left[2\nv\right]\left[3\nv+1\right]}&\frac{\left[3\nv\right]\left[4\nv-2\right]\left[6\nv+2\right]}{\left[\nv\right]\left[\nv+1\right]\left[2\nv-1\right]\left[3\nv+1\right]}\\
\left(\qq-\qq^{-1}\right)^{2}&\frac{\left[2\nv-2\right]\left[6\nv\right]}{\left[\nv-1\right]\left[\nv+1\right]\left[2\nv\right]}&-\frac{\left[6\nv\right]\left[6\nv+2\right]}{\left[\nv+1\right]\left[2\nv\right]\left[3\nv+1\right]}&\frac{\left[3\nv\right]\left[4\nv-2\right]\left[6\nv+2\right]}{\left[\nv\right]\left[\nv+1\right]\left[2\nv-1\right]\left[3\nv+1\right]}\\
0&0&0&0
\end{matrix}&0
\end{array}\!\right)
\\
\dilute\rho_{\obj3,X_2}(B_1)&=\left(\!
                   \begin{array}{c|c}
                     0&0\\\hline
                     0&\begin{matrix}1&1&0\\1&1&0\\0&0&0\end{matrix}
\end{array}\!\right)
\\
\dilute\rho_{\obj3,X_2}(E_1)&=0
\end{align*}

\begin{align*}
  \dilute\rho_{\obj3,Y_2}(\dilute S_1)&=\left(\!\begin{array}{c|c}
\begin{matrix}
\qq^{-2}&0&0\\
\qq^{-1}&-1&0\\
1&\qq^{-1}\pp^{2}\left(\pp^{2}-\qq\pp^{-1}\right)\left(\pp^{2}+\qq\pp^{-1}\right)&-\pp^{6}
\end{matrix}&0\\\hline
0&\begin{matrix}0&1&0\\1&0&0\\0&0&\qq^{-2}\end{matrix}
\end{array}\!\right)
\\
  \dilute\rho_{\obj3,Y_2}(T_1)&=\left(\!\begin{array}{c|c}
    0&\begin{matrix}0&0&0 \\ 0&0&0 \\ 1&1&0 \end{matrix}
    \\\hline
    \begin{matrix}-\frac{\left[4\nv-2\right]}{\left[\nv\right]\left[\nv+1\right]\left[2\nv-1\right]}&-\frac{\left[3\nv-1\right]\left[4\nv-2\right]\left[6\nv+2\right]}{\left[\nv\right]\left[\nv+1\right]\left[2\nv-1\right]\left[3\nv+1\right]}&\frac{\left[3\nv\right]\left[4\nv-2\right]\left[6\nv+2\right]}{\left[\nv\right]\left[\nv+1\right]\left[2\nv-1\right]\left[3\nv+1\right]}\\
-\frac{\left[4\nv-2\right]}{\left[\nv\right]\left[\nv+1\right]\left[2\nv-1\right]}&-\frac{\left[3\nv-1\right]\left[4\nv-2\right]\left[6\nv+2\right]}{\left[\nv\right]\left[\nv+1\right]\left[2\nv-1\right]\left[3\nv+1\right]}&\frac{\left[3\nv\right]\left[4\nv-2\right]\left[6\nv+2\right]}{\left[\nv\right]\left[\nv+1\right]\left[2\nv-1\right]\left[3\nv+1\right]}\\
0&0&0 \end{matrix}&0
\end{array}\!\right)
\\
\dilute\rho_{\obj3,Y_2}(B_1)&=\left(\!
                   \begin{array}{c|c}
                     0&0\\\hline
                     0&\begin{matrix}1&1&0\\1&1&0\\0&0&0\end{matrix}
\end{array}\!\right)
\\
  \dilute\rho_{\obj3,Y_2}(E_1)&=0
\end{align*}

\begin{align*}
  \dilute\rho_{\obj3,Y_2^*}(\dilute S_1)&=\left(\!\begin{array}{c|c}
    \begin{matrix}
      \qq^{2}\pp^{2}&0&0\\
      -\qq\pp&-1&0\\
      1&-\qq^{-1}\pp^{3}\left(\qq\pp^{2}-\pp^{-2}\right)\left(\qq\pp^{2}+\pp^{-2}\right)&-\pp^{6}
    \end{matrix}&0\\\hline
    0&\begin{matrix}0&1&0\\1&0&0\\0&0&\qq^2\pp^{2}\end{matrix}
  \end{array}\!\right)
\\
  \dilute\rho_{\obj3,Y^*_2}(T_1)&=\left(\!\begin{array}{c|c}
    0&\begin{matrix}0&0&0 \\ 0&0&0 \\ 1&1&0 \end{matrix}
    \\\hline
    \begin{matrix}
\frac{\left[6\nv+2\right]}{\left[\nv\right]\left[3\nv+1\right]}&\frac{\left[4\nv-2\right]\left[4\nv+1\right]\left[6\nv+2\right]}{\left[\nv\right]\left[\nv+1\right]\left[2\nv-1\right]\left[3\nv+1\right]}&\frac{\left[3\nv\right]\left[4\nv-2\right]\left[6\nv+2\right]}{\left[\nv\right]\left[\nv+1\right]\left[2\nv-1\right]\left[3\nv+1\right]}\\
\frac{\left[6\nv+2\right]}{\left[\nv\right]\left[3\nv+1\right]}&\frac{\left[4\nv-2\right]\left[4\nv+1\right]\left[6\nv+2\right]}{\left[\nv\right]\left[\nv+1\right]\left[2\nv-1\right]\left[3\nv+1\right]}&\frac{\left[3\nv\right]\left[4\nv-2\right]\left[6\nv+2\right]}{\left[\nv\right]\left[\nv+1\right]\left[2\nv-1\right]\left[3\nv+1\right]}\\
0&0&0\end{matrix}&0
\end{array}\!\right)
\\
\dilute\rho_{\obj3,Y_2^*}(B_1)&=\left(\!
                   \begin{array}{c|c}
                     0&0\\\hline
                     0&\begin{matrix}1&1&0\\1&1&0\\0&0&0\end{matrix}
\end{array}\!\right)
\\
\dilute\rho_{\obj3,Y_2^*}(E_1)&=0
\end{align*}
With the chosen bases, $P$ is block diagonal with antidiagonal blocks,
so that the generators $S_2$, $T_2$, $B_2$, $E_2$, are obtained from $S_1$, $T_1$, $B_1$, $E_1$, by rotation of each block by 180 degrees. The other generators, namely the $K_i$, $U_i$ and $O_i$, can be derived from the algebraic
relations \eqref{eq:genrelsa}--\eqref{eq:genrelsb}.

It is now simple, in principle, to check \eqref{eq:YBE} by substituting the matrices above and performing the matrix multiplications. In practice, a computer is still needed
to check equality of left- and right-hand-sides.

Cut lengths $1$ and $0$ are very similar to cut length $2$, and we skip the intermediate steps.

\subsubsection{Cut length one}
We can think of the cut length $1$ representations as arising from the action of $\dilute A(\obj3)$ on $\dilute B(\obj3,\obj1)/{\dilute J}^{(0)}(\obj3,\obj1)$. Here's a basis of the latter:
\begin{gather*}
\foreach \i in {10, 8, 6, 7, 9} { \diag{dilute3_1_\i} }
\\
\foreach \i in {3,2,1} { \diag{dilute3_1_\i} }
\\
\foreach \i in {0,4,5} { \diag{dilute3_1_\i} }
\end{gather*}

The matrices are then given by
{\tiny\setlength{\arraycolsep}{3pt}
\begin{align*}
  \dilute\rho_{\obj3,L}(\dilute S_1)&=\left(\!\begin{array}{c|c|c}
\begin{smallmatrix}
\qq^{-2}&0&0&0&0\\
\qq^{-1}\pp^{3}&-\pp^{6}&0&0&0\\
-1&\qq^{-1}\pp^{3}\left(\qq-1\right)\left(\qq+1\right)&-1&0&0\\
\qq\pp^{-3}&-\pp\left(\pp^{3}+\qq^{2}\pp^{-1}-\pp^{-1}\right)&\qq\pp\left(\pp^{4}-1+\pp^{-4}\right)&\qq^{2}\pp^{2}&0\\
\qq^{2}&-\pp^{5}\left(\qq-1\right)\left(\qq+1\right)\left(\qq^{-1}\pp^{2}+\qq\pp^{-2}\right)&\qq\pp^{6}\left(\qq^{-1}\pp^{2}+\qq\pp^{-2}\right)\left(\pp^{4}-1+\pp^{-4}\right)&\pp^{7}\left(\qq-1\right)\left(\qq+1\right)\left(\qq^{-1}\pp^{2}+\qq\pp^{-2}\right)&\pp^{12}
\end{smallmatrix}&0&0
\\\hline
0&\begin{matrix}0&1&0\\
1&0&0\\
0&0&-\pp^{6}\end{matrix}&0
\\\hline
0&0&\begin{matrix}1&0&0\\
0&0&1\\
0&1&0\end{matrix}
\end{array}\!\right)
\\
  \dilute\rho_{\obj3,L}(T_1)&=\left(\!\begin{array}{c|c|c}
0&\begin{matrix}0&0&0\\
-\frac{\left[3\nv\right]\left[4\nv-2\right]}{\left[\nv\right]\left[\nv+1\right]\left[2\nv-1\right]}&-\frac{\left[3\nv\right]\left[4\nv-2\right]}{\left[\nv\right]\left[\nv+1\right]\left[2\nv-1\right]}&0\\
\frac{\left[4\nv-2\right]}{\left[\nv\right]\left[\nv+1\right]\left[2\nv-1\right]}&\frac{\left[4\nv-2\right]}{\left[\nv\right]\left[\nv+1\right]\left[2\nv-1\right]}&0\\
-\frac{\left[2\nv\right]\left[2\nv-2\right]-\left[\nv-1\right]^{2}}{\left[\nv-1\right]\left[\nv\right]\left[\nv+1\right]}&-\frac{\left[2\nv\right]\left[2\nv-2\right]-\left[\nv-1\right]^{2}}{\left[\nv-1\right]\left[\nv\right]\left[\nv+1\right]}&0\\
-\frac{\left[4\nv-2\right]}{\left[\nv\right]\left[\nv+1\right]\left[2\nv-1\right]}&-\frac{\left[4\nv-2\right]}{\left[\nv\right]\left[\nv+1\right]\left[2\nv-1\right]}&0
\end{matrix}&0
\\\hline
\begin{matrix}
1&-\frac{\left[6\nv+2\right]}{\left[3\nv+1\right]}&0&0&0\\
1&-\frac{\left[6\nv+2\right]}{\left[3\nv+1\right]}&0&0&0\\
0&0&0&0&0
\end{matrix}&0&\begin{matrix}0&0&0\\0&0&0\\0&1&1\end{matrix}
\\\hline
0&\begin{matrix}0&0&0\\
0&0&\frac{\left[3\nv\right]\left[4\nv-2\right]\left[6\nv+2\right]}{\left[\nv\right]\left[\nv+1\right]\left[2\nv-1\right]\left[3\nv+1\right]}\\
0&0&\frac{\left[3\nv\right]\left[4\nv-2\right]\left[6\nv+2\right]}{\left[\nv\right]\left[\nv+1\right]\left[2\nv-1\right]\left[3\nv+1\right]}
\end{matrix}&0
\end{array}\!\right)
\\
  \dilute\rho_{\obj3,L}(B_1)&=\left(\!\begin{array}{c|c|c}
0&0&0
\\\hline
0&\begin{matrix}
1&1&0
\\
1&1&0
\\
0&0&0\end{matrix}&0
\\\hline
0&0&\begin{matrix}
0&0&0
\\
0&1&1
\\
0&1&1\end{matrix}
\end{array}\!\right)
\\
  \dilute\rho_{\obj3,L}(E_1)&=\left(\!\begin{array}{c|c|c}
0&0&\begin{matrix}
0&0&0
\\
0&0&0
\\
0&0&0
\\
0&0&0
\\
1&0&0
\end{matrix}
\\\hline
0&0&0
\\\hline
\begin{matrix}
1&-\frac{\left[4\nv-2\right]\left[6\nv+1\right]}{\left[\nv+1\right]\left[2\nv-1\right]}&\frac{\left[4\nv-2\right]\left[5\nv\right]\left[6\nv+1\right]}{\left[\nv+1\right]\left[2\nv-1\right]}&\frac{\left[4\nv-2\right]\left[4\nv\right]\left[6\nv+1\right]}{\left[\nv+1\right]\left[2\nv-1\right]\left[2\nv\right]}&\frac{\left[4\nv\right]\left[5\nv-1\right]\left[6\nv+1\right]}{\left[\nv+1\right]\left[2\nv\right]}\\
0&0&0&0&0\\
0&0&0&0&0
\end{matrix}
&0&0
\end{array}\!\right)
\end{align*}
}

\subsubsection{Cut length zero}
Finally, we can think of the cut length $0$ representations as arising from the action of $\dilute A(\obj3)$ on $\dilute B(\obj3,\obj0)$. Here's a basis of the latter:
\[
  \diag{dilute3_0_4}\qquad \foreach \i in {1,2,3}  { \diag{dilute3_0_\i} } \qquad \diag{dilute3_0_0}
\]

\begin{align*}
  \dilute\rho_{\obj3,I}(\dilute S_1)&=\left(\!\begin{array}{c|c|c}
    -\pp^6&0&0
    \\\hline
    0&\begin{matrix}
      0&1&0\\
      1&0&0\\
      0&0&\pp^{12}
    \end{matrix}&0\\
       \hline
0&0&1
\end{array}\!\right)
\\
\dilute\rho_{\obj3,I}(T_1)&=\left(\!
\begin{array}{c|c|c}
0&\begin{matrix}1&1&0\end{matrix}&0
  \\\hline
\begin{matrix}
\phi\\
\phi\\
0
\end{matrix}&0&0
  \\\hline
0&0&0
\end{array}\!\right)
\\
\dilute\rho_{\obj3,I}(B_1)&=\left(\!
\begin{array}{c|c|c}
0&0&0\\\hline
0&\begin{matrix}1&1&0\\1&1&0\\0&0&0\end{matrix}&0\\\hline
0&0&0
\end{array}\!\right)
\\
\dilute\rho_{\obj3,I}(E_1)&=\left(\!
\begin{array}{c|c|c}
0&0&0\\\hline
0&0&\begin{matrix}0\\0\\1\end{matrix}\\\hline
0&\begin{matrix}0&0&\delta\end{matrix}&0
\end{array}\!\right)
\end{align*}
and the generators indexed $2$ are obtained from those indexed $1$ by rotation of each block by 180 degrees.

\subsection{Proof of crossing symmetry}\label{sec:crossing}
The crossing relation, \eqref{eq:crossing}, can be reformulated as follows. Define
\[
  \check R^{\mathrm{poly}}(\uu)=
  \frac{\left[6\,\nv\right]^{2}\left[\nv\right]}{\left[3\,\nv\right]^{2}\left[2\,\nv\right]}
  \left[\nv+\xx+1\right]\ \left[\xx-1\right]\ \left[2\,\nv+\xx\right]\ \left[3\,\nv+\xx\right]\ \check R(\uu)
\]
Then the coefficients of $\check R^{\mathrm{poly}}(\uu)$ as a linear combination of trivalent diagrams are in $A_{\pp,\qq,\uu}$ and coprime (i.e., up to
a power of $\qq-\qq^{-1}$, they are coprime Laurent polynomials in $\pp,\qq,\uu$), and
\begin{equation}\label{eq:crossingb}
\Rot(\check R^{\mathrm{poly}}(\uu)) = \check R^{\mathrm{poly}}(\pp^{-3}\uu^{-1})
\end{equation}
The substitution $\uu\mapsto\pp^{-3}\uu^{-1}$ preserves $A_{\pp,\qq,\uu}$, so this formulation requires no spectral localisation.

Separate further the diagrams as follows: introduce
\begin{align*}
I&=\diag{dilute2_2_0}
\\
D&=\diag{dilute2_2_7}+\diag{dilute2_2_6}
\end{align*}
so that
\begin{align*}
1&=I+D^2+O
\\
\dilute S&=S+D+O
\\
B&=D+D^2
\end{align*}

Under rotation $\Rot$, these combinations of diagrams are either fixed points or two-cycles, according to
\[
(I,U)\qquad (S,S^{-1})\qquad (K,H)\qquad (E,D^2)\qquad T\qquad D\qquad O
\]
where we have been forced to introduce $H=\diag{diag2_2_3}$ as the rotation of $K$.

Now compare
\begin{multline*}
  \frac{\left[3\,\nv\right]^{2}\left[2\,\nv\right]}{\left[6\,\nv\right]^{2}\left[\nv\right]}
  \Rot(\check R^{\mathrm{poly}}(\uu))=
    [\nv+1][2\,\nv+\xx][3\,\nv+\xx](U+E+O)
\\
-\frac{\pp^{-1}\uu^{-1} S{}^{-1} - \pp\, \uu\, S}{\qq-\qq^{-1}}[\xx][2\,\nv+\xx][3\,\nv+\xx]
+[\xx][\nv+\xx][2\,\nv+\xx][3\,\nv+\xx](D+O)
\\
+\frac{\left[2\,\nv\right]\ \left[3\,\nv\right]\ \left[\nv+1\right]\ [\xx]}{\left[6\,\nv\right]}
\left(
-[\nv][\xx+3\,\nv]H+[\xx+3\,\nv]T-\frac{[\xx+3\,\nv]}{[\nv]}(D+E)-\frac{[2\xx]}{[\xx]}D^2\right.
\\
\left.
+\frac{1}{[\nv]}\left(\frac{[6\,\nv]}{[3\,\nv]}[\xx+2\,\nv]-[\xx+3\,\nv]\right)I
-\left(\frac{[2\,\xx+12\,\nv]}{[\xx+6\,\nv]}+\frac{2[\xx+3\,\nv]}{[\nv]}\right)O\right)
\end{multline*}

with
\begin{multline*}
  \frac{\left[3\,\nv\right]^{2}\left[2\,\nv\right]}{\left[6\,\nv\right]^{2}\left[\nv\right]}
  \check R^{\mathrm{poly}}(\pp^{-3}\uu^{-1})
  =
    [\nv+1][\nv+\xx][\xx](I+D^2+O)\\
    +\frac{\pp^{2}\uu  S - \pp^{-2} \uu^{-1}  S{}^{-1}}{\qq-\qq^{-1}}[3\,\nv+\xx][\nv+\xx][\xx]
    +[\xx][\nv+\xx][2\,\nv+\xx][3\,\nv+\xx](D+O)
\\
+\frac{\left[2\,\nv\right]\ \left[3\,\nv\right]\ \left[\nv+1\right]\ [3\,\nv+\xx]}{\left[6\,\nv\right]}
\left(
-[\nv][\xx]K+[\xx]T-\frac{[\xx]}{[\nv]}(D+D^2)+\frac{[2(3\,\nv+\xx)]}{[3\,\nv+\xx]}E\right.
\\
\left.
+\frac{1}{[\nv]}\left(\frac{[6\,\nv]}{[3\,\nv]}[\xx+\nv]-[\xx]\right)U
+\left(\frac{[2\,\xx-6\,\nv]}{[\xx-3\,\nv]}-\frac{2[\xx]}{[\nv]}\right)O\right)
\end{multline*}

We find that the coefficients of $T$ and $D$, which are explicitly
invariant under $\uu\mapsto \pp^{-3}\uu^{-1}$, are the same
in both expressions. Then comes $D^2\leftrightarrow E$: the coefficient of $D^2$ in the first expression
is $-\frac{[2\,\nv][3\,\nv][\nv+1]}{[6\,\nv]}[2\xx]$; in the second,
$[\nv+1][\nv+\xx][\xx]-\frac{[2\,\nv][3\,\nv][\nv+1]}{[\nv][6\,\nv]}[3\,\nv+\xx][\xx]$.
These are indeed equal.
Therefore the same is true of the coefficient of $E$.

With a bit more effort, one checks that the coefficient of $O$ is also invariant under $\uu\mapsto \pp^{-3}\uu^{-1}$.
In the end, taking the difference of the two displayed expressions gives
\[
[\nv][\xx][\xx+3\,\nv]
  \left(  \frac{S-S^{-1}}{\qq-\qq^{-1}}-\frac{[\nv+1][2\,\nv][3\,\nv]}{[6\,\nv]}(H-K+(\qq-\qq^{-1})^2(U-I))  \right)
\]
which is nothing but (a multiple of) the skein relation \eqref{eq:skeinc}.

\begin{rem}
An alternative proof goes as follows. Define
$\dilute U = U + E + O$. Then it is not hard to see that \eqref{eq:crossingb} is equivalent to the following identity in $\dilute A(\obj3)\otimes \coeffsu$:
\[
\dilute U_1 \dilute U_2 \check R^{\mathrm{poly}}_1(\uu) = \dilute U_1 \check R^{\mathrm{poly}}_2(\pp^{-3}\uu^{-1})
\]
The latter can be checked representation by representation, similarly to the proof of the Yang--Baxter equation in \S\ref{ssec:YBEproof}
(and $\dilute U_1$ is sent to zero in all representations except $\dilute \rho_{\obj3,L}$ and $\dilute \rho_{\obj3,I}$).
\end{rem}

\begin{rem}
In the course of the proof, we have introduced another form of the crossing symmetry relation, namely \eqref{eq:crossingb}, which is more aesthetically pleasing.
However, a drawback is that the form of the $R$-matrix that appears, $\check R^{\mathrm{poly}}(\uu)$, no longer satisfies the unitarity equation
\eqref{eq:unit}. One may wonder if it is possible to normalise the $R$-matrix in such a way that it satisfies {\em both}\/ unitary and crossing symmetry relations
in this improved form.
This is not possible if we stay within the realm of rational functions. If however we work in the {\em analytic}\/ setting, that is, we specialise $\pp$ and $\qq$
to complex numbers, and allow $\check R(\uu)$ to be a meromorphic function of $\uu$, then it is possible to find such a normalisation.

Assume $0<|\pp|<1$ and $\qq\ne0$ (analogous statements can be made for $|\pp|>1$). Define
\[
f(\uu)=\prod_{i=0}^\infty(1-\pp^{12i}\qq^{-2}\uu^2)(1-\pp^{12i+2}\qq^2 \uu^2)(1-\pp^{12i+4}\uu^2)(1-\pp^{12i+6}\uu^2)
\]
and
\[
\kappa(\uu)=\uu^4\frac{f(\uu^{-1})f(\pp^3\uu)}{f(\uu)f(\pp^3\uu^{-1})}
\]
To check this, put $P(\uu)=(1-\qq^{-2}\uu^2)(1-\pp^2\qq^2\uu^2)(1-\pp^4\uu^2)(1-\pp^6\uu^2)$.
The identity $f(\pp^6\uu)=f(\uu)/P(\uu)$ gives
\[
\kappa(\uu)\kappa(\uu^{-1})=1,\qquad
\kappa(\uu)\kappa(\pp^3\uu)
=\pp^{12}\uu^8\frac{P(\pp^{-3}\uu^{-1})}{P(\uu)},
\]
where the last expression is the prefactor in \eqref{eq:crossing}.
Thus $\kappa(\uu)^{-1}\check R(\uu)$ satisfies both \eqref{eq:unit} and crossing symmetry in the form \eqref{eq:crossingb}, as meromorphic identities on $\bC^\times$.
\end{rem}

\subsection{Specialisations}
In all exceptional cases except $A_1$ and $A_2$, the minimal affinisation of the $q$-deformed adjoint representation
$L_C$ is of the form $L_C\oplus I_C$, and one expects that the $R$-matrix $\check R(\uu)$ is mapped
(up to normalisation) under the functor $\dilute{\Xi_C}$ to the universal $R$-matrix of the corresponding quantised loop algebra in that
minimal affinisation. This can be checked by explicit computation
(in the most complicated case of $C=E_8$, this can also be seen by comparing
with the results of \cite{ZinnJustin2020}).

If $C\in\{A_1,A_2\}$, the minimal affinisation of $L_C$ is simply $L_C$, and so the interpretation
of our solution of the Yang--Baxter equation must be different.
We investigate these two cases below.

\editorialnote[gray]{}

\editorialnote{}

\subsubsection{\texorpdfstring{$A_1$}{A(1)}}
In this section $C=A_1$.
Because the tensor square of the fundamental representation of $U_q(\mathfrak{sl}(2))$ is $L_C\oplus I_C$, we expect a connection between our dilute diagrammatic category
and the Temperley--Lieb category. We briefly sketch this connection here.

We attach to the fundamental (two-dimensional) representation another type of line (depicted in green), with no vertices and the single rule that
\[
	\begin{tikzpicture}[bwwcircle]
		\draw[green,thick] (0,0) circle [radius=0.5];
	\end{tikzpicture}
 = \varepsilon := -(q+q^{-1})
\]
The correspondence between the two diagrammatic approaches is that
\def\tlid#1{	\begin{tikzpicture}[bwwrect={(0,0)}{(2,1)},rotate=90,green,thick,yscale=#1]
          \draw (0,0.25) -- (2,0.25);
          \draw (0,0.75) -- (2,0.75);
        \end{tikzpicture}
}
\def\tle#1{	\begin{tikzpicture}[bwwrect={(0,0)}{(2,1)},rotate=90,green,thick,yscale=#1]
          \draw (0,0.25) .. controls (0.8,0.25) and (0.8,0.75) .. (0,0.75);
          \draw (2,0.25) .. controls (1.2,0.25) and (1.2,0.75) .. (2,0.75);
        \end{tikzpicture}
}
\[
\tlid1
= \diag{dilute1_1_0}+\diag{dilute1_1_1}
\]
and conversely one can write
\begin{align*}
\diag{dilute1_1_1}&=\frac{1}{\varepsilon} \tle1
\\
\diag{dilute1_1_0}&=
	\begin{tikzpicture}[bwwrect={(0,0)}{(2,1)},rotate=90,green,thick]
          \draw (0,0.25) -- (2,0.25);
          \draw (0,0.75) -- (2,0.75);
          \draw[fill=sienna,thin,draw=black] (.75,.15) rectangle (1.25,.85);
        \end{tikzpicture}
:=
\tlid1 - \frac{1}{\varepsilon} \tle1
\end{align*}
which are the usual Jones projectors. In particular, one has $\delta=\varepsilon^2-1=q^2+1+q^{-2}$, which matches with our formula \eqref{eq:values} at $\pp=q^{1/3}$, $\qq=q$.

If we now define the trivalent vertex as
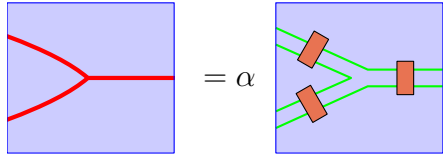
\[
\diag[.5]{diag2_1_0}=\alpha
\begin{tikzpicture}[execute at begin picture={\clip (-1.24,-1.616) rectangle ++(2.48,2.232);},x={(2.2204cm,0cm)},y={(0cm,-2.2204cm)},baseline={([yshift=-\the\dimexpr\fontdimen22\textfont2\relax]current  bounding  box.center)},line join=round,scale=.5]
\begin{scope}
\begin{scope}
\path[clip] (-1,.4) -- (-1,-1.4) -- (1,-1.4) -- (1,.4) -- cycle;
\path[fill={rgb,255:red,204;green,204;blue,255},draw=none] (-1,.4) -- (-1,-1.4) -- (1,-1.4) -- (1,.4) -- cycle;
\begin{scope}[every path/.append style={draw=green,thick}]
\draw (-1,.1) -- (.1,-.4) -- (1,-.4);
\draw (-1,-1.1) -- (.1,-.6) -- (1,-.6);
\draw (-1,-.1) -- (-.1,-.5) -- (-1,-.9);
\draw[fill=sienna,thin,draw=black] (.45,-.3) rectangle (.65,-.7);
\draw[fill=sienna,thin,draw=black,rotate=-30] (-1,-.65) rectangle (-.8,-.25);
\draw[fill=sienna,thin,draw=black,rotate=30] (-.5,-.25) rectangle (-.3,-.65);
\end{scope}
\path[draw=blue,line width=.01405cm] (-1,.4) -- (-1,-1.4) -- (1,-1.4) -- (1,.4) -- cycle;
\end{scope}
\end{scope}
\end{tikzpicture}
\]
where $\alpha$ is an irrelevant normalisation, cf \S\ref{ssec:norm}, then all the identities of Appendix~\ref{app:skein} are automatically satisfied.

Now introduce the fundamental $R$-matrix
\[
\check r(\uu)= \tlid{1.2} + \frac{\uu-\uu^{-1}}{q\,\uu-q^{-1}\uu^{-1}} \tle{1.2}
\]
Then it is not hard to show that setting $\pp=q^{1/3}$, $\qq=q$, in the definition \eqref{eq:Rmat} of $\check R(\uu)$
and choosing $\alpha$ appropriately, one has
\[
\frac{[\xx-1]}{[\xx+1]} \check R(\uu) = \check r_2(q^{1/3}\uu)\check r_1(\uu)\check r_3(\uu)\check r_2(q^{-1/3}\uu)
\]
where $\check r_i$, $i=1,2,3$, are defined as morphisms in $\Hom(\obj4,\obj4)$ of the Temperley--Lieb category.

The interpretation is that our $R$-matrix $\check R(\uu)$ is the $R$-matrix between representations of the quantised loop algebra $U_q(\mathfrak{sl}(2)[z^\pm])$
of the form of a tensor product of evaluation representations, namely, $V(z)\otimes V(q^{1/3}z)$, where $V(z)$ is the affinisation of the fundamental representation.
Note that this tensor product {\em is}\/ irreducible for the quantised loop algebra.

\subsubsection{\texorpdfstring{$A_2$}{A(2)}}
In this section $C=A_2$.
Because the tensor product of the two fundamental representations of $U_q(\mathfrak{sl}(3))$ is $V\otimes V^*\cong L_C\oplus I_C$, we expect a connection between our dilute diagrammatic category
and \cite[\S4]{kuperberg1996}. The situation is very similar to the $A_1$ case.

We introduce {\em oriented}\/ green lines, with
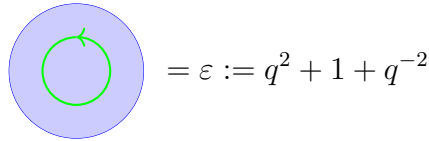
\[
	\begin{tikzpicture}[bwwcircle]
		\draw[green,thick,->] (.5,0) arc [radius=0.5,start angle=0, delta angle=360];
	\end{tikzpicture}
 = \varepsilon := q^2+1+q^{-2}
\]
and the same formula for the opposite orientation. We have
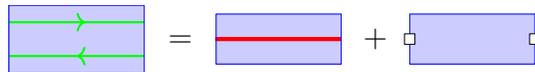
\tikzset{arrow/.style={postaction={decorate,thick,decoration={markings,mark = at position #1 with {\arrow{>}}}}},arrow/.default=0.5}
\tikzset{invarrow/.style={postaction={decorate,thick,decoration={markings,mark = at position #1 with {\arrow{<}}}}},invarrow/.default=0.5}
\def\tlid{	\begin{tikzpicture}[bwwrect={(0,0)}{(2,1)},rotate=90,green,thick]
          \draw[invarrow] (0,0.25) -- (2,0.25);
          \draw[arrow] (0,0.75) -- (2,0.75);
        \end{tikzpicture}
}
\def\tle{	\begin{tikzpicture}[bwwrect={(0,0)}{(2,1)},rotate=90,green,thick]
          \draw[invarrow=0.3] (0,0.25) .. controls (0.8,0.25) and (0.8,0.75) .. (0,0.75);
          \draw[arrow=0.3] (2,0.25) .. controls (1.2,0.25) and (1.2,0.75) .. (2,0.75);
        \end{tikzpicture}
}
\[
\tlid
= \diag{dilute1_1_0}+\diag{dilute1_1_1}
\]
and conversely one can write
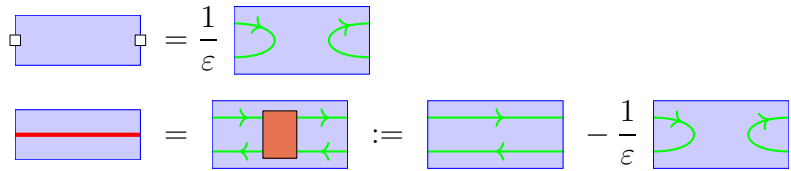
\begin{align*}
\diag{dilute1_1_1}&=\frac{1}{\varepsilon} \tle
\\
\diag{dilute1_1_0}&=
	\begin{tikzpicture}[bwwrect={(0,0)}{(2,1)},rotate=90,green,thick]
          \draw[invarrow=0.2,invarrow=0.8] (0,0.25) -- (2,0.25);
          \draw[arrow=0.2,arrow=0.8] (0,0.75) -- (2,0.75);
          \draw[fill=sienna,thin,draw=black] (.75,.15) rectangle (1.25,.85);
        \end{tikzpicture}
:=
\tlid - \frac{1}{\varepsilon} \tle
\end{align*}
Once again $\delta=\varepsilon^2-1=(q+q^{-1})^2(q^2+q^{-2})$, which matches \eqref{eq:values} at $\pp=q^{1/2}$, $\qq=q$.

The trivalent vertex is
\[
\diag[.5]{diag2_1_0}=\alpha
\begin{tikzpicture}[execute at begin picture={\clip (-1.24,-1.616) rectangle ++(2.48,2.232);},x={(2.2204cm,0cm)},y={(0cm,-2.2204cm)},baseline={([yshift=-\the\dimexpr\fontdimen22\textfont2\relax]current  bounding  box.center)},line join=round,scale=.5]
\begin{scope}
\begin{scope}
\path[clip] (-1,.4) -- (-1,-1.4) -- (1,-1.4) -- (1,.4) -- cycle;
\path[fill={rgb,255:red,204;green,204;blue,255},draw=none] (-1,.4) -- (-1,-1.4) -- (1,-1.4) -- (1,.4) -- cycle;
\begin{scope}[every path/.append style={draw=green,thick}]
\draw[invarrow=0.15,invarrow=0.55,invarrow=0.95] (-1,.1) -- (.1,-.4) -- (1,-.4);
\draw[arrow=0.15,arrow=0.55,arrow=0.95] (-1,-1.1) -- (.1,-.6) -- (1,-.6);
\draw[arrow=0.1,arrow=0.93] (-1,-.1) -- (-.1,-.5) -- (-1,-.9);
\draw[fill=sienna,thin,draw=black] (.45,-.3) rectangle (.65,-.7);
\draw[fill=sienna,thin,draw=black,rotate=-30] (-1,-.65) rectangle (-.8,-.25);
\draw[fill=sienna,thin,draw=black,rotate=30] (-.5,-.25) rectangle (-.3,-.65);
\end{scope}
\path[draw=blue,line width=.01405cm] (-1,.4) -- (-1,-1.4) -- (1,-1.4) -- (1,.4) -- cycle;
\end{scope}
\end{scope}
\end{tikzpicture}
\]
The various $R$-matrices between fundamental representations are:
\begin{align*}
\check r_{V,V}(\uu)&=
	\begin{tikzpicture}[bwwrect={(0,0)}{(2,1)},rotate=90,green,thick,yscale=1.2]
          \draw[invarrow] (0,0.25) -- (2,0.25);
          \draw[invarrow] (0,0.75) -- (2,0.75);
        \end{tikzpicture}
+ \frac{\uu-\uu^{-1}}{q\,\uu-q^{-1}\uu^{-1}}
	\begin{tikzpicture}[bwwrect={(0,0)}{(2,1)},rotate=90,green,thick,yscale=1.2]
          \draw[invarrow=0.2,invarrow=0.8,arrow=0.5] (0,0.25) -- (0.7,0.5) -- (1.3,0.5) -- (2,0.25);
          \draw[invarrow=0.4] (0,0.75) -- (0.7,0.5);
          \draw[invarrow] (1.3,0.5) -- (2,0.75);
        \end{tikzpicture}
\\
\check r_{V^*,V^*}(\uu)&=
	\begin{tikzpicture}[bwwrect={(0,0)}{(2,1)},rotate=90,green,thick,yscale=1.2]
          \draw[arrow] (0,0.25) -- (2,0.25);
          \draw[arrow] (0,0.75) -- (2,0.75);
        \end{tikzpicture}
+ \frac{\uu-\uu^{-1}}{q\,\uu-q^{-1}\uu^{-1}}
	\begin{tikzpicture}[bwwrect={(0,0)}{(2,1)},rotate=90,green,thick,yscale=1.2]
          \draw[arrow=0.2,arrow=0.8,invarrow=0.5] (0,0.25) -- (0.7,0.5) -- (1.3,0.5) -- (2,0.25);
          \draw[arrow=0.4] (0,0.75) -- (0.7,0.5);
          \draw[arrow=0.6] (1.3,0.5) -- (2,0.75);
        \end{tikzpicture}
\\
\check r_{V,V^*}(\uu)&=
	\begin{tikzpicture}[bwwrect={(0,0)}{(2,1)},rotate=90,green,thick,yscale=1.2]
          \draw[arrow=0.25,invarrow=0.75] (0,0.25) -- (2,0.25);
          \draw[invarrow=0.25,arrow=0.75] (0,0.75) -- (2,0.75);
          \draw[invarrow=0.5] (1,0.25) -- (1,0.75);
        \end{tikzpicture}
+ \frac{q^{1/2}\uu-q^{-1/2}\uu^{-1}}{q^{3/2}\uu-q^{-3/2}\uu^{-1}}
	\begin{tikzpicture}[bwwrect={(0,0)}{(2,1)},rotate=90,green,thick,yscale=1.2]
          \draw[arrow=0.3] (0,0.25) .. controls (0.8,0.25) and (0.8,0.75) .. (0,0.75);
          \draw[arrow=0.3] (2,0.25) .. controls (1.2,0.25) and (1.2,0.75) .. (2,0.75);
        \end{tikzpicture}
\\
\check r_{V^*,V}(\uu)&=
	\begin{tikzpicture}[bwwrect={(0,0)}{(2,1)},rotate=90,green,thick,yscale=1.2]
          \draw[invarrow=0.25,arrow=0.75] (0,0.25) -- (2,0.25);
          \draw[arrow=0.25,invarrow=0.75] (0,0.75) -- (2,0.75);
          \draw[arrow=0.5] (1,0.25) -- (1,0.75);
        \end{tikzpicture}
+ \frac{q^{1/2}\uu-q^{-1/2}\uu^{-1}}{q^{3/2}\uu-q^{-3/2}\uu^{-1}}
	\begin{tikzpicture}[bwwrect={(0,0)}{(2,1)},rotate=90,green,thick,yscale=1.2]
          \draw[invarrow=0.3] (0,0.25) .. controls (0.8,0.25) and (0.8,0.75) .. (0,0.75);
          \draw[invarrow=0.3] (2,0.25) .. controls (1.2,0.25) and (1.2,0.75) .. (2,0.75);
        \end{tikzpicture}
\end{align*}
where we have introduced trivalent vertices, with additional relations
\[
\begin{tikzpicture}[bwwcircle,green,thick,rotate=-90]
\draw[arrow] (-1,0) -- (-.3,0);
\draw[invarrow] (-.3,0) .. controls (-.3,.6) and (.3,.6) .. (.3,0);
\draw[invarrow] (-.3,0) .. controls (-.3,-.6) and (.3,-.6) .. (.3,0);
\draw[arrow] (.3,0) -- (1,0);
\end{tikzpicture}
=
-(q+q^{-1})\begin{tikzpicture}[bwwcircle,green,thick]
\draw[arrow] (-1,0) -- (1,0);
\end{tikzpicture}
\qquad
\begin{tikzpicture}[bwwcircle,green,thick]
\draw[arrow] (45:1) -- (45:.5);
\draw[invarrow] (135:1) -- (135:.5);
\draw[arrow] (225:1) -- (225:.5);
\draw[invarrow] (315:1) -- (315:.5);
\draw[invarrow=.125,invarrow=.375,invarrow=.625,invarrow=.875] (45:.5) -- (135:.5) -- (225:.5) -- (315:.5) -- cycle;
\end{tikzpicture}
=
\begin{tikzpicture}[bwwcircle,green,thick]
\draw[arrow,bend left] (45:1) to (135:1);
\draw[arrow,bend left] (225:1) to (315:1);
\end{tikzpicture}
+
\begin{tikzpicture}[bwwcircle,green,thick]
\draw[invarrow,bend left] (135:1) to (225:1);
\draw[invarrow,bend left] (315:1) to (45:1);
\end{tikzpicture}
\]

We then have
\[
\frac{[\xx-1]}{[\xx+1]} \check R(\uu) = \check r_{V^*,V;2}(\uu)\check r_{V,V;1}(\uu)\check r_{V^*,V^*;3}(\uu)\check r_{V,V^*;2}(\uu)
\]
The interpretation is that our $R$-matrix $\check R(\uu)$ is the $R$-matrix between (irreducible) representations of the quantised loop algebra $U_q(\mathfrak{sl}(3)[z^\pm])$
of the form of a tensor product of evaluation representations, namely, $V(z)\otimes V^*(z)$.


\appendix
\section{The diagrammatic relations}\label{app:skein}
In this appendix, we summarise the diagrammatic relations
that were found as we were building the various spaces of morphisms
$\Hom_{\threecat}(\obj r,\obj s)$.
  
\subsection{The basic relations}\label{ssec:basicrels}
The relations below are based on the first few Hom spaces, though the parameters
appearing in them can only be fixed via
the study of the two-string algebra in \S\ref{sec:twostring},
see in particular \eqref{eq:tadpole} and Definition~\ref{def:basicrels}.

In order for $\End_{\threecat}(\obj 0)$ to be of rank $1$,
we have to impose the relation
\[
  \circle = \delta \empty
\]
with
\[
  \delta = \frac{\qint 61 [5\,\nv-1][4\,\nv]}{[\nv+1][2\,\nv]}
\]

In order for $\Hom_{\threecat}(\obj 0,\obj 1)$ to be zero, we must have
\[
  	\tadpole = 0
\]

In order for $\Hom_{\threecat}(\obj 1,\obj 1)$ to be of rank $1$, we must have
\[
  \diag[.42]{bubble} = \phi \diag[.42]{round2_0}
\]
where we write $\phi$ in the form
\begin{align*}
  \phi &= \alpha\,\frac{[6\nv+2]}{[3\nv+1]}\frac{[4\nv-2]}{[2\nv-1]}\frac{[3\nv]}{[\nv]}
  \\
  &=\alpha\, (\pp^2+1+\pp^{-2})(\pp^2\qq^{-1}+\pp^{-2}\qq)(\pp^3\qq+\pp^{-3}\qq^{-1})
\end{align*}

In order for $\Hom_{\threecat}(\obj 1,\obj 2)$ to be of rank $1$, we must have
\[
  \diag[.42]{triangle} = \tau \diag[.42]{round3_0}
\]
with
\begin{align*}
  \tau &= \alpha\,\frac{[2\,\nv]}{[4\,\nv]}\left(\frac{[6\,\nv+2]}{[3\,\nv+1]}\frac{[4\,\nv-2]}{[2\nv-1]}\frac{[3\,\nv]}{[\nv]}+(\qq-\qq^{-1})^2[\nv+1]\frac{[5\nv]}{[\nv]}\right)
  \\
  &=\alpha\,(\pp+\pp^{-1})\left(\qq^{2}\pp^{2}+\pp^{4}+\qq^{2}-\pp^{2}-1-\pp^{-2}+\qq^{-2}+\pp^{-4}+\qq^{-2}\pp^{-2}\right)
\end{align*}

There is a remaining arbitrary parameter $\alpha$ in the definition of $\phi$ and $\tau$, which
can be chosen freely, only the ratio $\tau/\phi$ being fixed, cf \eqref{eq:ratio}. In \S\ref{sec:twostring}, \S\ref{sec:threestring}, as well as in the rest
of this appendix,
$\alpha$ is set to $1$ (cf Definition~\ref{def:basicrels}), whereas in \S\ref{sec:yangbaxter},
$\alpha=1/[\nv+1]$ (cf \S\ref{ssec:norm}).

  \subsection{The square-square relation}\label{ssec:squaresquare}
  From further study of $A(\obj2)=\End_{\threecat}(\obj 2)$, one finds
  the following relation:
\begin{equation}\label{eq:squaresquare}
  \diag{squaresquare}=
    \foreach\k in {0,...,4} {\ifnum\k=0\else +\fi b_{\k} \diag{round4_\k} }
  \end{equation}
  with
{\tiny\begin{align*}
  b_0&=-\frac{\left[\nv+1\right]^2 \left[6\,\nv+2\right]\ \left[4\,\nv-2\right]\ \left(\qq-\qq^{-1}\right)^{2}}{\left[\nv\right]^2\ \left[3\,\nv+1\right]\left[2\,\nv-1\right]}
  \\
  b_1&=\frac{\left[\nv+1\right]\ \left[6\,\nv+2\right]\ \left[4\,\nv-2\right]\left(\pp^{5}+\qq^{2}\pp+\pp^{3}+\pp^{-3}+\qq^{-2}\pp^{-1}+\pp^{-5}\right)}{\left[\nv\right]^{2}\left[3\,\nv+1\right]\ \left[2\,\nv-1\right]}
  \\
  b_2&=\frac{\left(\left[\nv\right]\ \left[3\,\nv-1\right]+\left[2\,\nv\right]\right)\left(\left[6\,\nv+2\right]\ \left[\nv\right]-\left[3\,\nv+1\right]\right)\left(\qq^{2}\pp+\pp^{3}-\pp-\pp^{-1}+\pp^{-3}+\qq^{-2}\pp^{-1}\right)}{\left[\nv\right]^{3}\left[3\,\nv+1\right]}
  \\
  b_3&=-\frac{[\nv+1]}{[\nv]^2}\left(\begin{gathered}\qq^{2}\pp^{5}-\qq^{2}\pp^{3}-\pp^{5}-\qq^{2}\pp-2\,\pp^{3}-\qq^{2}\pp^{-1}+2\,\pp+\qq^{-2}\pp^{3}\\
{}+\qq^{2}\pp^{-3}+2\,\pp^{-1}-\qq^{-2}\pp-2\,\pp^{-3}-\qq^{-2}\pp^{-1}-\pp^{-5}-\qq^{-2}\pp^{-3}+\qq^{-2}\pp^{-5}\end{gathered}\right)
  \\
  b_4&=\frac{\left[2\,\nv\right]\left(\qq^{2}\pp^{2}+\qq^{2}-3+\qq^{-2}+\qq^{-2}\pp^{-2}\right)}{\left[\nv\right]}
\end{align*}}

Alternatively, the coefficients can be written in terms of
the eigenvalues of $H$ and of the basic parameters above as:
\begin{align*}
  b_4&=\xi+\zeta+\zeta^*
  \\
  b_3&=-(\xi\zeta+\xi\zeta^*+\zeta\zeta^*)
  \\
  b_0&=\xi\zeta\zeta^*
  \\
  b_1&=\delta^{-1}(\phi-\xi)(\phi-\zeta)(\phi-\zeta^*)
  \\
  b_2&=\phi^{-1}(\tau-\xi)(\tau-\zeta)(\tau-\zeta^*)
\end{align*}
where
\begin{align*}
  \xi&=\left(\qq-\qq^{-1}\right)\left(\pp\qq-\pp^{-1}\qq^{-1}\right)
  \\
  \zeta&=-\frac{\left[4\,\nv-2\right]}{\left[\nv\right]\left[2\,\nv-1\right]}
  \\
  \zeta^*&=\frac{\left[\nv+1\right]\ \left[6\,\nv+2\right]}{\left[\nv\right]\ \left[3\,\nv+1\right]}
\end{align*}

\subsection{The square-pentagon relation}\label{ssec:squarepenta}
From the study in \S\ref{ssec:bimodules} of $B(\obj2,\obj3)=\Hom_{\threecat}(\obj2,\obj3)$, one finds the following relation:
\begin{align*}
  \diag{squarepenta}&=
  \foreach\k in {0,...,4} {\ifnum\k=0\else +\fi c_{\k} \diag{round5_\k} }
  \\
  \quad &
  \foreach\k in {5,...,9} {+c_{\k} \diag{round5_\k} } 
  \\
  \quad &
  \foreach\k in {10,...,14} {+c_{\k} \diag{round5_\k} }
  \\
  \quad &
  +c_{15} \diag{round5_15}
\end{align*}
where
\begin{align*}
  c_1&=c_4=
       -\frac{[\nv+1]^2}{[\nv]^4}\Bigl(\qq^{2}\pp^{3}-2\,\qq^{2}\pp-2\,\pp^{3}+\qq^{2}\pp^{-1}\\
  &\qquad+\pp+\pp^{-1}+\qq^{-2}\pp-2\,\pp^{-3}-2\,\qq^{-2}\pp^{-1}+\qq^{-2}\pp^{-3}\Bigr)
  \\
  c_2&=c_3=\frac{\left[6\,\nv\right]^{2}[\nv+1]^2}{\left[\nv\right]^3\left[2\,\nv\right]\ \left[3\,\nv\right]^{2}}
  \\
  c_5&=\frac{[\nv+1]}{[\nv]^2}\Bigl(2\,\qq^{2}\pp^{3}+\pp^{5}-\qq^{2}\pp-2\,\pp^{3}+2\,\qq^{2}\pp^{-1}\\
  &\qquad+2\,\qq^{-2}\pp-2\,\pp^{-3}-\qq^{-2}\pp^{-1}+\pp^{-5}+2\,\qq^{-2}\pp^{-3}\Bigr)
  \\
  c_6&=c_9=\frac{[\nv+1]}{[\nv]^2}\left(2\,\qq^{2}\pp+\pp^{3}-2\,\pp-2\,\pp^{-1}+\pp^{-3}+2\,\qq^{-2}\pp^{-1}\right)
  \\
  c_7&=c_8=\frac{[\nv+1]}{[\nv]^2}\Bigl(\qq^{2}\pp^{3}+\pp^{5}+\qq^{2}\pp-\pp^{3}+\qq^{2}\pp^{-1}\\
  &\qquad-\pp-\pp^{-1}+\qq^{-2}\pp-\pp^{-3}+\qq^{-2}\pp^{-1}+\pp^{-5}+\qq^{-2}\pp^{-3}\Bigr)
  \\
  c_{10}&=\frac{\pp^{3}\left[\nv\right]^{2}+\qq\,\pp^{2}\left[\nv\right]-2\,\left[\nv+1\right]-\qq^{-1}\left[\nv\right]+\pp^{-3}\left[\nv\right]^{2}+\pp^{-2}\left[\nv+1\right]}{\left[\nv\right]^2}
  \\
  c_{11}&=c_{14}=-\frac{\left[6\,\nv\right]\ [\nv+1]}{[\nv]\ \left[2\,\nv\right]\ \left[3\,\nv\right]}
  \\
  c_{12}&=c_{13}=-\frac{[\nv+1]}{\left[\nv\right]^2}
  \\
  c_{15}&=
          -2\frac{[\nv+1]}{[\nv]^2}+\frac{\left[6\,\nv\right]}{[\nv]^2\left[3\,\nv\right]}\left([\nv+1]^2+1\right)
\end{align*}
and we postpone the definition of $c_0$.

This relation can be verified by checking that l.h.s.\ and r.h.s.\ have the same inner product with every
element of the basis $\mathcal B(\obj 5)$ whose diagrams are on the r.h.s.
For example, one such inner product is obtained by closing the diagrams in the following fashion:
\[
  \diag{squarepenta-closed}
  = \delta \phi^2 \tau^2 
\]
where we have computed the resulting diagram by using the existing rules of \S\ref{ssec:basicrels}.
On the other hand, performing the same closure on the r.h.s.\ leads to  
\[
   \delta\phi c_2+\delta^2\phi c_3 + \delta\phi c_4 + \delta\phi^2 c_5 + \delta\phi^2 c_6 + \delta\phi\tau c_8 + \delta \phi^3 c_{10} + \delta\phi^3 c_{11} +  \delta\phi\tau^2 c_{12} + \delta\phi\tau^2 c_{14}+\delta\phi^2\tau c_{15} 
\]
One can check that these two expressions are equal.

The exact same computation, but with the external connectivity rotated one step clockwise, namely,
\[
  \diag{squarepenta-closed2}
  = \delta \phi \tau^3
\]
leads to the following expression for $c_0$:
\[
c_0 = \tau^3-\left(c_3+\delta c_4+\phi (c_6+c_7)+\tau c_9 +\phi^2(c_{11}+c_{12})+\tau^2(c_{10}+c_{13})+\tau\phi c_{15}\right)
\]


\providecommand{\bysame}{\leavevmode\hbox to3em{\hrulefill}\thinspace}
\providecommand{\MR}{\relax\ifhmode\unskip\space\fi MR }
\providecommand{\MRhref}[2]{%
  \href{http://www.ams.org/mathscinet-getitem?mr=#1}{#2}
}
\providecommand{\href}[2]{#2}

\end{document}